\documentclass[a4paper,11pt]{amsbook}
\usepackage[latin1]{inputenc}
\usepackage{amsmath,amssymb,amsthm,amsfonts,latexsym,dsfont,environ,mathtools,stmaryrd}
\usepackage[all,knot]{xy}
\usepackage{hyperref,bookmark}

\usepackage[usenames]{color}
\definecolor{red}{rgb}{1,0,0}
\definecolor{white}{rgb}{1,1,1}

\usepackage{pdfpages}
\usepackage{wrapfig,pgf,xcolor}

\usepackage{fancyhdr}

\fancypagestyle{plain}{
\fancyhead{}

}

\numberwithin{section}{chapter}

\hypersetup{
    pdftoolbar=true,
    pdfmenubar=true,
    pdffitwindow=false,
    pdfstartview={FitH},
    pdftitle={Generic properties of semi-Riemannian geodesic flows},
    pdfauthor={R. G. Bettiol},
    pdfsubject={M.Sc. dissertation},
    pdfkeywords={Generic properties, geodesic flows, semi-Riemannian manifolds, Bumpy Metric Theorem, general endpoints conditions}, 
    pdfnewwindow=true,
    colorlinks=true, 
    linkcolor=blue,
    citecolor=blue,
    urlcolor=blue,
}
\bookmarksetup{bold=false}

\theoremstyle{plain}
\newtheorem{theorem}{\sc Theorem}[chapter]
\newtheorem{claim}[theorem]{\sc Claim}
\newtheorem{lemma}[theorem]{\sc Lemma}
\newtheorem{proposition}[theorem]{\sc Proposition}
\newtheorem{corollary}[theorem]{\sc Corollary}

\newtheorem{abgen}[theorem]{\sc Abstract Genericity Criterion}
\newtheorem{eqgen}[theorem]{\sc Equivariant Genericity Criterion}
\newtheorem{transvthm}[theorem]{\sc Transversality Theorem}
\newtheorem{gbcthm}[theorem]{\sc Gauss--Bonnet--Chern Theorem}
\newtheorem{sardthm}[theorem]{\sc Sard Theorem}
\newtheorem{sardsthm}[theorem]{\sc Sard--Smale Theorem}
\newtheorem{bairethm}[theorem]{\sc Baire Theorem}
\newtheorem{weakbumpythm}[theorem]{\sc Weak Bumpy Metric Theorem}
\newtheorem{bumpythm}[theorem]{\sc Bumpy Metric Theorem}
\newtheorem{cibumpythm}[theorem]{\sc $C^\infty$ Bumpy Metric Theorem}
\newtheorem{stwthm}[theorem]{\sc Stone--Weierstrass Theorem}
\newtheorem{rieszthm}[theorem]{\sc Riesz Representation Theorem}

\theoremstyle{definition}
\newtheorem{definition}[theorem]{\sc Definition}
\newtheorem{example}[theorem]{\sc Example}

\theoremstyle{remark}
\newtheorem{remark}[theorem]{\it Remark}

\numberwithin{equation}{chapter}
\numberwithin{figure}{chapter}

\makeatletter
\renewenvironment{proof}[1][\proofname]{\par
  \pushQED{\qed}%
  \normalfont \topsep6\p@\@plus6\p@\relax
  \trivlist
  \item[\hskip\labelsep
        \it
    #1\@addpunct{.}]\ignorespaces
}{%
  \popQED\endtrivlist\@endpefalse
}
\makeatother

\allowdisplaybreaks[1]

\newcommand{\ev}{\operatorname{ev}}
\newcommand{\p}{\mathcal{P}}
\newcommand{\op}{\Omega_{\p}}
\newcommand{\met}{\ensuremath\operatorname{Met}}
\newcommand{\gr}{\ensuremath\operatorname{Gr}}
\newcommand{\chr}{\ensuremath\Gamma}
\newcommand{\sect}{{\boldsymbol{\Gamma}}}
\newcommand{\D}{\boldsymbol{\operatorname{D}}}
\newcommand{\s}{\ensuremath{\mathcal{S}}}
\newcommand{\R}{\mathds{R}}
\newcommand{\A}{\mathcal{A}}
\newcommand{\Z}{\mathds{Z}}
\newcommand{\N}{\mathds{N}}
\newcommand{\C}{\mathds{C}}
\newcommand{\Q}{\mathds{Q}}
\newcommand{\vol}{\operatorname{vol}}
\newcommand{\Vol}{\operatorname{Vol}}
\newcommand\SO{{\rm SO}}
\newcommand\GL{{\rm GL}}
\newcommand{\Ad}{\operatorname{Ad}}
\newcommand{\id}{\operatorname{id}}
\newcommand{\sgn}{\operatorname{sgn}}
\newcommand{\Ric}{\operatorname{Ric}}
\newcommand{\tr}{\operatorname{tr}}
\newcommand{\crit}{\operatorname{Crit}}
\newcommand{\im}{\operatorname{Im}}
\newcommand{\coker}{\operatorname{coker}}
\newcommand{\ind}{\operatorname{ind}}
\newcommand{\eig}{\operatorname{Eig}}
\newcommand{\diff}{\operatorname{Diff}}

\newcommand{\hess}{\ensuremath{\mathrm{Hess}}}
\newcommand{\supp}{\operatorname{supp}}
\newcommand{\Iso}{\ensuremath{\mathrm{Iso}}}
\newcommand{\Lin}{\operatorname{Lin}}
\newcommand{\Bilin}{\operatorname{Bilin}}
\newcommand{\cpc}{\operatorname{K}}

\newcommand{\ver}{\ensuremath{\mathrm{Ver}}}
\newcommand{\hor}{\ensuremath{\mathrm{Hor}}}
\newcommand{\codim}{\ensuremath{\mathrm{codim}}}
\newcommand{\ea}{\ensuremath{\mbox{\rm\AE}}}
\newcommand{\eo}{\ensuremath{\mbox{\rm\OE}}}
\newcommand{\M}{\mathcal M}
\newcommand{\la}{\longrightarrow}
\newcommand{\dd}{\mathrm d}
\newcommand{\curves}[1]{\mathfrak H\left[#1\right]}
\newcommand{\Curves}[2]{\mathfrak H_{#1}\left[#2\right]}
\newcommand{\normc}[2]{\|#2\|_{C^{#1}}}
\newcommand{\norml}[2]{\|#2\|_{L^{#1}}}
\newcommand{\fibr}[1]{{#1}^\diamond}
\newcommand*{\longhookrightarrow}{\ensuremath{\lhook\joinrel\relbar\joinrel\rightarrow}}
\newcommand{\llangle}{\langle\!\!\langle}
\newcommand{\rrangle}{\rangle\!\!\rangle}

\title{Generic properties of semi--Riemannian geodesic flows}
\author{Renato G. Bettiol (IME USP, Brazil)}
\makeindex

\begin{document}

\frontmatter

\includepdf[pages={1-3}]{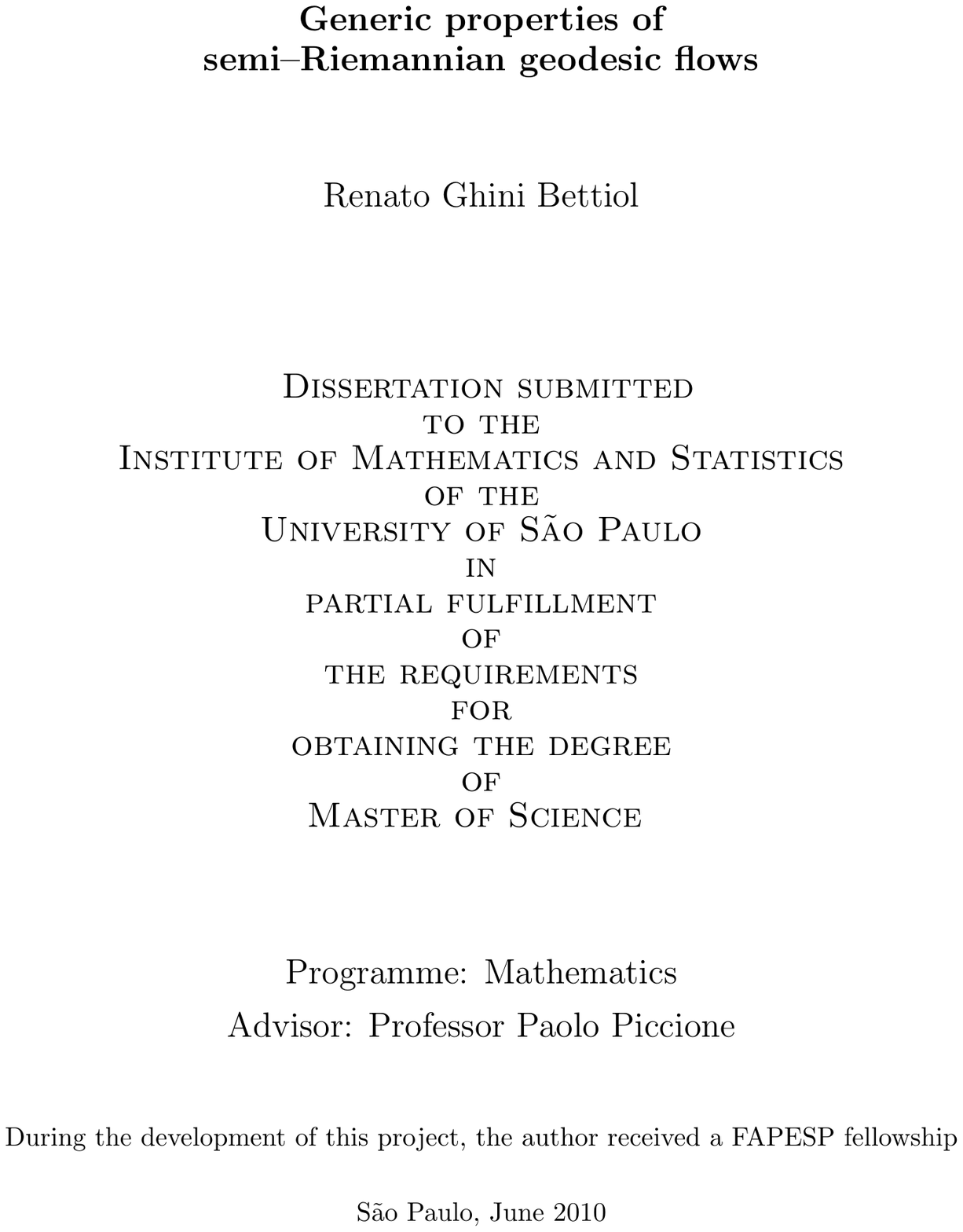}



\renewcommand{\chaptermark}[1]{\markboth{#1}{}}
\renewcommand{\sectionmark}[1]{\markright{#1}}

\cleardoublepage
{\color{white}.}

\renewcommand{\thepage}{\roman{page}}
\setcounter{page}{5}

\vskip 2.35cm
\centerline{\Large\bf Acknowledgements}
\vskip 2cm

{\small
The author gratefully thanks his advisor, Professor Paolo Piccione, and Professor Daniel Victor Tausk for their enormous support during countless fruitful conversations, and for the opportunity of learning from distinguished mathematicians as themselves. Not only they proportioned the best possible environment to stimulate the interest of a student in differential geometry, but also revealed extremely welcoming in the personal dimension.

Several other professors were of great influence, both personally and mathematically, specially Marcos Alexandrino, Leonardo Biliotti, Roberto Giamb\`o, Miguel Angel Javaloyes, Francesco Mercuri and Gaetano Siciliano. In addition, many thanks to Professors Luis Alias, Levi Lima, Guillermo Lobos and Jaime Ripoll for their kind invitations and to Fapesp for sponsoring this project. The author also acknowledges the many math departments that provided an excellent working environment at Universidad de Murcia in Spain, Universit\`a degli Studi di Parma in Italy, and Universidade Federal do Cear\'a, Universidade Federal do Rio Grande do Sul and Universidade Federal de S\~ao Carlos in Brazil.

Last, but not least, many thanks to Marcello Ghini Bettiol, Wagner Bettiol and Raquel Ghini, for their kind understanding and unconditional support during all the weekends away immersed in articles and books. Many thanks also to William Anderson and the Traditional Jazz Band Brasil, Rodrigo Andrade, Stephanie Blum, Ana Carolina Boero, Jorge Cham, Sheldon Cooper, Claudia Correa, Peter Hazard, Takeo Jumonji, Leandro Augusto Lichtenfelz, Fernando Henry Meirelles, Pedro Henrique Pontes, Rodrigo Roque, Lucas Kaufmann Sacchetto, Radu Saghin, Bianca Santoro, Cinthya Maria Schneider, Fabio Simas and Mariana Smit for their constant encouragement, without which this journey would have never been so smooth and pleasant.
}

\cleardoublepage
{\color{white}.}

\vskip 2.35cm
\centerline{\Large\bf Abstract}
\vskip 2cm

{\small
Let $M$ be a possibly non compact smooth manifold. We study genericity in the $C^k$--topology ($3\leq k\leq +\infty$) of nondegeneracy properties of semi--Riemannian geodesic flows on $M$. Namely, we prove a new version of the Bumpy Metric Theorem for a such $M$ and also genericity of metrics that do not possess any degenerate geodesics satisfying suitable endpoints conditions. This extends results of Biliotti, Javaloyes and Piccione \cite{biljavapic} for geodesics with fixed endpoints to the case where endpoints lie on a compact submanifold $\p\subset M\times M$ that satisfies an admissibility condition. Immediate consequences are generic non conjugacy between two points and non focality between a point and a submanifold (or also between two submanifolds). 
}

\vskip 2.35cm
\centerline{\Large\bf Resumo}
\vskip 2cm

{\small
Seja $M$ uma variedade suave possivelmente n\~ao compacta. Estuda--se a genericidade na topologia $C^k$ ($3\leq k\leq +\infty$) de propriedades de n\~ao degeneresc\^encia de fluxos geod\'esicos semi--Riemannianos em $M$. A saber, prova--se uma nova vers\~ao do Teorema de M\'etricas Bumpy para uma tal $M$ e tamb\'em a genericidade de m\'etricas que n\~ao possuem geod\'esicas degeneradas cujos pontos finais satisfazem certas condi\c{c}oes. Isso estende resultados anteriores de Biliotti, Javaloyes and Piccione \cite{biljavapic} para geod\'esicas com extremos fixos para o caso onde os extremos variam em uma subvariedade compacta $\p\subset M\times M$ que satisfaz uma condi\c{c}\~ao de admissibilidade. Consequ\^encias imediatas s\~ao genericidade de n\~ao conjuga\c{c}\~ao entre dois pontos e n\~ao focalidade entre um ponto e uma subvariedade (ou tamb\'em entre duas subvariedades).
}
\tableofcontents

\mainmatter
\renewcommand{\thepage}{\roman{page}}


\setcounter{page}{11}
\phantomsection
\renewcommand{\chaptermark}[1]{\markboth{#1}{}}
\renewcommand{\sectionmark}[1]{\markright{#1}}

\chapter*{Preface}

Genericity of properties of flows is a widely explored topic in dynamical systems, particularly regarding geodesic flows. A property satisfied by some elements of a metric space is called {\em generic} if the subset of elements that satisfy it contains a countable intersection of open dense subsets, i.e., a dense $G_\delta$. This subset is called a {\em generic} subset, and in particular, from the Baire Theorem, a generic subset is dense. In the case of the geodesic flow of a metric $g$, i.e., the flow on the tangent bundle $TM$ whose projection of trajectories on $M$ are the $g$--geodesics, one may analyze genericity of certain properties of metrics on $M$. Roughly, genericity of a such property means that it corresponds to the case of {\em typical} metrics on $M$, or, that an arbitrarily small perturbation of any given metric on $M$ produces a new metric on $M$ with this property. In this sense, generic properties of the geodesic flow give information on the expected behavior of a randomly chosen metric, and on the stability of this property. This stability is of great importance to infer conclusions using manifolds as mathematical models, since it guarantees that small inaccuracies in the observation are physically neglectable.

It is natural to expect that highly symmetric configurations are {\em not} generic, since they are unstable under perturbations. More precisely, consider for instance the isometry groups of a fixed manifold for varying Riemannian metrics. It is reasonable to predict that the subset of metrics on $M$ whose isometry group is trivial ought to be generic. In fact, this result was proved by Ebin \cite{ebin} in the seventies. There is, however, a subtle detail. In this article, genericity is established for {\em Riemannian structures} on $M$, i.e., equivalence classes of metrics on $M$ with respect to the action by pull--back of the diffeomorphisms group of $M$. Through the analysis of this action, particularly through the construction of a slice to the action, it is possible to infer several conclusions on the {\em orbit space} of Riemannian structures. Notice that an isometry of $(M,g)$ in this context is an element of the isotropy group of $g$.

We are interested in genericity of similar symmetry properties of metrics, concerning the existence of {\em degenerate} geodesics. Nevertheless, our approach is somewhat different. Namely, we aim to study generic subsets of the set $\met_\nu^k(M)$ of $C^k$ semi--Riemannian metrics of index $\nu$ on $M$, endowed with the topology induced from certain Banach spaces of tensors on $M$. In this sense, we prove genericity of certain properties of metrics, and not of {\em equivalence classes of metrics} as Ebin \cite{ebin}. In addition, such generic subsets will characterized by properties regarding the absence of degenerate geodesics, which constitute a sort of symmetric configuration, as it will be explained in the sequel.

A couple remarks are necessary at this point. First, we deal with non necessarily positive--definite metrics, i.e., semi--Riemannian metrics, which are nondegenerate symmetric $(0,2)$--tensors. Generic properties of {\em semi--Riemannian} geodesic flows constitute a fairly unexplored area, with a few very recent contributions by Biliotti, Javaloyes and Piccione \cite{biljavapic,biljavapic2} in 2009 and Bettiol and Giamb\`o \cite{metmna} in 2010. The main advantage of this more general context is that mathematical models of space--times in general relativity are also contemplated, together with Riemannian manifolds. Namely, space--times are modeled by four--dimensional manifolds endowed with a semi--Riemannian metric of index $\nu=1$ that satisfies the Einstein equations and have a time orientation, see Definition~\ref{def:spacetime}. Therefore, genericity of certain properties of such metrics clearly indicates that observation of these properties are physically relevant, since stable under small perturbations. More generally, all of our results are valid for higher indexes, and not only for the Lorentzian case $\nu=1$.

Second, the topology of $\met_\nu^k(M)$ is a delicate matter. Since we will deal with non necessarily compact manifolds, this space does not have a natural topology. For this reason, we introduce the concept of {\em $C^k$ Whitney type Banach spaces of tensors on $M$}, which are Banach spaces whose norm depends on a choice of an auxiliary Riemannian metric on $M$. In addition, the choice of another auxiliary metric $g_\mathrm A\in\met_\nu^k(M)$ will be necessary, to avoid empty interior intersections of $\met_\nu^k(M)$ and these Banach spaces, also maintaining its separability. This allows to induce a topology on a subset of $\met_\nu^k(M)$ formed by metrics that are asymptotically equal to $g_\mathrm A$ at infinity, turning it an open subset of a Banach space, in particular a metric space. Although this implies that all generic properties will be proved regarding the $C^k$--topology, standard intersection arguments will be applied to obtain the $C^\infty$ version of all our genericity statements.

Given the above considerations, let us briefly describe the nature of the generic properties of semi--Riemannian geodesic flows studied. A well--known result on generic properties of flows is the so--called \emph{Bumpy} Metric Theorem, stated by Abraham \cite{abraham}, and completely proved by Anosov \cite{anosov} in 1982. Metrics without degenerate periodic geodesics are called {\em bumpy}, since they are rather non symmetric objects. The classic Bumpy Metric Theorem asserts that the set of bumpy Riemannian metrics on a compact manifold $M$ is generic. In other words, the subset of Riemannian metrics on a compact manifold all of whose periodic geodesics do not have any periodic Jacobi field other than the tangent field is generic. Recently, Biliotti, Javaloyes and Piccione \cite{biljavapic2} managed to extend this classic result to the case of compact semi--Riemannian manifolds. In Section~\ref{sec:bumpy}, we prove a further extension of this result to non necessarily compact semi--Riemannian manifolds, the Bumpy Metric Theorem~\ref{thm:bumpy}.

This result paves the way to several possible applications, similarly to its Riemannian version. For instance, the classic Bumpy Metric Theorem was used by Klingenberg and Takens \cite{klitak} to establish further generic properties of the $k^{\mbox{\tiny th}}$ jet of the Poincar\'e map of periodic geodesics, and a similar statement holds in the case of semi--Riemannian manifolds. Nevertheless, counter examples by Meyer and Palmore \cite{mp} point out that abstract Hamiltonian systems cannot be considered for generalizations of the Bumpy Metric Theorem to a more comprehensive class of dynamical flows. Basically, the dynamics of solutions differ in distinct energy levels, and hence the nondegeneracy property fails to be generic. On the other hand, Gon\c{c}alves Miranda \cite{GonMir} proved genericity of nondegenerate periodic trajectories in the context of magnetic flows on a surface, which allows to establish an extension of the Kupka--Smale Theorem.

Furthermore, in Section~\ref{sec:gnggec} we use this Bumpy Metric Theorem~\ref{thm:bumpy} to establish another generic property concerning degenerate geodesics. We prove genericity of semi--Riemannian metrics without degenerate geodesics satisfying certain {\em general endpoints conditions}, or {\em GECs}. This was motivated by a result of Biliotti, Javaloyes and Piccione \cite{biljavapic} that asserts that given two distinct points $p,q\in M$, the set of semi--Riemannian metrics on $M$ for which $p$ and $q$ are not conjugate is generic. This is equivalent to the statement that all geodesics joining $p$ and $q$ are nondegenerate. This nondegeneracy is clearly in the sense that such geodesics are nondegenerate critical points $\gamma:[0,1]\to M$ of the energy functional $$E_g(\gamma)=\int_0^1 g(\dot\gamma,\dot\gamma)\;\dd t$$ for curves with fixed endpoints $\gamma(0)=p$ and $\gamma(1)=q$, i.e., critical points at which the second derivative of the functional is invertible.

Instead of fixing $p$ and $q$, we consider the energy functional for curves whose endpoints vary in a submanifold $\p\subset M\times M$, with certain reasonable properties. This submanifold $\p$ is called a {\em general endpoints condition}, or {\em GEC}. Critical points of the $g$--energy functional for such curves are $g$--geodesics $\gamma:[0,1]\to M$ with $$(\gamma(0),\gamma(1))\in\p \;\; \mbox{ and } \;\; (\dot\gamma(0),\dot\gamma(1))\in T_{(\gamma(0),\gamma(1))}\p^\perp,$$ where $^\perp$ denotes orthogonality with respect to $g\oplus (-g)$. Such geodesics will be called $(g,\p)$--geodesics. Theorem~\ref{thm:bigone} establishes genericity of semi--Riemannian metrics $g$ on $M$ all of whose $(g,\p)$--geodesics are nondegenerate.

In particular, considering for instance $\p=P\times\{q\}$, where $P$ is a submanifold of $M$, we obtain genericity of metrics for which $q$ is not focal to $P$, see Corollary~\ref{cor:nonfocality}. Moreover, setting $\p=\{p\}\times\{q\}$ we recover the result of Biliotti, Javaloyes and Piccione \cite{biljavapic}, with the additional advantage that $p$ and $q$ may be taken as the same point, see Corollary~\ref{cor:nonconjugacy}. The diagonal case $\p=\Delta$ however does not meet most requirements of Theorem~\ref{thm:bigone}, hence one should not expect to derive the Bumpy Metric Theorem from Theorem~\ref{thm:bigone}. In fact, the last {\em uses} the Bumpy Metric Theorem as part of its proof.

Motivation for studying such nondegeneracy generic properties of semi--Riemannian geodesic flows clearly come from possible applications in general relativity, but also from Morse theory. In fact, a crucial assumption to develop a Morse theory for geodesics between fixed points is that the two arbitrarily fixed distinct points must be non conjugate. Recent works by Abbondandolo and Majer \cite{AbbMej2,AbbMaj,AbbMaj2} connect Morse relations for critical points of the semi--Riemannian energy functional to the homology of a doubly infinite chain complex, the Morse--Witten complex, constructed out of the critical points of a strongly indefinite Morse functional, using the dynamics of the gradient flow. The Morse relations for critical points are obtained computing the homology of this complex, which in the standard Morse theory is isomorphic to the singular homology of the base manifolds. Abbondandolo and Majer \cite{AbbMej2} also managed to prove stability of this homology with respect to small perturbations of the metric structure. Thus, it is important to ask whether it is possible to perturb a metric in such a way that the non conjugacy property between two points is preserved. This is precisely the result of Biliotti, Javaloyes and Piccione \cite{biljavapic} above described, that corresponds to the particular case $\p=\{p\}\times\{q\}$ of our Theorem~\ref{thm:bigone}.

Let us give a more precise description of the admissibility hypotheses on a GEC $\p$ for Theorem~\ref{thm:bigone} to hold. First, given a metric $g$ on $M$ it is necessary to endow $\p\subset M\times M$ with a metric related to $g$ with some properties. For some technical reasons that will be clarified along the text, the adequate choice is to consider the product metric $g\oplus (-g)$ on $M$ and then its pull--back to $\p$. Nevertheless, since we are dealing with semi--Riemannian metrics, this is not always possible. Namely, the metric tensor might degenerate at the last step, for every choice of $g$. This is due to the fact that there exist topological obstructions to the existence of semi--Riemannian metrics of a given index, and in case $\p$ has such obstructions, the above procedure is always impossible. More generally, instead of studying the problem of nondegeneracy of certain submanifolds, we give a detailed study of obstructions to the existence of metrics of given index using characteristic classes in Section~\ref{sec:topobst}. For the GEC $\p$ to be admissible, it has to admit such induced metrics. In particular, it must be free of such topological obstructions.

Second, compactness of $\p$ is also necessary to obtain convenient convergent subsequences. Finally, if $\p$ intersects the diagonal $\Delta\subset M\times M$, it is necessary to ensure the existence of a lower bound to the Riemannian length of geodesics with endpoints in $\p$, for all metrics in a small open neighborhood of $g$. A submanifold $\p$ with the above three properties is called an admissible GEC, and to such $\p$'s Theorem~\ref{thm:bigone} may be applied. Admissibility of a large class of GECs that intersect $\Delta$ will be established. Namely, we prove in Proposition~\ref{prop:admissibility} that if $\p$ intersects $\Delta$ transversally, then $\p$ is admissible. In particular, this implies that a generic GEC is admissible.

With the above properties of $\p$ ensured, the proof of genericity of metrics $g$ without degenerate $(g,\p)$--geodesics uses transversality techniques and nonlinear Fredholmness of the $g$--energy functional to verify the hypotheses of an abstract genericity criterion, proved in Section~\ref{sec:abstractgenericity}. The keystone fact in use to establish this abstract criterion is the Sard--Smale Theorem, in a fashion clearly inspired by the previous works of White \cite{white} and Biliotti, Javaloyes and Piccione \cite{biljavapic,biljavapic2}. In addition, part of the techniques used in the proof of the Bumpy Metric Theorem~\ref{thm:bumpy} are transversality arguments very similar to these, with the additional complication imposed by the invariance of the energy functional under the action of $S^1$ reparameterizing periodic curves. For this reason, such abstract genericity criteria are studied separately in Chapter~\ref{chap4} and then applied in the proof of the two generic properties above in Chapters~\ref{chap5} and~\ref{chap6}.

We end with a few conventions and a short overview on the organization of the studied topics. By {\em smooth} we always mean of class $C^\infty$, by {\em operator} we always mean a linear map, and by {\em function} we always mean a map whose image is contained in the set $\R$ of real numbers. The symbol $M$ will always denote a finite--dimensional smooth manifold, and by {\em geodesic} we will always mean an {\em affinely parameterized} geodesic.

The text is divided in two parts, that respectively deal with topics of global analysis and semi--Riemannian geometry and with genericity of nondegenerate semi--Riemannian geodesics. The first part has four chapters and corresponds to the preliminary studies necessary for later applications. Basic objects of semi--Riemannian geometry as bundles, connections, metrics and curvature tensors are briefly recalled in Chapter~\ref{chap1}, together with some remarks on their importance in general relativity. In the last two sections of this first chapter, we respectively deal with topological obstructions to existence of semi--Riemannian metrics and a few auxiliary results. For instance, we prove that a non compact manifold always admits a Lorentzian metric, and compact manifolds admit Lorentzian metrics if and only if their Euler characteristic vanishes. In Chapter~\ref{chap2}, rudiments of functional analysis are recalled, beginning with notions of general theory of Fr\'echet, Banach and Hilbert spaces, compact and Fredholm operators and calculus on Banach spaces. Basic examples of function spaces and more auxiliary results are respectively given in the final sections of this chapter. Chapter~\ref{chap3} deals with infinite--dimensional manifolds, introducing basic terminology and proving elementary transversality results. For instance, it is proved that the preimage of a submanifold by a transverse map is a submanifold, in the context of Banach manifolds. This generalizes the classic result that the preimage of a regular value is a submanifold, which is also stated in this infinite--dimensional context, with the adequate adaptations. In addition, the sets of bounded tensors on a finite--dimensional manifold and Sobolev $H^1$ curves are respectively endowed with a Banach space and Hilbert manifold structures. A special attention is given to Banach spaces that will be used to induce a topology on $\met_\nu^k(M)$. Finally, abstract notions of continuous actions of Lie groups on Hilbert manifolds are studied along with the example of the reparameterization action of $S^1$ on the Hilbert manifold $H^1(S^1,M)$ of Sobolev $H^1$ periodic curves on a finite--dimensional manifold. Finally, a complete treatment of the semi--Riemannian geodesic variational problem under GECs is given in Chapter~\ref{chap35}. Namely, we compute first and second variations of the energy functional and analyze the kernel of its index form. In addition, we establish the existence of a sequence of submanifolds of $H^1(S^1,M)$ with special properties regarding the energy functional, that will be crucial in the proof of the Bumpy Metric Theorem~\ref{thm:bumpy}.

The second part has other four chapters and contains the proofs of our main results. Chapter~\ref{chap4} begins with a proof of the Sard--Smale Theorem and some remarks on genericity. The main results of this chapter are the four abstract genericity criteria, which are proved with the help of the Sard--Smale Theorem. The following Chapter~\ref{chap5} contains a study of iterate periodic geodesics and of a particularly degenerate class of these, called {\em strongly degenerate geodesics}, that play a fundamental role in the proof of the subsequent genericity results. In addition, it contains the proof of our version of the Bumpy Metric Theorem, as well as its version in the $C^\infty$--topology. Chapter~\ref{chap6} deals with admissibility of GECs and the proof of our second genericity result, Theorem~\ref{thm:bigone}, with a few applications. In addition, this generic property is also proved to hold in the $C^\infty$--topology. Finally, some final remarks are made in Chapter~\ref{chap7}, that concludes the text.


\renewcommand{\chaptermark}[1]{\markboth{\thechapter. #1}{}}
\renewcommand{\sectionmark}[1]{\markright{\thesection. #1}}

\newpage
\setcounter{page}{1}
\renewcommand{\thepage}{\arabic{page}}


\part{Global analysis and semi--Riemannian geometry}

\chapter{Rudiments of semi--Riemannian geometry}
\label{chap1}

In this chapter, we begin with a section recalling basic concepts of fiber bundles, vector bundles and connections on a smooth manifold, in order to establish notations and conventions. Furthermore, basic concepts of semi--Riemannian metrics such as geodesics, curvature tensors and Jacobi fields are defined and briefly explored in Section~\ref{sec:metricsetc}. The following section is dedicated to a few results on obstructions to existence of semi--Riemannian metrics of a prescribed index. More precisely, we prove that a smooth manifold $M$ admits a $C^k$ semi--Riemannian metric of index $\nu$ if and only if it admits a $C^k$ distribution of rank $\nu$, see Proposition~\ref{prop:metricdistribution}. In the final section, we prove some lemmas on accumulation of geodesics self intersections will be later used in our applications in Chapters~\ref{chap5} and~\ref{chap6}, see Lemma~\ref{le:finiteintersection} and Proposition~\ref{prop:selfintersections}.

Along this chapter, $M$ denotes a finite--dimensional smooth manifold, and by {\em smooth} we always mean of class $C^\infty$. We assume that the reader is familiar with fundamentals of differential and Riemannian geometry. Although the exposition of some elementary topics aims to keep the text self contained, it is beyond our objectives to give a thorough introduction to the theories of bundles, connections and semi--Riemannian geometry. For a detailed treatment of such topics, we respectively refer to \cite{kn1,kn2,picmertausk,pictaugstructure,walschap} and \cite{bee,oneill}.

\section{Fiber bundles and connections}
\label{sec:fiberbundlesconnections}

In this section, we briefly recall the concept of fiber bundle over $M$, in particular of vector bundle\footnote{Although in this section we shall only discuss vector bundles over $M$ with {\em finite--dimensional fibers}, several definitions and results naturally extend to the context of more general vector bundles over $M$, whose fibers are, for instance, Banach or Hilbert spaces. Such infinite--dimensional approach will be briefly used in the end of Section~\ref{sec:infinitedimmnflds}. A particularly effective reference for such bundles is Lang \cite{lang}.}, define abstract connections on vector bundles and study the particular case of tensor bundles. We follow closely the approach used in \cite{kn1,picmertausk}, to which we refer for a detailed exposition on the subject.

\begin{definition}\label{def:fiberbundle}
Let $\mathcal E$ be a  set and $\pi:\mathcal E\to M$ a map. A {\em trivialization}\index{Trivialization}\index{Fiber bundle!trivialization} of $(\mathcal E,M,\pi)$ is a bijective map $$\alpha:\pi^{-1}(U)\la U\times\mathcal E_0,$$ where $U\subset M$ is open, $\mathcal E_0$ is a finite--dimensional manifold and the following diagram commutes, where $p_1:U\times \mathcal E_0\to U$ denotes the projection.
\begin{equation*}
\xymatrix{\pi^{-1}(U)\ar[rr]^\alpha\ar@(d,l)[dr]_{\pi|_{\pi^{-1}(U)}} && U\times\mathcal E_0\ar@(d,r)[dl]^{p_1} \\
& U}
\end{equation*}
Denoting $p_2:U\times \mathcal E_0\to \mathcal E_0$ the other projection, for each $x\in U$ there is a bijection 
\begin{eqnarray*}
\alpha_x:\mathcal E_x &\la &\mathcal E_0 \\
e &\longmapsto & p_2(\alpha(e))
\end{eqnarray*}
between $\mathcal E_x=\pi^{-1}(x)$ and $\mathcal E_0$, called the {\em fiber}\index{Fiber} associated to the trivialization $\alpha$.

Two trivializations $\alpha:\pi^{-1}(U)\to U\times\mathcal E_0$ and $\beta:\pi^{-1}(V)\to V\times\mathcal E'_0$ are {\em $C^k$ compatible}\index{Trivialization!$C^k$ compatible} if either $U\cap V=\emptyset$ or the bijection $$\beta\circ\alpha^{-1}:(U\cap V)\times\mathcal E_0\ni(x,e_0)\longmapsto(x,\beta_x\circ\alpha_x^{-1}(e_0))\in(U\cap V)\times\mathcal E'_0$$ is a $C^k$ diffeomorphism. A family of pairwise $C^k$ compatible trivializations $\left\{\alpha_i:\pi^{-1}(U_i)\to U_i\times\mathcal E_0^i,i\in I\right\},$ with $M=\bigcup_{i\in I}U_i$ is called a {\em $C^k$ atlas of trivializations}\index{Atlas!of trivializations} for $(\mathcal E,M,\pi)$.

A {\em $C^k$ fiber bundle}\index{Fiber bundle} (or \emph{fibre bundle}) over $M$ consists of a map $\pi:\mathcal E\to M$ and a maximal $C^k$ atlas of trivializations for $(\mathcal E,M,\pi)$. In this case, $\mathcal E$ is called the {\em total space}\index{Fiber bundle!total space}, $M$ the {\em base}\index{Fiber bundle!base}, $\pi$ the {\em projection}\index{Fiber bundle!projection} and $\mathcal E_x=\pi^{-1}(x)$, $x\in M$, the {\em fibers}\index{Fiber}\index{Fiber bundle!fiber} of the fiber bundle.
\end{definition}

\begin{figure}[htf]
\begin{center}
\vspace{0.3cm}
\includegraphics[scale=0.27]{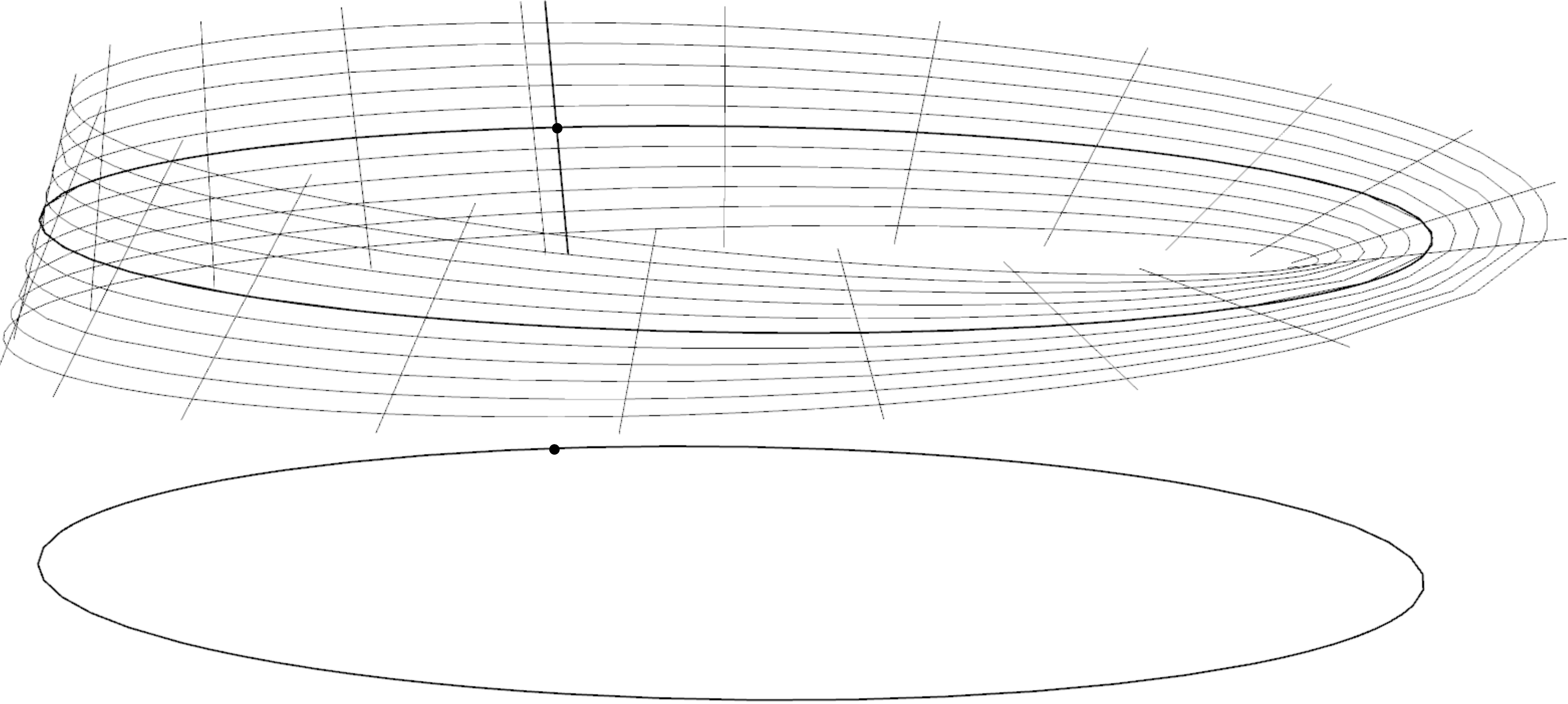}
\begin{pgfpicture}
\pgfputat{\pgfxy(-2,1.8)}{\pgfbox[center,center]{$S^1$}}
\pgfputat{\pgfxy(-8,1.6)}{\pgfbox[center,center]{$x$}}
\pgfputat{\pgfxy(-8,5.5)}{\pgfbox[center,center]{$\pi^{-1}(x)$}}
\end{pgfpicture}
\end{center}
\caption{Example of a fiber bundle. In this example, the total space $\mathcal E$ is the M\"obius strip, the base $M$ is $S^1$ and the fiber is $\R$.}
\end{figure}

\begin{remark}
Notice that different trivializations of $\mathcal E$ are not supposed to have the same fibers. However, if $M$ is connected, all fibers are diffeomorphic and hence there exists an atlas of trivializations whose corresponding fibers are the same.
\end{remark}

\begin{remark}
In the sequel, we call {\em trivialization}\index{Trivialization} only trivializations that belong to the given maximal atlas of a fiber bundle. In addition, the total space $\mathcal E$ alone is called {\em fiber bundle} over $M$ in case the projection and the maximal atlas of trivializations are implicit.
\end{remark}

A $C^k$ atlas of trivializations for a fiber bundle clearly induces a $C^k$ manifold structure on the total space $\mathcal E$ of the bundle, such that trivializations $\alpha:\pi^{-1}(U)\to U\times\mathcal E_0$ are $C^k$ diffeomorphisms defined on open subsets of $\mathcal E$. With this structure, the projection $\pi:\mathcal E\to M$ is a $C^k$ surjective submersion and fibers $\mathcal E_x$ are $C^k$ submanifolds of $\mathcal E$.

\begin{remark}\label{re:containedtrivial}
A subset $U\subset M$ is said to be {\em contained in a trivialization of $\mathcal E$} if $\pi^{-1}(U)$ is contained in the domain of a chart of the maximal atlas of $\mathcal E$. In this case, it is common to use the chart as an identification $\pi^{-1}(U)\cong U\times\mathcal E_x$.
\end{remark}

Locally, we also identify\footnote{In the sequel, we will be somewhat sloppy about this identification, since in some situations it is more convenient to identify the point $e=(x,e_x)$ with its fiber coordinate $e_x$, omitting the base point $x$.} a point $e\in\mathcal E$ with a pair of the form $(x,e_x)$, where $x=\pi(e)$ and $e_x\in\mathcal E_x$. In this local chart, $\pi$ is a projection, as guaranteed by the local form of submersions. The tangent space to the fiber $\mathcal E_x=\pi^{-1}(x)$ at $e\in\mathcal E$ is clearly given by the subspace $$T_e\mathcal E_x=\ker\dd\pi(x)\subset T_e\mathcal E,$$ called the \emph{vertical space} of $\mathcal E$ at $e$, and denoted $\ver_e\, \mathcal E$. Henceforth, $\mathcal E$ will be assumed endowed with such structures.

\begin{example}\label{ex:grassmann}
Consider the set $$\gr_r(M)=\bigcup_{x\in M} \{x\}\times\gr_r(T_xM),$$ where $\gr_r(T_xM)$ is the $r$--Grassmannian of $T_xM$, i.e., set of $r$--dimensional subspaces of $T_xM$. Then $\gr_r(M)$ is a smooth bundle over $M$, with compact fibers, called the {\em $r$--Grassmannian bundle}\index{Grassmannian bundle} over $M$.
\end{example}

\begin{remark}
As a particular case, notice that $\gr_1(\R^{n+1})$ is a (trivial) fiber bundle over $\R^{n+1}$, given by the product $\R^{n+1}\times\R P^n$.
\end{remark}

\begin{definition}\label{def:section}
Using the same notation as above, a $C^k$ map $s:M\to\mathcal E$ with $\pi\circ s=\id$ is called a {\em $C^k$ section}\index{Fiber bundle!section}\index{Section} of $\mathcal E$, see Figure \ref{fig:section}. The set of $C^k$ sections of $\mathcal E$ is denoted $\sect^k(\mathcal E)$. A section that is $C^k$ for every $k\in\N$ is said to be {\em smooth}, and the set of smooth sections of $\mathcal E$ is denoted $\sect^\infty(\mathcal E)$.
\end{definition}

\begin{figure}[htf]
\begin{center}
\includegraphics[scale=0.6]{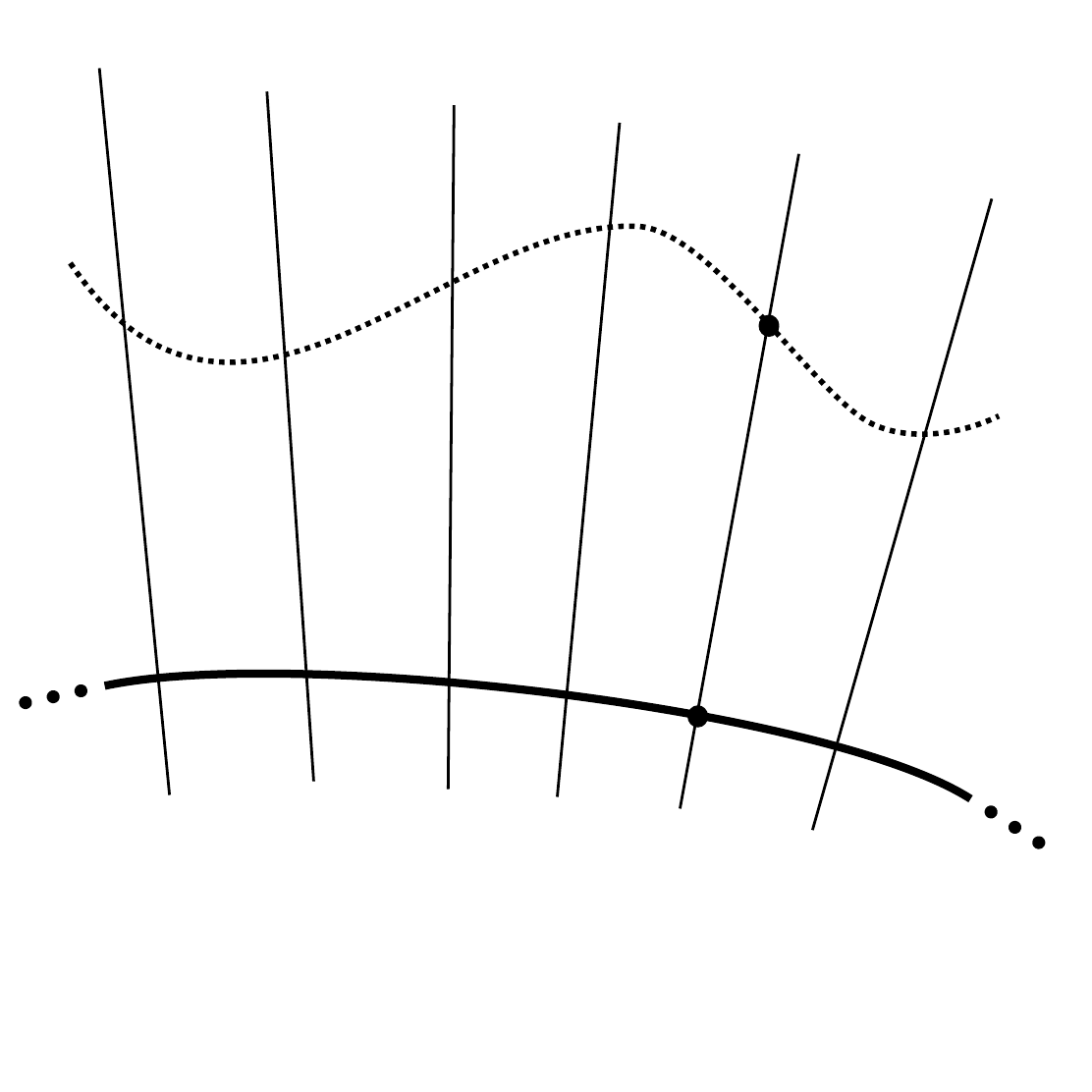}
\begin{pgfpicture}
\pgfputat{\pgfxy(-0.7,4.7)}{\pgfbox[center,center]{$s(x)\in\pi^{-1}(x)$}}
\pgfputat{\pgfxy(-2.5,1.6)}{\pgfbox[center,center]{$x\in M$}}
\pgfputat{\pgfxy(-7.6,5.2)}{\pgfbox[center,center]{$s\in\sect^{k}(\mathcal E)$}}
\end{pgfpicture}
\end{center}
\vspace{-1cm}
\caption{Section of a fiber bundle.}\label{fig:section}
\end{figure}

\begin{definition}\label{def:vectorbundle}
Consider $(\mathcal E,M,\pi)$ a fiber bundle over $M$, such that each fiber $\mathcal E_x$ has a finite--dimensional real vector space structure of dimension\footnote{It is easy to see from the regularity of $\mathcal E$ that the dimension of the fibers must be constant.} $n$. A {\em $C^k$ fiber--linear trivialization}\index{Trivialization!fiber--linear} $\alpha:\pi^{-1}(U)\to U\times\mathcal E_0$ of $\mathcal E$ is a trivialization, with $\mathcal E_0$ a real finite--dimensional vector space, such that the bijection $\alpha_x:\mathcal E_x\to\mathcal E_0$ is linear for every $x\in U$. A $C^k$ fiber bundle $\mathcal E$ endowed with such a real finite--dimensional vector space structure on each fiber $\mathcal E_x$ and a maximal atlas of $C^k$ fiber--linear trivializations is called a {\em $C^k$ vector bundle}\index{Vector bundle} of rank $r$.
\end{definition}

\begin{remark}
Vector bundles will be denoted $E$ instead of $\mathcal E$. Fibers of a vector bundle can be assumed to be equal to a fixed Euclidean space $\R^n$, and in the sequel we call \emph{trivialization}\index{Vector bundle!trivialization} of a vector bundle only fiber--linear trivializations.
\end{remark}

\begin{example}\label{ex:tmtmstar}\index{Bundle!tangent}\index{Bundle!cotangent}
The tangent and cotangent bundles $$TM=\bigcup_{x\in M} \{x\}\times T_xM,\;\quad\; TM^*=\bigcup_{x\in M} \{x\}\times T_xM^*$$ are clearly smooth vector bundles over $M$. Sections of $TM$ and $TM^*$ are respectively called {\em vector fields}\index{Vector field} and {\em $1$--forms}\index{$1$--form} on $M$.
\end{example}

There are two important operations that can be considered for vector bundles, described in the following result.

\begin{proposition}\label{prop:whitneysumtensor}
Let $E_1$ and $E_2$ be two $C^k$ vector bundles over $M$. Then the sets $$E_1\oplus E_2=\bigcup_{x\in M} \{x\}\times(E_1)_x\oplus (E_2)_x$$ and $$E_1\otimes E_2=\bigcup_{x\in M} \{x\}\times(E_1)_x\otimes (E_2)_x$$ admit a $C^k$ vector bundle structure. These are respectively called the {\em Whitney sum}\index{Vector bundle!Whitney sum} and {\em tensor product}\index{Vector bundle!tensor product} of $E_1$ and $E_2$.
\end{proposition}

For a proof of the above proposition, see for instance \cite{husemoller}. It is usually more convenient to describe trivializations of vector bundles using \emph{local frames}.

\begin{definition}\label{def:frame}
A {\em $C^k$ local frame}\index{Local frame}\index{Frame} (or {\em local referential}\index{Local referential}\index{Referential}) of a vector bundle $E$ of rank $r$ is a $r$--uple $\{\xi_i\}_{i=1}^r$ of $C^k$ \emph{local} sections of $E$ defined in an open subset $U\subset M$ such that $\{\xi_i(x)\}_{i=1}^r$ is a basis of $E_x$ for all $x\in U$.

A local frame can also be expressed in the form of linear isomorphisms $$p(x):\R^r\la E_x,$$ that are $C^k$ dependent on $x\in U$. At each $x\in U$, the isomorphism $p(x)$ maps the canonical orthonormal basis of $\R^r$ to the basis $\{\xi_i(x)\}_{i=1}^r$ of $E_x$. Usually, we will prefer this more synthetic description to deal with frames.
\end{definition}

A local frame defines a unique trivialization $\alpha:E|_U\to U\times\R^r$ for which $\alpha_x(v)$ are the coordinates of $v$ in the basis $\{\xi_i(x)\}_{i=1}^r$ of $E_x$, for every $v\in E_x$ and every $x\in U$. Conversely, every trivialization $\alpha:E|_U\to U\times\R^r$ of $E$ arises from a frame $\{\xi_i\}_{i=1}^r$. In fact, for each $x\in U$, let $\xi_i(x)$ be the vector in $E_x$ that is mapped by $\alpha_x$ to the $i^{\mbox{\tiny th}}$ vector of the canonical basis of $\R^r$.

\begin{definition}\label{def:subbundle}
A subset $E'\subset E$ is a {\em $C^k$ vector sub bundle}\index{Vector bundle!sub bundle} of $E$ if $E'_x=E'\cap E_x$ is a vector subspace of $E_x$ for every $x\in M$ and if every point of $M$ has an open neighborhood $U\subset M$ on which there exist $C^k$ sections $\xi_i:U\to E$ of $E$, $i=1,\ldots,r'$, such that $\{\xi_i(x)\}_{i=1}^{r'}$ is a basis for $E'_x$ for every $x\in U$.
\end{definition}

In this case, reducing $U$ if necessary, it is possible to extend $\{\xi_i\}_{i=1}^{r'}$ to a $C^k$ local referential $\{\xi_i\}_{i=1}^r$ of $E$, obtaining a trivialization $\alpha:E|_U\to U\times\R^r$ such that $\alpha_x(E'_x)=\R^{r'}\oplus\{0\}^{r-r'}$ for every $x\in U$. Thus, $E'$ has a natural vector bundle structure with projection $\pi|_{E'}$, and $E'$ is a submanifold of $E$.

\begin{example}\label{ex:distribution}
A $C^k$ sub bundle $\mathcal D$ of the tangent bundle $TM$, see Example~\ref{ex:tmtmstar}, is called a \emph{$C^k$ distribution}\index{Distribution} on $M$, and the dimension of fibers $\mathcal D_x$ is the \emph{rank} of $\mathcal D$.
\end{example}

\begin{definition}\label{def:hor}
Consider a vector bundle $E$ over $M$ and $\ver_e E=\ker\dd\pi(x)$ the vertical space at $e\in E$. Any choice of a complementary subspace of $T_eE$ is called a \emph{horizontal space} at $e$, and denoted $\hor_e E$. Such a \emph{choice} of horizontal spaces at each $e\in E$ gives rise to a \emph{horizontal sub bundle} $\hor\, E$, complementary to the \emph{vertical sub bundle} $\ver\, E$, whose fibers are vertical spaces $\ver_e E$.\index{Vector bundle!horizontal}\index{Vector bundle!vertical}
\end{definition}

\begin{figure}[htf]
\begin{center}
\vspace{-0.5cm}
\includegraphics[scale=0.6]{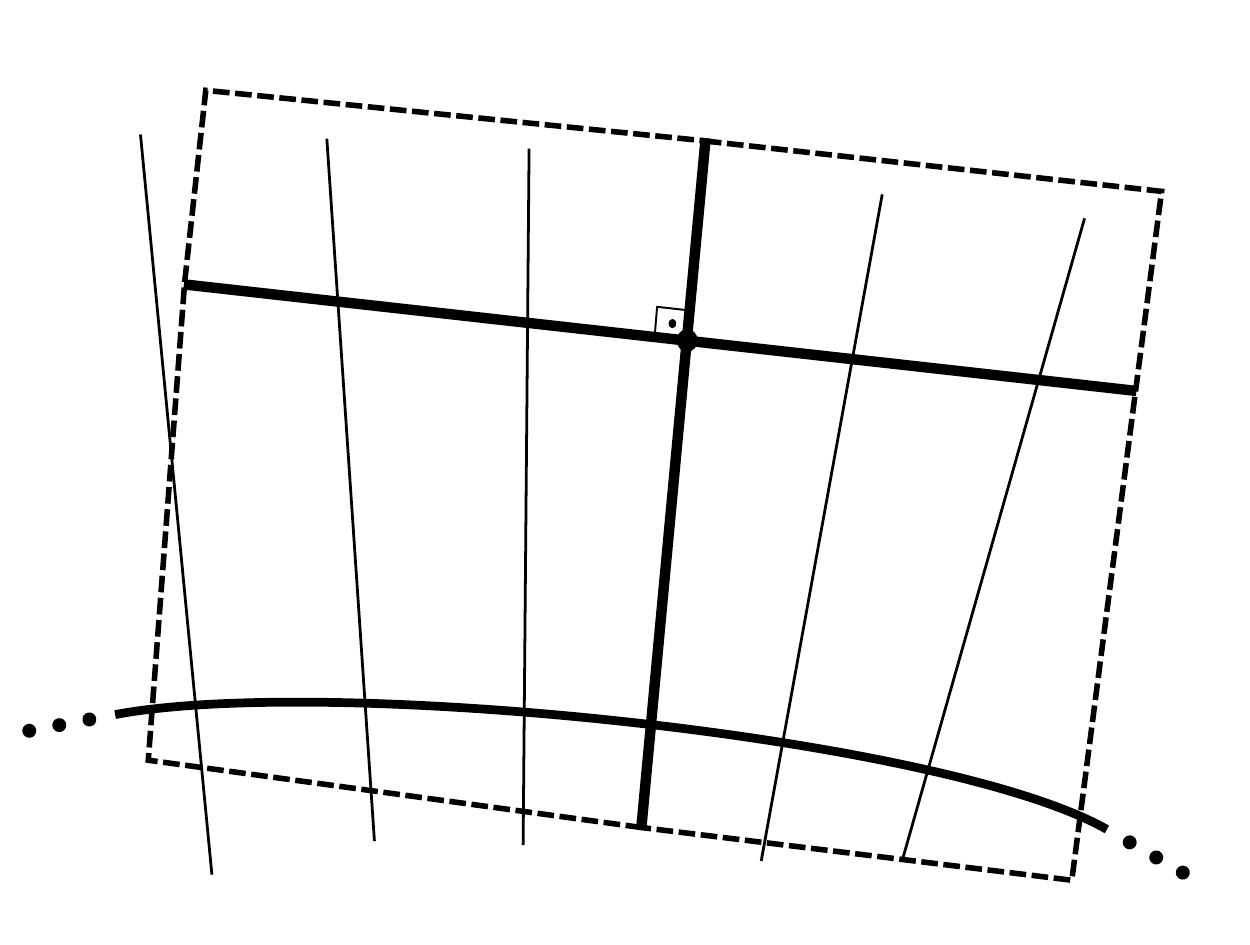}
\begin{pgfpicture}
\pgfputat{\pgfxy(-3.3,3.9)}{\pgfbox[center,center]{$e$}}
\pgfputat{\pgfxy(-0.2,3.4)}{\pgfbox[center,center]{$\hor_e E$}}
\pgfputat{\pgfxy(-3.1,5.2)}{\pgfbox[center,center]{$\ver_e E$}}
\end{pgfpicture}
\end{center}
\caption{Decomposition of $T_eE$ as sum of vertical and horizontal subspaces, $T_eE=\ver_e E\oplus\hor_e E$.}\label{fig:horver}
\end{figure}

We stress that in a \emph{general situation} there is no canonical choice of a horizontal space at $e\in E$, and in fact, such a choice defines a \emph{connection} on $E$ (see Definition~\ref{def:connection} and Remark~\ref{re:connectionhor}). The only fixed choice of horizontal spaces is possible on the \emph{null section} of $E$, as observed in Remark~\ref{re:nullsection}.

\begin{remark}\label{re:secvetsp}
The set of $C^j$ sections of a $C^k$ vector bundle $E$ for any $j=0,\ldots,k$ has a natural real vector space structure induced by the fibers. More precisely, for each $K_1,K_2\in\sect^k(E)$ and $\lambda\in\R$, consider $$(K_1+\lambda K_2)(x)=K_1(x)+\lambda K_2(x),$$ for all $x\in M$, where the right--hand side operations are vector operations of $E_x$. It can be easily verified that the above equation defines a real vector space structure on $\sect^k(E)$.
\end{remark}

In Section~\ref{sec:banachspacetensors}, we will endow (subspaces of) $\sect^k(E)$ with a Banach norm when $E$ is a \emph{tensor bundle} over $M$ (see Definition~\ref{def:tensorbundle} and Proposition~\ref{prop:banachspaceofsections}).

\begin{definition}\label{def:tendstozero}
Suppose that each fiber $E_x$ is endowed with a norm $\|\cdot\|_x$, varying continuously with the point $x$. A section $s\in\sect^k(E)$ is said to {\em tend to zero at infinity}\index{Section!tends to zero at infinity} if for every $\varepsilon>0$ there exists a compact subset $K\subset M$ such that $\|s(x)\|_x<\varepsilon$ for all $x\in M\setminus K$. The vector subspace of such sections is denoted $\sect_0^k(E)$. Notice however that this definition depends on the choice of the norms $\|\cdot\|_x$. Notice also that if $M$ is compact, all sections automatically satisfy this (empty) condition for any norms $\|\cdot\|_x$.
\end{definition}

\begin{remark}\label{re:nullsection}
The zero $\mathbf 0_E\in\sect^k(E)$ of this vector space, called \emph{null section} of $E$, is the map $$\mathbf 0_E:M\ni x\longmapsto (x,0_x)\in E,$$ where $0_x\in E_x$ is the zero. Hence there is a natural identification of $\mathbf 0_E$ with an embedding of the base manifold $M$ in $E$, and by \emph{null section}\index{Section!null} of $E$ we will also mean the \emph{image} of such embedding. This will be formalized in the context of infinite--dimensional vector bundles over infinite--dimensional manifolds in Remark~\ref{re:nullsectionsubmnfld}.

\begin{figure}[htf]
\begin{center}
\vspace{-0.5cm}
\includegraphics[scale=0.6]{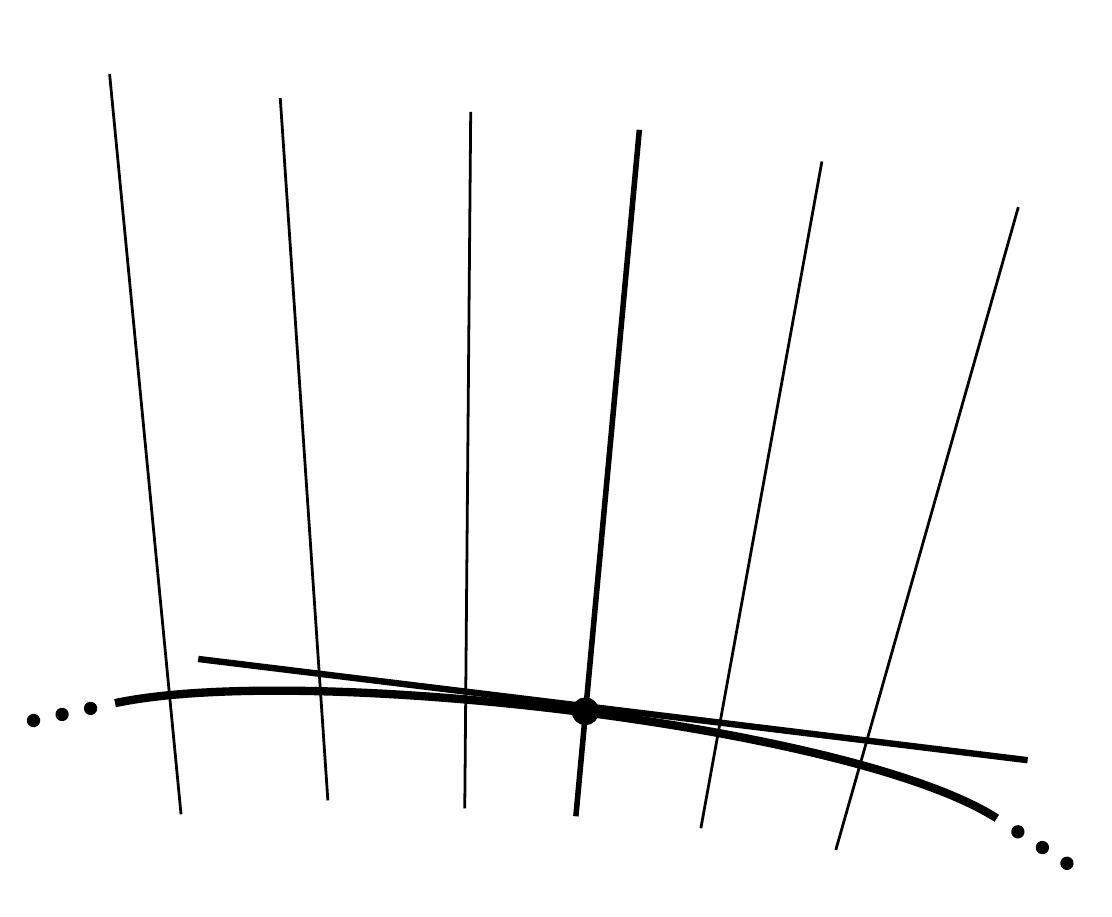}
\begin{pgfpicture}
\pgfputat{\pgfxy(-4,1.7)}{\pgfbox[center,center]{$(x,0_x)$}}
\pgfputat{\pgfxy(0.1,0.9)}{\pgfbox[center,center]{$T_xM$}}
\pgfputat{\pgfxy(-0.2,0)}{\pgfbox[center,center]{$\mathbf 0_E$}}
\pgfputat{\pgfxy(-3,5.1)}{\pgfbox[center,center]{$E_x$}}
\end{pgfpicture}
\end{center}
\caption{Null section $\mathbf 0_E$, horizontal space $\hor_{(x,0_x)} E$ given by $T_{(x,0_x)}\mathbf 0_E\cong T_xM$ and vertical space $\ver_{(x,0_x)} E\cong E_x$.}\label{fig:nullsection}
\end{figure}

Tangent vectors to the null section are called \emph{horizontal} vectors of $E$. Notice that a canonical choice of a horizontal space, i.e., a complementary subspace to $\ver_e E$ (see Definition~\ref{def:hor}), is possible only when $e\in\mathbf 0_E$, given by $$\hor_{(x,0_x)} E=T_{(x,0_x)}\mathbf 0_E \cong T_xM,$$ from the above identification. In this case, the tangent space to $E$ naturally decomposes in the sum of horizontal and vertical spaces, respectively tangent to the null section and to the fibers, i.e.,
\begin{eqnarray}\label{eq:tangentnullsec}
T_{(x,0_x)} E &\cong & T_{(x,0_x)}\mathbf 0_E\oplus T_{(x,0_x)}E_x \nonumber \\
& \cong & T_xM\oplus E_x,
\end{eqnarray}
see Figure \ref{fig:nullsection}. This decomposition naturally generalizes to the infinite--dimensional context, see Remark~\ref{re:ttx}.
\end{remark}

Analogously to the case of real valued functions, the \emph{support}\index{Section!support} of a section $s\in\sect^k(E)$ is defined as 
\begin{equation}\label{eq:support}
\supp s=\overline{M\setminus s^{-1}(\mathbf 0_E)}.
\end{equation}

\begin{example}\label{ex:functions}
The vector space $C^k(M)$ of $C^k$ functions on $M$ is identified with the space of $C^k$ sections of the \emph{trivial} vector bundle $M\times\R$, by the linear isomorphism $$C^k(M)\ni f\longmapsto (\id,f)\in\sect^k(M\times\R)$$ where by $(\id,f)$ we mean the section that maps each $x\in M$ to $(x,f(x))\in M\times\R$. It is a trivial but rather important observation that, with such identification, all results obtained for the structure of the space of $C^k$ sections of a vector bundle over $M$ are automatically valid for the space of $C^k$ functions on $M$.
\end{example}

Let us give a definition that applies Proposition~\ref{prop:whitneysumtensor} inductively on the {\em tangent bundle} $TM$ and the {\em cotangent bundle} $TM^*$, which are clearly (smooth) vector bundles over $M$.

\begin{definition}\label{def:tensorbundle}
A tensor power\footnote{The {\em tensor power} $(\otimes^s TM^*)$ denotes $TM^*\otimes\dots\otimes TM^*$ $s$ times, and analogously for $\otimes^r TM$, recall Example~\ref{ex:tmtmstar}.} $(\otimes^s TM^*)\otimes (\otimes^r TM)$ is called \emph{$(r,s)$--type tensor bundle over $M$}.\index{Tensor bundle}\index{Tensor bundle!$(r,s)$--type} Clearly, its fibers are $$\underbrace{T_xM^*\otimes\dots\otimes T_xM^*}_{s}\otimes \underbrace{T_xM\otimes\dots\otimes T_xM}_{r}=(\otimes^s T_xM^*)\otimes (\otimes^r T_xM).$$ Sections of this bundle are called \emph{$(r,s)$--tensors}.\index{$(r,s)$--tensor}
\end{definition}

Obviously, vector fields, $1$--forms and Riemannian metrics are $(r,s)$--tensors, more precisely, $(1,0)$, $(0,1)$ and $(0,2)$ tensors, respectively. More precisely, it is possible to classify some $(0,s)$--tensors as skew--symmetric or symmetric, respectively.

\begin{definition}\label{def:symskewsym}
A symmetric power\footnote{The {\em symmetric power} $\vee^s TM^*$ denotes $TM^*\vee\dots\vee TM^*$ $s$ times. Recall that if $V$ and $W$ are real vector spaces, $V\vee W$ is a quotient of the tensor product $V\otimes W$ formed by {\em symmetric} tensors.} $\vee^s TM^*$ can be identified with a sub bundle of $\otimes^s TM^*$, whose sections are {\em symmetric} $(0,s)$--tensors (see Definition~\ref{def:subbundle}).\index{Symmetric tensor}\index{Tensor bundle!symmetric} This means that a section $K\in\sect^k(\vee^s TM^*)$ at any $x\in M$, $$K(x):\prod_{i=1}^s T_xM\la\R,$$ is a symmetric $s$--multilinear form. Clearly, the fibers of $\vee^s TM^*$ are $$\underbrace{T_xM^*\vee\dots\vee T_xM^*}_{s}=\vee^s T_xM^*.$$

Analogously, the skew--symmetric power\footnote{The {\em skew--symmetric power} $\wedge^s TM^*$ denotes $TM^*\wedge\dots\wedge TM^*$ $s$ times. Recall that if $V$ and $W$ are real vector spaces, $V\wedge W$ is a quotient of the tensor product $V\otimes W$ formed by {\em skew--symmetric} tensors.} $\wedge^s TM^*$ can be identified with a sub bundle of $\otimes^s TM^*$, consisting of the {\em skew--symmetric} $(0,s)$--tensors, also called {\em differential $s$--forms} (see Definition~\ref{def:subbundle}).\index{Tensor bundle!skew--symmetric}\index{skew--symmetric tensor}\index{Differential form} This means that a section $K\in\sect^k(\wedge^s TM^*)$ at any $x\in M$, $$K(x):\prod_{i=1}^s T_xM\la\R,$$ is a skew--symmetric $s$--multilinear form. Clearly, the fibers of $\wedge^s TM^*$ are $$\underbrace{T_xM^*\wedge\dots\wedge T_xM^*}_{s}=\wedge^s T_xM^*.$$
\end{definition}

\begin{definition}\label{def:pullbackbundle}
If $\mathcal E$ is a $C^k$ fiber bundle over $M$ with projection $\pi:\mathcal E\to M$ and $f:N\to M$ is a $C^k$ map between smooth finite--dimensional manifolds, one can \emph{pull back} $\mathcal E$ to a fiber bundle over $N$. The {\em pull--back bundle}\index{Fiber bundle!pull--back}\index{Pull--back!bundle} $f^*\mathcal E$ is the $C^k$ fiber bundle over $N$  given by $$f^*\mathcal E=\bigcup_{x\in N}\{x\}\times \mathcal E_{f(x)},$$ and the projection $\widehat\pi:f^*\mathcal E\to N$ maps $\{x\}\times\mathcal E_{f(x)}$ to $x\in N$. If $\alpha:\pi^{-1}(U)\to U\times\mathcal E_0$ is a trivialization of $\mathcal E$, then the map $$\widehat\alpha:\widehat\pi^{-1}(f^{-1}(U))\ni (x,e)\longmapsto (x,\alpha_{f(x)}(e))\in f^{-1}(U)\times\mathcal E_0$$ is a trivialization of $f^*\mathcal E$, with the same regularity.

Given a section $s\in\sect^k(\mathcal E)$, one can {\em pull back} $s$ to a section of the pull--back bundle $f^*\mathcal E$. The {\em pull--back section}\index{Pull--back!Section} $f^*s$ is the $C^k$ section of $f^*\mathcal E$ given by
\begin{equation}
(f^*s)(x) = s(f(x)), \quad x\in N.
\end{equation}
Therefore, we have the following diagram relating $s$ and $f^*s$.
\begin{equation*}
\xymatrix@+20pt{
f^*\mathcal E\ar[d]^{\widehat\pi} & \mathcal E\ar[d]_\pi \\
N\ar@(l,l)[u]^{f^*s}\ar[r]^f & M\ar@(r,r)[u]_s
}
\end{equation*}
Notice that not every element of $\sect^k(f^*\mathcal E)$ is of this form, see Remark~\ref{re:extensible} below.
\end{definition}

\begin{remark}\label{re:restbundle}
A special case of \emph{pull--back bundle} is the \emph{restriction} bundle. If $i:N\hookrightarrow M$ is a submanifold and $(\mathcal E,M,\pi)$ is a fiber bundle, then $i^*\mathcal E$, denoted also $\mathcal E|_N$, is a fiber bundle over $N$, whose trivializations are restrictions of trivializations of $\mathcal E$ to $N$.
\end{remark}

\begin{figure}[htf]
\begin{center}
\includegraphics[scale=0.6]{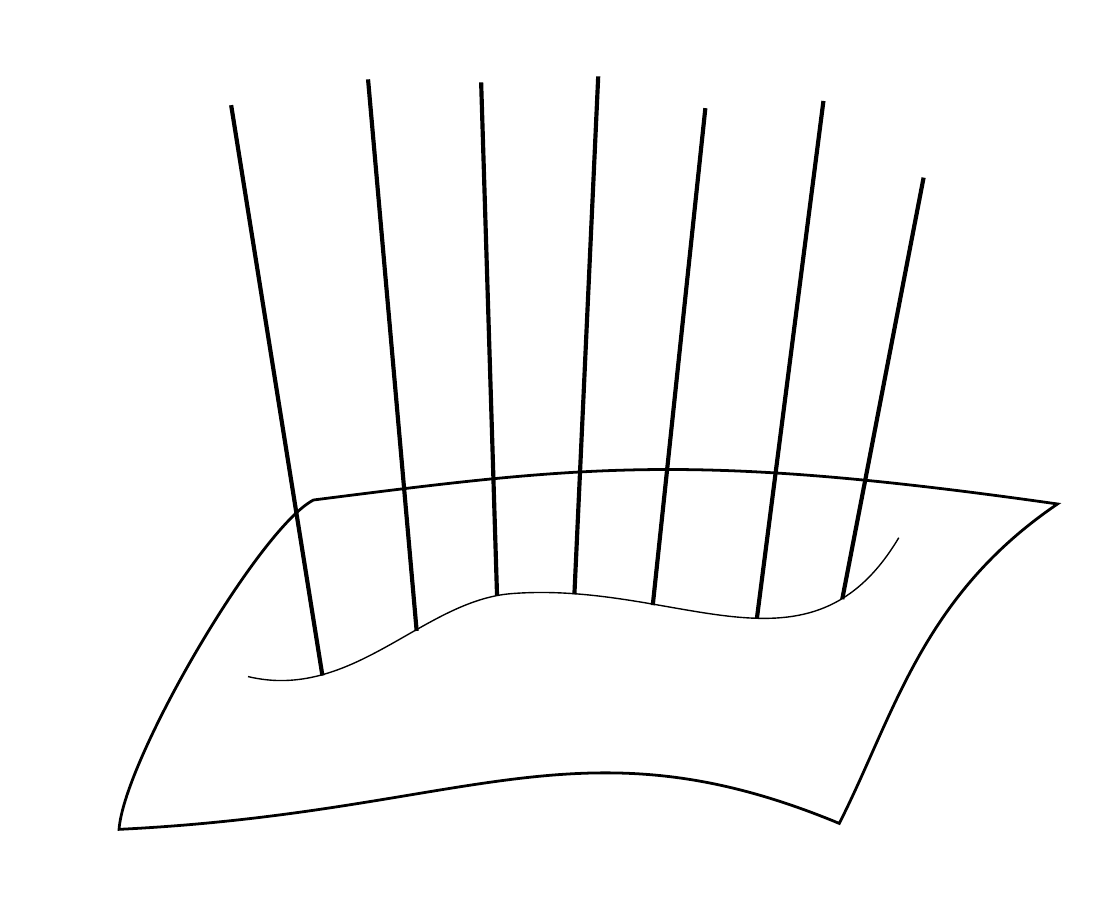}
\begin{pgfpicture}
\pgfputat{\pgfxy(-0.7,4)}{\pgfbox[center,center]{$\mathcal E|_N$}}
\pgfputat{\pgfxy(-0.8,1)}{\pgfbox[center,center]{$M$}}
\pgfputat{\pgfxy(-3,1.5)}{\pgfbox[center,center]{$N$}}
\end{pgfpicture}
\end{center}
\caption{Pull--back bundle $\mathcal E|_N$ over $N\subset M$.}
\end{figure}

\begin{example}\label{ex:vectoralongcurve}
If $\gamma:[a,b]\to M$ is a $C^k$ curve on $M$, the pull--back $\gamma^*TM$ is identified with the restriction of the tangent bundle $TM$ to the image of $\gamma$. Thus, an element $X\in\sect^k(\gamma^*TM)$ is a $C^k$ vector field along $\gamma$, i.e., $X:[a,b]\to TM$, with $X(t)\in T_{\gamma(t)}M$ for all $t\in [a,b]$.\index{Vector field!along a curve} Notice that although $TM$ is a smooth bundle over $M$, the regularity of the pull--back bundle $\gamma^*TM$ is the same of $\gamma$.
\end{example}

\begin{remark}\label{re:extensible}
A vector field $X\in\sect^k(\gamma^*TM)$ that is the pull--back section of some $\widetilde X\in\sect^k(TM)$ is called an {\em extensible} vector field along $\gamma$. Clearly there are vector fields $X$ along $\gamma$ that are not induced as restrictions of globally defined vector fields $\widetilde X\in\sect^k(TM)$. Consider for instance the tangent field $\dot\gamma$ of a curve with transverse self intersections, i.e. $\gamma(t)=\gamma(s)$ and $\dot\gamma(t)\ne\dot\gamma(s)$. This is clearly a non extensible vector field.

In particular, this example recalls that not every section of a pull--back bundle is a pull--back section.
\end{remark}

In the final part of this section, we study connections on vector bundles. A more detailed treatment of the abstract theory of connections can be found in \cite{kn1,picmertausk,pictaugstructure}.

\begin{definition}\label{def:connection}
A {\em connection}, or {\em affine connection},\index{Connection}\index{Vector bundle!connection} on a $C^k$ vector bundle $E$ over $M$ is a $\R$--linear operator $$\nabla:\sect^k(E)\la\sect^{k-1}(TM^*\otimes E),$$ satisfying the Leibniz rule $\nabla(fs)=\dd f\otimes s+f\nabla s,$ for all $f\in C^k(M)$ and $s\in\sect^k(E)$. The term \emph{affine}, often omitted, arises from the fact that the space of all connections on $E$ is an affine space. The image $\nabla s$ is called \emph{covariant derivative} of $s$.\index{Covariant derivative!of a section}
\end{definition}

\begin{remark}
A more common and less useful equivalent definition of connection is the following. A connection is a map $$\nabla:\sect^k(TM)\times\sect^k(E)\owns (X,s)\longmapsto\nabla_X s\in\sect^{k-1}(E)$$ that is $C^k(M)$-linear in $X$, $\R$-linear in $s$ and satisfies the Leibniz rule $\nabla_X (fs)=(X(f))s+f\nabla_X s,$ for all $f\in C^k(M)$ and $X\in\sect^k(TM)$.
\end{remark}

\begin{remark}\label{re:nabladepend}
From Definition~\ref{def:connection}, it is evident that the value of $\nabla_X s$ at a point $x\in M$ depends in different ways of the values of $X$ and $s$. Namely, it only depends of the value $X(x)$ of $X$ at the point $x$, however depends on the values of $s$ in a {\em neighborhood} of $x$. This fact is usually remarked as $\nabla$ being {\em tensorial} only on $X$, and not on $s$.
\end{remark}

\begin{remark}\label{re:connectionhor}
It is possible to prove that the choice of a connection on $E$ is equivalent to the choice of a horizontal bundle $\hor\, E$ with certain properties\footnote{For more details, see Mercuri, Piccione and Tausk \cite[Definition 2.1.6 and Proposition 2.1.12]{picmertausk}}. Let us briefly indicate how to construct such equivalence. Consider $s\in\sect^k(E)$ and $X\in\sect^k(TM)$. Given a horizontal bundle $\hor\, E$, the tangent space at each point $s(x)\in E$ decomposes in the direct sum $T_{s(x)}E=\hor_{s(x)}E\oplus\ver_{s(x)}E$, and the value of $\nabla_X s$ is defined to be the vertical component of $\dd s(x)X$. Conversely, given a connection $\nabla$, the bundle $$\hor_e E=\Big\{\dd s(x)X-\nabla_X s:X\in T_xM, s\in\sect^k(E), s(x)=e\Big\}$$ defines a horizontal bundle that satisfies the appropriate conditions.
\end{remark}

\begin{example}\label{ex:derfunctions}
As pointed out in Example~\ref{ex:functions}, the space $C^k(M)$ of functions on $M$ is identified with $\sect^k(M\times\R)$. Any connection on this trivial bundle acts as the usual derivative of functions, namely for any $f\in C^k(M)$ and $x\in M$,
\begin{equation}\label{eq:derfunctions}
\nabla f(x):T_xM\ni v\longmapsto (x,v(f))\in\{x\}\times\R.
\end{equation}
Notice that setting $s\in C^k(M)$ to be the constant function equal to $1$, it also follows from the Leibniz rule that
\begin{eqnarray*}
(\nabla f)(X) &=& \nabla (1f)(X) \\
&=& (\dd f\otimes 1 +f\nabla 1)(X) \\
&=& \dd f(X) \\
&=& X(f),
\end{eqnarray*}
for any $X\in\sect^k(TM)$. The reason for the covariant derivatives of functions be necessarily the usual derivative is obvious when a connection is identified with a choice of a horizontal bundle as discussed in Remark~\ref{re:connectionhor}. Clearly, $M\times\R$ has a unique possible decomposition in horizontal and vertical bundles, given by, respectively, the tangent spaces to each factor $M$ and $\R$. Hence, there is a unique connection on $M\times\R$, namely, the usual derivative.
\end{example}

\begin{definition}\label{def:pullbackconnection}
Given a connection $\nabla$ on a $C^k$ vector bundle $E$ over $M$ and $f:N\to M$ a $C^k$ map between smooth finite--dimensional manifolds, one can {\em pull back} $\nabla$ to a connection on $f^*E$, see Definition~\ref{def:pullbackbundle}, by setting
\begin{eqnarray*}
f^*\nabla:\sect^k(TN)\times\sect^k(f^*E) &\la &\sect^{k-1}(TN^*\otimes f^*E)\\
(X,f^*s) &\longmapsto &f^*(\nabla_{\dd f(X)} s).
\end{eqnarray*}
\end{definition}

\begin{example}\label{ex:nablasection}
Given a smooth frame\footnote{See Definition~\ref{def:frame}.} $p(x):\R^m\to E_x$ for all $x\in M$ of a vector bundle $E$ over $M$, it is possible to define a connection $\dd^p$ on $E$ associated to $p$ as follows. Let $s\in\sect^k(E)$. Define, for each direction $X\in\sect^k(TM)$,
\begin{equation}\label{eq:nablap}
\big((\dd^p)_X s\big)(x)=p(x)\left[\dd\tilde{s}(x)X(x)\right],
\end{equation}
where $\dd$ is the ordinary differentiation in Euclidean space and $\tilde{s}(x):M\to\R^m$ is the representation of $s$ with respect to the frame $p$ at $x\in M$, i.e.,
\begin{equation}\label{eq:tildes}
\tilde{s}(x)=p(x)^{-1}\big(s(x)\big).
\end{equation}
One can easily verify that $\dd^p$ is a connection in the sense of Definition~\ref{def:connection}. This special connection will be used in the sequel to explore local expressions of tensors, in case $E=TM$.
\end{example}

We finish this section with a couple of definitions for connections on the tangent bundle $TM$, that are particularly important in semi--Riemannian geometry. They allow to {\em parallel translate} vectors along curves, {\em connecting} tangent spaces of $M$ at different points. Notice that given a vector field $X\in\sect^k(TM)$, the covariant derivative $\nabla X\in\sect^{k-1}(TM^*\otimes TM)$ is simply the section that to each $x\in M$ associates the linear operator $$\nabla X(x):T_xM\ni v\longmapsto \nabla_v X\in T_xM.$$

\begin{definition}\label{def:symflat}
There are two important tensors related to connections on the tangent bundle $TM$. Given $\nabla$ a connection on $TM$, define the
\begin{itemize}
\item[(i)] {\em torsion}\index{Connection!torsion} of $\nabla$ to be the skew--symmetric $(1,2)$--tensor
\begin{equation}\label{eq:T}
T^\nabla(X,Y)=\nabla_X Y-\nabla_Y X-[X,Y];
\end{equation}
\item[(ii)] {\em curvature}\index{Connection!curvature} of $\nabla$ to be the $(1,3)$--tensor
\begin{eqnarray}\label{eq:R}
R^\nabla(X,Y)Z &=& \nabla_X\nabla_Y Z-\nabla_Y\nabla_X Z-\nabla_{[X,Y]}Z \\
&=&[\nabla_X,\nabla_Y]Z-\nabla_{[X,Y]}Z,\nonumber
\end{eqnarray}
\end{itemize}
for all $X,Y,Z\in\sect^k(TM)$. Recall that $[\cdot,\cdot]$ denotes the {\em Lie bracket} of vector fields on $M$. Finally, if $T^\nabla$ or $R^\nabla$ vanishes identically, $\nabla$ is respectively called {\em symmetric} or {\em flat}.\index{Connection!symmetric}\index{Connection!flat}
\end{definition}

\begin{definition}\label{def:christoffeltens}
Consider two connections $\nabla$ and $\nabla'$ on the tangent bundle $TM$. The $(1,2)$--tensor given by the difference 
\begin{equation}
\chr=\nabla-\nabla'
\end{equation}
is called the {\em Christoffel tensor of $\nabla$ relatively to $\nabla'$}.\index{Christoffel tensor}

Moreover, given a frame $p(x):\R^m\to T_xM$ for all $x\in M$, it is also possible to define the {\em Christoffel tensor of $\nabla$ relatively to $p$}, as the $(1,2)$--tensor given by the difference
\begin{equation}\label{eq:chrnablap}
\chr=\nabla-\dd^p.
\end{equation}
\end{definition}

\begin{remark}\label{re:chrsym}
Notice that if $\nabla$ and $\nabla'$ are {\em symmetric} connections, the Christoffel tensor of one relatively to the other is also {\em symmetric}.
\end{remark}

\section{Metrics and basic objects}
\label{sec:metricsetc}

In this section we briefly recall basic objects of semi--Riemannian geometry, such as metrics, geodesics, curvature tensors and Jacobi fields among others. For a detailed introduction to the subject, we refer to classic textbooks such as \cite{jost,kn1,kn2,lee,oneill,petersen} and for interpretations and applications to general relativity see \cite{bee,besse,hawking}.

\begin{definition}\label{def:srmetric}
A tensor $g\in\sect^k(TM^*\vee TM^*)$ is a {\em $C^k$ semi--Riemannian metric}\index{Semi--Riemannian!metric}\index{Metric}\index{Metric!semi--Riemannian} of index $\nu$ on $M$ if for all $x\in M$, the bilinear form $$g(x):T_xM\times T_xM\la\R$$ is nondegenerate (see Definition~\ref{def:nondegenerate}) and has index $\nu$, i.e., the dimension of the negative autospace of $g(x):T_xM\to T_xM^*\cong T_xM,$ see \eqref{ident:bilin}, is equal to $\nu$. In case $\nu=0$, this means that $g$ is positive--definite, and then $g$ is called a {\em Riemannian metric}\index{Riemannian!metric}\index{Metric!Riemannian} on $M$. The pair $(M,g)$ is called a {\em semi--Riemannian manifold}.\index{Semi--Riemannian!manifold} In case $\nu=1$, the metric $g$ is called a {\em Lorentzian metric}\index{Lorentzian!metric}\index{Metric!Lorentzian} on $M$, and $(M,g)$ is said to be a {\em Lorentzian manifold}.\index{Lorentzian!manifold}

The set of all $C^k$ semi--Riemannian metrics on $M$ of index $\nu$ is denoted $\met_\nu^k(M)$. Naturally, we also denote $\met_\nu^\infty(M)=\bigcap_{k\in\N}\met_\nu^k(M)$. We will also usually drop the base point $x$ in the notation of the metric, for instance, we will commonly use $g(v,w)$ instead of $g(x)(v,w)$, when there is not ambiguity concerning the base point of the vectors $v,w\in T_xM$.
\end{definition}

\begin{remark}\label{re:fixedgr}
Recall that $M$ is supposed smooth in this text. Using smooth partitions of the unity on $M$, it is possible to prove that $\met_0^\infty(M)\neq\emptyset$, see for instance \cite{jost,petersen}. For many different reasons\footnote{As an example, see Remarks~\ref{re:dependgr} and~\ref{re:dependgr2}, with reference to the use of this auxiliary Riemannian metric in Definition~\ref{def:fiberproduct} and in the subsequent developments.}, we will constantly need an auxiliary Riemannian metric on $M$, that we now fix. Henceforth, $g_\mathrm R\in\met_0^{\infty}(M)$ will denote this fixed smooth Riemannian metric on $M$.

Although it is quite simple to verify that $\met_0^k(M)\neq\emptyset$, for $\nu\geq1$ the set $\met_\nu^k(M)$ might be empty depending on the topology of $M$. In fact, there are obstructions to the existence of semi--Riemannian metrics, which will be studied in Section~\ref{sec:topobst}.
\end{remark}

Since we will be dealing with non necessarily positive--definite metric tensors, the {\em norm} $g(x)(v,v)$ of a vector $v\in T_xM$ might be null or even negative. This gives a classification of tangent vectors (and other associated objects) regarding this sign, called their {\em causal character}.\index{Causal character}

\begin{definition}\label{def:causalchar}
Let $g\in\met_\nu^k(M)$ and $x\in M$. Vectors $v\in T_xM$ are classified regarding their causal character as
\begin{itemize}
\item[(i)] {\em timelike}, if $g(v,v)<0$;\index{Causal character!timelike}
\item[(ii)] {\em nonspacelike} or {\em causal} if $g(v,v)\leq 0$;\index{Causal character!causal}
\item[(iii)] {\em lightlike} or {\em null} if $g(v,v)=0$;\index{Causal character!lightlike}
\item[(iv)] {\em spacelike} if $g(v,v)>0$,\index{Causal character!spacelike}
\end{itemize}
see Figure \ref{fig:lightcone}. A curve $\gamma:[a,b]\to M$ is called {\em timelike}, {\em lightlike} or {\em spacelike} if the tangent vector $\dot\gamma(t)\in T_{\gamma(t)}M$ is respectively {\em timelike}, {\em lightlike} or {\em spacelike}, for all $t\in [a,b]$.
\end{definition}

\begin{figure}[htf]
\begin{center}
\includegraphics[scale=0.8]{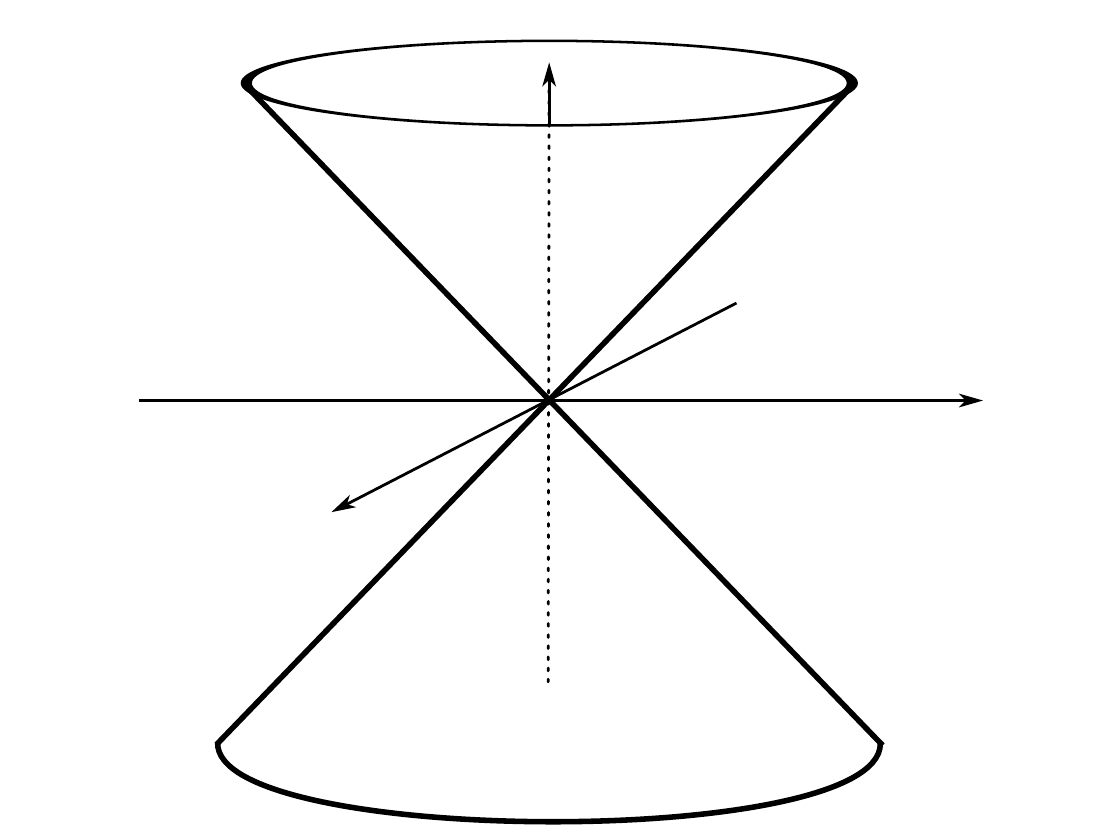}
\begin{pgfpicture}
\pgfputat{\pgfxy(-4.6,0.7)}{\pgfbox[center,center]{$g(v,v)<0$}}
\pgfputat{\pgfxy(-1,0.7)}{\pgfbox[center,center]{$g(v,v)=0$}}
\pgfputat{\pgfxy(-1.9,4)}{\pgfbox[center,center]{$g(v,v)>0$}}
\pgfputat{\pgfxy(-1.2,2.8)}{\pgfbox[center,center]{$T_pM$}}
\end{pgfpicture}
\end{center}
\caption{Possible causal characters of vectors in $T_pM$ and the {\em lightcone} $T^0_pM$, subset of $T_pM$ formed by lightlike vectors.}\label{fig:lightcone}
\end{figure}

\begin{definition}\label{def:gorthframe}
A frame $\{\xi_i(t)\}_{i=1}^m$ along a curve $\gamma:[a,b]\to M$ is a frame\footnote{Recall Definition~\ref{def:frame}.} of the vector bundle $\gamma^*TM$. In addition, given $g\in\met_\nu^k(M)$, it is said to be {$g$--orthonormal}\index{Frame!orthonormal} if for every $t\in [a,b]$, $$g(\xi_i(t),\xi_j(t))=\begin{cases} \delta_i=\pm1, & \text{if } i=j \\ 0, & \text{if }i\neq j. \end{cases}$$
\end{definition}

\begin{definition}\label{def:unittangentbundle}
Let $g\in\met_0^k(M)$ be a Riemannian metric on $M$. The sub bundle $$T^1 M=\bigcup_{x\in M} \{x\}\times\{v\in T_xM: g(x)(v,v)=1\}$$ is called the {\em unit tangent bundle}\index{Bundle!unit tangent} over $M$ with respect to $g$. For semi--Riemannian metrics $g\in\met_\nu^k(M)$, it is also possible to define unit tangent bundles for each causal character. Namely, consider the sub bundles of $TM$ given by $T^1M$ above, $$T^{0}M=\bigcup_{x\in M} \{x\}\times\{v\in T_xM: g(x)(v,v)=0\}$$ and $$T^{-1}M=\bigcup_{x\in M} \{x\}\times\{v\in T_xM: g(x)(v,v)=-1\}.$$ The sub bundle $T^0M$ is called the {\em $g$--light cone bundle} over $M$, and at each $x\in M$, its fiber $T^0_xM=\{v\in T_xM:g(x)(v,v)=0\}$ is called the {\em $g$--light cone}\index{Light cone} at $x\in M$, see Figure \ref{fig:lightcone}.
\end{definition}

\begin{remark}
The $g$--light cone at $x\in M$ divides the tangent space $T_xM$ in two parts. Namely, vectors {\em inside} the light cone are timelike, vectors on the light cone are lightlike and vectors {\em outside} the light cone are spacelike, see Figure \ref{fig:lightcone}. This separation is easily seen, since these components correspond respectively to $f^{-1}(-\infty,0)$, $f^{-1}(0)$ and $f^{-1}(0,+\infty)$ where $$f:T_xM\ni v\longmapsto g(x)(v,v)\in\R.$$
\end{remark}

\begin{definition}\label{def:spacetime}
A vector field $X\in\sect^k(M)$ is timelike if $X(x)\in T_xM$ is timelike for every $x\in M$. A Lorentzian manifold $(M,g)$ with a given timelike vector field $X$ is said to be {\em time oriented}\index{Lorentzian manifold!time oriented} by $X$. A {\em space--time}\index{Space--time} is a time oriented Lorentzian manifold.
\end{definition}

\begin{remark}
Not every Lorentzian manifold admits a time orientation. Nevertheless, if a Lorentzian manifold is not time orientable, it admits a time orientable two--fold cover. This can be proved using a few techniques developed in Section~\ref{sec:topobst} to deal with this type of topological obstructions.
\end{remark}

\begin{definition}\label{def:isometries}
A diffeomorphism $f:(M,g^M)\rightarrow (N,g^N)$ with $f^*g^N=g^M$ is called an {\em isometry}.\index{Isometry} This means that $$g^M(x)(v,w)=g^N(f(x))\big(\dd f(x)v,\dd f(x)w\big),$$ for all $x\in M$ and $v,w\in T_xM$.

The set of all isometries of a given semi--Riemannian manifold $(M,g)$ is clearly\footnote{This is a simple consequence of the chain rule for maps in $M$.} a group under composition of maps, denoted $\Iso(M,g)$ or simply $\Iso(M)$.
\end{definition}

\begin{remark}
A classic result of Myers and Steenrod \cite{myerssteenrod} proves that if $g$ is a Riemannian metric, every closed subgroup of $\Iso(M,g)$ in the compact--open topology\footnote{A subset $G$ of $\Iso(M,g)$ is {\em closed in the compact--open topology} if the following condition holds. Let $K\subset M$ be a compact subset and $\{f_{n}\}_{n\in\N}$ a sequence of isometries in $G$ that converges uniformly in $K$ to a continuous map $f:M\rightarrow M$. Then $f\in G$.} is a Lie group. In particular, $\Iso(M,g)$ itself is a Lie group.

Using this result, it is possible to prove that the group of {\em affine diffeomorphisms} of $M$, i.e., diffeomorphisms that preserve a connection, is a Lie group. This is done by regarding it as the isometry group of another Riemannian manifold. From this fact, it also follows that the isometry group $\Iso(M,g)$ of a semi--Riemannian manifold is a Lie group.
\end{remark}

\begin{remark}
If $g$ is a Riemannian metric, compactness of $M$ implies compactness of $\Iso(M,g)$. This is false for a general semi--Riemannian metric. Nevertheless, there are some interesting results in the literature, for instance D'Ambra \cite{dambra} proved that the isometry group of a real analytic simply connected compact Lorentzian manifold is compact. Recently, Piccione and Zeghib \cite{piczeghib} proved this result without the analyticity hypothesis, assuming the existence of a somewhere timelike Killing vector field.
\end{remark}

Given a semi--Riemannian metric $g$ on $M$ there is a canonical way to associate a connection $\nabla^g$ on $TM$, see Definition~\ref{def:connection}, that is {\em compatible} with $g$, as the following classic result asserts.

\begin{theorem}\label{thm:levicivita}
Let $(M,g)$ be a semi--Riemannian manifold. There exists a unique symmetric\footnote{See Definition~\ref{def:symflat}.} connection $\nabla^g$ on $TM$, called the {\em Levi--Civita connection}\index{Connection!Levi--Civita} of $g$, that is {\em compatible with $g$}, i.e.
\begin{equation}
Xg(Y,Z)=g(\nabla^g_X Y,Z)+g(Y,\nabla^g_X Z), \quad X,Y,Z\in\sect^k(TM).
\end{equation}
\end{theorem}

\begin{remark}
In the case of the fixed Riemannian metric $g_\mathrm R$, its Levi--Civita connection will be denoted $\nabla^\mathrm R$.
\end{remark}

The key fact on the proof of this theorem is the equation known as {\em Koszul formula}.\index{Koszul formula}
\begin{multline}\label{eq:koszul}
g(\nabla^g_Y X,Z)=\tfrac{1}{2}\Big( Xg(Y,Z)-Zg(X,Y)+Yg(Z,X) \\ -g([X,Y],Z)-g([X,Z],Y)-g([Y,Z],X)\Big).
\end{multline}
It exhibits the natural candidate to the Levi--Civita connection and shows that it is uniquely determined by the metric. A complete proof of Theorem~\ref{thm:levicivita} can be found in any basic Riemannian geometry textbook such as \cite{jost,petersen}.

\begin{remark}\label{re:chrg}
The Koszul formula also allows to compute, as follows, the Christoffel tensor of the Levi--Civita connection $\nabla^g$ relatively to a fixed connection $\nabla$ (see Definition~\ref{def:christoffeltens}).
\begin{equation}\label{eq:christoffeltensor}
g(\chr^g(X,Y),Z)=\tfrac12\Big(\nabla g(X,Z,Y)+\nabla g(Y,Z,X)-\nabla g(Z,X,Y)\Big).
\end{equation}
Notice that since $\nabla^g$ is symmetric, if $\nabla$ is symmetric, then $\chr^g$ is also symmetric as a consequence of Remark~\ref{re:chrsym}.
\end{remark}

\begin{definition}\label{def:christoffelsymb}
Consider a local chart $(U,\varphi)$ of $M$ and the local frame of $TM$ at $U$ given by the {\em coordinate basis}\footnote{The frame $p(x):\R^m\to T_xM$ given by the coordinate basis consists of local sections $\{\xi_i(x)\}_{i=1}^m$, where $\xi_i=\varphi^*e_i$ is the pull--back by $\varphi$ of the canonical orthonormal basis of $\R^m$, see Definition~\ref{def:frame}.} induced by $\varphi$. The {\em Christoffel symbols}\index{Christoffel symbol} of $g$ are the functions $\chr^k_{ij}$ in $U$ that give the local expression of the Christoffel tensor $\chr^g$ of $\nabla^g$ relatively to $p$ at $U$, defined by
\begin{equation}\label{eq:christoffelsymb}
\chr^g(\xi_i,\xi_j)=\sum_{k=1}^m \chr^k_{ij}\xi_k.
\end{equation}
\end{definition}

\begin{remark}
Let us compute the Christoffel tensor of the Levi--Civita connection $\nabla^g$ relatively to a frame $p$ (see Definition~\ref{def:christoffeltens}). From Definition~\ref{def:christoffeltens}, this Christoffel tensor is given by formula \eqref{eq:chrnablap},
\begin{equation*}
\chr^g(\gamma)(X,Y)=\nabla_{X} Y-\dd^p_{X} Y,
\end{equation*}
where $\dd^p$ is the connection induced by the frame $p$, defined in Example~\ref{ex:nablasection} by formula \eqref{eq:nablap}. More precisely,
\begin{equation*}
\big((\dd^p)_{X} Y\big)(x)=p(x)\left[\dd\tilde{Y}(x)X(x)\right],
\end{equation*}
where $\tilde{Y}$ is the representation of $Y$ with respect to $p$, given by \eqref{eq:tildes}. Thus, we obtain
\begin{equation}\label{eq:nablachr}
\nabla^g_X Y(x)=p(x)\big(\dd\tilde{Y}(x)X\big)+\chr^g(x)(X,Y), \quad x\in M.
\end{equation}
Formula \eqref{eq:nablachr} is usually known as the {\em covariant derivative formula}, expressed in local coordinates using the components of the frame $p$ and Christoffel symbols, see Definition~\ref{def:christoffelsymb} above.

Furthermore, thinking a connection as a choice of a horizontal bundle, as explained in Remark~\ref{re:connectionhor}, we have that $\chr^g(x)(X,Y)$ corresponds to the horizontal component and $\nabla^g_X Y(x)$ to the vertical component of $p(x)\big(\dd\tilde{Y}(x)X\big)$.
\end{remark}

We now aim to endow each tensor bundle over $M$, see Definition~\ref{def:tensorbundle}, with a natural connection by using the fixed Riemannian metric $g_\mathrm R$ on $M$ and its Levi--Civita connection
\begin{equation}\label{eq:nablar}
\nabla^\mathrm R:\sect^k(TM)\la\sect^{k-1}(TM^* \otimes TM).
\end{equation}
Let us first comment on a particular case, namely the case of the cotangent bundle $TM^*$.

\begin{proposition}\label{prop:cotangentconnection}
The Levi--Civita connection $\nabla^\mathrm R$ induces a natural connection on the cotangent bundle $TM^*$, denoted by the same symbol and given by
\begin{equation}\label{eq:nablarform}
\begin{aligned}
\nabla^\mathrm R:\sect^k(TM^*)\ni\omega &\longmapsto \nabla^\mathrm R\omega\in\sect^{k-1}(TM^*\otimes TM^*)\\
(\nabla^\mathrm R\omega)(Y,X) &=\nabla^\mathrm R_Y\omega(X)-\omega(\nabla^\mathrm R_Y X),
\end{aligned}
\end{equation}
for any $C^{k-1}$ vector fields $X,Y$.
\end{proposition}

\begin{proof}
First, we observe that $\nabla^\mathrm R_Y$ of a $C^{k-1}$ function $f$, such as $\omega(X)$, is simply its usual derivative $Y(f)$, as discussed in Example~\ref{ex:derfunctions}, and that it is an elementary verification that \eqref{eq:nablarform} is a well--defined $\R$-linear operator.

Thus, it only remains to verify that $\nabla^\mathrm R$ satisfies the Leibniz rule. In fact, given $\omega\in\sect^k(TM^*)$, $f\in C^k(M)$ and $X,Y\in\sect^k(TM)$,
\begin{eqnarray*}
\nabla^\mathrm R(f\omega)(Y,X) &=& \nabla^\mathrm R_Y(f\omega(X))-f\omega(\nabla^\mathrm R_Y X) \\
&=& Y(f\omega(X))-f\omega(\nabla^\mathrm R_Y X) \\
&=& Y(f)\omega(X)+fY(\omega(X))-f\omega(\nabla^\mathrm R_Y X) \\
&=& \dd f(Y)\omega(X)+f\left[ \nabla^\mathrm R_Y\omega(X)-\omega(\nabla^\mathrm R_Y X)\right] \\
&=& \left[\dd f\otimes\omega+f\nabla^\mathrm R\omega\right](Y,X).\qedhere
\end{eqnarray*}
\end{proof}

\begin{theorem}\label{thm:tensorconnection}
The Levi--Civita connection $\nabla^\mathrm R$ induces a natural connection on the $(r,s)$--type tensor bundle over $M$,\index{Connection!induced}\index{Tensor bundle!induced connection} denoted by the same symbol and given by
\begin{equation*}
\nabla^\mathrm R:\sect^k({TM^*}^{(s)}\otimes TM^{(r)})\ni K\longmapsto\nabla^\mathrm R K\in\sect^{k-1}({TM^*}^{(s+1)}\otimes TM^{(r)})
\end{equation*}
\begin{eqnarray}\label{eq:inducedconnection}
&(\nabla^\mathrm R K)(Y,X_1,\dots,X_s,\omega_1,\dots,\omega_r)=\nabla^\mathrm R_Y(K(X_1,\dots,X_s,\omega_1,\dots,\omega_r)) &\nonumber \\
&\hspace{4.6cm}-\textstyle\sum_{i=1}^s K(X_1,\dots,\nabla^\mathrm R_Y X_i,\dots,X_s,\omega_1,\dots,\omega_r)& \\
&\hspace{4.6cm}-\textstyle\sum_{j=1}^r K(X_1,\dots,X_s,\omega_1,\dots,\nabla^\mathrm R_Y\omega_j,\dots,\omega_r),&\nonumber
\end{eqnarray}
for any $C^{k-1}$ vector fields $Y,X_i$ and $1$--forms $\omega_j$.
\end{theorem}

Recall that $\nabla^\mathrm R_Y$ of a $C^{k-1}$ $1$--form $\omega$ is given by \eqref{eq:nablarform}, which can also be deduced from the general expression \eqref{eq:inducedconnection} by setting $r=0$ and $s=1$. In addition, $\nabla^\mathrm R_Y$ of a $C^{k-1}$ function $f$ is its usual derivative $Y(f)$, as observed in Example~\ref{ex:derfunctions}. Notice that \eqref{eq:inducedconnection} is a natural extension of \eqref{eq:derfunctions}, \eqref{eq:nablar} and \eqref{eq:nablarform}.

A proof of the above theorem is a simple verification that the expression \eqref{eq:inducedconnection} defines a linear operator that satisfies the Leibniz rule. Since it is totally analogous to the particular case studied in Proposition~\ref{prop:cotangentconnection}, it will be omitted. Henceforth, we will denote $\nabla^\mathrm R$ any connection on a tensor bundle over $M$ induced as above.

\begin{corollary}\label{cor:jcovder}
If $j\leq k$, any $C^k$ $(r,s)$--tensor $K$ has a {\em $j^{\mbox{\tiny th}}$ covariant derivative}\index{Covariant derivative} $(\nabla^\mathrm R)^j K$, which is a $C^{k-j}$ $(r,s+j)$--tensor.
\end{corollary}

In fact, define $(\nabla^\mathrm R)^2 K =\nabla^\mathrm R(\nabla^\mathrm R K),$ and inductively, $$(\nabla^\mathrm R)^jK=\nabla^\mathrm R\left[(\nabla^\mathrm R)^{j-1}K\right].$$ This allows to compute high order covariant derivatives of any $(r,s)$--tensors, and will be used in Section~\ref{sec:banachspacetensors} to endow (subspaces of) $\sect^k(E)$ with a Banach space norm, for tensor bundles $E$ over $M$.

Let us now continue to explore elementary aspects of semi--Riemannian geometry, defining covariant differentiation of vector fields along curves and geodesics.

\begin{proposition}\label{prop:covder}
Let $g\in\met_\nu^k(M)$ and $\gamma:[a,b]\to M$ a $C^k$ curve. The Levi--Civita connection $\nabla^g$ induces a unique operator
\begin{equation}
\D^g:\sect^k(\gamma^*TM)\la\sect^{k-1}(\gamma^*TM)
\end{equation}
called {\em covariant derivative}\index{Covariant derivative} operator, that satisfies the Leibniz rule
\begin{equation}
\D^g(fX)=f'(t)X+f\D^g X, \quad X\in\sect^k(\gamma^*TM),f\in C^k(\R,\R),
\end{equation}
and satisfies
\begin{equation}\label{eq:dgx}
\D^g X=\nabla^g_{\dot\gamma} \widetilde X
\end{equation}
if $X$ is induced\footnote{See Remark~\ref{re:nabladepend}.} from a vector field $\widetilde X\in\sect^k(TM)$.
\end{proposition}

\begin{remark}
The covariant derivative operator along curves induced by $\nabla^\mathrm R$ will be denoted $\D^\mathrm R$. In Remark~\ref{re:covderh1} we will comment on how to reduce the regularity hypotheses from class $C^k$ to weaker assumptions and still have a covariant derivative operator $\D^\mathrm R$ defined {\em almost everywhere}.
\end{remark}

For the following results, consider a fixed semi--Riemannian metric $g\in\met_\nu^k(M)$.

\begin{definition}\label{def:geodesic}
An {\em affinely parameterized} $C^2$ curve $\gamma:[a,b]\to M$ is a {\em $g$--geodesic}\index{Geodesic} if it satisfies $\D^g\dot\gamma=0$. In local coordinates, this is a second--order system of ODEs called the {\em $g$--geodesic equation}, that involves the Christoffel symbols of $g$. As usual, when the metric $g$ is evident from the context it will be omitted.
\end{definition}

\begin{remark}\label{re:geodequation}
Let $\gamma:[a,b]\to M$ be a $g$--geodesic. Since $\dot\gamma$ can always be locally extended\footnote{From Remark~\ref{re:nabladepend}, to compute $\nabla^g_{\dot\gamma}\dot\gamma$ at $\gamma(t)$, the vector field $\dot\gamma$ must be defined in an open neighborhood of $\gamma(t)$. Even if $\gamma$ has self intersections of the type $\gamma(t)=\gamma(s)$ and $\dot\gamma(t)\ne\dot\gamma(s)$, see Remark~\ref{re:extensible}, the vectors $\dot\gamma$ at $t$ and $s$ can be {\em locally} extended in different ways {\em around $t$ and $s$}. Since the matter is local, we may use each different extension to compute $\nabla^g_{\dot\gamma}\dot\gamma$ at $t$ and $s$ separately.}, we may apply the covariant derivative formula \eqref{eq:nablachr} for $\dot\gamma$. Using a local frame $p(x):\R^m\to T_xM$, $x\in U$, as an identification with Euclidean space, the $g$--geodesic equation locally reads\index{Geodesic!equation}
\begin{equation}\label{eq:geodequation}
\ddot\gamma(t)+\chr^g(\gamma(t))(\dot\gamma(t),\dot\gamma(t))=0, \quad t\in\gamma^{-1}(U).
\end{equation}
Usually, the geodesic equation is expressed in terms of the Christoffel symbols, see Definition~\ref{def:christoffelsymb}. To obtain this equation, it suffices to express \eqref{eq:geodequation} in terms of a local chart and use \eqref{eq:christoffelsymb}. In this text, we use exclusively {\em coordinate--free} notation such as \eqref{eq:geodequation}, refusing to work with incomprehensible formulas that yield a plethora of indexes.
\end{remark}

\begin{corollary}\label{cor:geodck}
Let $g\in\met^k_\nu(M)$. If $\gamma:[a,b]\to M$ is a $g$--geodesic, then $\gamma$ is of class $C^{k+1}$.
\end{corollary}

\begin{proof}
Since $\gamma$ is a $g$--geodesic, in local coordinates it satisfies the $g$--geodesic equation \eqref{eq:geodequation}, i.e.,
\begin{equation}\label{eq:geodequationk}
\ddot\gamma=-\chr^g(\gamma)(\dot\gamma,\dot\gamma).
\end{equation}
Notice that the Christoffel tensor $\chr^g$ is of class $C^{k-1}$, since it involves first derivatives of $g$ which is $C^k$, see \eqref{eq:christoffeltensor}. From Definition~\ref{def:geodesic}, it follows that $\gamma$ is of class $C^2$. Inductively, suppose that $\gamma$ is of class $C^j$, for some $2\leq j\leq k$. Then the map $$t\longmapsto-\chr^g(\gamma(t))(\dot\gamma(t),\dot\gamma(t))$$ is a composite of $C^{j-1}$ maps, hence of class $C^{j-1}$. Thus, from \eqref{eq:geodequationk} it follows that $\ddot\gamma$ is of class $C^{j-1}$, hence $\gamma$ is of class $C^{j+1}$. This argument works for $2\leq j\leq k$. Therefore, applying it for $j=k$, we may conclude that $\gamma$ is of class $C^{k+1}$.
\end{proof}

\begin{corollary}
If $g\in\met_\nu^\infty(M)$, then $g$--geodesics are smooth curves.
\end{corollary}

\begin{proof}
The result follows directly from Corollary~\ref{cor:geodck}, since $g$ is of class $C^k$ for all $k\in\N$.
\end{proof}

Applying the classic ODE theorem that guarantees existence and uniqueness of solutions, one can prove the following result.

\begin{proposition}\label{prop:geodexist}
For any $x\in M$, $t_0\in\R$ and $v\in T_xM$, there exist an open interval $I\subset\R$ containing $t_0$ and a $g$--geodesic $\gamma:I\rightarrow M$ satisfying the initial conditions $\gamma(t_0)=x$ and $\dot\gamma(t_0)=v$. In addition, any two $g$--geodesics with those initial conditions agree on their common domain.
\end{proposition}

Furthermore, from uniqueness of the solution, it is possible to obtain a maximal $g$--geodesic with this prescribed initial data.

\begin{example}\label{ex:periodicgeod}
Let us briefly introduce a very important class of geodesics, namely {\em periodic geodesics}. A $g$--geodesic $\gamma:[a,b]\to M$ is said to be {\em periodic}\index{Geodesic!periodic} if $$\gamma(a)=\gamma(b)\;\;\mbox{ and }\;\;\dot\gamma(a)=\dot\gamma(b).$$ If only the first condition is satisfied, i.e., $\gamma$ is a geodesic and a periodic curve, then $\gamma$ is called a {\em $g$--geodesic loop}.\index{Geodesic!loop}

\begin{figure}[htf]
\begin{center}
\vspace{-0.7cm}
\includegraphics[scale=0.8]{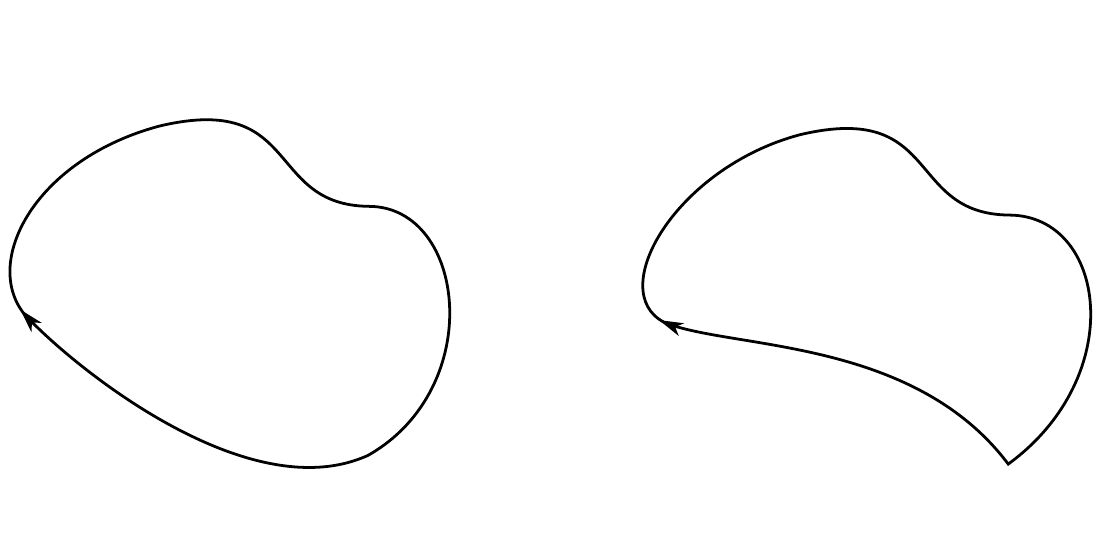}
\end{center}
\vspace{-0.5cm}
\caption{A periodic geodesic (on the left) and a geodesic loop (on the right).}
\end{figure}

From Proposition~\ref{prop:geodexist}, it is clear that under these conditions, the geodesic $\gamma$ can be extended to any interval that contains $[a,b]$, being $\R$ its maximal domain. The extensions of $\gamma$ are still geodesics, that may not be periodic themselves, but are recognizable as {\em portions} of periodic geodesics by counting its self intersections, see Proposition~\ref{prop:selfintersections}.

Each periodic geodesic $\gamma:[a,b]\to M$ has a {\em prime geodesic}\index{Geodesic!periodic!prime} as generator, in the sense that there exists a {\em minimal} interval $[a_0,b_0]\subset [a,b]$ whose endpoints satisfy the conditions $\gamma(a_0)=\gamma(b_0)$ and $\dot\gamma(a_0)=\dot\gamma(b_0)$. In Section~\ref{sec:actions} a precise definition of prime curve will be given, in terms of the reparameterization action of $S^1$ on the space of periodic curves, see Definition~\ref{def:prime}. This {\em prime} geodesic is therefore not given as $n$--fold iteration of any other periodic geodesic. The number $\omega=|b_0-a_0|$ is called the {\em period} of $\gamma$. In case $[a_0,b_0]$ does not coincide with $[a,b]$, $\gamma$ is either a portion of a periodic geodesic (if the endpoints $\gamma(a)$ and $\gamma(b)$ do not coincide), or an {\em iterate geodesic}\index{Geodesic!periodic!iterate} (if the endpoints coincide), see Definition~\ref{def:prime}.

An important property of two periodic geodesics is if they are {\em geometrically distinct}\index{Geodesic!periodic!geometrically distinct} or not. Two periodic geodesics $\gamma_1$ and $\gamma_2$ are geometrically distinct if their images do not coincide. A prime geodesic and any of its iterates are never geometrically distinct. More precisely, any two periodic geodesics given as iterates of the same prime geodesic are not geometrically distinct, since they obviously have the same image.
\end{example}

\begin{definition}\label{def:parallel}
A vector field $X\in\sect^k(TM)$ is said to be {\em $g$--parallel along} $\gamma$ if $\D^g X=0$. In addition, a vector field is called {\em $g$--parallel} if it is $g$--parallel\index{Vector field!parallel} along every curve.
\end{definition}

\begin{remark}
A $g$--geodesic $\gamma$ can be hence characterized as a curve whose tangent field $\dot\gamma$ is $g$--parallel along $\gamma$.
\end{remark}

Another construction that involves covariant differentiation along curves is parallel translation.

\begin{proposition}\label{prop:paralleltransport}
Let $\gamma:[a,b]\rightarrow M$ be a $C^k$ curve, $t_0\in [a,b]$ and $v_0\in T_{\gamma(t_0)}M$. There exists a unique $g$--parallel vector field $X\in\sect^k(\gamma^*TM)$ such that $X(t_0)=v_0$. This vector field is called the {\em $g$--parallel translate}\index{Parallel translation} of $v_0$ along $\gamma$.
\end{proposition}

This is another basic result, whose proof uses elementary ODE techniques and can be found, for instance in \cite{jost,petersen}. It is also easy to verify that $g$--parallel translation is an isometry of $(M,g)$, in the sense of Definition~\ref{def:isometries}. Having existence and uniqueness of geodesics with prescribed initial data, an important question is how do geodesics change under perturbations of initial data. This change is characterized by the {\em semi--Riemannian exponential map}, which will be defined using the {\em geodesic flow} of a metric.

\begin{definition}\label{def:geodflow}
The {\it geodesic flow}\index{Geodesic flow}\index{Geodesic!flow} of $g\in\met_\nu^k(M)$ is the flow $$\Phi^g:U\subset\R\times TM\la TM,$$ defined in an open subset $U$ of $\R\times TM$ that contains $\{0\}\times TM$, of the {\em unique vector field}\footnote{This vector field on $TM$ is called the {$g$--geodesic vector field}.} on the tangent bundle whose integral curves are of the form $t\mapsto(\gamma(t),\dot\gamma(t))$, where $\gamma$ is a $g$--geodesic, satisfying
\begin{itemize}
\item[(i)] $\gamma(t)=\pi\circ\Phi^g(t,(x,v))$ is the unique geodesic with initial conditions $\gamma(0)=x$ and $\dot\gamma(0)=v$ (see Proposition~\ref{prop:geodexist});
\item[(ii)] $\Phi^g(t,(x,cv))=\Phi^g(ct,(x,v))$, for all $c\in\R$ such that this equation makes sense.
\end{itemize}
\end{definition}

\begin{figure}[htf]
\begin{center}
\includegraphics[scale=0.8]{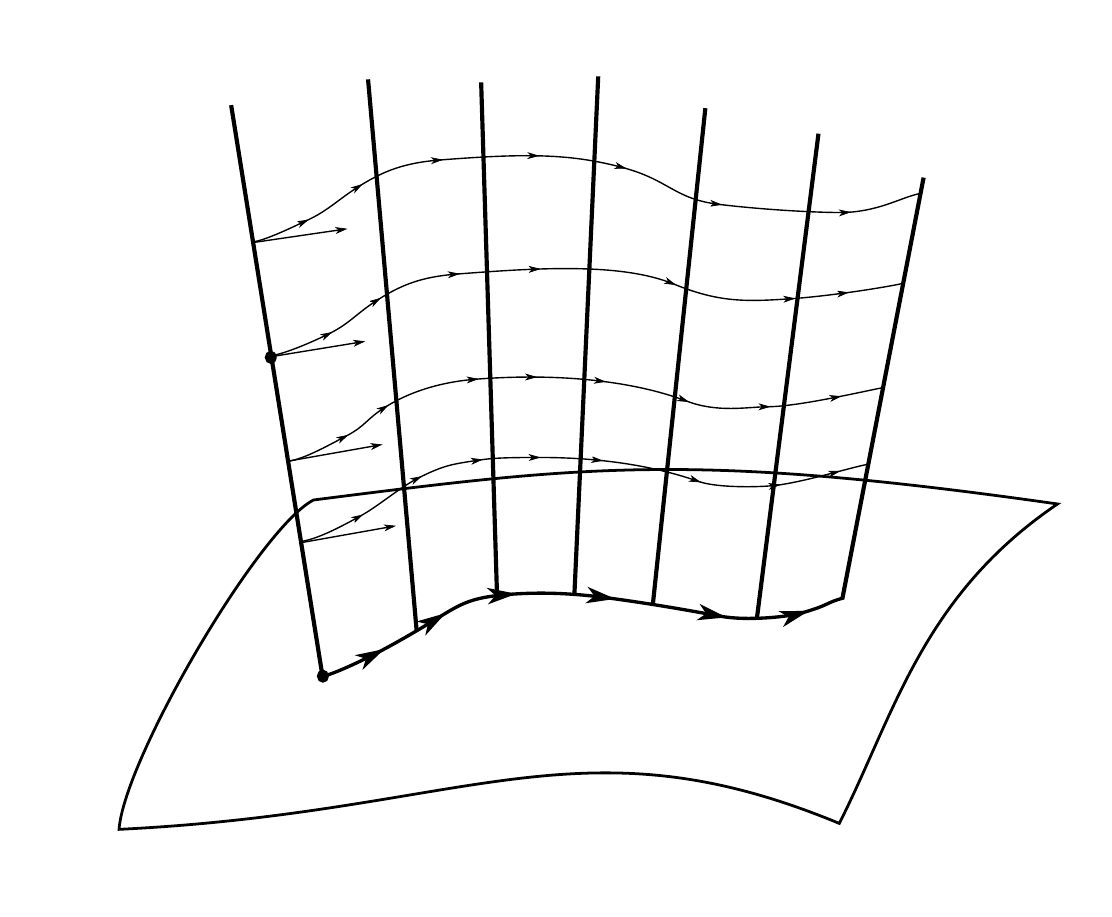}
\begin{pgfpicture}
\pgfputat{\pgfxy(-6.5,1.5)}{\pgfbox[center,center]{$\gamma(0)=x$}}
\pgfputat{\pgfxy(-7.7,4.5)}{\pgfbox[center,center]{$(x,v)$}}
\pgfputat{\pgfxy(-1.4,1.2)}{\pgfbox[center,center]{$M$}}
\pgfputat{\pgfxy(-4.5,2.2)}{\pgfbox[center,center]{$\gamma$}}
\pgfputat{\pgfxy(-0.6,5.1)}{\pgfbox[center,center]{$\Phi^g(\varepsilon,(x,v))$}}
\pgfputat{\pgfxy(-4.5,7.3)}{\pgfbox[center,center]{$TM$}}
\end{pgfpicture}
\end{center}
\caption{Representation of the geodesic flow $\Phi^g$ on $TM$, with a geodesic $\gamma(t)=\pi\circ\Phi^g(t,(x,v))$, for $t\in[0,\varepsilon]$, that satisfies $\gamma(0)=x$ and $\dot\gamma(0)=v$.}
\end{figure}

The flow $\Phi^g$ is well--defined as a consequence of Proposition~\ref{prop:geodexist}. Moreover, supposing that the geodesic vector field exists, it is easy to obtain conditions in local coordinates that this field must satisfy (corresponding to the geodesic equation mentioned in Definition~\ref{def:geodesic}). Defining the vector field as its solutions, elementary ODE results guarantee existence and $C^k$ regularity of $\Phi^g$. In addition, its domain $U$ is obviously related with the maximal intervals for solutions of the $g$--geodesic equation.

\begin{remark}
For instance, in the case of compact Riemannian manifolds it is easy to prove that the domain $U$ of $\Phi^g$ may be taken as the whole $\R\times TM$, since compact Riemannian manifolds are, in particular, {\em geodesically complete}. Since we are dealing with semi--Riemannian manifolds, completeness notions are not well--posed, and there are simple examples of compact semi--Riemannian manifolds whose geodesic flow has domain $U\neq\R\times TM$.
\end{remark}

In Chapter~\ref{chap4}, we will discuss the abstract meaning of generic properties aiming to study generic properties of the geodesic flow $\Phi^g$, in Chapters~\ref{chap5} and~\ref{chap6}. By a generic property of $\Phi^g$, we basically mean a property that is common to most metrics $g$ on $M$. In particular, such properties are stable under small perturbations, i.e., given a certain metric $g_0$, it is always possible to perturb it (in the adequate topology) and obtain a new metric $g_0+\varepsilon g$, such that the geodesic flow $\Phi^{g_0+\varepsilon g}$ satisfies this property. In this sense, genericity of certain properties suggests the typical dynamic behavior of the geodesic flow on certain manifold.

\begin{proposition}
Let $g\in\met_\nu^k(M)$. For each $x\in M$, there exists an open neighborhood $U$ of the origin of $T_xM$ such that it is possible to define the {\em $g$--exponential map}\index{Exponential map} by
\begin{eqnarray*}
\exp_x:U\subset T_xM&\la &M \\
v &\longmapsto &\Phi^g(1,(x,v)).
\end{eqnarray*}
\end{proposition}

It is simple to verify that $\exp_x$ is smooth. The exponential map can be clearly used as a local chart, and through this observation it is possible to define special neighborhoods with particular regularities, as follows.

\begin{definition}\label{def:normalradius}
From the Inverse Function Theorem, it follows that for each $x\in M$, there exist a neighborhood $V$ of the origin in $T_xM$ and a neighborhood $U$ of $x$, such that $\exp_x|_V:V\rightarrow U$ is a diffeomorphism. Such neighborhood $U$ is called a {\em $g$--normal neighborhood}\index{Normal neighborhood} of $p$. A $g$--normal neighborhood of $p$ is called {\em $g$--convex}\index{Convex neighborhood}\index{Normal neighborhood!convex} if it is a $g$--normal neighborhood of all of its points.\footnote{Convex neighborhoods of a given point exist for every semi--Riemannian metric $g$, and their size depends continuously on $g$ relatively to the $C^k$--topology, see O'Neill \cite{oneill}.}

For the auxiliary Riemannian metric $g_\mathrm R$, a positive number $r>0$ is called a {\em normal radius}\index{Normal radius} of a point $x\in M$ if $\exp_x(B(0,r))$ is a normal neighborhood of $x\in M$, where $B(0,r)$ denotes the open ball of radius $r$ around the origin of $T_xM$ with respect to the norm induced by $g_\mathrm R$. Finally, $r>0$ is called a {\em totally normal}\index{Normal radius!totally} radius for $x\in M$ if $r$ is a normal radius for $x$ and for all the points in the open set $\exp_x(B(0,r))$.
\end{definition}

We now discuss another essential concept in semi--Riemannian geometry, {\em curvature}\index{Curvature}. For this, consider again a fixed semi--Riemannian metric $g\in\met_\nu^k(M)$.

\begin{definition}
The {\em curvature tensor}\index{Curvature!of a metric} of $g$ is the $(1,3)$--tensor $R^g$ defined as the curvature tensor $R^{\nabla^g}$ of the Levi--Civita connection $\nabla^g$ of $g$, in the sense of Definition~\ref{def:symflat}. More precisely,
\begin{eqnarray}\label{eq:Rg}
R^g(X,Y)Z &=& \nabla^g_X\nabla^g_Y Z-\nabla^g_Y\nabla^g_X Z-\nabla^g_{[X,Y]}Z \\
&=&[\nabla^g_X,\nabla^g_Y]Z-\nabla^g_{[X,Y]}Z,\nonumber
\end{eqnarray}
\end{definition}

\begin{remark}
There is no convention in the literature for the sign of $R^g$. We choose to use the sign convention \eqref{eq:R}, the same adopted in \cite{jost,lee,petersen}. Other texts however may define the curvature tensor as $-R^g$. Notice that changing this choice of sign automatically implies changing other formulas such as the Jacobi equation \eqref{eq:jacobi}.
\end{remark}

\begin{definition}\label{def:flatmetric}
A metric $g\in\met_\nu^k(M)$ whose curvature tensor $R^g$ vanishes identically is called a {\em flat metric}.\index{Metric!flat}
\end{definition}

\begin{remark}
Usually, the curvature tensor $R^g$ is used together with $g$ in the form of the $(0,4)$--tensor $g(R^g(X,Y)Z,W)$. One can easily verify many different symmetries of this tensor, for instance,
\begin{itemize}
\item[(i)] $g(R^g(X,Y)Z,W)=-g(R^g(Y,X)Z,W)=g(R^g(Y,X)W,Z);$
\item[(ii)] $g(R^g(X,Y)Z,W)=g(R^g(Z,W)X,Y).$
\end{itemize}
\end{remark}

There are several possible interpretations of curvature. A first naive approach, immediate from Definition~\ref{def:symflat}, is that it measures second covariant derivatives' failure to commute. To present less trivial interpretations, we now introduce the concept of {\em Jacobi field}.

\begin{definition}\label{def:jacobifield}
Let $\gamma:[a,b]\to M$ be a $g$--geodesic. A {\em Jacobi field}\index{Jacobi field} along $\gamma$ with respect to $g$ is a vector field $J\in\sect^2(\gamma^*TM)$ that satisfies the {\em $g$--Jacobi equation}\index{Jacobi equation} along $\gamma$, given by
\begin{equation}\label{eq:jacobi}
(\D^g)^2J(t)=R^g(\dot\gamma(t),J(t))\dot\gamma(t), \quad t\in [a,b].
\end{equation}
\end{definition}

\begin{corollary}\label{cor:jacobick}
Let $g\in\met^k_\nu(M)$. If $\gamma:[a,b]\to M$ is a $g$--geodesic and $J$ is a $g$--Jacobi field along $\gamma$, then $J$ is of class $C^k$.
\end{corollary}

\begin{proof}
From Corollary~\ref{cor:geodck}, since $\gamma$ is a $g$--geodesic it is of class $C^{k+1}$. Since $J$ is a $g$--Jacobi field along $\gamma$, it satisfies the $g$--Jacobi equation \eqref{eq:jacobi}. Notice that the curvature tensor $R^g$ is of class $C^{k-2}$, since it involves second derivatives of $g$ which is $C^k$, see \eqref{eq:Rg}. From Definition~\ref{def:jacobifield}, it follows that $J$ is of class $C^2$. Inductively, suppose that $J$ is of class $C^j$, for some $2\leq j\leq k-2$. Then the map $$t\longmapsto R^g(\dot\gamma(t),J(t))\dot\gamma(t)$$ is a composite of $C^{j}$ maps, hence of class $C^{j}$. Thus, from \eqref{eq:jacobi} it follows that $(\D^g)^2 J$ is of class $C^{j}$, hence $J$ is of class $C^{j+2}$. This argument works for $2\leq j\leq k-2$. Therefore, applying it for $j=k-2$, we may conclude that $J$ is of class $C^{k}$.
\end{proof}

\begin{remark}
If $\gamma$ is a $g$--geodesic, its tangent field $\dot\gamma$ satisfies the $g$--Jacobi equation \eqref{eq:jacobi}, since $\D^g\dot\gamma$ and $R^g(\dot\gamma,\dot\gamma)\dot\gamma$ vanish identically. Notice that this trivial example verifies the assertion of Corollary~\ref{cor:jacobick}, since from Corollary~\ref{cor:geodck}, $\gamma$ is $C^{k+1}$ hence $\dot\gamma$ is $C^k$. The solutions $J=\dot\gamma$ and $J=0$ are called the {\em trivial} solutions of the $g$--Jacobi equation.
\end{remark}

\begin{corollary}
If $g\in\met_\nu^\infty(M)$, then $g$--Jacobi fields along $g$--geodesics are smooth.
\end{corollary}

\begin{proof}
The result follows directly from Corollaries~\ref{cor:geodck} and~\ref{cor:jacobick}, since $g$ and $J$ are of class $C^k$ for all $k\in\N$.
\end{proof}

\begin{remark}
The Jacobi equation \eqref{eq:jacobi} is obtained as a linearization of the geodesic equation \eqref{eq:geodequation}. Thus, Jacobi fields describe how quickly two geodesics with the same starting point {\em move away} one from each other.
\end{remark}

In this sense, the curvature tensor of a metric also contains information on the behavior of the geodesic flow, see Definition~\ref{def:geodflow}. Another possible interpretation is that curvature describes how parallel transport along a loop differs from the identity, see Example~\ref{ex:periodicgeod} and Proposition~\ref{prop:paralleltransport}. Finally, $R^g$ also measures non integrability of a special kind of distribution defined in the frame bundle. These fundamental interpretations of Riemannian curvature are explained for instance in \cite{bishop,jost}.

\begin{definition}\label{def:conjugate}
Two points $p,q\in M$ are said to be {\em $g$--conjugate}\index{Conjugate points} if there exists a $g$--geodesic $\gamma:[a,b]\to M$ with $\gamma(a)=p$ and $\gamma(b)=q$ and a $g$--Jacobi field $J:[a,b]\to M$ along $\gamma$ such that $J(a)=0$ and $J(b)=0$.
\end{definition}

Since the Jacobi equation \eqref{eq:jacobi} is a linearization of the geodesic equation \eqref{eq:geodequation}, two points are $g$--conjugate if there exists a $g$--geodesic $\gamma$ joining them and a {\em variation} of $\gamma$ by $g$--geodesics, whose variational field vanishes at the endpoints of $\gamma$, i.e., a map
\begin{equation}\label{eq:variationgammast}
(-\varepsilon,\varepsilon)\times [a,b]\ni (s,t)\longmapsto\gamma_s(t)\in M
\end{equation}
with $\gamma_s=\gamma$ for $s=0$ and $J(t)=\frac{\partial}{\partial s}\gamma_s(t)\big|_{s=0}$ satisfying $J(a)=0$ and $J(b)=0$. This variational field $J$ is the $g$--Jacobi field of Definition~\ref{def:conjugate}. In fact, it is easy to prove that every $g$--Jacobi field along $\gamma$ arises as the variational field of a variation of $\gamma$ by other $g$--geodesics, see \cite{jost,petersen}.

\begin{remark}
Notice that the above observation that conjugacy of $p$ and $q$ is equivalent to existence of a variation of $\gamma$ by other geodesics whose variational field vanishes at endpoints {\em does not imply} that if $p$ and $q$ are conjugate, then there exists more than one geodesic joining them. In fact, vanishing of the variational field at the endpoints only implies that the endpoints $\gamma_s(b)$ of the geodesics in the variation \eqref{eq:variationgammast} are $q$ {\em up to first order} in the parameter $s$. Indeed, it is not difficult to find examples of two conjugate points joined by only one $g$--geodesic, as illustrated in Figure \ref{fig:conjugatepoints1}.
\end{remark}

\begin{figure}[htf]
\begin{center}
\vspace{-0.5cm}
\includegraphics[scale=1]{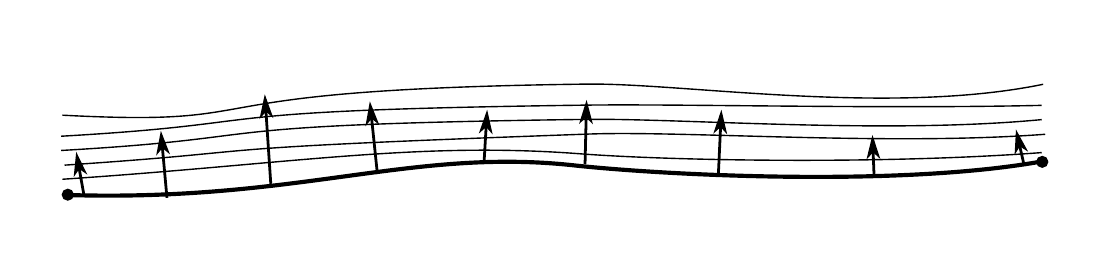}
\begin{pgfpicture}
\pgfputat{\pgfxy(-6,0.8)}{\pgfbox[center,center]{$\gamma$}}
\pgfputat{\pgfxy(-5.2,1.8)}{\pgfbox[center,center]{$J$}}
\pgfputat{\pgfxy(-2.5,2.1)}{\pgfbox[center,center]{$\gamma_s$}}
\pgfputat{\pgfxy(-11.1,0.8)}{\pgfbox[center,center]{$p$}}
\pgfputat{\pgfxy(-0.5,1)}{\pgfbox[center,center]{$q$}}
\end{pgfpicture}
\end{center}
\vspace{-0.3cm}
\caption{Conjugate points $p$ and $q$ joined by a unique geodesic $\gamma$, with variation by geodesics $\gamma_s$ and induced variational field $J$.}\label{fig:conjugatepoints1}
\end{figure}

\begin{example}\label{ex:sphereconjpoints}
Consider $\R^m$ endowed with the Euclidean metric and the embedded round $(m-1)$--sphere $S^{m-1}$. Then, any point $p\in S^{m-1}$ is conjugate to itself and to $-p$. In this case, these points are joined by infinitely many geodesics. A Jacobi field $J$ that vanishes at $p$ and $-p$ may be easily obtained as the variational field of a variation of any geodesic joining $p$ and $-p$ by other geodesics that join $p$ and $-p$, as shown in Figure \ref{fig:conjugatepoints2}.
\end{example}

\begin{figure}[htf]
\begin{center}
\includegraphics[scale=1]{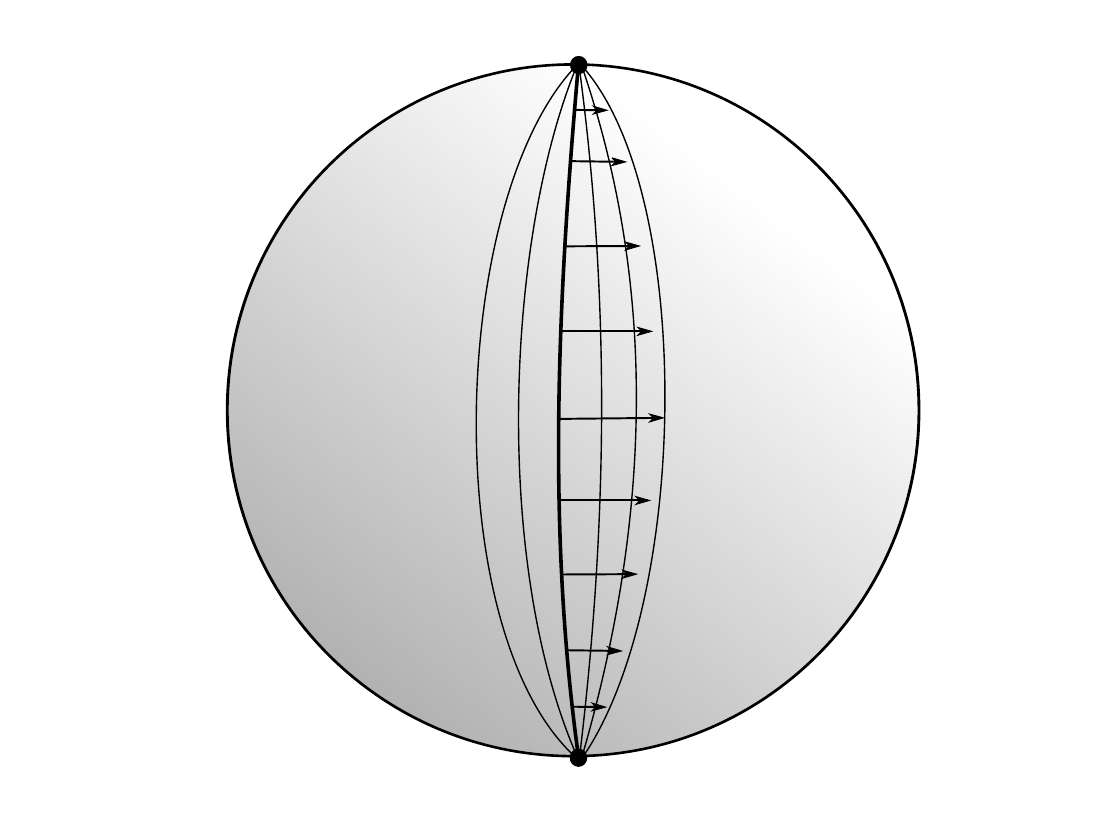}
\begin{pgfpicture}
\pgfputat{\pgfxy(-5.5,8.2)}{\pgfbox[center,center]{$p$}}
\pgfputat{\pgfxy(-5.5,0.3)}{\pgfbox[center,center]{$-p$}}
\pgfputat{\pgfxy(-4.3,4.4)}{\pgfbox[center,center]{$J$}}
\pgfputat{\pgfxy(-2.8,1.5)}{\pgfbox[center,center]{$S^{m-1}$}}
\end{pgfpicture}
\end{center}
\vspace{-0.3cm}
\caption{Antipodal points $p$ and $-p$ on $S^{m-1}$ with the round metric are conjugate along any geodesic joining them.}\label{fig:conjugatepoints2}
\end{figure}

\begin{proposition}\label{prop:conjugatecritexp}
Given $p\in M$, the set of points on $M$ that are $g$--conjugate to $p$ coincides with the critical values of the $g$--exponential map $\exp_p:T_pM\to M$.
\end{proposition}

\begin{proof}
Let $v\in T_pM$ be a critical point of $\exp_p:T_pM\to M$, denote $q=\exp_p v$ and consider the derivative
\begin{equation}\label{eq:dexppv}
\dd\exp_p(v):T_vT_pM\la T_qM.
\end{equation}
Since $T_pM$ is a vector space, let us identify $T_vT_pM\cong T_pM$. From the fact that $v$ is a critical point, it follows that the above map is not surjective. Since \eqref{eq:dexppv} is a linear map between finite--dimensional vector spaces, its nonsurjectivity implies that it has nontrivial kernel. Thus, let $w\in T_pM$ be a nonzero vector in the kernel of \eqref{eq:dexppv}, and consider the short segment $v(t)$ given by $$v:(-\varepsilon,\varepsilon)\ni s\longmapsto v+sw\in T_pM.$$ Notice that $v(0)=v$, and $v'(0)=w$. Consider the $g$--geodesic $\gamma:[0,1]\to M$ given by $\gamma(t)=\exp_p tv$ and the variation $$\gamma_s(t)=\exp_p tv(s).$$ Then the variational field $J=\frac{\partial}{\partial s}\gamma_s(t)\big|_{s=0}$ is a nontrivial $g$--Jacobi field along $\gamma$ that vanishes at the endpoints. Namely, $J(0)$ is clearly null and $J(1)$ coincides with the image by $\exp_p$ of $v'(0)=w$, that is in the kernel of this map. Therefore, $q=\exp_p v$ is $g$--conjugate to $p$ if $v$ is a critical point of $\exp_p$.

Conversely, let $J$ be a $g$--Jacobi field along $\gamma(t)=\exp_p tv$, where $q=\exp_p v$, that vanishes at the endpoints of $\gamma$. Then $J$ is the variational field of a certain variation $\gamma_s(t)$ of $\gamma$ by $g$--geodesics, i.e., there exists a curve $$(-\varepsilon,\varepsilon)\ni s\mapsto v(s)\in T_pM$$ such that $v(0)=v$ and $\gamma_s(t)=\exp_p tv(s)$. Thus $$J(t)=\dd\exp_p(tv)tv'(0),$$ and setting $t=1$, since $J(1)=0$, it follows that $v$ is a critical point of $\exp_p$.
\end{proof}

\begin{figure}[htf]
\begin{center}
\includegraphics[scale=1]{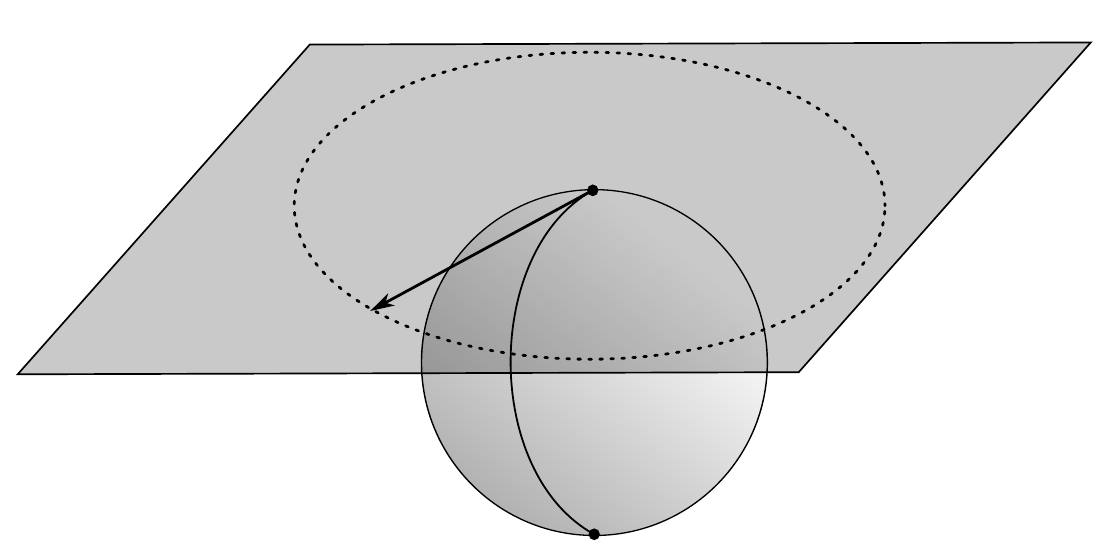}
\begin{pgfpicture}
\pgfputat{\pgfxy(-5.5,4)}{\pgfbox[center,center]{$p$}}
\pgfputat{\pgfxy(-5.5,-0.2)}{\pgfbox[center,center]{$-p$}}
\pgfputat{\pgfxy(-0.7,3.5)}{\pgfbox[center,center]{$T_pS^{m-1}$}}
\pgfputat{\pgfxy(-3.5,0.5)}{\pgfbox[center,center]{$S^{m-1}$}}
\end{pgfpicture}
\end{center}
\caption{In the case of $S^{m-1}$, all points in the $(m-2)$--sphere of radius $\pi$ around the origin of $T_pS^{m-1}$ are critical points of $\exp_p:T_pS^{m-1}\to S^{m-1}$. All these critical points are mapped to the antipodal conjugate point $-p$, see Example \ref{ex:sphereconjpoints}.}\label{fig:conjugatepoints3}
\end{figure}

\begin{remark}
This important notion of conjugacy between two points will be generalized in the sequel by the notion of {\em focality} between a point and a submanifold, or between two submanifolds, see Definitions~\ref{def:focalpoint} and~\ref{def:focalsubmanifolds}.
\end{remark}

\begin{definition}
The {\em Ricci curvature}\index{Curvature!Ricci} of $g$ is a $(0,2)$--tensor field given by the trace of the curvature endomorphism on its first and last indexes. More precisely, if $\{\xi_i(x)\}_{i=1}^m$ is a $g$--orthonormal frame, let $\delta_i=g(\xi_i,\xi_i)=\pm1$,
\begin{eqnarray*}
\Ric^g(X,Y) &=& \tr_g R^g(\,\cdot,X)Y \\
&=& \sum_{i=1}^m \delta_i g(R^g(\xi_i,X)Y,\xi_i).
\end{eqnarray*}
\end{definition}

\begin{remark}
In the Riemannian case, Ricci curvature should be thought as an approximation of the Laplacian of the metric, i.e., a measure of the volume distortion on $M$. For a more precise formulation of this interpretation, see Morgan and Tian \cite{MorganTian}.
\end{remark}

\begin{definition}
The {\em scalar curvature}\index{Curvature!scalar} of $g$ is a function $S^g$ given by the trace of the Ricci curvature. More precisely, if $\{\xi_i(x)\}_{i=1}^m$ is a $g$--orthonormal frame, let $\delta_i=g(\xi_i,\xi_i)=\pm1$,
\begin{eqnarray*}
S^g &=& \tr_g \Ric^g \\
&=& \sum_{i=1}^m \delta_i\Ric(\xi_i,\xi_i).
\end{eqnarray*}
\end{definition}

\begin{definition}
A metric $g$ is called an {\em Einstein metric}\index{Metric!Einstein}\index{Einstein manifold} if it satisfies the {\em Einstein equation}\index{Einstein equation}
\begin{equation}\label{eq:einstein}
\Ric^g-\tfrac12 S^gg+\Lambda g=8\pi T,
\end{equation}
where $\Lambda\in\R$ is the cosmological constant and $T$ is the\index{Energy--momentum tensor} {\em energy--momentum tensor}\footnote{This is a $(2,0)$--tensor on $M$ that contains all the information on the physical distribution of matter and energy in the space--time $M$. For instance, when dealing with a space--time without matter, i.e., a vacuum, this tensor vanishes identically. For a physical interpretation of $T$, see \cite{besse,hawking,gravitation}.} of $M$. Although \eqref{eq:einstein} may be considered for any manifold, it is usually studied on four--dimensional space--times,\footnote{Recall Definition~\ref{def:spacetime}.} i.e., for $\nu=1$ and $m=4$, due to the physical relevance of this particular case in general relativity.
\end{definition}

Einstein metrics appear in general relativity as perfect models for gravitational problems. It relates the simplest $(2,0)$--tensors on a space--time $M$. The constant $8\pi$ in the right--hand side of \eqref{eq:einstein} is responsible for an adequate scaling that allows to consider classic gravitation as a limit case of relativistic gravitation.

Let $g$ be an Einstein metric such that $(M,g)$ is a space--time. Points $x\in M$ are called {\em events}, and $g$--geodesics $\gamma:[a,b]\to M$ are either timelike, lightlike or spacelike, according to the causal character of its tangent field $\dot\gamma$, see Definition~\ref{def:causalchar}. General relativity asserts that a timelike geodesic corresponds to the path of an observer moving at less than the speed of light, only under influence of gravitational forces. Similarly, lightlike geodesics correspond to trajectories of a flash of light, and spacelike geodesics are the geometric equivalent to a trajectory with speed higher than the speed of light. The theory also states that particles with mass cannot move faster than light, hence spacelike geodesics are not admissible paths for the motion of an object.

\begin{remark}
Expanding the Einstein equation \eqref{eq:einstein} in local coordinates, we obtain a system of second--order PDEs. Einstein himself was not able to find examples of space--times $(M,g)$ that are exact solutions of \eqref{eq:einstein}, but only approximate linearized solutions. The first exact solution found was the {\em Schwartzschild metric}, which in coordinates $(t,r,\theta,\phi)$ is given by
\begin{equation}\label{eq:schwartzschild}
\dd s^2=-\left(1-\frac{2m}{r}\right)\dd t^2+\frac{\dd r^2}{1-\frac{2m}{r}}+r^2\dd\theta^2+r^2\sin^2\theta\dd\phi^2,
\end{equation}
assuming that the energy--momentum tensor $T$ and the cosmological constant $\Lambda$ vanish. This model describes the gravitational field outside a spherical non--rotating body of mass $m$ such as a (non--rotating) star, planet, or black hole.\footnote{The Schwartzschild black hole is characterized by a surrounding spherical surface, called the event horizon, which is situated at the Schwartzschild radius, often called the radius of a black hole. Any non--rotating and non--charged mass that is smaller than its Schwartzschild radius forms a black hole. The solution of the Einstein equations \eqref{eq:einstein} is valid for any mass $m$, so in principle, according to general relativity, a Schwartzschild black hole of any mass could exist if conditions became sufficiently favorable to allow for its formation.} It is also a good approximation to the gravitational field of a slowly rotating body like the Earth or Sun. Later, other solutions as {\em Robertson--Walker metrics} and {\em Kerr metrics} where obtained. This last models the gravitational field outside a rotating black hole, see \cite{bee,hawking,gravitation}.
\end{remark}

\begin{definition}\label{def:minkowski}
The {\em Minkowski space--time}\index{Space--time!Minkowski} is the Lorentzian manifold $(\R^4,\dd s_{M}^2)$, with the so--called {\em Minkowski metric},\index{Metric!Minkowski} that may be written in coordinates $(t,x,y,z)$ as
\begin{equation}\label{eq:minkowski}
\dd s_{M}^2=-\dd t^2+\dd x^2+\dd y^2+\dd z^2.
\end{equation}
\end{definition}

\begin{remark}
The Minkowski metric is flat, and is a trivial solution of the Einstein equation \eqref{eq:einstein} with vanishing cosmological constant and energy--momentum tensor. It hence models the gravitational field of a perfect vacuum, i.e., an {\em empty}\footnote{i.e., without matter.} space--time.
\end{remark}

\begin{definition}\label{def:asymptflat}
A semi--Riemannian metric $g\in\met_\nu^k(M)$ is said to be {\em asymptotically flat}\index{Metric!asymptotically flat} if there exists $h\in\sect^k_0(TM^*\vee TM^*)$ such that $g-h$ is a flat metric, see Definitions~\ref{def:tendstozero} and~\ref{def:flatmetric}.
\end{definition}

\begin{remark}\label{re:physics0}
In general relativity, it is common to consider Lorentzian metrics on $\R^4$ that are asymptotically flat, i.e., tend to the Minkowski metric \eqref{eq:minkowski} at infinity. The physical meaning of this asymptotically flatness can be described as follows. Since by the Einstein equation \eqref{eq:einstein}, curvature of space--time (that corresponds to gravitation) is a consequence of the presence of matter, the gravitational field of an asymptotically flat space--time, as well as any matter or other fields which may be present, become negligible in magnitude at large distances from some region. Recall that flatness of a space--time corresponds to absence of matter, hence the Minkowski space--time models perfect vacuum. In this sense, it is reasonable to consider space--times all of whose {\em non negligible} matter is present in some region, since this allows to model {\em isolated systems}, i.e., systems whose exterior influences can be neglected.

As an illustrative example, consider the problem of modeling the gravitational field around a single star. Instead of imagining a universe containing a single star and nothing else, it seems to be more physically meaningful to model the interior of the star together with an exterior region in which gravitational effects due to the presence of other objects, such as nearby stars, can be neglected. Since typical distances between astrophysical bodies tend to be much larger than the diameter of each body, this idealization usually helps to greatly simplify the construction and analysis of such models. For instance, the Schwartzschild metric \eqref{eq:schwartzschild} deals with such an idealized model of the gravitational field outside a spherical non--rotating body.

For more detailed interpretation of asymptotically flat space--times, see Hawking \cite{hawking}. Furthermore, a few stability results for the Minkowski space--time were studied by Christodoulou \cite{chris1,chris2}.
\end{remark}

We now approach a delicate matter concerning length of curves and distance maps in semi--Riemannian geometry. Using the auxiliary Riemannian metric $g_\mathrm R$, we may define the length of a curve as usual.
 
\begin{definition}\label{def:riemlenght}
The {\em $g_\mathrm R$--length}\index{Length}\index{Riemannian!length} of a curve $\gamma:[a,b]\rightarrow M$ is $$L_\mathrm R(\gamma)=\int_a^b \sqrt{g_\mathrm R(\dot\gamma(t),\dot\gamma(t))}\;\dd t.$$
\end{definition}
 
Notice however that replacing $g_\mathrm R$ with a semi--Riemannian metric $g\in\met_\nu^k(M)$ with index $\nu\neq0$, the integrand above is not well--defined. In particular, using this same length definition would imply that non constant lightlike curves would always have always null length. For this reason, in semi--Riemannian geometry it is more usual to deal with the {\em energy} of a curve, rather than its length.

\begin{definition}\label{def:energy}
The {\em $g$--energy}\index{Energy} of a curve $\gamma:[a,b]\rightarrow M$ is $$E_g(\gamma)=\tfrac12\int_a^b g(\dot\gamma(t),\dot\gamma(t))\;\dd t.$$
Notice that $E_g(\gamma)$ might be negative, for instance if $\gamma$ is timelike.
\end{definition}

In Chapter~\ref{chap35}, we will study the relation between $g$--geodesics and curves that minimize $g$--energy, which is totally analogous to the Riemannian case.

\begin{definition}\label{def:distance}
The {\em $g_\mathrm R$--distance}\index{Riemannian!distance}\index{Distance} of two points $p,q\in M$ is given by the infimum $d_\mathrm R(p,q)$ of lengths of all piecewise regular curve segments joining $p$ and $q$. 
\end{definition}

\begin{remark}
The pair $(M,d_\mathrm R)$ is a metric space, and the topology induced by this distance coincides with the topology from the atlas of $M$.
\end{remark}

It is also possible to define semi--Riemannian distance functions, nevertheless we will not use this concept in our applications. In the case of the Riemannian distance $d_\mathrm R$, completeness of the metric space $(M,d_\mathrm R)$ is related to a {\em geodesic} notion of completeness by the celebrated Hopf--Rinow Theorem, see \cite{jost,lee,petersen}. Several related modern topics of research deal with similar relations and completeness notions in the semi--Riemannian case. We will not discuss this topic, which is beyond the objectives of this text.

We end this section recalling some basic definitions regarding {\em submanifolds} of a semi--Riemannian manifold $(M,g)$. Consider the inclusion $i:P\hookrightarrow M$ of a submanifold $P\subset M$. The restriction $i^*g$ may be a degenerate\footnote{This happens in case there exists $x\in P$ such that $i^*g(x)$ is a degenerate symmetric bilinear form on $T_xP$, see Definition~\ref{def:nondegenerate}.} tensor, in which case the submanifold $P$ is called {\em degenerate}.\index{Submanifold!degenerate} Furthermore, as we will see in the next section, there exists topological obstructions to the existence of metrics of given index, hence if a submanifold $P$ has such obstructions, then any restriction $i^*g$ will necessarily degenerate at some point. In order to develop our results that concern submanifolds, {\em nondegeneracy} will be a necessary hypothesis.

\begin{definition}\label{def:nondegmet}
Consider $P$ a submanifold of $M$, $g\in\met_\nu^k(M)$, and $i:P\hookrightarrow M$ its inclusion. Then $P$ is said to be {\em $g$--nondegenerate}\index{Submanifold!nondegenerate} if the restricted metric tensor $i^*g$ is nondegenerate. The set of such metrics on $M$ is denoted
\begin{equation}\label{eq:nondegmet}
\met_{\nu}^k(M,P)=\{g\in\met_{\nu}^k(M): P \mbox{ is }g\mbox{--nondegenerate}\}.
\end{equation}
The submanifold $P$ is said to be {\em $g$--degenerate}\index{Submanifold!degenerate} for every $g\in\met_\nu^k(M)\setminus\met_\nu^k(M,P)$.
\end{definition}

\begin{remark}\label{re:metmpmightbeempty}
For $\nu=0$, trivially $\met_0^k(M,P)=\met_0^k(M)$ for any submanifold $P$. Nevertheless, if $0<\nu <m$, the subset $\met_{\nu}^k(M,P)$ might be empty, since there are topological obstructions to the existence of semi--Riemannian metrics of fixed index on $P$, which will be studied in the next section using characteristic classes, in particular the Euler class.
\end{remark}

\begin{proposition}
A submanifold $P$ of $M$ is $g$--degenerate if and only if there exists $p\in P$ such that\footnote{Recall that $T_p^0M$ is the $g$--light cone of $M$ at $p$, see Definition~\ref{def:unittangentbundle}.} $T_pP\cap T^0_pM\neq\{0\}$ and given any nonzero $v\in T_pP\cap T^0_pM$, the subspace $T_pP$ is contained in $T_v T^0_pM$.
\end{proposition}

\begin{proof}
It suffices to prove that the subspace $T_pP$ of $T_pM$ intersects the kernel\footnote{See Definition~\ref{def:nondegenerate}.} of the bilinear form $g(p)|_{T_pP\times T_pP}$ non trivially if and only if $T_pP\cap T_p^0M\neq\{0\}$ and given any nonzero $v\in T_pP\cap T^0_pM$, the subspace $T_pP$ is contained in $T_v T^0_pM$. This implies that $g$ degenerates at $p\in P$, and hence $P$ is $g$--degenerate.

Suppose $T_pP$ intersects the kernel of $g(p)|_{T_pP\times T_pP}$ non trivially. Then there exists a nonzero $v\in T_pP$ such that $g(p)(v,w)=0$ for all $w\in T_pP$. In particular, $g(p)(v,v)=0$, hence $T_pP$ intersects $T^0_pM$ non trivially. Consider a nonzero $v\in T_pP\cap T_p^0M$. Deriving $g(p)(v,v)=0$, it is easy to see that the tangent space to $T_p^0M$ is given by $$T_vT^0_pM=v^{\perp_g}=\{w\in T_pM:g(p)(v,w)=0\}.$$ Thus, if $v$ is in the kernel of $g(p)|_{T_pP\times T_pP}$, then $T_pP\subset T_vT^0_vM$. The converse is obvious.
\end{proof}

\begin{definition}
Let $P$ be a submanifold of $M$ and $g\in\met_\nu^k(M,P)$. The {\em $g$--normal bundle}\index{Normal bundle} $TP^\perp$ to $P$ is the smooth sub bundle of the tangent bundle $TM$ whose base is $P$ and whose fibers at each $p\in P$ are given by $T_pP^{\perp_g}$, i.e., the orthogonal complement of $T_pP$ in $T_pM$ with respect to $g(p)$. In the presence of more than one metric on the ambient, when not clear from the context, we will include a subindex $^{\perp_g}$ to denote with reference to which metric normal objects should be considered.
\end{definition}

\begin{definition}\label{def:sfform}
If $g\in\met_{\nu}^k(M,P)$, the \emph{second fundamental form}\index{Second fundamental form}\index{Submanifold!second fundamental form} of $P$ in the normal direction $\eta\in TP^\perp$ is the symmetric bilinear tensor $\s^P_\eta\in\sect^k(TP^*\vee TP^*)$, given by
\begin{equation}\label{eq:sff}
\s^P_\eta(v,w)=g(\nabla^{g}_v \overline{w},\eta),
\end{equation}
where $\overline{w}$ is an extension\footnote{It is simple to verify that indeed this definition does not depend on the chosen extension of $w$.} of $w$ tangent to $P$. Using the fact that $P$ is nondegenerate, we will also identify $\s^P_\eta$ at a point $p\in P$ with the $g$--symmetric linear operator
\begin{eqnarray*}
&\s^P_\eta(p):T_pP\longrightarrow T_pP& \\
&g(\s^P_\eta(p)v,w)=\s^P_\eta(v,w), \quad  v,w\in T_pP,&
\end{eqnarray*}
using \eqref{ident:bilin}. This operator $\s^P_\eta$ is called the \emph{shape operator}\index{Shape operator}\index{Submanifold!shape operator} of $P$.

If the second fundamental form $\s_\eta^P$ vanishes identically for any normal direction $\eta$, then $P$ is called a \emph{totally geodesic}\index{Totally geodesic submanifold}\index{Submanifold!totally geodesic} submanifold. This property is equivalent to each $g$--geodesic of $P$ being a $g$--geodesic of $M$.
\end{definition}

\begin{definition}\label{def:focalpoint}
A point $q\in M$ is said to be {\em $g$--focal}\index{Focal!point} to a submanifold $P$ if there exists a $g$--geodesic $\gamma:[a,b]\to M$ with $\gamma(a)\in P$, $\dot\gamma(a)\in T_{\gamma(a)}P^\perp$ and $\gamma(b)=q$, and a $g$--Jacobi field $J:[a,b]\to M$ satisfying $J(a)\in T_{\gamma(a)}P$, $J(b)=0$ and $$\D^g J(a)+\s^P_{\gamma(a)}(J(a))\in T_{\gamma(a)}P^\perp.$$
\end{definition}

\begin{remark}
The above definition clearly generalizes the notion of conjugacy between two points, see Definition~\ref{def:conjugate}. Analogously to Proposition~\ref{prop:conjugatecritexp}, it is easy to prove that a point $q$ is focal to a submanifold $P$ if and only if it is a critical value of the normal $g$--exponential map $\exp^\perp:TP^\perp\to M$, given by the restriction of the $g$--exponential map to the $g$--normal bundle to $P$.
\end{remark}

\begin{example}
Consider $\R^m$ endowed with the Euclidean metric and the embedded round $(m-1)$--sphere $S^{m-1}$. Then, it is easy to verify that the origin of $\R^m$ is focal to $S^{m-1}$, since any geodesic orthogonal to $S^{m-1}$ admits a Jacobi field satisfying the conditions of Definition \ref{def:focalpoint}, as shown in figure below.

\begin{figure}[htf]
\begin{center}
\vspace{-0.4cm}
\includegraphics[scale=1]{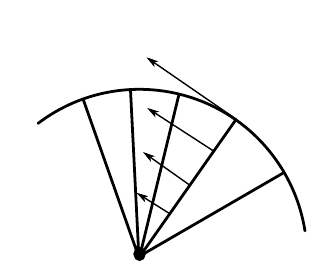}
\begin{pgfpicture}
\pgfputat{\pgfxy(0.1,1)}{\pgfbox[center,center]{$S^{m-1}$}}
\pgfputat{\pgfxy(-2.1,-0.2)}{\pgfbox[center,center]{$0$}}
\end{pgfpicture}
\end{center}
\end{figure}
\end{example}

\begin{definition}\label{def:focalsubmanifolds}
Two submanifolds $P$ and $Q$ of a semi--Riemannian manifold $(M,g)$ are said to be {\em $g$--focal}\index{Focal!submanifolds} if there exists a $g$--geodesic $\gamma:[a,b]\to M$ with $\gamma(a)\in P$, $\dot\gamma(a)\in T_{\gamma(a)}P^\perp$ and $\gamma(b)\in Q$, $\dot\gamma(b)\in T_{\gamma(b)}Q^\perp$ and a $g$--Jacobi field $J:[a,b]\to M$ along $\gamma$ satisfying $J(a)\in T_{\gamma(a)}P$, $J(b)\in T_{\gamma(b)}Q$ and
\begin{equation}\label{eq:focalsubmanifolds}
\begin{aligned}
\D^g J(a)+\s_{\dot\gamma(a)}^P(J(a)) &\in T_{\gamma(a)}P^\perp \\
\D^g J(b)+\s_{\dot\gamma(b)}^Q(J(b)) &\in T_{\gamma(b)}Q^\perp,
\end{aligned}
\end{equation}
where $^\perp$ is orthogonality with respect to the metrics on $P$ and $Q$ induced by $g$.
\end{definition}

Once more, the above definition clearly generalizes the previous notions of conjugacy between points and focality between a point and a submanifold. For a geometrical interpretation of focality between submanifolds we refer to Piccione and Tausk \cite{PicTauJMP}. Notice also that there are clear physical approaches to focality of two manifolds, for instance considering wavefronts.

\section{Topological obstructions to existence of metrics}
\label{sec:topobst}

Using partitions of the unity, it is not difficult to prove that every manifold can be endowed with a Riemannian metric\footnote{Recall Remark~\ref{re:fixedgr}.}, see for instance \cite{jost,lee,petersen}. Nevertheless, there are topological obstructions to the existence of {\em semi--Riemannian} metrics. A relevant topic in modern research is to determine practical necessary and sufficient topological conditions for the existence of semi--Riemannian metrics of a given index. An adequate approach for this type of problem consists of using obstruction theory and characteristic classes.

In this section, we prove a well--known condition of this type, see Proposition~\ref{prop:metricdistribution}. Nevertheless, this is not a computationally manageable condition for arbitrary indexes. For more specific indexes however, it is possible to improve such statement. Namely, we will explore the Lorentzian case $\nu=1$, in which the obstruction described in Proposition~\ref{prop:metricdistribution} is a well--known characteristic class. We will also relate it with celebrated topological invariants, for compact manifolds. Finally, we discuss some examples in low dimensions, particularly concerning existence of semi--Riemannian metrics on spheres, based in Steenrod \cite{steenrod}.

Notice that every result on obstructions to the existence of metrics may be applied to submanifolds of a given semi--Riemannian manifold. Hence, it may be regarded as a result on the obstruction to the nondegeneracy of submanifolds, see Definition~\ref{def:nondegmet} and Remark~\ref{re:metmpmightbeempty}.

\begin{proposition}\label{prop:metricdistribution}
A smooth manifold $M$ admits a $C^k$ semi--Riemannian metric $g\in\met_\nu^k(M)$ if and only if $M$ admits a $C^k$ distribution\footnote{See Example~\ref{ex:distribution}.} of rank $\nu$.
\end{proposition}

\begin{proof}
Assume $\mathcal D\subset TM$ is a $C^k$ distribution of rank $\nu$ on $M$ and consider $g_\mathrm R$ an auxiliary smooth Riemannian metric on $M$, see Remark~\ref{re:fixedgr}. Define a section $g\in\sect^k(TM^*\vee TM^*)$ by setting
\begin{equation}
g(v,w)=\begin{cases}g_\mathrm R(v,w),& v,w\in\mathcal D^\perp\\[.3cm]
0,&v\in\mathcal D,w\in\mathcal D^\perp\\[.3cm]
-g_\mathrm R(v,w),&v,w\in\mathcal D,
\end{cases}
\end{equation}
where $^\perp$ clearly denotes $g_\mathrm R$--orthogonality. It is then a simple verification that $g\in\met_\nu^k(M)$.

Conversely, assume that $g\in\met_\nu^k(M)$ and let $A\in\sect^k(TM^*\otimes TM)$ be the unique $g_{\mathrm R}$--symmetric $(1,1)$--tensor on $M$ that represents $g$ in terms of $g_\mathrm R$, i.e., such that $$g=g_{\mathrm R}(A\cdot,\cdot).$$ Notice that at each $x\in M$, $A(x)$ is a symmetric $m\times m$ real matrix of index $\nu$, hence diagonalizable. Denote by $\sigma(A(x))\subset\R$ the set of eigenvalues of $A(x)$, and by $\eig_{A(x)}(\lambda_j)$ the eigenspace of $A(x)$ correspondent to $\lambda_j\in\sigma(A(x))$. Define a distribution $\mathcal D$ by
\begin{equation}
\mathcal D_x=\bigoplus_{\substack{\lambda_j\in\sigma(A(x)) \\ \lambda_j<0}} \eig_{A(x)}(\lambda_j), \quad x\in M.
\end{equation}
This is clearly a distribution of rank $\nu$ on $M$. Notice that proving that $\mathcal D$ is $C^k$ is equivalent to proving that the map
\begin{equation}\label{eq:xpx}
M\ni x\longmapsto (x,p_x)\in\bigcup_{q\in M} \{q\}\times\Lin(T_qM,\mathcal D_q),
\end{equation}
where $p_x:T_xM\to\mathcal D_x$ is the $g_\mathrm R(x)$--orthogonal projection onto $\mathcal D_x$, is a $C^k$ section of this vector bundle. This map \eqref{eq:xpx} can be clearly decomposed as
\begin{equation*}
x\xmapsto{\;\;A\;\;} A(x)\xmapsto{\;\;\eta\;\;} p_x,
\end{equation*}
where $\eta$ is the map that to each symmetric matrix $H\in\GL(m,\R)$ of index $\nu$ associates the orthogonal projection $\eta(H)\in\Lin(\R^m,\R^\nu)$ onto the direct sum of its negative eigenspaces. A standard functional analytical argument gives
\begin{equation}
\eta(H)=\sum_{j=1}^\nu \frac1{2\pi i}\oint_{\gamma_j}\frac{\dd z}{z-H},
\end{equation}
where $\{\gamma_j\}_{j=1}^\nu$ are smooth curves in the complex plane $\C$ that make one turn around each negative eigenvalues $\lambda_j$ of $H$ counterclockwisely, see figure below.\footnote{Note that in case $H$ is diagonal, from the Cauchy formula, the $j^{\mbox{\tiny th}}$ line integral is equal to the diagonal matrix with $1$ in the $j^{\mbox{\tiny th}}$ position and $0$ in the others. Hence, the sum that results $\eta(H)$ is the diagonal matrix with $1$ in the coordinates that correspond to negative eigenvalues of $H$ and $0$ in the others. This is exactly the matrix of the projection onto the direct sum of all negative eigenspaces of $H$.} Since $\eta$ is clearly smooth and $A$ is $C^k$, it follows that \eqref{eq:xpx} is also $C^k$, concluding the proof.
\end{proof}

\begin{figure}[htf]
\vspace{-1cm}
\begin{center}
\includegraphics[scale=1]{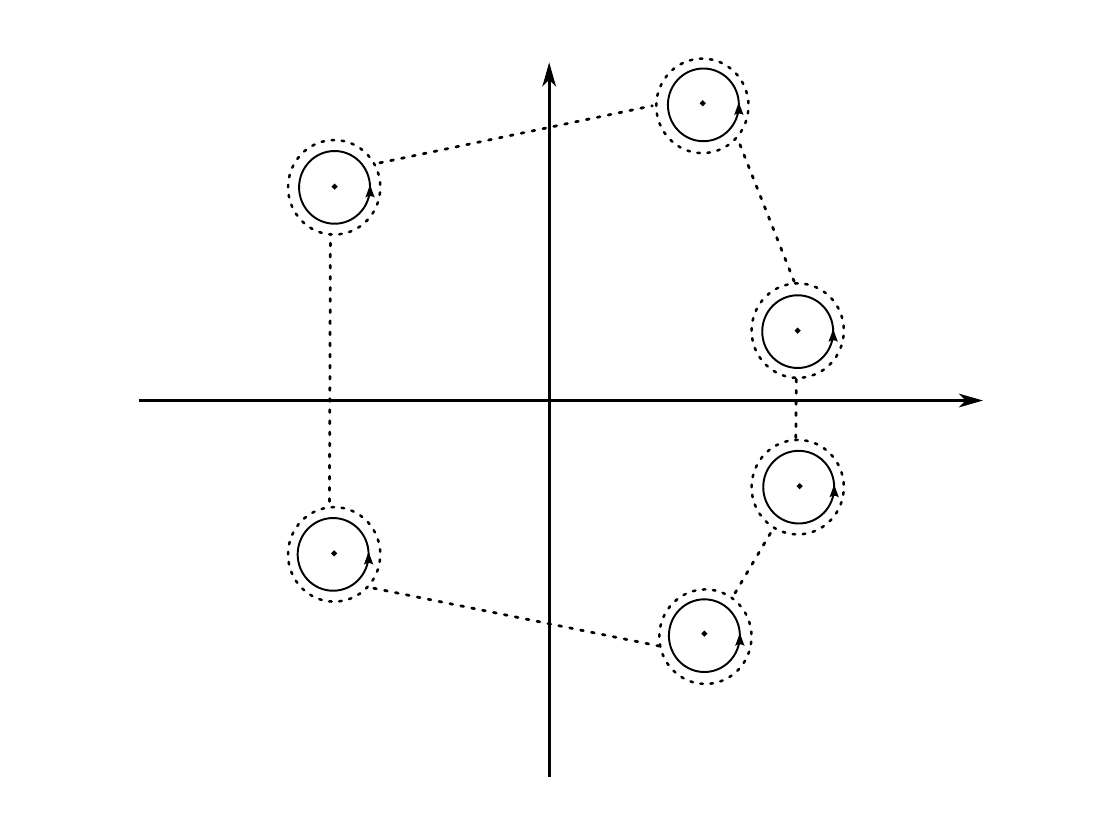}
\begin{pgfpicture}
\pgfputat{\pgfxy(-8.1,2.7)}{\pgfbox[center,center]{$\lambda_j$}}
\pgfputat{\pgfxy(-8.8,2.8)}{\pgfbox[center,center]{$\gamma_j$}}
\end{pgfpicture}
\end{center}
\vspace{-1cm}
\end{figure}

\begin{remark}
For the rest of this section, we drop the observations about regularity of metrics. Since the obstructions are topological and $M$ is assumed smooth, if $M$ has no obstructions to the existence of a $C^k$ semi--Riemannian metric of index $\nu$, it automatically admits semi--Riemannian metrics of this index of class $C^r$ for any other $r$. Thus, we shall omit the regularity of metrics in this section.
\end{remark}

\begin{remark}\index{$\nu$--topological obstruction}
In the sequel, by {\em having $\nu$--topological obstructions} we mean having obstructions to the existence of metrics of index $\nu$. In addition, if the index $\nu$ is evident from the context, it may be omitted.
\end{remark}

\begin{corollary}\label{cor:stupidcor}
If $M$ is {\em contractible}\footnote{i.e., has the same homotopy type of a point.}, then $M$ admits semi--Riemannian metrics of all possible signatures.
\end{corollary}

\begin{proof}
This is immediate from the fact that vector bundles with contractible basis are trivial. Hence $M$ admits distributions of all possible ranks, and the result follows from Proposition~\ref{prop:metricdistribution}.
\end{proof}

\begin{remark}
Corollary~\ref{cor:stupidcor} is not as cloddish as it seems. In fact, there are elaborate constructions of contractible smooth manifolds not homeomorphic to the Euclidean space, such as {\em Whitehead manifolds}.\footnote{Whitehead manifolds are contractible $3$--manifolds not homeomorphic to $\R^3$, and were discovered by Henri Whitehead in his attempts to prove the Poincar\'e conjecture, see Kirby \cite{kirby}.}
\end{remark}

We now treat the case $\nu=1$ of Lorentzian metrics on $M$. The correspondent obstruction may be characterized by a well--known {\em characteristic class}, i.e., a rule that associates a cohomology class of $M$ to each vector bundle $E$ over $M$, measuring how {\em twisted} it is, and particularly if it admits nontrivial sections. Existence of a such nontrivial section for $TM$ means that there is a globally defined non vanishing vector field on $M$. This clearly implies\footnote{Notice that the converse is not necessarily true, see Remark~\ref{re:linebundvec}.} the existence of a rank $1$ distribution on $M$ spanned by this field, which, by Proposition~\ref{prop:metricdistribution}, guarantees the existence of a Lorentzian metric on $M$. We shall later go back to general indexes $\nu$ in the end of this section.

The definition of the {\em Euler class} of an oriented vector bundle can be given in several different ways. Namely, one may give an axiomatic characterization, or an explicit formula using the curvature of a connection on this bundle, or finally a typical algebraic topology construction using the orientation class of this bundle. We shall adopt the last, see \cite{kirk,husemoller,kn2,milnor} for a more comprehensive study.

Let $E$ be an oriented $C^0$ vector bundle over $M$ of rank $r$, and denote $\dot E=E\setminus\mathbf 0_E$ the complementary of the null section in $E$, see Remark~\ref{re:nullsection}. Consider the inclusion of pairs $j_x:(E_x,\dot E_x)\to (E,\dot E)$, where $\dot E_x=E_x\setminus\{\mathbf 0_E (x)\}$, and the induced homomorphism $j_x^*:H^r(E,\dot E;\Z)\to H^r(E_x,\dot E_x;\Z)$ between the respective cohomologies. Standard arguments prove that there exists a unique $U\in H^r(E,\dot E;\Z)$, called {\em orientation class}\index{Orientation class} of $E$, such that $j_x^*(U)$ is a generator of $H^r(E_x,\dot E_x;\Z)$ for all $x\in M$, see for instance \cite{husemoller}. Denote by $i:(E,\emptyset)\hookrightarrow (E,\dot E)$ the inclusion and $i^*:H^*(E,\dot E;\Z)\to H^*(E;\Z)$ the restriction homomorphism induced between the respective cohomology rings.

\begin{definition}
The {\em Euler class}\index{Euler class} of an oriented vector bundle $E$ over $M$ of rank $r$ is the cohomology class $e(E)\in H^r(M,\Z)$ defined by $$e(E)=(\pi^*)^{-1}i^*(U),$$ where $U$ is the orientation class of $E$ and $\pi^*:H^r(M,\Z)\to H^r(E,\Z)$ is the homomorphism induced by the projection of $E$.
\end{definition}

\begin{remark}
In case $E$ is non orientable, it is necessary to use cohomology with {\em twisted coefficients} to obtain a substitute for the orientation class $U$ in the above case. One may define analogously the Euler class of non orientable vector bundles as a cohomology class in such twisted cohomology.
\end{remark}

\begin{remark}
Notice that for any $s\in\sect^0(E)$, we have $\pi\circ s=\id$ and hence $s^*=(\pi^*)^{-1}:H^*(E,\Z)\to H^*(M,\Z)$. Thus
\begin{eqnarray*}
e(E) &=& (\pi^*)^{-1}i^*(U)\\
&=& s^*\circ i^*(U)\\
&=& (i\circ s)^*(U).
\end{eqnarray*}
\end{remark}

The main reason to study the Euler class of oriented vector bundles over $M$ is the following.

\begin{theorem}\label{thm:obsteuler}
Let $E$ be an orientable vector bundle over $M$ of rank $r$. The {\em primary} obstruction to the existence of a nontrivial\footnote{i.e., $s\ne\mathbf 0_E$.} section $s\in\sect^0(E)$ is the Euler class $e(E)$. In case $r=m$, this is the {\em unique} obstruction.
\end{theorem}

We will not give a proof of this result. Basically, it is done by studying the obstruction to lift a nontrivial section defined in the $n$--skeleton of a CW--complex to its $(n+1)$--skeleton. Up to minor identifications, the crucial fact in use is that a continuous function $f:S^{n-1}\to\R$ can be continuously extended to a function $\widetilde f:B^n\to\R$ defined on the $n$--ball $B^n$ that has boundary $S^{n-1}$ if and only if it is homotopic to a constant. Since $E_x\setminus\{\mathbf 0_E(x)\}$ has the same homotopy type of $S^{r-1}$, and $\pi_k(S^{r-1})=0$ for $k<r-1$, one easily verifies that there is no obstruction to lift a nontrivial section until reaching the $(r-1)$--skeleton. The final step to obtain the desired non vanishing section of $E$ is to lift it to the $r$--skeleton, which corresponds to $M$. In case $r=m$, the only obstruction is in this final step, and it is characterized by the vanishing of the Euler class. Nevertheless, for $r>m$, there are further obstructions, being $e(E)$ the first of them. See Davis and Kirk \cite{kirk} for a proof.

\begin{corollary}\label{cor:lorentzobst}
The obstruction class to the existence of Lorentzian metrics on $M$ is given by $e(TM)$.
\end{corollary}

\begin{proof}
Vanishing of $e(TM)$ guarantees the existence of a globally defined section $v\in\sect^0(TM)$ that never vanishes. The conclusion follows by considering the rank $1$ distribution on $M$ spanned by $v$ and applying Proposition~\ref{prop:metricdistribution}.
\end{proof}

\begin{remark}\label{re:linebundvec}
In the proof of Corollary~\ref{cor:lorentzobst} we used a rank $1$ distribution spanned by a vector field. Nevertheless, not every rank $1$ distribution is of this form. More precisely, it is spanned by a non vanishing vector field if and only if it is orientable. It is also possible to prove that on a {\em simply connected} manifold (compact or not), every rank $1$ distribution is spanned from a globally defined non vanishing vector field. Let us briefly comment an example of rank $1$ distribution that is not spanned by any globally defined non vanishing vector field, discussed in Palomo and Romero \cite{romero}.

Consider $G=S^1\times\SO(3)$. Since $G$ is a Lie group, it is parallelizable, and hence every vector field $X\in\sect^k(TG)$ can be regarded as a map
\begin{equation}\label{eq:gvfield}
X:G\la\mathfrak g,
\end{equation}
where $\mathfrak g$ is a $4$--dimensional real vector space. Thus, every rank $1$ distribution $\mathcal D$ can be thought as a map $\mathcal D:G\to\R P^3\subset\mathfrak g$. Composing $\mathcal D$ with a fixed diffeomorphism $f:\R P^3\to\SO(3)$, it follows that every rank $1$ distribution $\mathcal D$ can be regarded as map $$f\circ\mathcal D:G\la\SO(3).$$ Consider the distribution $\mathcal D_2$ induced in this way by the projection $G\to\SO(3)$ on the second factor. Assuming that $\mathcal D_2$ is spanned by a vector field \eqref{eq:gvfield} on $G$ and using that $\mathfrak g\setminus\{0\}$ is simply connected, one can easily conclude\footnote{Consider the homomorphisms between the fundamental groups induced by these maps.} that $\SO(3)$ is simply connected, which is false. Therefore, $\mathcal D_2$ is {\em not} spanned by a globally defined non vanishing vector field on $G$.

Finally, the {\em existence} of a globally non vanishing vector field on $M$ is equivalent to the {\em existence} of a rank $1$ distribution on $M$, that may not be spanned by this vector field. This follows from the fact that both statements are equivalent to the vanishing of the Euler class $e(TM)$.
\end{remark}

\begin{proposition}\label{prop:noncompactlorentz}
Every non compact manifold admits a Lorentzian metric.
\end{proposition}

\begin{proof}
If $M$ is non compact, standard arguments prove that $H^m(M,\Z)=0$. Hence $e(TM)$ trivially vanishes and hence, from Corollary~\ref{cor:lorentzobst}, it follows that $M$ admits Lorentzian metrics.
\end{proof}

In order to better describe this obstruction in the compact case, we present an axiomatic characterization of the Euler class, that can be found for instance in Kobayashi and Nomizu \cite{kn2}.

\begin{proposition}\label{prop:knaxioms}
The Euler class $e$ for oriented vector bundles of rank $r$ is characterized by the following axioms.
\begin{itemize}
\item[(i)] $e(E)\in H^r(M,\Z)$ and $e(E)$ is trivial if $r$ is odd;
\item[(ii)] If $f:N\to M$ is a smooth map, then $$e(f^*E)=f^*(e(E));$$
\item[(iii)] If $E_1$ and $E_2$ are oriented vector bundles over $M$ of rank $2$, then $$e(E_1\oplus E_2)=e(E_1)\wedge e(E_2);$$
\item[(iv)] Let $E_\C$ be the natural complex line bundle over $\C P^1$. Then $e(E_\C)$ coincides with the first Chern class $c_1(E_\C)$.
\end{itemize}
Axioms (ii), (iii) and (iv) are called the {\em naturality}, {\em Whitney sum} and {\em normalization} axioms, respectively. If another characteristic class satisfies (i)--(iv), then it must coincide with the Euler class.
\end{proposition}

\begin{corollary}\label{cor:lorentzodddim}
If $M$ is odd dimensional and orientable, then $M$ admits Lorentzian metrics.
\end{corollary}

\begin{proof}
From axiom (i) in Proposition~\ref{prop:knaxioms}, $e(TM)=0$. The result is then immediate from Corollary~\ref{cor:lorentzobst}.
\end{proof}

\begin{corollary}
If $M=M_1\times M_2$ is a product manifold, then $M$ admits Lorentzian metrics if at least one of the factors $M_i$ admits Lorentzian metrics.
\end{corollary}

\begin{proof}
Once more, we use Corollary~\ref{cor:lorentzobst}. From the Whitney sum axiom (i) in Proposition~\ref{prop:knaxioms}, $e(TM)=e(TM_1)\wedge e(TM_2)$ is trivial if either $e(TM_1)$ or $e(TM_2)$ is trivial. Notice that one could also build {\em directly} the Lorentzian product metric by considering the direct sum of the Lorentzian metric on one factor and any Riemannian metric on the other factor.
\end{proof}

\begin{corollary}
Suppose $M$ is orientable and has even dimension $m$. Let $\{e_i\}_{i=1}^m$ be a $g_\mathrm R$--orthonormal frame and define $$\Omega_{ij}(v,w)=g_\mathrm R(R^\mathrm R(v,w)e_j,e_i), \quad i,j=1,\ldots,m.$$ Then the Euler class of $TM$ is given by
\begin{equation}\label{eq:eulerclass}
e(TM)=\frac1{\left(\tfrac{m}{2}\right)!\pi^{m/2}2^m}\sum_{\sigma\in\mathfrak S_m}\sgn(\sigma) \Omega_{\sigma(1)\sigma(2)}\wedge\Omega_{\sigma(3)\sigma(4)} \wedge\ldots\wedge\Omega_{\sigma(m-1)\sigma(m)},
\end{equation}
where $\mathfrak S_m$ denotes the symmetric group on $m$ elements and $\sgn(\sigma)$ the sign of the permutation $\sigma$.
\end{corollary}

Formula \eqref{eq:eulerclass} is proved simply verifying that the axioms (i)--(iv) of Proposition~\ref{prop:knaxioms}.

\begin{remark}
We are being a bit sloppy about the coefficients of cohomologies above. Proposition~\ref{prop:knaxioms} states that $e(TM)\in H^m(M,\Z)$, and formula \eqref{eq:eulerclass} clearly gives an expression of a differential form, i.e., an element of $H^m(M,\R)$. There is however a natural identification between the cohomology rings $H^*(M,\Z)\hookrightarrow H^*(M,\R)$.
\end{remark}

We now state the celebrated Gauss--Bonnet--Chern Theorem, that relates the Euler class $e(TM)$ of the tangent bundle of $M$ with the Euler characteristic $\chi(M)$. It also extends the classic Gauss--Bonnet Theorem for $2$--manifolds to any even dimensional manifold. A complete proof can be found for instance in Mercuri, Piccione and Tausk \cite{picmertausk}.

\begin{gbcthm}\label{thm:gbcthm}\index{Theorem!Gauss--Bonnet--Chern}
Let $M$ be compact and oriented and consider $e(TM)$ the expression for the Euler class given by \eqref{eq:eulerclass}. Then
\begin{equation}
\int_M e(TM)=\chi(M).
\end{equation}
\end{gbcthm}

Thus, for compact orientable manifolds, $\chi(M)=0$ if and only if the Euler class of $TM$ is trivial. Hence, we may give the following characterization of the obstruction to the existence of Lorentzian metrics on compact orientable manifolds.

\begin{proposition}\label{prop:existlorentzian}
A compact orientable manifold $M$ admits Lorentzian metrics if and only if $\chi(M)=0$.
\end{proposition}

\begin{proof}
This is immediate from Corollary~\ref{cor:lorentzobst} and the Gauss--Bonnet--Chern Theorem~\ref{thm:gbcthm}.
\end{proof}

Notice that this result gives a complete description of the obstruction for $\nu=1$. Recall that non compact orientable manifolds always admit Lorentzian metrics, see Proposition~\ref{prop:noncompactlorentz}. Odd dimensional compact manifolds also admit Lorentzian metrics, from Corollary~\ref{cor:lorentzodddim}. Finally, even dimensional compact manifolds admit Lorentzian metrics if and only if its Euler characteristic is different from $0$. In a low dimensional context, it is possible to give even more detailed results, for instance the following.

\begin{corollary}\label{cor:2dimlorentz}
The only two--dimensional compact manifolds that admit Lorentzian metrics are the torus $S^1\times S^1$ and the Klein bottle.
\end{corollary}

\begin{proof}
It is a classic result that these are the only two--dimensional compact manifolds whose Euler characteristic is zero.
\end{proof}

Except for the case $\nu=1$, it is in general a fairly difficult task to give universal necessary and sufficient conditions for the existence of semi--Riemannian metrics of index $\nu$ on $M$, or distributions of rank $\nu$ on $M$. One may try to characterize the obstruction to the existence of such distributions in the same fashion of Theorem~\ref{thm:obsteuler}. This would be done lifting nontrivial sections of a Grassmannian bundle $\gr_\nu(M)$ through $n$--skeletons. In addition, it would be necessary to compute the homotopy groups of $\gr_\nu(M)$, possibly using its homogeneous space structure and homotopy tools, such as the {\em Bott periodicity}. Finally, it is very likely that even for vector bundles of rank $m$, the obstruction is not unique, as in Theorem~\ref{thm:obsteuler}.

Nevertheless, in a low dimensional context it is still possible to use a few tricks. For instance, notice that if $g\in\met_\nu^k(M)$ then $-g\in\met_{m-\nu}^k(M)$. Therefore, 
\begin{equation}\label{dualitymetric}
\met_\nu^k(M)\ne\emptyset\;\;\mbox{ if and only if }\;\;\met_{m-\nu}^k(M)\ne\emptyset.
\end{equation}
In particular, this implies the following.

\begin{corollary}
If $M$ is a tri--dimensional compact orientable manifold, then $M$ admits semi--Riemannian metrics of all possible indexes.
\end{corollary}

\begin{proof}
Clearly $M$ admits metrics of index $\nu=0$. From Corollary~\ref{cor:lorentzodddim}, $M$ also admits metrics of index $\nu=1$. From \eqref{dualitymetric}, since $m=3$, it follows that $M$ admits metrics of all possible indexes $\nu=0,1,2,3$.
\end{proof}

Another approach to is to obtain more specific characterizations for simple and well--studied manifolds. The simplest $m$--dimensional connected non contractible\footnote{Recall Corollary~\ref{cor:stupidcor}.} manifold is the $m$--sphere $S^m$. A complete and detailed discussion on fiber bundles over spheres together with a classification of such bundles is given in Walschap \cite{walschap}. We end this section with the following two results by Steenrod on this topic, indirectly proven in \cite{steenrod} and \cite{steenrod2} respectively.

\begin{theorem}\label{thm:spheres1}
For the following values of $\nu$ and $m$, the $m$--sphere $S^m$ admits semi--Riemannian metrics of index $\nu$,
\begin{itemize}
\item[(i)] $m$ even, $\nu=0$ and $\nu=m$;
\item[(ii)] $m$ odd, $\nu\in\{0,1,m-1,m\}$;
\item[(iii)] $m\equiv3\mod4$, $0\leq\nu\leq 3$ and $m-3\leq\nu\leq m$;
\item[(iv)] $m\equiv7\mod8$, $0\leq\nu\leq 7$ and $m-7\leq\nu\leq m$.
\end{itemize}
\end{theorem}

\begin{remark}
Notice that (i) and (ii) above correspond respectively to the existence of Riemannian and Lorentzian metrics, using \eqref{dualitymetric}. Moreover, recall that (ii) is a consequence of Corollary~\ref{cor:lorentzodddim}.
\end{remark}

\begin{theorem}\label{thm:spheres2}
For the following values of $\nu$ and $m$, the $m$--sphere $S^m$ {\em does not} admit semi--Riemannian metrics of index $\nu$,
\begin{itemize}
\item[(i)] $m$ even, $1\leq\nu\leq m-1$;
\item[(ii)] $m+1\equiv0\mod2^r$, where $2^r$ is the {\em highest power}\footnote{i.e., $\frac{m+1}{2^r}$ is odd.} of $2$ dividing $m+1$, and $2^r\leq\nu\leq m-2^r$.
\end{itemize}
\end{theorem}

\section{A few lemmas}

In this final section, we prove a few lemmas of semi--Riemannian geometry that will be used in the following chapters. We start with a few results concerning self intersection of geodesics and parallelism of Jacobi fields and tangent fields.

\begin{figure}[htf]
\begin{center}
\includegraphics[scale=0.8]{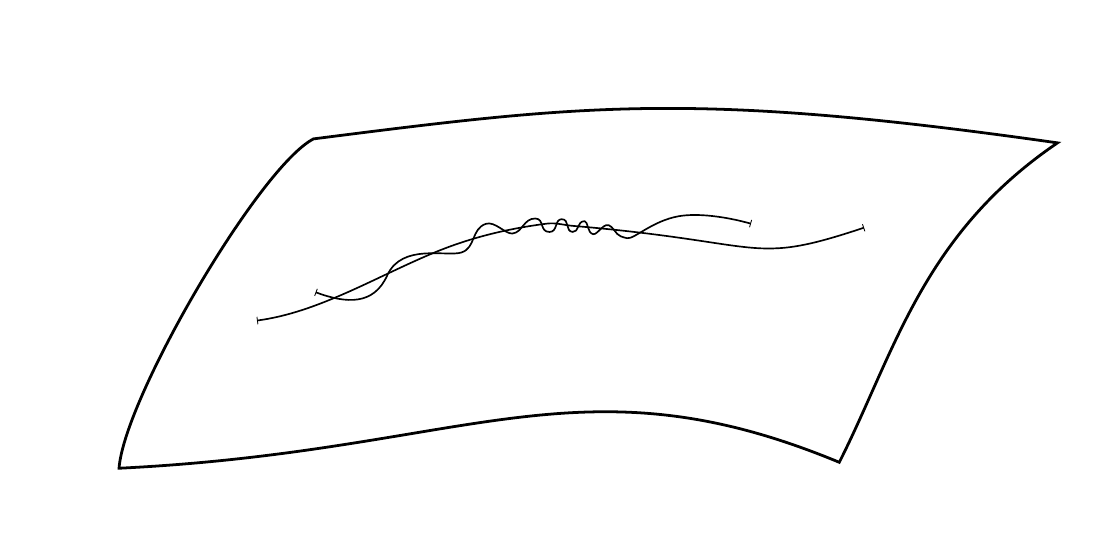}
\begin{pgfpicture}
\pgfputat{\pgfxy(-6.8,1.6)}{\pgfbox[center,center]{$\gamma_1$}}
\pgfputat{\pgfxy(-4,3)}{\pgfbox[center,center]{$\gamma_2$}}
\pgfputat{\pgfxy(-1.2,1.7)}{\pgfbox[center,center]{$M$}}
\end{pgfpicture}
\end{center}
\caption{Two $g$--geodesics intersect only finitely many times, unless one of them is an affine reparameterization of the other.}
\end{figure}

\begin{lemma}\label{le:finiteintersection}
Let $\gamma_i:[a_i,b_i]\rightarrow M$ two $g$--geodesics. Then the set of points where these geodesics intersect is finite, unless one is an affine reparameterization of the other.
\end{lemma}

\begin{proof}
Since the images of $\gamma_1$ and $\gamma_2$ are compact, if there were infinitely many intersection points, there would be an accumulation intersection point $p=\gamma_1(t)=\gamma_2(s)$. Consider $U$ a normal neighborhood of $p$. If $\dot{\gamma_1}(t)$ and $\dot{\gamma_2}(s)$ are linearly independent, since there are infinitely many points near $p$ such that $\gamma_1$ and $\gamma_2$ coincide in $U$, there is an obvious contradiction to injectivity of the exponential map on $U$. Otherwise, if $\dot{\gamma_1}(t)$ and $\dot{\gamma_2}(s)$ are linearly dependent, then $\gamma_1$ and $\gamma_2$ are affine reparameterizations of each other.
\end{proof}

\begin{proposition}\label{prop:selfintersections}
Let $\gamma:[0,1]\rightarrow M$ be a $g$--geodesic in $M$. If the set $$\mathcal{I}=\big\{(t,s)\in [0,1]\times [0,1] : t\neq s, \gamma(t)=\gamma(s)\big\}$$ is infinite, then $\gamma$ is a portion of a periodic geodesic with period $\omega<1$, see Example~\ref{ex:periodicgeod}.
\end{proposition}

\begin{proof}
If $\mathcal{I}$ is infinite, there exists an accumulation point $\left(\overline{t},\overline{s}\right)\in\mathcal{I}$. The local injectivity of $\gamma$ implies that $\overline{t}\neq\overline{s}$, suppose $\overline{t}<\overline{s}$. Take $\varepsilon>0$ small, and define $\gamma_1=\gamma_{\varepsilon}\big|_{\left[\overline{t}-\varepsilon,\overline{t}+\varepsilon\right]}$ and $\gamma_2=\gamma_{\varepsilon}\big|_{\left[\overline{s}-\varepsilon,\overline{s}+\varepsilon\right]}$, where $\gamma_\varepsilon$ is the extension of $\gamma$ to $[-\varepsilon,1+\varepsilon]$. Since $\gamma_1$ and $\gamma_2$ are defined on compact intervals and intersect infinitely many times, from Lemma~\ref{le:finiteintersection}, one is an affine reparameterization of the other. Moreover, both are restrictions of the same geodesic $\gamma_\varepsilon$, hence $\gamma_1(t+\omega)=\gamma_2(t)$ for $t\in \left[\overline{t}-\varepsilon,\overline{t}+\varepsilon\right]$, where $\omega=\overline{s}-\overline{t}\leq 1$. Therefore $\dot{\gamma_1}(\overline{t})=\dot{\gamma_2}(\overline{s})$. Hence, from Proposition~\ref{prop:geodexist}, $\gamma$ is a portion of a periodic geodesic with period $\omega\leq 1$. If $\overline{t}=0$ and $\overline{s}=1$, one can easily derive a contradiction with local injectivity of $\gamma$ around $0$, which implies $\omega <1$.
\end{proof}

\begin{lemma}\label{le:parallelfinite}
Let $\gamma:[a,b]\rightarrow M$ be a $g$--geodesic and $J$ a nontrivial Jacobi field along $\gamma$, that is not everywhere parallel to $\dot{\gamma}$. Then $$\mathcal{D}=\{t\in[a,b]:J(t) \mbox{ is parallel to }\dot{\gamma}\}$$ consists only of isolated points, hence is finite.
\end{lemma}

\begin{proof}
Consider a basis of $T_{\gamma(a)}M$ given by $(\dot{\gamma}(a),e_2,\dots,e_m)$ and its parallel transport along $\gamma$ creating a frame $(e_1(t),e_2(t),\dots,e_m(t))$, with $e_1(t)=\dot{\gamma}(t)$, see Definition~\ref{def:frame}. Then, writing $$J=\sum_{i=1}^m J_i(t)e_i(t),$$ it follows that $J$ is parallel to $\dot{\gamma}$ at time $t$ if and only if $J_i(t)=0$, for $i\geq 2$. Suppose that there exists a limit $t_\infty\in [a,b]$ of a sequence $\{t_n\}_{n\in\N}$ of different elements of $\mathcal{D}$. From continuity of $J$ it follows that $t_\infty\in\mathcal{D}$. Thus for each $i\geq 2$, the coordinate function $J_i(t)$ has a convergent sequence of zeros $\{t_n\}_{n\in\N}$ and hence $J_i'(t_\infty)=0$. Therefore, the covariant derivative $\D^g J(t_\infty)$ is also parallel to $\dot{\gamma}$.

It is then possible to find $c_1,c_2\in\R$ such that $\tilde{J}=(c_1+c_2t)\dot{\gamma}(t)$ satisfies $\tilde{J}(t_\infty)=J(t_\infty)$ and $\D^g \tilde{J}(t_\infty)=\D^g J(t_\infty)$. Since the Jacobi equation is a second order linear ODE, $\tilde{J}=J$. Hence $J$ is always parallel to $\dot{\gamma}$, a contradiction.
\end{proof}

The next result gives an estimate of the difference of the normalized tangent vectors to a geodesic segment at its endpoints and the Riemannian length of this segment. We shall use such estimate to define {\em admissibility} of general endpoints conditions, see Definition~\ref{def:admgbc}.

\begin{lemma}\label{le:short}
Let $U\subset\R^m$ be an open subset and $g_\infty\in\met_\nu^k(U)$. Then for all compact subsets $K\subset U$ there exists a positive number $c>0$ and an open neighborhood $\mathcal O$ of $g_\infty$ in the weak Whitney $C^1$--topology,\footnote{See Section~\ref{sec:banachspacetensors} for basic definitions of topologies on spaces of tensors over $M$.} such that for all $g\in\mathcal O$ and all non constant $g$--geodesic $\gamma:[a,b]\rightarrow U$ with $\gamma([a,b])\subset K$, the following inequality holds \begin{equation} \left\| \frac{\dot\gamma(b)}{\|\dot\gamma(b)\|}-\frac{\dot\gamma(a)}{\|\dot\gamma(a)\|}\right\|\leq c\int_a^b \|\dot\gamma(t)\|\;\mathrm{d}t,\end{equation} where $\|\cdot\|$ is the Euclidean norm.
\end{lemma}

\begin{proof}
Given $g\in\met_\nu^k(U)$, denote by $\chr^g$ the Christoffel tensor of the Levi--Civita connection $\nabla^g$ relatively to the Levi--Civita connection of the Euclidean metric on $U$, see Definition~\ref{def:christoffeltens}. Thus, for all $x\in U$, $\Gamma^g(x):\R^m\times\R^m\to\R^m$ is a symmetric bilinear map depending continuously on $x$, and if $\gamma$ is a $g$--geodesic, it satisfies the $g$--geodesic equation \eqref{eq:geodequation}, $$\ddot\gamma+\Gamma^g(\gamma)(\dot\gamma,\dot\gamma)=0,$$ where $\ddot\gamma$ denotes the ordinary second derivative of $\gamma$ in $\R^m$. This association $g\mapsto\Gamma^g$ is clearly continuous when $\met_\nu^k(U)$ is endowed with the weak Whitney $C^1$--topology and the space of $\Gamma^g$'s is endowed with the weak Whitney $C^0$--topology. If $K\subset U$ is a given compact subset, set $\kappa=\max_{x\in K} \|\Gamma^{g_\infty}(x)\|+1$ and define $$\mathcal O=\{g\in\met_\nu^k(U):\|\Gamma^g(x)\| <\kappa, \; \mbox{ for all }  x\in K\},$$ which is obviously an open neighborhood of $g_\infty$ in the weak Whitney $C^1$--topology.

Let us show that such $\mathcal O$ satisfies the conclusion, with $c=2\kappa$. Indeed, if $g\in\mathcal O$ and $\gamma$ is a non constant $g$--geodesic with image lying in $K$, then at each time $t\in [a,b]$,
\begin{eqnarray*}
\left\|\frac{\dd}{\dd t}\frac{\dot\gamma}{\|\dot\gamma\|}\right\| &=& \left\|\frac{\ddot\gamma}{\|\dot\gamma\|}-\frac{\dot\gamma\langle\dot\gamma,\ddot\gamma\rangle}{\|\dot\gamma\|^3}\right\| \\
&=& \left\|-\frac{\Gamma^g(\gamma)(\dot\gamma,\dot\gamma)}{\|\dot\gamma\|}+\frac{\langle\dot\gamma,\Gamma^g(\gamma)(\dot\gamma,\dot\gamma)\rangle}{\|\dot\gamma\|^3}\dot\gamma\right\| \\
&\leq& \frac{\|\Gamma^g(\gamma)\|\|\dot\gamma\|^2}{\|\dot\gamma\|} + \frac{\|\Gamma^g(\gamma)\|\|\dot\gamma\|^4}{\|\dot\gamma\|^3} \\
&\leq& 2\kappa\|\dot\gamma\|.
\end{eqnarray*}
Integrating the above inequality in $[a,b]$, it follows that
\begin{eqnarray*}
\left\| \frac{\dot\gamma(b)}{\|\dot\gamma(b)\|}-\frac{\dot\gamma(a)}{\|\dot\gamma(a)\|}\right\| &\leq& \left\|\int_a^b \frac{\dd}{\dd t}\frac{\dot\gamma}{\|\dot\gamma\|}\;\dd t\right\| \\
&\leq& \int_a^b \left\|\frac{\dd}{\dd t}\frac{\dot\gamma}{\|\dot\gamma\|}\right\|\;\dd t \\
&\leq& 2\kappa\int_a^b\|\dot\gamma(t)\|\;\dd t.\qedhere
\end{eqnarray*}
\end{proof}

We end with a result from Biliotti, Javaloyes and Piccione \cite{biljavapic} that guarantees the existence of a local section of a vector bundle with prescribed values along a curve for the section and its covariant derivative in a transverse direction.

\begin{lemma}\label{le:extension}
Let $E$ be a smooth vector bundle over $M$ endowed with a connection $\nabla$. Consider $\gamma\in C^{k+1}([a,b],M)$ and $v\in\sect^k(\gamma^*TM)$ a vector field along $\gamma$, such that $v(t_0)$ is not parallel to $\dot{\gamma}(t_0)$ for some $t_0\in\; ]a,b[$. Then there exists an open interval $I\subset [a,b]$ containing $t_0$ with the property that, given sections $H,K\in\sect^k(\gamma^*E)$ with compact support in $I$ and given any open set $U$ containing $\gamma(I)$, there exists $h\in\sect^k(E)$ with compact support contained in $U$, such that
\begin{equation}
h(\gamma_0(t))=0\;\;\mbox{ and }\;\;\nabla_{J(t)} h=K(t), \quad\mbox{ for all } t\in I.
\end{equation}
\end{lemma}

\begin{proof}
This proof is in great part adapted from \cite[Lemma 2.4]{biljavapic}. Let $I\subset\,]a,b[$ be a sufficiently small open interval such that $\gamma\vert_I$ is a $C^{k+1}$ embedding and such that $v(t)$ is not parallel to $\dot\gamma(t)$ for all $t\in I$. Consider $S$ a $C^{k+1}$ hypersurface of $M$ containing $\gamma(I)$, such that $v(t)\not\in T_{\gamma(t)}S$ for all $t\in I$, and $V\in\sect^k(S^*(TM))$ a section along $S$ such that $V(\gamma(t))=v(t)$ for all $t\in I$. 

By possibly reducing the size of both $I$ and $S$, we may assume the existence of $\varepsilon>0$ and a diffeomorphism $$\phi:S\times\left]-\varepsilon,\varepsilon\right[\ni(x,\lambda)\longmapsto\phi(x,\lambda)\in \widetilde U\subset U,$$ where $\widetilde U$ is an open subset of $M$ contained in $U$ and that contains $\gamma(I)$, such that $\frac{\partial\phi}{\partial\lambda}(x,0)=V(x)$ for all $x\in S$. For instance, such a diffeomorphism can be obtained using the exponential map $\exp^\mathrm R$ of the Levi--Civita connection $\nabla^{\mathrm R}$ of $M$ by setting $$\phi(x,\lambda)=\exp^{\mathrm R}_x(\lambda V(x)), \quad (x,\lambda)\in S\times\left]-\varepsilon,\varepsilon\right[.$$ Clearly, $\widetilde U$ may be chosen small enough to be contained in a trivialization of $E$. Let $r\in\N$ be the rank of $E$ and $p(x,\lambda):\R^r\to E_{\phi(x,\lambda)}$ a $C^{k+1}$ referential of $\phi^*(E\vert_{\widetilde U})$, with the property that $\nabla_{V(x)}p(x,\lambda)=0,$ i.e., $p$ is parallel along the curves $\left]-\varepsilon,\varepsilon\right[\ni\lambda\mapsto\phi(x,\lambda)\in\widetilde U$. For instance, such referential $p$ may be chosen selecting an arbitrary $C^{k+1}$ referential of $E$ along $S$, and extending it by parallel transport along the curves $\lambda\mapsto\phi(x,\lambda)$. The problem of determining the required section $h$ is then reduced to determining a $C^k$ map with compact support $$\widetilde h:S\times\left]-\varepsilon,\varepsilon\right[\la\R^r$$ such that
\begin{equation*}
\widetilde h(\gamma(t),0)=p(\gamma(t),0)^{-1}H(t) \;\mbox{ and }\; \frac{\partial\widetilde h}{\partial\lambda}(\gamma(t),0)=p(\gamma(t),0)^{-1}K(t),
\end{equation*}
for all $t\in I$. Once such $\widetilde h$ has been determined, the desired section $h$ is obtained setting $h=0$ outside $\widetilde U$ and $$h(\phi(x,\lambda))=p(x,\lambda)\circ\widetilde h(x,\lambda), \quad (x,\lambda)\in S\times\left]-\varepsilon,\varepsilon\right[.$$

This map $\widetilde h$ can be constructed as follows. Let $\widetilde H,\widetilde K:S\to\R^r$ be $C^k$ maps having compact support, such that
\begin{equation*}
p(\gamma(t),0)\circ\widetilde H(\gamma(t))=H(t) \;\mbox{ and }\; p(\gamma(t),0)\circ\widetilde K(\gamma(t))=K(t)
\end{equation*}
for all $t\in I$. Then, define $$\widetilde h(x,\lambda)=\widetilde H(x)+f(\lambda)\widetilde K(x),$$ where $f:\left]-\varepsilon,\varepsilon\right[\to\R$ is a $C^k$ function with compact support such that $f(\lambda)=\lambda$ near $\lambda=0$. This concludes the construction and the proof.
\end{proof}

\chapter{Rudiments of functional analysis}
\label{chap2}

In this chapter, we aim to recall some basic facts of functional analysis that will be later used, together with a few lemmas. We begin with a section discussing elementary topics of topological vector spaces, Fr\'echet, Banach and Hilbert spaces. In addition, a few conventions are made and notation is fixed. Nevertheless, we will not state classic results such as the Hahn--Banach Theorem, the Banach--Schauder Theorem and the Closed Graph Theorem, that will be assumed. In Section~\ref{sec:compactfredholm}, we state some facts about compact and Fredholm operators, specially regarding their stability, without giving proofs. Sections~\ref{sec:calcbanspaces} and~\ref{sec:funcspaces} deal respectively with differential calculus on Banach spaces, and classic function spaces with several regularities, such as $C^k$, $L^p$ and Sobolev $W^{k,p}$. Finally, Section~\ref{sec:afewmorelemmas} concludes the chapter with several auxiliary lemmas related to the previous topics.

Throughout the text, all vector spaces are supposed to be {\em real},\footnote{Notice however that several results are automatically valid for complex vector spaces. Nevertheless, all of our applications will require only real vector spaces, thus these results are stated in this context.} unless otherwise stated, and the term {\em operator}\index{Operator} will be used exclusively for linear maps. The given treatment of elementary topics only aims to keep the text self contained, and for a detailed treatment we refer to \cite{brezis,fabian,kreyszig,reedsimon,rudinfa,ys}.

\section{Basic concepts of Fr\'echet, Banach and Hilbert spaces}

In this section we recall the basic elements of functional analysis that will be used in the following chapters. Although most definitions are repeated here, several important results will be omitted, or only stated without a proof. Complete references for most topics mentioned in the sequel are the textbooks above mentioned.

\begin{definition}
A {\em topological vector space}, or {\em TVS},\index{Topological Vector Space} is a (real) vector space $V$ endowed with a topology for which the vector space operations
\begin{eqnarray*}
V\times V\ni(v,w) &\longmapsto & v+w\in X, \\
\R\times V\ni(\lambda,v) &\longmapsto &\lambda v\in V,
\end{eqnarray*}
are continuous. In addition, a TVS is {\em locally convex}\index{Topological Vector Space!locally convex} if  every neighborhood of the origin contains an open neighborhood $U$ of the origin such that if $v,w\in U$ and $0\leq t\leq 1$, then $tv+(1-t)w\in U$.
\end{definition}

It is possible to characterize a locally convex TVS with additional topological assumptions using semi--norms as follows. This equivalent approach will be useful for the definition of Fr\'echet spaces.

\begin{definition}
A {\em semi--norm}\index{Semi--norm} on $V$ is a function $p:V\to\R$ such that for all $\lambda\in\R$ and $v,w\in V$,
\begin{itemize}
\item[(i)] $p(v)\geq 0$;
\item[(ii)] $p(\lambda v)=|\lambda|p(v)$;
\item[(iii)] $p(v+w)\leq p(v)+p(w)$.
\end{itemize}
Notice that $p(v)=0$ for possibly nonzero vectors $v$. If, in addition, $p(v)=0$ implies $v=0$, then $p$ is called a {\em norm}\footnote{Usually denoted $\|\cdot\|$ rather than $p$.} on $V$, and $V$ is called a {\em normed vector space}.\index{Norm}\index{Normed space}
\end{definition}

\begin{remark}
A normed vector space is a TVS. More precisely, let $V$ be a vector space endowed with a norm $\|\cdot\|$. Then $d(v,w)=\|v-w\|$ defines a metric on $V$, whose induced topology turns $V$ into a TVS.
\end{remark}

\begin{lemma}\label{le:tvshell}
Let $V$ be a TVS. The following statements are equivalent.
\begin{itemize}
\item[(i)] $V$ is locally convex and \emph{pseudo--metrizable};\footnote{This means that $V$ admits a \emph{pseudo--metric} $d:V\times V\to\R$, i.e., a metric for which $d(v,w)=0$ with possibly $v\neq w$, that induces the same topology on $V$.}
\item[(ii)] $V$ is locally convex and first--countable;
\item[(iii)] The topology of $V$ is induced by a countable family of semi--norms $\{p_i\}_{i\in\N}$, i.e., $U\subset X$ is open if and only if for every $v\in U$ there exists $i_0\geq 1$ and $\varepsilon>0$ such that $\{w:p_i(v-w)<\varepsilon\mbox{ for all }i\leq i_0\}$ is a subset of $U$.
\end{itemize}
A {\em pre--Fr\'echet space}\index{Fr\'echet space!pre--Fr\'echet space} is a TVS whose topology satisfies any (hence all) of the above conditions and for which every unitary set $\{v\}$ is closed.\footnote{In particular, this implies that $V$ is Hausdorff, see Rudin \cite{rudinfa}.}
\end{lemma}

A proof of such equivalences can be found in any elementary textbook on TVSs, for instance \cite{rudinfa,schaefer,ys}. A pre--Fr\'echet space only lacks completeness to become a Fr\'echet space.

Since {\em a priori} there is no metric on a pre--Fr\'echet space $V$, the definition of Cauchy sequence is the following. A sequence $\{v_i\}_{i\in\N}$ in $V$ is a {\em Cauchy sequence} if for any open neighborhood $U$ of the origin there exists $i_0$ such that $v_i-v_j\in U$ for $i,j\geq i_0$. Analogously, such a sequence {\em converges} if there exists $v_\infty\in V$ such that for any open neighborhood $U$ of the origin there exists $i_0$ such that $v_i-v_\infty\in U$ for $i\geq i_0$. Clearly, a TVS is said to be {\em complete} if all Cauchy sequences converge.

\begin{definition}\label{def:frechet}
A {\em Fr\'echet space}\index{Fr\'echet space} is a complete pre--Fr\'echet space.
\end{definition}

\begin{remark}
Completeness in the non--metric sense above is equivalent to completeness with a {\em (translation) invariant} metric, i.e. a metric $d$ on $V$ such that $d(v+z,w+z)=d(v,w)$ for all $v,w,z\in V$. If $V$ admits a {\em complete} invariant metric $d$ that induces the above topology, then $V$ is a Fr\'echet space.
\end{remark}

It is important to mention that various nonequivalent definitions of TVS and Fr\'echet space can be found in the literature. However, our applications are mostly concerned with Banach spaces.

\begin{definition}\label{def:banach}
A {\em Banach space}\index{Banach space} is a (real) vector space $V$ endowed with a norm $\|\cdot\|:V\to\R$ that induces a complete metric on $V$. Notice that $V$ is automatically a TVS with the topology induced from such metric. A {\em Banachable space} is a TVS for which {\em there exists} a norm that turns it into a Banach space.
\end{definition}

Products and direct sums of Banach spaces are automatically Banach spaces, considering the natural norms. Recall that a subspace $W$ of a Banach space $V$ is a Banach space if and only if it is {\em closed} in $V$. Furthermore, the closure $\overline W$ of a subspace $W\subset V$ is a subspace of $V$, in particular, a Banach space. If $W$ is closed in $V$, then the quotient $V/W$ also has a Banach norm, given by the infimum of the norms of all elements of an equivalence class.

\begin{definition}\label{def:complement}
Let $V$ be a Banach space and $W\subset V$ a subspace. Then $W$ is \emph{complemented}\index{Banach space!complemented subspace} if there exists a closed subspace $W'\subset V$ such that $V=W\oplus W'$. Such a space $W'$ is called a (topological) \emph{complement}\index{Banach space!(topological) complement} of $W$ in $V$.
\end{definition}

Usually, the notion of being complemented is only considered for closed subspaces. Recall that in finite--dimensional vector spaces, all subspaces are automatically (closed) and complemented. However, in infinite dimension, there exist closed subspaces that are not complemented.

\begin{example}\label{ex:noncomplement}
The space $c_0$ of sequences in $\R$ that converge to $0$ is a closed and non complemented subspace of the space $\ell^\infty$ of all bounded sequences in $\R$.
\end{example}

Using the Hahn--Banach Theorem, one can verify the following sufficient condition for subspaces to be complemented.

\begin{lemma}\label{le:finitecomplemented}
Every finite--dimensional and finite--codimensional subspaces of a Banach space are complemented.
\end{lemma}

Before discussing further properties of complements, we recall some other basic definitions regarding operators.

\begin{definition}\label{def:proj}
An operator $p:V\to V$ is a \emph{projection}\index{Banach space!projection} onto a subspace $W$ if $\im p=W$ and $p(w)=w$, for all $w\in W$.
\end{definition}

\begin{remark}
The existence of a complement of a subspace $W$ is equivalent to existing a continuous linear projection $p$ onto $W$.
\end{remark}

\begin{lemma}\label{le:contbound}
Let $V_i$, $i=1,\ldots,r$ and $W$ be normed vector spaces and
\begin{equation*}
T:V_1\times\ldots\times V_r\la W
\end{equation*}
a multilinear\footnote{Recall that an operator $T:V_1\times\ldots\times V_r\to W$ is said to be a {\em multilinear form} if it is linear in each component. In particular, for $r=1$, a multilinear form is simply an operator $T:V\to W$.} form.\index{Bilinear form} Then the following are equivalent.
\begin{itemize}
\item[(i)] $T$ is continuous;
\item[(ii)] $T$ is continuous in the origin;
\item[(iii)] $T$ is bounded.
\end{itemize}
Henceforth the terms {\em continuous} and {\em bounded} referring to multilinear forms will be used indistinguishably.
\end{lemma}

\begin{definition}\label{def:multilinnorm}
If $T:V_1\times\ldots\times V_r\to W$ is multilinear, then the {\em (operator) norm}\index{Operator!norm}\index{Operator!continuous} of $T$ is given by
\begin{equation}\label{eq:normt}
\|T\|=\sup_{\substack{\|v_i\|=1 \\ i=1,\ldots,r}} \|T(v_1,\ldots,v_r)\|.
\end{equation}
Notice that $\|T\|<+\infty$ if and only if $T$ satisfies one (hence all) of the conditions in Lemma~\ref{le:contbound}. Expression \eqref{eq:normt} defines a norm on the vector spaces of {\em bounded} multilinear forms, turning them into normed vector spaces. The topology of such spaces will be henceforth considered to be the one induced by \eqref{eq:normt}.

Particular cases are spaces of bounded operators between normed vector spaces $T:V\to W$ and bounded bilinear forms $T:V_1\times V_2\to\R$, respectively denoted $\Lin(V,W)$ and $\Bilin(V_1,V_2)$. For simplicity, we also denote $\Lin(V)=\Lin(V,V)$ and $\Bilin(V)=\Bilin(V,V)$.
\end{definition}

\begin{remark}
There exists a natural isomorphism 
\begin{equation}\label{ident:bilin}
\Lin(V,W^*)\ni B\longmapsto\widetilde B\in\Bilin(V,W),
\end{equation}
where $\widetilde B(v,w)=B(v)(w)$. Henceforth, any such operator $B$ and bilinear form $\widetilde B$ will be identified and denoted by the same symbol.
\end{remark}

Clearly, $\Lin(V,\R)\cong V^*$ is the (topological) dual of $V$, consisting of continuous linear functionals on $V$.

\begin{remark}
If $W$ is a Banach space, the spaces of bounded multilinear forms $T:V_1\times\ldots\times V_r\to W$ are also Banach spaces, for $r\in\N$. In particular, $V^*$ is a Banach space.
\end{remark}

\begin{remark}
The above considerations about a Banach structure on the vector space of continuous multilinear forms between Banach spaces is no longer valid for more general TVSs. For instance, if $V$ and $W$ are Fr\'echet spaces, the vector space of continuous operators $T:V\to W$ may be not Fr\'echet.
\end{remark}

\begin{remark}
Obviously, a generalized {\em Cauchy--Schwartz inequality}\index{Cauchy--Schwartz inequality} holds,
\begin{equation}\label{eq:cauchyschwartz}
\|T(v_1,\dots,v_r)\|\leq\|T\|\|v_1\|\dots\|v_r\|,
\end{equation}
for all $v_i\in V_i$, $i=1,\dots,r$.
\end{remark}

Before proceeding, we prove two abstract lemmas that will be later used.

\begin{lemma}\label{le:grandetauskao}
Let $V$ be a normed vector space, $S$ a closed subspace of $V$ with finite--codimension and $\alpha:V\to\R$ a linear functional that vanishes identically on $S$. Then $\alpha$ is continuous, i.e., $\alpha\in V^*$.
\end{lemma}

\begin{proof}
Since $\alpha$ vanishes identically on $S$, it induces a functional in the quotient $\overline\alpha:V/S\to\R$, as in the following diagram.
\begin{equation*}
\xymatrix@+20pt{V\ar[d]_q\ar[r]^\alpha & \R \\ \dfrac{V}{S}\ar@(r,d)[ur]_{\overline\alpha} &}
\end{equation*}
Since $S$ is closed, the quotient map $q:V\to V/S$ that induces the usual quotient norm on $V/S$ is continuous. Moreover, since $\codim_V S<+\infty$, this space $V/S$ is finite--dimensional, hence the functional $\overline\alpha:V/S\to\R$ is continuous. Therefore, $\alpha=\overline\alpha\circ q$ is continuous.
\end{proof}

\begin{lemma}\label{le:cdreduction}
Let $X$ be a topological space, $V$ and $W$ Banach spaces and $T\in\Lin(V,W)$ a operator with closed image. Then each of the maps $f$ and $f_0$ in the diagram is continuous if and only if the other is continuous.
\begin{equation*}
\xymatrix{& X\ar@(dl,u)[ld]_{f_0}\ar@(dr,u)[rd]^f & \\
V\ar[rr]_T & & W}
\end{equation*}
\end{lemma}

\begin{proof}
Since $T(V)$ is closed, $T:V\to T(V)$ is a homeomorphism, with $T(V)$ endowed with the subspace topology. The result follows immediately from basic topology facts.
\end{proof}

\begin{definition}\label{def:topiso}
A continuous linear isomorphism of TVSs that has a continuous inverse, i.e., a linear homeomorphism, is called a {\em topological isomorphism}.\index{Topological isomorphism}\index{Operator!topological isomorphism}
\end{definition}

\begin{remark}\label{re:omt}
From the Open Mapping Theorem (or Banach--Schauder Theorem), every continuous isomorphism between Banach spaces is a topological isomorphism.
\end{remark}

\begin{definition}
An operator $T\in\Lin(V,W)$ between two normed vector spaces is an {\em isometric immersion}\index{Isometric immersion!linear} if $$\|Tv\|=\|v\|$$ for all $v\in V$. Such an operator is automatically injective and bounded, with $\|T\|=1$, see \eqref{eq:normt}. A bijective isometric immersion is called an {\em isometry},\index{Linear isometry}\index{Isometry!linear} whose inverse is also automatically an isometry. An isometry is clearly a topological isomorphism.
\end{definition}

\begin{lemma}
Complements to the same subspace are topologically isomorphic.
\end{lemma}

\begin{proof}
If $W'$ and $W''$ are both complements of $W$ and $p:W'\oplus W''\to W'$ is a projection, then $\ker p=W''$, and analogously with the other projection. Thus we have the sequence of (algebraic) isomorphisms $$W'\cong\frac{W\oplus W'}{W}\cong \frac{W\oplus W''}{W}\cong W''.$$ Continuity of the isomorphisms above is obvious. From Remark~\ref{re:omt}, it follows that such isomorphisms are homeomorphisms, concluding the proof.
\end{proof}

\begin{lemma}\label{le:comp1}
If $W$ is a complemented subspace of $V$, then all complements of $W$ are topologically isomorphic to $V/W$.
\end{lemma}

\begin{proof}
Consider the following exact sequence of vector spaces and operators
\begin{equation}\label{eq:esq}
0\la W\stackrel{i}{\longhookrightarrow} V\stackrel{q}{\la}\frac{V}{W}\la 0,
\end{equation}
where $i$ is the inclusion and $q$ the quotient operator. Since $V/W$ is a vector space (in particular, a free module), the above sequence splits. Thus $V\cong W\oplus V/W$. From Lemma~\ref{le:comp1}, all complements to $W$ are topologically isomorphic, and the proof is complete.
\end{proof}

\begin{definition}\label{def:hilbert}
A {\em Hilbert space}\index{Hilbert space} is a (real) vector space $H$ endowed with an inner product\footnote{i.e., a symmetric positive--definite bilinear form.} $\langle\,\cdot,\cdot\,\rangle:H\times H\to\R$ whose corresponding norm turns $H$ into a Banach space. A {\em Hilbertable space} is a TVS for which {\em there exists} an inner product that turns it into a Hilbert space.
\end{definition}

\begin{remark}
Analogously to Banach spaces, products, direct sums and quotients of Hilbert spaces are Hilbert spaces, with the natural inner products.
\end{remark}

For any subspace $W$ of a Hilbert space $H$, define $$W^\perp=\{v\in V:\langle v,w\rangle=0 \mbox{ for all } w\in W\},$$ which is always a {\em closed} subspace. If $W$ is closed, then $W^\perp$ is a complement of $W$ in the sense of Definition~\ref{def:complement}, called its \emph{orthogonal complement}\index{Hilbert space!orthogonal complement}. Hence all closed subspaces of a Hilbert space are complemented, which is obviously not true for general Banach spaces. Indeed, if all closed subspaces of a Banach space $V$ are complemented, then $V$ is Hilbertable, see Brezis \cite{brezis}.

\begin{remark}
A subspace $W$ of a Hilbert space $H$ is dense if and only if $W^\perp=\{0\}$. Consequently, $(W^{\perp})^\perp=\overline{W}$.
\end{remark}

The {\em orthogonal} projection onto a subspace $W$ is a projection in the sense of Definition~\ref{def:proj} that will be denoted $$p_W:V\longrightarrow W.$$ If $v_1,v_2\in H$ are such that $\langle v_1,v_2\rangle=0$, then $$\|v_1+v_2\|^2=\|v_1\|^2+\|v_2\|^2.$$ In particular, this implies that $p_W$ has unitary norm \eqref{eq:normt}. Moreover, $p_W(v)$ is the global minimum of the function $H\owns x\mapsto d(x,v)\in\R$.

Using the inner product, each vector $v\in H$ induces a bounded functional
\begin{equation}\label{eq:hilbertdual}
H\owns v\longmapsto\langle v,\cdot\,\rangle\in H^*,
\end{equation}
which is a linear isometric immersion as a consequence of Cauchy--Schwartz inequality. A converse is given by the following well--known result.

\begin{rieszthm}\label{thm:riesz}\index{Theorem!Riesz Representation}
If $H$ is a Hilbert space, then \eqref{eq:hilbertdual} is an isometry.
\end{rieszthm}

Consequently, the dual $H^*$ of a Hilbert space $H$ is canonically identified (isometrically) with $H$, and henceforth we will implicitly use
\begin{equation}\label{ident:dual}
H\cong H^*
\end{equation}
for any Hilbert spaces $H$.

\begin{definition}\label{def:represents}
Let $H_1$ and $H_2$ be Hilbert spaces, and $B:H_1\times H_2\to\R$ a bilinear form. The unique operator $T_B:H_1\to H_2$ such that
\begin{equation}\label{eq:brepres}
B(v,w)=\langle T_Bv,w\rangle
\end{equation}
for all $v\in H_1$ and $w\in H_2$ is called the operator that {\em represents}\index{Operator!that represents a bilinear form}\index{Bilinear form!represented by an operator} $B$ (in terms of $\langle\cdot,\cdot\rangle$).
\end{definition}

\begin{remark}\label{re:isombilinlin}
The operator $\Bilin(H_1,H_2)\ni B\mapsto T_B\in\Lin(H_1,H_2)$ is an isometry, considering the norms given by \eqref{eq:normt} in Definition~\ref{def:multilinnorm}.
\end{remark}

The Riesz Representation Theorem~\ref{thm:riesz} also allows to associate to each operator $T\in\Lin(H)$ its {\em adjoint} operator\index{Operator!adjoint} $T^*\in\Lin(H)$, uniquely defined by
\begin{equation}\label{eq:derrr}
\langle Tv,w\rangle=\langle v,T^*w\rangle, \quad v,w\in H.
\end{equation}
Namely, for each $w\in H$, the functional $\langle T\,\cdot,w\rangle$ corresponds by \eqref{eq:hilbertdual} to a unique $T^*w$, such that \eqref{eq:derrr} holds. From the above property, it is easy to derive several elementary consequences, among which the following important relation between kernel and image of an operator and its adjoint,
\begin{equation}\label{eq:ateminhamaesabe}
\ker T^*=(\im T)^\perp.
\end{equation}
Notice also that an operator and its adjoint have the same norm \eqref{eq:normt}.

\begin{definition}
If an operator $T\in\Lin(H)$ coincides with its adjoint $T^*=T$, then it is said to be {\em self--adjoint}.\index{Operator!self--adjoint}
\end{definition}

\begin{remark}
Notice that if $T\in\Lin(H)$ is self--adjoint, then \eqref{eq:ateminhamaesabe} reads $\ker T=(\im T)^\perp$.
\end{remark}

\begin{lemma}\label{le:symselfadjoint}
If $B\in\Bilin(H)$ is a symmetric bilinear form, then the unique operator $T_B$ that represents it with respect to $\langle\cdot,\cdot\rangle$ is self--adjoint.
\end{lemma}

\begin{proof}
This follows immediately by comparing \eqref{eq:brepres} and \eqref{eq:derrr}, using that $B$ is symmetric.
\end{proof}

\begin{definition}\label{def:nondegenerate}
For any bilinear form $B\in\Bilin(H)$, $$\ker B=\ker T_B.$$ In case this space is trivial, $B$ is called {\em nondegenerate}.\index{Bilinear form!nondegenerate}
\end{definition}

Observe that $B\in\Bilin(H,H)$ on a Hilbert space $H$ is nondegenerate if and only if the operator
\begin{equation}\label{eq:BHBHBHBH}
H\ni x\longmapsto B(x,\cdot\,)\in H^*
\end{equation}
is injective. Equivalently, $B$ is nondegenerate if the operator that represents $B$ with respect to the Hilbert space inner product of $H$ is injective.

\begin{definition}\label{def:stronglynondegenerate}
A continuous bilinear form $B\in\Bilin(H)$ is called {\em strongly nondegenerate}\index{Bilinear form!strongly nondegenerate} if the operator \eqref{eq:BHBHBHBH} is an isomorphism, or equivalently, if the operator that represents $B$ with respect to the Hilbert space inner product of $H$ is an isomorphism.
\end{definition}

\begin{lemma}\label{le:polarization}
Suppose $B\in\Bilin(H)$ is a symmetric bilinear form. Then the following {\em polarization formula}\index{Polarization formula} holds
\begin{equation}\label{eq:polarization}
B(v,w)=\tfrac12\big(B(v+w,v+w)-B(v,v)-B(w,w)\big),\quad v,w\in H.
\end{equation}
\end{lemma}

\begin{proof}
Follows at once by expanding the right--hand side of \eqref{eq:polarization} using bilinearity of $B$.
\end{proof}

\section{Compact and Fredholm operators}
\label{sec:compactfredholm}

A linear endomorphism of a finite--dimensional vector space is surjective if and only if it is injective. This is clearly false for operators between infinite--dimensional spaces. In this section, we briefly recall that such important property still holds (see Lemma~\ref{le:fred0}) for a special class of operators between Banach spaces, that include {\em sufficiently small} perturbations of isomorphisms, namely, Fredholm operators. To this aim, we also recall the concept of compact operator.

Furthermore, we state a few well--known results on stability of this property in the space of continuous operators, without giving proofs. In addition, although some of the following constructions can be identically done in the case of locally convex TVSs, we will restrict our attention to Banach spaces. In Section~\ref{sec:infinitedimmnflds}, this notion of Fredholmness will also be extended to a nonlinear context of Banach manifolds, see Definition~\ref{def:nonlinfred}. Complete proofs of most results stated in this section may be found in any basic functional analysis textbook, such as \cite{brezis,fabian,rudinfa,ys}.

\begin{lemma}\label{le:cpcop}
Let $K\in\Lin(V,W)$ be an operator between Banach spaces. The following are equivalent.
\begin{itemize}
\item[(i)] The image by $K(B_V)$ of the unitary ball $B_V$ of $V$ (centered in the origin) is {\em relatively compact}\footnote{A {\em relatively compact}\index{Relatively compact} subset of a topological space is a subset whose closure is compact. Moreover, for subsets $A$ of a {\em complete} metric space, such as Banach space, being relatively compact is equivalent to being {\em totally limited}, i.e. for all $\varepsilon>0$ there exists a finite cover of $A$ of subsets whose diameter is less then $\varepsilon$.} in $W$;
\item[(ii)] If $A\subset V$ is any limited subset, then $K(A)\subset W$ is relatively compact;
\item[(iii)] For any sequence $\{v_n\}_{n\in\N}$ in $V$, the sequence $\{Kv_n\}_{n\in\N}$ in $W$ admits a convergent subsequence.
\end{itemize}
\end{lemma}

\begin{definition}\label{def:cpcop}
An operator $K\in\Lin(V,W)$ between Banach spaces is a {\em compact operator}\index{Compact operator}\index{Operator!compact} if any (hence all) of the conditions in Lemma~\ref{le:cpcop} is satisfied. The vector space of all compact operators $K\in\Lin(V,W)$ will be denoted $\cpc(V,W)$, and for simplicity $\cpc(V)=\cpc(V,V)$.
\end{definition}

We now give a central result of compact operators, whose proof can be found in \cite{fabian,kreyszig}.

\begin{proposition}\label{prop:propcpcop}
Let $V,W$ and $Z$ be Banach spaces. Then $\cpc(V,W)$ is a closed subspace of $\Lin(V,W)$, hence a Banach space. Let $T_1\in\Lin(V,W)$ and $T_2\in\Lin(W,Z)$. Then $T_2T_1\in\cpc(V,Z)$ if either $T_1$ or $T_2$ is compact. In particular, $\cpc(V)$ is an {\em ideal} of $\Lin(V)$ under the composition product.
\end{proposition}

\begin{remark}\label{re:cpcinductive}
From a simple inductive argument, a finite composition of bounded operators is compact provided that at least one of the factors is compact.
\end{remark}

\begin{definition}\label{def:fredholmop}
An operator $T\in\Lin(V,W)$ between Banach spaces is a {\em Fredholm operator}\index{Fredholm!operator}\index{Operator!Fredholm} if the subspaces $\ker T$ and $\coker T$ are finite--dimensional.\footnote{Recall that $\coker T=W/\im T$. Hence $\coker T$ has finite dimension if and only if $\im T$ has finite codimension, since $\dim\coker T=\codim\im T$.} The {\em Fredholm index}\index{Fredholm!operator!index}\index{Operator!Fredholm!index} of $T$ is then defined by
\begin{equation}\label{eq:indt}
\ind(T)=\dim\ker T-\dim\coker T.
\end{equation}
\end{definition}

\begin{example}
Topological isomorphisms\footnote{See Definition~\ref{def:topiso}.} are clearly Fredholm operators of index zero.
\end{example}

\begin{remark}\label{re:fredcomp}
Applying Lemma~\ref{le:finitecomplemented}, it follows that if $T\in\Lin(V,W)$ is Fredholm, both $\ker T$ and $\im T$ are complemented subspaces.
\end{remark}

\begin{lemma}\label{le:fred0}
If $T\in\Lin(V,W)$ is a Fredholm operator of index zero, it is injective if and only if it is surjective.
\end{lemma}

\begin{proof}
This is an obvious consequence of the definition of index \eqref{eq:indt}, since if $$\ind(T)=\dim\ker T-\codim\im T$$ is zero, the kernel of $T$ is trivial if and only if the image of $T$ is the whole $W$, hence $T$ is injective if and only if it is surjective.
\end{proof}

In this sense, the index of a Fredholm operator measures the difference between its non injectivity and non surjectivity. Let us remark the case of self--adjoint operators.

\begin{lemma}\label{le:selfadjointzero}
Let $H$ be a Hilbert space and $T\in\Lin(H)$ a self--adjoint Fredholm operator. Then the $\ind(T)=0$.
\end{lemma}

\begin{proof}
If $T$ is self--adjoint,
\begin{eqnarray*}
\dim\ker T &=& \dim\ker T^*\\
&\stackrel{\eqref{eq:ateminhamaesabe}}{=}& \dim (\im T)^\perp \\
&=& \codim\im T\\
&=& \dim\coker T.\qedhere
\end{eqnarray*}
\end{proof}

Another important property of Fredholm operators is that their image is always a closed subspace of the counter domain, as proved in the next result.

\begin{proposition}\label{prop:fredclosed}
Let $T\in\Lin(V,W)$ be a Fredholm operator. Then the image $\im T$ is closed.
\end{proposition}

\begin{proof}
From Remark~\ref{re:fredcomp}, there exists a finite--dimensional complement $S\subset W$ of $\im T$. Consider the operator
\begin{eqnarray*}
\tilde T:V\oplus S &\la& W \\
(v,s) &\longmapsto & (Tv,s).
\end{eqnarray*}
This operator is clearly surjective, hence open, by the Open Mapping Theorem. Therefore it is a quotient map, i.e. $X$ is open (respectively, closed) in $W$ if and only if $\tilde T^{-1}(X)$ is open (respectively, closed) in $V\oplus S$. Since $\tilde T^{-1}(T(V))=V\oplus \{0\}$ is clearly closed, also $T(V)$ is closed.
\end{proof}

Fredholmness of isomorphisms are {\em stable} in many different ways. For instance, sufficiently small perturbations of isomorphisms with respect to the norm \eqref{eq:normt} are still Fredholm, preserving also the index equal to zero. Another important and well--known stability result is the following, whose proof can be found in \cite{fabian,ys}.

\begin{proposition}\label{prop:fredholmisocomp}
Let $V$ be a Banach space, $T\in\Lin(V)$ a topological isomorphism and $K\in\cpc(V)$ a compact operator. Then $T+K\in\Lin(V)$ is a Fredholm operator and $$\ind(T+K)=\ind(T)=0.$$
\end{proposition}

\begin{remark}
In fact, it is possible to state a more general version of the above result as follows. If $T\in\Lin(V)$ is a Fredholm operator and $K\in\cpc(V)$ is compact, then $T+K\in\Lin(V)$ is a Fredholm operator with the same index of $T$.
\end{remark}

Stability of Fredholmness may be stated in a stronger sense as follows.

\begin{proposition}\label{prop:fredstable}
Let $V$ and $W$ be Banach spaces. The subset of $\Lin(V,W)$ formed by Fredholm operators is open in the topology induced by \eqref{eq:normt}. More precisely, given a Fredholm operator $T_0\in\Lin(V,W)$, there exists $\varepsilon>0$ such that if $T\in\Lin(V,W)$ satisfies $\|T_0-T\|<\varepsilon$, then $T$ is also Fredholm and $\ind(T)=\ind(T_0)$.
\end{proposition}

\begin{remark}
It is easy to verify that the subset of Fredholm operators {\em of given index} between two Banach spaces is a connected component of the above set of Fredholm operators. This follows from the local continuity of the Fredholm index given by Proposition \ref{prop:fredstable}. Thus, for each index, the subset of Fredholm operators of that index form an open subset of the space of all continuous operators between these Banach spaces.
\end{remark}

Let us finish this section stating an another result on composition of Fredholm operators. Recall that proofs of most results in this section may be found in \cite{fabian,rudinfa,tausk,ys}.

\begin{proposition}
Let $V,W$ and $Z$ be Banach spaces. If $T\in\Lin(V,W)$ and $S\in\Lin(W,Z)$ are Fredholm operators, then $ST\in\Lin(V,Z)$ is also Fredholm, and $$\ind(ST)=\ind(S)+\ind(T).$$
\end{proposition}

\section{Calculus on Banach spaces}
\label{sec:calcbanspaces}

In this section we briefly recall some basic aspects of differential calculus on Banach spaces. This will be the {\em linear} basis to develop calculus on Banach {\em manifolds}, in Section~\ref{sec:infinitedimmnflds}. Most concepts are immediate generalizations of their finite--dimensional counterparts, hence the correspondent discussion will be relatively short.

\begin{definition}\label{def:diffBanach}
Let $V$ and $W$ be Banach spaces, $U\subset V$ an open subset and $f:U\to W$ a map. It is said that $f$ is {\em differentiable}\index{Differentiable map} at a point $x\in U$ if there exists a continuous operator $T:V\to W$ such that the map $r$ defined in $$f(x+h)=f(x)+T(h)+r(h)$$ satisfies $\lim_{h\to0}\frac{r(h)}{\Vert h\Vert}=0$.
\end{definition}

\begin{remark}
Let $f$ be differentiable at $x\in V$. It is easy to verify that $$T(v)=\lim_{t\to0}\frac{f(x+tv)-f(x)}t,$$ for all $v\in V$. Hence $T$ is unique when it exists, and thus $T$ is called the {\em differential}\index{Derivative of a map} of $f$ at $x$, denoted $T=\dd f(x)$.
\end{remark}

\begin{remark}\label{re:diffinBanachable}
The statement {\em $f$ is differentiable at $x$ and $\dd f(x)=T$} is clearly invariant under substitution of the norms in $V$ and $W$ by equivalent ones. In particular, differentiability is a well--defined notion for Banachable spaces.
\end{remark}

\begin{definition}\label{def:differentialBanach}
If $f$ is differentiable at every point of $U$, we say that $f$ is {\em differentiable in $U$} and in such case, it is possible to consider the map
\begin{eqnarray*}
\dd f:U &\la& \Lin(V,W) \\
x &\longmapsto &\dd f(x),
\end{eqnarray*}
called the {\em differential}, or {\em derivative}, of $f$.
\end{definition}

Since $\Lin(V,W)$ is again a Banach space, one may ask whether $\dd f$ is a differentiable map. If it is, we obtain a {\em second (ordinary) derivative} $$\dd^2f=\dd(\dd f):U\la\Lin(V,\Lin(V,W)).$$ In general, if $f$ can be differentiated $k$ times, we can consider its {\em $k^{\mbox{\tiny th}}$ (ordinary) derivative},\index{Derivative of a map!higher order} defined recursively by $\dd^kf=\dd\big(\dd^{k-1}f)$, which is a map of the form $$\dd^kf:U\la\underbrace{\Lin(V,\Lin(V,\cdots,\Lin}_{\text{$k$ $\Lin$'s}}(V,W))\cdots).$$ The counter domain of $\dd^kf$ may be identified with a simpler space. More precisely, there is an isometry of this space with the Banach space of all continuous $k$--multilinear forms $B:V\times\ldots\times V\to W$.

\begin{definition}\label{def:ckbanach}
Analogously to the finite--dimensional case, a map $f:U\subset V\to W$ that is $k$ times differentiable (in the sense of Definition~\ref{def:diffBanach}) and has continuous $k^{\mbox{\tiny th}}$ derivative $\dd^kf$ is said to be of {\em class $C^k$}.\index{$C^k$ map} In addition, if $f$ is of class $C^k$ for all $k\in\N$, then $f$ is {\em of class $C^\infty$}.
\end{definition}

A general theory of differentiable calculus on Banach spaces can be developed analogously to the finite--dimensional case with the above basis. More precisely, extended versions of elementary results as the chain rule, the mean value inequality, Schwartz's Theorem (on the symmetry of the higher order derivatives), the Inverse Function Theorem and others can be easily proved. For a detailed exposition on this subject, we refer to Lang \cite{lang}.

Some further aspects of this theory will appear in Section~\ref{sec:infinitedimmnflds}, in the context of Banach manifolds. Namely, we will define {\em critical} and {\em regular} points and explore the classic concepts of {\em degeneracy} and {\em transversality} in this context.

We end this section with a technical analytical result, namely a weak differentiation principle, that gives a practical method for proving differentiability of maps between Banach spaces in concrete examples.

\begin{definition}\label{def:separapontos}
Let $Y$ be a Banach space. A {\em separating family}\index{Separating family} for $Y$ is a set $\mathcal F$ of continuous operators $\lambda:Y\to Z_\lambda$, with $Z_\lambda$ a Banach space, such that for each non zero $v\in Y$ there exists $\lambda\in\mathcal F$ with $\lambda(y)\ne0$.
\end{definition}

\begin{lemma}\label{le:weakdiffprinc}
Let $X$ and $Y$ be Banach spaces, $f:U\to Y$ a map defined on an open subset $U\subset X$ and $\mathcal F$ a separating family for $Y$. If there exists a continuous map $g:U\to\Lin(X,Y)$ such that for every $x\in U$, $v\in X$, $\lambda\in\mathcal F$, the directional derivative $\frac{\partial(\lambda\circ f)}{\partial v}(x)$ exists and equals $\lambda(g(x)v)$, then $f$ is $C^1$ and $\dd f=g$.
\end{lemma}

\begin{proof}
Let $x\in U$ be fixed and consider the map $r$ defined in $$f(x+h)=f(x)+g(x)h+r(h).$$ From Definition~\ref{def:diffBanach}, it suffices to prove that $\lim_{h\to 0}\frac{r(h)}{\|h\|}=0$. For a sufficiently small $h$, the closed line segment $[x,x+h]$ is contained in $U$. It follows from the hypotheses on $\mathcal F$ that for each $\lambda\in\mathcal F$ the curve $$[0,1]\ni t\longmapsto(\lambda\circ f)(x+th)$$ is differentiable, with $$\frac{\dd}{\dd t}(\lambda\circ f)(x+th)=\lambda(g(x+th)h).$$

From the Fundamental Theorem of Calculus\footnote{This can be rigorously done only with a theory of integration for Banach space valued curves. One possibility is to use the {\em Bochner integral} (see \cite{ys}), however there are also simpler approaches in this case. For instance, one can use the notion of {\em weak integration}.},
\begin{eqnarray*}
\lambda(r(h)) &=& \int_0^1\frac{\dd}{\dd t}(\lambda\circ f)(x+th)\;\dd t-\lambda(g(x)h) \\
&=& \lambda\left(\int_0^1g(x+th)h\;\dd t-g(x)h\right).
\end{eqnarray*}
Since $\mathcal F$ separates points in $Y$, it follows that
\begin{eqnarray*}
\|r(h)\| &=& \left\|\int_0^1g(x+th)h\;\dd t-g(x)h\right\| \\
&=& \left\|\left(\int_0^1 g(x+th)-g(x)\;\dd t\right)h\right\| \\
&\le & \left(\sup_{t\in[0,1]}\|g(x+th)-g(x)\|\right)\|h\|.
\end{eqnarray*}
From continuity of $g$, $\lim_{h\to 0}\frac{r(h)}{\|h\|}=0$, which concludes the proof.
\end{proof}

\section{Function spaces}
\label{sec:funcspaces}

In this section, we recall the definitions of several classic function spaces of various regularities. The main purpose of this part of the text is to establish notations and make a few conventions. For a more detailed treatment of this subject, we refer to \cite{brezis,pugh,rudinfa,tausk,ys}.

In addition to basic notions of $C^k$ and $L^p$ spaces, we give a slightly longer description of Sobolev $W^{k,p}$ and $H^k$ spaces, however not following the usual approach using distributional derivatives. This is only possible since the domain of the considered maps will always be one--dimensional, and an equivalent approach using absolutely continuous maps is hence feasible. Finally, a few celebrated results on compactness or density of some immersions between the mentioned spaces are recalled, without proofs.

\begin{definition}\label{def:czero}
Let $C_b([a,b],\R^m)$ denote the vector space of {\em bounded}\footnote{The subindex $b$ stands for {\em bounded}, and has no relation with the upper end of the real interval $[a,b]$.} maps $f:[a,b]\to\R^m$, with the {\em uniform convergence norm}\index{Uniform convergence norm}
\begin{equation}
\normc{0}{f}=\sup_{x\in [a,b]} \|f(x)\|,
\end{equation}
where $\|\cdot\|$ denotes an arbitrary norm on $\R^m$. This is clearly a Banach space, see \cite{pugh} for a proof.
\end{definition}

\begin{lemma}\label{le:convderivadas}
Let $\{f_i:[a,b]\to\R^m\}_{i\in\N}$ be a sequence of $C^k$ maps that converges locally uniformly to a map $f_\infty\in C^0([a,b],\R^m)$, such that also the first $k$ derivatives $\{f_i^{(j)}\}_{i\in\N}$ converge locally uniformly to $f^j_\infty\in C^0([a,b],\R^m)$, for $1\leq j\leq k$. Then $f_\infty\in C^k([a,b],\R^m)$ and $f_\infty^{(j)}=f^j_\infty$.
\end{lemma}

A proof of this result is elementary and can be found, for instance in \cite{pugh,rudin}. A more sophisticated version of this lemma concerning $C^k$ sections of vector bundles will be given in Lemma~\ref{le:climage}.

\begin{definition}\label{def:ck}
The subset $C^0([a,b],\R^m)$ of continuous maps is a closed subspace of $C_b([a,b],\R^m)$, hence a Banach space. For each positive integer $k$, define $C^k([a,b],\R^m)$\index{$C^k$ map} to be the vector space of maps $f:[a,b]\to\R^m$ of class $C^k$. From Lemma~\ref{le:convderivadas}, the injective operator
\begin{eqnarray}\label{eq:icksomakc0}
C^k([a,b],\R^m) &\longhookrightarrow& \bigoplus_{j=0}^{k}C^0([a,b],\R^m)\\
f &\longmapsto & \left(f,f',\dots,f^{(k)}\right)\nonumber
\end{eqnarray}
has closed image and hence induces a TVS structure on $C^k([a,b],\R^m)$ making it a Banachable space.

An explicit Banach norm for this space is the so--called {\em $C^k$--norm}\index{$C^k$ norm}\index{Norm!$C^k$}\footnote{With respect to the norm $\|\cdot\|$ in $\R^m$.}
\begin{equation}\label{eq:normck}
\normc{k}{f}=\max_{0\leq j\leq k}\left\{\sup_{x\in [a,b]}\|f^{(j)}(x)\|\right\}=\max_{0\leq j\leq k}\normc{0}{f^{(j)}},
\end{equation}
where $\|\cdot\|$ denotes an arbitrary norm on $\R^m$. Endowing the counter domain with such norm, \eqref{eq:icksomakc0} is an isometric immersion, and $C^k([a,b],\R^m)$ endowed with \eqref{eq:normck} is a Banach space.
\end{definition}

\begin{remark}
Another norm on $C^k([a,b],\R^m)$ equivalent to \eqref{eq:normck} is
\begin{equation*}
\|f\|=\sum_{j=0}^k \normc{0}{f^{(j)}}.
\end{equation*}
\end{remark}

\begin{remark}
Notice that for any fixed $t_0\in [a,b]$ there exists a topological isomorphism
\begin{eqnarray*}
C^k([a,b],\R^m) &\la & (\R^m)^k\oplus C^0([a,b],\R^m) \\
f &\longmapsto& \left(f(t_0),f'(t_0),\dots,f^{(k-1)}(t_0),f^{(k)}\right)
\end{eqnarray*}
that induces other norms in $C^k([a,b],\R^m)$ equivalent to \eqref{eq:normck}, for instance
\begin{equation}\label{eq:normckeq}
\|f\|=\max\big\{\|f(t_0)\|,\|f'(t_0)\|,\dots,\|f^{(k-1)}(t_0)\|,\normc{0}{f^{(k)}}\big\}.
\end{equation}
\end{remark}

\begin{remark}\label{re:ckck-1}
For every positive integer $k$, the inclusion map $$C^k([a,b],\R^m)\longhookrightarrow C^{k-1}([a,b],\R^m)$$ is a compact operator, in particular continuous (see Definition~\ref{def:cpcop}). Inductively, from Proposition~\ref{prop:propcpcop}, for all $1\le j\le k$ the inclusions $$C^k([a,b],\R^m)\longhookrightarrow C^{k-j}([a,b],\R^m)$$ are also compact.
\end{remark}

\begin{remark}\label{re:cinftyfrechet}
Consider the countable intersection $$C^\infty([a,b],\R^m)=\bigcap_{k\in\N}C^k([a,b],\R^m).$$ Maps in this subspace are said to be {\em smooth}, or of class $C^\infty$. Every attempt to endow this space with a Banach space norm similar to norms in $C^k([a,b],\R^m)$ such as \eqref{eq:normck} or \eqref{eq:normckeq} trivially fail. In fact, $C^\infty([a,b],\R^m)$ is not a Banach space, but only a Fr\'echet space, see Definition~\ref{def:frechet}. The sequence of norms $\{\|\cdot\|_{C^k}\}_{k\in\N}$, given by \eqref{eq:normck}, gives a countable family of semi--norms that induce the topology of $C^\infty([a,b],\R^m)$, see Lemma~\ref{le:tvshell}. In addition, there are classic results on the density of $C^\infty([a,b],\R^m)$ in other spaces of functions with less regularity, see the Stone--Weierstrass Theorem~\ref{thm:stoneweierstrass}, Proposition~\ref{prop:cinftylp} and Corollaries~\ref{cor:cinftywkp} and~\ref{cor:cinftyhk}.
\end{remark}

An important subspace of $C^\infty([a,b],\R^m)$ is $C^\infty_c(\,]a,b[,\R^m)$, formed by maps that have compact support contained in $\,]a,b[$, see \eqref{eq:support}. This subspace will be used for some variational lemmas in Section~\ref{sec:afewmorelemmas}.

We now mention an important class of spaces of functions that are basic in analysis. These were first introduced by Riesz in the beginning of the twentieth century. Recall that in this text, measurability and integrals are in the {\em Lebesgue} sense, and by \emph{for almost all} (or \emph{almost everywhere} and {\em almost always}) we mean outside a set of Lebesgue measure zero. This handy convention will be used throughout the text.

\begin{definition}\label{def:lp}
Let $f:[a,b]\to\R^m$ be a measurable map. For every $p\in\left[1,+\infty\right[\;$, define the {\em $L^p$--norm}\index{Norm!$L^p$}\index{$L^p$ norm}\footnote{With respect to the norm $\|\cdot\|$ in $\R^m$.}
\begin{equation}\label{eq:normlp}
\norml{p}{f}=\left(\int_a^b\|f(t)\|^p\,\dd t\right)^\frac 1p\in[0,+\infty],
\end{equation}
where $\|\cdot\|$ denotes an arbitrary norm on $\R^m$. Maps $f:[a,b]\to\R^m$ with finite $L^p$--norm are called {\em $L^p$ maps}.\index{$L^p$ map}

The {\em Minkowski inequality}\index{Minkowski inequality} states that for every measurable maps $f,g:[a,b]\to\R^m$
$$\norml{p}{f+g}\leq\norml{p}{f}+\norml{p}{g}.$$
It is also easy to see that $\norml{p}{f}=0$ if and only if $f(t)=0$ for almost all $t\in[a,b]$. Hence the set of all measurable maps $f:[a,b]\to\R^m$ with $\norml{p}{f}<+\infty$ is a subspace of the vector space of all $\R^m$--valued maps on $[a,b]$, endowed with the semi--norm \eqref{eq:normlp}. Consider the induced norm, also denoted $\norml{p}{\cdot}$, on the vector space $L^p([a,b],\R^m)$ defined as the quotient\footnote{This means that an element of $L^p([a,b],\R^m)$ is an equivalence class of $L^p$ functions, where the equivalence relation $\sim$ is $f\sim g\Leftrightarrow\text{$f=g$ almost everywhere}$. Nevertheless, the elements of $L^p([a,b],\R^m)$ are usually thought as functions, with a subtle abuse of notation.} by such subspace. Endowed with such norm, $L^p([a,b],\R^m)$ is a Banach space. For a proof see for instance \cite{kreyszig,ys}.

Notice that the topology on $L^p([a,b],\R^m)$ does not depend on the choice of the norm $\|\cdot\|$ on $\R^m$. In addition, if this norm is induced by an inner product $\langle\cdot,\cdot\rangle$ on $\R^m$ and if $p=2$, then the $L^p$--norm is induced by the {\em $L^2$--inner product}
\begin{equation}\label{eq:productldois}
\langle f,g\rangle_{L^2}=\int_a^b\big\langle f(t),g(t)\big\rangle\,\dd t.
\end{equation}
Thus $L^2([a,b],\R^m)$ endowed with $\langle\cdot,\cdot\rangle_{L^2}$ is a Hilbert space.
\end{definition}

\begin{example}\label{thm:multC0L2}
If $B:\R^m\times\R^n\to\R^p$ is a bilinear form, then 
\begin{eqnarray*}
&\widehat B:C^0([a,b],\R^m)\times L^2([a,b],\R^n)\longrightarrow L^2([a,b],\R^p)&\\
&\widehat B(f,g)(t)=B(f(t),g(t)), \quad t\in[a,b]&
\end{eqnarray*}
is bilinear and continuous. More precisely, $$\big\Vert\widehat B(f,g)\big\Vert_{L^2}^2=\int_a^bB(f(t),g(t))^2\,\dd t\le\Vert B\Vert^2\Vert f\Vert^2_{C^0}\int_a^b\big\Vert g(t)\big\Vert^2\,\dd t,$$ and therefore $\big\Vert\widehat B\big\Vert\le\Vert B\Vert$. In particular, we shall use the continuity of
\begin{eqnarray}
&\widehat B:C^0([a,b],\Lin(\R^m,\R^n))\times L^2([a,b],\R^m)\la L^2([a,b],\R^n)&\label{eq:compositeC0L2L2} \\
&\widehat B(T,f)(t)=T(t)f(t),\quad t\in[a,b].&\nonumber
\end{eqnarray}
\end{example}

\begin{definition}\label{def:abscont}
A map $f:[a,b]\to\R^m$ is said to be {\em absolutely continuous}\index{Absolutely continuous map} if for every $\varepsilon>0$ there exists $\delta>0$ such that if $\,]x_i,y_i[,i=1,\dots,r$ are disjoint open intervals contained in $[a,b]$ with $\sum_{i=1}^r y_i-x_i<\delta$ then $$\sum_{i=1}^r\|f(y_i)-f(x_i)\|<\varepsilon.$$
\end{definition}

The notion of absolutely continuous map is characterized in the following result, whose proof can be found in Rudin \cite{rudin}.

\begin{proposition}\label{prop:abscontcharact}
A map $f:[a,b]\to\R^m$ is absolutely continuous if and only if the following conditions hold.
\begin{itemize}
\item[(i)] the derivative $$f'(t)=\lim\limits_{h\to0}\frac{f(t+h)-f(t)}h$$ exists for almost every $t\in [a,b]$;
\item[(ii)] the (almost everywhere defined) map $f':[a,b]\to\R^m$ is integrable;
\item[(iii)] for all $t\in[a,b]$, $$f(t)=f(a)+\int_a^t f'(s)\;\dd s.$$
\end{itemize}
Moreover, if $\phi:[a,b]\to\R^m$ is an integrable map, then the map $f:[a,b]\to\R^m$ defined by $f(t)=\int_a^t\phi(s)\;\dd s$ is absolutely continuous and $f'=\phi$ almost everywhere.
\end{proposition}

\begin{definition}\label{def:wkp}
For every positive integer $k$ and $p\in [1,+\infty[\;$, define\index{$W^{k,p}$ map}\index{Sobolev class $W^{k,p}$}
\begin{equation*}
W^{k,p}([a,b],\R^m)=\left\{f\in C^{k-1}([a,b],\R^m):\begin{array}{c} f^{(k-1)} \mbox{ absolutely continuous}\\ \mbox{and }f^{(k)}\in L^p([a,b],\R^m)\end{array}\right\}.
\end{equation*}
In particular, $W^{1,1}([a,b],\R^m)$ is the space of all absolutely continuous maps, and $W^{0,p}([a,b],\R^m)=L^p([a,b],\R^m)$. An adapted version of Lemma~\ref{le:convderivadas} guarantees that the injective operator
\begin{eqnarray}\label{eq:wkpc0lp}
W^{k,p}([a,b],\R^m) &\longhookrightarrow& \bigoplus_{j=0}^{k}C^0([a,b],\R^m)\oplus L^p([a,b],\R^m) \\
f &\longmapsto & \left(f,f',\dots,f^{(k)}\right)\nonumber
\end{eqnarray}
has closed image and hence induces a TVS structure on $W^{k,p}([a,b],\R^m)$ making it a Banachable space. An explicit Banach norm for this space is, for instance,\index{$W^{k,p}$ norm}\index{Norm!$W^{k,p}$}
\begin{equation}\label{eq:normwkp}
\|f\|_{W^{k,p}}=\normc{k-1}{f}+\norml{p}{f^{(k)}}.
\end{equation}
\end{definition}

\begin{remark}\label{re:wkpderivativeae}
Notice that from Proposition~\ref{prop:abscontcharact}, if $f\in W^{k,p}([a,b],\R^m)$, then $f^{(k)}$ is defined for almost all $t\in [a,b]$. This defines an element of $L^p([a,b],\R^m)$, since elements of this space are equivalence classes defined by the equivalence relation of being equal almost everywhere in $[a,b]$. Henceforth, we will omit the term {\em almost everywhere}, and by any equality or definition involving the $k^{\mbox{\tiny th}}$ derivative of a map in $W^{k,p}([a,b],\R^m)$ we implicitly assume that it is to be thought at {\em almost every} point in $[a,b]$, although no direct mention to this fact will be made.
\end{remark}

\begin{remark}\label{re:wkpiso}
Notice that for any fixed $t_0\in [a,b]$ there exists a topological isomorphism
\begin{eqnarray*}
W^{k,p}([a,b],\R^m) &\longhookrightarrow & (\R^m)^k\oplus L^p([a,b],\R^m) \\
f &\longmapsto& \left(f(t_0),f'(t_0),\dots,f^{(k-1)}(t_0),f^{(k)}\right)
\end{eqnarray*}
that induces other norms in $W^{k,p}([a,b],\R^m)$ equivalent to \eqref{eq:normwkp}, for instance
\begin{equation}
\|f\|=\norml{p}{f^{(k)}}+\sum_{j=0}^{k-1}\|f^{(j)}(t_0)\|.
\end{equation}
\end{remark}

\begin{definition}\label{def:hk}
For $p=2$, the Banachable space $W^{k,p}([a,b],\R^m)$ described in Definition~\ref{def:wkp} is denoted $$H^k([a,b],\R^m)=W^{k,2}([a,b],\R^m)$$ and its elements are called maps of {\em Sobolev class $H^k$}.\index{Sobolev class $H^k$!map}\index{$H^k$ map} This is a Hilbertable space, that can be endowed, for instance, with the inner product given by
\begin{equation}\label{eq:productwkp}
\langle f,g\rangle=\langle f^{(k)},g^{(k)}\rangle_{L^2}+\sum_{j=0}^{k-1}\langle f^{(j)}(t_0),g^{(j)}(t_0)\rangle.
\end{equation}
\end{definition}

\begin{remark}\label{re:eqhk}
Notice that for any fixed $t_0\in [a,b]$ the topological isomorphism described in Remark~\ref{re:wkpiso} in the case of $H^k([a,b],\R^m)$ is given by
\begin{eqnarray}\label{eq:imhkl2}
H^k([a,b],\R^m) &\longhookrightarrow & (\R^m)^k\oplus L^2([a,b],\R^m) \\
f &\longmapsto& \left(f(t_0),f'(t_0),\dots,f^{(k-1)}(t_0),f^{(k)}\right).\nonumber
\end{eqnarray}
It clearly induces other inner products in $H^k([a,b],\R^m)$ equivalent to \eqref{eq:productwkp}, for instance\index{$H^k$ norm}\index{Norm!$H^k$}
\begin{equation}\label{eq:innprodH1}
\langle f,g\rangle=\langle f(t_0),g(t_0)\rangle +\langle f,g\rangle_{L^2}.
\end{equation}
\end{remark}

\begin{remark}
The restriction of the $L^2$--inner product of $L^2([a,b],\R^m)$ to $H^1([a,b],\R^m)$ gives a limited inner product, however not equivalent to \eqref{eq:productwkp} with $k=1$. In other words, such restriction induces a different topology on $H^1([a,b],\R^m)$.
\end{remark}

\begin{remark}\label{re:h1c0l2}
Notice that \eqref{eq:wkpc0lp} for $k=1$ and $p=2$ is given by
\begin{equation*}
H^1([a,b],\R^m)\ni f\longmapsto (f,f')\in C^0([a,b],\R^m)\oplus L^2([a,b],\R^m).
\end{equation*}
As mentioned in Definition~\ref{def:wkp}, an adaptation of Lemma~\ref{le:convderivadas} guarantees that this is a injective operator with closed image. Furthermore, this consideration implies that $H^1([a,b],\R^m)$ is a Banachable space, that can be endowed with a norm that induces the same topology as \eqref{eq:productwkp}, given by \eqref{eq:normwkp}, i.e., $$\|f\|=\normc{0}{f}+\norml{2}{f'}$$
\end{remark}

\begin{remark}
Definitions~\ref{def:czero},~\ref{def:ck},~\ref{def:lp},~\ref{def:wkp} and~\ref{def:hk} were given considering maps with counter domain $\R^m$, however $\R^m$ could be obviously replaced with any finite--dimensional vector space $V$, with the additional hypothesis that the norm of $V$ comes from an inner product in the case of formulas \eqref{eq:productldois} and \eqref{eq:productwkp}. The topologies on the spaces given in such definitions does not depend on the choice of a norm on $V$.

In addition, the correspondent topologies on these spaces will be henceforth called {\em $C^k$--topology}, {\em $C^\infty$--topology}, {\em $L^p$--topology}, {\em $H^k$--topology} and so on, depending on the regularity of the space dealt with.\index{Topology!$C^k$}\index{Topology!$C^\infty$}\index{Topology!$L^p$}\index{Topology!$H^k$}
\end{remark}

Most of the above function spaces are related in several ways. The following classic results give a few inclusions between these function spaces, some of which are compact or have dense image, as studied in the sequel.

\begin{proposition}\label{prop:inclusions}
The following inclusion maps are continuous:
\begin{itemize}
\item[(i)] $C^l([a,b],\R^m)\hookrightarrow C^k([a,b],\R^m)$, for $0\leq k\leq l$;
\item[(ii)] $C^0([a,b],\R^m)\hookrightarrow L^p([a,b],\R^m)$, for $p\in[1,+\infty]$;
\item[(iii)] $L^q([a,b],\R^m)\hookrightarrow L^p([a,b],\R^m)$, for $1\leq p\leq q\leq+\infty$;
\item[(iv)] $W^{k,q}([a,b],\R^m)\hookrightarrow W^{k,p}([a,b],\R^m)$, for $1\leq p\leq q\leq+\infty$, $k\geq 1$;
\item[(v)] $W^{k+1,p}([a,b],\R^m)\hookrightarrow C^k([a,b],\R^m)$, for $p\in[1,+\infty]$, $k\geq 0$;
\item[(vi)] $C^k([a,b],\R^m)\hookrightarrow W^{k,p}([a,b],\R^m)$, for $p\in[1,+\infty]$, $k\geq 0$.
\end{itemize}
In particular, setting $p=2$ in {\rm (v)} and {\rm (vi)}, the inclusions $H^{k+1}([a,b],\R^m)\hookrightarrow C^k([a,b],\R^m)$ and $C^k([a,b],\R^m)\hookrightarrow H^k([a,b],\R^m)$ are continuous, for $k\geq0$.
\end{proposition}

\begin{remark}
It is actually possible to give very precise estimates for the norms \eqref{eq:normt} of the above inclusions, depending on the norms chosen in each space.
\end{remark}

The following is a classic result of basic analysis that asserts that every continuous real function may be uniformly approximated by smooth functions, more precisely by polynomials. See \cite{pugh,rudin} for a proof.

\begin{stwthm}\label{thm:stoneweierstrass}\index{Theorem!Stone--Weierstrass}
The space $C^\infty([a,b],\R^m)$ is dense in $C^0([a,b],\R^m)$.
\end{stwthm}

\begin{proposition}\label{prop:cinftylp}
If $p\in[1,+\infty[$, then the space $C^\infty([a,b],\R^m)$ is dense in $L^p([a,b],\R^m)$. In particular, $C^k([a,b],\R^m)$ is dense in $L^p([a,b],\R^m)$ for all $k\geq 0$.
\end{proposition}

We now explore a couple of corollaries that follow immediately from the two above results using Proposition~\ref{prop:inclusions}.

\begin{corollary}\label{cor:cinftywkp}
If $p\in[1,+\infty[$, $k\geq1$, then $C^\infty([a,b],\R^m)$ is dense in $W^{k,p}([a,b],\R^m)$. In particular, for all $j\geq 0$, $C^{k+j}([a,b],\R^m)$ is dense in $W^{k,p}([a,b],\R^m)$.
\end{corollary}

Setting $p=2$, we obtain the analogous result for $H^k([a,b],\R^m)$.

\begin{corollary}\label{cor:cinftyhk}
If $p\in[1,+\infty[$, $k\geq1$, then $C^\infty([a,b],\R^m)$ is dense in $H^k([a,b],\R^m)$. In particular, for all $j\geq 0$, $C^{k+j}([a,b],\R^m)$ is dense in $H^k([a,b],\R^m)$.
\end{corollary}

The above density results are extremely useful to prove properties of functions with regularities weaker then $C^k$, as Sobolev class $H^k$. Most of the times, it is only possible to infer a certain formula for maps in $C^k([a,b],\R^m)$. Since this is a dense subset of $H^k([a,b],\R^m)$, if the formula is known to be continuous, it follows\footnote{It is an elementary topological fact that if two continuous maps between metric spaces coincide in a dense subset, then they must coincide everywhere.} that it holds for the entire $H^k([a,b],\R^m)$.

We now focus on the study of the Hilbert space $H^k([a,b],\R^m)$, on which the infinite--dimensional manifold of {\em Sobolev curves on a finite--dimensional manifold $M$} will be modeled, see Definition~\ref{def:H1abM} and Section~\ref{sec:hum}. For this, we recall (a very simple corollary of) the celebrated {\em Sobolev Embedding Theorem}, whose proof can be found, for instance, in \cite{adams,mazja}. Once more, it deals with some of the inclusions mentioned in Proposition~\ref{prop:inclusions}, regarding its compactness, see Definition~\ref{def:cpcop}.

\begin{proposition}\label{prop:hkckmenosum}
For all $k\geq 0$, the inclusion map $H^{k+1}([a,b],\R^m)\hookrightarrow C^k([a,b],\R^m)$ is a compact operator.
\end{proposition}

\begin{corollary}\label{cor:hkckj}
For all $1\leq j\leq k$, the inclusions $H^k([a,b],\R^m)\hookrightarrow C^{k-j}([a,b],\R^m)$ are compact operators.
\end{corollary}

\begin{proof}
This is a simple consequence of Propositions~\ref{prop:propcpcop},~\ref{prop:hkckmenosum} and Remark~\ref{re:ckck-1}.
\end{proof}

\begin{lemma}\label{le:h1c0extension}
Let $B:H^1([a,b],\R^m)\times H^1([a,b],\R^m)\to\R$ be a continuous bilinear form. If $B$ admits a continuous extension $$\widehat B:H^1([a,b],\R^m)\times C^0([a,b],\R^m)\la\R,$$ then $B$ is represented by a compact operator $K\in\cpc(H^1([a,b],\R^m))$, see Definitions~\ref{def:represents} and~\ref{def:cpcop}.
\end{lemma}

\begin{proof}
Assume the operator $\widehat B$ that extends $B$ in the following diagram is continuous. From Corollary~\ref{cor:hkckj}, the inclusion $i$ is a compact operator.
$$\xymatrix@+17.5pt{
C^0([a,b],\R^m)\ar[r]^(0.33){\widehat B}& \big[H^1([a,b],\R^m)\big]^*\cong H^1([a,b],\R^m) \\
H^1([a,b],\R^m)\ar@{^{(}->}[u]^i \ar@(r,d)[ur]_{\widehat B\circ i} &
}$$
From Proposition~\ref{prop:propcpcop}, the composite map $\widehat B\circ i$ is compact. Since $\widehat B$ coincides with $B$ on its domain, it follows that the bilinear form $B$ is represented by $K=\widehat B\circ i$, which is a compact operator of $H^1([a,b],\R^m)$.
\end{proof}

The following result will be used in Section~\ref{sec:hum} to establish (smooth) compatibility of charts of $H^1([a,b],M)$.

\begin{theorem}\label{thm:curvesalphasmooth}
Let $U\subset\R\times\R^m$ be an open subset and consider the set of all Sobolev $H^1$ curves $\gamma:[a,b]\to\R^m$ whose graph is contained in $U$,
\begin{equation*}
\curves U=\left\{\gamma\in H^1([a,b],\R^m):(t,\gamma(t))\in U, \, \mbox{ for all } t\in[a,b]\right\}.
\end{equation*}
Given a $C^k$ map $\alpha:U\to\R^n$, define
\begin{eqnarray}\label{eq:curvesalpha}
&\curves\alpha:\curves U\la H^1([a,b],\R^n)&\\
&\curves\alpha(\gamma)(t)=\alpha(t,\gamma(t)), \quad t\in [a,b]&\nonumber
\end{eqnarray}
Then $\curves U$ is open in $H^1([a,b],\R^m)$ and $\curves\alpha$ is a $C^{k-1}$ map. In addition, if $k\geq 2$, for all $\gamma\in\curves U$, $v\in H^1([a,b],\R^n)$ and $t\in [a,b]$,
\begin{equation}\label{eq:diffcurvesalpha}
\dd\curves\alpha(\gamma)(v)(t)=\frac{\partial\alpha}{\partial x}(t,\gamma(t))v(t).
\end{equation}
\end{theorem}

\begin{proof}
This proof is in great part adapted from the proof of \cite[Theorem 4.2.16]{picmertausk}. It will be given through the following two claims.

\begin{claim}\label{cl:curvesalphacontinua}
The subset $\curves U$ is open in $H^1([a,b],\R^m)$ and $\curves\alpha$ is continuous.
\end{claim}

Since $U\subset\R\times\R^m$ is open,
\begin{equation*}
\Curves{C^0}{U}=\left\{\gamma\in C^0([a,b],\R^m):(t,\gamma(t))\in U, \, \mbox{ for all } t\in[a,b]\right\}.
\end{equation*}
is an open subset of $C^0([a,b],\R^m)$. Moreover, from Proposition~\ref{prop:hkckmenosum} the inclusion $i:H^1([a,b],\R^m)\hookrightarrow C^0([a,b],\R^m)$ is continuous, hence $$\curves U=i^{-1}(\Curves{C^0}{U})$$ is an open subset of $H^1([a,b],\R^m)$.

Furthermore, from Remark~\ref{re:h1c0l2} and Lemma~\ref{le:cdreduction}, continuity of $\curves\alpha$ follows from continuity of the composite maps
\begin{gather}
\label{eq:composite1}
\curves U \xrightarrow{\;\;\curves\alpha\;\;} H^1([a,b],\R^m)\xhookrightarrow{\;\;i\;\;} C^0([a,b],\R^m) \\
\label{eq:composite2}
\curves U\xrightarrow{\;\;\curves\alpha\;\;}H^1([a,b],\R^m)\xrightarrow{\;\;\dd\;\;} L^2([a,b],\R^m).
\end{gather}

Moreover, from continuity of $\alpha:U\to\R^n$, it follows that the map
\begin{eqnarray*}
&\Curves{C^0}{\alpha}:\Curves{C^0}{U} \la C^0([a,b],\R^m)&\\
&\Curves{C^0}{\alpha}(\gamma)(t)=\alpha(t,\gamma(t)), \quad t\in [a,b]&
\end{eqnarray*}
is continuous. In addition, from Proposition~\ref{prop:hkckmenosum}, $i$ is also continuous, hence \eqref{eq:composite1} is continuous.

As for continuity of \eqref{eq:composite2}, evaluating it explicitly on $\gamma\in\curves U$,
$$\frac{\dd}{\dd t}\curves\alpha(\gamma)(t)=\frac{\partial\alpha}{\partial t}(t,\gamma(t))+\frac{\partial\alpha}{\partial
x}(t,\gamma(t))\dot\gamma(t).$$
Thus \eqref{eq:composite2} is given by the sum of the restriction of $\Curves{C^0}{\frac{\partial\alpha}{\partial t}}$ to $\curves U$ and of the derivation map $\dd:\curves U\to L^2([a,b],\R^m)$. More precisely, \eqref{eq:composite2} is given by
$$\xymatrix@+30pt{
\curves{U} \ar@(rd,l)[rrd]_{\eqref{eq:composite2}} \ar[rr]^(0.33){\Curves{C^0}{\frac{\partial\alpha}{\partial x}}\;\oplus\;\dd} & & C^0([a,b],\Lin(\R^m,\R^n))\oplus L^2([a,b],\R^n)\ar@(d,u)[d]^{\eqref{eq:compositeC0L2L2}} \\
& & L^2([a,b],\R^n)
}$$
and hence is continuous. This concludes the proof of Claim~\ref{cl:curvesalphacontinua}.

We now use the weak differentiation principle given in Lemma~\ref{le:weakdiffprinc} to establish differentiability of $\curves\alpha$.

\begin{claim}\label{cl:curvesalphaC1}
If $\alpha:U\to\R^n$ is a map of class $C^2$ defined on an open subset $U\subset\R^m$ then $\curves\alpha$ is of class $C^1$ and \eqref{eq:diffcurvesalpha} holds.
\end{claim}

The separating family $\mathcal F$ for $H^1([a,b],\R^n)$ is the family of {\em evaluation maps}. For every $t\in[a,b]$, consider
\begin{eqnarray*}
\ev_t:H^1([a,b],\R^n)&\la &\R^n \\
\gamma &\longmapsto & \gamma(t)
\end{eqnarray*}
and $\mathcal F=\{\ev_t:t\in[a,b]\}$. Let $g$ be given by
\begin{eqnarray*}
&g:\curves U\longrightarrow\Lin(H^1([a,b],\R^m),H^1([a,b],\R^n))& \\
&g(\gamma)(v)(t)=\displaystyle\frac{\partial\alpha}{\partial x}(t,\gamma(t))v(t),\quad t\in[a,b],&
\end{eqnarray*}
for all $\gamma\in\curves U$, $v\in H^1([a,b],\R^m)$. It is then clear that $$\frac{\partial(\ev_t\circ\curves\alpha)}{\partial v}(\gamma)=\frac{\dd}{\dd s}\alpha(t,\gamma(t)+sv(t))\Big\vert_{s=0}=g(\gamma)(v)(t).$$ The only nontrivial part of the proof, which we omit, is the continuity of $g$. Such continuity follows from the continuity of $$\curves{\frac{\partial\alpha}{\partial x}}:\curves U\la H^1([a,b],\Lin(\R^m,\R^n)).$$ More precisely,
\begin{eqnarray}\label{eq:LINLINLIN}
&\mathcal O:H^1([a,b],\Lin(\R^m,\R^n))\longrightarrow\Lin(H^1([a,b],\R^m), H^1([a,b],\R^n))\nonumber & \\
&\mathcal O(T)(v)(t)=T(t)v(t), \quad t\in[a,b], &
\end{eqnarray}
is a continuous operator. For details on how to prove such continuity statements, we refer to Piccione, Mercuri and Tausk \cite[Section 4.2]{picmertausk}. Notice that $\dd\curves\alpha$ is equal to the composition of $\curves{\frac{\partial\alpha}{\partial x}}$ with the continuous operator \eqref{eq:LINLINLIN}. Applying Lemma~\ref{le:weakdiffprinc} with the above setting, it follows that $\curves\alpha$ is $C^1$ and that \eqref{eq:diffcurvesalpha} holds, concluding the proof of Claim~\ref{cl:curvesalphaC1}.

Finally, a simple inductive argument on $k$ guarantees that we may replace $C^2$ and $C^1$ with $C^k$ and $C^{k-1}$, respectively, in Claim~\ref{cl:curvesalphaC1}. This concludes the proof of Theorem~\ref{thm:curvesalphasmooth}.
\end{proof}

\section{A few more lemmas}
\label{sec:afewmorelemmas}

Once more, we end the chapter with some lemmas that will be later used. These are functional analysis results that are easily proved and repeated here for the sake of self--containment.

\begin{proposition}\label{prop:cksep}
The space $C^0(K,\R)$ endowed with the uniform convergence norm is separable if and only if $K$ is metrizable.
\end{proposition}

For a proof of this classic result, see Fabi\'an et al. \cite{fabian}. We now prove some elementary analytical lemmas that involve some of the spaces of functions studied in Section~\ref{sec:funcspaces}.

\begin{lemma}\label{le:milagre}
Let $f,g:[a,b]\to\R^m$ be continuous maps, with $f$ absolutely continuous, and suppose that $f'=g$ almost everywhere. Then $f\in C^1([a,b],\R^m)$ and $f'=g$.
\end{lemma}

\begin{proof}
From Proposition~\ref{prop:abscontcharact}, the derivative $f':[a,b]\to\R^m$ exists almost everywhere, and by hypothesis coincides almost everywhere with $g$, which is continuous. Thus,
\begin{eqnarray*}
f(t)&=&f(a)+\int_a^t f'(s)\;\dd s\\
&=& f(a)+\int_a^t g(s)\;\dd s.
\end{eqnarray*}
Therefore, from the Fundamental Theorem of Calculus, $f$ is differentiable and has continuous derivative $g$. Hence $f\in C^1([a,b],\R^m)$, concluding the proof.
\end{proof}

\begin{corollary}\label{cor:milagre}
Let $f:[a,b]\to M$ be an absolutely continuous map and $g\in\sect^0(f^*TM)$, with $\dd f=g$ almost everywhere. Then $f\in C^1([a,b],M)$ and $\dd f=g$.
\end{corollary}

\begin{proof}
Considering local charts it is possible to reduce this problem to Euclidean space. The proof then follows directly from Lemma~\ref{le:milagre}.
\end{proof}

\begin{lemma}\label{le:distder0}
Let $\alpha\in L^1([a,b],\R^m)$ and consider an inner product $\langle\,\cdot,\cdot\,\rangle$ on $\R^m$. Suppose that for every $\lambda\in C^\infty_c(\,]a,b[,\R^m)$,
\begin{equation}\label{eq:distder0hyp}
\int_a^b \langle\alpha(t),\lambda'(t)\rangle\;\dd t=0.
\end{equation}
Then $\alpha$ is constant almost everywhere.
\end{lemma}

\begin{proof}
First, consider $\alpha$ and $\lambda$ expressed in coordinates $\alpha=(\alpha_i)_{i=1}^m$ and $\lambda=(\lambda_i)_{i=1}^m$, so that $$\int_a^b \langle\alpha(t),\lambda'(t)\rangle\;\dd t=\int_a^b\sum_{i=1}^m \alpha_i\lambda'_i\;\dd t=\sum_{i=1}^m \int_a^b\alpha_i\lambda'_i\;\dd t$$ Obviously, the above expression vanishes if and only if each integral with the product of the $i^{\mbox{\tiny th}}$ coordinates of $\alpha$ and $\lambda'$ vanishes. Thus, we reduce the problem to the case $m=1$, where the counter domain the considered maps is one--dimensional and the inner product $\langle\alpha(t),\lambda'(t)\rangle$ is an ordinary product of real functions.

Second, notice that we may use the Fundamental Theorem of Calculus characterize the derivative of maps $\lambda\in C^\infty_c(\,]a,b[,\R)$ as follows,
\begin{equation*}
\mathfrak T=\big\{\lambda':\lambda\in C^\infty_c(\,]a,b[,\R)\big\}=\left\{\mu\in C^\infty_c(\,]a,b[,\R):\int_a^b \mu(t)\;\dd t=0\right\},
\end{equation*}
see figure below. Hence, \eqref{eq:distder0hyp} is equivalent to $\int_a^b \alpha(t)\mu(t)\;\dd t=0$ for all $\mu\in C^\infty_c(\,]a,b[,\R)$ such that $\int_a^b \mu(t)\;\dd t=0$.
\begin{figure}[htf]
\begin{center}
\vspace{-0.5cm}
\includegraphics[scale=1.2]{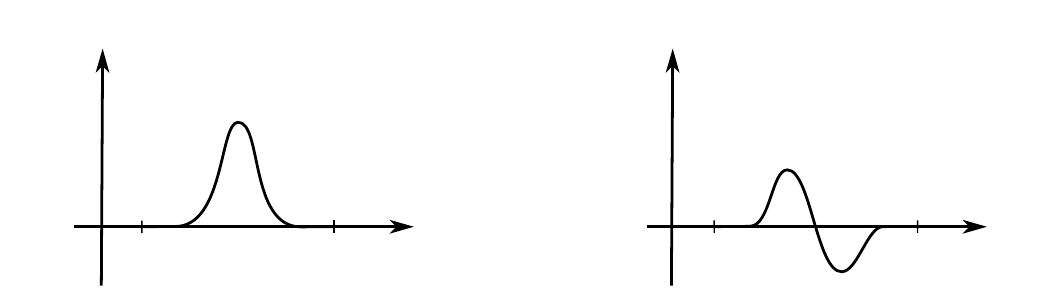}
\begin{pgfpicture}
\pgfputat{\pgfxy(-4.55,1)}{\pgfbox[center,center]{$a$}}
\pgfputat{\pgfxy(-2.2,1)}{\pgfbox[center,center]{$b$}}
\pgfputat{\pgfxy(2.55,1)}{\pgfbox[center,center]{$a$}}
\pgfputat{\pgfxy(4.9,1)}{\pgfbox[center,center]{$b$}}
\pgfputat{\pgfxy(-4.7,3.5)}{\pgfbox[center,center]{$\lambda$}}
\pgfputat{\pgfxy(2.3,3.5)}{\pgfbox[center,center]{$\mu$}}
\end{pgfpicture}
\vspace{-0.5cm}
\end{center}
\end{figure}

\noindent
Consider $\lambda_0\in C^\infty_c(\,]a,b[,\R)$ such that $\int_a^b \lambda_0(t)\;\dd t=1$, and set $c=\int_a^b \alpha(t)\lambda_0(t)\;\dd t$. Then,
\begin{equation}\label{eq:distder0eq1}
\int_a^b (\alpha(t)-c)\lambda_0(t)\;\dd t=\int_a^b \alpha(t)\lambda_0(t)\;\dd t-c\int_a^b \lambda_0(t)\;\dd t=c-c=0
\end{equation}
In addition, by hypothesis
\begin{equation}\label{eq:distder0eq2}
\int_a^b (\alpha(t)-c)\mu\;\dd t=\int_a^b \alpha(t)\mu(t)\;\dd t -c\int_a^b \mu(t)\;\dd t=0.
\end{equation}
Notice that $C^\infty_c(\,]a,b[,\R)$ naturally decomposes as the direct sum of $\mathfrak T$ and the one--dimensional subspace spanned by $\lambda_0$. Thus, $\mu$ and $\lambda_0$ span the entire $C^\infty_c(\,]a,b[,\R)$, hence from \eqref{eq:distder0eq1} and \eqref{eq:distder0eq2}, it follows that $$\int_a^b (\alpha(t)-c)\xi(t)\;\dd t=0$$ for all $\xi\in C^\infty_c(\,]a,b[,\R)$. From the Fundamental Lemma of Calculus of Variations, $\alpha(t)-c$ must be null almost everywhere, hence $\alpha$ is constant almost everywhere.
\end{proof}

\begin{remark}
The result in Lemma~\ref{le:distder0} has a clear interpretation in terms of {\em distributions}.\footnote{Not to be confused with subbundles of the tangent bundle. In mathematical analysis, distributions are objects that generalize functions, making it possible to differentiate functions whose derivative does not exist in the classical sense. In particular, any locally integrable function has a distributional derivative. Distributions are widely used to formulate generalized solutions of PDEs.} Namely, it states that if a distributional derivative is zero, then the distribution is almost everywhere constant.
\end{remark}

Let us now recall some basic properties of operators in Banach spaces.

\begin{lemma}\label{le:invertibility}
Let $V$ be a Banach space and $T\in\Lin(V)$ such that $\|T\|<1$, where $\|\cdot\|$ is the operator norm \eqref{eq:normt}. Then $\id-T\in\Lin(V)$ is an invertible operator.
\end{lemma}

\begin{proof}
The operator power series $\sum_{n=0}^{+\infty} T^n$ is absolutely convergent, since $$\left\|\sum_{n=0}^{+\infty} T^n\right\|\leq\sum_{n=0}^{+\infty} \|T\|^n=\exp\|T\|<+\infty.$$ Since $V$ is Banach, $\sum_{n=0}^{+\infty} T^n$ converges in $\Lin(V)$. For each $N\in\N$, $$\lim_{N\to +\infty}(\id-T)\left(\sum_{n=0}^N T^n\right)= \lim_{N\to +\infty}\id-T^{N+1}=\id$$ and analogously for $\left(\sum_{n=0}^N T^n\right)(\id-T)$. Hence the operator $\id-T$ is invertible, with inverse given by
\begin{equation*}
(\id-T)^{-1}=\sum_{n=0}^{+\infty} T^n\in\Lin(V).\qedhere
\end{equation*}
\end{proof}

\begin{lemma}\label{le:tsobreab}
Let $V$ and $W$ be Banach spaces. Then the following is an open subset of $\Lin(V,W)$,
\begin{equation*}
\mathcal O=\{T\in\Lin(V,W): \im T=W \mbox{ and } \ker T\mbox{ is complemented}\}.
\end{equation*}
\end{lemma}

\begin{proof}
First, notice that $T\in\mathcal O$ if and only if $T$ admits a right inverse, i.e., if there exists $S\in\Lin(W,V)$ with $TS=\id$. Indeed, the existence of such $S$ is equivalent to surjectivity of $T$, and it is easily seen that $\im S$ is a closed complement to $\ker T$.

Let $T_0\in\mathcal O$ and consider $S_0\in\Lin(W,V)$ its right inverse and $\varepsilon<\frac{1}{\|S_0\|}$, where the norm considered is the usual \eqref{eq:normt}. We claim that if $T\in\Lin(V,W)$ satisfies $\|T-T_0\|<\varepsilon$, then $T\in\mathcal O$. Since $T_0S_0=\id$,
\begin{eqnarray*}
\|\id-TS_0\| &=& \|T_0S_0-TS_0\| \\
&\leq &\|T-T_0\|\|S_0\| \\
&<& 1.
\end{eqnarray*}
Lemma~\ref{le:invertibility} applied to $\id-TS_0$ gives that $TS_0$ is invertible. Thus, the operator $S_0(TS_0)^{-1}\in\Lin(W,V)$ is a right inverse for $T$, implying that $T\in\mathcal O$. This concludes the proof that $\mathcal O$ is open in $\Lin(V,W)$.
\end{proof}

The following is a direct consequence of the universal property of tensor products of vector spaces, which can be found for instance in Ash \cite{ash}.

\begin{lemma}\label{le:upt}
Let $Z$, $V_1,\dots,V_n$ and $W_1,\dots,W_n$ be finite--dimensional real vector spaces and consider the following diagram
\begin{equation}\label{eq:diagtb}
\xymatrix@+5pt{
V_1\times\ldots\times V_n\times W_1\times\ldots\times W_n \ar[rrr]^(.7){T}\ar[d]_\otimes & & & Z\\
V_1\otimes\ldots\otimes V_n\times W_1\otimes\ldots\otimes W_n \ar@(r,d)[rrru]_B & &
}
\end{equation}
Given any multilinear form $T$, there exists a unique bilinear form $B$ that makes this diagram commutative.
\end{lemma}

\begin{lemma}\label{le:lemmatransvthm}
Let $V_1,V_2$ and $H$ be Banach spaces, $L\in\Lin(V_1\oplus V_2,H)$ a surjective operator, and $\Pi\in\Lin(V_1\oplus V_2,V_1)$ the projection onto the first variable. Then the composite operators
\begin{eqnarray}
L|_{V_2}: &V_2 &\longhookrightarrow \; V_1\oplus V_2 \, \xrightarrow{\;\; L\;\;} \, H \label{eq:tcomp1} \\
\Pi|_{\ker L}: &\ker L &\longhookrightarrow \; V_1\oplus V_2 \, \xrightarrow{\;\;\Pi\;\;} \, V_1 \label{eq:tcomp2}
\end{eqnarray}
have isomorphic kernels and cokernels. In particular, $L|_{V_2}$ is surjective if and only if $\Pi|_{\ker L}$ is surjective. Moreover, $L|_{V_2}$ is Fredholm if and only if $\Pi|_{\ker L}$ is Fredholm, and in this case, both have the same index. In addition, if either (hence both) the above operators is Fredholm, then $\ker L$ is complemented in $V_1\oplus V_2$.
\end{lemma}

\begin{proof}
By direct comparison, it is easy to conclude that $\{0\}\oplus\ker L|_{V_2}=\ker \Pi|_{\ker L}$, thus $L|_{V_2}$ and $\Pi|_{\ker L}$ have isomorphic kernels. Moreover, $L$ induces an isomorphism
\begin{equation}
\frac{V_1\oplus V_2}{V_2+\ker L}\xrightarrow{\;\;\cong\;\;}\frac{H}{L(V_2)}=\coker L|_{V_2},
\end{equation}
and $\Pi$ induces an isomorphism
\begin{equation}
\frac{V_1\oplus V_2}{V_2+\ker L}\xrightarrow{\;\;\cong\;\;}\frac{V_1}{\Pi(\ker L)}=\coker \Pi|_{\ker L}.
\end{equation}
Therefore, $\coker L|_{V_2}$ and $\coker \Pi|_{\ker L}$ are isomorphic. In particular, surjectivity of $L|_{V_2}$ is equivalent to surjectivity of $\Pi|_{\ker L}$, since these are in turn equivalent to $\coker L|_{V_2}$, and hence $\coker \Pi|_{\ker L}$, being trivial. This also proves that $L|_{V_2}$ is Fredholm if and only if $\Pi|_{\ker L}$ is Fredholm, and in this case both have the same index, since they have isomorphic kernels and cokernels.

Suppose now that $\Pi|_{\ker L}$ is a Fredholm operator. Then $\Pi(\ker L)$ is a closed and complemented subspace of $V_1\oplus V_2$, since it has finite codimension. Thus, there exists a closed subspace $S_1$ of $V_1$ such that $V_1=S_1\oplus\Pi(\ker L)$. In addition, $\ker \Pi|_{\ker L}=\ker L\cap(\{0\}\oplus V_2)$ is finite--dimensional, hence complemented in $\{0\}\oplus V_2$. Let $S_2$ be a closed subspace of $\{0\}\oplus V_2$ such that $S_2\oplus [\ker L \cap(\{0\}\oplus V_2)]=\{0\}\oplus V_2$.

We claim that $S_1\oplus S_2$ is a complement of $\ker L$ in $V_1\oplus V_2$. In fact, suppose $(x,y)\in (S_1\oplus S_2)\cap\ker L$. Then $L(x,y)=0$, hence $x\in S_1\cap\Pi(\ker L)=\{0\}$, i.e., $x=0$. Moreover, $(0,y)\in S_2\cap\ker L=\{0\}$, thus $y=0$. Therefore, $(S_1\oplus S_2)\cap\ker L=\{0\}$. Finally, let $(x,y)\in V_1\oplus V_2$. Since $x\in V_1$, there exist $x'\in V_1$ and $s\in S_1$ such that $x=x'+s$, where $(x',y')\in\ker L$ for some $y'\in V_2$. Since $y'-y\in V_2$, there exists $(0,z)\in\ker L$ and $s'\in S_2$ such that $(0,y'-y)=(0,z)+(0,s)$. Therefore, $(x,y)=(s,0)-(0,s')+(x',y'-z)$, where $(s,0)\in S_1$, $(0,s')\in S_2$ and $(x',y'-z)\in\ker L$. Thus $(S_1\oplus S_2)\oplus\ker L=V_1\oplus V_2$, concluding the proof.
\end{proof}

%
%

We end this section proving a series of lemmas for operators of the form $L_1\oplus L_2:V\oplus H\to H$, where $V$ is a Banach space and $H$ is a Hilbert space, to be used in the proof of the Abstract Genericity Criterion~\ref{crit:abstractgenericity}. In our applications, $L_2\in\Lin(H)$ will be a self--adjoint Fredholm operator.

\begin{lemma}\label{le:abs1}
Let $V$ be a Banach space, $H$ a Hilbert space, $L_1:V\to H$ and $L_2:H\to H$ continuous operators, with $\im L_2$ closed,\footnote{Recall that if $L_2$ is Fredholm, this is automatically verified as a consequence of Proposition~\ref{prop:fredclosed}.} and consider their direct sum \begin{eqnarray*} L=L_1\oplus L_2:V\oplus H &\longrightarrow & H\\ (v,h) &\longmapsto & L_1v+L_2h.\end{eqnarray*} Then, $L$ is surjective if and only if the projection onto $(\im L_2)^\perp$ is surjective, i.e., \begin{equation}\label{lemma1eq1}p_{(\im L_2)^\perp}\im L_1=(\im L_2)^\perp.\end{equation} In addition, if $L_2$ is self--adjoint and $p_{\ker L_2}\im L_1$ is closed in $\ker L_2$,\footnote{This hypothesis is also automatically verified in case $L_2$ is Fredholm.} then $L$ is surjective if and only if for all $h\in\ker L_2\setminus\{0\}$ there exists $v\in V$ such that $\langle L_1v,h\rangle\neq 0$.
\end{lemma}

\begin{proof}
If \eqref{lemma1eq1} holds, given $k\in H$, consider $p_{(\im L_2)^\perp}k$. From \eqref{lemma1eq1}, there exists $v\in V$ such that \begin{equation}\label{lemma1eq2}
p_{(\im L_2)^\perp}k=p_{(\im L_2)^\perp}L_1v.
\end{equation} Hence $k-L_1v\in\im L_2$. Therefore, there exists $h\in H$ such that $k=L_1v+L_2h=L(v,h)$. Conversely, if $L$ is surjective, given $k\in (\im L_2)^\perp$ there exists $v\in V$ and $h\in H$ such that $k=L_1v+L_2h$. Applying $p_{(\im L_2)^\perp}$, we obtain $$k=p_{(\im L_2)^\perp}L_1v,$$ thus \eqref{lemma1eq1} is verified.

In addition, suppose $L_2$ self--adjoint and $p_{\ker L_2}\im L_1$ closed in $\ker L_2$. Then $(\im L_2)^\perp=\ker L_2$, and $$p_{\ker L_2}\circ L_1:V\longrightarrow\ker L_2$$ is not surjective if and only if $(\im(p_{\ker L_2}\circ L_1))^\perp\neq 0$, which is equivalent to existing $h\in\ker L_2\setminus\{0\}$ such that $\langle p_{\ker L_2}L_1v,h\rangle=\langle L_1v,h\rangle=0$, for all $v\in V$. This concludes the proof.
\end{proof}

\begin{lemma}\label{le:abs2}
Let $L:U\to V$ be an operator between vector spaces, and $S\subset V$ a subspace with finite codimension. Then $L^{-1}(S)$ has finite codimension in $U$ and $$\codim_U L^{-1}(S)=\codim_V S- \codim_V (\im L+S),$$ where by $\codim_B A=\dim B/A$ we mean the codimension of $A$ in $B$.
\end{lemma}

\begin{proof} 
Denote by $q:V\to V/S$ the quotient map. Then $q\circ L:U\to V/S$ has kernel $L^{-1}(S)$. Thus $q\circ L$ induces an injective operator $T$ on the quotient, having the same image of $q\circ L$, such that the following diagram commutes (vertical arrows are projections).
$$\xymatrix@+17.5pt{
U \ar[d]\ar[r]^L & V\ar[d]^q \\
\displaystyle\frac{U}{L^{-1}(S)}\ar[r]^(.55)T &\displaystyle\frac{V}{S}}$$
Since $S$ has finite codimension in $V$ and $T$ is injective, $L^{-1}(S)$ has finite codimension in $U$. Then the following sequence of vector spaces and operators is exact, and analogously to \eqref{eq:esq}, splits.
$$0\xrightarrow{\;\;\;\;}\frac{U}{L^{-1}(S)}\xrightarrow{\;\;T\;\;}V/S\xrightarrow{\;\;\;\;}\frac{V/S}{\im T}\xrightarrow{\;\;\;\;} 0.$$ Hence, \begin{eqnarray}\label{lemma2eq1}
V/S &\cong& \frac{U}{L^{-1}(S)}\oplus\frac{V/S}{\im T} \nonumber\\
&\cong& \frac{U}{L^{-1}(S)}\oplus\frac{V/S}{\im (q\circ L)}.\end{eqnarray} In addition,
\begin{equation}\label{lemma2eq2}
\frac{V}{\im L+S}\cong\frac{V/S}{\im (q\circ L)},
\end{equation} since the map \begin{eqnarray*}V &\la &\frac{V/S}{\im (q\circ L)} \\
v &\longmapsto &(v+S)+\im (q\circ L)\end{eqnarray*} is clearly surjective and has kernel $\im L+S$. Finally, it follows that
\begin{eqnarray*}
\codim_V S &=& \dim V/S \\
&\stackrel{\eqref{lemma2eq1}}{=}& \dim\frac{U}{L^{-1}(S)}+\dim\frac{V/S}{\im (q\circ L)} \\
&\stackrel{\eqref{lemma2eq2}}{=}& \codim_U L^{-1}(S)+\codim_V(\im L+S).\qedhere
\end{eqnarray*}
\end{proof}

\begin{proposition}\label{prop:sumcomplement}
Let $U$, $V$ and $W$ be Banach spaces, $L_1:U\to W$, $L_2:V\to W$ continuous operators, with $\ker L_2$ complemented in $V$ and $\im L_2$ finite codimensional in $W$.\footnote{Once more, notice that if $L_2$ is Fredholm, these hypotheses are automatically verified.} Consider the direct sum $L=L_1\oplus L_2:U\oplus V\to W$, as in Lemma~\ref{le:abs1}. Then $\ker L$ is complemented in $U\oplus V$.
\end{proposition}

\begin{proof}
Consider the quotient map $q:W\to W/\im L_2$. Obviously, the restriction $q|_{\im L_1}:\im L_1\to W/\im L_2$ has kernel $\im L_1\cap\im L_2$ and hence induces an injective operator $$\frac{\im L_1}{\im L_1\cap\im L_2}\la\frac{W}{\im L_2}.$$ Since the counter domain is finite--dimensional, it follows that $\im L_1\cap\im L_2$ has finite codimension in $\im L_1$, hence is complemented, see Lemma~\ref{le:finitecomplemented}. Let $Z\subset W$ be a (finite--dimensional) complement of $\im L_1\cap\im L_2$ in $\im L_1$. Then $L_1$ induces an injective operator $$\frac{L_1^{-1}(Z)}{\ker L_1}\la Z,$$ and hence $\ker L_1$ is complemented in $L_1^{-1}(Z)$, since it has finite codimension in this space (see Lemma~\ref{le:finitecomplemented}). Let $U'$ be a (finite--dimensional) complement of $\ker L_1$ in $L_1^{-1}(Z)$ and $V'$ a complement of $\ker L_2$ in $V$. We claim that $U'\oplus V'$ is a complement of $\ker L$ in $U\oplus V$.

In fact, if $(x,y)\in (U'\oplus V')\cap\ker L$, since $U'\subset L_1^{-1}(Z)$, then $L_1(x)\in Z$. In addition, $L_1(x)=-L_2(y)\in\im L_2$, hence $L_1(x)\in Z\cap (\im L_1\cap\im L_2)=\{0\}$. This implies $L_1(x)=L_2(y)=0$, hence $x\in\ker L_1\cap U'=\{0\}$ and $y\in\ker L_2\cap V'=\{0\}$. Therefore $(U'\oplus V')\cap\ker L=\{0\}$.

Moreover, for any $(x,y)\in U\oplus V$, choose $u\in U$ and $z\in Z$, such that $L_1(x)=L_1(u)+z$ and $L_1(u)\in\im L_2$. Since $z\in Z\subset\im L_1$, there exists $a\in U'$ such that $z=L_1(a)$. Thus $x=u+a+b$, for some $b\in\ker L_1$. Since $L_1(u)\in\im L_2$, there exists $w\in V'$ such that $L_1(u)=L_2(w)$. Then $y=c+v$, for some $c\in\ker L_2$ and $v\in V'$. It follows that $(a,v+w)\in U'\oplus V'$, $L(u+b,c-w)=L_1(u)-L_2(w)=0$ and $(x,y)=(a,v+w)+(u+b,c-w).$ Therefore $(U'\oplus V')+\ker L=U\oplus V$. This concludes the proof that $U'\oplus V'$ is a complement of $\ker L$ in $U\oplus V$.
\end{proof}

\chapter{Banach and Hilbert manifolds}
\label{chap3}

In this chapter, we will discuss elementary concepts and results on infinite--dimensional manifolds and their role in global analysis. More precisely, we will focus on three important topics, namely spaces of sections of vector bundles over non compact manifolds, the Sobolev $H^1$ curves on a non compact manifold, and actions of Lie groups on Hilbert manifolds.

In Section~\ref{sec:infinitedimmnflds}, we define basic concepts and explore some classic transversality results in the context of Banach manifolds. In the following section, we deal with the space of $C^k$ sections of tensor bundles over a finite--dimensional non compact manifold $M$. The main result of this section, Proposition~\ref{prop:affineworks}, gives a separable Banach manifold structure to a set of $C^k$ semi--Riemannian metrics on $M$. Further comments on the similarity of this domain of metrics and usual metrics considered in general relativity are given in Remark~\ref{re:physics}. In Section~\ref{sec:hum}, we study the Hilbert manifold structure of the set of Sobolev $H^1$ curves on $M$. Finally, in Section~\ref{sec:actions}, basic notions of actions of finite--dimensional Lie groups on Hilbert manifolds are given, and the special case of the reparameterization action of $S^1$ on the Hilbert manifold of Sobolev $H^1$ periodic curves on $M$ is studied, see Example~\ref{ex:s1h1}.

As in the previous chapters, $M$ is considered throughout the text as a possibly non compact smooth $m$--dimensional manifold, endowed with an auxiliary Riemannian metric $g_\mathrm R$.

\section{Infinite--dimensional manifolds}
\label{sec:infinitedimmnflds}

In this section, the definitions of Banach and Hilbert manifolds are given, as well as a few key facts regarding transversality. This brief exposition aims to recall fundamentals of this theory making the text self contained, and by no means to give a full treatment of the subject. A thorough discussion of fundamentals of infinite--dimensional differential geometry can be found in Lang \cite{lang}.

\begin{definition}\label{def:chartatlasmnfld}
Let $X$ be a set. A {\em local chart}\index{Chart} on $X$ is a pair $(U,\varphi)$, where $U\subset X$ and $\varphi:U\to \varphi(U)$ is a bijection between $U$ and an open subset $\varphi(U)$ of some Banach space. Two charts $(U,\varphi)$ and $(V,\psi)$ are said to be {\em $C^k$ compatible}\index{Chart! $C^k$ compatible} if either $U\cap V=\emptyset$ or $$\psi\circ\varphi^{-1}:\varphi(U\cap V)\longrightarrow\psi(U\cap V)$$ is a $C^k$ diffeomorphism between open sets. A $C^k$ {\em atlas}\index{Atlas} on $X$ is a set of pairwise $C^k$ compatible charts on $X$, whose domains cover $X$. Finally, a {\em $C^k$ Banach manifold}\index{Banach manifold} is a set $X$ endowed with a maximal $C^k$ atlas.
\end{definition}

\begin{definition}
A {\em $C^k$ Hilbert manifold}\index{Hilbert manifold} $X$ is a Banach manifold whose $C^k$ maximal atlas has charts with Hilbert spaces as counter domain.
\end{definition}

\begin{remark}
A Banach (or Hilbert) manifold $X$ will be always supposed to be \emph{smooth}, i.e., $C^k$ for all $k\in\N$, unless otherwise specified\footnote{For instance, in some further applications using transversality of $C^k$ maps, the obtained manifolds will be $C^k$ and not smooth.}. We will call \emph{charts} only charts that belong to the given maximal atlas of $X$. In addition, given $x\in X$ a point in a Banach manifold, a local chart $(U,\varphi)$ of $X$ with $x\in U$ is called a {\em local chart around $x$}.
\end{remark}

\begin{remark}\label{re:cdomaincharts}
Notice that charts of an atlas on $X$ are not supposed to share the same counter domain. However, differentiating the compatibility condition it follows that counter domains are linearly isomorphic. Since the set of points on $X$ for which there exists a chart with counter domain linearly isomorphic to some fixed Banach (or Hilbert) space is open and closed, each connected component of $X$ admits an atlas with charts taking value on the same space.
\end{remark}

\begin{remark}
Analogously to the finite--dimensional case, an atlas induces a topology on $X$, such that domains of charts are open subsets and charts are homeomorphisms. At this point, there is no reason to assume any separation or countability axiom on this topology.
\end{remark}

The notion of $C^k$ map can be clearly extended from the context of Banach spaces (see Definition~\ref{def:ckbanach}) to Banach manifolds using charts, since differentiability is a local matter. In addition, central results of differential calculus on Banach spaces are automatically valid on Banach manifolds, analogously to the finite--dimensional case. 

\begin{remark}\label{re:txx}
Regarding the tangent space $T_xX$ to an infinite--dimensional manifold $X$ at the point $x$, it can be obtained as the set of equivalence classes of curves passing through $x$, or equivalence classes of triples $(U,\varphi,v)$, where $(U,\varphi)$ is a chart, $x\in U$ and $v\in\varphi(U)$, see respectively Mercuri, Piccione and Tausk \cite{picmertausk} and Lang \cite{lang}. In both cases, the result is a Banachable space. Analogously, for Hilbert manifolds, the tangent space at any point is a Hilbertable space (see Definitions~\ref{def:banach} and~\ref{def:hilbert}).
\end{remark}

Let us recall some basic concepts analogous to the finite--dimensional case. For the following definitions, consider $f:X\to Y$ a $C^k$ map between Banach manifolds, with $k\geq 1$. Observe that the differential $\dd f(x)$ is naturally defined\footnote{Analogously to the finite--dimensional case, the differential of a $C^k$ map at a point is defined using local charts and its differentials, and the representation of the map. Thus, the definition is locally the same as Definition~\ref{def:differentialBanach}.} as a continuous operator $$\dd f(x):T_xX\la T_{f(x)}Y.$$

\begin{definition}\label{def:critical}
A point $x\in X$ is a {\em critical point}\index{Critical point} of $f$ if $\dd f(x)$ is not surjective or if $\ker\dd f(x)$ is not complemented. The set of all critical points of $f$, called the {\em critical set}\index{Critical point!set} of $f$, is denoted $\crit(f)$. A value $y\in\im f$ is a {\em critical value} of $f$ if there exists a critical point $x$ in its preimage $f^{-1}(\{y\})$.\index{Critical value}
\end{definition}

\begin{remark}
Notice that if $f:X\to\R$ is a function, the condition that $\dd f(x)$ is not surjective is equivalent to $\dd f(x)$ being zero, since the counter domain is one--dimensional.
\end{remark}

\begin{lemma}\label{le:critclosed}
Let $f:X\to Y$ be a $C^k$ map between Banach manifolds. Then the critical set $\crit(f)$ is closed in $X$.
\end{lemma}

\begin{proof}
Consider a sequence of critical points $\{x_n\}_{n\in\N}$ that converge to $x_\infty\in X$. Considering a local chart around $x_\infty$, we may assume that $f$ is defined between Banach spaces. Thus, we may consider the sequence of continuous operators $\{\dd f(x_n)\}_{n\in\N}$ between these spaces. Since $x_n$ are critical points, either $\dd f(x_n)$ is non surjective or $\ker\dd f(x_n)$ is not complemented. From Lemma~\ref{le:tsobreab}, the subset of continuous operators with at least one of these properties is closed, hence the limit $\dd f(x_\infty)$ is either non surjective or has non complemented kernel. Therefore, $x_\infty\in\crit(f)$, concluding the proof.
\end{proof}

\begin{definition}\label{def:regular}
A point $x\in X$ is a {\em regular point}\index{Regular!point} of $f$ if $\dd f(x)$ is surjective and $\ker \dd f(x)$ is a complemented subspace. A value $y\in Y$ is a {\em regular value} of $f$ if every point $x \in f^{-1}(y)$ in its preimage is a regular point of $f$.\index{Regular!value} Notice that if $y\notin\im f$, then $y$ is automatically a regular value of $f$.
\end{definition}

\begin{remark}
We consider points $x\in X$ such that $\ker\dd f(x)$ is not complemented\footnote{Recall that subspaces of Banach spaces may not be complemented, see Definition~\ref{def:complement} and Example~\ref{ex:noncomplement}.} to be {\em critical points}. This is in order to have the property that a point $x\in X$ is either regular or critical. In the regular case, a closed complement of $\ker\dd f(x)$ will be necessary for instance in Proposition~\ref{prop:regularvalue} to have that the preimage of a regular value is a submanifold, see Definition~\ref{def:submnfldchart}.
\end{remark}

We now approach the matter of immersions, embeddings and submersions of Banach manifolds, where further assumptions must be made relatively to the finite--dimensional setting. Namely, such assumptions deal again with the problem that in infinite--dimensional TVS, a vector subspace is not necessarily closed, see Definition~\ref{def:complement} and Lemma~\ref{le:finitecomplemented}.

\begin{definition}\label{def:immersion}
A $C^k$ map $f:X\to Y$ between Banach manifolds is an \emph{immersion at $x\in X$}\index{Banach manifold!immersion} if $\dd f(x):T_xX\to T_{f(x)}Y$ is injective and its image is a closed and complemented subspace (see Definition~\ref{def:complement}). If this property holds for all $x\in X$, then $f$ is said to be an {\em immersion}. In addition, if $f:X\to f(X)$ is a homeomorphism, where $f(X)$ is endowed with the subspace topology, then $f$ is called an {\em embedding}.\index{Banach manifold!embedding}
\end{definition}

\begin{lemma}\label{le:invemb}
Let $f:X\to Y$ be a $C^k$ map between Banach manifolds. If $f$ admits a $C^1$ left inverse $g:Y\to X$, then $f$ is an embedding.
\end{lemma}

\begin{proof}
Since $f$ admits a left inverse, it is injective. Differentiating the identity $g\circ f=\id$ at $x\in X$, it follows that $$\dd g(f(x))\dd f(x)=\id,$$ hence also $\dd f(x)$ is injective, since it admits the left inverse $\dd g(f(x))$.

Moreover, $\im\dd f(x)$ is a closed subspace. Indeed, if $\{w_n\}_{n\in\N}$ is a convergent sequence in $T_{f(x)}Y$ to $w_\infty$, the sequence $\{v_n\}_{n\in\N}$, given by $v_n=\dd g(f(x))w_n$, converges to $v_\infty=\dd g(f(x))w_\infty\in T_xX$ due to continuity of this left inverse of $\dd f(x)$. Therefore $w_\infty=\dd f(x)v_\infty$ is in the image of $\dd f(x)$, which is hence closed. In addition, $\ker\dd g(f(x))$ is clearly a closed complement of $\im\dd f(x)$. Thus $f$ is an injective immersion.

Finally, the inverse of the bijection $f:X\to f(X)$ coincides with $g|_{f(X)}$, which is a continuous map. Therefore $f:X\to f(X)$ is bijective, continuous and has continuous inverse, and hence is a homeomorphism. Therefore $f$ is an embedding, concluding the proof.
\end{proof}

\begin{corollary}\label{cor:secembed}
Let $E$ be a $C^1$ (Banach) vector bundle\footnote{As remarked in Chapter~\ref{chap1}, most definitions given for bundles over finite--dimensional smooth manifolds can be extended to an infinite--dimensional context. Since this extension is absolutely natural and intuitive, we will not give details of the theory of fiber bundles and connections over Banach manifolds, stressing however its analogy with its finite--dimensional counterpart presented in Chapter~\ref{chap1}.} over a Banach manifold $X$. Then every section $s\in\sect^k(E)$ is an embedding.
\end{corollary}

\begin{proof}
Given $s:X\to E$ a section, if $\pi:E\to X$ is the projection of $E$, then $\pi\circ s=\id$. Hence $s$ admits a $C^1$ left inverse, and thus, from Lemma~\ref{le:invemb} is an embedding.
\end{proof}

\begin{remark}\label{re:nullsectionsubmnfld}
Consider again $E$ a $C^1$ (Banach) vector bundle over a Banach manifold $X$. Corollary~\ref{cor:secembed} gives a formal proof of the legitimacy of the identification of the null section $\mathbf 0_E$ with the base manifold $X$, addressed in the finite--dimensional case in Remark~\ref{re:nullsection}. Since $$\mathbf 0_E:X\ni x\longmapsto (x,0_x)\in E,$$ where $0_x\in E_x$ is the zero, the image of the embedding $\mathbf 0_E:X\to E$ consists of points of the form $(x,0_x)$, which is hence canonically identified with $x\in X$. 

Henceforth, using this identification, we shall refer to the null section $\mathbf 0_E$ both as a section or submanifold of $E$.
\end{remark}

\begin{definition}\label{def:submersion}
A $C^k$ map $f:X\to Y$ between Banach manifolds is a \emph{submersion at $x\in X$}\index{Banach manifold!submersion} if $x$ is a regular point of $f$, i.e., if the differential $\dd f(x):T_xX\to T_{f(x)}Y$ is surjective and $\ker\dd f(x)$ is complemented in $T_xX$. If this property holds for all $x\in X$, then $f$ is said to be a {\em submersion}.
\end{definition}

\begin{remark}\label{re:localforms}
There are statements on the local form of immersions and of submersions similar to their finite--dimensional correspondents that hold under the generalized definitions above. These are {\em local} results and follow from the Inverse Function Theorem (in Banach manifolds), analogously to the usual finite--dimensional case. More precisely, consider a $C^k$ map $f:X\to Y$ between Banach manifolds. If $f$ is an immersion at $x\in X$, given a local chart around $x$, there exists a local chart around $f(x)$ such that $f$ is represented in these local charts by a linear inclusion of (open subsets of) Banach spaces. Similarly, if $f$ is a submersion at $x\in X$, given a local chart around $f(x)$, there exists a local chart around $x$ such that $f$ is represented in these local charts by a linear projection of (open subsets of) Banach spaces.
\end{remark}

\begin{remark}
Observe that if $X$ and $Y$ in Definition~\ref{def:submersion} are Hilbert manifolds, all closed subspaces are automatically complemented (by its {\em orthogonal} complement) and the definition coincide with the finite--dimensional version. In Definition~\ref{def:immersion} however, even if $X$ and $Y$ are Hilbert manifolds, the image of the differential $\dd i(x)$ may not be a closed subspace.
\end{remark}

Let us consider the particular case of $C^k$ {\em functions} defined on Banach manifolds, i.e., maps $f:X\to\R$ of class $C^k$. At each $x\in X$, the differential $\dd f(x):T_xX\to\R$ is a bounded linear functional, hence $\dd f$ may be regarded as a $C^{k-1}$ section of the cotangent bundle of $X$, $$\dd f:X\la TX^*.$$ Standard arguments prove that analogously to the finite--dimensional case, Banach vector bundles always admit connections that satisfy the same conditions of Definition~\ref{def:connection}. Thus, we may consider a connection $\nabla$ on $TX^*$.

\begin{definition}\label{def:hessnabla}
Let $\nabla$ be a connection on $TX^*$. The {\em $\nabla$--Hessian}\index{$\nabla$--Hessian} of a function $f\in C^k(X)$ if given by the $(0,2)$--tensor
\begin{equation}
\hess^\nabla(f)=\nabla(\dd f)\in\sect^{k-2}(TX^*\otimes TX^*).
\end{equation}
\end{definition}

\begin{lemma}
For any connection $\nabla$ on $TX^*$ and vector fields $X,Y\in\sect^{k-2}(TX)$,
\begin{equation}\label{eq:hessnablaf}
\hess^\nabla(f)(X,Y)=X(Y(f))-\nabla_X Y(f).
\end{equation}
In addition, $\nabla$ is symmetric,\footnote{See Definition~\ref{def:symflat}.} if and only if $\hess^\nabla$ is symmetric.
\end{lemma}

\begin{proof}
From the Definition~\ref{def:hessnabla} and elementary properties of connections,
\begin{eqnarray*}
\hess^\nabla(f)(X,Y) &=& (\nabla(\dd f))(X,Y)\\
&=& (\nabla_X \dd f)(Y) \\
&\stackrel{\eqref{eq:nablarform}}{=}& X(\dd f(Y))-\dd f(\nabla_X Y) \\
&=& X(Y(f))-\nabla_X Y(f),
\end{eqnarray*}
which proves \eqref{eq:hessnablaf}. Moreover,
\begin{eqnarray*}
\hess^\nabla(f)(X,Y) - \hess^\nabla(f)(Y,X)&\stackrel{\eqref{eq:hessnablaf}}{=}& X(Y(f))-\nabla_X Y(f) \\
&& \hspace{0.5cm}-Y(X(f))+\nabla_Y X(f)\\
&=& -\big(\nabla_X Y-\nabla_Y X-[X,Y]\big)(f) \\
&\stackrel{\eqref{eq:T}}{=}& -T^\nabla(X,Y)(f),
\end{eqnarray*}
hence $\hess^\nabla$ is symmetric if and only if $\nabla$ is symmetric.
\end{proof}

\begin{corollary}\label{cor:nablavsf}
If $x_0\in X$ is a critical point of $f\in C^k(X)$, then \eqref{eq:hessnablaf} does not depend on the choice of the connection $\nabla$.
\end{corollary}

\begin{proof}
Since $x_0$ is a critical point, $\dd f(x_0)=0$. Thus, for any $X,Y\in\sect^{k-2}(TX^*)$,
\begin{eqnarray}\label{eq:connectionsagreeat0}
\hess^\nabla(f)(x_0)(X,Y) &=& X(Y(f))(x_0)-\nabla_X Y(f)(x_0) \nonumber\\
&=& X(Y(f))(x_0)-\dd f(x_0)[(\nabla_X Y)(x_0)] \\
&=& X(Y(f))(x_0).\nonumber\qedhere
\end{eqnarray}
\end{proof}

Therefore, we may establish the following.

\begin{definition}\label{def:hess}
Consider a function $f\in C^k(X)$ and $x_0$ a critical point of $f$. Then the {\em Hessian}\index{Hessian} of $f$ is defined as the symmetric bilinear form $$\hess(f)(x_0)=\hess^\nabla(f)(x_0),$$ for {\em any choice}\footnote{From Corollary~\ref{cor:nablavsf}, the above definition does not depend on this choice.} of connection $\nabla$ on $TX^*$.
\end{definition}

\begin{remark}\label{re:ttx}
Corollary~\ref{cor:nablavsf} is used above to guarantee that $\hess(f)(x_0)$ is well--defined on critical points $x_0$ of $f$ without the use of a connection on $TX^*$. Let us use a different approach to verify the same result.

Recall that from Remark~\ref{re:connectionhor}, the choice of a connection on $TX^*$ is equivalent to the choice of a horizontal subbundle of $TX^*$ with certain properties. Notice also that at a critical point $x_0$, the section $\dd f(x_0)$ coincides with the null section $\mathbf 0_{TX^*}$. Precisely in this case it is possible to have a canonical choice of horizontal (and vertical) subbundle of $TX^*$, using the identification of $\mathbf 0_{TX^*}$ with $X$ as submanifold of $TX^*$, see Remark~\ref{re:nullsectionsubmnfld}. Namely, there is a canonical decomposition of $T_{(x_0,0_{x_0})}\mathbf 0_{TX^*}$ in horizontal and vertical parts, where $0_{x_0}\in T_{x_0}X^*$ the zero of this vector space, respectively tangent to $\mathbf 0_{TX^*}$ and to the fibers of $TX^*$. This is totally analogous to its finite--dimensional counterpart \eqref{eq:tangentnullsec}, that gives a decomposition in horizontal and vertical parts of the tangent space to a vector bundle at its null section.

More precisely, the tangent space to the null section of $TX^*$ at $(x_0,0_{x_0})$ is canonically identified as
\begin{equation}\label{eq:ttxident1}
T_{(x_0,0_x)}\mathbf 0_{TX^*}\cong T_{x_0}X,
\end{equation}
and hence
\begin{equation}\label{eq:ttxident2}
T_{(x_0,0_{x_0})}TX^*\cong T_{x_0}X\oplus T_{x_0}X^*.
\end{equation}
The existence of this natural choice of horizontal and vertical parts for $T_{(x_0,0_{x_0})}TX^*$ at critical points $x_0$ guarantees that $\hess(f)(x_0)$ is well--defined without a connection. Indeed, this is equivalent to state that all connections agree at the null section, as proved with \eqref{eq:connectionsagreeat0}.

Such decomposition of the tangent space to cotangent bundle at the null section will also be used several times in the sequel for other purposes.
\end{remark}

\begin{remark}\label{re:equivalenthess}
There are another (equivalent) ways of defining $\hess(f)$ on a critical point $x_0$ without using a connection on $TX^*$. Namely, let $$\hess(f)(x_0)(v,v)=\dfrac{\dd^2}{\dd t^2}(f\circ\gamma)\Big|_{t=0},$$ where $\gamma:(-\varepsilon,\varepsilon)\to X$ is a $C^2$ curve with $\gamma(0)=x_0$ and $\dot\gamma(0)=v$. Polarizing the above formula, one obtains an equivalent definition of $\hess(f)(x_0)$.

In the same sense, since we have the $C^{k-1}$ map $\dd f:X\to TX^*$, it is possible to derive again this map at critical points $x_0$ of $f$ obtaining
\begin{eqnarray}\label{eq:hessfx0}
\dd^2 f(x_0): T_{x_0}X &\la &T_{(x_0,0_{x_0})}TX^*\cong T_{x_0}X\oplus T_{x_0}X^* \nonumber \\
v &\longmapsto & \big(v,\hess(f)(x_0)(v,\cdot\,)\big).
\end{eqnarray}
This could also be adopted as definition of $\hess(f)(x_0)$ at a critical point.
\end{remark}

\begin{remark}\label{re:hessnoncrit}
We will only use the Hessian of functions {\em on its critical points}. Although Definition~\ref{def:hess} extends the concept of Hessian of functions in Euclidean space to Banach manifolds, differently from the finite--dimensional case, $\hess(f)$ cannot be defined in general without the use of a connection, as in Definition~\ref{def:hessnabla}.
\end{remark}

With the notion of Hessian of a function on its critical points, we may now classify critical points according to its degeneracy.

\begin{definition}\label{def:degnondegstdeg}
Let $x_0\in X$ be a critical point of a function $f\in C^k(X)$, and consider the Hessian of $f$ under identification \eqref{ident:bilin},
\begin{equation}\label{eq:hessfx0op}
\hess(f)(x_0):T_{x_0}X\la T_{x_0}X^*.
\end{equation}
Then the critical point $x_0$ is said to be
\begin{itemize}
\item[(i)] {\em degenerate}\index{Critical point!degenerate} if $\hess(f)(x_0)$ has nontrivial kernel;
\item[(ii)] {\em nondegenerate}\index{Critical point!nondegenerate} if $\hess(f)(x_0)$ is injective;
\item[(iii)] {\em strongly nondegenerate}\index{Critical point!strongly nondegenerate} if $\hess(f)(x_0)$ is an isomorphism.
\end{itemize}
The function $f\in C^k(X)$ is said to be a {\em Morse function}\index{Morse function} if all of its critical points are strongly nondegenerate.
\end{definition}

Using an analogy with the finite--dimensional case, this gives a qualitative description of the behavior of $f$ near these critical points. For instance, in the nondegenerate case, if the Hessian of a function $f:\R^m\to\R$ is positive--definite or negative--definite at a critical point, then this point is a local minimum or maximum of $f$, respectively. Degeneracy occurs for instance when the critical point is a saddle point. In the infinite--dimensional context however, there are several degrees of degeneracy, since non injectivity and non surjectivity of the Hessian are not equivalent (unless it is a Fredholm operator, recall~\ref{def:fredholmop}).

\begin{lemma}
Let $X$ be a Hilbert manifold and suppose $f\in C^k(X)$ has only nondegenerate critical points. If at every critical point $x_0$ of $f$ the operator $\hess(f)$ is Fredholm, then $f$ is Morse.
\end{lemma}

\begin{proof}
Let $x_0$ be any critical point of $f$. Since it is nondegenerate, the operator $\hess(f)(x_0)$ regarded as \eqref{eq:hessfx0op} has trivial kernel. Assuming this is a Fredholm operator, since it represents a symmetric bilinear form, is clearly self--adjoint. From Lemma~\ref{le:selfadjointzero}, it has index zero. Thus, from Lemma~\ref{le:fred0}, it is also surjective, hence $x_0$ is strongly nondegenerate. Therefore, $f$ is Morse.
\end{proof}

Let us now define the nonlinear version of Fredholmness.

\begin{definition}\label{def:nonlinfred}
A $C^k$ map $f:X\to Y$ between Banach manifolds is a \emph{nonlinear Fredholm map}\index{Fredholm!nonlinear map} if $\dd f(x):T_xX\to T_{f(x)}Y$ is a Fredholm map for all $x\in X$. The \emph{index}\index{Fredholm!nonlinear map!index} $\ind(f)$ at each connected component of $X$ is defined as the index $\ind(\dd f(x))$ of the linear Fredholm map $\dd f(x)$ at any $x$ in such connected component, since from continuity of the index, it is constant in each connected component of $M$.
\end{definition}

\begin{remark}
For further applications, $X$ will be connected and hence the index of any nonlinear Fredholm map will not depend on a choice of connected component as in the general case of the above definition.
\end{remark}

In order to continue our brief exposition of basic elements of infinite--dimensional differential geometry, we introduce the concept of submanifolds.

\begin{definition}\label{def:submnfldchart}
Consider a subset $S$ of a Banach manifold $X$. A chart $(U,\varphi)$ of $X$ is a {\em submanifold chart}\index{Chart!submanifold} of $S$ if there exists a closed complemented subspace $Y\subset X$ such that $$\varphi(U\cap S)=\varphi(U)\cap Y.$$ If there exists a $C^k$ atlas of $X$ with submanifold charts whose domain cover $S$, then $S$ is a $C^k$ {\em embedded Banach submanifold}\index{Banach manifold!submanifold} of $X$.
\end{definition}

\begin{remark}\label{re:submnflds}
Obviously, the submanifold charts of a submanifold $S$ can be restricted to form a $C^k$ atlas of $S$, so that $S$ is also a $C^k$ Banach manifold. The inclusion map $i:S\hookrightarrow X$ is a $C^k$ embedding. In addition, for each $x\in S$, the differential $\dd i(x)$ identifies the tangent space $T_xS$ with a closed complemented subspace of $T_xX$.

Conversely, if an inclusion $i:S\hookrightarrow X$ satisfies the above properties, then $i(S)$ is a $C^k$ embedded submanifold of $X$.
\end{remark}

We now deal with \emph{transversality}\footnote{For a detailed discussion of transversality in the finite--dimensional context, together with applications to differential topology and Morse theory, see Guillemin and Pollack \cite{guilleminpollack}. A concise extension of most results to the infinite--dimensional case is given in Lang \cite{lang}.}, that will play a fundamental role in the sequel. It can be naively seen as a condition to generalize the following basic result concerning regular values of a $C^k$ map, whose proof can be found in Lang \cite{lang} using simply the local form of submersions.

\begin{proposition}\label{prop:regularvalue}
Consider a $C^k$ map $f:X\to Y$ between Banach manifolds. If $c\in Y$ is a regular value\footnote{See Definition~\ref{def:regular}.} of $f$, then $f^{-1}(c)$ is a $C^k$ Banach submanifold of $X$. Moreover, its tangent space at any $x\in f^{-1}(c)$ is given by the complemented Banach subspace of $T_xX$ $$T_xf^{-1}(c)=\ker\left(\dd f(x)\right).$$
\end{proposition}

\begin{remark}\label{re:regularvaluehilbert}
The above result has an obvious version for Hilbert manifolds, with identical proof.
\end{remark}

A natural extension of Proposition~\ref{prop:regularvalue} is to consider the preimage $f^{-1}(Z)$ of a submanifold $Z\subset Y$, instead of the preimage of a regular value $f^{-1}(c)$. As we will prove, \emph{transversality} of $f$ to $Z$ is a sufficient condition for $f^{-1}(Z)$ to be a submanifold of $X$.

\begin{definition}\label{def:transversality}
Consider a $C^k$ map $f:X\to Y$ between Banach manifolds, and $Z\subset Y$ a (smooth) submanifold. Then {\em $f$ is transverse to $Z$ at $x$} if $\dd f(x)^{-1}\left[T_{f(x)}Z\right]$ is complemented in $T_xX$ and
\begin{equation}\label{eq:transversality}
\im \dd f(x) +T_{f(x)}Z=T_{f(x)}Y.
\end{equation}
If this happens for all $x\in f^{-1}(Z)$, then $f$ is said to be {\em transverse} to $Z$.\index{Transverse map}
\end{definition}

\begin{remark}\label{re:transversalityeq}
Equation \eqref{eq:transversality} is clearly equivalent to surjectivity of the following composite map, for all $x\in f^{-1}(Z)$
\begin{equation}\label{eq:transversality2}
T_xX\xrightarrow{\;\;\dd f(x)\;\;} T_{f(x)}Y\xrightarrow{\;\;q\;\;}\dfrac{T_{f(x)}Y}{T_{f(x)}Z},
\end{equation}
where $q$ denotes the quotient map. Notice that if $f$ is a submersion (recall Definition~\ref{def:submersion}), then it automatically satisfies this condition.
\end{remark}

\begin{remark}\label{re:abouttransv}
The name \emph{transversality} is justified with the situation where $f:X\hookrightarrow Y$ is an immersion. Transversality of $f$ to a submanifold $Z\subset Y$ is equivalent to the submanifolds $f(X)$ and $Z$ being {\em transverse},\index{Transverse submanifolds} i.e., the sum of their tangent spaces at any $y\in f(X)\cap Z$ is the entire $T_yY$, see Figure \ref{fig:transversality}. In addition, it is easy to see that endowing the space $C^k(X,Y)$ with an adequate natural topology, transversality to $Z$ is an {\em open} condition. This means that if $f$ is transverse to $Z$, there exists an open neighborhood of $f$ of $C^k$ maps transverse to $Z$. Transversality to a fixed submanifold is also {\em open} in another sense, see Remark~\ref{re:transvopen}. In Chapter~\ref{chap4}, we will also prove that under certain hypotheses {\em transversality} is a generic phenomenon (see the Transversality Theorem~\ref{thm:transversality}).

\begin{figure}[htf]
\begin{center}
\vspace{-0.3cm}
\includegraphics[scale=1]{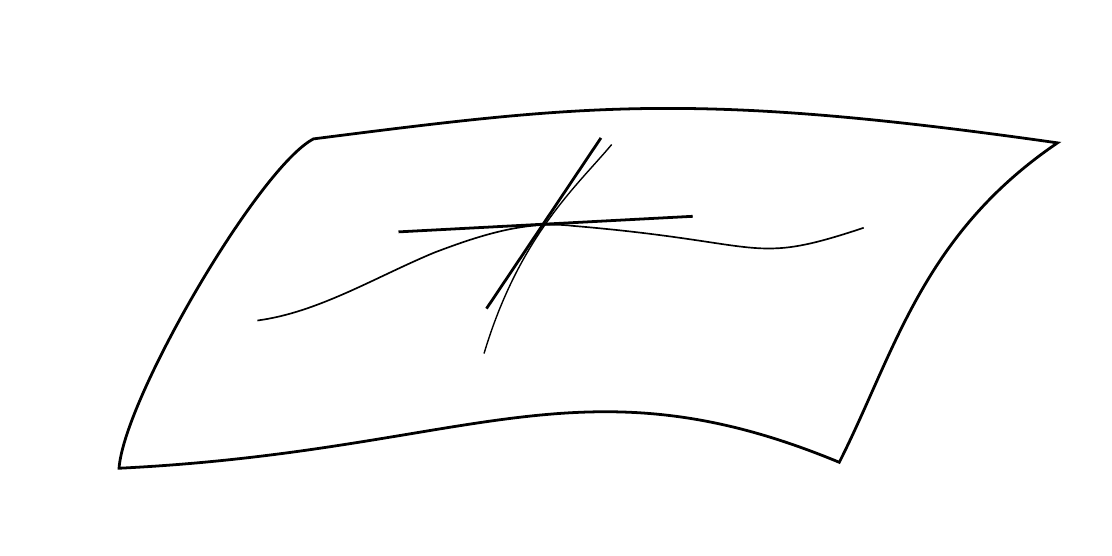}
\begin{pgfpicture}
\pgfputat{\pgfxy(-6.4,1.7)}{\pgfbox[center,center]{$f(X)$}}
\pgfputat{\pgfxy(-4,2.8)}{\pgfbox[center,center]{$Z$}}
\pgfputat{\pgfxy(-1.7,1.8)}{\pgfbox[center,center]{$Y$}}
\end{pgfpicture}
\end{center}
\vspace{-0.1cm}
\caption{Transverse submanifolds $f(X)$ and $Z$ of $Y$.}\label{fig:transversality}
\end{figure}
\end{remark}

\begin{proposition}\label{prop:transvsubmnfld}
If a $C^k$ map $f:X\to Y$ is transverse to a (smooth) submanifold $Z\subset Y$, then $f^{-1}(Z)$ is a $C^k$ embedded submanifold of $X$, and the tangent space at any $x\in f^{-1}(Z)$ is given by
\begin{equation}\label{eq:tf-1}
T_xf^{-1}(Z)=\dd f(x)^{-1}\left[T_{f(x)}Z\right].
\end{equation}
\end{proposition}

\begin{proof}
Since the condition of being a submanifold is local, it suffices to prove that for any $x\in f^{-1}(Z)$, there exists an open neighborhood $U\ni x$, such that $f^{-1}(Z)\cap U$ is a $C^k$ submanifold. For a fixed $x\in f^{-1}(Z)$, there exists a (smooth) submanifold chart $(V,\psi)$ of $Z$ around $f(x)$ in $Y$, with $\psi:V\to V_1\oplus V_2$, $\psi(Z\cap V)\subset V_1$. Let $p_2:V_1\oplus V_2\to V_2$ be the projection. Then $Z\cap V=(p_2\circ\psi)^{-1}(0)$, and $0\in V_2$ is a regular value for $p_2\circ\psi$. Since $f$ is continuous, there exists an open neighborhood $U\ni x$ such that $f(U)\subset V$.

Consider the composite map $g$ given by
\begin{equation}
g:U\xrightarrow{\;\;f\;\;} V\xrightarrow{\;\;\psi\;\;} V_1\times V_2\xrightarrow{\;\;p_2\;\;} V_2,
\end{equation}
i.e., $g=p_2\circ\psi\circ f$, which is clearly $C^k$. We claim that $\dd g(x)$ is surjective. In fact, for each $a\in T_0 V_2$, since $0$ is a regular value for $p_2\circ\psi$, there exists $b\in T_{f(x)}Y$, such that $\dd (p_2\circ\psi)(f(x))b=a$. From transversality of $f$ to $Z$ (recall \eqref{eq:transversality}), there exists $c\in T_xX$ and $d\in T_{f(x)}Z$ such that $\dd f(x)c+d=b$. Since $p_2\circ\psi$ is constant on $Z$, $\dd (p_2\circ\psi)(f(x))d=0$. It follows that 
\begin{eqnarray*}
\dd g(x)c &=& \dd (p_2\circ\psi\circ f)(x)c \\
&=& \dd (p_2\circ\psi)(f(x))\dd f(x)c \\
&=& \dd (p_2\circ\psi)(f(x))\dd f(x)c+\dd(p_2\circ\psi)(f(x))d \\
&=& \dd (p_2\circ\psi)(f(x))b \\
&=& a.
\end{eqnarray*}

Thus, reducing $U$ if necessary, $g$ is a $C^k$ submersion at all points of $g^{-1}(0)$, which is hence a regular value of $g$. From Proposition~\ref{prop:regularvalue}, it follows that $g^{-1}(0)=f^{-1}(Z)\cap U$ is a $C^k$ submanifold of $X$, which proves that $f^{-1}(Z)$ is a $C^k$ submanifold, as discussed above. In addition, $$T_xf^{-1}(Z)=T_xg^{-1}(0)=\ker\dd g(x).$$ Since $(p_2\circ\psi)^{-1}(0)=Z\cap V$ and $0$ is a regular value for this map, applying the second part of Proposition~\ref{prop:regularvalue}, it follows that $$T_{f(x)}Z=\ker\dd (p_2\circ\psi)(f(x)).$$ Moreover, $\dd g(x)=\dd (p_2\circ\psi)(f(x))\dd f(x)$, hence $v\in\ker\dd g(x)$ if and only if $\dd f(x)v\in\ker\dd (p_2\circ\psi)(f(x))$. This implies that $T_xf^{-1}(Z)$ coincides with $\dd f(x)^{-1}\left[T_{f(x)}Z\right]$, concluding the proof. Notice that $T_x f^{-1}(Z)$ is an automatically complemented subspace as a consequence of Definition~\ref{def:transversality}.
\end{proof}

\begin{remark}\label{re:transvsubmnfldhilbert}
Analogously to Remark~\ref{re:regularvaluehilbert}, the above result has an obvious version for Hilbert manifolds, with identical proof.
\end{remark}

Let us give an example of a nontrivial use of transversality to characterize Morse functions.

\begin{proposition}\label{prop:morseifftransv}
Consider $X$ a Hilbert manifold and $f\in C^k(X)$ a function such that its Hessian, regarded as \eqref{eq:hessfx0op}, $$\hess(f)(x_0):T_{x_0}X\la T_{x_0}X^*\cong T_{x_0}X$$ is a Fredholm operator at every critical point $x_0$ of $f$. Then $f$ is Morse if and only if $\dd f$ is transverse to the null section $\mathbf 0_{TX^*}$ of the cotangent bundle.
\end{proposition}

\begin{proof}
Recall that $\dd f$ may be regarded as section $\dd f\in\sect^{k-1}(TX^*)$. First, notice that $(\dd f)^{-1}(\mathbf 0_{TX^*})$ is precisely the set of critical points of $f$. Let $x_0$ be one such critical point. We will use the decomposition of $T_{(x_0,0_{x_0})}TX^*$ in horizontal and vertical parts given in Remark~\ref{re:ttx}. Namely, from \eqref{eq:ttxident2}, the operator $\dd^2 f(x_0)$ is given by \eqref{eq:hessfx0}. Therefore, $\big(\dd^2 f(x_0)\big)^{-1}\left[T_{(x_0,0_{x_0})}\mathbf 0_{TX^*}\right]$ is automatically complemented on $T_{x_0}X$, since from \eqref{eq:ttxident1} and \eqref{eq:hessfx0} it is possible to identify this space with the whole $T_{x_0}X$.

Thus, the condition that $\dd f$ be transverse to $\mathbf 0_{TX^*}$ is now precisely the surjectivity of the following composite map for every critical point $x_0$ of $f$, see Remark~\ref{re:transversalityeq}.
\begin{equation}\label{eq:transversalityhessf}
T_{x_0}X\xrightarrow{\;\;\dd^2 f(x_0)\;\;} T_{(x_0,0_{x_0})}TX^*\xrightarrow{\;\;\;\;}\dfrac{T_{(x_0,0_{x_0})}TX^*}{T_{(x_0,0_{x_0})}\mathbf 0_{TX^*}}\stackrel{\eqref{eq:ttxident1}\;\eqref{eq:ttxident2}}{\cong} T_{x_0}X^*
\end{equation}
It is easy to see that this composite map coincides with $\hess(f)(x_0)$. Since this is a self--adjoint Fredholm operator, from Lemma~\ref{le:selfadjointzero} it has index zero. Thus, it is injective if and only if it is surjective. This proves that $\dd f$ is transverse to $\mathbf 0_{TX^*}$ if and only if every critical point $x_0$ is strongly nondegenerate. Hence $f$ is Morse if and only if $\dd f$ is transverse to $\mathbf 0_{TX^*}$.
\end{proof}

\begin{corollary}
Let $f\in C^k(X)$ be a function on a Hilbert manifold $X$. Then if $\dd f:X\to TX^*$ is a nonlinear Fredholm map, transversality of $\dd f$ to $\mathbf 0_{TX^*}$ is equivalent to $f$ being Morse.
\end{corollary}

\begin{proof}
If $\dd f:X\to TX^*$ is a nonlinear Fredholm map, then $\hess(f)(x_0):T_{x_0}X\to T_{x_0}X^*\cong T_{x_0}X$ is Fredholm at every critical point $x_0$ of $f$, and Proposition~\ref{prop:morseifftransv} applies.
\end{proof}

To end our discussion of transversality, we prove the following auxiliary result of independent interest which asserts that the transversality condition is open.

\begin{lemma}\label{le:transvopen}
Let $f:X\to Y$ be a $C^k$ map between Banach manifolds and $Z$ a Banach submanifold of $Y$. Then the following is an open subset of $f^{-1}(Z)$,
\begin{equation}
\mathfrak A=\{x\in f^{-1}(Z):f \mbox{ is transverse to }Z \mbox{ at } x\}
\end{equation}
\end{lemma}

\begin{proof}
Recall that, from Definition~\ref{def:transversality}, $x\in\mathfrak A$ if and only if $\dd f(x)^{-1}\left[T_{f(x)}Z\right]$ is complemented in $T_xX$ and
$\im \dd f(x) +T_{f(x)}Z=T_{f(x)}Y.$ Given $x_0\in\mathfrak A$, we will prove that there exists $U$ an open neighborhood of $x_0$, with $U\cap f^{-1}(Z)\subset\mathfrak A$.

Since $Z$ is a submanifold, there exists $\varphi:V\to B$ a submanifold chart\footnote{See Definition~\ref{def:submnfldchart}.} of $Z$ around $f(x_0)$, where $B$ is a Banach space. Standard arguments prove that this submanifold chart may be taken satisfying $V\cap Z=\phi^{-1}(0)$, where $\phi:V\to B$ is a smooth map having $0\in B$ as regular value. Consider
\begin{equation*}
U=\{x\in f^{-1}(V):\im\dd(\phi\circ f)(x)=B\mbox{ and } \ker\dd(\phi\circ f)(x) \mbox{ is complemented}\}
\end{equation*}
We claim that this is an open neighborhood of $x_0\in X$. Indeed, let $\{x_n\}_{n\in\N}$ be a sequence of elements of $X\setminus U$ that converges to $x_\infty\in X$. By taking a local chart around $x_\infty$, we may assume that $\varphi\circ f$ is defined between open subsets of Banach spaces. Lemma~\ref{le:tsobreab} then implies that $x_\infty\in X\setminus U$, hence $U$ is open in $X$. Finally, $U$ clearly satisfies $x_0\in U\cap f^{-1}(Z)\subset\mathfrak A$, concluding the proof that $\mathcal A$ is open.
\end{proof}

\begin{remark}
Lemma~\ref{le:transvopen} asserts that given any submanifold $Z$ of $Y$, the subset $\mathfrak A$ of the preimage $f^{-1}(Z)$ where $f$ is transverse to $Z$ is open in $f^{-1}(Z)$. This means that even if $f^{-1}(Z)$ is not a submanifold of $X$, the {\em natural} part of $f^{-1}(Z)$ that is candidate to be a submanifold of $X$ is given by the intersection of an open subset of $X$ and $f^{-1}(Z)$. This has important consequences and also illustrates the situation that occurs in the presence of singularities, for instance in the case of the so--called {\em good orbifolds}, see Alexandrino and Bettiol \cite{ilgaag}.
\end{remark}

\begin{remark}\label{re:transvopen}
Another interesting interpretation of Lemma~\ref{le:transvopen} is that the condition of transversality to a fixed submanifold $Z$ is an {\em open} condition. Indeed, suppose $X$ is a submanifold of $Y$ and consider $f:X\to Y$ the inclusion. Then the set of points $x\in X\cap Z$ where these submanifolds are transverse is open in $X\cap Z$.
\end{remark}

\begin{proposition}\label{prop:dovalorinterm}
Let $X$ and $Y$ be Banach manifolds, $A$ a topological space and $f:A\times X\to Y$ a continuous map. Suppose there exists $S$ a submanifold of $Y$ with $\codim_Y S=1$, $a_0\in A$ and $x_0\in X$ such that $f(a_0,x_0)\in S$, $f(a_0,\cdot\,):X\to Y$ is of class $C^1$ and $$T_{f(a_0,x_0)}Y=T_{f(a_0,x_0)}S\oplus\im\frac{\partial f}{\partial x}(a_0,x_0).$$ Then there exists an open neighborhood $U$ of $a_0$ in $A$ such that for all $a\in U$, $S\cap\im f(a,\cdot\,)\neq\emptyset$.
\end{proposition}

\begin{proof}
Since the matter is local, by taking local charts we may assume without loss of generality that $X$ is an open subset of a Banach space, $x_0\in X$ the origin of this Banach space, $Y$ a Banach space and $S$ a closed subspace\footnote{In fact, for this to be possible it suffices to choose a local chart around $f(a_0,x_0)$ that is a submanifold chart of $S$ in $Y$, see Definition~\ref{def:submnfldchart}.} of $Y$. Since $\codim_Y S=1$ and $S$ is closed, there exists a continuous functional $\alpha\in Y^*$ such that $S=\ker\alpha$.

Without loss of generality, we may also suppose that $X$ is an open subset of $\R$. In fact, let $q:Y\to Y/S$ be the quotient map. Since $$Y=S\oplus\im\frac{\partial f}{\partial x}(a_0,0),$$ also $q\circ\frac{\partial f}{\partial x}(a_0,0)\vert_{T_xX}$ is surjective. From $\dim Y/S=1$, it follows that $q\circ\frac{\partial f}{\partial x}(a_0,0)\vert_{T_xX}$ is already surjective when restricted to (an adequate) one--dimensional subspace. Thus, we may assume\footnote{Notice that proving the result for a {\em smaller} $X$ automatically implies that the property remains valid for any {\em larger} $X$, since an enlargement of $X$ weakens the assertion $S\cap\im f(a,\cdot\,)\neq\emptyset$.} $X\subset\R$.

Consider now the composite $\widetilde f=\alpha\circ f:A\times X\subset A\times\R\to\R$. Then, since $f(a_0,0)\in S$, it follows that $\widetilde f(a_0,0)=0$ and the operator $\frac{\partial\widetilde f}{\partial x}(a_0,0):\R\to\R$, given by $\alpha\circ\frac{\partial f}{\partial x}(a_0,0)$ is surjective, for $\frac{\partial f}{\partial x}(a_0,0)$ is surjective. Hence it is an isomorphism, since it is surjective between vector spaces of the same (finite) dimension. Thus $\frac{\partial\widetilde f}{\partial x}(a_0,0)\neq 0$, hence there exists $\delta>0$ such that for $0<\varepsilon<\delta$, $$\widetilde f(a_0,-\varepsilon)\widetilde f(a_0,\varepsilon)<0.$$ From continuity of $\widetilde f$, it follows that there exists a neighborhood $U$ of $a_0\in A$ such that for $a\in U$, $\widetilde f(a_0,-\varepsilon)\widetilde f(a_0,\varepsilon)<0,$ hence $0\in\im\widetilde f(a,\cdot\,)$. Since $S=\ker\alpha$, this is equivalent to $S\cap\im f(a,\cdot\,)\neq\emptyset$, concluding the proof.
\end{proof}

\begin{remark}\label{re:codimalta}
Although it may seem that the hypothesis $\codim_Y S=1$ was strongly used, it may be replaced by $\codim_Y S=n<+\infty$. The functional $\alpha\in Y^*$ must then be replaced by a continuous operator $\alpha:Y\to\R^n$ with $\ker\alpha=S$, which has codimension $n$. Moreover, $X$ is assumed an open subset of $\R^n$, instead of $\R$, and the final argument of sign change when passing through the zero becomes a topological degree argument. Continuity of the topological degree then implies the existence of the desired open neighborhood of $a_0$, concluding the proof.
\end{remark}

We end this section briefly introducing {\em Riemann--Hilbert structures} on Hilbert manifolds. This is a natural generalization of Riemannian structures on finite--dimensional manifolds, recalled in Section~\ref{sec:metricsetc}. We will only define the concept of Riemannian metric in the infinite--dimensional setting, as a natural extension of Definition~\ref{def:srmetric}. For a comprehensive study of infinite--dimensional Riemannian geometry, see Lang \cite{lang}.

Obviously, since we shall endow each tangent space $T_xX$ of the infinite--dimensional manifold $X$ with an inner product, $X$ must be supposed a Hilbert manifold. Analogously to Definition~\ref{def:section} and Remark~\ref{re:secvetsp}, consider $\sect^k(TX^*\otimes TX^*)$ the vector space of $C^k$ sections $$B:X\la TX^*\otimes TX^*$$ of the tensor bundle $TX^*\otimes TX^*$ over the Hilbert manifold $X$. Consider also symmetric and skew--symmetric subbundles, analogously to Definition~\ref{def:symskewsym}.

\begin{definition}\label{def:metricinfinitemnfld}
A section $G\in\sect^k(TX^*\vee TX^*)$ is a {\em $C^k$ Riemannian metric}\index{Riemannian!metric!infinite--dimensional} on a Hilbert manifold $X$ if for all $x\in X$, the bilinear form $$G(x):T_xX\times T_xX\la\R$$ is $G(x)$ is a Hilbert inner product on the Hilbertable space $T_xX$.

A Hilbert manifold $X$ endowed with a Riemannian metric is said to be endowed with a {\em Riemann--Hilbert structure}.\index{Riemann--Hilbert structure}
\end{definition}

\begin{remark}\label{re:infinitedimmetricspace}
As in the finite--dimensional case, the presence of a Riemannian metric on a Hilbert manifold $X$ induces a metric space structure on $X$, and the definition of distance is similar to Definition~\ref{def:distance}.
\end{remark}

\begin{remark}\label{re:metricinfinitesmfld}
If $S\subset X$ is a Hilbert submanifold\footnote{See Definition~\ref{def:submnfldchart}.} of $X$, the restriction $G|_S$ is a section of the subbundle $TS^*\vee TS^*$ that automatically satisfies conditions in Definition~\ref{def:metricinfinitemnfld} over the Hilbert submanifold $S$. Therefore, the restriction of a Riemannian metric on a Hilbert manifold to a Hilbert submanifold gives a Riemannian metric on this submanifold.
\end{remark}

Although we will not discuss further topics of infinite--dimensional Riemannian geometry, it is possible to extend most definitions of Section~\ref{sec:metricsetc} to this context. For instance, notions of covariant derivative, geodesics and completeness can be defined. In particular, this is commonly used to analyze groups of diffeomorphisms of finite--dimensional manifolds, which are infinite--dimensional groups whose operation is left composition, which is continuous, but not differentiable. {\em Geodesics} in such groups are one--parameter families of diffeomorphisms that satisfy appropriate conditions, and its study allows to infer several conclusions about the underlying finite--dimensional manifold.

\section{Banach spaces of tensors}
\label{sec:banachspacetensors}

In this section, we are interested in endowing (affine subspaces of) the vector space $\sect^k(TM^*\vee TM^*)$ of $C^k$ symmetric $(0,2)$--tensors with a Banach space structure, recall Definitions~\ref{def:symskewsym} and~\ref{def:banach}. Such necessity arises from the fact that genericity of geodesic flow properties will be stated in terms of $C^k$ semi--Riemannian metrics, that are objects of this space of sections. However, the set $\met_\nu^k(M)$ lacks an adequate topological structure that allows to study the intended generic properties. For this reason, $\met_\nu^k(M)$ is seen as a subset of $\sect^k(TM^*\vee TM^*)$ whose topological and differentiable structures will be inherited from the intersection
\begin{equation}\label{eq:s}
S\cap\met_\nu^k(M)
\end{equation}
with a Banach subspace $S$ of $\sect^k(TM^*\vee TM^*)$. For technical reasons, we will need $S$ to be also separable among some other convenient conditions, see Definition~\ref{def:ckwhitbanachspace} and Remark~\ref{re:emptyinterior}.

\begin{remark}
The subject of this section is part of a widely explored area, namely Banach manifold structures on sets of sections of vector bundles.\footnote{This theory obviously contains special cases of sets of maps, for instance $C^k(M)$, see Example~\ref{ex:functions}.} The classic theory for sections of bundles over {\em compact} manifolds was developed mostly by Palais \cite{palais}. Recently, several extensions of this work to the {\em non compact} setting have been studied in \cite{eells,pictau}. In particular, we refer to Piccione and Tausk \cite{pictau} for a comprehensive description of the manifold structure of sets of maps between non compact manifolds.
\end{remark}

A few important considerations are worth mentioning, regarding the process to induce the desired structures on $\met_\nu^k(M)$ using \eqref{eq:s}.

First, the results in this section regarding Banach structures on vector spaces $\sect^k(E)$ of sections could be analogously done, at least its great majority, for any vector bundle $E$ over $M$, provided that $E$ has a connection, an inner product in each fiber $E_x$ and $M$ has a Riemannian metric. Nevertheless, we will only treat the case of tensor bundles
\begin{equation}\label{eq:tensorbundle}
E=(\otimes^s TM^*)\otimes (\otimes^r TM),
\end{equation}
recall Definitions~\ref{def:vectorbundle},~\ref{def:tensorbundle} and~\ref{def:connection}. By $E$ in this section, we always mean such a tensor bundle. In this particular case, both the necessary connection on $E$ and the inner product in each fiber $E_x$ are induced by the presence of (any) Riemannian metric on $M$ (see Theorem~\ref{thm:tensorconnection} and Definition~\ref{def:fiberproduct}).

Some further descriptions of $\sect^k(E)$ will be only studied for the particular case necessary for our applications, given by $r=0$ and $s=2$, i.e. $E=TM^*\otimes TM^*$; more precisely, concerning just symmetric sections. As a rule, results are usually stated in the most general context possible, however we stress that constructions of this section will be used exclusively in the case above mentioned of $\sect^k(TM^*\vee TM^*)$.

Second, the whole space $\sect^k(TM^*\vee TM^*)$ does not have a canonical Banach space structure for non compact base manifolds $M$, see Remarks~\ref{re:dependgr} and~\ref{re:dependgr2}. Since we are interested in the $C^k$ topology on the space of metrics, we shall define an appropriate $C^k$--norm of tensors, see \eqref{eq:sectionnorm}, and study the subspace $\sect^k_b(TM^*\vee TM^*)$ of sections with finite norm, see Definition~\ref{def:sectionnorm}). This will be proved to be a non separable Banach space (see Proposition~\ref{prop:banachspaceofsections} and Remark~\ref{re:separability}). In order to gain separability of the space of tensors, we shall make further restrictions to the subspace $\sect^k_0(TM^*\vee TM^*)$ of tensors whose norm {\em tends to zero at infinity}, see Definition~\ref{def:tendstozero} and Proposition~\ref{prop:tensorszeroseparable}. This will be our {\em typical} subspace $S$ used in \eqref{eq:s}. More generally, we describe sufficient abstract conditions on such a subspace $S$ in order to induce the adequate structures on $\met_\nu^k(M)$, see Definition~\ref{def:ckwhitbanachspace} and Remark~\ref{re:iiiiii}.

Third, we will fix an auxiliary metric $g_\mathrm A\in\met_\nu^k(M)$ such that the eigenvalue with minimum absolute value of the $g_\mathrm R$--symmetric operator $(g_\mathrm A)_x$ stays uniformly away from $0$. The appropriate space of metrics will be taken as an affine subspace of $\sect^k_b(TM^*\vee TM^*)$ containing $g_\mathrm A$, typically given by $$g_\mathrm A+\sect^k_0(TM^*\vee TM^*),$$ see Figure \ref{fig:sectionspaces}. This will guarantee that the intersection of such affine Banach space with $\met_\nu^k(M)$ is open and nonempty. Under these conditions genericity statements on metrics make sense. Without displacing $\sect^k_0(TM^*\vee TM^*)$ from the origin, the intersection with $\met_\nu^k(M)$ would have empty interior (see Remark~\ref{re:emptyinterior}), hence genericity on open subsets would be an empty statement. We also observe that this is a common setting for studying semi--Riemannian metrics, inspired by asymptotically flat space--times, see Remarks~\ref{re:physics0} and~\ref{re:physics}.

Fourth, the choice of $C^k$ regularity instead of smooth conditions is due to the fact that smooth assumptions would give rise to a Fr\'echet structure rather than a Banach structure in $\sect^\infty(TM^*\vee TM^*)$, see Remark~\ref{re:sectinftyfrechet}. Since the main tools necessary, such as the Sard--Smale Theorem~\ref{thm:sardsmale}, are only available for (separable) Banach spaces, we restrict to the $C^k$ case. However, once genericity of a certain property is established in the $C^k$--topology for all $k\geq k_0$, it is possible to apply a standard argument to extend it to the smooth case. This will be done for our genericity statements on Sections~\ref{sec:smooth1} and~\ref{sec:smooth2}.

Recall that $g_\mathrm R$ denotes a fixed smooth Riemannian metric on $M$ and $E$ a tensor bundle over $M$ given by \eqref{eq:tensorbundle}. We now briefly discuss how $g_\mathrm R$ induces a natural inner product, hence also a norm (called the {\em Hilbert--Schmidt norm}),\index{Hilbert--Schmidt norm} on each fiber
\begin{equation}\label{eq:tensorbundlefiber}
E_x=(\otimes^s T_xM^*)\otimes (\otimes^r T_xM).
\end{equation}

\begin{wrapfigure}{l}{2cm}
\vspace{-0.7cm}
\xymatrix@+15pt{T_xM^* \ar[r]^{g_\mathrm R(x)^{-1}} & T_xM\ar[d]^{i_x} \\
T_xM\ar[u]^{i_x}\ar[r]_{g_\mathrm R(x)} & T_xM^*
}
\vspace{-0.5cm}
\end{wrapfigure}
Since each tangent space $T_xM$ has the inner product $g_\mathrm R(x)$, the dual space $T_xM^*$ has the dual inner product $g_\mathrm R(x)^{-1}$. It is the only linear map that makes the diagram commutative, where $i_x:T_xM\to T_xM^*$ is the natural isometry that maps each vector $v$ to the functional $g_\mathrm R(x)(v,\cdot\,)$. Notice that using the identifications \eqref{ident:bilin}, $i_x=g_\mathrm R(x)$. Now we use Lemma~\ref{le:upt} with $n=r+s$, setting both $V_i$'s and $W_i$'s equal to $T_xM$ $r$ times and to $T_xM^*$ $s$ times. More precisely, consider the map $$T:\underbrace{\prod_{i=1}^s T_xM^*\times\prod_{j=1}^r T_xM}_{V_1\times\dots\times V_n}\times \underbrace{\prod_{i=1}^s T_xM^*\times\prod_{j=1}^r T_xM}_{W_1\times\dots\times W_n}\la \R$$ \begin{multline*}
(v_1^*,\dots,v_s^*,v_1,\dots,v_r;w_1^*,\dots,w_s^*,w_1,\dots,w_r) \longmapsto \\
\longmapsto \prod_{j=1}^r g_\mathrm R(x)(v_j,w_j)\prod_{i=1}^s g_\mathrm R(x)^{-1}(v_i^*,w_i^*).
\end{multline*}
This map $T$ is clearly multilinear. From Lemma~\ref{le:upt}, there exists a bilinear map $B_x:E_x\times E_x\to\R$ such that the diagram \eqref{eq:diagtb} commutes. Taking basis of the considered spaces (see Remark~\ref{re:tbasis}) it is an easy verification $B_x$ is an inner product on $E_x$, i.e. symmetric and positive--definite, as claimed above. Such inner product is usually called a {\em Hilbert--Schmidt inner product}\index{Hilbert--Schmidt inner product} of multilinear forms on $T_xM$, that is induced by the fixed Riemannian metric $g_\mathrm R$.

\begin{remark}\label{re:tbasis}
Suppose $\{e_i(x)\}_{i=1}^m$ is a $g_\mathrm R(x)$--orthonormal basis of $T_xM$ and $\{e^*_i(x)\}_{i=1}^m$ the dual basis on $T_xM^*$. Elementary tensor calculus shows that $$e_{j_1}^*(x)\otimes\ldots\otimes e_{j_s}^*(x)\otimes e_{i_1}(x)\otimes\ldots\otimes e_{i_r}(x)$$ with $i_a,j_b\in\{1,\ldots,m\}$ for $1\leq a\leq r$ and $1\leq b\leq s$ form a $B_x$--orthonormal basis of $E_x$.
\end{remark}

\begin{remark}
The Hilbert--Schmidt inner product $B_x$ may be also described as follows. From Remark~\ref{re:tbasis}, a choice of a $g_\mathrm R(x)$--orthonormal basis of $T_xM$ induces a choice of $B_x$--orthonormal basis in each tensor power $E_x$ of $T_xM$. Thus it is possible to calculate traces of operators in such spaces. Elementary linear algebra calculations show that for each $X,Y\in E_x$,
\begin{equation}\label{eq:btr}
B_x(X,Y)=\tr(X^*Y).
\end{equation}
\end{remark}

\begin{definition}\label{def:fiberproduct}
Consider the Hilbert--Schmidt inner product $$B_x:E_x\times E_x\la\R$$ described above. The vector space norm on $E_x$ induced by $B_x$ will be denoted $\|\cdot\|_\mathrm R$, without reference to the \emph{base point} $x$.
\end{definition}

\begin{remark}
The subindex $_\mathrm R$ stress the dependence on the fixed Riemannian metric $g_\mathrm R$ of $M$, see Remarks~\ref{re:dependgr} and~\ref{re:dependgr2}.
\end{remark}

Notice that using the above results it is possible to analyze the growth control of derivatives of tensors in $\sect^k(E)$. Consider the Levi--Civita connection $\nabla^\mathrm R$ of $g_\mathrm R$. From Theorem~\ref{thm:tensorconnection}, $\nabla^\mathrm R$ induces a connection on $E$, denoted by the same symbol. Furthermore, from Corollary~\ref{cor:jcovder} it makes sense to compute the $j^{\mbox{\tiny th}}$ covariant derivative $(\nabla^\mathrm R)^j K$ of any $K\in\sect^k(E)$, provided that $j\leq k$. Finally, Definition~\ref{def:fiberproduct} allows to compute the norm of such covariant derivatives.

We are now ready to define a norm on (a subspace of) $\sect^k(E)$. Essentially, this is a natural generalization of the $C^k$--norm of maps between Euclidean spaces \eqref{eq:normck}, replacing maps with tensors and standard derivatives with covariant derivatives.

\begin{definition}\label{def:sectionnorm}
Denote by $\sect^k_b(E)$ the vector subspace consisting of sections $K\in\sect^k(E)$ such that\index{Section!$C^k$ norm}
\begin{equation}\label{eq:sectionnorm}
\|K\|_k=\max_{0\leq j\leq k}\Big\{\sup_{x\in M} \Big\|(\nabla^\mathrm{R})^j K(x)\Big\|_\mathrm{R}\Big\}
\end{equation}
is bounded.\footnote{If $M$ is compact, clearly $\sect_b^k(E)=\sect^k(E)$.} Then \eqref{eq:sectionnorm} defines a norm on $\sect_b^k(E)$, and this vector space will be implicitly assumed to be endowed with the norm \eqref{eq:sectionnorm}.
\end{definition}

It is a standard verification that \eqref{eq:sectionnorm} is a well--defined norm on $\sect_b^k(E)$, similar to the case of $C^k$--norms of maps between Euclidean spaces, described in Definition~\ref{def:ck}, see for instance Rudin \cite{rudin}. Before proving completeness of this normed vector space of tensors, we recall the following well--known result, that is a natural generalization of Lemma~\ref{le:convderivadas}.

\begin{lemma}\label{le:climage}
Let $\{s_n\}_{n\in\N}$ be a sequence of sections in $\sect_b^k(E)$ that converges locally uniformly to a limit section $s_\infty\in\sect_b^0(E)$, such that also the covariant derivatives $\{(\nabla^\mathrm R)^j s_n\}_{n\in\N}$ converge locally uniformly to sections $s^j_\infty\in\sect_b^0(E)$, for $1\leq j\leq k$. Then $s_\infty\in\sect_b^k(E)$ and $(\nabla^\mathrm R)^j s_\infty=s^j_\infty$.
\end{lemma}

\begin{proof}
We will only give a brief sketch of this proof. Since the matter is local\footnote{i.e., it suffices to prove that for each $x\in M$, $s_\infty$ is $C^k$ in a neighborhood of $x$ and $(\nabla^\mathrm R)^j s_\infty(x)=s^j_\infty(x)$ for $1\leq j\leq k$.}, by taking a chart around each point, without loss of generality we can assume that we are in an open subset of $\R^m$, contained in a trivialization of $E$ (see Remark~\ref{re:containedtrivial}). From an elementary result of analysis in Euclidean spaces, given a locally uniformly convergent sequence of $C^1$ maps $f_n:U\subset\R^m\to\R^p$, provided that also the derivatives $f'_n$ converge locally uniformly, the limit $f_\infty$ is $C^1$ and $f'_\infty=\lim_{n\in\N} f'_n$ (see Lemma~\ref{le:convderivadas}). By induction, this is true replacing the class $C^1$ with $C^k$ and the standard derivative $f'$ with higher order derivatives $D^\alpha f$, with multi--indexes $|\alpha|\leq k$.

With our identifications, the connection $\nabla^\mathrm R$ gives a connection in the trivial bundle $U\times\R^p$. For any connection $\nabla$ in $U\times\R^p$, there exists a $(1,2)$--tensor $\chr$ such that for every $n\in\N$, $$\nabla f_n=\dd f_n+\chr(\,\cdot,f_n),$$ see \eqref{eq:chrnablap} in Definition~\ref{def:christoffeltens}. Analogously, $\nabla^j f_n$ are related to $D^\alpha f_n$, $|\alpha|\leq j$ in terms of the same tensor $\chr$. Therefore, applying the result stated above for the sequence $\{s_n\}_{n\in\N}$ and using the relations between the standard and covariant derivatives of $s_n$, it follows that the limit section is also $C^k$ and $(\nabla^\mathrm R)^j s_\infty=\lim_{n\in\N} (\nabla^\mathrm R)^j s_n$, for $1\leq j\leq k$.
\end{proof}

\begin{corollary}\label{cor:splitimmersion}
Consider the map
\begin{eqnarray*}
\mathfrak{i}_b:\sect_b^k(E) &\longhookrightarrow& \sect^0_b\left(\bigoplus_{j=0}^k \left({TM^*}^{(j)}\otimes E\right) \right) = \bigoplus_{j=0}^k \sect^0_b\left({TM^*}^{(j)}\otimes E\right) \\
s &\longmapsto & \left(s,(\nabla^\mathrm R)s,\dots,(\nabla^\mathrm R)^{k-1} s,(\nabla^\mathrm R)^{k}s\right),
\end{eqnarray*}
where ${TM^*}^{(j)}$ denotes, as usual, the tensor bundle given by the $j^{\mbox{\tiny th}}$ tensor power of $TM^*$. Endowing the counter domain with the norm given by the maximum of the norms of each component, the above map is a linear isometric immersion with closed image.
\end{corollary}

The verification that $\mathfrak{i}_b$ is a linear isometric immersion is immediate from the norms considered in each space. Furthermore, Lemma~\ref{le:climage} guarantees that any (globally) uniformly convergent sequence of elements in $\im\mathfrak i_b$ has its limit also in $\im\mathfrak i_b$, which is hence closed.

\begin{proposition}\label{prop:banachspaceofsections}
The vector space $\sect_b^k(E)$ is a Banach space.
\end{proposition}

\begin{proof}
We first reduce the problem to the case $k=0$. From Corollary~\ref{cor:splitimmersion}, $\im\mathfrak i_b\subset\bigoplus_{j=0}^k \sect^0_b({TM^*}^{(j)}\otimes E)$ is closed, hence it suffices to prove that $\sect^0_b(E')$ is complete for any tensor bundle $E'$. The closed subspace $\im\mathfrak i_b$ will then be a complete space, which is isometric to $\sect_b^k(E)$, concluding the proof.

At this point, the proof is a simple generalization of elementary completeness results of continuous function spaces between Euclidean spaces with the uniform convergence norm. Consider a Cauchy sequence $\{s_n\}_{n\in\N}$ in $\sect^0_b(E')$, with respect to the uniform convergence norm. Then for each $x\in M$ and $n,m\in\N$, $$\|s_n(x)-s_m(x)\|_\mathrm R\leq\sup_{x\in M}\|s_n(x)-s_m(x)\|_\mathrm R=\|s_n-s_m\|_0,$$ hence $\{s_n(x)\}_{n\in\N}$ is Cauchy in $E'_x$, for all $x\in M$. From completeness of the finite--dimensional vector space $E'_x$, there exists $$s_\infty(x)=\lim_{n\in\N} s_n(x).$$ Define $s_\infty\in\sect^0(E')$ by the expression above, for all $x\in M$. It only remains to prove that $\{s_n\}_{n\in\N}$ converges uniformly to $s_\infty$ and that $s_\infty\in\sect_b^0(E')$. For each $\varepsilon>0$ there exists $N\in\N$, such that if $n,m\geq N$, 
\begin{equation*}
\|s_n-s_m\|_0<\tfrac{\varepsilon}{2}.
\end{equation*} Furthermore, for each $x\in M$, there exists $m=m(x)\geq N$, such that $$\|s_m(x)-s_\infty(x)\|_\mathrm R<\tfrac{\varepsilon}{2}.$$ Thus, if $n\geq N$, for all $x\in M$, 
\begin{eqnarray}
\|s_n(x)-s_\infty(x)\|_\mathrm R &\leq& \|s_n(x)-s_{m(x)}(x)\|_\mathrm R+\|s_{m(x)}(x)-s_\infty(x)\|_\mathrm R \nonumber \\
&<& \tfrac{\varepsilon}{2}+\tfrac{\varepsilon}{2}=\varepsilon.\label{eq:banachspaceofsectionsconvergence}
\end{eqnarray}
Hence $\|s_\infty\|_0<+\infty$, since for all $x\in M$, from \eqref{eq:banachspaceofsectionsconvergence}, $$\|s_N(x)-s_\infty(x)\|_\mathrm R<\varepsilon,$$ and $\|s_N(x)\|_0<+\infty$. Therefore $s_\infty\in\sect_b^0(E')$. Finally, from \eqref{eq:banachspaceofsectionsconvergence}, the Cauchy sequence $\{s_n\}_{n\in\N}$ converges uniformly to $s_\infty$, concluding the proof that $\sect_b^0(E')$ is complete.
\end{proof}

\begin{remark}\label{re:sectinftyfrechet}
Consider the countable intersection $$\sect^\infty_b(E)=\bigcap_{k\in\N}\sect^k_b(E).$$ Sections in this subspace are smooth, or of class $C^\infty$. Analogously to the case of $C^\infty([a,b],\R^m)$ described in Remark~\ref{re:cinftyfrechet}, this is not a Banach space, as $\sect^k_b(E)$. In fact, it is a Fr\'echet space, see Definition~\ref{def:frechet}. The sequence of norms $\{\|\cdot\|_{k}\}_{k\in\N}$, given by \eqref{eq:sectionnorm}, gives a countable family of semi--norms that induce the topology of $\sect_b^\infty(E)$, see Lemma~\ref{le:tvshell}.

Although it would be desirable to deal with smooth sections instead of $C^k$ sections, most of our tools apply to Banach spaces {\em only}. Thus, we will analyze genericity of some properties of sections, particularly metric tensors, in the $C^k$--topology, rather than the $C^\infty$--topology. Nevertheless, as mentioned in the beginning of this chapter, it will be later possible to establish the same results in the $C^\infty$--topology, using standard intersection arguments described in Sections~\ref{sec:smooth1} and~\ref{sec:smooth2}.

In addition, density results such as the Stone--Weierstrass Theorem~\ref{thm:stoneweierstrass}, Proposition~\ref{prop:cinftylp} and Corollaries~\ref{cor:cinftywkp} and~\ref{cor:cinftyhk} automatically extend to this context of sections of vector bundles, see the Stone--Weierstrass Theorem~\ref{thm:stw2}, Proposition~\ref{prop:cinftylp2} and Corollary~\ref{cor:cinftyhk2}.
\end{remark}

\begin{remark}\label{re:separability}
The Banach space $\sect_b^k(E)$ is non separable. A proof of this result is elementary and very similar to the proof that the Banach space $\ell_\infty$ of bounded sequences of real numbers with the $\sup$ norm is non separable. Both use the same classic technique for proving that a Banach space is non separable.

Suppose it is possible to construct an uncountable subset $S$ of a Banach space $X$ such that the distance between any two points of $S$ is a strictly positive number. Considering sufficiently small open balls of $X$ centered in each point of $S$, one concludes any dense subset $D$ of $X$ is automatically uncountable, since each open ball of $X$ contains at least one element of $D$. Hence, under this hypothesis, $X$ is not separable.
\end{remark}

\begin{remark}\label{re:dependgr}
Although it is not natural to consider spaces of sections endowed with a structure that depends\footnote{Notice that \eqref{eq:sectionnorm} involves the norms $\|\cdot\|_\mathrm R$, which are induced by $g_\mathrm R$ in the fibers of tensor bundles over $M$. {\em A priori}, different choices of $g_\mathrm R$ would give rise to different Banach space structures on $\sect^k_b(E)$.} on the choice of a Riemannian metric $g_\mathrm R$ on $M$, this is the best possible setting for the desired applications.

In case $M$ is compact, it is easy to prove that norms of the form \eqref{eq:sectionnorm} on $\sect_b^k(E)$ are always equivalent\footnote{Hence $\sect^k_b(E)$ is naturally a Banachable space, provided $M$ is compact. In this text, we are interested mostly with {\em non compact} manifolds, hence all the present discussion is necessary.}, for different choices of $g_\mathrm R$. Namely, if $g'_\mathrm R$ is another Riemannian metric, from continuity of both metrics and compactness of $M$, there exist positive constants $c_1,c_2\in\R$ such that $$c_1g_\mathrm R(v,v)\leq g'_\mathrm R(v,v)\leq c_2g_\mathrm R(v,v),$$ for all $v\in T_xM$.\footnote{Clearly, this is always locally true, since for each $x\in M$, both $(g_\mathrm R)(x)$ and $(g'_\mathrm R)(x)$ are inner products in the Euclidean space $T_xM$, hence equivalent, i.e., there exist positive constants $c_1(x)$ and $c_2(x)$ such that $c_1(x)(g_\mathrm R)(x)(v,v)\leq (g'_\mathrm R)(x)(v,v)\leq c_2(x)(g_\mathrm R)(x)(v,v)$, for all $v\in T_xM$. From continuity of the metrics, there exist positive continuous functions $c_1,c_2:M\to\R$ such that the above inequality holds for every $x\in M$. Then define $c_1=\min_{x\in M} c_1(x)$ and $c_2=\max_{x\in M} c_2(x)$.} Analogously, one has a similar comparison between the Levi--Civita connections of $g_\mathrm R$ and $g'_\mathrm R$, hence between the $j^{\mbox{\tiny th}}$ covariant derivatives of sections with respect to such connections. Using these inequalities it follows that the norms given by the expression \eqref{eq:sectionnorm} using either $g_\mathrm R$ or $g'_\mathrm R$ are equivalent.
\end{remark}

\begin{remark}\label{re:dependgr2}
Another \emph{would--be approach} to deal with the dependence of the structure of $\sect_b^k(E)$ on the fixed Riemannian metric $g_\mathrm R$ is to endow $\sect_b^k(E)$ with the compact--open topology, relinquishing the Banach space structure discussed above. Nevertheless, this would give rise to a Fr\'echet structure on $\sect_b^k(E)$, and our applications use in a nontrivial way several hypothesis that are no longer valid passing from Banach spaces to Fr\'echet spaces, for instance those necessary to apply the Sard--Smale Theorem.
\end{remark}

The Banach space structure described in Proposition~\ref{prop:banachspaceofsections} is a particular case of the following concept considered by Biliotti, Javaloyes and Piccione \cite{biljavapic}.

\begin{definition}\label{def:ckwhitbanachspace}
A vector subspace $\mathds E$ of $\sect^k(E)$ is called a {\it $C^k$ Whitney type Banach space of sections of $E$}\index{$C^k$ Whitney type Banach space of sections}\index{Section!$C^k$ Whitney type Banach space}\index{Banach space!$C^k$ Whitney type} if
\begin{itemize}
\item[(i)] there exists a bounded linear inclusion $i:\mathds E\hookrightarrow\sect^k_b(E)$;
\item[(ii)] $\mathds E$ contains all sections in $\sect^k(E)$ having compact support;
\item[(iii)] $\mathds E$ is endowed with a Banach space norm $\|\cdot\|_{\mathds E}$ with the property that $\|\cdot\|_{\mathds E}$--convergence of a sequence implies convergence in the weak Whitney $C^k$--topology.\footnote{Recall that $\sect^k(E)$ endowed with the weak Whitney $C^k$--topology (i.e., the topology of uniform convergence of the first $k$ derivatives in compact subsets) is locally homeomorphic to a Fr\'echet space, hence first countable. Therefore this topology can be characterized by convergence of sequences. However, $\sect^k(E)$ may be a non separable (or equivalently, non second countable) space. Examples of both separable and non separable $C^k$ Whitney type Banach spaces of sections are given in the sequel.}
\end{itemize}
The third condition means that given any sequence $\{s_n\}_{n\in\N}$ in $\mathds E$ and an element $s_\infty\in\mathds{E}$ such that $\lim_{n\in\N} \|s_n-s_\infty\|_\mathds{E}=0$, then for each compact set $K\subset M$, the sequence of restrictions $\{s_n|_K\}_{n\in\N}$ tends uniformly to $s_\infty|_K$ in the $C^k$--topology as $n$ tends to $\infty$.
\end{definition}

\begin{remark}\label{re:iiiiii}
These conditions are sufficient to endow the intersection $\mathds E\cap\met_\nu^k(M)$ with the adequate topological structure. More precisely, (i) guarantees continuity of left composition with bundle morphisms with non compact base, see Piccione and Tausk \cite{pictau}. This will be used, for instance, to prove that the generalized energy functional considered in Section~\ref{sec:genenfunc} is sufficiently differentiable. In addition, (ii) will be used for technical constructions regarding compact support sections to be used for a local metric perturbation argument. Finally, (iii) endows $\mathds E$ with a sufficiently fine topology that allows the sequences we are interested in to converge at the same time it {\em detects small perturbations}.
\end{remark}

\begin{example}\label{ex:sectb}
The Banach space structure on $\sect_b^k(E)$ given in Proposition~\ref{prop:banachspaceofsections} clearly satisfies (i) and (ii). Furthermore, $\|\cdot\|_k$--convergence clearly implies uniform convergence of the first $k$ derivatives on compact subsets, verifying (iii). Hence $\sect_b^k(E)$ is an example of $C^k$ Whitney type Banach space of sections.
\end{example}

We now approach a delicate matter concerning \emph{separability} of $C^k$ Whitney type Banach spaces of sections, which will be a necessary hypothesis in our applications. It is not difficult to prove that $\sect_b^k(E)$ is \emph{not} separable, see Remark~\ref{re:separability}. However, the next result gives an example of a separable subspace of $\sect_b^k(E)$.

\begin{proposition}\label{prop:tensorszeroseparable}
Let $\sect_0^k(E)$ be the subspace of $\sect_b^k(E)$ consisting of tensors all of whose covariant derivatives tend to zero at infinity\footnote{i.e., for any $\varepsilon>0$, there exists a compact subset $K\subset M$ such that $\|(\nabla^\mathrm R)^j s(x)\|_\mathrm R<\varepsilon$ for all $x\in M\setminus K$ and $0\leq j\leq k$.} (see Definition~\ref{def:tendstozero}). Then $\sect_0^k(E)$ is a {\em separable} $C^k$ Whitney type Banach space of sections.
\end{proposition}

\begin{proof}
The verification that $\sect_0^k(E)$ endowed with the norm \eqref{eq:sectionnorm} satisfies the conditions of Definition~\ref{def:ckwhitbanachspace} is simple and similar to the case of $\sect_b^k(E)$ discussed in Example~\ref{ex:sectb}. We will only prove its separability.

First, we observe that once more it is possible to reduce the problem to the case $k=0$. Consider the restriction of the linear map $\mathfrak i_b$ defined in Corollary~\ref{cor:splitimmersion} to the subspace $\sect_0^k(E)$,
\begin{eqnarray*}
\sect_0^k(E) &\longhookrightarrow& \sect^0_0\left(\bigoplus_{j=0}^k \left({TM^*}^{(j)}\otimes E\right) \right) = \bigoplus_{j=0}^k \sect^0_0\left({TM^*}^{(j)}\otimes E\right) \\
s &\longmapsto & \left(s,(\nabla^\mathrm R)s,\dots,(\nabla^\mathrm R)^{k-1} s,(\nabla^\mathrm R)^{k}s\right).
\end{eqnarray*}
For the same reasons in Corollary~\ref{cor:splitimmersion}, this is an isometric immersion when the domain is endowed with the norm \eqref{eq:sectionnorm} and the counter domain with the norm given by the maximum of the uniform convergence norms of each component. Thus, proving that $\bigoplus_{j=0}^k \sect^0_0({TM^*}^{(j)}\otimes E)$ is separable automatically implies that $\sect_0^k(E)$ is separable.\footnote{In fact, this follows from the simple observation below, which is clearly true in the particular cases of normed vector spaces. \begin{remark} Let $i:A\hookrightarrow B$ be an isometric immersion of metric spaces. If $B$ is separable, then $A$ is also separable. This can be easily proved from the fact that all subsets of a separable metric space are also separable, see for instance Kaplansky \cite{kaplansky}.\end{remark}} Therefore it suffices to prove separability of $\sect_0^0(E')$ endowed with the uniform convergence norm, where $E'$ is some tensor bundle over $M$.

Second, denote $\sect^0_c(E')\subset\sect^0_b(E')$ the subspace of sections with compact support. It is easy to see that $\sect^0_c(E')$ is dense in $\sect^0_0(E')$, analogously to elementary results of analysis in Euclidean spaces concerning maps with compact support and maps that tend to zero at infinity. Hence, we reduced the problem to proving separability of $\sect^0_c(E')$.

Third, for each compact set $K\subset M$, let
\begin{equation}
\sect^0_K(E')=\left\{s\in\sect^0(E'):\supp s\subset K\right\},
\end{equation}
and consider an exhaustion of $M$ by compact subsets, i.e., a sequence of compact subsets $\{K_n\}_{n\in\N}$ of $M$, with $K_n$ contained in the interior of $K_{n+1}$ for all $n\in\N$, and $M=\bigcup_{n\in\N} K_n$. It is clear that $\sect^0_c(E')=\bigcup_{n\in\N}\sect^0_{K_n}(E')$. Since $\sect^0_c(E')$ is given by the {\em countable} union of $\sect^0_{K_n}(E')$'s, once more we have reduced the problem, to prove that each $\sect^0_{K_n}(E')$ is separable.

Since each $K_n$ is compact, it is possible to split it in the disjoint union $K_n=\bigsqcup_{i=1}^{r_n} K_n^i$, with $K_n^i\subset K_n$ compact and contained in a trivialization $\alpha_i$ of $E'$ (see Definition~\ref{def:vectorbundle} and Remark~\ref{re:containedtrivial}). Thus one has the following sequence of isometric immersions given by restriction maps $$\sect^0_{K_n}(E')\longhookrightarrow C^0(K_n,E')\longhookrightarrow\bigoplus_{i=1}^{r_n} C^0(K_n^i,E'),$$ where the last term is endowed with the norm given by the maximum of the uniform convergence norms of each component. Again, with such isometric immersions, we reduced the problem to proving that each $C^0(K_n^i,E')$ with the uniform convergence norm is separable.

For $K_n^i$ is contained in the domain of a trivialization $\alpha_i$ of $E'$, there is a natural identification $$C^0(K_n^i,E')\cong C^0(K_n^i,\R^m)\cong\bigoplus_{l=1}^m C^0(K_n^i,\R),$$ where $m$ is the dimension of the fibers of $E'$. Since $K_n^i$ are compact and metrizable, it follows from Proposition~\ref{prop:cksep} that $C^0(K_n^i,\R)$ with the uniform convergence norm are separable, concluding the proof.
\end{proof}

\begin{remark}\label{re:emptyinterior}
We now obtain a candidate to {\em typical} $C^k$ Whitney type Banach space of sections that endows $\met_\nu^k(M)$ with the adequate structures. Namely, consider $\sect^k_0(TM^*\vee TM^*)$, see Definitions~\ref{def:tendstozero} and~\ref{def:symskewsym}. From Proposition~\ref{prop:tensorszeroseparable}, this is a separable $C^k$ Whitney type Banach space.

Nevertheless, if $M$ is non compact, the intersection $$\mathcal{F}=\sect^k_0(TM^*\vee TM^*)\cap\met_\nu^k(M)$$ has empty interior in the topology of $\sect^k_0(TM^*\vee TM^*)$, hence all genericity statements concerning open subsets of $\mathcal{F}$ are automatically empty. This can be easily proved as follows. Let $g\in\mathcal F$. As mentioned in the proof of Proposition~\ref{prop:tensorszeroseparable}, the subset $\sect^k_c(TM^*\vee TM^*)$ of sections in $\sect^k_0(TM^*\vee TM^*)$ with compact support is dense. Therefore there exists a sequence $\{s_n\}_{n\in\N}$ in $\sect^k_c(TM^*\vee TM^*)$ that tends to $g$. Since each element $s_n$ has compact support, it is a {\em degenerate} symmetric tensor outside its support. Hence $s_n$ cannot be a metric. Therefore, $g$ is not an interior point and $\mathcal F$ fails to have nonempty interior.
\end{remark}

The easiest solution for this problem is considering an {\em affine subspace} of $\sect^k_b(TM^*\vee TM^*)$ isomorphic to a separable $C^k$ Whitney type Banach space of sections $\mathds E$, for instance a suitable displacement of the subspace $\sect^k_0(TM^*\vee TM^*)$ by a metric $g_\mathrm A$, that satisfies a condition similar to \eqref{eq:boundedfromzero}. In order to prove that this indeed solves the problem, we first need some auxiliary results.

\begin{lemma}\label{le:symhs}
Let $s\in\sect^k(TM^*\vee TM^*)$ and fix $x\in M$. Denote by $\lambda_j$, $j=1,\ldots,m$ the eigenvalues of the $g_\mathrm R$--symmetric operator $s(x)$ of $T_xM$. Then
\begin{equation}\label{eq:symhs}
\|s(x)\|_\mathrm R=\sqrt{\sum_{j=0}^m \lambda_j^2}.
\end{equation}
\end{lemma}

\begin{proof}
Choose a $g_\mathrm R$--orthonormal basis of $T_xM$. Such choice gives an isomorphism between $T_xM$ and $\R^m$. From Remark~\ref{re:tbasis}, this choice also induces a choice of orthonormal basis in all tensor powers of $T_xM$, of the form \eqref{eq:tensorbundlefiber}, with respect to the Hilbert--Schmidt inner product induced by $g_\mathrm R$. Thus not only $T_xM$, but also its tensor powers can be now identified with tensor powers of $\R^m$. In the rest of the proof we shall use such isomorphisms as identifications.\footnote{These identifications obviously depend on the choice of the $g_\mathrm R$--orthonormal basis of $T_xM$, that determines an isomorphism of $T_xM$ with the Euclidean space. Notice however that the proof is a linear algebra fact, and could be done abstractly for finite--dimensional real vector spaces.}

Denote by $A$ the $m\times m$ real matrix that represents the symmetric operator $s(x)$ in the orthonormal basis above described. Since $A$ is symmetric, there exists an orthogonal matrix $P$ such that the conjugation $D=P^*AP$ is a diagonal matrix. Clearly, this diagonal matrix is formed by eigenvalues of the $g_\mathrm R$--symmetric operator $s(x)$, i.e.,
$$D=\left[ \begin{array}{l l l}
\lambda_1 &  & \\
& \ddots & \\
& & \lambda_m\end{array}
\right]$$
From \eqref{eq:btr},
\begin{eqnarray*}
\|s(x)\|^2_\mathrm R &=& \tr(A^*A) \\
&=& \tr(P^*A^*AP) \\
&=& \tr(P^*A^*PP^*AP) \\
&=& \tr(D^*D) \\
&=& \sum_{j=0}^m \lambda_j^2.\qedhere
\end{eqnarray*}
\end{proof}

Since each $s\in\sect^k(TM^*\vee TM^*)$ at $x\in M$ is a symmetric bilinear form $s(x):T_xM\times T_xM\to\R$, one may also compute the usual operator norm of $s(x)$, given by \eqref{eq:normt} and simply denoted $\|\cdot\|$. Since $T_xM$ is a finite--dimensional real vector space, this norm is given by
\begin{equation}
\|s(x)\|=\max_{\substack{\|v_i\|_\mathrm R=1 \\ i=1,2}} s(x)(v_1,v_2).
\end{equation}

The Hilbert--Schmidt norm and the usual operator norm are clearly equivalent, since we are in finite--dimensional vector spaces. More precisely, the following result gives the constants of such equivalence.

\begin{corollary}
For each $s\in\sect^k(TM^*\vee TM^*)$ and $x\in M$,
\begin{equation}\label{eq:snormt}
\|s(x)\|\le\|s(x)\|_\mathrm R\le\sqrt{m}\|s(x)\|.
\end{equation}
\end{corollary}

\begin{proof}
Using elementary linear algebra, \eqref{eq:snormt} can be also given by
\begin{equation*}
\|s(x)\|=\max_{j} |\lambda_j|,
\end{equation*}
where $\lambda_j$ are the eigenvalues of the $g_\mathrm R$--symmetric operator $s(x)$ of $T_xM$. The inequalities \eqref{eq:snormt} follow from direct comparison of \eqref{eq:symhs} and the above equation.
\end{proof}

We now give a simple condition for nondegeneracy of $C^k$ symmetric $(0,2)$--tensors on $M$, given it has sufficiently small distance to a {\em suitably nondegenerate} tensor uniformly on $M$. By {\em suitably nondegenerate} tensor we mean a tensor $s\in\sect^k_b(TM^*\vee TM^*)$ all of whose eigenvalues $\lambda_j$ at any $x\in M$ have absolute value bounded\footnote{Notice that such limitation on eigenvalues $\lambda_j$ of symmetric tensors $s$ can be expressed in terms of $\|s(x)\|_\mathrm R$ or $\|s(x)\|$. Moreover, from \eqref{eq:snormt}, both $\|\cdot\|_\mathrm R$ and $\|\cdot\|$ give {\em essentially} the same conditions.} away from zero, uniformly on $x$. More precisely, using identifications \eqref{ident:bilin}, the above condition is equivalent to $s(x):T_xM\to T_xM^*$ satisfying $$0<\tfrac{1}{c}\leq\min_{0\leq j\leq m} |\lambda_j|,$$ for some $c>0$, uniformly on $M$.

A fancy way of expressing this condition on $s$ is requiring that
\begin{equation}\label{eq:boundedfromzero}
\sup_{x\in M}\|s(x)^{-1}\|_\mathrm R\leq c<+\infty.
\end{equation}
Notice that $s(x)^{-1}:T_xM^*\to T_xM$ lies in the same of $s(x)$, using identifications \eqref{ident:dual} on the finite--dimensional vector space $T_xM$.

\begin{lemma}\label{le:salvador}
Fix $x\in M$. For all constants $c'>c>0$, there exists $\varepsilon>0$ such that if $s\in\sect^k_b(TM^*\vee TM^*)$ satisfies $\|s(x)^{-1}\|_\mathrm R\leq c$, then for any $s'\in\sect^k_b(TM^*\vee TM^*)$,
\begin{equation}\label{eq:nondegeq1}
\|s'(x)-s(x)\|_\mathrm R<\varepsilon
\end{equation}
implies that $s'(x)$ is nondegenerate and $\|s'(x)^{-1}\|_\mathrm R\leq c'$.
\end{lemma}

\begin{proof}
Let $c'>c>0$ be given and consider $s\in\sect^k_b(TM^*\vee TM^*)$ such that $\|s(x)^{-1}\|_\mathrm R\leq c$. Recall that from identifications \eqref{ident:bilin} and \eqref{ident:dual}, by $s(x)^{-1}$ we mean the inverse of the operator $$s(x):T_xM\la T_xM^*\cong T_xM.$$ Since $\GL(T_xM)$ is a Lie group, the inversion map $$\iota:\GL(T_xM)\ni g\longmapsto g^{-1}\in\GL(T_xM)$$ is continuous. Hence, there exists $\delta>0$ such that if $s'\in\sect^k_b(TM^*\vee TM^*)$ satisfies $\|s'(x)-s(x)\|_\mathrm R<\delta$ then 
\begin{equation}\label{eq:salvadoreq1}
\|\iota(s'(x))-\iota(s(x))\|_\mathrm R=\|s'(x)^{-1}-s(x)^{-1}\|_\mathrm R<c'-c.
\end{equation}

Set $\varepsilon=\min\left\{\delta,\frac{1}{c}\right\}$ and consider $s'\in\sect^k_b(TM^*\vee TM^*)$ be such that \eqref{eq:nondegeq1} holds. Elementary computations give $$s'(x)=s(x)(\id+s(x)^{-1}(s'(x)-s(x))).$$ From Lemma~\ref{le:invertibility}, setting $T=-s(x)^{-1}(s'(x)-s(x))$, in order to prove nondegeneracy of $s'(x)$, it suffices to prove that $\|T\|<1$. Notice that this is not the {\em Hilbert--Schmidt norm}, but the {\em operator norm}. Using Cauchy--Schwartz inequality \eqref{eq:cauchyschwartz} and inequalities \eqref{eq:snormt},
\begin{eqnarray*}
\|T\|&=& \|s(x)^{-1}(s'(x)-s(x))\| \\
&\leq & \|s(x)^{-1}\|\|s'(x)-s(x)\| \\
&\leq & \|s(x)^{-1}\|_\mathrm R\|s'(x)-s(x)\|_\mathrm R \\
&<& c\,\frac{1}{c}=1.
\end{eqnarray*}
Thus $s'(x):T_xM\to T_xM^*\cong T_xM$ is invertible, hence nondegenerate as a symmetric bilinear form of $T_xM$ (see Definition~\ref{def:nondegenerate}).

In addition, from \eqref{eq:salvadoreq1},
\begin{eqnarray*}
\|s'(x)^{-1}\|_\mathrm R-\|s(x)^{-1}\|_\mathrm R &\leq & \Big|\|s'(x)^{-1}\|_\mathrm R-\|s(x)^{-1}\|_\mathrm R\Big| \\
&\leq & \|s'(x)^{-1}-s(x)^{-1}\|_\mathrm R \\
&<& c'-c.
\end{eqnarray*}
Hence $\|s'(x)^{-1}\|_\mathrm R\leq c'$, concluding the proof.
\end{proof}

We are now ready to give a solution for the problem presented in Remark~\ref{re:emptyinterior} regarding emptiness of the interior of $$\mathcal F=\sect_0^k(TM^*\vee TM^*)\cap\met_\nu^k(M).$$ This will be done replacing $\sect_0^k(TM^*\vee TM^*)$ with a suitable displacement of this subspace, originating an {\em affine subspace} of $\sect_b^k(TM^*\vee TM^*)$.

\begin{proposition}\label{prop:affineworks}
Let $g_\mathrm A\in\met_\nu^k(M)$ be an auxiliary metric satisfying \eqref{eq:boundedfromzero}, i.e.
\begin{equation*}
\sup_{x\in M}\|g_\mathrm A^{-1}(x)\|_\mathrm R<+\infty.
\end{equation*}
Let $\mathds E$ be a $C^k$ Whitney type Banach space of sections of $E=TM^*\vee TM^*$ that tend to zero at infinity, see Definitions~\ref{def:ckwhitbanachspace} and~\ref{def:tendstozero}. Consider the affine space $g_\mathrm A+\mathds E$, with the topology induced by the translation of $g_\mathrm A$. Then the following is a (nonempty) open subset of this affine space
\begin{equation}\label{eq:agnu}
\A_{g_\mathrm A,\nu}=(g_\mathrm A+\mathds E)\cap\met_\nu^k(M).
\end{equation}
\end{proposition}

\begin{remark}
The above condition \eqref{eq:boundedfromzero} on $g_\mathrm A$ implies that the eigenvalue with minimum absolute value of the $g_\mathrm R$--symmetric operator $(g_\mathrm A)_x$ stays uniformly away from $0$. Notice that, once more, these considerations depend on the choice of the fixed Riemannian metric $g_\mathrm R$, see Remarks~\ref{re:dependgr} and~\ref{re:dependgr2}.
\end{remark}

\begin{proof}
The subset $\A_{g_\mathrm A,\nu}$ is clearly nonempty since $g_\mathrm A\in\A_{g_\mathrm A,\nu}$. Fix any $g\in\A_{g_\mathrm A,\nu}$. We will prove that $g$ is an interior point of $\A_{g_\mathrm A,\nu}$ in the topology induced by $\mathcal E$, which is hence an open subset of $g_\mathrm A+\mathds E$.

First, we use Lemma~\ref{le:salvador} to prove that also $g$ satisfies \eqref{eq:boundedfromzero}. Set $$c_1=\sup_{x\in M}\|g_\mathrm A(x)^{-1}\|_\mathrm R$$ and $c'_1=c_1+1$. From Lemma~\ref{le:salvador}, there exists $\varepsilon_1>0$ such that if for some $x\in M$, $\|g(x)-g_\mathrm A(x)\|_\mathrm R<\varepsilon_1$, then $g(x)$ is nondegenerate,\footnote{This is clearly redundant, since $g\in\A_{g_\mathrm A,\nu}$ was already taken as a metric.} and $$\|g(x)^{-1}\|_\mathrm R<c'_1=c_1+1<+\infty.$$ Since $g-g_\mathrm A\in\mathds E$ is a tensor that tends to zero at infinity (see Definition~\ref{def:tendstozero}), it follows that there exists a compact subset $K\subset M$ such that for $x\in M\setminus K$, $\|g(x)-g_\mathrm A(x)\|_\mathrm R<\varepsilon_1$, hence $\|g(x)^{-1}\|<c'_1$. In addition, from continuity of $g$, the function $$x\longmapsto \|g(x)^{-1}\|_\mathrm R$$ is clearly continuous hence bounded from above by $L>0$ for $x\in K$. Thus, for all $x\in M$, $\|g(x)^{-1}\|_\mathrm R\leq\max\,\{L,c'_1\}$. Therefore, any $g\in\A_{g_\mathrm A,\nu}$ satisfies
\begin{equation*}
\sup_{x\in M}\|g(x)^{-1}\|_\mathrm R<+\infty.
\end{equation*}

Second, we use Lemma~\ref{le:salvador} again to prove that $g$ is an interior point of $\A_{g_\mathrm A,\nu}$. From the last inequality, there exists a finite constant $c_2>0$, $$c_2=\sup_{x\in M}\|g(x)^{-1}\|_\mathrm R.$$ From Lemma~\ref{le:salvador} with $c_2$ and $c'_2=c_2+1$, there exists $\varepsilon_2>0$ such that if $h\in\sect_b^k(TM^*\vee TM^*)$ satisfies $\|h-g\|_k<\varepsilon_2$, in particular, $\|h(x)-g(x)\|_\mathrm R<\varepsilon_2$ for all $x\in M$, then $h(x)$ is nondegenerate for all $x\in M$. Therefore if $\|h-g\|_k<\varepsilon_2$, then $h$ is a semi--Riemannian metric on $M$.

Third, we prove that the index of $h$ is also $\nu$, hence $h\in\met_\nu^k(M)$ provided that $\|h-g\|_k<\varepsilon_2$. Observe that the same argument used above to prove nondegeneracy of $h$ works for any other $s\in\sect_b^k(TM^*\vee TM^*)$ with $\|s-g\|_k<\varepsilon_2$. Hence any continuous path joining $g$ and $h$ inside the open ball $B_k(g,\varepsilon_2)$ of $\sect_b^k(TM^*\vee TM^*)$ with center $g$ and radius $\varepsilon_2$ is entirely formed by {\em nondegenerate} tensors, i.e. metrics. This implies that $h$ has index $\nu$, since $\met_\nu^k(M)$ is path--connected and if at some point $s_0$ in a continuous path of metrics from $g$ to $h$ inside $B_k(g,\varepsilon_2)$ there was a change in the number of negative eigenvalues, then $s_0$ would obviously be a degenerate tensor.

From continuity of the inclusion $\mathds E\hookrightarrow\sect^k_b(TM^*\vee TM^*)$, there exists an open neighborhood $U$ of $g-g_\mathrm A\in\mathds E$, such that if $h-g_\mathrm A\in U$, then $\|h-g\|_k<\varepsilon_2$. Thus, if $h-g_\mathrm A\in U$, then $h\in\met_\nu^k(M)$. Therefore $g_\mathrm A+V$ is the desired open neighborhood of $g$ in $g_\mathrm A+\mathds E$, with $U\subset\met_\nu^k(M)$. This concludes the proof that $g$ is an interior point and that $\A_{g_\mathrm A,\nu}$ is an open set.
\end{proof}

\begin{figure}[htf]
\begin{center}
\vspace{-0.4cm}
\includegraphics[scale=1]{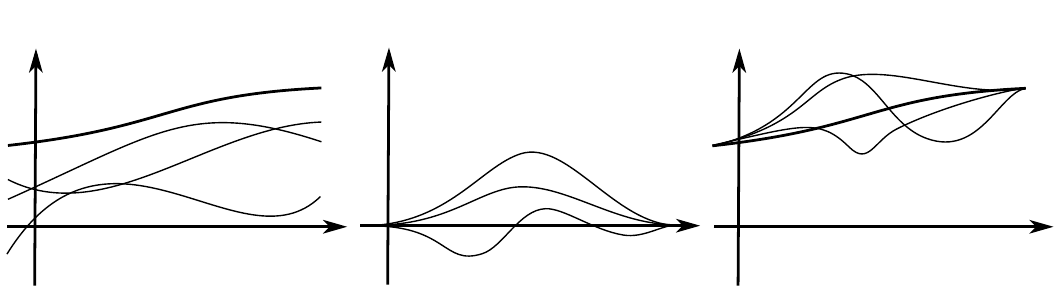}
\begin{pgfpicture}
\pgfputat{\pgfxy(-8.4,2.45)}{\pgfbox[center,center]{$g_\mathrm A$}}
\pgfputat{\pgfxy(-1.3,2.45)}{\pgfbox[center,center]{$g_\mathrm A$}}
\pgfputat{\pgfxy(-9,0.15)}{\pgfbox[center,center]{$\sect^k_b(E)$}}
\pgfputat{\pgfxy(-5.25,0.15)}{\pgfbox[center,center]{$\sect^k_0(E)$}}
\pgfputat{\pgfxy(-1.62,0.15)}{\pgfbox[center,center]{$g_\mathrm A+\sect^k_0(E)$}}
\end{pgfpicture}
\end{center}
\caption{Sections of $E=TM^*\vee TM^*$ that respectively are bounded, tend to zero and are asymptotically equal to $g_\mathrm A$, illustrated as functions.}\label{fig:sectionspaces}
\end{figure}

\begin{remark}\label{re:physics}
The setting $\A_{g_\mathrm A,\nu}$ established above for the domain of semi--Riemannian metrics generalizes a particularly fashionable concept among physicists. In general relativity, it is common to consider asymptotically flat space--times,\footnote{See Definition~\ref{def:asymptflat} and Remark~\ref{re:physics0}.} i.e., Lorentzian manifolds whose curvature vanishes at large distances from some region. This means that at large distances, the geometry becomes essentially the same as that of a Minkowski space--time, see Definition~\ref{def:minkowski}. Observe that the Minkowski metric \eqref{eq:minkowski} satisfies \eqref{eq:boundedfromzero}, hence can be taken as $g_\mathrm A$ in Proposition~\ref{prop:affineworks}. Moreover, notice that since tensors in $\mathds E$ tend to zero at infinity, any metric of the form $g=g_\mathrm A+g'$, with $g'\in\mathds E$, is asymptotically equal to $g_\mathrm A$ at infinity. In particular, the curvature of $g$ is asymptotically equal to the curvature of $g_\mathrm A$ at infinity. Thus, our set of metrics $$\A_{g_\mathrm A,\nu}=(g_\mathrm A+\mathds E)\cap\met_\nu^k(M)$$ established in Proposition~\ref{prop:affineworks} for much more general $g_\mathrm A$'s and $\mathds E$'s, is a natural extension of asymptotically flat space--times, that occur for $\nu=1$ and $m=4$, see Definition~\ref{def:asymptflat}.
\end{remark}

\begin{remark}
The open subset $\A_{g_\mathrm A,\nu}$ studied above will be the domain of metrics in our applications concerning generic properties of the geodesic flow. Namely, genericity of these properties will be proved in this separable Banach manifold, since it has all the necessary structure. In addition, some particular generic properties of metrics that will be studied in Chapter~\ref{chap6} require certain submanifolds to be nondegenerate. In this case, due to possible topological obstructions\footnote{See Section~\ref{sec:topobst}.} to nondegeneracy of certain submanifolds of $M$, the actual domain of metrics used will be an open subset of $\A_{g_\mathrm A,\nu}$. 
\end{remark}

\begin{wrapfigure}{r}{2cm}
\vspace{-0.5cm}
\xymatrix@+15pt{
& TM\ar[d]\\
[a,b]\ar[ur]^v\ar[r]^\gamma & M
}
\vspace{-0.5cm}
\end{wrapfigure}
We end this section with a few remarks on some special Banach spaces of tensors, namely spaces $\sect^k(\gamma^*TM)$ of $C^k$ vector fields\index{Vector field!along a curve} along $C^l$ curves $\gamma:[a,b]\to M$, provided that $k\leq l$, see Example~\ref{ex:vectoralongcurve}. Notice that since $E=TM$ is trivially a tensor bundle over $M$, these are a particular case of the Banach spaces $\sect^k_b(E)$ studied above. More precisely, Proposition~\ref{prop:banachspaceofsections} also implies that $\sect^k(\gamma^*TM)$ is a Banach space, since its elements are bounded sections of the pull--back bundle\footnote{See Remark~\ref{re:restbundle}.} $\gamma^*TM$, for this bundle has compact base $\gamma([a,b])\subset M$ and its continuous sections are hence bounded.

There is a clear similarity between the Banach spaces $\sect^k(\gamma^*TM)$ and $C^k([a,b],\R^m)$, since elements of both are $C^k$ maps having the interval $[a,b]$ as domain and taking values on a vector space, or a family $\{T_{\gamma(t)}M\}_{t\in [a,b]}$ of vector spaces. Indeed, let $\{e_i(t)\}_{i=1}^m$ be a $C^k$ referential along $\gamma$. Then any vector field $v\in\sect^k(\gamma^*TM)$ can be expressed in terms of this referential, $$v(t)=\sum_{i=1}^m \lambda_i(t)e_i(t),$$ with $C^k$ coordinate functions $\lambda_i:[a,b]\to\R$ for $1\leq i\leq m$. Thus $v\in\sect^k(\gamma^*TM)$ may be identified with $\lambda=(\lambda_i)_{i=1}^m\in C^k([a,b],\R^m)$. Notice that this identification is {\em not} canonical, since it depends on the choice of a referential along $\gamma$. Through this identification, it is possible to transfer most results about $C^k([a,b],\R^m)$ to $\sect^k(\gamma^*TM)$, such as the Stone--Weierstrass Theorem that gives density of $\sect^\infty(\gamma^*TM)$ in $\sect^0(\gamma^*TM)$, see Theorem~\ref{thm:stw2}.

Moreover, we may also consider vector fields along curves with weaker regularities, as $L^p$ or Sobolev class $H^k$.

\begin{definition}\label{def:campol2}
Consider $\gamma:[a,b]\to M$ a continuous curve. A vector field $v\in\sect(\gamma^*TM)$ is {\em of class $L^2$}\index{Vector field!along a curve!$L^2$} if the curve
\begin{equation}
[a,b]\ni t\longmapsto (\gamma(t),v(t))\in TM
\end{equation}
is measurable and for every local chart $\varphi:U\to\varphi(U)$ of $TM$ and for every interval $[c,d]\subset[a,b]$ with $\gamma([c,d])\subset U$, the curve $$[c,d]\ni t\longmapsto \dd\varphi(\gamma(t))v(t)\in\R^{2m}$$ is class $L^2$ in the sense of Definition~\ref{def:lp}. The vector subspace of $\sect(\gamma^*TM)$ formed by $L^2$ vector fields along $\gamma$ is denoted $\sect^{L^2}(\gamma^*TM)$.
\end{definition}

\begin{definition}\label{def:campoh1}
Consider $\gamma:[a,b]\to M$ a curve of Sobolev class $H^k$. A vector field $v\in\sect^0(\gamma^*TM)$ is {\em of Sobolev class $H^k$}\index{Sobolev class $H^k$!vector field}\index{Vector field!along a curve!$H^k$} if the curve
\begin{equation}\label{eq:vaolongodegamma}
[a,b]\ni t\longmapsto (\gamma(t),v(t))\in TM
\end{equation}
is continuous and for every local chart $\varphi:U\to\varphi(U)$ of $TM$ and for every interval $[c,d]\subset[a,b]$ with $\gamma([c,d])\subset U$, the curve $$[c,d]\ni t\longmapsto \dd\varphi(\gamma(t))v(t)\in\R^{2m}$$ is of Sobolev class $H^k$ in the sense of Definition~\ref{def:hk}. The vector subspace of $\sect^0(\gamma^*TM)$ formed by Sobolev class $H^k$ vector fields along $\gamma$ is denoted $\sect^{H^k}(\gamma^*TM)$.
\end{definition}

\begin{remark}
In Section~\ref{sec:hum}, it will be clear that this is equivalent to \eqref{eq:vaolongodegamma} being a Sobolev $H^k$ curve in a more general sense, see Definition~\ref{def:H1abM}.
\end{remark}

The identifications established between $C^k([a,b],\R^m)$ and $\sect^k(\gamma^*TM)$ in the case of a $C^l$ curve can also be made\footnote{Notice that regularity of vector fields along $\gamma:[a,b]\to M$ are {\em at most} the same as the regularity of $\gamma$. For instance, concepts as a Sobolev $H^1$ vector field along a continuous curve, or a $C^3$ vector field along a $C^2$ curve do note make sense. In case the regularity of a curve is not specified, it is implicit that it is at least equal to the regularity of the considered vector fields along it.} in the context of $L^2$ or Sobolev $H^k$ curves. Namely, $\sect^{L^2}(\gamma^*TM)$ and $\sect^{H^k}(\gamma^*TM)$ are Banach spaces respectively identified with $L^2([a,b],\R^m)$ and $H^k([a,b],\R^m)$. Notice however that these identifications are {\em not} canonical, since they depend on the choice of a referential along $\gamma$. Through these, it is possible to state similar results to Propositions~\ref{prop:inclusions} and~\ref{prop:cinftylp} and Corollary~\ref{cor:cinftyhk}, as follows. Notice that the proof of these results is immediate from their versions regarding maps in Euclidean space.

\begin{proposition}\label{prop:inclusions2}
The following inclusion maps are continuous:
\begin{itemize}
\item[(i)] $\sect^l(\gamma^*TM)\hookrightarrow\sect^k(\gamma^*TM)$, for $\gamma$ of class $C^l$ and $0\leq k\leq l$;
\item[(ii)] $\sect^0(\gamma^*TM)\hookrightarrow\sect^{L^2}(\gamma^*TM)$, for $\gamma$ of class $C^0$;
\item[(iii)] $\sect^k(\gamma^*TM)\hookrightarrow\sect^{H^k}(\gamma^*TM)$, for $\gamma$ of class $C^k$ and $k\geq 0$.
\end{itemize}
\end{proposition}

\begin{stwthm}\label{thm:stw2}\index{Theorem!Stone--Weierstrass}
If $\gamma$ is of class $C^\infty$, then $\sect^\infty(\gamma^*TM)$ is dense in $\sect^0(\gamma^*TM)$.
\end{stwthm}

\begin{proposition}\label{prop:cinftylp2}
If $\gamma$ is smooth, then the space $\sect^\infty(\gamma^*TM)$ is dense in $\sect^{L^2}(\gamma^*TM)$. In particular, if $\gamma$ is of class $C^k$, then $\sect^k(\gamma^*TM)$ is dense in $\sect^{L^2}(\gamma^*TM)$.
\end{proposition}

\begin{corollary}\label{cor:cinftyhk2}
If $\gamma$ is smooth, then the space $\sect^\infty(\gamma^*TM)$ is dense in $\sect^{H^k}(\gamma^*TM)$, for $k\geq 1$. In particular, if $\gamma$ is of class $C^{k+j}$ for some $j\geq 0$, then $\sect^{k+j}(\gamma^*TM)$ is dense in $\sect^{H^k}(\gamma^*TM)$.
\end{corollary}

\section[Hilbert manifold of Sobolev $H^1$ curves]{\texorpdfstring{Hilbert manifold of Sobolev $H^1$ curves}{Hilbert manifold of Sobolev H1 curves}}
\label{sec:hum}

In this section, we endow the set of Sobolev curves $H^1([a,b],M)$ on a $m$--dimensional (smooth) manifold $M$ with a separable Hilbert manifold structure, following closely the approach of Mercuri, Piccione and Tausk \cite{picmertausk}. We also describe the structure of the tangent spaces $T_\gamma H^1([a,b],M)$ and endow $H^1([a,b],M)$ with a (infinite--dimensional) Riemann metric, see Definition~\ref{def:metricinfinitemnfld}.

\begin{definition}\label{def:H1abM}
A curve $\gamma:[a,b]\to M$ on $M$ is of {\em Sobolev class $H^k$} if it is continuous and for every local chart $\varphi:U\to\varphi(U)$ of $M$ and for every interval $[c,d]\subset[a,b]$ with $\gamma([c,d])\subset U$, the curve $\varphi\circ\gamma\vert_{[c,d]}:[c,d]\to\R^m$ is of class $H^k$ in the sense of Definition~\ref{def:hk}. The set of all Sobolev $H^k$ curves $\gamma:[a,b]\to M$ is denoted $H^k([a,b],M)$.
\end{definition}

\begin{remark}\label{re:covderh1}
If $\gamma:[a,b]\to M$ is a curve of Sobolev class $H^1$, there exists a version of covariant derivative operator\index{Covariant derivative} $$\D^g:\sect^k(\gamma^*TM)\la\sect^{k-1}(\gamma^*TM)$$ of Proposition~\ref{prop:covder} defined\footnote{Notice that if $\gamma$ has this regularity, it is meaningless to consider $C^k$ vector fields along $\gamma$.} on $\sect^{H^1}(\gamma^*TM)$. More precisely, any $v\in\sect^{H^1}(\gamma^*TM)$ can be regarded as a curve $v:[0,1]\to TM$ of Sobolev class $H^1$, with $v(t)\in T_{\gamma(t)}M$, for all $t\in [0,1]$, recall Definition~\ref{def:campoh1}. From Proposition~\ref{prop:abscontcharact}, it is easy to conclude\footnote{Formally, it would be necessary to consider a referential and express the covariant derivative $\D^\mathrm R v$ locally in terms of ordinary derivatives in Euclidean space, and then use Proposition~\ref{prop:abscontcharact}.} that the covariant derivative $\D^\mathrm Rv$ is defined {\em almost everywhere} using formulas \eqref{eq:nablachr} and \eqref{eq:dgx}. This gives a continuous covariant derivative operator $$\D^\mathrm R:\sect^{H^1}(\gamma^*TM)\la\sect^{L^2}(\gamma^*TM).$$ Henceforth, when dealing with covariant derivatives of Sobolev $H^1$ vector fields we always mean in this form, and almost everywhere.
\end{remark}

\begin{definition}\label{def:familycharts}
A {\em one--parameter family of (smooth) charts}\index{Chart!one--parameter family} on $M$ is a smooth map $\fibr\varphi:U\to\R^{m+1}$ defined on an open subset $U$ of $\R\times M$ such that $$\fibr\varphi:U\ni(t,x)\longmapsto(t,\varphi(t,x))\in\R\times\R^m$$ is a smooth diffeomorphism onto an open subset $\fibr\varphi(U)\subset\R\times\R^m$. For each $t\in\R$, set $U_t=\{x\in M:(t,x)\in U\}$, which is an open (and possibly empty) subset of $M$, and $$\varphi_t:U_t\ni x\longmapsto\varphi(t,x)\in\R^m.$$ A one--parameter family of charts with $\varphi_t$ and $U_t$ defined as above will be denoted by the pair $\varphi=(\varphi_t,U_t)$.
\end{definition}

\begin{remark}
Clearly, if $\varphi=(\varphi_t,U_t)$ is a one--parameter family of charts, then $\varphi_t:U_t\to\varphi_t(U_t)$ is a local chart on $M$ for every $t$. Conversely, it follows from the Inverse Function Theorem that if $\varphi$ is smooth and each $\varphi_t$ is a local chart, then $\varphi$ is a one--parameter family of charts.
\end{remark}

\begin{definition}\label{def:curvesalphaM}
If $U\subset\R\times M$ is an open subset, consider the set of Sobolev $H^1$ curves $\gamma:[a,b]\to M$ whose graph is contained in
$U$, $$\curves U=\left\{\gamma\in H^1([a,b],M):(t,\gamma(t))\in U, \mbox{ for all }t\in[a,b]\right\}.$$ If $N$ is a smooth manifold and $\alpha:U\to N$ is smooth, it is possible to define
\begin{equation}\label{eq:curvesalphaM}
\begin{aligned}
\curves\alpha:\curves U&\la H^1([a,b],N)\\
\curves\alpha(\gamma)(t)&=\alpha(t,\gamma(t)), \quad t\in [a,b],
\end{aligned}
\end{equation}
analogously to \eqref{eq:curvesalpha}.
\end{definition}

\begin{figure}[htf]
\begin{center}
\vspace{-0.4cm}
\includegraphics[scale=1]{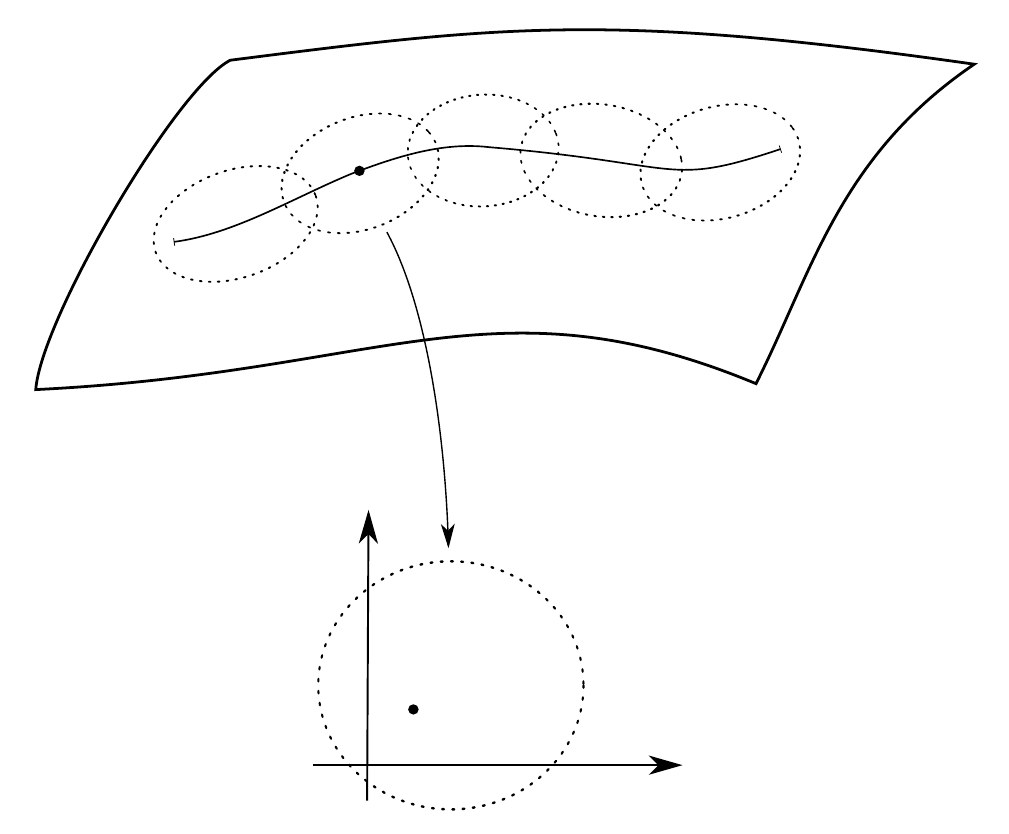}
\begin{pgfpicture}
\pgfputat{\pgfxy(-7.2,7.5)}{\pgfbox[center,center]{$U_t$}}
\pgfputat{\pgfxy(-6.6,6.4)}{\pgfbox[center,center]{$\gamma(t)$}}
\pgfputat{\pgfxy(-2.1,7)}{\pgfbox[center,center]{$\gamma$}}
\pgfputat{\pgfxy(-5.5,4.5)}{\pgfbox[center,center]{$\varphi_t$}}
\pgfputat{\pgfxy(-4,2.6)}{\pgfbox[center,center]{$\R^m$}}
\pgfputat{\pgfxy(-5.5,1)}{\pgfbox[center,center]{$\varphi(t,x)$}}
\pgfputat{\pgfxy(-1.6,5.8)}{\pgfbox[center,center]{$M$}}
\end{pgfpicture}
\end{center}
\caption{A one--parameter family $\varphi=(\varphi_t,U_t)$ of smooth charts on $M$ and a curve $\gamma:[a,b]\to M$ with graph contained in $U\subset \R\times M$.}
\end{figure}

Notice that if $\varphi=(\varphi_t,U_t)$ is a one--parameter family of charts in $M$, then $\curves\varphi$ gives a bijection from $\curves U$ to $\curves{\varphi(U)}$.

\begin{lemma}
Consider two one--parameter families of charts $\varphi=(\varphi_t,U_t)$ and $\psi=(\psi_t,V_t)$. If $U\cap V\ne\emptyset$ then it is possible to define a {\em transition function} from $\varphi$ to $\psi$ by
\begin{equation}\label{eq:alphaM}
\begin{aligned}
\alpha:\underbrace{\bigcup_{t\in\R}\{t\}\times\varphi_t(U_t\cap V_t)}_{\fibr\varphi(U\cap V)} &\la\underbrace{\bigcup_{t\in\R}\{t\}\times\psi_t(U_t\cap V_t)}_{\fibr\psi(U\cap V)}\\
(t,v) &\longmapsto (t,(\psi_t\circ\varphi_t^{-1})(v))
\end{aligned}
\end{equation}
The transition function $\alpha=\fibr\psi\circ(\fibr\varphi)^{-1}$ is a smooth diffeomorphism between open subsets of $\R\times\R^m$, and $$\curves\alpha=\curves\psi\circ(\curves\varphi)^{-1}:\curves{\fibr\varphi(U\cap V)}\longrightarrow\curves{\fibr\psi(U\cap V)}$$ is a smooth diffeomorphism between open subsets of $H^1([a,b],\R^m)$.
\end{lemma}

\begin{proof}
From Definitions~\ref{def:H1abM} and~\ref{def:curvesalphaM}, the map \eqref{eq:alphaM} is clearly a smooth diffeomorphism. Applying Theorem~\ref{thm:curvesalphasmooth}, it follows that $\curves\alpha$ is also a smooth diffeomorphism.
\end{proof}

\begin{corollary}\label{cor:atlash1}
For every one--parameter family of charts $\varphi=(\varphi_t,U_t)$, the pair $(\curves{U},\curves\varphi)$, with
\begin{equation}\label{eq:curvesvarphiM}
\begin{aligned}
\curves\varphi:\curves U\subset &H^1([a,b],M)\la H^1([a,b],\R^m) \\
\curves\varphi(\gamma)(t) &=\varphi(t,\gamma(t)), \quad t\in [a,b],
\end{aligned}
\end{equation}
is a chart on $H^1([a,b],M)$, see Definition~\ref{def:curvesalphaM}. Moreover, charts of this form are pairwise smoothly compatible.
\end{corollary}

In order to obtain a differentiable atlas for $H^1([a,b],M)$, see Definition~\ref{def:chartatlasmnfld}, we now need to show that the domains of such charts $\curves\varphi$ cover $H^1([a,b],M)$. This will be a consequence of the following result.

\begin{proposition}\label{prop:existVBN}
Given a continuous curve $\gamma:[a,b]\to M$ on a differentiable manifold $M$ then there exists a one--parameter family of charts $\varphi=(\varphi_t,U_t)$ on $M$ such that $U$ contains the graph of $\gamma$.
\end{proposition}

\begin{proof}
Consider the auxiliary Riemannian metric $g_\mathrm R$ on $M$. It is a standard argument in Riemannian geometry (see, for instance, \cite{jost,kn1,petersen}) that for every compact subset $K\subset M$ there exists $r>0$ that is a totally normal radius for all points of $K$, see Definition~\ref{def:normalradius}.

Consider an arbitrary continuous extension of $\gamma$ to a curve defined in $\R$. Let $r>0$ be a totally normal radius for all points of the compact set $K=\gamma([a-1,b+1])$. By standard approximation arguments, see Hirsch \cite{hirsch}, there exists a smooth curve $\mu:\left]a-1,b+1\right[\to M$ such that $d_\mathrm R(\gamma(t),\mu(t))<r$ for all $t\in\left]a-1,b+1\right[$, where $d_\mathrm R$ is the Riemannian distance\footnote{Recall Definition \ref{def:distance}.} on $M$.

Choose an arbitrary parallel frame along $\mu$, such that $\sigma_t:T_{\mu(t)}M\to\R^m$ is an isomorphism for all $t\in\left]a-1,b+1\right[$. Let $$U_t=\exp_{\mu(t)}(B(0,r)),$$ where $B(0,r)$ is an open ball of radius $r$ around the origin of $T_{\mu(t)}M$. Define $$\varphi_t=\sigma_t\circ\exp_{\mu(t)}^{-1}:U_t\la\R^m$$ as the composition of the inverse of the diffeomorphism $\exp:B(0,r)\to U_t$ with the isomorphism $\sigma_t$, for all $t\in\left]a-1,b+1\right[$. It is a simple verification that $\varphi=(\varphi_t,U_t)$ is a one--parameter family of charts such that $U$ contains the graph of $\gamma$.
\end{proof}

\begin{corollary}\label{cor:atlascobre}
For every $\gamma\in H^1([a,b],M)$, there exists a (smooth) local chart $(\curves{U},\curves\varphi)$ of the form \eqref{eq:curvesvarphiM} around $\gamma$.
\end{corollary}

\begin{theorem}\label{thm:atlasparaH1}
The set $H^1([a,b],M)$ is a smooth Hilbert manifold, locally modeled on the Hilbert space $H^1([a,b],\R^m)$.
\end{theorem}

\begin{proof}
From Corollaries~\ref{cor:atlash1} and~\ref{cor:atlascobre}, the set $\{\curves\varphi\}_\varphi$, where $\varphi$ runs over all possible one--parameter families of charts on $M$, is a smooth atlas for $H^1([a,b],M)$.
\end{proof}

Regarding separability of $H^1([a,b],M)$, we have the following result.

\begin{proposition}\label{prop:embh1emb}
If $f:M\hookrightarrow N$ is a $C^k$ embedding of finite--dimensional smooth manifolds, then the following is a $C^k$ embedding of Hilbert manifolds
\begin{eqnarray*}
&\widehat{f}:H^1([a,b],M)\xhookrightarrow{\;\;\;\;} H^1([a,b],N)& \\
&\widehat f(\gamma)(t)=f(\gamma(t)), \quad t\in [a,b].&
\end{eqnarray*}
\end{proposition}

\begin{proof}
Since $f(M)$ is an embedded submanifold of $N$, it is easy to see that there exists an open neighborhood $U$ of $f(M)$ and a $C^k$ {\em retraction}\footnote{Recall that if $X$ is a topological space and $A\subset X$ is a subspace, a continuous map $r:X\to A$ is a {\em retraction} if the restriction of $r$ to $A$ is the identity map, i.e., $r(a)=a$ for all $a\in A$.} $r:V\to f(M)$, constructed for instance using submanifold charts given in Definition~\ref{def:submnfldchart}. Consider the map
\begin{eqnarray*}
&\widehat r:H^1([a,b],U)\la H^1([a,b],M) & \\
&\widehat r(\gamma)(t)=r(\gamma(t)), \quad t\in [a,b].&
\end{eqnarray*}
This is clearly a $C^k$ left inverse for $\widehat f:H^1([a,b],M)\to H^1([a,b],V)$. From Lemma~\ref{le:invemb}, $\widehat f$ is an embedding. Finally, since $H^1([a,b],V)$ is endowed with the topology of subspace of $H^1([a,b],N)$, the map $\widehat f$ can be regarded as an embedding into the larger space $H^1([a,b],N)$.
\end{proof}

\begin{corollary}\label{cor:h1separable}
The Hilbert manifold $H^1([a,b],M)$ is separable.
\end{corollary}

\begin{proof}
Since $M$ is a smooth $m$--dimensional manifold, from the Whitney Embedding Theorem\footnote{This is a classic result on existence of embeddings in Euclidean space, see \cite{hirsch,whitney}.}, there exists a (smooth) embedding $$f:M\la\R^{2m}.$$ From Proposition~\ref{prop:embh1emb}, $f$ induces a smooth embedding $\widehat f:H^1([a,b],M)\to H^1([a,b],\R^{2m})$. Using the topological isomorphism \eqref{eq:imhkl2}, one has the following composite embedding $$H^1([a,b],M)\xhookrightarrow{\;\;\widehat f\;\;} H^1([a,b],\R^{2m})\xrightarrow{\;\;\cong\;\;} (\R^{2m})^k \oplus L^2([a,b],\R^{2m}).$$ Separability of $L^2([a,b],\R^{2m})$ is a classic result, see Reed and Simon \cite{reedsimon}. Since all the spaces above are metrizable, their separability is a hereditary property\footnote{Recall that a metric space is separable if and only if it is second--countable. Obviously, second--countability is a hereditary property.}, and this concludes the proof.
\end{proof}

\begin{remark}
The proof given above using the Whitney Embedding Theorem is an {\em indirect} proof. It is also possible to give a more direct proof of this result using the continuous inclusion of $H^1([a,b],M)$ in $C^0([a,b],M)$ given in Corollary~\ref{cor:hkckj}. One can construct explicitly a countable dense subset of $C^0([a,b],M)$ using {\em distance functions}, with an argument similar to the one used to prove Proposition~\ref{prop:cksep}, for instance in Fabi\'an et al. \cite{fabian}.
\end{remark}

We now study the tangent space to $H^1([a,b],M)$ at some curve $\gamma$, which is a Hilbertable space that can be constructed in various abstract ways. For instance, one could use equivalence classes of curves or any other general construction for tangent spaces of Hilbert manifolds (see Remark~\ref{re:txx}).

In order to obtain a more concrete description of $T_\gamma H^1([a,b],M)$, consider for each $t_0\in[a,b]$ the evaluation map\index{Evaluation map}
\begin{eqnarray}
\ev_{t_0}:H^1([a,b],M) &\la& M \\
\gamma &\longmapsto &\gamma(t_0).\nonumber
\end{eqnarray}
If $\varphi=(\varphi_t,U_t)$ is a one--parameter family of charts in $M$, then the following diagram is commutative.
\begin{equation}
\xymatrix@+20pt{
H^1([a,b],M) &\ar@{_{(}->}[l] \curves U\ar[d]_{\ev_{t_0}}\ar[r]^{\curves\varphi} & \curves{\varphi(U)}\ar[d]^{\ev_{t_0}} \ar@{^{(}->}[r] & H^1([a,b],\R^m) \\
&U_t\ar[r]_{\varphi_{t_0}} & \varphi_t(U_t) &
}\label{eq:Evals}
\end{equation}

\begin{remark}\label{re:evalsmooth}
From \eqref{eq:Evals}, the map $$\ev_{t_0}:H^1([a,b],M)\la M$$ is represented in the local charts $\curves\varphi$ and $\varphi_{t_0}$ by $$\ev_{t_0}:H^1([a,b],\R^m)\la\R^m.$$ Hence $\ev_{t_0}:H^1([a,b],M)\to M$ is smooth for every $t_0\in[a,b]$.
\end{remark}

We now identify the tangent bundle $TH^1([a,b],M)$ with the Hilbert manifold $H^1([a,b],TM)$ of curves on $TM$, which will allow to identify the tangent space $T_\gamma H^1([a,b],M)$ with a space of sections of $\gamma^*TM$, i.e., vector fields along $\gamma$.

\begin{proposition}\label{prop:TH1H1T}
For every $\gamma\in H^1([a,b],M)$ and $\mathfrak v\in T_\gamma H^1([a,b],M)$, consider
\begin{equation}\label{eq:deval}
v(t)=\dd(\ev_t)(\gamma)\mathfrak v, \quad t\in[a,b]
\end{equation}
so that $v:[a,b]\to TM$ is a vector field along $\gamma$, i.e. a section of $\gamma^*TM$. Then the curve $v:[a,b]\to TM$ is of Sobolev class $H^1$ and the following is a smooth diffeomorphism of Hilbert manifolds.
\begin{equation}\label{eq:TH1H1T}
TH^1([a,b],M)\ni\mathfrak v\longmapsto v\in H^1([a,b],TM).
\end{equation}
\end{proposition}

\begin{proof}
To prove the above claim, we will verify that the representation of \eqref{eq:TH1H1T} with appropriate local charts is a smooth diffeomorphism. Let $\varphi=(\varphi_t,U_t)$ be a one--parameter family of charts in $M$. For every $t$, we have that $\dd\varphi_t:TU_t\to\varphi_t(U_t)\times\R^m$ is a local chart in $TM$ defined on the open subset $TU_t\subset TM$. The differential of $\curves\varphi$ gives a local chart
\begin{equation}\label{eq:TH1H1Tdphi}
\dd\curves\varphi:T\curves U\longrightarrow\curves{\varphi(U)}\times H^1([a,b],\R^m)
\end{equation}
on the tangent bundle $TH^1([a,b],M)$.

In addition, it is easy to see that $\psi=(\dd\varphi_t,TU_t)$ is a one--parameter family of charts in $TM$, and
\begin{equation}\label{eq:TH1H1Tcpsi}
\curves\psi:\curves{TU}\la\curves{\varphi_t(U_t)\times\R^m}\cong\curves{\varphi(U)}\times H^1([a,b],\R^m),
\end{equation}
is a local chart on $H^1([a,b],TM)$. Differentiating \eqref{eq:Evals} it follows that the following diagram is commutative.
$$\xymatrix@C-10pt{
T\curves U\ar@(d,dl)[dr]\ar[rr]^{\eqref{eq:TH1H1T}} && \curves{TU}\ar@(d,dr)[dl]\\
&\curves{\varphi(U)}\times H^1([a,b],\R^m)}$$
where \eqref{eq:TH1H1T} is considered restricted to the appropriate domain and counter domain, and the vertical arrows are given by the maps \eqref{eq:TH1H1Tdphi} and \eqref{eq:TH1H1Tcpsi} respectively. Thus \eqref{eq:TH1H1T} is represented by the identity with respect to local charts.
\end{proof}

\begin{remark}\label{re:TH1H1T}
Henceforth, the we will use \eqref{eq:TH1H1T} as an identification $$TH^1([a,b],M)\cong H^1([a,b],TM),$$ and, in particular, for each $\gamma\in H^1([a,b],M)$, the following are also identified $$T_\gamma H^1([a,b],M)\cong\sect^{H^1}(\gamma^*TM).$$
\end{remark}

Using this identification of $TH^1([a,b],M)$, we may establish certain properties of a special {\em double} evaluation map, namely the {\em endpoints map}. These will later be used to describe the adequate setting for endpoints conditions on geodesics variational problems, see Lemma~\ref{le:opqsmnfld} and Proposition~\ref{prop:opmsubmfld}.

\begin{proposition}\label{prop:endpointsmap}
The following {\em endpoints map}\index{Endpoints map} is a smooth submersion.
\begin{equation}\label{eq:ev01}
\begin{aligned}
\ev_{01}=(\ev_0,\ev_1):H^1([0,1],M) &\la M\times M \\
\gamma &\longmapsto (\gamma(0),\gamma(1)).
\end{aligned}
\end{equation}
\end{proposition}

\begin{proof}
From Remark~\ref{re:evalsmooth}, the endpoints map $\ev_{01}$ is smooth. Moreover, from \eqref{eq:deval}, using identification \eqref{eq:TH1H1T} explained in Remark~\ref{re:TH1H1T}, it follows that
\begin{equation}\label{eq:dev01}
\begin{aligned}
\dd(\ev_{01})(\gamma):T_\gamma H^1([0,1],M)&\la T_{(\gamma(0),\gamma(1))}(M\times M)\\
v &\longmapsto (v(0),v(1)).
\end{aligned}
\end{equation}
Given any $(v_0,v_1)\in T_{(\gamma(0),\gamma(1))}(M\times M)$, it is easy to construct a vector field\footnote{Recall Remark~\ref{re:TH1H1T}.} $v\in\sect^{H^1}(\gamma^*TM)$ with $v(0)=v_0$ and $v(1)=v_1$. This implies that $\dd(\ev_{01})(\gamma)$ is surjective at all $\gamma\in H^1([0,1],M)$. Moreover, $\ker\dd(\ev_{01})(\gamma)$ has finite codimension, and hence is automatically complemented\footnote{Actually, being a closed subspace of a Hilbert space, it is obviously complemented.} (see Lemma~\ref{le:finitecomplemented}). Thus, $\ev_{01}$ is a submersion, concluding the proof.
\end{proof}

\begin{corollary}\label{cor:h1s1}
The set $H^1(S^1,M)$ of closed curves on $M$ of Sobolev class $H^1$ is a smooth separable Hilbert manifold.
\end{corollary}

\begin{proof}
Let $\Delta\subset M\times M$ be the diagonal submanifold. From Proposition~\ref{prop:endpointsmap}, the endpoints map $\ev_{01}$ given by \eqref{eq:ev01} is a smooth submersion. In particular, $\ev_{01}$ is transverse to $\Delta$, see Remark~\ref{re:transversalityeq} and Definition~\ref{def:transversality}. Thus, from Proposition~\ref{prop:transvsubmnfld}, $$H^1(S^1,M)=\ev_{01}^{-1}(\Delta)$$ is a smooth submanifold of $H^1([0,1],M)$, hence a smooth Hilbert manifold. Separability of $H^1(S^1,M)$ follows from Corollary~\ref{cor:h1separable}, since these are first--countable spaces, and hence separability is hereditary, since it is equivalent to second--countability.
\end{proof}

\begin{remark}
Moreover, $H^1(S^1,M)$ may be regarded as the set of Sobolev $H^1$ maps from $S^1$ to $M$, in a similar sense to Definition~\ref{def:H1abM}. This follows from the fact that there is a clear identification between the circle and the quotient of an interval by its endpoints, and a curve $\gamma:[0,2\pi]\to M$ passes to the quotient $S^1$ in the following diagram
\begin{equation*}
\xymatrix@+10pt{[0,2\pi]\ar@(r,lu)[rd]^\gamma\ar[d] & \\ \dfrac{[0,2\pi]}{\{0,2\pi\}}\cong S^1\ar[r]_(.6){\overline\gamma} & M}
\end{equation*}
if and only if $\gamma(0)=\gamma(2\pi)$. The curve $\overline\gamma$ is of the same Sobolev class of $\gamma$ as a consequence of the following fact. If the restrictions of a continuous map to subsets that form a finite partition of its domain are of Sobolev class $H^1$, then this map must also be of Sobolev class $H^1$, see Rudin \cite{rudin}.

Therefore, there is a clear identification between elements $\gamma\in\ev_{01}^{-1}(\Delta)$ and {\em Sobolev $H^1$ curves} $\gamma:S^1\to M$, in the sense that $\gamma$ is absolutely continuous (see Definition~\ref{def:abscont} and Proposition~\ref{prop:abscontcharact}) and $\dot\gamma\in \sect^{L^2}(\gamma^*TM)$.
\end{remark}

\begin{proposition}\label{prop:RiemannH1}
Let $(M,g_\mathrm R)$ be a finite--dimensional Riemannian manifold and $\nabla^\mathrm R$ its Levi--Civita connection. For each $\gamma\in H^1([a,b],M)$,
\begin{equation}\label{eq:RiemannH1}
\llangle v,w\rrangle_\gamma=g_\mathrm R(v(a),w(a))+\int_a^b g_\mathrm R(\D^\mathrm R v,\D^\mathrm R w)\,\dd t
\end{equation}
defines\footnote{Recall Remark~\ref{re:covderh1}.} a Hilbert space inner product on $T_\gamma H^1([a,b],M)$. Moreover, the family $$H^1([a,b],M)\ni\gamma\longmapsto\llangle\cdot,\cdot\rrangle_\gamma,$$ defines a Riemannian metric on $H^1([a,b],M)$.
\end{proposition}

For a proof of Proposition~\ref{prop:RiemannH1} we refer to Mercuri, Piccione and Tausk \cite[Proposition 4.4.10]{picmertausk}. Observe also that for $M=\R^m$, the above Riemannian metric coincides with \eqref{eq:innprodH1}, in Remark~\ref{re:eqhk}. We also observe that, from \cite[Theorem 4.4.16]{picmertausk}, the Riemann--Hilbert manifold $H^1([a,b],M)$ endowed with \eqref{eq:RiemannH1} is {\em complete} provided that $(M,g_\mathrm R)$ is complete. This is later used in the same reference to prove that the standard energy functional with domain $H^1([a,b],M)$ satisfies the Palais--Smale condition. For more details on this topic, we refer to Mercuri, Piccione and Tausk \cite{picmertausk}.

\section{Actions of Lie groups on Hilbert manifolds}
\label{sec:actions}

In this section, we briefly recall some basic concepts related to actions of Lie groups\index{Lie group}\footnote{By {\em Lie group} we mean a (possibly infinite--dimensional) smooth Hilbert manifold endowed with a group structure, such that the map $G\times G\ni (g_1,g_2)\mapsto g_1g_2^{-1}\in G$ is smooth. We will also denote $\mathfrak g\cong T_eG$ the tangent space at the identity, that carries a Lie algebra structure given by the Lie bracket $[\cdot,\cdot]:\mathfrak g\times\mathfrak g\to\mathfrak g$.} on Hilbert manifolds. As a general idea, the existence of an (isometric) action of a group $G$ on a space $Y$ means that $Y$ has {\em symmetries of $G$ type}. For instance, if $G=S^1$, it means that $Y$ is {\em rotationally symmetric}, see an example in Figure \ref{fig:rotationallysym}. This general idea can be extrapolated to an infinite--dimensional context, as we will remark in the sequel. Moreover, in the presence of such symmetries codified in the form of a group action, it is possible to simplify the study of {\em invariant} functionals. These represent variational problems with symmetries, in the sense that the related functional is invariant under a certain group action. This is the case in many different problems of geometric calculus of variations, and analysis of criticality and degeneracy can be severely simplified by a clever use of such symmetries, as it will be shown in Chapter~\ref{chap5}.

\begin{figure}[htf]
\begin{center}
\vspace{-0.4cm}
\includegraphics[scale=1]{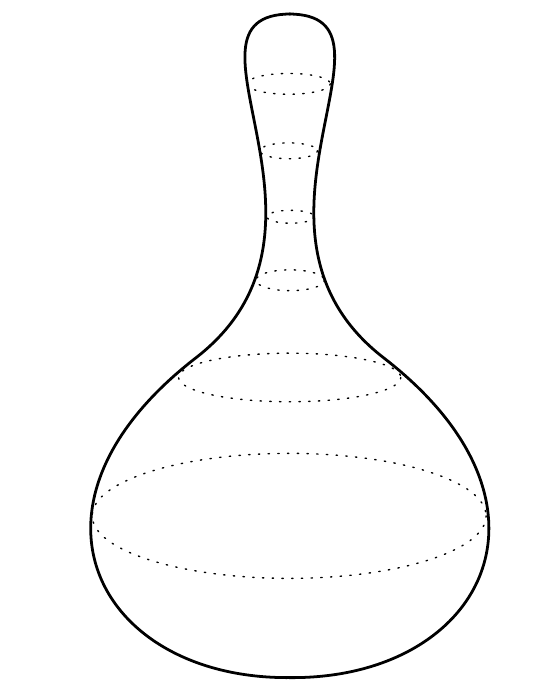}
\end{center}
\caption{An example of rotationally symmetric object. This symmetry may be characterized as an action of the Lie group $S^1$.}\label{fig:rotationallysym}
\end{figure}

Results in this direction and also concerning Morse theory on Hilbert manifolds acted upon by Lie groups, the so--called {\em equivariant} Morse theory, have been extensively studied in the literature.\footnote{Actually, most results are on {\em linear} actions of Lie groups on Hilbert {\em spaces}. The nonlinear version of this problem, for actions on Hilbert {\em manifolds}, still leaves several open questions.} These have applications for instance in the problem of determining if compact manifolds have infinitely many geometrically distinct periodic geodesics. In connection with this topic, in Section~\ref{sec:equivariantgenericity} we give an abstract result on nondegeneracy of critical points of an invariant functional, that is later used in the proof of the Bumpy Metric Theorem~\ref{thm:bumpy}. For a finite--dimensional introduction to the subject, we refer to \cite{ilgaag,duistermaat,liedover,gov0,gov1,gov3}, and for an infinite--dimensional discussion and applications of equivariant Morse theory, we refer to \cite{pacella,palaismorsehilb,palaissymcrit,PalaisTerng}. In this section, $G$ denotes a finite--dimensional Lie group with Lie algebra $\mathfrak g$ and $Y$ denotes a Hilbert manifold.

\begin{definition}
A {\em (left) action}\footnote{{\em Right actions} are maps of the form $Y\times G\to Y$, analogously defined.}\index{Action} of $G$ on $Y$ is a map $\mu:G\times Y\rightarrow Y$ such that
\begin{itemize}
\item[(i)] $\mu(e,y)=y$, for all $y\in Y$;
\item[(ii)] $\mu(g_1,\mu(g_2,y))=\mu(g_1g_2,y)$, for all $g_1,g_2\in G, y\in Y$.
\end{itemize}
The action is said to be of class $C^k$ or smooth, if $\mu$ is a respectively $C^k$ or smooth map.
\end{definition}

\begin{example}
A simple example is the following. Let $H$ be a Hilbert space and consider the map $$\mu:\Lin(V)\times V\ni (T,v)\longmapsto Tv\in V.$$ Another important example is the \emph{adjoint action} of a Lie group $G$ on its Lie algebra $\mathfrak{g}$, given by the adjoint representation $$\Ad:G\times\mathfrak{g}\ni (g,X)\longmapsto \Ad(g)X=\frac{\dd}{\dd t}\left(g\exp(tX)g^{-1}\right)\Big|_{t=0}\in\mathfrak{g}.$$ Other typical examples are actions of a Lie subgroup $H\subset G$ on $G$ by left multiplication or conjugation. All of these are smooth actions.
\end{example}

\begin{example}\label{ex:hkmn}
Consider $M$ and $N$ finite--dimensional smooth Riemannian manifolds, $M$ compact and $H^k(M,N)$ the Hilbert manifold of embeddings $M\hookrightarrow N$ of Sobolev class $H^k$. Then it is easy to see that the group $\Iso(N)$ of isometries of the ambient space acts smoothly on the Hilbert manifold $H^k(M,N)$ by left composition.
\end{example}

The next example will be explored in more details along this section, since it is in connection with the periodic and iterate geodesics problem.

\begin{example}\label{ex:s1h1}
Let $G=S^1$ and $Y=H^1(S^1,M)$, recall Corollary~\ref{cor:h1s1}. Then $G$ clearly acts\index{Action!reparameterization}\index{Reparameterization action} on $Y$ by right composition, reparameterizing curves, i.e.,
\begin{equation}\label{eq:mus1h1}
\begin{aligned}
\rho:S^1\times H^1(S^1,M)&\la H^1(S^1,M)\\
\rho(e^{i\theta},\gamma)(z)&=\gamma(e^{i\theta}z), \quad z\in S^1.
\end{aligned}
\end{equation}

\begin{figure}[htf]
\begin{center}
\vspace{-0.6cm}
\includegraphics[scale=1]{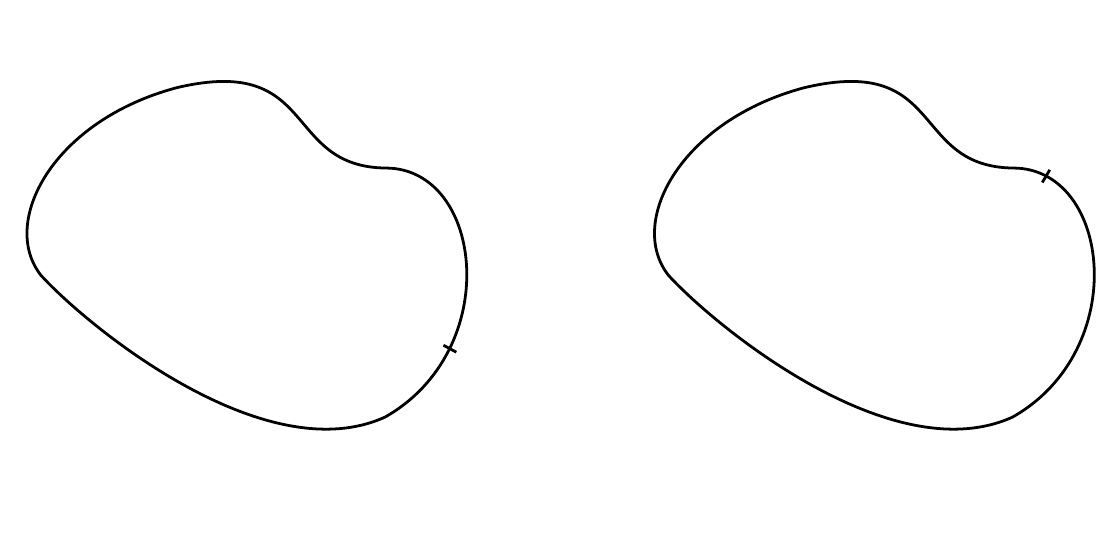}
\begin{pgfpicture}
\pgfputat{\pgfxy(-9,0.3)}{\pgfbox[center,center]{$\gamma$}}
\pgfputat{\pgfxy(-6.2,2)}{\pgfbox[center,center]{$\gamma(z)$}}
\pgfputat{\pgfxy(-1.4,3.4)}{\pgfbox[center,center]{$\gamma(e^{i\theta}z)$}}
\pgfputat{\pgfxy(-2.3,0.3)}{\pgfbox[center,center]{$\rho(e^{i\theta},\gamma)$}}
\end{pgfpicture}
\end{center}
\caption{Reparameterization $\rho(e^{i\theta},\gamma)$ of $\gamma$, that {\em rotates} $\gamma$ of $e^{i\theta}$.}\label{fig:s1action}
\end{figure}

This action is continuous but not differentiable. Indeed, if it were differentiable, for any $\gamma\in H^1(S^1,M)$ and $w\in S^1$, the following composite map would also be differentiable $$S^1\xrightarrow{\;\;\rho_\gamma\;\;} H^1(S^1,M)\xrightarrow{\;\;\ev_{w}\;\;} M,$$ where $\rho_\gamma(z)=\rho(z,\gamma)$ and $\ev_{w}(\gamma)=\gamma(w)$ is the evaluation map at $w$. Nevertheless, this composite map is given by $$S^1\ni z\longmapsto\gamma(zw)\in M,$$ and is a simple reparameterization of $\gamma$, hence not differentiable, since $\gamma\in H^1(S^1,M)$ is only {\em differentiable almost everywhere}, see Proposition~\ref{prop:abscontcharact}. Moreover, there would also be a regularity problem, since the derivative $\dot\gamma$ would be in $\sect^{L^2}(\gamma^*TM)$, and not in the {\em correct} tangent space $\sect^{H^1}(\gamma^*TM)$.

However, notice that the same action $\rho$ regarded in different spaces, for instance $$\rho:S^1\times H^2(S^1,M)\la H^1(S^1,M),$$ does not have these pathologies. In fact, the above map is of class $C^1$, as we will see in Remark~\ref{re:weakreg}.
\end{example}

In case of lack of regularity of the action, as in the above example, there are weaker assumptions that can be usually made. These in general related to regularity of some {\em auxiliary maps}, defined as follows.

\begin{definition}
Given an action $\mu:G\times Y\to Y$, consider
\begin{equation}\label{eq:muaux}
\begin{aligned}
\mu^g:Y&\la Y\quad& \mu_y:G&\la Y \\
y&\longmapsto\mu(g,y)\quad& g&\longmapsto\mu(g,y).
\end{aligned}
\end{equation}
In case the action is $C^k$, these maps are also clearly $C^k$ and their derivatives at $y\in Y$ and $e\in G$ are respectively
\begin{equation}\label{eq:ddmuaux}
\begin{aligned}
\dd\mu^g(y):T_yY&\la T_yY\quad& \dd\mu_y(e):\mathfrak g&\la T_yY \\
\dd\mu^g(y)&=\frac{\partial\mu}{\partial y}(g,y)\quad& \dd\mu_y(e)&=\frac{\partial\mu}{\partial g}(e,y).
\end{aligned}
\end{equation}
\end{definition}

\begin{definition}\label{def:actdiffeos}
An action $\mu:G\times Y\to Y$ is said to be an {\em action by diffeomorphisms}\index{Action!by diffeomorphisms} if $\mu^g:Y\to Y$ is a diffeomorphism of $Y$ for all $g\in G$.
\end{definition}

Let us recall the definition of some basic objects related to an action.

\begin{definition}\label{def:actionobjects}
Given an action $\mu:G\times Y\to Y$, the subgroup
\begin{equation*}
G_y=\{g\in G:\mu(g,y)=y\}
\end{equation*}
is called the {\em isotropy group}\index{Isotropy group}\index{Action!isotropy group} or {\em stabilizer}\index{Action!stabilizer}\index{Stabilizer} of $y\in Y$ and 
\begin{equation*}
G(y)=\{\mu(g,y):g\in G\}=\im \mu_y
\end{equation*}
is called the {\em orbit}\index{Orbit}\index{Action!orbit} of $y\in Y$.

A subset $S\subset Y$ is said to be {\em $G$--invariant}\index{$G$--invariant subset}\index{Action!invariant subset} if $\mu^g(S)\subset S$ for all $g\in G$. In particular, orbits are obviously $G$--invariant. In addition, if $\bigcap_{y\in Y} G_y=\{e\}$, the action is said to be {\em effective}\index{Action!effective} and if $G_y=\{e\}$, for all $y\in M$, it is said to be {\em free}.\index{Action!free} Finally, if given $x,y\in Y$ there exists $g\in G$ with $\mu(g,x)=y$, the action is said to be {\em transitive}.\index{Action!transitive}
\end{definition}

Every orbit $G(y)$ of a $C^k$ action is an {\em immersed} submanifold of $Y$, in a sense weaker than Definition~\ref{def:submnfldchart}. In fact, the map $\mu_y:G\to Y$ is constant on $G_y$ cosets and hence {\em passes to the quotient} inducing a map $\overline{\mu_y}:G/G_y\to Y$. It maps each class $gG_y$ to $\mu_y(g)$, and is clearly well--defined in this way. Furthermore, it is injective and has image equal to the orbit $G(y)$, since it coincides with the image of $\mu_y$.

\begin{wrapfigure}{l}{2cm}
\vspace{-0.5cm}
\xymatrix@+15pt{
G\ar[r]^{\mu_y}\ar[d] & Y \\
\dfrac{G}{G_y}\ar@(r,d)[ru]_{\overline{\mu_y}}
}
\vspace{-0.5cm}
\end{wrapfigure}
This allows to identify $G(y)$ with the quotient\footnote{Notice that, in general, although isotropy groups are Lie subgroups, they are not {\em normal subgroups}. This means that the quotient $G/G_y$ in general is not a Lie group. Nevertheless, here we regard $G/G_y$ as a quotient manifold.} manifold $G/G_y$, as shown in the diagram, where the vertical arrow denotes the quotient map. Since the action is $C^k$, the map $\overline{\mu_y}$ is in fact a $C^k$ immersion. Thus, orbits $G(y)$ are $C^k$ {\em immersed} submanifolds of the Hilbert manifold $Y$, but may not have the induced topology from the ambient, i.e., its inclusion is an immersion, not an embedding. Notice that without further assumptions, the orbits in general need not be closed in $Y$ and even the tangent spaces $T_yG(y)$ need not be closed in $T_yY$.

A way to ensure these properties is to make further properness and Fredholmness assumptions, for instance assuming that for every $y\in Y$, the map $\mu_y:G\to Y$ is a nonlinear Fredholm map, see Definition~\ref{def:nonlinfred}. In this case, each isotropy group $G_y$ has finite dimension and each orbit $G(y)$ has finite codimension (hence has closed and complemented tangent space). More details on general abstract theory of actions of Lie groups on Hilbert manifolds can be found in Palais and Terng \cite{PalaisTerng}.

\begin{figure}[htf]
\begin{center}
\includegraphics[scale=1]{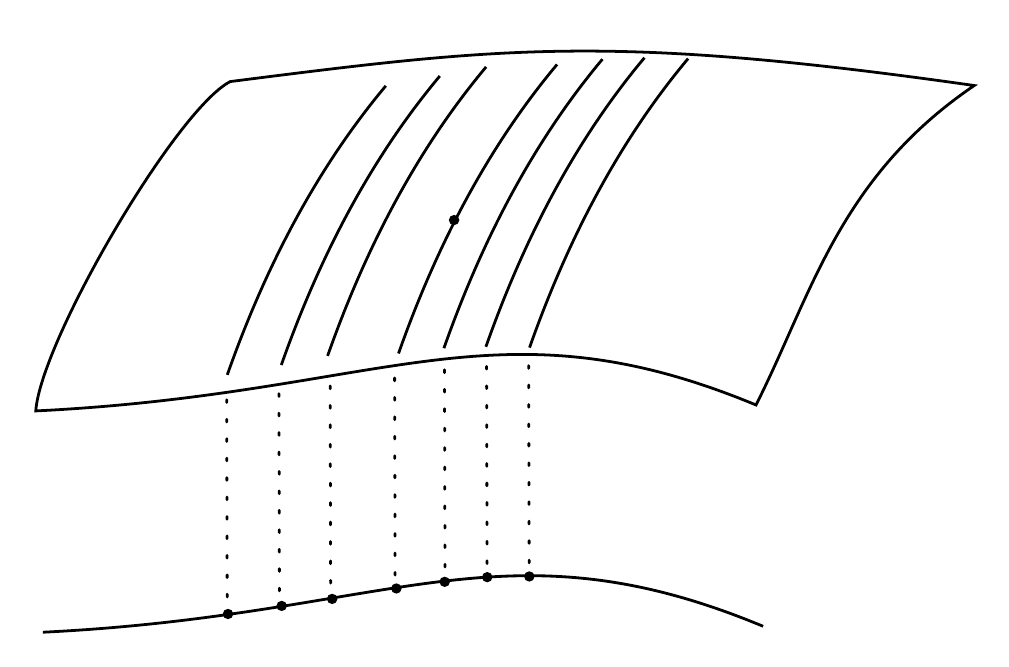}
\begin{pgfpicture}
\pgfputat{\pgfxy(-4.7,6.6)}{\pgfbox[center,center]{$G(y)$}}
\pgfputat{\pgfxy(-6.1,4.5)}{\pgfbox[center,center]{$y$}}
\pgfputat{\pgfxy(-1.4,3.6)}{\pgfbox[center,center]{$Y$}}
\pgfputat{\pgfxy(-2.1,0.3)}{\pgfbox[center,center]{$Y/G$}}
\end{pgfpicture}
\end{center}
\caption{Partition of $Y$ by orbits and orbit space $Y/G$.}\label{fig:orbitsubmnfld}
\end{figure}

Moreover, if two orbits $G(x)$ and $G(y)$ have nontrivial intersection, then they obviously coincide. Hence, orbits of an action of $G$ on $Y$ constitute a {\em partition} of $Y$ by immersed submanifolds, and we may consider the {\em orbit space}\index{Orbit!space} given by the quotient $Y/G$. Since the action is continuous, it is possible to endow $Y/G$ with a quotient topology. Nevertheless, this is in general {\em not a manifold}.\footnote{Under some additional and very restrictive hypotheses, the orbit space is a manifold. For instance, if the action if {\em free} and {\em proper}, then this quotient is a manifold. Usually, the way to deal with orbit spaces is to use the concept of {\em orbifold}, which is essentially a manifold with {\em well--behaved} singularities. More precisely, it has an underlying open dense subset which is a manifold. For an introduction to this subject, we refer to \cite{ilgaag,moerdijk}.}

\begin{remark}\label{re:closediso}
If $\mu:G\times Y\to Y$ is a continuous action, then clearly all isotropy groups are {\em closed} subgroups of $G$. If $G$ is finite--dimensional, then this implies that $G_y$ is a {\em Lie subgroup} of $G$, in particular a Lie group, see Alexandrino and Bettiol \cite{ilgaag}. Nevertheless, closed subgroups of infinite--dimensional Lie groups need not be Lie subgroups.\footnote{A counter--example was given by Bourbaki in 1975, and can be found in \cite{adamsratiuschmid}. Let us briefly describe it, for the reader's convenience. Consider $\ell_2$ the Hilbert space of real sequences $(x_1,x_2,\ldots)$ such that $\sum_{n\in\N} x_n^2<+\infty$. Define $$G_n=\{x\in\ell_2:x_m\in\tfrac1m\Z,1\leq m\leq n\}$$ and observe that $G_n$ is a closed Lie subgroup of $\ell_2$ for all $n\in\N$. Consequently, $H=\bigcap_{n\in\N} G_n$ is a closed subgroup. Nevertheless, it is possible to prove that $H$ is totally disconnected and not discrete, therefore cannot be a submanifold, hence a Lie subgroup.} It is possible to prove that if the closed subgroup is also {\em locally compact}, then it is a finite--dimensional Lie subgroup. Further properness assumptions on the action imply this property for isotropy subgroups, however we will mainly deal with finite--dimensional Lie groups.
\end{remark}

Henceforth, unless otherwise stated, assume that all Lie groups $G$ are finite--dimensional. In particular, orbits of $C^k$ actions are finite--dimensional $C^k$ immersed submanifolds of $Y$ and hence have closed and complemented tangent spaces. Moreover, isotropy subgroups of a continuous action are Lie subgroups. Notice however that no assumptions are being made on $Y$, which is a (possibly infinite--dimensional) Hilbert manifold acted upon by $G$.

Let us explore some of these objects in the case of the reparameterization action \eqref{eq:mus1h1} described in Example~\ref{ex:s1h1},
\begin{eqnarray*}
&\rho:S^1\times H^1(S^1,M)\la H^1(S^1,M)&\\
&\rho(e^{i\theta},\gamma)(z)=\gamma(e^{i\theta}z), \quad z\in S^1.&
\end{eqnarray*}
Recall this is a continuous action, but not differentiable. Nevertheless, the next result gives further regularity properties of this action.

\begin{lemma}\label{le:mus1h1diffeos}
The action $\rho:S^1\times H^1(S^1,M)\to H^1(S^1,M)$ is an action by diffeomorphisms.\footnote{Recall Definition~\ref{def:actdiffeos}.} More precisely, for each $z\in S^1$, $$\rho^z:H^1(S^1,M)\ni\gamma\longmapsto\rho(z,\gamma)=\gamma(z\,\cdot)\in H^1(S^1,M)$$ is a global diffeomorphism, whose derivative at $\gamma\in H^1(S^1,M)$ is given by
\begin{eqnarray}\label{eq:dmuz}
\dd\rho^z(\gamma): T_\gamma H^1(S^1,M) &\la &T_{\rho^z(\gamma)}H^1(S^1,M)\\
v&\longmapsto& v(z\,\cdot).\nonumber
\end{eqnarray}
\end{lemma}

\begin{proof}
Fix $z\in S^1$. Standard arguments prove that $\rho^z$ is differentiable.\footnote{Actually, one possibility is computing the candidate to $\dd\rho^z(\gamma)$ as follows, and then proving it satisfies the definition of derivative of $\rho^z$.} To compute its derivative, we can use evaluation maps in the following way. Let $w\in S^1$ and consider the composite
\begin{equation*}
H^1(S^1,M)\xrightarrow{\;\;\rho^z\;\;}H^1(S^1,M)\xrightarrow{\;\;\ev_w\;\;}M
\end{equation*}
which maps each $\gamma$ to $\gamma(zw)$, hence coincides with $\ev_{zw}$. Its derivative at $\gamma\in H^1(S^1,M)$ is then given by
\begin{equation*}
\begin{aligned}
\dd(\ev_{zw})(\gamma)v &= \dd(\ev_w\circ\rho^z)(\gamma)v\\
&= \dd(\ev_w)(\rho^z(\gamma))\dd\rho^z(\gamma)v\\
&=\big[\dd\rho^z(\gamma)v\big](w),
\end{aligned}
\end{equation*}
for all $v\in T_\gamma H^1(S^1,M)$. In addition, from \eqref{eq:deval},
\begin{equation*}
\dd(\ev_{zw})(\gamma)v = v(zw).
\end{equation*}
Thus, it follows that $\big[\dd\rho^z(\gamma)v\big](w)=v(zw)$, i.e., $$\dd\rho^z(\gamma)v=v(z\,\cdot),$$ hence \eqref{eq:dmuz} holds. Furthermore, notice that this is clearly a continuous and invertible operator. Hence, from the Inverse Function Theorem, $\rho^z:H^1(S^1,M)\to H^1(S^1,M)$ is a local diffeomorphism. In addition, $\rho^z$ is clearly injective, since it admits the left inverse $\rho^{z^{-1}}$. Therefore, $\rho^z$ is a global diffeomorphism, concluding the proof.\footnote{Notice that for each $z\in S^1$, the inverse $\left(\rho^z\right)^{-1}$ is given by $\rho^{z^{-1}}$, which is also a global diffeomorphism, by the same argument.}
\end{proof}

\begin{lemma}\label{le:cyclicgroup}
Let $S^1_\gamma$ be the isotropy subgroup of a non constant curve $\gamma:S^1\to M$. Then $S^1_\gamma$ is a finite cyclic subgroup of $S^1$, hence isomorphic to $\Z_n$ for some $n\in\N$.
\end{lemma}

\begin{proof}
Recall that the isotropy subgroup $S^1_\gamma$ is the subgroup of $S^1$ formed by elements $e^{i\theta}\in S^1$ such that $\gamma(e^{i\theta}z)=\gamma(z)$ for all $z\in S^1$. If $S^1_\gamma=S^1$, it is easy to see that $\gamma$ is a constant curve, which is not the case. Thus we may assume that $S^1_\gamma$ is a proper subgroup of $S^1$. It is well--known that if a proper subgroup of $S^1$ is infinite, then it must be dense. From Remark~\ref{re:closediso}, since the action is continuous, $S^1_\gamma$ is closed and hence finite. Finally, finite subgroups of a field are cyclic.\footnote{Indeed, suppose $G$ is a finite subgroup of a field. Then for a given divisor $d|\# G$ of the order of $G$, either $G$ has no element of order $d$ or at least one. In this last case, $G$ contains a cyclic group of order $d$, which by hypothesis, must contain all solutions of $x^d=1$ in $G$, since $G$ is contained in a field. Thus, in this case, $G$ contains exactly $\phi(d)$ elements of order $d$, where $\phi$ is the Euler phi function. In addition, $\# G=\sum_{d|\# G}\phi(d)$, where $d$ runs over all divisors of $\# G$. No divisor $d|\# G$ is left out, since there are either zero or $\phi(d)$ elements of order $d$ in $G$. However the sum would not add up to $\# G$ if any zero ever occurred, hence $G$ contains elements of all orders, in particular of order $\# G$, proving that $G$ is cyclic.} Thus, since $S^1_\gamma$ is a subgroup of $\C$, it must be a finite cyclic group, hence isomorphic to $\Z_n$, for $n=\# S^1_\gamma$.
\end{proof}

\begin{figure}[htf]
\begin{center}
\includegraphics[scale=0.7]{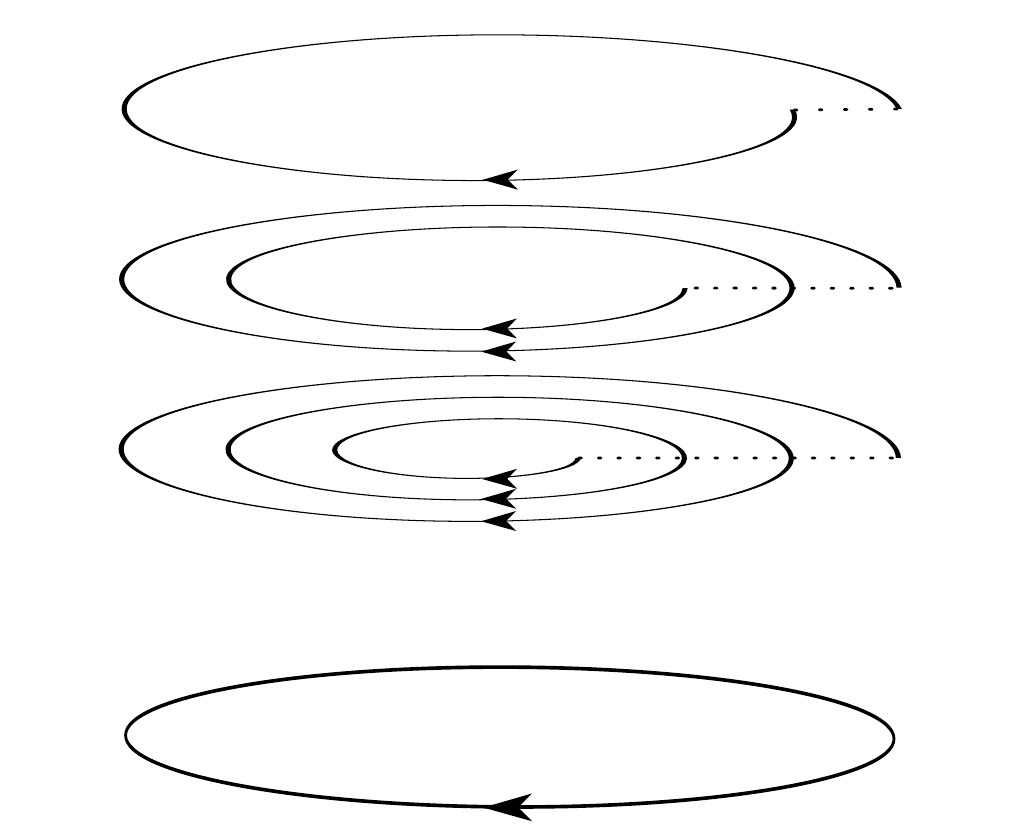}
\begin{pgfpicture}
\pgfputat{\pgfxy(0.1,5.2)}{\pgfbox[center,center]{$S^1_\gamma\cong\{1\}$}}
\pgfputat{\pgfxy(0.1,3.8)}{\pgfbox[center,center]{$S^1_\gamma\cong\Z_2$}}
\pgfputat{\pgfxy(0.1,2.6)}{\pgfbox[center,center]{$S^1_\gamma\cong\Z_3$}}
\pgfputat{\pgfxy(-0.3,0.6)}{\pgfbox[center,center]{$\gamma$}}
\end{pgfpicture}
\end{center}
\caption{A prime curve $\gamma$ and iterates given as $n$--fold covers with respective isotropy groups.}\label{fig:iterates}
\end{figure}

\begin{definition}\label{def:prime}
A curve $\gamma\in H^1(S^1,M)$ is called {\em prime} if $S^1_\gamma$ is trivial, otherwise it is called an {\em iterate}.\index{Prime curve}\index{Geodesic!periodic!prime}\index{Iterate curve}\index{Geodesic!periodic!iterate} Denote $H^1_*(S^1,M)$ the subset of $H^1(S^1,M)$ formed by prime curves.
\end{definition}

\begin{remark}
The order of the isotropy group $S^1_\gamma$ measures how many times {\em $\gamma$ winds itself} around its image. In this sense, prime curves are the periodic curves that make one single twist. More precisely, if $S^1_\gamma$ has order $n$, then it is easy to see that $\gamma$ is the $n$--fold iteration of a prime curve $\gamma_0:S^1\to M$. This means $\gamma$ has the same image of $\gamma_0$, but {\em runs over it} $n$ times, while $\gamma_0$ does it only once.

Another approach for this analysis is considering the {\em period} of a curve $\gamma:S^1\to M$. Namely, the {\em period} of $\gamma$ is defined as the generator of $S^1_\gamma$. Since this is a finite abelian group, its generator is the element of maximal order $\#S^1_\gamma$, which corresponds to the minimal time $\theta$ for $\gamma(e^{i\theta}z)$ to coincide again with $\gamma(z)$ for all $z\in S^1$ after $\theta=0$. If $S^1_\gamma$ has order $n$, its period is an element of order $n$, and hence $\gamma$ {\em makes $n$ turns around its image}, meaning once more it is the $n$--fold iteration of a prime curve $\gamma_0$.
\end{remark}

\begin{remark}
It is possible to prove that $H^1_*(S^1,M)$ is open in $H^1(S^1,M)$. Furthermore, it is clearly $S^1$--invariant.
\end{remark}

Non constant periodic curves with the same image form an infinite class of {\em geometrically indistinct} curves. For the sake of counting periodic geodesics for instance, it is convenient to have this infinite family counted as a single geodesic, otherwise every manifold that admits a periodic geodesic would trivially have {\em infinitely many} periodic geodesics. Thus, it is useful to have a {\em distinguished} representant of a such class, given by the prime curve that {\em generates} all the other iterates.

Let us now examine the orbit of a curve $\gamma\in H^1(S^1,M)$. Due to lack of regularity of this action, the maps $\rho_\gamma$ and $\overline{\rho_\gamma}$ used to identify $S^1(\gamma)$ with $S^1/S^1_\gamma$ are only {\em homeomorphisms}, and not {\em diffeomorphisms}. Thus, $S^1(\gamma)$ is {\em homeomorphic} to the quotient $S^1/S^1_\gamma$. Hence, if $\gamma$ is constant, its orbit is a point. If it is non constant, then its orbit is homeomorphic to $S^1$, since $S^1_\gamma$ is finite from Lemma~\ref{le:cyclicgroup}.

In addition, in case $\gamma$ has more regularity, the orbit is a submanifold. More precisely, if for instance\footnote{Notice that this is the case when $\gamma$ is a geodesic.} $\gamma$ is of class $C^2$, then $S^1(\gamma)$ is a $C^1$ submanifold of $H^1(S^1,M)$. Once more, the proof follows from a simple analysis of the maps $\rho_\gamma$ and $\overline{\rho_\gamma}$.

\begin{definition}\label{def:dy}
A $C^k$ action $\mu:G\times Y\to Y$ induces a natural (finite--dimensional) subspace of $T_yY$ at every $y\in Y$, tangent to $G(y)$, given by
\begin{equation}\label{eq:dy}
\mathcal D_y=\im\dd\mu_y(e).
\end{equation}
If $\dim\mathcal D_y$ does not depend on $y$, then $\mathcal D=\{\mathcal D_y:y\in Y\}$ is a smooth distribution\footnote{Recall Example~\ref{ex:distribution}.} of $Y$.
\end{definition}

\begin{figure}[htf]
\begin{center}
\includegraphics[scale=0.7]{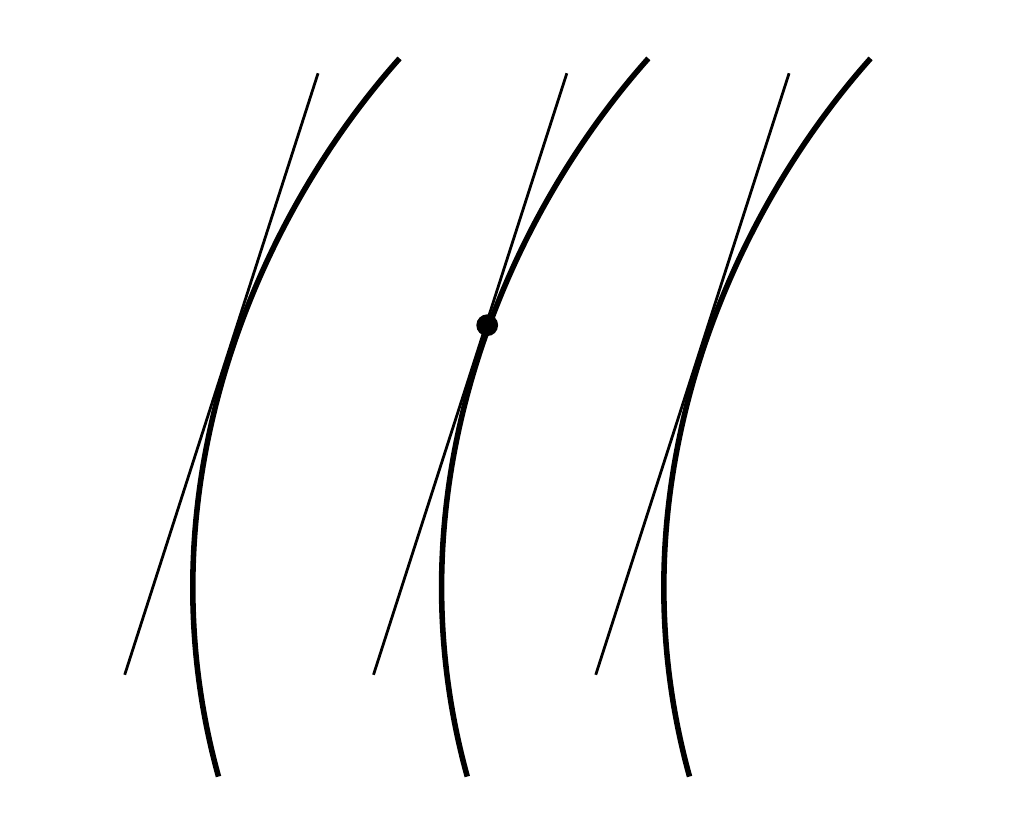}
\begin{pgfpicture}
\pgfputat{\pgfxy(-3.3,5.7)}{\pgfbox[center,center]{$\mathcal D_y$}}
\pgfputat{\pgfxy(-3.85,0)}{\pgfbox[center,center]{$G(y)$}}
\pgfputat{\pgfxy(-4.2,3.6)}{\pgfbox[center,center]{$y$}}
\end{pgfpicture}
\end{center}
\end{figure}

\begin{remark}\label{re:weakreg}
Notice that it is possible to relax the regularity assumptions on the action $\mu$ to define the distribution $\mathcal D$. In fact, suppose $\mu:G\times Y\to Y$ is a (possibly not differentiable) action and that there exists a $G$--invariant dense subset\footnote{This subset $Y_1$ in fact will have a differentiable structure, however its inclusion $Y_1\hookrightarrow Y$ will be continuous but not a homeomorphism. Thus, $Y_1$ may not be regarded as a submanifold of $Y$.} $Y_1$ of $Y$, with and $$\mu_y:G\la Y$$ differentiable for all $y\in Y_1$. Then we may consider for each $y\in Y_1$, $$\mathcal D_{y}=\im\dd\mu_{y}(e)\subset T_{y}Y.$$

This is the case of the action $\rho$ considered in Example~\ref{ex:s1h1}, with $Y_1=H^2(S^1,M).$ Let $\gamma\in H^2(S^1,M)$. Then standard arguments prove that $$\rho_\gamma:S^1\la H^1(S^1,M)$$ is of class $C^1$ for every $\gamma\in Y_1$. In fact, $\dd\rho_\gamma(1)$ is identified with $\dot\gamma$ in the following way. Consider the composite 
\begin{equation*}
\begin{aligned}
S^1\xrightarrow{\;\;\rho_\gamma\;\;} & \;\,H^1(S^1,M)\xrightarrow{\;\;\ev_w\;\;}M\\
\;z\longmapsto & \quad\;\gamma(z\,\cdot)\quad\longmapsto\gamma(zw)
\end{aligned}
\end{equation*}
This is simply a reparameterization of $\gamma$, and coincides with the composite map
\begin{equation*}
\begin{aligned}
S^1\xrightarrow{\;\;R_w\;\;} & \;\,S^1\xrightarrow{\;\;\gamma\;\;}M\\
z\longmapsto & \;zw\longmapsto\gamma(zw)
\end{aligned}
\end{equation*}
Hence, using \eqref{eq:deval}, we may compute the derivative
\begin{equation*}
\begin{aligned}
\dd(\ev_w\circ\rho_\gamma)(z)v&= \dd(\ev_w)(\rho_\gamma(z))\dd\rho_\gamma(z)v\\
&=\big[\dd\rho_\gamma(z)v\big](w),
\end{aligned}
\end{equation*}
for every $v\in T_zS^1$. Moreover, it coincides with the derivative
\begin{equation*}
\begin{aligned}
\dd(\gamma\circ R_w)(z)v&= \dd\gamma(zw)\dd R_w(z)v\\
&=\dot\gamma(zw)vw.
\end{aligned}
\end{equation*}
Therefore, $\dd\rho_\gamma(z)(w)=\dot\gamma(zw)w$ and hence $\dd\rho_\gamma(1)(w)=\dot\gamma(w)w$. This allows to identify
\begin{equation}\label{eq:dgammaspangamma}
\mathcal D_\gamma=\im\dd\rho_\gamma(1)=\operatorname{span}\,\dot\gamma\subset T_\gamma H^1(S^1,M),
\end{equation}
i.e., the subspace $\im\dd\rho_\gamma(1)$ is identified with the one--dimensional subspace spanned by $\dot\gamma$ in $\sect^{H^1}(\gamma^*TM)$. Thus, the distribution $\mathcal D$ of $Y$ is well--defined on points of $Y_1$, which in this case is dense in $Y$ from Corollary~\ref{cor:cinftyhk}.
\end{remark}

We will henceforth drop the assumption of differentiability of the actions and assume the existence of the dense subset $Y_1$, such that \eqref{eq:dy} is a defined on $Y_1$ as in Remark~\ref{re:weakreg}. More precisely, for each $y\in Y_1$, there is a subspace $\mathcal D_y$ of $T_yY$, defined by \eqref{eq:dy}. In fact, we always keep in mind the above example of the action $\rho$, for which all this theory is developed. Let us extend some definitions of Section~\ref{sec:infinitedimmnflds} to this context.

\begin{definition}\label{def:transversed}
A submanifold $S\subset Y$ is said to be \emph{transverse to $\mathcal D$ at $y\in S\cap Y_1$} if $$T_yY=T_yS\oplus\mathcal D_y.$$ The submanifold $S$ is {\em transverse to $\mathcal D$} if it is transverse to $\mathcal D$ at every $y\in S\cap Y_1$.
\end{definition}

\begin{remark}
Notice that $\mathcal D$ is {\em integrable}, since orbits of the action are its integral submanifolds. More precisely, since $Y_1$ is $G$--invariant, if $y\in Y_1$, the orbit $G(y)$ is contained in $Y_1$. The subspace $\mathcal D_y$ of $T_yY_1$ then coincides with the tangent space to $G(y)$ considered as a submanifold of $Y_1$. From this viewpoint, the above definition of transversality of $S$ to $\mathcal D$ {\em does not} coincides with the definition of transversality of $S$ to $G(y)\subset Y_1$ for all $y\in Y_1$, see Remark~\ref{re:abouttransv}. Indeed, the condition above is {\em stronger} than transversality of $S$ to $G(y)$ for all $y\in Y_1$, since transversality does not require the intersection to be discrete.
\end{remark}

We end this section with an interesting result on continuous actions $\mu$ that have in addition further regularity of $\mu_y$ for some $y\in Y_1$. It indirectly uses stability of transversality and degree theory to obtain an open neighborhood of $y$ by considering the image under the group action of a submanifold transverse to the orbits at $y$. As stated in Remark~\ref{re:weakreg}, such regularity hypotheses are satisfied in the case of $\rho:S^1\times H^1(S^1,M)\to H^1(S^1,M)$ with $Y_1=H^2(S^1,M)$, since for one such $\gamma\in Y_1$, the map $\rho_\gamma$ is $C^1$.

\begin{proposition}\label{prop:tauskehocara}
Let $\mu:G\times Y\to Y$ be a continuous action of a one--dimensional Lie group $G$ and suppose there exists $y\in Y_1$ such that $\mu_y:G\to Y$ is of class $C^1$. If a submanifold $S$ of $Y$ is such that $y\in S$ and $T_yY=T_yS\oplus\mathcal D_y$, then $\mu(G\times S)$ is a neighborhood of $y\in Y$.
\end{proposition}

\begin{proof}
Since $G$ is one--dimensional, $\codim_Y S=1$. In addition, from the hypotheses on $\mu$ and $y\in S$, it follows that Proposition~\ref{prop:dovalorinterm} applies, with $f$ being the action $\mu$, the Banach manifold $X$ being $Y$, $x_0=y\in Y$, and the topological space $A$ being $G$, $a_0=e\in G$ the identity. Notice that $$f(a_0,x_0)=\mu(e,y)=y=x_0$$ and $$\im\frac{\partial f}{\partial x}(a_0,x_0)=\im\frac{\partial\mu}{\partial y}(e,y)=\im\dd\mu_y(e)=\mathcal D_y.$$ Thus, from Proposition~\ref{prop:dovalorinterm} there exists an open neighborhood $U$ of $e$ in $G$ such that for all $g\in U$, we have $S\cap\im\mu_g\neq\emptyset$. This means that if $g\in U$, then $g\in\mu(G\times S)$, and hence $\mu(G\times S)$ is a neighborhood of $y\in Y$.
\end{proof}

\begin{remark}
Proposition~\ref{prop:tauskehocara} obviously holds for any finite--dimensional Lie group $G$. The hypothesis that $\dim G=1$ was only used to obtain $\codim_Y S=1$, since $S$ is transverse to the orbits, and then apply Proposition~\ref{prop:dovalorinterm}. Nevertheless, as stated in Remark~\ref{re:codimalta}, by using a topological degree argument, this hypothesis may be replaced by $\codim_Y S=n<+\infty$. Hence, applying this more general version of Proposition~\ref{prop:dovalorinterm} we obtain the same result above for any finite--dimensional Lie groups. Nevertheless, the simpler version of this result given above is already enough for our applications, that will be concerned with $G=S^1$ and its action on $H^1(S^1,M)$.
\end{remark}

\chapter{Geodesic variational problems}
\label{chap35}

In this chapter, we are interested in a classic problem of geometric calculus of variations. In general, problems of geometric calculus of variations are in the interface of nonlinear analysis and differential geometry, studying variational problems that arise in a geometric context. Let us give a brief introduction to the subject, inspired mostly by Jost \cite{jostcv}.

The oldest and most famous geometric variational problem is the geodesic problem. If $(M,g_\mathrm R)$ is a Riemannian manifold and $\gamma:[0,1]\to M$ is a Sobolev $H^1$ curve, we may consider its {\em $g_\mathrm R$--length} and its {\em $g_\mathrm R$--energy}, respectively given by
\begin{equation}\label{eq:grenergylength}
L_\mathrm R(\gamma)=\int_0^1\sqrt {g_{\mathrm R}(\dot\gamma,\dot\gamma)}\;\dd t\;\; \mbox{ and } \;\; E_{\mathrm R}(\gamma)=\tfrac12\int_0^1 g_\mathrm R(\dot\gamma,\dot\gamma)\;\dd t.
\end{equation}
Recall that in Section~\ref{sec:hum}, the set $H^1([0,1],M)$ was endowed with a separable Hilbert manifold structure, see Theorem~\ref{thm:atlasparaH1}. Notice that since $\gamma\in H^1([0,1],M)$, its tangent field is regarded as $\dot\gamma\in\sect^{L^2}(\gamma^*TM)$, and this derivative is only almost everywhere defined, see Proposition~\ref{prop:abscontcharact}. Nevertheless, the integrals above are perfectly well--defined. In fact, Sobolev class $H^1$ is the minimal regularity assumption needed to have enough analytical tools to study the above two functionals.

Let us discuss some relations between these functionals and the respective variational problems. Using the Cauchy--Schwartz inequality for the $L^2$--inner product, see \eqref{eq:productldois}, it is easy to conclude that
\begin{equation*}
L^2_\mathrm R(\gamma)\le 2E_{\mathrm R}(\gamma),
\end{equation*}
with the equality holding if and only if
\begin{equation}\label{eq:gammageod}
g_\mathrm R(\dot\gamma,\dot\gamma)=\mbox{const.}
\end{equation}
Critical points of $E_{\mathrm R}$ are $g_\mathrm R$--geodesics,\footnote{Notice that there is an implicit regularity result, since the functional $E_{\mathrm R}$ is defined for Sobolev $H^1$ curves, and a $g_\mathrm R$--geodesic is a $C^2$ curve. In Proposition~\ref{prop:critgenenfunc}, we prove that if $\gamma\in H^1([0,1],M)$ is a critical point of $E_{\mathrm R}$, then $\gamma\in C^2([0,1],M)$. In fact, if $g_\mathrm R$ is of class $C^k$, it follows from Corollary~\ref{cor:geodck} that $\gamma$ is of class $C^{k+1}$.} in the sense of Definition~\ref{def:geodesic}, see Proposition~\ref{prop:critgenenfunc}. It is easy to see that a critical point of $E_{\mathrm R}$ is a critical point of $L_\mathrm R$ if and only if \eqref{eq:gammageod} holds, and vice versa. Hence critical points of these functionals are geometrically the same.

An important observation however, is that critical points of $E_{\mathrm R}$ are affinely parameterized curves, while this is not necessarily true for $L_\mathrm R$. Recall that by {\em geodesic} we mean an affinely parameterized curve that satisfies the geodesic equation, see Definition~\ref{def:geodesic}.

Our goal is to study geodesic variational problems for {\em semi--Riemannian} geodesics, hence the natural (actually {\em compulsory}) option is considering the energy functional instead of the lenght functional. More precisely, we would to consider the functionals \eqref{eq:grenergylength} replacing $g_\mathrm R$ with a semi--Riemannian metric $g\in\met_\nu^k(M)$. Nevertheless, this can only be done in the energy functional, since the integrand of the length functional is not even well--defined if $g$ is not positive--definite. Thus, a convenient setting for the $g$--geodesic variational problem in our case is to find extrema of the $g$--energy functional defined on Sobolev $H^1$ curves on $M$, $$E_g:H^1([0,1],M)\owns\gamma\longmapsto\tfrac{1}{2}\int_0^1 g(\dot{\gamma},\dot{\gamma}) \;\dd t\in\R.$$ Moreover, since we will be interested in analyzing the set of metrics for which the geodesic variational problem has only (strongly) nondegenerate minimizers, we consider a family of such geodesic variational problems, parameterized by semi--Riemannian metrics.

The adequate abstract structure for the space of parameters in this case is that of a Banach manifold. Consider $\mathcal A_{g_\mathrm A,\nu}$ as in Proposition~\ref{prop:affineworks}. Recall that this is an open subset of an affine Banach space formed by semi--Riemannian metrics of index $\nu$ that are asymptotically equal to an auxiliary metric $g_\mathrm A$ at infinity, see Section~\ref{sec:banachspacetensors}.
We may then define the objects of this parameterized family of geodesic variational problems as follows. Consider the {\em generalized energy functional}
\begin{equation*}
E:\mathcal A_{g_\mathrm A,\nu}\times H^1([0,1],M)\owns (g,\gamma)\longmapsto E_g(\gamma)=\tfrac{1}{2}\int_0^1 g(\dot{\gamma},\dot{\gamma}) \;\dd t\in\R.
\end{equation*}
The first variable of this functional should be thought of as a parameter $g\in\mathcal A_{g_\mathrm A,\nu}$, while the second variable $\gamma\in H^1([0,1],M)$ is the {\em real} variable of which we are interested in finding extrema. In this sense, we will frequently use the notation $E_g:H^1([0,1],M)\to\R$ for the restricted functional $E(g,\cdot\,)$. We will later give a formal definition of this functional and study its regularity, see Definition~\ref{def:energyfunc} and Proposition~\ref{prop:fck}.

Let us briefly remark that the $g$--geodesic variational problem is part of a wide class of variational problems in classical mechanics, namely {\em Hamiltonian variational problems}. In this sense, one can regard geodesics as Hamiltonian flows, since these are solutions of the associated Hamilton-Jacobi equation. In fact, consider the geodesic Hamiltonian on $M$ defined by
\begin{eqnarray*}
H_g: TM^*& \la &\R\\
(x,p)&\longmapsto &\tfrac12g(x)^{-1}(p,p).
\end{eqnarray*}
Since $g$ is a semi--Riemannian metric on $M$, at each $x\in M$ we may use \eqref{ident:bilin} to consider $g(x):T_xM\to T_xM^*$, and its inverse $g(x)^{-1}:T_xM^*\to T_xM$. Thus $g(x)^{-1}$ gives an inner product in the dual space $T_xM^*$, which is used to give the correct formulation of the geodesic Hamiltonian as above. Notice that this is the well--known kinetic Hamiltonian for a particle with unitary mass, where $p$ represents its momentum.

The Hamilton--Jacobi equation for $H_g$ coincides with the geodesic equation \eqref{eq:geodequation} mentioned in Remark~\ref{re:geodequation}. By using this approach, it is possible to intepret geodesics as the trajectories described by particles that are not experiencing any forces. Compare these concepts for instance in $\R^m$ endowed with the Euclidean metric. On the one hand, geodesics in Euclidean space are straight lines. On the other hand, Newton's First Law asserts that a particle moving in a straight line will continue to move in a straight line if it experiences no external forces. The reason for the {\em straight} motion of this particle in $\R^m$ is conservation of momentum, which in the presence of curvature is described in terms of the metric $g$.

In this sense, the geodesic flow of a metric $g\in\met_\nu^k(M)$ is a Hamiltonian flow, see Definition~\ref{def:geodflow}. There are several important properties of this particular flow that distinguishes it from general Hamiltonian flows. For instance, consider the geodesic flow's {\em energy levels} 
\begin{equation}
H_g^{-1}(\zeta)=\{(x,p)\in TM^*: H_g(x,p)=\zeta\}.
\end{equation}
These form a partition of the cotangent bundle $TM^*$, which is {\em well--behaved} in several ways. Some of its properties are no longer valid for more general Hamiltonian flows, and this causes genericity results of nondegeneracy similar to the ones proved in this text to fail for more general classes of Hamiltonian flows, see Chapter~\ref{chap7}. Several important Hamiltonian aspects of the geodesic flow are well studied in the literature, regarding dynamical concepts for instance as being Anosov, or having positive topological entropy. Great contributions in this area were given by several authors, from which we highlight Contreras-Barandiar\'an, Ma\~n\'e and Paternain. For a thorough study of the geodesic flow from this viewpoint we refer to Paternain \cite{paternain}.

\section{GECs}
\label{sec:gecs}

As discussed above, Sobolev $H^1$ is a convenient regularity to develop the basic arguments of calculus of variations in the case of the geodesic variational problem. However, if the domain of the $g$--energy functional was the entire $H^1([0,1],M)$, extrema would trivially be constant curves. Thus we must require further {\em endpoints conditions} on curves, which corresponds to restricting the $g$--energy functional to submanifolds of $H^1([0,1],M)$. The main goal of this section is to establish the most general setting for endpoints condition on curves $\gamma\in H^1([0,1],M)$, analyze the structure of the correspondent submanifolds and discuss a few examples.

Let us start with a simple example in which curves have fixed endpoints.

\begin{definition}
A {\em fixed endpoints}\index{Fixed endpoints condition}\index{GEC!fixed endpoints} condition on $M$ is a fixed pair of points $(p,q)\in M\times M$. The correspondent restraint on a curve $\gamma\in H^1([0,1],M)$ is $\gamma(0)=p$ and $\gamma(1)=q$, see Figure \ref{fig:omegapq}. In this context, the domain of curves that satisfy such endpoints condition is
\begin{equation}\label{eq:opq}
\Omega_{p,q}(M)=\{\gamma\in H^1([0,1],M):\gamma(0)=p,\gamma(1)=q\}.
\end{equation}
\end{definition}

\begin{figure}[htf]
\begin{center}
\vspace{-0.5cm}
\includegraphics[scale=1]{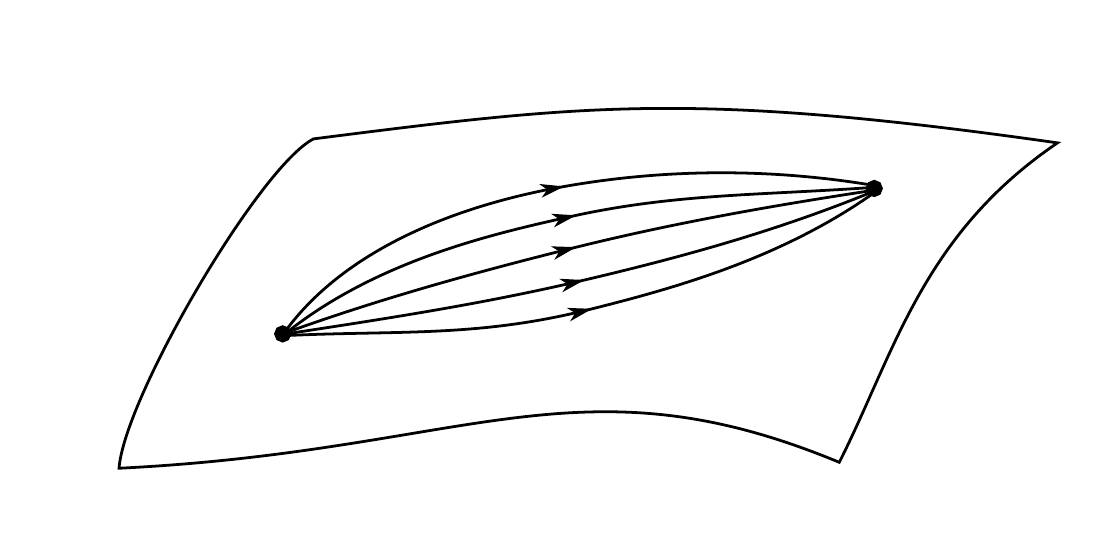}
\begin{pgfpicture}
\pgfputat{\pgfxy(-9,2.2)}{\pgfbox[center,center]{$p$}}
\pgfputat{\pgfxy(-2,3.8)}{\pgfbox[center,center]{$q$}}
\pgfputat{\pgfxy(-1.8,2)}{\pgfbox[center,center]{$M$}}
\end{pgfpicture}
\vspace{-0.75cm}
\caption{Some curves $\gamma\in\Omega_{p,q}(M)$.}\label{fig:omegapq}
\end{center}
\end{figure}

\begin{lemma}\label{le:opqsmnfld}
The subset $\Omega_{p,q}(M)$ is a (smooth) separable submanifold of $H^1([0,1],M)$, whose tangent space at $\gamma$ is given by
\begin{equation}\label{eq:tgammaopq}
T_\gamma\Omega_{p,q}(M)=\{v\in\sect^{H^1}(\gamma^*TM):v(0)=0, v(1)=0\}.
\end{equation}
\end{lemma}

\begin{proof}
Consider the endpoints map \eqref{eq:ev01},
\begin{eqnarray*}
\ev_{01}=(\ev_0,\ev_1):H^1([0,1],M) &\la& M\times M \\
\gamma &\longmapsto &(\gamma(0),\gamma(1)).
\end{eqnarray*}
From Proposition~\ref{prop:endpointsmap}, this is a smooth submersion. In particular, $(p,q)\in M\times M$ is a regular value, see Definitions~\ref{def:regular} and~\ref{def:submersion}. From Proposition~\ref{prop:regularvalue} and Remark~\ref{re:regularvaluehilbert}, it follows that $\Omega_{p,q}(M)$ is a (smooth) Hilbert submanifold of $H^1([0,1],M)$.

It also follows from Proposition~\ref{prop:regularvalue} that the tangent space to $\Omega_{p,q}(M)$ at $\gamma$ is given by the complemented subspace $\ker\dd(\ev_{01})(\gamma)$. Moreover, from Remarks~\ref{re:submnflds} and~\ref{re:TH1H1T}, $T_\gamma\Omega_{p,q}(M)$ is a Hilbert subspace of $\sect^{H^1}(\gamma^*TM)$. From \eqref{eq:dev01} it is clear that $v\in\ker\dd(\ev_{01})(\gamma)$ if and only if $v(0)=0$ and $v(1)=0$. Therefore formula \eqref{eq:tgammaopq} holds.

Finally, regarding separability of $\Omega_{p,q}(M)$, Corollary~\ref{cor:h1separable} guarantees that $H^1([0,1],M)$ is separable. Since these are metric spaces (see Remarks~\ref{re:infinitedimmetricspace} and~\ref{re:metricinfinitesmfld} and Proposition~\ref{prop:RiemannH1}), separability is equivalent to second--countability, which is a hereditary property. Therefore, the submanifold $\Omega_{p,q}(M)$ is separable.
\end{proof}

\begin{corollary}\label{cor:opqsmnfld}
The Hilbert manifold of curves $\Omega_{p,q}(M)$ can be endowed with the Riemannian metric
\begin{equation}\label{eq:metricopqm}
\llangle v,w\rrangle =\int_0^1 g_\mathrm{R}(\D^\mathrm Rv,\D^\mathrm Rw)\;\dd t, \quad v,w\in T_\gamma\Omega_{p,q}(M)
\end{equation}
where $\D^\mathrm R:\sect^{H^1}(\gamma^*TM)\to\sect^{L^2}(\gamma^*TM)$ is the covariant derivative operator along $\gamma$ induced by the fixed Riemannian metric $g_\mathrm R$.
\end{corollary}

\begin{proof}
This is an immediate consequence of Lemma~\ref{le:opqsmnfld}, Proposition~\ref{prop:RiemannH1} and Remark~\ref{re:metricinfinitesmfld}. Notice that from Proposition~\ref{prop:RiemannH1}, the Riemannian metric on $\Omega_{p,q}(M)$ is given by the restriction of \eqref{eq:RiemannH1} to $T_\gamma\Omega_{p,q}(M)$. Hence, this restricted metric is given by formula \eqref{eq:metricopqm}, once the term $g_\mathrm R(v(0),w(0))$ obviously vanishes since for all $v\in T_\gamma\Omega_{p,q}(M)$, from \eqref{eq:tgammaopq}, $v(0)=0$ and $v(1)=0$.
\end{proof}

The above results guarantee that $\Omega_{p,q}(M)$ is a sufficiently regular domain for developing calculus of variations. Thus, it would be possible to continue and study extrema of the functional $E_g:\Omega_{p,q}(M)\to\R$, which are be geodesics joining $p$ and $q$. Nevertheless, we would like to consider {\em more general} endpoints conditions.

Several attempts to generalize this fixed endpoints condition are possible, for instance instead of fixing two points $p,q\in M$, fix two submanifolds $P,Q\subset M$, and allow $\gamma(0)\in P$ and $\gamma(1)\in Q$, as in Figure \ref{fig:omegapzaoqzao}. To our knowledge, the most comprehensive generalization is considering a {\em submanifold} $\p\subset M\times M$. The correspondent endpoints condition for curves $\gamma$ is given by $(\gamma(0),\gamma(1))\in\p.$ This makes arbitrary choices of endpoints conditions possible.
 
%

\begin{figure}[htf]
\begin{center}
\vspace{-0.75cm}
\includegraphics[scale=1]{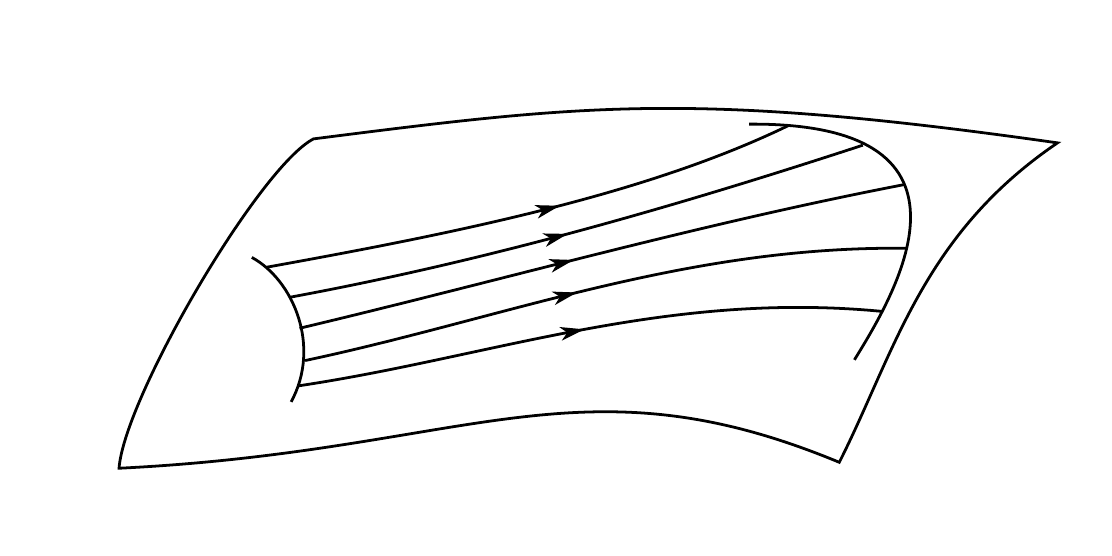}
\begin{pgfpicture}
\pgfputat{\pgfxy(-8.8,2.2)}{\pgfbox[center,center]{$P$}}
\pgfputat{\pgfxy(-2,3.8)}{\pgfbox[center,center]{$Q$}}
\pgfputat{\pgfxy(-1.8,2)}{\pgfbox[center,center]{$M$}}
\end{pgfpicture}
\vspace{-0.5cm}
\caption{Some curves joining the submanifolds $P$ and $Q$.}\label{fig:omegapzaoqzao}
\end{center}
\end{figure}

\begin{definition}\label{def:gec}
A \emph{general endpoints condition}\index{GEC} on $M$ (or simply {\em GEC}) is a submanifold $\p\subset M\times M$.
\end{definition}


The subset of curves that satisfy a GEC $\p$ will be denoted
\begin{equation}\label{opm}
\op(M)=\{\gamma\in H^1([0,1],M):(\gamma(0),\gamma(1))\in\p\}.
\end{equation}
We now develop a result totally analogous to Lemma~\ref{le:opqsmnfld} and Corollary~\ref{cor:opqsmnfld}, replacing \eqref{eq:opq} with \eqref{opm}. The proof of this result will be given in more details then Lemma~\ref{le:opqsmnfld}, and is obviously an extension of such result to GECs.

\begin{proposition}\label{prop:opmsubmfld}
The subset $\op(M)$ is a separable Hilbert submanifold of $H^1([0,1],M)$. Moreover, the tangent space to $\op(M)$ at $\gamma$ is given by
\begin{equation}\label{eq:tgammaopm}
T_\gamma\op(M)=\{v\in\sect^{H^1}(\gamma^*TM):(v(0),v(1))\in T_{(\gamma(0),\gamma(1))}\p\},
\end{equation}
see Figure \ref{fig:tomegap}, and can be endowed with the inner product induced from \eqref{eq:RiemannH1},
\begin{equation}\label{eq:riemhilbop}
\llangle v,w\rrangle =g_\mathrm R(v(0),w(0))+\int_0^1 g_\mathrm{R}(\D^{\mathrm{R}}v,\D^{\mathrm{R}}w)\;\dd t.
\end{equation}
\end{proposition}

\begin{proof}
Consider again the endpoints map \eqref{eq:ev01},
\begin{eqnarray*}
\ev_{01}=(\ev_0,\ev_1):H^1([0,1],M) &\la& M\times M\\
\gamma &\longmapsto& (\gamma(0),\gamma(1)).
\end{eqnarray*}
From Proposition~\ref{prop:endpointsmap}, this is a smooth submersion. In particular, $\ev_{01}$ is transverse to $\p$, see Remark~\ref{re:transversalityeq} and Definition~\ref{def:transversality}. From Proposition~\ref{prop:transvsubmnfld}, $$\op(M)=\ev_{01}^{-1}(\p)$$ is a (smooth) submanifold of $H^1([0,1],M)$.

It also follows from Proposition~\ref{prop:transvsubmnfld} that the tangent space to $\op(M)$ at $\gamma$ is the Hilbertable subspace of $T_\gamma H^1([0,1],M)$ given by
\begin{equation}\label{eq:tgammaopm1}
T_\gamma\op(M)=\dd(\ev_{01})(\gamma)^{-1}\big[T_{(\gamma(0),\gamma(1))}\p\big].
\end{equation}
Moreover, from Remarks~\ref{re:submnflds} and~\ref{re:TH1H1T}, $T_\gamma\op(M)$ is a Hilbert subspace of $\sect^{H^1}(\gamma^*TM)$, and formula \eqref{eq:tgammaopm} is an immediate consequence of \eqref{eq:dev01} and \eqref{eq:tgammaopm1}.

Regarding separability of $\op(M)$, Corollary~\ref{cor:h1separable} guarantees that the ambient manifold $H^1([0,1],M)$ is separable. Since these are metrizable spaces (see Remarks~\ref{re:infinitedimmetricspace} and~\ref{re:metricinfinitesmfld} and Proposition~\ref{prop:RiemannH1}), separability is equivalent to second--countability, which is a hereditary property. Therefore, the submanifold $\op(M)$ is separable.

Finally, the Riemannian metric \eqref{eq:RiemannH1} on $H^1([0,1],M)$ that was described in Proposition~\ref{prop:RiemannH1} can be restricted to $\op(M)$, see Remarks~\ref{re:metricinfinitesmfld} and~\ref{re:covderh1}, resulting in formula \eqref{eq:riemhilbop} at each $\gamma\in\op(M)$, and this concludes the proof.
\end{proof}

\begin{figure}[htf]
\begin{center}
\vspace{-0.75cm}
\includegraphics[scale=1]{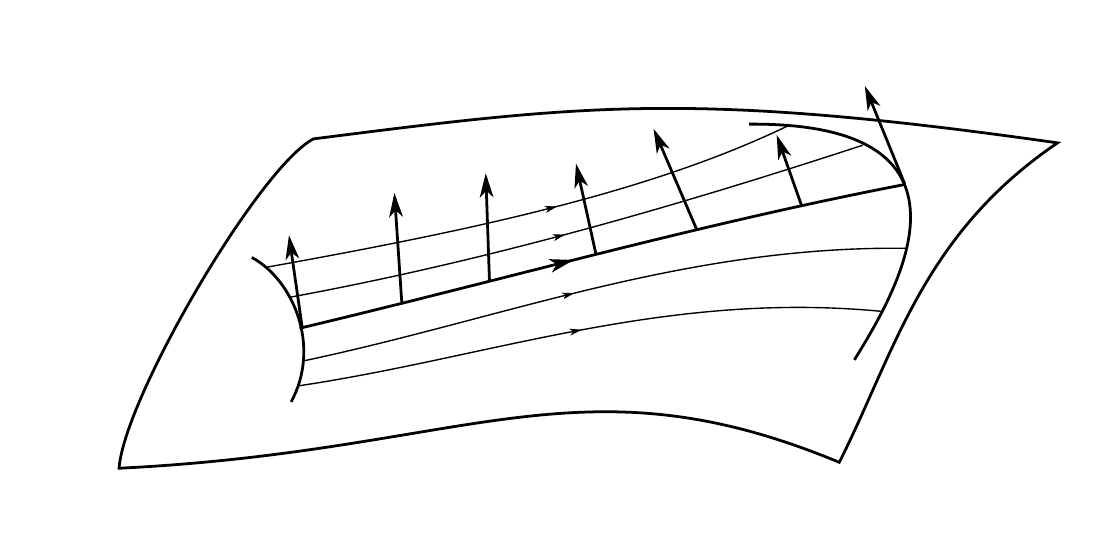}
\begin{pgfpicture}
\pgfputat{\pgfxy(-8.8,2.2)}{\pgfbox[center,center]{$P$}}
\pgfputat{\pgfxy(-1.9,3.6)}{\pgfbox[center,center]{$Q$}}
\pgfputat{\pgfxy(-1.8,2)}{\pgfbox[center,center]{$M$}}
\pgfputat{\pgfxy(-5,2.9)}{\pgfbox[center,center]{$\gamma$}}
\pgfputat{\pgfxy(-3.8,4)}{\pgfbox[center,center]{$v$}}
\end{pgfpicture}
\vspace{-0.5cm}
\caption{A vector $v\in T_\gamma\op(M)$ represented as a vector field along $\gamma$.}\label{fig:tomegap}
\end{center}
\end{figure}

\begin{example}\label{ex:gecs}
The fixed endpoints condition $\p=\{p\}\times\{q\}$ is a GEC. Notice that setting $\p=\{p\}\times\{q\}$, Proposition~\ref{prop:opmsubmfld} coincides with Lemma~\ref{le:opqsmnfld}. This fixed endoints condition is illustrated in Figure~\ref{fig:omegapq}. In particular, \eqref{eq:tgammaopm} coincides with \eqref{eq:tgammaopq} for such $\p$. In other words, as expected, the tangent space $T_\gamma\op(M)$ is formed by Sobolev class $H^1$ sections $v$ of $\gamma^*TM$ such that $v(0)=0$ and $v(1)=0$, as described in Lemma~\ref{le:opqsmnfld} by formula \eqref{eq:tgammaopq}.

\begin{figure}[htf]
\begin{center}
\vspace{-0.75cm}
\includegraphics[scale=1]{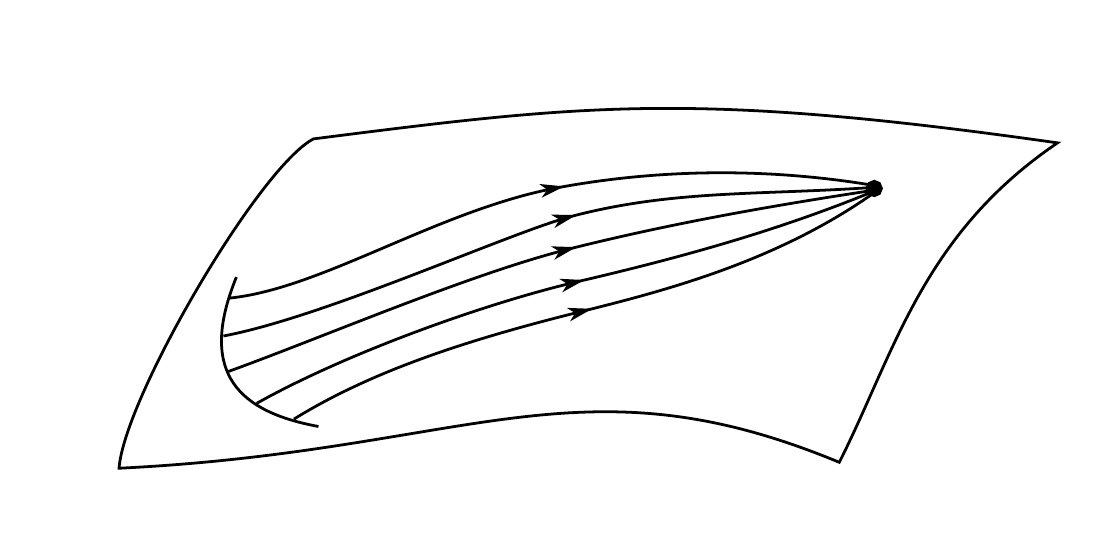}
\begin{pgfpicture}
\pgfputat{\pgfxy(-9.5,1.7)}{\pgfbox[center,center]{$P$}}
\pgfputat{\pgfxy(-2.1,3.6)}{\pgfbox[center,center]{$q$}}
\pgfputat{\pgfxy(-1.8,2)}{\pgfbox[center,center]{$M$}}
\end{pgfpicture}
\vspace{-0.5cm}
\end{center}
\end{figure}

Another interesting example of GEC is $\p=P\times Q$, where $P$ and $Q$ are submanifolds of $M$, as illustrated in Figure~\ref{fig:omegapzaoqzao}. The curves $\gamma\in\op(M)$ satisfy $\gamma(0)\in P$ and $\gamma(1)\in Q$, and the condition on the sections $v$ of $\gamma^*TM$ that form the tangent space to $\op(M)$ at $\gamma$ is, as expected, $v(0)\in T_{\gamma(0)}P$ and $v(1)\in T_{\gamma(1)}Q$. This follows at once from Proposition~\ref{prop:opmsubmfld} by formula \eqref{eq:tgammaopm}. Notice also that we could also consider $\p=P\times\{q\}$, replacing the submanifold $Q$ with a point $Q=\{q\}$, as illustrated above. Analogous results on the endpoints conditions for curves and tangent spaces are easily verified.

As a last example of GEC, consider the case of {\em periodic} curves on $M$, given by the diagonal\footnote{Here $\Delta\subset M\times M$ is the diagonal of the product manifold $M\times M$, however in the sequel we will be somewhat sloppy about the use of the symbol $\Delta$. It will denote the diagonal not only of $M\times M$, but also of any product space, for instance $\Delta$'s own tangent space, which is the diagonal $\Delta\subset T_xM\oplus T_xM$. There is no ambiguity, since it will always be clear from the context which diagonal is being considered.} $$\Delta=\{(p,p):p\in M\}$$ Curves $\gamma\in\Omega_\Delta(M)$ satisfy $\gamma(0)=\gamma(1)$, and the condition on the sections $v$ of $\gamma^*TM$ that form the tangent space to $\Omega_\Delta(M)$ at $\gamma$ is $v(0)=v(1)$. Recall that we had already proved that $\Omega_\Delta(M)$ is a submanifold, and identified it with $H^1(S^1,M)$, see Corollary~\ref{cor:h1s1}.
\end{example}

\begin{remark}\label{re:transpose}
Note that the \emph{transpose}\index{GEC!transpose} of a GEC $\p$, defined by \begin{equation}\label{eq:transpose}\p^t=\{(p,q)\in M\times M:(q,p)\in\p\},\end{equation} is also a GEC, and the manifolds $\op(M)$ and $\Omega_{\p^t}(M)$ can be canonically identified using the diffeomorphism given by backwards reparametrization of curves, see Figure~\ref{fig:ptranspose}. Hence solutions of the geodesic variational problems with endpoints conditions $\p$ and $\p^t$ are also obviously identified. Due to such symmetry, every result stated for some GEC $\p$ is also automatically valid for its transpose $\p^t$.
\end{remark}

\begin{figure}[htf]
\begin{center}
\vspace{-0.3cm}
\includegraphics[scale=1]{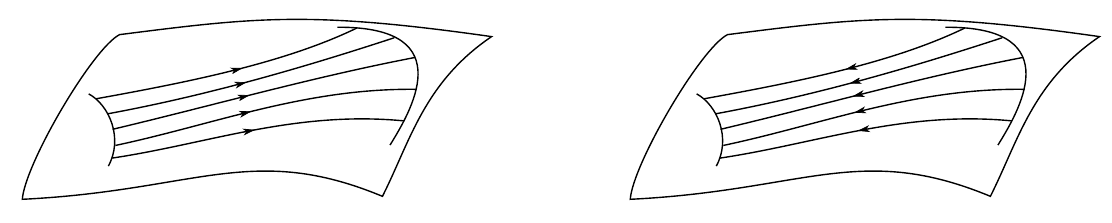}
\begin{pgfpicture}
\pgfputat{\pgfxy(-9,0)}{\pgfbox[center,center]{$\p$}}
\pgfputat{\pgfxy(-2.5,0)}{\pgfbox[center,center]{$\p^t$}}
\end{pgfpicture}
\caption{Curves satisfying a GEC $\p$ and its transpose $\p^t$.}\label{fig:ptranspose}
\end{center}
\end{figure}

\section{Generalized energy functional}
\label{sec:genenfunc}

In this section, we study a generalized energy functional for the geodesic variational problem under general endpoints conditions. This energy functional has a parameter $g$, which is the metric used to compute the energy of Sobolev $H^1$ curves on $M$ that satisfy a GEC.

\begin{definition}\label{def:energyfunc}\index{Energy!generalized functional}\index{Generalized energy functional}
Consider $\mathds E$ a separable $C^k$ Whitney type Banach space of sections of $E=TM^*\vee TM^*$ that tend to zero at infinity, $g_\mathrm A\in\met_\nu^k(M)$ a metric that satisfies \eqref{eq:boundedfromzero} and $\A_{g_\mathrm A,\nu}=(g_\mathrm A+\mathds E)\cap\met_\nu^k(M).$ Furthermore, let $\p$ be a GEC on $M$ and consider the separable Hilbert manifold $\op(M)$. The {\em generalized energy functional} for $M$ is defined by
\begin{equation}\label{eq:efunct}
E:\A_{g_\mathrm A,\nu}\times\op(M)\owns (g,\gamma)\longmapsto E_g(\gamma)=\tfrac{1}{2}\int_0^1 g(\dot{\gamma},\dot{\gamma}) \;\dd t\in\R.
\end{equation} 
\end{definition}

\begin{remark}
From Propositions~\ref{prop:affineworks} and~\ref{prop:opmsubmfld}, the domain $\A_{g_\mathrm A,\nu,\p}\times\op(M)$ is an open subset of the product $(g_\mathrm A+\mathds E)\times \op(M)$. 
\end{remark}

\begin{proposition}\label{prop:fck}
The generalized energy functional $E$ given by \eqref{eq:efunct} is of class $C^k$. More precisely, it is smooth with respect to the first variable $g\in\A_{g_\mathrm A,\nu}$ and $C^k$ with respect to the second variable\footnote{This $C^k$ regularity clearly comes from the regularity $C^k$ chosen for the metrics $g\in\A_{g_\mathrm A,\nu}$. For our applications to be possible, we will henceforth implicitly suppose $k\geq3$.} $\gamma\in\op(M)$.
\end{proposition}

\begin{proof}
There are essentially two ways of proving the desired regularity of \eqref{eq:efunct}. We will briefly comment on the first approach and then sketch parts of the proof using the second approach.

The first idea is to use the local charts of $\A_{g_\mathrm A,\nu}\times\op(M)$ given by $$T\times\curves\varphi:\A_{g_\mathrm A,\nu}\times(\curves{U}\cap\op(M))\la\mathds E\times H^1([0,1],\R^m),$$ where $T:\sect_b^k(E)\to\sect_b^k(E)$ is the translation of $-g_\mathrm A$, a global chart for $\A_{g_\mathrm A,\nu}$ that maps $g_\mathrm A+\mathds E$ to $\mathds E$ and $\curves\varphi$ is a submanifold chart of $\op(M)$. Representing $E$ in such charts, one obtains a (fairly complicated) local expression for $E$ defined in an open subset of the product $\mathds E\times H^1([0,1],\R^m)$.

Determining the regularity of $E$ is now reduced to determining the regularity of a map defined in an open subset of a Banach space, in the sense of Definitions~\ref{def:diffBanach} and~\ref{def:ckbanach}. Such verification involves several preliminary lemmas to guarantee the adequate regularity of auxiliary maps such as left composition with certain vector bundle morphisms. Given the high technicality of the involved computations, we will not follow this approach. The interested reader may find the basic tools necessary in Palais \cite{palais} for the case in which $M$ is compact, and in Piccione and Tausk \cite{pictau} for the noncompact case.

A second approach is the following. Denote $X=g_\mathrm A+\mathds E$, $Y=\op(M)$ and $\mathcal U=\A_{g_\mathrm A,\nu}\times\op(M)$, which is clearly an open subset of the product $X\times Y$. The functional $E:\mathcal U\to\R$ is linear in the first variable. This means that for each fixed $\gamma_0$,
\begin{equation}\label{eq:fy}
E(\,\cdot\,,\gamma_0):\A_{g_\mathrm A,\nu}\la\R
\end{equation}
is linear, hence smooth. In addition, standard arguments prove that for each fixed $g_0$, 
\begin{equation}\label{eq:fx}
E_{g_0}=E(g_0,\,\cdot\,):\op(M)\la\R
\end{equation}
is of class $C^k$, see for instance \cite{jost,petersen}. Moreover, the derivatives of \eqref{eq:fy} and \eqref{eq:fx} at any $(g_0,\gamma_0)\in\mathcal U$, i.e., the partial derivatives of $E$, are respectively\footnote{Formula \eqref{eq:dfdgamma} will be later justified, see \eqref{eq:dfdgammaa}.}
\begin{eqnarray}
&\label{eq:dfdg}\displaystyle\frac{\partial E}{\partial g}(g_0,\gamma_0)h =\frac{1}{2}\int_0^1 h(\dot{\gamma_0},\dot{\gamma_0})\;\dd t, \quad h\in\mathds E& \\
&&\nonumber\\
&\label{eq:dfdgamma}\displaystyle\frac{\partial E}{\partial\gamma}(g_0,\gamma_0)v=\int_0^1 g_0(\dot{\gamma_0},\D^{g_0}v)\;\dd t, \quad v\in T_\gamma\op(M).&
\end{eqnarray}
These clearly induce continuous maps $$\frac{\partial E}{\partial g}:\mathcal U\la X^* \;\;\mbox{ and }\;\; \frac{\partial E}{\partial\gamma}:\mathcal U\la Y^*.$$ Therefore $$\dd E:\mathcal U\ni (g,\gamma)\longmapsto\left(\frac{\partial E}{\partial g}(g,\gamma),\frac{\partial E}{\partial\gamma}(g,\gamma)\right)\in X^*\times Y^*$$ is also continuous, and hence $E$ is of class $C^1$.

In order to prove that $E$ is of class $C^k$, the same standard argument above applies. Namely, if each partial derivative of order $r$ exists and is continuous as a {\em map of two variables, $g$ and $\gamma$},\footnote{Notice that each first partial derivative is clearly continuous as a function of the respective variable. However, to infer continuity of $\dd E$ it is necessary to verify continuity with respect to {\em both} variables. In the above case, this is a simple calculation, however the verification of continuity of higher order derivatives may imply greater computation efforts.} then $E$ is of class $C^r$. Being $E$ linear on the first variable $g$, it suffices to prove the above statement for derivatives with respect to $\gamma$.

Observe that the first derivative $\frac{\partial E}{\partial\gamma}(g_0,\gamma_0)$ computed above involves the covariant derivative $\D^{g_0}$ induced by the Levi--Civita connection $\nabla^{g_0}$ of $g_0$, hence the Christoffel tensors of $g_0$, which are computed in terms of the first derivatives of the metric coefficients, see \eqref{eq:christoffeltensor}. The second derivative $\frac{\partial^2 E}{\partial\gamma^2}(g_0,\gamma_0)$ involves the curvature tensor $R^{g_0}$ of $\nabla^{g_0}$, i.e., the second derivative of $g$, see \eqref{eq:Rg} and \eqref{eq:koszul}. Higher order derivatives of $E$ with respect to $\gamma$ at $(g_0,\gamma_0)$ are computed in terms of higher order covariant derivatives of $R^{g_0}$.

Using this standard setting, the reader may verify that $E$ is indeed $C^k$, having the same regularity as the metric tensors in the translated $C^k$ Whitney type Banach space $\mathds E$.
\end{proof}

Henceforth, assume the domain of the generalized energy functional \eqref{eq:efunct} to be the open subset
\begin{equation}
\mathcal U=\A_{g_\mathrm A,\nu}\times\op(M),
\end{equation}
i.e., fix the auxiliary parameters $g_\mathrm A$ and $\nu$ and a general endpoints condition $\p$ for the geodesic variational problem. Let us now study its extrema, which are the critical points of $E(g,\cdot\,)=E_g:\op(M)\to\R$, that will be obtained by a classic {\em first variation} argument. For this, we need $g$ to induce a metric on $\p$. It will be later evident\footnote{See Remark~\ref{re:whyg}.} that a convenient choice is to consider the metric
\begin{equation}\label{eq:overlineg}
\overline{g}=g\oplus(-g)
\end{equation}
on the product $M\times M$ and then its restriction to $\p$. Clearly, there are topological obstructions on $\p$ for this to be possible. This problem will be dealt with later, by reducing\footnote{Actually, further topological assumptions on $\p$ will be necessary, such as compactness.} the domain $\A_{g_\mathrm A,\nu}$ of parameters to an open subset of metrics $g$ where $\p$ is nondegenerate with respect to $\overline g$, see Proposition~\ref{prop:nondegenerateopen}. For now, we ignore this problem by considering only parameters that are metrics $g\in\A_{g_\mathrm A,\nu}$ such that $\overline g$ does not degenerate on $\p$. In this way, notions as $\overline g$--orthogonality involving tangent vectors to $\p$ are legitimate.

\begin{proposition}\label{prop:critgenenfunc}
A point $(g_0,\gamma_0)\in\mathcal U$ satisfies $\frac{\partial E}{\partial\gamma}(g_0,\gamma_0)=0$ if and only if $\gamma_0\in\op(M)$ is a $g_0$--geodesic (in particular, of class $C^2$) and
\begin{equation}\label{eq:gpgeod}
(\dot{\gamma_0}(0),\dot{\gamma_0}(1))\in T_{(\gamma_0(0),\gamma_0(1))}\p^\perp,
\end{equation}
where $^\perp$ denotes orthogonality with respect to $\overline{g_0}$.
\end{proposition}

\begin{proof}
Suppose first $\gamma_0\in\op(M)$ is a $g_0$--geodesic satisfying \eqref{eq:gpgeod}, in particular $\gamma_0\in C^2([a,b],M)$ and consider a $C^2$ variation $$(-\varepsilon,\varepsilon)\times [0,1]\ni (s,t)\longmapsto\gamma_s(t)\in M$$ with $\gamma_s=\gamma_0$ for $s=0$ and $(\gamma_s(0),\gamma_s(1))\in\p$ for all $s$. We will denote $\dot{\gamma_s}(t_0)\in T_{\gamma_s(t_0)}M$ the derivative $\frac{\partial}{\partial t}\gamma_s(t)\big|_{t=t_0}$. The infinitesimal variation $v$ associated induces a vector field $v\in\sect^1(\gamma^*TM)$ given by $$v(t)=\frac{\partial}{\partial s}\gamma_s(t)\Big|_{s=0}.$$ Thus, since from Proposition~\ref{prop:fck} the functional $E$ is of class $C^k$, we may compute
\begin{eqnarray}\label{eq:dfdgammaa}
\begin{aligned}
\frac{\partial E}{\partial\gamma}(g_0,\gamma_0)v &= \frac{\partial}{\partial s}E(g_0,\gamma_s)\Big|_{s=0} \\
&= \tfrac12\int_0^1 \frac{\partial}{\partial s} g_0\left(\dot{\gamma_s},\dot{\gamma_s}\right)\Big|_{s=0}\;\dd t \\
&= \int_0^1 g_0\left(\dot{\gamma_0},\D^{g_0} v\right)\;\dd t.
\end{aligned}
\end{eqnarray}
Notice that the above formula {\em a priori} does not hold for any Sobolev $H^1$ class variation $v\in T_{\gamma_0}\op(M)$, but only for $v\in\sect^1(\gamma_0^*TM)$.

Nevertheless, using again that $E$ is of class $C^k$, its derivative $$\frac{\partial E}{\partial\gamma}(g_0,\gamma_0):T_{\gamma_0}\op(M)\la\R$$ is continuous. From Corollary~\ref{cor:cinftyhk2}, $\sect^1(\gamma_0^*TM)$ is dense in $\sect^{H^1}(\gamma_0^*TM)$ hence in $T_{\gamma_0}\op(M)$. Thus, the continuous map $\frac{\partial E}{\partial\gamma}(g_0,\gamma_0)$ coincides in a dense subset with \eqref{eq:dfdgammaa}, which is also continuous. It follows that for all $v\in T_{\gamma_0}\op(M)$,
\begin{equation}\label{eq:dfdgammav}
\frac{\partial E}{\partial\gamma}(g_0,\gamma_0)v =\int_0^1 g_0\left(\dot{\gamma_0},\D^{g_0} v\right)\;\dd t.
\end{equation}
Notice that even if $v$ is {\em only} of Sobolev class $H^1$, the above integral is well--defined, since $\D^{g_0}v\in\sect^{L^2}(\gamma_0^*TM)$, see Remark~\ref{re:covderh1}. Using \eqref{eq:dfdgammav} and the $g_0$--geodesic equation, we may compute for all $v\in T_{\gamma_0}M$,
\begin{eqnarray}\label{eq:firstvariation}
\begin{aligned}
\frac{\partial E}{\partial\gamma}(g_0,\gamma_0)v &\stackrel{\eqref{eq:dfdgammav}}{=} \int_0^1 g_0(\dot{\gamma_0},\D^{g_0}v)\;\dd t\\
&= -\int_0^1 g_0(\D^{g_0}\dot{\gamma_0},v)\;\dd t+g_0(\dot{\gamma_0},v)\Big|_0^1\\
&= g_0(\dot{\gamma_0}(1),v(1))-g_0(\dot{\gamma_0}(0),v(0))\\
&= \overline{g_0}\big((\dot{\gamma_0}(0),\dot{\gamma_0}(1)),(v(0),v(1))\big)\\
&\stackrel{\eqref{eq:gpgeod}}{=} 0.
\end{aligned}
\end{eqnarray}
Thus, $\frac{\partial E}{\partial\gamma}(g_0,\gamma_0)=0$, i.e., $\gamma_0$ is a critical point of $E_{g_0}:\op(M)\to\R$.

Conversely, suppose $(g_0,\gamma_0)\in\mathcal U$ satisfies $\frac{\partial E}{\partial\gamma}(g_0,\gamma_0)=0$. Before any computations, we first have to ensure that $\gamma_0$ is sufficiently regular. Consider\footnote{Existence of such frame with weak regularity (Sobolev class $H^1$) is not evident, however follows from standard techniques of ODEs.} $\{e_i(t)\}_{i=1}^m$ a {\em $g_0$--parallel orthonormal frame} of $\gamma_0^*TM$, in other words, a $g_0$--orthonormal frame\footnote{Recall Definition~\ref{def:gorthframe}.} formed by vectors $e_i\in\sect^{H^1}(\gamma_0^*TM)$ along $\gamma_0$ that are $g_0$--parallel.\footnote{From Definition~\ref{def:parallel}, a vector field $v$ along $\gamma_0$ is $g_0$--parallel if it satisfies $\D^{g_0}v=0$. Notice however that in this context, this ODE is supposed to hold almost everywhere, since the considered vector fields are not $C^k$, but only Sobolev $H^1$.} Let $v\in T_{\gamma_0}H^1([0,1],M)$, and decompose it with respect to this frame,
\begin{equation*}
v(t)=\sum_{i=1}^m \lambda_i(t)e_i(t), \quad t\in [0,1],
\end{equation*}
where $\lambda_i:[0,1]\to\R$. Since $\frac{\partial E}{\partial\gamma}(g_0,\gamma_0)v=0$ for all $v\in T_{\gamma_0}H^1([0,1],M)$, in particular this holds for $v$'s such that for all $1\leq i\leq m$, $\lambda_i\in C^\infty_c(\,]0,1[,\R)$, for these $v$'s are clearly in $H^1([0,1],\R)$.

Denote $\lambda=(\lambda_i)_{i=1}^m\in C^\infty_c(\,]0,1[,\R^m)$ and notice that, since the frame is parallel,
\begin{equation*}
\D^{g_0} v(t)=\sum_{i=1}^m \lambda'_i(t)e_i(t), \quad t\in\, ]0,1[,
\end{equation*}
and $\lambda'=(\lambda'_i)_{i=1}^m\in C^\infty_c(\,]0,1[,\R^m)$. In addition, consider for all $1\leq i\leq m$,
\begin{equation*}
\alpha_i(t)=g_0(\dot{\gamma_0}(t),e_i(t)), \quad t\in [0,1],
\end{equation*}
and $\alpha=(\alpha_i)_{i=1}^m\in L^2([0,1],\R^m)$.

To have an expression of the form \eqref{eq:dfdgammav}, notice that the same density argument works, considering also $\gamma$ as a variable. More precisely, consider the maps
\begin{eqnarray*}
T\op(M)\ni (\gamma,v) &\longmapsto&\dfrac{\partial E}{\partial\gamma}(g_0,\gamma)v\in\R \\
T\op(M)\ni (\gamma,v) &\longmapsto&\displaystyle\int_0^1 g_0(\dot\gamma,\D^{g_0}v)\;\dd t\in\R.
\end{eqnarray*}
Both are continuous and coincide in the dense subset formed by pairs $(\gamma,v)$, where $\gamma\in C^2([0,1],M)\cap\op(M)$ and $v\in\sect^1(\gamma^*TM)$. Thus, the above maps coincides in the entire $T\op(M)$.

Therefore, we may compute
\begin{eqnarray*}
\frac{\partial E}{\partial\gamma}(g_0,\gamma_0)v &=& \int_0^1 g_0(\dot{\gamma_0},\D^{g_0}v)\;\dd t \\
&=& \int_0^1\sum_{i=1}^m g_0(\dot{\gamma_0},\lambda_i'e_i)\;\dd t \\
&=& \int_0^1\sum_{i=1}^m \lambda'_i g_0(\dot{\gamma_0},e_i)\;\dd t \\
&=& \int_0^1\sum_{i=1}^m \alpha_i\lambda'_i\;\dd t \\
&=& \int_0^1\langle\alpha,\lambda'\rangle\;\dd t.
\end{eqnarray*}
Since $\gamma_0$ is a critical point of $E_{g_0}$, the above expression vanishes for all $v$, hence for all $\lambda\in C^\infty_c(\,]0,1[,\R^m)$. From Lemma~\ref{le:distder0}, it follows that $\alpha$ is constant almost everywhere. This means that there exists $a=(a_i)_{i=1}^m\in\R^m$ such that $$\alpha_i=g_0(\dot{\gamma_0}(t),e_i(t))=a_i$$ for almost every $t\in [0,1]$. Let $\delta_i=g_0(e_i,e_i)=\pm1$. Then,
\begin{equation}\label{eq:dgamma0t}
\dot{\gamma_0}(t)=\sum_{i=1}^m \delta_ia_ie_i(t),
\end{equation}
for almost every $t\in [0,1]$. Notice that the right--hand side of \eqref{eq:dgamma0t} is continuous and, since $\gamma_0$ is of Sobolev class $H^1$, it is absolutely continuous. From Corollary~\ref{cor:milagre}, it follows that $\gamma_0$ is of class $C^1$ and the above equality holds for every $t\in [0,1]$. Thus, the frame $\{e_i(t)\}_{i=1}^m$ is a $g_0$--parallel frame along a $C^1$ curve, hence also of class $C^1$. Therefore, it follows again from \eqref{eq:dgamma0t} that $\dot{\gamma_0}$ is of class $C^1$, hence $\gamma_0$ is of class $C^2$. This gives the necessary regularity to proceed. Moreover, \eqref{eq:dgamma0t} implies that the tangent field $\dot{\gamma_0}$ is $g_0$--parallel, hence
\begin{equation}\label{eq:g0geodeq}
\D^{g_0}\dot{\gamma_0}=0
\end{equation}
i.e., $\gamma_0$ is a $g_0$--geodesic.

Finally, since $\gamma_0$ is of class $C^2$, the same integration by parts to obtain expression \eqref{eq:firstvariation} holds in this case. More precisely, we may compute
{\allowdisplaybreaks
\begin{eqnarray*}
\frac{\partial E}{\partial\gamma}(g_0,\gamma_0)v &\stackrel{\eqref{eq:dfdgammav}}{=}& \int_0^1 g_0(\dot{\gamma_0},\D^{g_0}v)\;\dd t\\
&=& -\left.\int_0^1 g_0(\D^{g_0}\dot{\gamma_0},v)\;\dd t+g_0(\dot{\gamma_0},v)\right|_0^1\\
&\stackrel{\eqref{eq:g0geodeq}}{=}& g_0(\dot{\gamma_0}(1),v(1))-g_0(\dot{\gamma_0}(0),v(0))\\
&=& -\overline{g_0}\big((\dot{\gamma_0}(0),\dot{\gamma_0}(1)),(v(0),v(1))\big).
\end{eqnarray*}}Since the above expression vanishes for all $v$'s, it follows that $\gamma_0$ must also satisfy \eqref{eq:gpgeod}, concluding the proof.
\end{proof}

\begin{definition}\label{def:gpgeod}
A curve $\gamma\in\op(M)$ is called a \emph{$(g,\p)$--geodesic}\index{$(g,\p)$--geodesic}\index{Geodesic!$(g,\p)$--geodesic} if $$\frac{\partial E}{\partial \gamma}(g,\gamma)=0,$$ i.e., if $\gamma$ is a critical point of $E_g:\op(M)\to\R$.
From Proposition~\ref{prop:critgenenfunc}, this is equivalent to $\gamma$ being a $g$--geodesic that satisfies $$(\dot\gamma(0),\dot\gamma(1))\in T_{(\gamma(0),\gamma(1))}\p^\perp,$$ where $^\perp$ denotes orthogonality relatively to $\overline g$.
\end{definition}

\begin{remark}\label{re:gammack}
From Corollary~\ref{cor:geodck}, since $g\in\met^k_\nu(M)$, if $\gamma$ is a $g$--geodesic then $\gamma$ is of class $C^{k+1}$. In particular, $(g,\p)$--geodesics are $C^{k+1}$.
\end{remark}

\begin{example}\label{ex:gecs2}
Consider the GECs given in Example~\ref{ex:gecs}. If $\p=\{p\}\times\{q\}$, $(g,\p)$--geodesics are $g$--geodesics joining $p$ and $q$. Since the tangent space to $\p$ is trivial, condition \eqref{eq:gpgeod} is also trivial. In case $P$ and $Q$ are submanifolds of $M$ and $\p=P\times Q$, the $(g,\p)$--geodesics are $g$--geodesics $g$--orthogonal to $P$ and $Q$ at its endpoints. This follows at once since $T_{(\gamma(0),\gamma(1))}\p=T_{\gamma(0)}P\oplus T_{\gamma(1)}Q$.

Finally, let us examine the special case $\p=\Delta$. The tangent space to $\Delta$ at $(x,x)\in\Delta$ is the diagonal $\Delta\subset T_xM\oplus T_xM$. Thus, {\em for every} metric $g\in\A_{g_\mathrm A,\nu}$, the product metric $\overline g=g\oplus (-g)$ is identically null at $\Delta$, and hence $\overline g$--orthogonality at such points is meaningless. Nevertheless, the computations above imply that $\gamma\in\Omega_\p(M)$ is a $(g,\p)$--geodesic if and only if it is a periodic $g$--geodesic, see Example~\ref{ex:periodicgeod}. Recall that this means not only $\gamma(0)=\gamma(1)$, but also $\dot\gamma(0)=\dot\gamma(1)$. Identifying $\Omega_\Delta(M)$ and $H^1(S^1,M)$, this simply means that the Sobolev $H^1$ curve $\gamma:S^1\to M$ is of class $C^2$ and satisfies the $g$--geodesic equation.
\end{example}

\section{Generalized index form}
\label{sec:genindexform}

In order to describe degeneracy of $(g,\p)$--geodesics, we need to analyze the second variation of $E_g:\op(M)\to\R$ at its critical points, which is given by the {\em index form} of this generalized energy functional \eqref{eq:efunct}. More precisely, we are interested in obtaining an explicit formula for the second derivative $$\frac{\partial^2 E}{\partial \gamma^2}(g_0,\gamma_0):T_{\gamma_0}\op(M)\times T_{\gamma_0}\op(M)\la\R$$ at points $(g_0,\gamma_0)\in\mathcal U$ such that $\frac{\partial E}{\partial\gamma}(g_0,\gamma_0)=0$, and determining its Fredholmness. In addition, we will also obtain an explicit formula at such points $(g_0,\gamma_0)\in\mathcal U$ for the mixed derivative $$\frac{\partial^2 E}{\partial g\partial\gamma}(g_0,\gamma_0):\mathds E\oplus T_{\gamma_0}\op(M)\la\R.$$

\begin{proposition}\label{prop:indexform}
Fix $(g_0,\gamma_0)\in\mathcal U$ such that $\frac{\partial E}{\partial\gamma}(g_0,\gamma_0)=0$. The {\em index form} of $E$ at $(g_0,\gamma_0)$ is given by
\begin{multline}\label{eq:indexform}
\frac{\partial^2 E}{\partial \gamma^2}(g_0,\gamma_0)(v,w)=\int_0^1 g_0(\D^{g_0}v,\D^{g_0}w)-g_0(R^{g_0}(\dot{\gamma_0},v)w,\dot{\gamma_0})\;\dd t \\ -\s^\p_{(\dot{\gamma_0}(0),\dot{\gamma_0}(1))}\Big((v(0),v(1)),(w(0),w(1))\Big),
\end{multline}
for all $v,w\in T_\gamma\op(M)$, where $\s^\p_\eta$ is the second fundamental form of $\p$ with normal $\eta\in T\p^\perp$, with respect to the ambient metric $\overline{g_0}$.
\end{proposition}

\begin{proof}
From Proposition~\ref{prop:critgenenfunc}, $\gamma_0:[0,1]\to M$ is a $g_0$--geodesic. Using Corollary~\ref{cor:geodck}, it follows that $\gamma_0$ is automatically of class $C^{k+1}$, in particular $C^3$, see Remark~\ref{re:gammack}. This allows us to use, for instance, $C^3$ variations of $\gamma$ and some integration by parts among other analytical tools.

Formula \eqref{eq:indexform} above is obtained through a standard argument, using variations of $\gamma_0$ by other curves in $M$ that satisfy the same GEC. Namely, consider a $C^3$ variation $$(-\varepsilon,\varepsilon)\times [0,1]\ni (s,t)\longmapsto\gamma_s(t)\in M$$ with $\gamma_s=\gamma_0$ for $s=0$ and $(\gamma_s(0),\gamma_s(1))\in\p$. We will denote $\dot{\gamma_s}(t_0)\in T_{\gamma_s(t_0)}M$ the derivative $\frac{\partial}{\partial t}\gamma_s(t)\big|_{t=t_0}$. The infinitesimal variation $v$ associated induces a vector field $v\in\sect^2(\gamma^*TM)$ given by $$v(t)=\left.\frac{\partial}{\partial s}\gamma_s(t)\right|_{s=0}.$$ From Proposition~\ref{prop:fck}, the functional $E$ is of class $C^k$ and hence we may compute as in \eqref{eq:dfdgammaa},
\begin{eqnarray*}
\frac{\partial E}{\partial\gamma}(g_0,\gamma_0)v &=& \frac{\partial}{\partial s}E(g_0,\gamma_s)\Big|_{s=0}\\
&=& \int_0^1 g_0\left(\frac{\D}{\partial s}\dot{\gamma_s},\dot{\gamma_0}\right)\;\dd t.
\end{eqnarray*}
Deriving again and applying integration by parts, it follows that
{\allowdisplaybreaks
\begin{eqnarray*}
\frac{\partial^2 E}{\partial\gamma^2}(g_0,\gamma_0)(v,v) &=& \frac{\partial^2}{\partial s^2}E(g_0,\gamma_s)\Big|_{s=0} \\
&=& \int_0^1 g_0\left(\frac{\D}{\partial s}\dot{\gamma_s},\frac{\D}{\partial s}\dot{\gamma_s}\right)+g_0\left(\frac{\D}{\partial s}\frac{\D}{\partial s}\dot{\gamma_s},\dot{\gamma_0}\right)\;\dd t \\
&=& \int_0^1 g_0\left(\D^{g_0}v,\D^{g_0}v\right)+g_0\left(\frac{\D}{\partial s}\frac{\D}{\partial t}v,\dot{\gamma_0}\right)\;\dd t \\
&=& \int_0^1 g_0\left(\D^{g_0}v,\D^{g_0}v\right)+g_0(R^{g_0}(v,\dot{\gamma_0})v,\dot{\gamma_0}) \\
&&\hspace{0.5cm}+g_0\left(\frac{\D}{\partial t}\frac{\D}{\partial s}\frac{\partial}{\partial s}\gamma_s,\dot{\gamma_0}\right)\;\dd t \\
&=& \int_0^1 g_0\left(\D^{g_0}v,\D^{g_0}v\right)-g_0(R^{g_0}(\dot{\gamma_0},v)v,\dot{\gamma_0})\;\dd t \\
&&\hspace{0.5cm}+\left.g_0\left(\frac{\D}{\partial s}\frac{\partial}{\partial s}\gamma_s,\dot{\gamma_0}\right)\right|_0^1  \\
&=& \int_0^1 g_0\left(\D^{g_0}v,\D^{g_0}v\right)-g_0(R^{g_0}(\dot{\gamma_0},v)v,\dot{\gamma_0})\;\dd t \\
&&\hspace{0.5cm}+g_0\Big(\nabla^{g_0}_{v(1)}v(1),\dot{\gamma_0}(1)\Big)-g_0\Big(\nabla^{g_0}_{v(0)}v(0),\dot{\gamma_0}(0)\Big)  \\
&=& \int_0^1 g_0\left(\D^{g_0}v,\D^{g_0}v\right)-g_0(R^{g_0}(\dot{\gamma_0},v)v,\dot{\gamma_0})\;\dd t \\
&&\hspace{0.5cm}-\overline{g_0}\Big(\nabla^{\overline{g_0}}_{(v(0),v(1))}(v(0),v(1)),(\dot{\gamma_0}(0),\dot{\gamma_0}(1)\Big)  \\
&=& \int_0^1 g_0\left(\D^{g_0}v,\D^{g_0}v\right)-g_0(R^{g_0}(\dot{\gamma_0},v)v,\dot{\gamma_0})\;\dd t \\
&&\hspace{0.5cm}-\s_{(\dot{\gamma_0}(0),\dot{\gamma_0}(1))}^\p\Big((v(0),v(1)),(v(0),v(1))\Big).
\end{eqnarray*}}By applying standard polarization arguments\footnote{Recall \eqref{eq:polarization} in Lemma~\ref{le:polarization}.} to the last expression, since it is bilinear and symmetric, we obtain \eqref{eq:indexform},
\begin{multline*}
\frac{\partial^2 E}{\partial \gamma^2}(g_0,\gamma_0)(v,w)=\int_0^1 g_0(\D^{g_0}v,\D^{g_0}w)-g_0(R^{g_0}(\dot{\gamma_0},v)w,\dot{\gamma_0})\;\dd t \\ -\s^\p_{(\dot{\gamma_0}(0),\dot{\gamma_0}(1))}\Big((v(0),v(1)),(w(0),w(1))\Big).
\end{multline*}
Notice that the above formula {\em a priori} does not hold for any vectors $v,w\in T_{\gamma_0}\op(M)$, but only for $v,w\in\sect^2(\gamma_0^*TM)$.

Nevertheless, from Proposition~\ref{prop:fck}, the functional $E$ is of class $C^k$. Thus, its derivative $$\dfrac{\partial^2 E}{\partial\gamma^2}(g_0,\gamma_0):T_{\gamma_0}\op(M)\times T_{\gamma_0}\op(M)\la\R$$ is a symmetric {\em continuous} bilinear form. From Corollary~\ref{cor:cinftyhk2}, $\sect^2(\gamma_0^*TM)$ is dense in $\sect^{H^1}(\gamma_0^*TM)$ hence in $T_{\gamma_0}\op(M)$. Thus, the continuous bilinear form $\frac{\partial^2 E}{\partial \gamma^2}(g_0,\gamma_0)$ coincides in a dense subset with the above formula, which is also continuous. It follows that for all $v,w\in T_{\gamma_0}\op(M)$ the formula \eqref{eq:indexform} holds,\footnote{Notice that even if $v,w$ are {\em only} of Sobolev class $H^1$, the integral in \eqref{eq:indexform} is well--defined, since $\D^{g_0}v\in\sect^{L^2}(\gamma_0^*TM)$, see Remark~\ref{re:covderh1}.} concluding the proof.
\end{proof}

\begin{proposition}\label{prop:fredholmness}
Fix $(g_0,\gamma_0)\in\mathcal U$ such that $\frac{\partial E}{\partial\gamma}(g_0,\gamma_0)=0$. The index form \eqref{eq:indexform}, represented using \eqref{ident:bilin} as \begin{equation}\label{eq:indexformrep}
\frac{\partial^2 E}{\partial\gamma^2}(g_0,\gamma_0):T_{\gamma_0}\op(M)\la T_{\gamma_0}\op(M)^*\cong T_{\gamma_0}\op(M)
\end{equation}
is a self--adjoint Fredholm operator of this Hilbert space.
\end{proposition}

\begin{proof}
Recall that from Proposition~\ref{prop:critgenenfunc}, $\gamma_0:[0,1]\to M$ is a $g_0$--geodesic, in particular of class $C^2$. Self--adjointness of \eqref{eq:indexformrep} is evident since it represents\footnote{See Definition~\ref{def:represents}.} a symmetric bilinear form of this Hilbert space.

For each $t\in [0,1]$, denote by $A_t\in\GL(T_{\gamma_0(t)}M)$ the $g_\mathrm R$--symmetric automorphism that represents $g_0$ in terms of the fixed Riemannian metric $g_\mathrm{R}$, that is, such that $g_0=g_\mathrm{R}(A_t\cdot,\cdot)$. Then the map
\begin{eqnarray}
\Phi:T_{\gamma_0}\op(M)&\la & T_{\gamma_0}\op(M)\\
v &\longmapsto & \tilde v \nonumber
\end{eqnarray}
where $\tilde{v}(t)=A_tv(t)$, is an isomorphism. We now prove that \eqref{eq:indexformrep} is a compact perturbation of the isomorphism $\Phi$. From Proposition~\ref{prop:fredholmisocomp}, it will then follow that it is a Fredholm operator. Recall that the inner product $\llangle\cdot,\cdot\rrangle$ in $T_\gamma\op(M)$ is given by \eqref{eq:riemhilbop}, hence
\begin{eqnarray*}
\llangle\Phi v,w\rrangle &=& g_\mathrm R(A_0 v(0),w(0))+\int_0^1 g_\mathrm R(\D^\mathrm R\Phi v,\D^\mathrm R w)\;\dd t\\
&=& g_\mathrm R(A_0 v(0),w(0))+\int_0^1 g_\mathrm R(A'v,\D^\mathrm R w)+g_\mathrm R(A\D^\mathrm Rv,\D^\mathrm R w)\;\dd t,\\
\end{eqnarray*}
where $A'$ is the covariant derivative\footnote{$A$ can be thought as a $C^k$ section of $\gamma_0^*(TM^*\vee TM)$, see Example~\ref{ex:vectoralongcurve}. From Theorem~\ref{thm:tensorconnection}, $\nabla^\mathrm{R}$ canonically induces a connection on $TM^*\vee TM^*$. Furthermore, Definition~\ref{def:pullbackconnection} guarantees the existence of a pull--back connection on $\gamma_0^*(TM^*\vee TM)$, which is used to compute $A'$.} of $A$. Denote by $\chr^\mathrm{R}=\D^{g_0}-\D^\mathrm{R}$ the Christoffel tensor of $\nabla^{g_0}$ relatively to $\nabla^\mathrm{R}$, see Definition~\ref{def:christoffeltens}. The difference
\begin{eqnarray*}
D:\sect^{H^1}(\gamma_0^*TM)\times\sect^{H^1}(\gamma_0^*TM)&\la &\R \\
(v,w) &\longmapsto &\frac{\partial^2 E}{\partial \gamma^2}(g_0,\gamma_0)(v,w)-\llangle\Phi v,w\rrangle
\end{eqnarray*}
is clearly a continuous bilinear symmetric form, that can be now computed as follows.
{\allowdisplaybreaks
\begin{eqnarray*}
\begin{aligned}
D(v,w) &= \int_0^1 g_0(\D^{g_0}v,\D^{g_0}w)-g_0(R^{g_0}(\dot{\gamma_0},v)w,\dot{\gamma_0})\;\dd t  \\
&\hspace{0.5cm} -\s^\p_{(\dot{\gamma_0}(0),\dot{\gamma_0}(1))}\Big((v(0),v(1)),(w(0),w(1))\Big)-g_\mathrm R(A_0 v(0),w(0))  \\
&\hspace{0.5cm} -\int_0^1 g_\mathrm R(A'v,\D^\mathrm R w)+g_\mathrm R(A\D^\mathrm Rv,\D^\mathrm R w)\;\dd t \\
&= \int_0^1 \Big[g_0(\D^{g_0}v,\D^{g_0}w)-g_\mathrm R(A'v,\D^\mathrm R w)-g_\mathrm R(A\D^\mathrm Rv,\D^\mathrm R w) \\
&\hspace{0.5cm} -g_0(R^{g_0}(\dot{\gamma_0},v)w,\dot{\gamma_0})\Big]\;\dd t \\
&\hspace{0.5cm} -g_\mathrm R(A_0 v(0),w(0))-\s^\p_{(\dot{\gamma_0}(0),\dot{\gamma_0}(1))}\Big((v(0),v(1)),(w(0),w(1))\Big)  \\
&= \int_0^1\Big[-g_\mathrm R(A'v,\D^\mathrm{R}w)+g_\mathrm{R}(A\D^\mathrm{R}v,\chr^\mathrm{R}w)+g_\mathrm{R}(A\chr^\mathrm{R}v,\D^\mathrm{R}w) \\
&\hspace{0.5cm} +g_\mathrm{R}(A\chr^\mathrm{R}v,\chr^\mathrm{R}w)+g_\mathrm{R}(AR^{g_0}(\dot{\gamma_0},v)\dot{\gamma_0},w)\Big]\;\dd t  \\
&\hspace{0.5cm} -g_\mathrm R(A_0 v(0),w(0))-\s^\p_{(\dot{\gamma_0}(0),\dot{\gamma_0}(1))}\Big((v(0),v(1)),(w(0),w(1))\Big).
\end{aligned}
\end{eqnarray*}}We now briefly explain why the above bilinear form $D$ is represented by a compact operator $$T_D:T_{\gamma_0}\op(M)\la\big[T_{\gamma_0}\op(M)\big]^*\cong T_{\gamma_0}\op(M).$$ Notice that each term of the above integral is a continuous bilinear form in $T_{\gamma_0}\op(M)$ that does not contain more than one derivative of its arguments. More precisely, each of these bilinear forms can be written as the composition of the continuous covariant derivative operator\footnote{See Remark~\ref{re:covderh1}.} $\D^\mathrm R:\sect^{H^1}(\gamma_0^*TM)\to\sect^{L^2}(\gamma_0^*TM)$ and other $\sect^{L^2}(\gamma_0^*TM)$--continuous operators, such as $\chr^\mathrm R$ and $A$ for instance. Each $g_\mathrm R$--product of a couple of such composite operators along $\gamma_0^*TM$ is hence in $L^1([0,1],\R)$, and its integral is therefore continuous. All the other integrand terms without covariant derivatives are also clearly $\sect^{H^1}(\gamma_0^*TM)$--continuous, since they are $g_\mathrm R$--products of composite $\sect^{L^2}(\gamma_0^*TM)$--continuous operators, hence also in $L^1([0,1],\R)$. Therefore, up to convenient identifications\footnote{See Remark~\ref{re:TH1H1T}.} of $T_{\gamma_0}\op(M)$, each integrand term is a bilinear form on $\sect^{H^1}(\gamma_0^*TM)$ that is $\sect^{H^1}(\gamma_0^*TM)$--continuous in one variable (or in both, if there are no covariant derivatives involved) and $\sect^0(\gamma_0^*TM)$--continuous in the other variable. It then follows from Lemma~\ref{le:h1c0extension} that each of these terms is represented by a compact operator of $T_{\gamma_0}\op(M)$.

With regard to the last two terms of the above expression, which are not integrands, they are obviously represented by a compact operator of $T_{\gamma_0}\op(M)$ since they are composite operators involving a linearized evaluation map of the form \eqref{eq:dev01}, which has finite rank and is hence compact.

Thus, $D$ is represented by a compact operator $T_D\in\cpc(T_{\gamma_0}\op(M))$, given by the sum of the compact operators above described that represent each term of $D$, see Proposition~\ref{prop:propcpcop}. This implies, by Proposition~\ref{prop:fredholmisocomp}, that \eqref{eq:indexformrep} is Fredholm, concluding the proof.
\end{proof}

We end this section calculating the mixed derivative $\frac{\partial^2 E}{\partial g\partial\gamma}(g_0,\gamma_0)$, which will be later useful for our genericity results.

\begin{proposition}\label{prop:mixedderivative}
Fix $(g_0,\gamma_0)\in\mathcal U$ such that $\frac{\partial E}{\partial\gamma}(g_0,\gamma_0)=0$. Consider $\nabla$ any symmetric connection on $M$ and denote by $\D$ the covariant derivative operator of vector fields along $\gamma_0$ induced\footnote{Recall Proposition~\ref{prop:covder}. Although only stated for Levi--Civita connections, the result holds for any symmetric connection on $TM$, see Definition~\ref{def:symflat}. Alternatively, consider $\nabla$ to be the Levi--Civita connection of some metric on $M$ and apply directly Proposition~\ref{prop:covder}.} by $\nabla$. Then for all $v\in T_{\gamma_0}\op(M)$ and $h\in\mathds E$,
\begin{equation}\label{eq:mixedderivative}
\frac{\partial^2 E}{\partial g\partial\gamma}(g_0,\gamma_0)(h,v)=\int_0^1 h(\dot{\gamma_0},\D v)+\tfrac{1}{2}\nabla h(v,\dot{\gamma_0},\dot{\gamma_0})\;\dd t,
\end{equation}
\end{proposition}

\begin{proof}
For this proof, it is convenient to use the Schwartz Lemma. Let us briefly explain the context where this calculation simplifier will be employed. Recall that from Proposition~\ref{prop:fck}, the energy functional $E$ is of class $C^k$. Since the domain $\mathcal U=\A_{g_\mathrm A,\nu}\times\op(M)$ is the product of an open subset $\A_{g_\mathrm A,\nu}$ of an affine Banach space $g_\mathrm A+\mathds{E}$ and a Hilbert manifold $\op(M)$, the first partial derivative can be thought as $$\frac{\partial E}{\partial g}:\A_{g_\mathrm A,\nu}\times\op(M)\la\mathds{E}^*,$$ which is explicitly given by \eqref{eq:dfdg}. Deriving $\frac{\partial E}{\partial g}(g_0,\,\cdot\,)$, one obtains
\begin{equation}\label{eq:yx}
\frac{\partial}{\partial\gamma}\frac{\partial E}{\partial g}(g_0,\gamma_0):T_{\gamma_0}\op(M)\la\mathds E^*,
\end{equation}
which may also be seen as a bilinear form on $T_{\gamma_0}\op(M)\times\mathds E$. If instead of deriving $E$ first in $g$, one derives first in $\gamma$ and then in $g$, the result is
\begin{equation}\label{eq:xy}
\frac{\partial}{\partial g}\frac{\partial E}{\partial \gamma}(g_0,\gamma_0):\mathds E\la T_{\gamma_0}\op(M)^*,
\end{equation}
which is a bilinear form on $\mathds E\times T_{\gamma_0}\op(M)$. Using local charts and the Schwartz Lemma, it follows that these maps are transpose to each other, that is, for all $(h,v)\in\mathds E\times T_{\gamma_0}\op(M)$,
\begin{equation}\label{eq:xyyxtranspose}
\frac{\partial^2 E}{\partial g\partial\gamma}(g_0,\gamma_0)(h,v)=\frac{\partial^2 E}{\partial\gamma\partial g}(g_0,\gamma_0)(v,h).
\end{equation}
Thus, since we are interested in computing the mixed derivative \eqref{eq:xy}, however it turns out to be easier to compute \eqref{eq:yx}, we now use the above observation.

Recall that from \eqref{eq:dfdg}, since $E$ is linear in the first variable, for all $h\in\mathds E$, $$\frac{\partial E}{\partial g}(g_0,\gamma)h=\tfrac{1}{2}\int_0^1 h(\dot{\gamma},\dot{\gamma})\;\dd t.$$ We would now like to derivate the above expression with respect to $\gamma$, to obtain a formula for \eqref{eq:yx}. From Corollary~\ref{cor:geodck}, $\gamma_0$ is of class $C^2$, see also Remark~\ref{re:gammack}. Consider a $C^2$ variation $$(-\varepsilon,\varepsilon)\times [0,1]\ni (s,t)\longmapsto\gamma_s(t)\in M$$ with $\gamma_s=\gamma_0$ for $s=0$ and $(\gamma_s(0),\gamma_s(1))\in\p$ for all $s$. We will denote $\dot{\gamma_s}(t_0)\in T_{\gamma_s(t_0)}M$ the derivative $\frac{\partial}{\partial t}\gamma_s(t)\big|_{t=t_0}$. The infinitesimal variation $v$ associated induces a vector field $v\in\sect^1(\gamma^*TM)$ given by $$v(t)=\left.\frac{\partial}{\partial s}\gamma_s(t)\right|_{s=0}.$$ Consider a symmetric connection $\nabla$ on $M$ and the covariant derivative operator $\D$ of vector fields along $\gamma_0$ induced by $\nabla$. Thus, since from Proposition~\ref{prop:fck} the functional $E$ is of class $C^k$, we may compute
\begin{equation}\label{eq:ddfdgammadg}
\begin{aligned}
\frac{\partial^2 E}{\partial\gamma\partial g}(g_0,\gamma_0)(v,h) &= \frac{\partial}{\partial s}\left.\frac{\partial E}{\partial g}(g_0,\gamma_s)h\right|_{s=0}\\
&= \tfrac12\int_0^1 \frac{\partial}{\partial s} h(\dot{\gamma_s},\dot{\gamma_s})\Big|_{s=0}\;\dd t\\
&= \int_0^1 h(\dot{\gamma}_0,\D v)+\tfrac{1}{2}\nabla h(v,\dot{\gamma_0},\dot{\gamma_0})\;\dd t.
\end{aligned}
\end{equation}
It is easy to see that the following construction does not depend on the choice of $\nabla$. Indeed, replacing $\nabla$ with $\nabla'$ above, the difference between the obtained expressions vanishes identically from the symmetries of the Christoffel tensor $\nabla-\nabla'$, see Definition~\ref{def:christoffeltens} and \eqref{eq:christoffeltensor}. Furthermore, notice that the above formula {\em a priori} does not hold for any Sobolev $H^1$ class variation $v\in T_{\gamma_0}\op(M)$, but only for $v\in\sect^1(\gamma_0^*TM)$.

Nevertheless, using again that $E$ is of class $C^k$, its derivative $$\frac{\partial^2 E}{\partial\gamma\partial g}(g_0,\gamma_0):\mathds E\oplus T_{\gamma_0}\op(M)\la\R$$ is continuous. From Corollary~\ref{cor:cinftyhk2}, $\sect^1(\gamma_0^*TM)$ is dense in $\sect^{H^1}(\gamma_0^*TM)$ hence in $T_{\gamma_0}\op(M)$. Thus, the continuous map $\frac{\partial^2 E}{\partial\gamma\partial g}(g_0,\gamma_0)$ coincides in a dense subset with \eqref{eq:ddfdgammadg}, which is also continuous. It follows that for all $(h,v)\in\mathds E\times T_{\gamma_0}\op(M)$,
\begin{equation*}
\frac{\partial^2 E}{\partial\gamma\partial g}(g_0,\gamma_0)(v,h)=\int_0^1 h(\dot{\gamma_0},\D v)+\tfrac{1}{2}\nabla h(v,\dot{\gamma_0},\dot{\gamma_0})\;\dd t.
\end{equation*}
Notice that even if $v$ is {\em only} of Sobolev class $H^1$, the above integral is well--defined, since $\D v\in\sect^{L^2}(\gamma_0^*TM)$, see Remark~\ref{re:covderh1}. This gives a formula for \eqref{eq:yx}, and hence for its transpose \eqref{eq:xy}. From \eqref{eq:xyyxtranspose}, it follows that \eqref{eq:mixedderivative} holds, concluding the proof.
\end{proof}

\section{\texorpdfstring{$\p$--Jacobi fields}{P--Jacobi fields}}
\label{sec:pjacobi}

In the last sections, we studied critical points $\gamma_0$ of the $g_0$--energy functional $E_{g_0}=E(g_0,\cdot\,)$ with endpoints condition $\p$, i.e., $$(g_0,\gamma_0)\in\mathcal U, \quad\frac{\partial E}{\partial\gamma}(g_0,\gamma_0)=0.$$ In order to characterize degeneracy of such critical points, it is necessary to study the kernel of the index form $\frac{\partial^2 E}{\partial \gamma^2}(g_0,\gamma_0)$ given by \eqref{eq:indexform}, recall Definition~\ref{def:degnondegstdeg}. Such kernel is formed by special $g_0$--Jacobi fields along $\gamma_0$ that describe the variational character of the endpoints condition $\p$, see Definitions~\ref{def:jacobifield} and~\ref{def:pjacobi}. Thus, one expects these Jacobi fields to satisfy a {\em linearized endpoints condition} that involves submanifold objects associated to $\p$ such as its second fundamental form $\s^\p$, as confirmed by the next result.

\begin{proposition}\label{prop:pjacobi}
Fix $(g_0,\gamma_0)\in\mathcal U$ such that $\frac{\partial E}{\partial\gamma}(g_0,\gamma_0)=0$. Then the kernel $\ker\frac{\partial^2 E}{\partial \gamma^2}(g_0,\gamma_0)$ of the index form \eqref{eq:indexform} is the subspace of $T_{\gamma_0}\op(M)$ formed by $g_0$--Jacobi fields $J\in\sect^2(\gamma_0^*TM)$ along $\gamma_0$, such that 
\begin{equation}\label{eq:prepcampodijacobi}
(\D^{g_0} J(0),\D^{g_0} J(1))+\s^\p_{(\dot{\gamma_0}(0),\dot{\gamma_0}(1))}\Big(J(0),J(1)\Big)\in T_{(\gamma_0(0),\gamma_0(1))}\p^\perp,
\end{equation}
where $^\perp$ denotes orthogonality with respect to $\overline{g_0}$.
\end{proposition}

\begin{proof}
Recall that from Proposition~\ref{prop:critgenenfunc}, $\gamma_0:[0,1]\to M$ is a $g_0$--geodesic, in particular of class $C^2$. Suppose $J\in\sect^2(\gamma_0^*TM)$ is a $g_0$--Jacobi field along $\gamma_0$ satisfying \eqref{eq:prepcampodijacobi}. Then, using the $g_0$--Jacobi equation \eqref{eq:jacobi} and \eqref{eq:prepcampodijacobi}, we may compute, using identifications \eqref{ident:bilin},
{\allowdisplaybreaks
\begin{eqnarray*}
\frac{\partial^2 E}{\partial \gamma^2}(g_0,\gamma_0)(J,v) &\stackrel{\eqref{eq:indexform}}{=}& \int_0^1 g_0(\D^{g_0}J,\D^{g_0}v)-g_0(R^{g_0}(\dot{\gamma_0},J)v,\dot{\gamma_0})\;\dd t \\
&&\hspace{0.5cm}-\s^\p_{(\dot{\gamma_0}(0),\dot{\gamma_0}(1))}\Big((J(0),J(1)),(v(0),v(1))\Big)\\
&=&\int_0^1 g_0\big(-(\D^{g_0})^2J,v\big)+g_0(R^{g_0}(\dot{\gamma_0},J)\dot{\gamma_0},v)\;\dd t \\
&&\hspace{0.5cm}+g_0(\D^{g_0}J,v)\Big|_0^1\\
&&\hspace{0.5cm}-\s^\p_{(\dot{\gamma_0}(0),\dot{\gamma_0}(1))}\Big((J(0),J(1)),(v(0),v(1))\Big)\\
&\stackrel{\eqref{eq:jacobi}}{=}& g_0(\D^{g_0}J(1),v(1))-g_0(\D^{g_0}J(0),v(0))\\
&&\hspace{0.5cm}-\s^\p_{(\dot{\gamma_0}(0),\dot{\gamma_0}(1))}\Big((J(0),J(1)),(v(0),v(1))\Big)\\
&=&-\overline{g_0}\Big(\big(\D^{g_0}J(0),\D^{g_0}J(1)\big),(v(0),v(1))\Big)\\
&&\hspace{0.5cm}-\overline{g_0}\Big(\s^\p_{(\dot{\gamma_0}(0),\dot{\gamma_0}(1))}\Big(J(0),J(1)\Big),(v(0),v(1))\Big)\\
&\stackrel{\eqref{eq:prepcampodijacobi}}{=}&0.
\end{eqnarray*}}Thus, $J\in\ker\frac{\partial^2 E}{\partial \gamma^2}(g_0,\gamma_0)$.

Conversely, suppose $J\in\ker\frac{\partial^2 E}{\partial \gamma^2}(g_0,\gamma_0)$. Before any computations, we first have to ensure that $J$ is sufficiently regular. Since $\gamma_0$ is of class $C^2$, we may consider $\{e_i(t)\}_{i=1}^m$ a $g_0$--parallel orthonormal frame of $\gamma_0^*TM$, i.e., a $g_0$--orthonormal frame\footnote{Recall Definition~\ref{def:gorthframe}.} formed by vectors $e_i\in\sect^{H^1}(\gamma_0^*TM)$ along $\gamma_0$ that are $g_0$--parallel.\footnote{From Definition~\ref{def:parallel}, a vector field $v$ along $\gamma_0$ is $g_0$--parallel if it satisfies $\D^{g_0}v=0$. Notice however that in this context, this ODE is supposed to hold almost everywhere, since the considered vector fields are not $C^k$, but only Sobolev $H^1$.} Define for all $1\leq i\leq m$ and $t\in [0,1]$,
\begin{align*}
J_i(t) &= g_0(J(t),e_i(t));\\
R_i(t) &= g_0\big(R^{g_0}(\dot{\gamma_0}(t),J(t))\dot{\gamma_0}(t),e_i(t)\big);\\
\intertext{and}
b_i(t) &= R_i(0)+\int_0^t R_i(s)\;\dd s.
\end{align*}
Notice that $J_i\in H^1([0,1],\R)$, $R_i\in C^0([0,1],\R)$ and $b_i\in C^1([0,1],\R)$, with
\begin{equation}
b_i'(t)=R_i(t), \quad t\in [0,1],\quad 1\leq i\leq m.
\end{equation}
Since the frame is orthonormal, consider $\delta_i=g_0(e_i,e_i)=\pm1$. Notice that from the above definitions,
\begin{equation}\label{eq:jji}
J(t)=\sum_{i=1}^m \delta_iJ_i(t)e_i(t).
\end{equation}
Using that the frame is parallel, it is also possible to express the covariant derivative of $J$ in terms of the ordinary derivatives of the coordinate functions $J_i$. More precisely, for almost every $t\in [0,1]$,
\begin{equation*}
\D^{g_0} J(t)=\sum_{i=1}^m \delta_iJ'_i(t)e_i(t),
\end{equation*}
where $J'_i\in L^2([0,1],\R)$ is the (almost everywhere defined) ordinary derivative of $J_i:[0,1]\to\R$. In addition, define for all $t\in [0,1]$,
\begin{equation}
\alpha_i(t)=J'_i(t)-b_i(t),
\end{equation}
and $\alpha=(\alpha_i)_{i=1}^m\in L^2([0,1],\R^m)$.

Let $v\in T_{\gamma_0}H^1([0,1],M)$, and decompose it with respect to the same frame,
\begin{equation*}
v(t)=\sum_{i=1}^m \lambda_i(t)e_i(t), \quad t\in [0,1],
\end{equation*}
where $\lambda_i:[0,1]\to\R$. Since $\frac{\partial^2 E}{\partial\gamma^2}(g_0,\gamma_0)(J,v)=0$ for all $v\in T_{\gamma_0}H^1([0,1],M)$, in particular this holds for $v$'s such that for all $1\leq i\leq m$, $\lambda_i\in C^\infty_c(\,]0,1[,\R)$, for these $v$'s are clearly in $H^1([0,1],\R)$. 

Denote $\lambda=(\lambda_i)_{i=1}^m\in C^\infty_c(\,]0,1[,\R^m)$ and notice that, since the frame is parallel,
\begin{equation*}
\D^{g_0} v(t)=\sum_{i=1}^m \lambda'_i(t)e_i(t), \quad t\in\, ]0,1[,
\end{equation*}
and $\lambda'=(\lambda'_i)_{i=1}^m\in C^\infty_c(\,]0,1[,\R^m)$. We may then compute
{\allowdisplaybreaks
\begin{eqnarray*}
\frac{\partial^2 E}{\partial \gamma^2}(g_0,\gamma_0)(J,v) &\stackrel{\eqref{eq:indexform}}{=}& \int_0^1 g_0(\D^{g_0}J,\D^{g_0}v)-g_0(R^{g_0}(\dot{\gamma_0},J)v,\dot{\gamma_0})\;\dd t \\
&&\hspace{0.5cm}-\s^\p_{(\dot{\gamma_0}(0),\dot{\gamma_0}(1))}\Big((J(0),J(1)),(v(0),v(1))\Big)\\
&=& \int_0^1 g_0(\D^{g_0}J,\D^{g_0}v)+g_0(R^{g_0}(\dot{\gamma_0},J)\dot{\gamma_0},v)\;\dd t \\
&=& \int_0^1\sum_{i=1}^m g_0(\D^{g_0}J,\lambda'_ie_i)+g_0(R^{g_0}(\dot{\gamma_0},J)\dot{\gamma_0},\lambda_ie_i)\;\dd t \\
&=& \int_0^1\sum_{i=1}^m \lambda'_ig_0(\D^{g_0}J,e_i)+\lambda_ig_0(R^{g_0}(\dot{\gamma_0},J)\dot{\gamma_0},e_i)\;\dd t \\
&=& \int_0^1\sum_{i=1}^m \lambda'_iJ'_i+\lambda_iR_i\;\dd t \\
&=& \int_0^1\sum_{i=1}^m \lambda'_iJ'_i+\lambda_ib'_i\;\dd t \\
&=& \int_0^1\sum_{i=1}^m \lambda'_iJ'_i-\lambda'_ib_i\;\dd t \\
&=& \int_0^1\sum_{i=1}^m \lambda'_i(J'_i-b_i)\;\dd t \\
&=& \int_0^1\sum_{i=1}^m \alpha_i\lambda'_i\;\dd t \\
&=& \int_0^1\langle\alpha,\lambda'\rangle\;\dd t.
\end{eqnarray*}}Since $J\in\ker\frac{\partial^2 E}{\partial \gamma^2}(g_0,\gamma_0)$, the above expression vanishes for all $v$, hence for all $\lambda\in C^\infty_c(\,]0,1[,\R^m)$. From Lemma~\ref{le:distder0}, it follows that $\alpha$ is constant almost everywhere. This means that there exists $a=(a_i)_{i=1}^m\in\R^m$ such that $$\alpha_i=J'_i(t)-b_i(t)=a_i, \quad 1\leq i\leq m,$$ for almost every $t\in [0,1]$. Then, 
\begin{equation}\label{eq:jiprime}
J'_i(t)=a_i+b_i(t), \quad 1\leq i\leq m,
\end{equation}
for almost all $t\in [0,1]$. Notice that the right--hand side of \eqref{eq:jiprime} is continuous and, since $J_i$ is of Sobolev class $H^1$, it is absolutely continuous. From Lemma~\ref{le:milagre}, it follows that $J_i$ are of class $C^1$ and the above equality holds for every $t\in [0,1]$. The same argument applies to \eqref{eq:jji}, in that it holds almost everywhere, its right--hand side is continuous and given as the covariant derivative of an absolutely continuous vector field. Thus, from Corollary~\ref{cor:milagre}, it follows that $J$ is of class $C^1$ and \eqref{eq:jji} holds for all $t\in [0,1]$.

Therefore, since $J$ and $J_i$ are of class $C^1$, it follows that the frame $\{e_i(t)\}_{i=1}^m$ is also of class $C^1$. Then, from \eqref{eq:jji}, we have that $\D^{g_0} J$ is of class $C^1$, hence $J$ is of class $C^2$. This gives the necessary regularity to proceed.

Finally, since $J\in\sect^2(\gamma_0^*TM)$, we may compute for $v\in T_{\gamma_0}\op(M)$ such that $v(0)=0$ and $v(1)=0$,
{\allowdisplaybreaks
\begin{eqnarray*}
\frac{\partial^2 E}{\partial \gamma^2}(g_0,\gamma_0)(J,v) &\stackrel{\eqref{eq:indexform}}{=}&\int_0^1 g_0(\D^{g_0}J,\D^{g_0}v)-g_0(R^{g_0}(\dot{\gamma_0},J)v,\dot{\gamma_0})\;\dd t \\
&&\hspace{0.5cm}-\s^\p_{(\dot{\gamma_0}(0),\dot{\gamma_0}(1))}\Big((J(0),J(1)),(v(0),v(1))\Big)\\
&=& \int_0^1 g_0(-(\D^{g_0})^2J,v)+g_0(R^{g_0}(\dot{\gamma_0},J)\dot{\gamma_0},v)\;\dd t\\
&& \hspace{0.5cm}+g_0(\D^{g_0}J,v)\Big|_0^1 \\
&=&\int_0^1 g_0\Big(-(\D^{g_0})^2 J+R^{g_0}(\dot{\gamma_0},J)\dot{\gamma_0},v\Big)\;\dd t.
\end{eqnarray*}}Since the above expression vanishes for all such $v$'s, it follows that $J$ must satisfy the $g_0$--Jacobi equation along $\gamma_0$,
\begin{equation}\label{eq:g0jacobieq}
(\D^{g_0})^2 J=R^{g_0}(\dot{\gamma_0},J)\dot{\gamma_0}.
\end{equation}
In addition, for $v$'s that do not vanish at the endpoints, we have, applying identifications \eqref{ident:bilin},
{\allowdisplaybreaks
\begin{eqnarray*}
\frac{\partial^2 E}{\partial \gamma^2}(g_0,\gamma_0)(J,v) &\stackrel{\eqref{eq:indexform}}{=}&\int_0^1 g_0(\D^{g_0}J,\D^{g_0}v)-g_0(R^{g_0}(\dot{\gamma_0},J)v,\dot{\gamma_0})\;\dd t \\
&&\hspace{0.5cm}-\s^\p_{(\dot{\gamma_0}(0),\dot{\gamma_0}(1))}\Big((J(0),J(1)),(v(0),v(1))\Big)\\
&=& \int_0^1 g_0(-(\D^{g_0})^2J,v)+g_0(R^{g_0}(\dot{\gamma_0},J)\dot{\gamma_0},v)\;\dd t\\
&& \hspace{0.5cm}+g_0(\D^{g_0}J,v)\Big|_0^1 \\
&&\hspace{0.5cm}-\s^\p_{(\dot{\gamma_0}(0),\dot{\gamma_0}(1))}\Big((J(0),J(1)),(v(0),v(1))\Big)\\
&=&\int_0^1 g_0\Big(-(\D^{g_0})^2 J+R^{g_0}(\dot{\gamma_0},J)\dot{\gamma_0},v\Big)\;\dd t. \\
&& \hspace{0.5cm}+g_0(\D^{g_0}J(1),v(1))-g_0(\D^{g_0}J(0),v(0)) \\
&&\hspace{0.5cm}-\overline{g_0}\Big(\s^\p_{(\dot{\gamma_0}(0),\dot{\gamma_0}(1))}\Big(J(0),J(1)\Big),(v(0),v(1))\Big)\\
&\stackrel{\eqref{eq:g0jacobieq}}{=}&-\overline{g_0}\Big((\D^{g_0}J(0),\D^{g_0}J(1)),(v(0),v(1))\Big).\\
&&\hspace{0.5cm}-\overline{g_0}\Big(\s^\p_{(\dot{\gamma_0}(0),\dot{\gamma_0}(1))}\Big(J(0),J(1)\Big),(v(0),v(1))\Big).
\end{eqnarray*}}Since the above expression vanishes for all such $v$'s, it follows that $J$ must also satisfy \eqref{eq:prepcampodijacobi}, concluding the proof.
\end{proof}

\begin{definition}\label{def:pjacobi}
Let $\gamma$ be a $(g,\p)$--geodesic. A vector field $J\in\sect^{H^1}(\gamma_0^*TM)$ is called a {\em $\p$--Jacobi field}\index{$\p$--Jacobi field}\index{Jacobi field!$\p$--Jacobi field} along $\gamma$ with respect to $g$ if $J$ is in the kernel of the index form of $E_g$, i.e., if $J\in\ker\frac{\partial^2 E}{\partial\gamma^2}(g,\gamma)$.
From Proposition~\ref{prop:pjacobi}, this is equivalent to $J$ satisfying
\begin{itemize}
\item[(i)] the $g$--Jacobi equation, i.e., $$(\D^{g_0})^2 J=R^{g_0}(\dot{\gamma_0},J)\dot{\gamma_0};$$
\item[(ii)] the {\em linearized endpoints condition}\index{GEC!linearized} associated to $\p$ at $(g,\gamma)$, given by
\begin{equation}\label{eq:pcampodijacobi}
(\D^{g} J(0),\D^{g} J(1))+\s^\p_{(\dot{\gamma}(0),\dot{\gamma}(1))}\Big(J(0),J(1)\Big)\in T_{(\gamma(0),\gamma(1))}\p^\perp,
\end{equation}
where $^\perp$ denotes orthogonality with respect to $\overline{g}$.
\end{itemize}
Finally, when $g_0$ and $\gamma_0$ are evident from the context, we will simply refer to $J$ as $\p$--Jacobi field.
\end{definition}

\begin{remark}\label{re:jck}
From Corollary~\ref{cor:jacobick}, since $g\in\met^k_\nu(M)$, if $J$ is a $g$--Jacobi field along a $g$--geodesic $\gamma$, then $J$ is of class $C^{k}$. In particular, $\p$--Jacobi fields are $C^{k}$.
\end{remark}

\begin{example}\label{ex:jac}
Consider $\gamma$ a $(g,\p)$--geodesic where $\p$ is one of the GECs given in Example~\ref{ex:gecs}, see also Example~\ref{ex:gecs2}. If $\p=\{p\}\times\{q\}$, $\gamma$ simply joins $p$ and $q$. According to expected, since the tangent space to $\p$ is trivial, $T_{(p,q)}\p=\{0\}$, the $\p$--Jacobi fields are $g$--Jacobi fields along $\gamma$ that vanish at its endpoints. In case $P$ and $Q$ are submanifolds of $M$ and $\p=P\times Q$, $\gamma$ is $g$--orthogonal to $P$ and $Q$ at its endpoints. In addition, we may compute the second fundamental form\footnote{Recall Definition~\ref{def:sfform}} of $\p$ as $$\s_{(\dot{\gamma}(0),\dot{\gamma}(1))}^{P\times Q}=\s_{\dot{\gamma}(0)}^P-\s_{\dot{\gamma}(1)}^Q,$$ considering the metrics induced by $\overline{g}$. Thus, using \eqref{eq:prepcampodijacobi} it is easy to see that $\p$--Jacobi fields are $g$--Jacobi fields along $\gamma$ that satisfy $J(0)\in T_{\gamma(0)}P$, $J(1)\in T_{\gamma(1)}Q$ and
\begin{equation*}
\begin{aligned}
\D^g J(0)+\s_{\dot\gamma(0)}^P(J(0)) &\in T_{\gamma(0)}P^\perp \\
\D^g J(1)+\s_{\dot\gamma(1)}^Q(J(1)) &\in T_{\gamma(1)}Q^\perp,
\end{aligned}
\end{equation*}
where $^\perp$ is orthogonality with respect to the metrics on $P$ and $Q$ induced by $g$.

As for the special case $\p=\Delta$, where $\overline g=g\oplus (-g)$ vanishes identically, it is possible to adapt the above computations to obtain the following. A $(g,\p)$--geodesic, as observed in Example~\ref{ex:gecs2} is a periodic $g$--geodesic, see Example~\ref{ex:periodicgeod}. In addition, the condition for a $g$--Jacobi field $J$ along $\gamma$ to be a $\p$--Jacobi field is simply its periodicity, i.e., $J(0)=J(1)$ and also $\D^g J(0)=\D^g J(1)$.

Geometrically, existence of a nontrivial $\p$--Jacobi field in the previous cases can be interpreted as follows. In the first case, it simply means that $p$ and $q$ are conjugate along $\gamma$, see Definition~\ref{def:conjugate}. In the second, if $Q=\{q\}$ is a point, it means that $q$ is a focal point of $P$, see Definition~\ref{def:focalpoint}. Finally, for $\p=P\times Q$, existence of a nontrivial $\p$--Jacobi field is equivalent to focality of $P$ and $Q$, see Definition~\ref{def:focalsubmanifolds}.
\end{example}

We end this section with a few last results and remarks on $\p$--Jacobi fields, that will be important to determine the genericity of the parameters $g$ for which $E_g$ has only nondegenerate critical points, i.e., $(g,\p)$--geodesics that do not admit any nontrivial $\p$--Jacobi field.

\begin{proposition}\label{prop:tangentaintpjacobi}
Consider $\gamma\in\op(M)$ a nonconstant $(g,\p)$--geodesic. Although the tangent field $\dot\gamma$ is a $g$--Jacobi field along $\gamma$, it is \emph{not} a \emph{$\p$--Jacobi field} along $\gamma$.
\end{proposition}

\begin{proof}
Recall that $J=\dot\gamma$ is trivially a solution of the $g$--Jacobi equation. From \eqref{eq:tgammaopm}, all $\p$--Jacobi fields $J$ along $\gamma$ must be tangent to $\p$ at  $(\gamma(0),\gamma(1))$, since $J\in T_\gamma\op(M)$. Nevertheless, from Definition~\ref{def:gpgeod}, $\dot\gamma$ is $\overline{g}$--orthogonal to $\p$ at $(\gamma(0),\gamma(1))$. Since we are assuming that $\overline g$ does not degenerate on $\p$, it follows that $\dot\gamma$ is not a $\p$--Jacobi field, unless $\gamma$ is constant.
\end{proof}

\begin{remark}
Notice that this observation includes the case of geodesics loops, which may be $(g,\p)$--geodesics if $\p\cap\Delta\neq\emptyset$. In addition, it also covers the possibility $\p=\{p\}\times\{q\}$, even if $p=q$. In such case, the tangent space $T_{(p,q)}\p$ is trivial, hence all $\p$--Jacobi fields $J$ along $\gamma$ have to satisfy $J(0)=0$ and $J(1)=0$. Therefore, $\dot\gamma$ is not a $\p$--Jacobi field once more.

Nevertheless, the same does not hold for $\p=\Delta$, since $\overline g$ degenerates. In fact, the tangent field to a periodic geodesic is periodic, and this means that in this special case, the tangent field is a $\Delta$--Jacobi field. This is one of the main reasons it is necessary to treat this case separately in our applications, since nondegeneracy of $\overline g$ will be a necessary and constant assumption.
\end{remark}

\begin{remark}\label{re:degeneracynotions}
Suppose $\p\cap\Delta\neq\emptyset$ and let $\gamma$ be a periodic $g$--geodesic that is also a $(g,\p)$--geodesic. As a consequence of Proposition~\ref{prop:tangentaintpjacobi}, the notions of degeneracy of $\gamma$ differ when it is considered as a \emph{periodic geodesic} and as a \emph{$(g,\p)$--geodesic}. More precisely, the tangent field $\dot\gamma$ is always a Jacobi field along $\gamma$, therefore $\gamma$ is always a degenerate critical point of the $g$--energy functional. Such degeneracy is caused by the reparameterization action of the circle $S^1$ on $H^1(S^1,M)$, studied in Section~\ref{sec:actions}. In fact, as we will see in Section~\ref{sec:s1invariance}, there is a more precise equivariant concept of nondegeneracy that is adequate in this case. In this sense, $\gamma$ will be degenerate if there are non trivial periodic Jacobi fields along it, that are {\em not constant multiples} of $\dot\gamma$, see Remark~\ref{re:invkernel} and Definition~\ref{def:gmorse}.

Furthermore, since from Proposition~\ref{prop:tangentaintpjacobi} the tangent field $\dot\gamma$ is not a $\p$--Jacobi field along $\gamma$, it follows that if $\gamma$ is nondegenerate {\em as a periodic geodesic}, then it is also nondegenerate {\em as a $(g,\p)$--geodesic}. However, the converse is not true, since $\gamma$ may admit a Jacobi field which is not a constant multiple of $\dot\gamma$, neither a $\p$--Jacobi field.
\end{remark}

Let $\gamma$ be a $(g,\p)$--geodesic. Not only the tangent field $\dot\gamma$ is not a $\p$--Jacobi field (see Proposition~\ref{prop:tangentaintpjacobi}), but also $\p$--Jacobi fields along $\gamma$ are only parallel to $\dot\gamma$ at a finite number of points. Such claim is a consequence of Lemma~\ref{le:parallelfinite} combined with the following result.

\begin{lemma}\label{le:jacobinotparallel}
Let $\gamma:[0,1]\rightarrow M$ be a $(g,\p)$--geodesic. If $J$ is a nontrivial $\p$--Jacobi field along $\gamma$, then it is not everywhere parallel to $\dot{\gamma}$.
\end{lemma}

\begin{proof}
First, let us consider the trivial case when $J$ does not vanish at the endpoints of $\gamma$. Since $J$ is a $\p$--Jacobi field, from \eqref{eq:tgammaopm}, $(J(0),J(1))\in T_{(\gamma(0),\gamma(1))}\p$. Hence $J(0)$ and $J(1)$ are not respectively parallel to $\dot{\gamma}(0)$ and $\dot{\gamma}(1)$, because they are not trivial and $(\dot{\gamma}(0),\dot{\gamma}(1))\in T_{(\gamma(0),\gamma(1))}\p^\perp$.

If $J(0)=0$ and $J(1)=0$, the argument is modified as follows. Suppose that there exists $\lambda:[0,1]\rightarrow M$ such that $J(t)=\lambda(t)\dot{\gamma}(t)$. Since $J$ is a solution of the $g$--Jacobi equation \eqref{eq:jacobi}, $\lambda$ must be an affine function, that is, $\lambda(t)=c_1+c_2t$ for some $c_1,c_2\in\R$. Using that $J(0)=0$ and $J(1)=0$, it follows that $\lambda(0)=\lambda(1)=0$, which implies that $J$ is the trivial solution.
\end{proof}

\begin{remark}\label{re:whyg}
At this point, the reader may have already recognized the naturality of the choice $\overline{g}=g\oplus (-g)$ for the ambient metric on $M\times M$ instead of any other. First, it appears naturally on expressions such as \eqref{eq:firstvariation}, \eqref{eq:indexform} and \eqref{eq:prepcampodijacobi}, directly or in the form of the identification \eqref{ident:bilin} to express the second fundamental form $\s^\p$. Second, the crucial reason is that if the induced metric in $\p$ was different, it would be possible that the tangent field $\dot\gamma$ was a $\p$--Jacobi field, see Proposition~\ref{prop:tangentaintpjacobi}. Furthermore, notice that the nondegeneracy of $\p$ with respect to this $\overline{g}$ is essential in the proof of Lemma~\ref{le:jacobinotparallel}.
\end{remark}

\begin{corollary}\label{cor:pjacobiparallelfinite}
Let $\gamma:[0,1]\rightarrow M$ be a $(g,\p)$--geodesic. If $J$ is a nontrivial $\p$--Jacobi field along $\gamma$, then the following set is finite, $$\{t\in [0,1]: J(t) \mbox{ is parallel to } \dot{\gamma}(t)\}.$$
\end{corollary}

\begin{proof}
From Lemma~\ref{le:jacobinotparallel}, $J$ is not everywhere parallel to $\dot\gamma$. The conclusion then follows from Lemma~\ref{le:parallelfinite}.
\end{proof}

\section{\texorpdfstring{Periodic geodesics and $S^1$--invariance}{Periodic geodesics and S1--invariance}}
\label{sec:s1invariance}

Geometric variational problems are often invariant under the action of a group, i.e., the related functional is constant on the orbits. For instance, let us analyze the two geodesic functionals \eqref{eq:grenergylength},
\begin{equation*}
L_\mathrm R(\gamma)=\int_0^1\sqrt {g_{\mathrm R}(\dot\gamma,\dot\gamma)}\;\dd t\;\; \mbox{ and } \;\; E_{\mathrm R}(\gamma)=\tfrac12\int_0^1 g_\mathrm R(\dot\gamma,\dot\gamma)\;\dd t
\end{equation*}
mentioned in the beginning of this chapter. Suppose their domain to be $H^1([0,1],M)$. Then the length functional $L_\mathrm R$ is invariant under reparameterizations, i.e., {\em diffeomorphisms} of $[0,1]$. Here, the action of the group of diffeomorphisms\footnote{This is actually not a Lie group, but only a {\em topological group}, whose action on $H^1([0,1],M)$ is only continuous. Therefore, most tools that will be developed do not apply to this case. Nevertheless, as remarked in the beginning of the chapter, there is a clear geometric relation between critical points of $L_\mathrm R$ and $E_\mathrm R$. Thus, all the analysis for geodesic variational problems can be done using $E_\mathrm R$, which in addition allows to consider semi--Riemannian metrics instead of only Riemannian metrics.} $\diff([0,1])$ on $H^1([0,1],M)$ is given by right composition. On the other hand, the energy functional $E_\mathrm R$ is only invariant under {\em isometries} of $[0,1]$, also acting by right composition. Since there are endpoints conditions involved, we only consider reparameterizations on $[0,1]$ that preserve the orientation of curves, i.e., have positive derivative. Thus, the group of possible reparameterizations for $L_\mathrm R$ is $f\in\diff([0,1])$ such that $f'>0$ and for $E_\mathrm R$ it is reduced to the trivial group. Among other reasons, this indicates that the analysis of critical points for $E_\mathrm R$ is easier then the correspondent for $L_\mathrm R$, since it is not invariant under any group actions. This is precisely because critical points of $E_\mathrm R$ are {\em affinely parameterized} geodesics and there are no other possible reparameterizations, as in the case of critical points of $L_\mathrm R$.

Nevertheless, in the case of {\em periodic} curves, the domain of these functionals is the submanifold $H^1(S^1,M)$, see Corollary~\ref{cor:h1s1}. Here, {\em there are} nontrivial isometric reparameterizations of $S^1$ that leave $E_\mathrm R$ invariant, namely rotations of the domain $S^1$. The correspondent action $$\rho:S^1\times H^1(S^1,M)\la H^1(S^1,M)$$ is precisely the one given by \eqref{eq:mus1h1}, introduced in Section~\ref{sec:actions}. Therefore, in this case we are dealing with a $S^1$--invariant functional, hence with an equivariant variational problem, see Example~\ref{ex:mus1h1inv}.

We will first give a brief abstract introduction to $G$--invariant functionals in the following context. We assume $G$ is a finite--dimensional Lie group, $Y$ is a Hilbert manifold, $$\mu:G\times Y\la Y$$ is a differentiable action and $f:Y\to\R$ is a $C^k$ functional invariant under this action. Second, we explore the above example of the energy functional for periodic curves $E_g:H^1(S^1,M)\to\R$, developing asome special tools to deal with its the lack of regularity.

\begin{definition}\label{def:finv}
A $C^k$ functional $f:Y\to\R$ is {\em $G$--invariant}\index{$G$--invariant functional}\index{Action!invariant functional} if it is constant along the orbits of $G$, i.e., for all $y\in Y$ and $g\in G$, 
\begin{equation}\label{eq:finv}
f(\mu(g,y))=f(y).
\end{equation}
\end{definition}

Invariance of a functional under $G$ means that it would be essentially possible to define this functional {\em modulo $G$}.\footnote{As we will see, this means that $f$ is constant along each orbit. This suggests dividing out by $G$, i.e., considering $f$ defined in the orbit space $Y/G$ instead of $Y$.} For instance, consider $M$ and $N$ Riemannian manifolds, with $M$ compact, and the action of $\Iso(N)$ on $H^k(M,N)$ described in Example~\ref{ex:hkmn}. The volume functional of an embedding $$\Vol(\Phi)=\int_M\sqrt{\det\big(\dd\Phi(x)^*\dd\Phi(x)\big)}\;\dd\vol_M(x),$$ where $\vol_M(x)$ is the volume form of $M$, is clearly $\Iso(N)$--invariant, since isometries of the ambient space preserve the Riemannian structure of submanifolds. In other words, it is not important to consider any particular positioning of an embedded submanifold or rigid motions of the ambient to compute its volume. Let us mention another example that will become the center of our attention in the sequel.

\begin{example}\label{ex:mus1h1inv}
Consider $g$ is a semi--Riemannian metric on $M$ and the $g$--energy functional on closed curves,
\begin{eqnarray}\label{eq:efunctper}
E_g:H^1(S^1,M)&\la &\R\\
\gamma&\longmapsto&\tfrac12\int_{S^1} g(\dot\gamma(z),\dot\gamma(z))\;\dd z.\nonumber
\end{eqnarray}
As remarked above, this is a $S^1$--invariant functional considering the action \eqref{eq:mus1h1} described in Example~\ref{ex:s1h1}, $$\rho:S^1\times H^1(S^1,M)\la H^1(S^1,M).$$ In fact, for each $w\in S^1$,
{\allowdisplaybreaks
\begin{eqnarray*}
E_g(\rho^z(\gamma)) &=& \tfrac12\int_{S^1} g(\dd\rho^z(\gamma)\dot\gamma(w),\dd\rho^z(\gamma)\dot\gamma(w))\;\dd w\\
&\stackrel{\eqref{eq:dmuz}}{=}& \tfrac12\int_{S^1} g(\dot\gamma(zw),\dot\gamma(zw))\;\dd w\\
&=& \tfrac12\int_{S^1} g(\dot\gamma(\zeta),\dot\gamma(\zeta))\;\dd\zeta \\
&=& E_g(\gamma),
\end{eqnarray*}}where the third equality holds by a simple change of variables $\zeta=zw$. Hence $E_g$ is $S^1$--invariant.

Analogously to Propositions~\ref{prop:fck},~\ref{prop:critgenenfunc},~\ref{prop:indexform},~\ref{prop:fredholmness} and~\ref{prop:mixedderivative}, the functional $E_g$ is $C^k$ and its critical points are periodic $g$--geodesics. Moreover, if $\gamma_0\in H^1(S^1,M)$ is a critical point of $E_g$, then the second derivative of $E_g$ is a continuous bilinear symmetric form on $T_{\gamma_0} H^1(S^1,M)$ that is represented by a Fredholm operator of this Hilbert space, and the elements in its kernel are periodic Jacobi fields along $\gamma_0$. Finally, formula \eqref{eq:mixedderivative} also holds for \eqref{eq:efunctper}.
\end{example}

In order to better describe the behaviour of a $G$--invariant functional, we use some objects related to the action $\mu:G\times Y\to Y$, described in Section~\ref{sec:actions}. Recall that $$\mathcal D_y=\im\dd\mu_y(1)$$ gives a subspace of $T_yY$ for all $y\in Y$ if $\mu$ is differentiable. Also under weaker regularity assumptions, such as existence of a $G$--invariant dense subset $Y_1\subset Y$ described in Remark~\ref{re:weakreg}, it is possible to consider this subspace of $T_yY$ for $y\in Y_1$.

\begin{lemma}\label{le:invkernel}
If $f:Y\to\R$ is a $G$--invariant functional, then $\mathcal D_y$ is containted in $\ker\dd f(y)$ for every $y\in Y$, or $y\in Y_1$ in case the action is non differentiable.\footnote{The context for non differentiable actions is the one established in Remark~\ref{re:weakreg}, recalled above.}
\end{lemma}

\begin{proof}
Since $f$ is $G$--invariant, it is constant along the orbits $G(y)$. The subspace $\mathcal D_y$ is always tangent to $G(y)$, hence is clearly in the kernel of $\dd f(y)$. More precisely, deriving \eqref{eq:finv} with respect to $g$ at $g=e$, we have $$\dd f(y)\dd\mu_y(e)=0,$$ which implies $\im\dd\mu_y(e)\subset\ker\dd f(y)$. Notice that if the action is non differentiable, then $\mathcal D_y$ can only be considered for points $y\in Y_1$. However, in this case $\dd\mu_y(e)$ can be computed at such points in $Y_1$.
\end{proof}

\begin{remark}\label{re:invkernel}
In particular, Lemma~\ref{le:invkernel} implies that a $G$--invariant functional $f$ is {\em not} a Morse functional in the sense of Definition~\ref{def:degnondegstdeg}, since $\ker\dd f(y)$ is never trivial. Nevertheless, it is possible to define a $G$--equivariant Morse condition, requiring $\ker\dd f(y)$ to be the smallest possible, see Definition~\ref{def:gmorse}.
\end{remark}

\begin{remark}
If a $C^k$ functional $f:Y\to\R$ is $G$--invariant, then its derivative $\dd f(y):T_yY\to\R$ is {\em $G$--equivariant}.\index{$G$--equivariant map} Invariance of $f$ is expressed by \eqref{eq:finv}, and can also be seen as commutativity of the following diagram.
\begin{equation*}
\xymatrix@+10pt{
Y\ar[rr]^{\mu^g}\ar[rd]_f & & Y\ar[ld]^f\\
& \R &
}
\end{equation*}
Deriving \eqref{eq:finv} at $y$, we obtain $G$--equivariance of $\dd f$, i.e.
\begin{equation}\label{eq:dfeq}
\dd f(\mu(g,y))\dd\mu^g(y)=\dd f(y),
\end{equation}
or in the form of a commutative diagram,
\begin{equation*}
\xymatrix@+8pt{
T_yY\ar[rr]^{\dd\mu^g(y)}\ar[rd]_{\dd f(y)} & & T_{\mu(g,y)}Y\ar[ld]^{\dd f(\mu(g,y))}\\
& \R &
}
\end{equation*}
\end{remark}

\begin{lemma}\label{le:critorb}
Suppose the action of $G$ on $Y$ is by diffeomorphisms and let $f:Y\to\R$ be a $G$--invariant $C^k$ functional. If $y_0\in Y$ is a critical point of $f$, then the whole orbit $G(y_0)$ is critical.
\end{lemma}

\begin{proof}
Since the action is by diffeomorphisms, for all $g\in G$, the map $\mu^g:Y\to Y$ is a diffeomorphism. In particular, its derivative $\dd\mu^g(y)$ is an isomorphism for all $y\in Y$. From $G$--equivariance \eqref{eq:dfeq},
\begin{eqnarray*}
\dd f(y) &=& \dd(f\circ\mu^g)(y)\\
&=& \dd f(\mu^g(y))\dd\mu^g(y).
\end{eqnarray*}
If $y_0$ is a critical point of $f$, then the above expression vanishes. However, since $\dd\mu^g(y_0)$ is an isomorphism, it follows that $\dd f(\mu^g(y_0))=0$, i.e., $\mu^g(y_0)$ is a critical point of $f$. Therefore, $\dd f(y)$ vanishes for all $y=\mu(g,y_0)$, i.e., for all $y\in G(y_0)$.
\end{proof}

We now define the {\em $G$--Morse} condition for $G$--invariant functionals. Recall that from Lemma~\ref{le:invkernel}, the kernel $\ker\dd f(y)$ of a $G$--invariant functional is necessarily nontrivial since it contains $\mathcal D_y$, hence $G$--invariant functionals are {\em never} Morse functionals in the sense of Definition~\ref{def:degnondegstdeg}. This concept provides a $G$--equivarant version of the Morse condition, that requires the kernel of $\dd f(y)$ to be {\em at most} $\mathcal D_y$, i.e., the smallest possible subspace of $T_yY$. Again, in case $\mathcal D_y$ is only defined for $y\in Y_1$, we suppose $\crit(f)\subset Y_1$, i.e., that all critical points of $f$ are in the dense subset $Y_1$, where $\mathcal D$ is defined. Notice that this is the case of the $g$--energy functional $E_g$ for periodic curves \eqref{eq:efunctper} with respect to a $C^k$ semi--Riemannian metric $g$, since analogously to Remark~\ref{re:gammack}, periodic geodesics $\gamma$ are automatically $C^{k+1}$, and in particular are in $Y_1=H^2(S^1,M)$, from Proposition~\ref{prop:inclusions}. Moreover, from \eqref{eq:dgammaspangamma}, the subspace $\mathcal D_\gamma$ coincides with the one--dimensional subspace of $\sect^{H^1}(\gamma^*TM)$ spanned by $\dot\gamma$.

\begin{definition}\label{def:gmorse}
Let $f$ be a $G$--invariant functional whose critical points are in the $G$--invariant dense subset $Y_1$, where $\mathcal D$ is well--defined. Then $f$ is {\em $G$--nondegenerate}\index{Critical point!$G$--nondegenerate} at a critical point $y_0$ if $\hess(f)(y_0)$ restricted to a closed complement of $\mathcal D_{y_0}$ is an isomorphism. If all critical points of $f$ are $G$--nondegenerate, then $F$ is said to be {\em $G$--Morse}.\index{$G$--Morse functional}
\end{definition}

\begin{remark}
The above definition is clearly an equivariant extension of Definition~\ref{def:degnondegstdeg}, which corresponds to the case $G=\{e\}$ acting trivially on $Y$. In this case, $\mathcal D_y$ is also trivial, hence both definitions coincide.
\end{remark}

\begin{remark}
It is easily seen that the above definition of $G$--Morse functional does not depend on the choice of a closed complement of $\mathcal D_y$.
\end{remark}

\begin{lemma}\label{le:gmorseorbit}
Suppose the action of $G$ on $Y$ is by diffeomorphisms and let $f$ be a $G$--invariant functional and $y_0\in Y$ a critical point of $f$. Then $y_0$ is $G$--nondegenerate if and only if every $y\in G(y_0)$ is also $G$--nondegenerate.
\end{lemma}

\begin{proof}
Let $g\in G$ and consider $y=\mu(g,y_0)$. From Lemma~\ref{le:critorb}, $y$ is a critical point of $f$, and from $G$--equivariance of $\dd f$ it is easy to see that the following diagram commutes.
\begin{equation*}
\xymatrix@+25pt{
T_{y_0}Y \ar[d]_{\dd\mu^g(y_0)}\ar[r]^(.4){\hess(f)(y_0)} & T_{y_0}Y^*\cong T_{y_0}Y\ar[d]^{\dd\mu^g(y_0)} \\
T_{y}Y \ar[r]_(.4){\hess(f)(y)} & T_{y}Y^*\cong T_{y_0}Y
}
\end{equation*}
Suppose $y_0$ is a $G$--nondegenerate critical point of $f$. Then the Hessian $\hess(f)(y_0)$ restricted to some closed complement of $\mathcal D_{y_0}$ is an isomorphism. Clearly, the isomorphism $\dd\mu^g(y_0)$ maps $\mathcal D_{y_0}$ to $\mathcal D_y$, and since it also maps the closed complement of $\mathcal D_{y_0}$ where $\hess(f)(y_0)$ is an isomorphism to some closed complement of $\mathcal D_y$ in $T_yY$, see Figure~\ref{fig:gnondeg}. Commutativity of the diagram above implies that $\hess(f)(y)$ restricted to this closed complement of $\mathcal D_y$ is an isomorphism. Hence $y$ is also $G$--nondegenerate, concluding the proof.
\end{proof}

\begin{figure}[htf]
\begin{center}
\vspace{-0.7cm}
\includegraphics[scale=1]{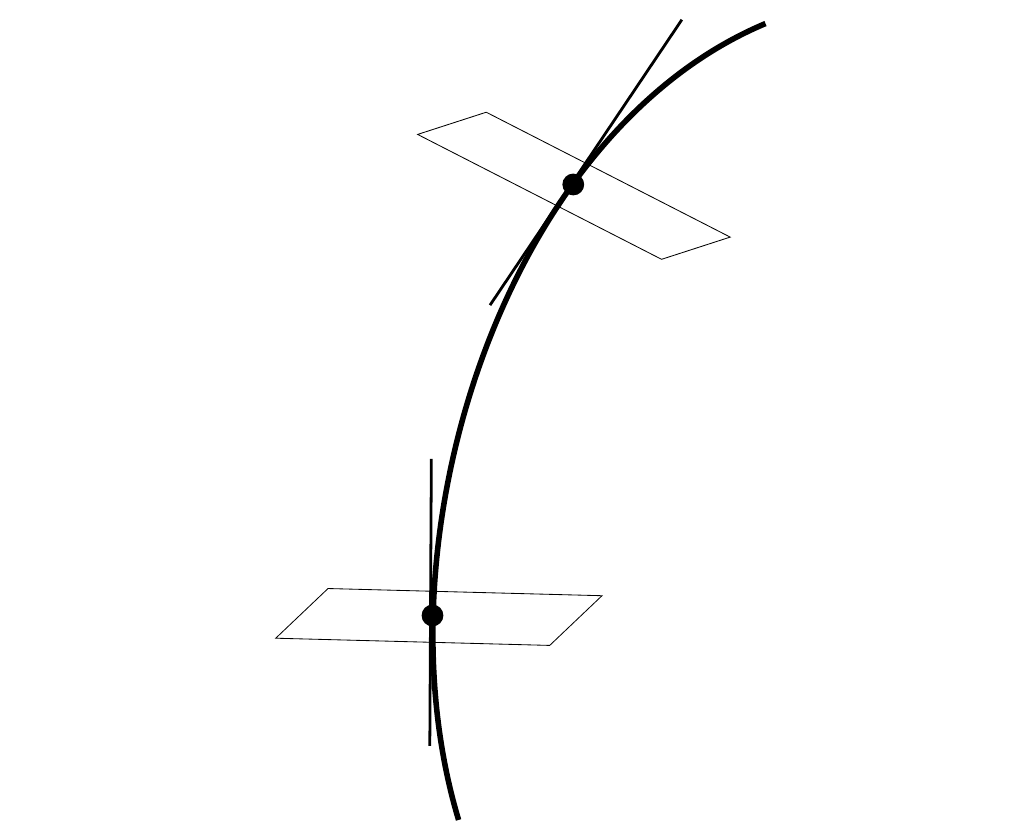}
\begin{pgfpicture}
\pgfputat{\pgfxy(-4.2,7.8)}{\pgfbox[center,center]{$\mathcal D_y$}}
\pgfputat{\pgfxy(-6.5,3.2)}{\pgfbox[center,center]{$\mathcal D_{y_0}$}}
\pgfputat{\pgfxy(-5,0)}{\pgfbox[center,center]{$G(y_0)$}}
\pgfputat{\pgfxy(-6.5,2.2)}{\pgfbox[center,center]{$y_0$}}
\pgfputat{\pgfxy(-4.8,6.8)}{\pgfbox[center,center]{$y$}}
\end{pgfpicture}
\caption{Points $y_0$ and $y$ in the same orbit, subspaces $\mathcal D_{y_0}$ and $\mathcal D_{y}$ and closed complements, where $\hess(f)$ is an isomorphism in case the orbit is formed by $G$--nondegenerate critical points of $f$.}\label{fig:gnondeg}
\end{center}
\end{figure}

We now establish the definition of {\em generalized slice} for an action with respect to a $G$--invariant functional. This is the key idea to analyze the $G$--Morse property for $G$--invariant functionals, since it will reduce the problem to analyzing the classic Morse property of $f$ restricted to transverse submanifolds to the orbits, see Proposition~\ref{prop:morsegmorse}.

\begin{definition}\label{def:genslice}
Suppose the action of $G$ on $Y$ is by diffeomorphisms and there exists a $G$--invariant dense subset $Y_1\subset Y$ as in Remark~\ref{re:weakreg}, such that $\mathcal D_y$ is well--defined for $y\in Y_1$. Let $f:Y\to\R$ be a $G$--invariant $C^k$ functional whose critical points are contained in $Y_1$. A {\em generalized slice}\index{Generalized slice}\footnote{Usually, the notion of {\em slice} of an action at a point $y$ is used to simplify the analysis of the behaviour of the orbit $G(y)$, by introducing a transverse submanifold with certain {\em $G$--invariance properties}. In addition, the image of a slice by $\mu|_G$ gives an open neighborhood of the orbit $G(y)$. By the analogy with this situation, we will call the family $\{S_n\}_{n\in\N}$ a generalized slice, although we stress it does not have any invariance under isotropy groups at all, as standard slices. For further details on slices we refer to Alexandrino and Bettiol \cite{ilgaag} for the finite--dimensional case, and to Palais and Terng \cite{PalaisTerng} for the infinite--dimensional case.} for the action of $G$ on $Y$ with respect to $f$ is a pair $(\mathfrak U,\{S_n\}_{n\in\N})$, where $\mathfrak U$ is an open subset of $Y$ that contains $Y_1$ and $\{S_n\}_{n\in\N}$ is a countable family of submanifolds of $Y$, satisfying
\begin{itemize}
\item[(i)] for all $y\in\mathfrak U$ there exists $n\in\N$ with $G(y)\cap S_n\neq\emptyset$;
\item[(ii)] for all $n\in\N$, given $y\in S_n$, if $y$ is a critical point of $f\vert_{S_n}$, then $y$ is a critical point of $f$ (in particular, $y\in Y_1$);
\item[(iii)] for all $n\in\N$, if $y\in S_n$ is a critical point of $f$ then $T_yY$ decomposes as direct sum $T_yY=T_yS_n\oplus\mathcal D_y$.
\end{itemize}
\end{definition}

\begin{figure}[htf]
\begin{center}
\vspace{-0.4cm}
\includegraphics[scale=1]{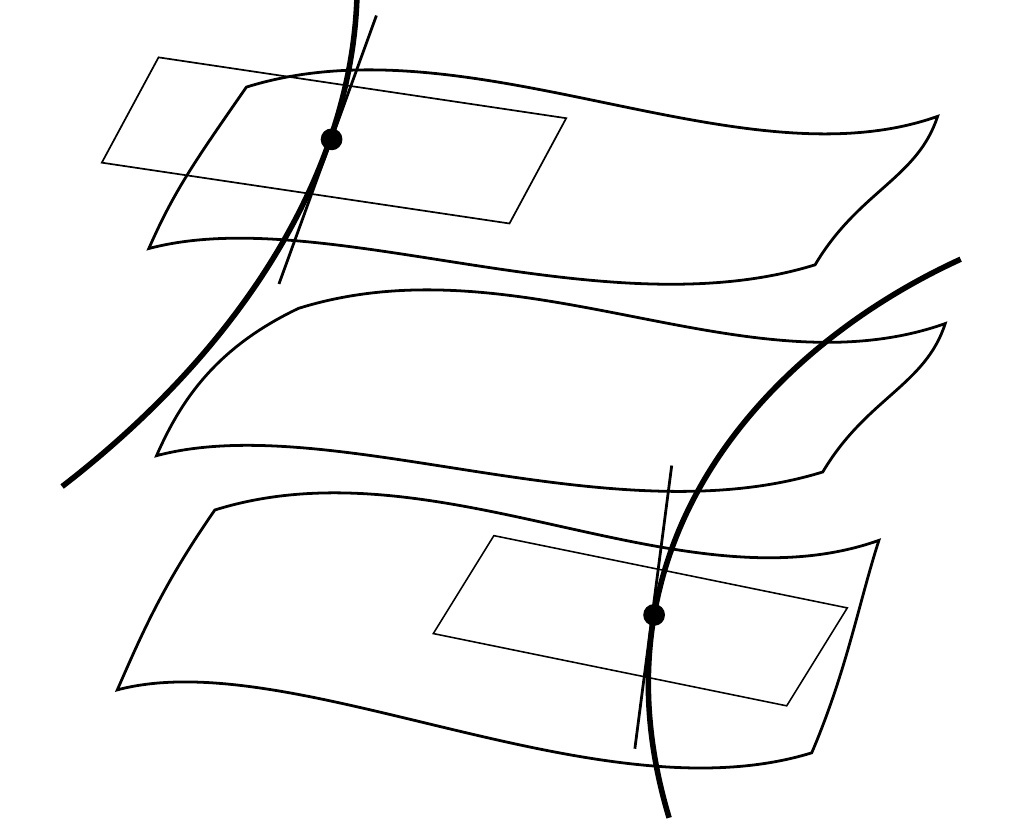}
\begin{pgfpicture}
\pgfputat{\pgfxy(-4,3.8)}{\pgfbox[center,center]{$\mathcal D_y$}}
\pgfputat{\pgfxy(-4.2,0.2)}{\pgfbox[center,center]{$G(y)$}}
\pgfputat{\pgfxy(-3.4,2)}{\pgfbox[center,center]{$y$}}
\pgfputat{\pgfxy(-0.9,2.3)}{\pgfbox[center,center]{$S_n$}}
\pgfputat{\pgfxy(-0.5,4.4)}{\pgfbox[center,center]{$S_{n+1}$}}
\pgfputat{\pgfxy(-0.5,6.6)}{\pgfbox[center,center]{$S_{n+2}$}}
\pgfputat{\pgfxy(-4.8,2.5)}{\pgfbox[center,center]{$T_yS_n$}}
\end{pgfpicture}
\vspace{0.3cm}
\caption{Some elements of a generalized slice $(\mathfrak U,\{S_n\}_{n\in\N})$ and transverse intersection with critical orbits of $f$.}\label{fig:genslice}
\end{center}
\end{figure}

\begin{proposition}\label{prop:morsegmorse}
In the above conditions, suppose there exists a generalized slice $(\mathfrak U,\{S_n\}_{n\in\N})$ for the action of $G$ on $Y$ with respect to $f$. Then $f$ is $G$--Morse if and only if $f\vert_{S_n}$ is Morse\footnote{In the sense of Definition~\ref{def:degnondegstdeg}.} for all $n\in\N$.
\end{proposition}

\begin{proof}
First, notice that $y_0\in S_n$ is a critical point of $f$ if and only if it is a critical point of the restriction $f\vert_{S_n}$. This is immediate from Definition~\ref{def:genslice}, since $T_{y_0}S_n$ and $\mathcal D_{y_0}$ are complementary subspaces of $T_{y_0}Y$ and from Lemma~\ref{le:invkernel}, $\mathcal D_{y_0}\subset\ker\dd f(y_0)$. Second, $\hess(f\vert_{S_n})(y_0)$ is the restriction to $T_{y_0}S_n\times T_{y_0}S_n$ of $\hess(f)(y_0)$, see Definition~\ref{def:hess}. Thus, if $f$ is $G$--Morse and $y_0\in S_n$ is a critical point of $f$, then $y_0$ is $G$--nondegenerate, and hence the restriction of $\hess(f)(y_0)$ to the (closed) complement $T_{y_0}S_n$, given by $\hess(f\vert_{S_n})(y_0)$, is an isomorphism. Therefore, $f\vert_{S_n}$ is Morse.

Conversely, suppose $f\vert_{S_n}$ is Morse for every $n\in\N$ and let $y_0$ be a critical point of $f$. From Lemma~\ref{le:gmorseorbit}, for $y_0$ to be $G$--nondegenerate, it suffices to have that one single point $y\in G(y_0)$ is $G$--nondegenerate, since this implies that the whole orbit $G(y_0)$ is formed by $G$--nondegenerate critical points. In addition, since $y_0\in Y_1\subset\mathfrak U$, from Definition~\ref{def:genslice}, there exists $n_0\in\N$ such that $G(y_0)\cap S_{n_0}\neq\emptyset$. Let $y\in G(y_0)\cap S_{n_0}$. Since $f\vert_{S_{n_0}}$ is Morse, its Hessian, that is given by the restriction of $\hess(f)(y)$ to $T_yS_{n_0}$, is an isomorphism. Moreover, $T_yS_{n_0}$ is a closed complement to $\mathcal D_y$, hence $y$, and also $y_0$, are $G$--nondegenerate. Therefore $f$ is $G$--Morse.
\end{proof}

We now establish the feasibility of condition (i) of Definition~\ref{def:genslice} for {\em differentiable} actions on {\em separable} Hilbert manifolds, with respect to any $G$--invariant functional $f$. In this stronger context, $\mathcal D_y$ is well--defined for every $y\in Y$, and $Y_1=Y$. Nevertheless, notice that the case of the $g$--energy functional for periodic curves does not fit this situation, since the action is not differentiable. In the sequel we will {\em manually} construct a generalized slice for the reparameterization action \eqref{eq:mus1h1}, which allows to use Proposition~\ref{prop:morsegmorse}.

\begin{lemma}\label{le:transvd}
Suppose the action of $G$ on $Y$ is differentiable and let $S\subset Y$ be a transverse submanifold\footnote{Recall Definition~\ref{def:transversed}.} to $\mathcal D$ at some $y\in S$. Then the following hold.
\begin{itemize}
\item[(i)] there exists an open submanifold $S'\subset S$ containing $y$ which is transverse to $\mathcal D$;
\item[(ii)] $\mu(G\times S)=\{\mu(g,s):g\in G, s\in S\}$ contains an open neighborhood of $y$ in $Y$.
\end{itemize}
\end{lemma}

\begin{proof}
A result similar to Lemma~\ref{le:transvopen} implies that the set of $y\in S$ such that $T_yY=T_yS\oplus\mathcal D_y$ is open, hence (i) holds.\footnote{See also Remark~\ref{re:transvopen} on the openness of the transversality condition to a fixed submanifold. Notice however that we are dealing here with transversality to the distribution $\mathcal D$, i.e., to all orbits, which are integral submanifolds, at the same time.}

As for (ii), consider the derivative of $\mu:G\times S\to Y$ at $(e,y)$, given by
\begin{eqnarray*}
\dd\mu(e,y):\mathfrak g\oplus T_yS &\la & T_yY=T_yS\oplus \mathcal D_y\\
(X,v) &\longmapsto &(\dd\mu_y(e)X,v).
\end{eqnarray*}
This is clearly an isomorphism, hence from the Inverse Function Theorem, there exist open neighborhoods of $(e,y)\in G\times S$ and $y\in Y$ where $\mu$ is a diffeomorphism. Therefore, the image $\mu(G\times S)$ contains an open neighborhood of $y$ in $Y$.
\end{proof}

\begin{proposition}\label{prop:genslice}
If the action of $G$ on $Y$ is differentiable and $Y$ is separable, then there exists a countable family $\{S_n\}_{n\in\N}$ of transverse submanifolds to $\mathcal D$ such that every orbit intercepts some $S_n$.
\end{proposition}

\begin{proof}
Let $\mathfrak U=Y$. From (i) in Lemma~\ref{le:transvd}, every $y\in Y$ is contained in a submanifold $S_y$ of $Y$ which is transverse to $\mathcal D$. From (ii), there exists an open neighborhood $U_y$ of $y$ in $Y$ such that every point in $U_y$ belongs to the orbit of some element of $S_y$. Since $Y$ is separable (and metrizable), it is second--countable, hence the open cover $\{U_y\}_{y\in Y}$ of $Y$ admits a countable subcover $\{U_{y_n}\}_{n\in\N}$. Then, the corresponding family $\{S_n\}_{n\in\N}$ given by $S_n=S_{y_n}$ clearly satisfies the claimed property.
\end{proof}

This shows that condition (i) of Definition~\ref{def:genslice} can be verified under these stronger hypotheses by taking a family $\{S_n\}_{n\in\N}$ as above. Nevertheless, we are interested in establishing the existence of a generalized slice for the reparameterization action \eqref{eq:mus1h1} of $S^1$ on $H^1(S^1,M)$ with respect to the $S^1$--invariant $g$--energy functional $E_g$. This will be done {\em manually}, directly using the structure of $H^1(S^1,M)$, since the action is not differentiable and hence the above arguments fail.

\begin{proposition}\label{prop:existslice}
For every $g\in\met_\nu^k(M)$, there exists a generalized slice $(\mathfrak U,\{S_n\}_{n\in\N})$ for the reparameterization action of $S^1$ on $H^1(S^1,M)$, given by \eqref{eq:mus1h1}, with respect to the $S^1$--invariant $g$--energy functional $E_g$, given by \eqref{eq:efunctper}.
\end{proposition}

\begin{proof}
Let $g\in\met_\nu^k(M)$ be any metric. Through a sequence of five claims we will prove existence of a generalized slice for the action of $S^1$ on $H^1(S^1,M)$ with respect to $E_g$.

Let us first establish some notations. For each $\gamma$, there is a natural inclusion of $T_\gamma H^1(S^1,M)$ in $\sect^{L^2}(\gamma^*TM)$. This inclusion induces an $L^2$--topology on $T_\gamma H^1(S^1,M)$, which will be denoted $\tau_{L^2}$. This is the smallest topology that turns the above inclusion continuous. The Hilbert space of Sobolev $H^1$ sections of $\gamma^*TM$ endowed with the topology $\tau_{L^2}$ will be denoted $(T_\gamma H^1(S^1,M),\tau_{L^2})$.

In addition, for each $\gamma\in H^1(S^1,M)$, we denote $\mathcal D_\gamma$ the one--dimensional subspace of $\sect^{L^2}(\gamma^*TM)$ spanned by the tangent field $\dot\gamma$. If $\gamma\in H^2(S^1,M)$, this coincides with the subspace defined by \eqref{eq:dy}, using differentiability of $\rho_\gamma$, see Remark~\ref{re:weakreg}. In this way, we may consider $\mathcal D_\gamma\subset\sect^{L^2}(\gamma^*TM)$ for each $\gamma\in H^1(S^1,M)$, and not only for $\gamma\in H^2(S^1,M)$.

\begin{claim}\label{cl:of}
For each $\gamma\in H^2(S^1,M)$ there exists a submanifold $S_\gamma$ of $H^1(S^1,M)$ such that for all $\alpha\in S_\gamma$ there exists a closed subset $A_\alpha$ of $\sect^{L^2}(\alpha^*TM)$ such that
\begin{itemize}
\item[(i)] $\sect^{L^2}(\alpha^*TM)=\mathcal D_\alpha\oplus A_\alpha$;
\item[(ii)] $T_\alpha S_\gamma=A_\alpha\cap T_\alpha H^1(S^1,M)$, in particular, $T_\alpha S_\gamma$ is closed \hfill\break\hfill in $(T_\alpha H^1(S^1,M),\tau_{L^2})$.
\end{itemize}
\end{claim}

Let us consider a special chart around $\gamma\in H^1(S^1,M)$. Namely, there exists an open neighborhood $U$ of the origin of $T_\gamma H^1(S^1,M)$ such that the map
\begin{equation*}
\begin{aligned}
\phi:U\ni u&\longmapsto\phi(u)\in H^1(S^1,M)\\
\phi(u)(z)&=\exp^\mathrm R_{\gamma(z)}(u(z)), \quad z\in S^1
\end{aligned}
\end{equation*}
is a local chart, where $\exp_x^\mathrm R$ is the exponential map of the auxiliary Riemannian metric $g_\mathrm R$ on the basepoint $x\in M$. Notice that $\phi(0)=\gamma$ corresponds to the exponentiation of the null section along $\gamma$, which hence coincides with $\gamma$ itself. The differential of this chart can be easily computed as being\footnote{Here we identify the tangent space $T_u T_\gamma H^1(S^1,M)$ to this vector space as $T_\gamma H^1(S^1,M)$ itself.}
\begin{equation}\label{eq:dphiu}
\begin{aligned}
\dd\phi(u):T_\gamma H^1(S^1,M)&\la T_{\phi(u)}H^1(S^1,M)& \\
(\dd\phi(u)v)(z)&=\dd\exp^\mathrm R_{\gamma(z)}(u(z))v(z), \quad z\in S^1.
\end{aligned}
\end{equation}
Furthermore, notice that $\dd\phi(0)=\id$. The linear isomorphism $\dd\phi(u)$ extends to a linear isomorphism $\widetilde{\dd\phi(u)}$ between the spaces of $L^2$--sections along the Sobolev $H^1$ curves $\gamma$ and $\phi(u)$ respectively, as the following diagram illustrates.
\begin{equation*}
\xymatrix@+20pt{\sect^{L^2}(\gamma^*TM)\ar[r]^{\widetilde{\dd\phi(u)}} & \sect^{L^2}(\phi(u)^*TM) \\ T_\gamma H^1(S^1,M)\ar@{^{(}->}[u]\ar[r]_{\dd\phi(u)} & T_{\phi(u)}H^1(S^1,M)\ar@{^{(}->}[u]}
\end{equation*}
In fact, this is a direct verification, by checking that the expression \eqref{eq:dphiu} is well--defined and is continuous in the $L^2$--topology,\footnote{The mentioned expression consists of a left multiplication by a continuous curve of operators, and is hence $L^2$--continuous.} hence extends to an isomorphism $\widetilde{\dd\phi(u)}$ as claimed above.

Define $A_\gamma$ as any closed complement of $\mathcal D_\gamma$ in $\sect^{L^2}(\gamma^*TM)$, so that $$\sect^{L^2}(\gamma^*TM)=\mathcal D_\gamma\oplus A_\gamma.$$ Notice that $A'=A_\gamma\cap T_\gamma H^1(S^1,M)$ is clearly closed in $T_\gamma H^1(S^1,M)$ in both the $L^2$--topology and its natural topology. Thus, it follows that
\begin{equation*}
T_\gamma H^1(S^1,M)=\mathcal D_\gamma\oplus A'.
\end{equation*}
Define $S_\gamma=\phi(A'\cap U)$. Since $A'$ is closed, $S$ is a submanifold of $H^1(S^1,M)$ in the sense of Definition~\ref{def:submnfldchart}. Moreover, from the observation that $\dd\phi(0)=\id$, it follows that $A'=T_\gamma S_\gamma$. For each $u\in A'\cap U$, let us denote $\alpha=\phi(u)\in S_\gamma$ the correspondent Sobolev $H^1$ curve. The for each $\alpha\in S_\gamma$, define $A_\alpha=\widetilde{\dd\phi(u)}(A_\gamma)$. This is clearly a closed subspace of $\sect^{L^2}(\alpha^*TM)$, since $\widetilde{\dd\phi(u)}$ is a linear isomorphism and $A_\gamma$ is closed in $\sect^{L^2}(\gamma^*TM)$.

Let us now verify that the above choices of submanifolds $S_\gamma$ and closed subsets $A_\alpha$ of $\sect^{L^2}(\alpha^*TM)$ for each $\alpha\in S_\gamma$ satisfy conditions (i) and (ii) of Claim~\ref{cl:of}, concluding the proof of this claim. Notice that $T_\alpha S_\gamma=\dd\phi(u)A'$ is contained in $T_\alpha H^1(S^1,M)$. Thus, $T_\alpha S_\gamma$ is contained in the intersection $A_\alpha\cap T_\alpha H^1(S^1,M)$. Using that $\widetilde{\dd\phi(u)}$ is an isomorphism, it follows easily that $T_\alpha S_\gamma=A_\alpha\cap T_\alpha H^1(S^1,M)$, which proves (ii).

As for (i), in order to prove that $\sect^{L^2}(\alpha^*TM)=\mathcal D_\alpha\oplus A_\alpha$, it suffices\footnote{Notice that the second decomposition is simply the pull--back of the first decomposition by the linear isomorphism $\widetilde{\dd\phi(u)}:\sect^{L^2}(\gamma^*TM)\to\sect^{L^2}(\alpha^*TM)$.} to prove that $\sect^{L^2}(\gamma^*TM)=A_\gamma\oplus\widetilde{\dd\phi(u)}^{-1}(\mathcal D_\alpha)$. For this, notice that the map
\begin{equation}\label{eq:mapamilagre}
T_\gamma H^1(S^1,M)\ni u\longmapsto\widetilde{\dd\phi(u)}^{-1}(\dot\alpha)\in\sect^{L^2}(\gamma^*TM)
\end{equation}
is continuous. In fact, we may compute directly $$\widetilde{\dd\phi(u)}^{-1}(\dot\alpha)(z)=\Big[\dd\exp^\mathrm R_{\gamma(z)}(u(z))\Big]^{-1}(\dot\alpha(z)), \quad z\in S^1,$$ where $\alpha(z)=\phi(u)(z)=\exp^\mathrm R_{\gamma(z)}(u(z))$, for all $z\in S^1$.

Notice that the desired transversality
\begin{equation}\label{eq:desireddec}
\sect^{L^2}(\gamma^*TM)=A_\gamma\oplus\widetilde{\dd\phi(u)}^{-1}(\mathcal D_\alpha)
\end{equation}
holds at $\alpha=\gamma$. From continuity of \eqref{eq:mapamilagre} and the fact that $A_\gamma$ is a closed subspace of $\sect^{L^2}(\gamma^*TM)$, it follows that by possibly reducing $S_\gamma$ if necessary, \eqref{eq:desireddec} also holds for all $\alpha\in S_\gamma$. Therefore, condition (i) is proved, concluding the proof of Claim~\ref{cl:of}.

\begin{claim}\label{cl:bah}
For all $\alpha\in S_\gamma\cap H^2(S^1,M)$, there is a decomposition $$T_\alpha H^1(S^1,M)=\mathcal D_\alpha\oplus T_\alpha S_\gamma.$$
\end{claim}

Let $\alpha\in S_\gamma\cap H^2(S^2,M)$. Then $\mathcal D_\alpha$ and $T_\alpha S_\gamma$ are closed subspaces whose intersection is, from (ii) of Claim~\ref{cl:of}, given by $$\mathcal D_\alpha\cap T_\alpha S_\gamma=\mathcal D_\alpha\cap A_\alpha\cap T_\alpha H^1(S^1,M),$$ which is clearly empty since the sum in (i) of Claim~\ref{cl:of} is direct. Moreover, if $v\in T_\alpha H^1(S^1,M)$, then from (i) of Claim~\ref{cl:of} it follows that there exists $v_1\in A_\alpha$ and $v_2\in\mathcal D_\alpha$ such that $v=v_1+v_2$. Since $\alpha\in H^2(S^1,M)$, the subspace $\mathcal D_\alpha$ is contained in $T_\alpha H^1(S^1,M)$, hence $v_2\in\sect^{H^1}(\alpha^*TM)$. Therefore, $v_1=v-v_2$ is in $T_\alpha H^1(S^1,M)$, and also in $A_\alpha$. Thus, from (ii) in Claim~\ref{cl:of}, it follows that $v_1\in T_\alpha S_\gamma$. This implies that any $v\in T_\alpha H^1(S^1,M)$ decomposes as a sum $v=v_1+v_2$ with $v_1\in T_\alpha S_\gamma$ and $v_2\in\mathcal D_\alpha$. This means that $T_\alpha H^1(S^1,M)=\mathcal D_\alpha + T_\alpha S_\gamma$, and since the intersection of these subspaces was proved to be trivial, it follows that this sum is direct, concluding the proof of Claim~\ref{cl:bah}.

\begin{claim}\label{cl:vale2}
If $\alpha\in S_\gamma$ is a critical point of $E_g\vert_{S_\gamma}$, then $\alpha$ is a critical point of $E_g$.
\end{claim}

Suppose $\alpha\in S_\gamma$ is such that $\dd E_g(\alpha)\vert_{T_\alpha S_\gamma}=0$. We first prove that this implies $\alpha\in H^2(S^1,M)$, and then use Claim~\ref{cl:bah} to conclude the argument. 

The differential $\dd E_g(\alpha)$ is a continuous functional\footnote{See Proposition~\ref{prop:fck}.} in $T_\alpha H^1(S^1,M)$, which can also be regarded with the topology induced by $\sect^{L^2}(\alpha^*TM)$. In order to infer that $\dd E_g(\alpha)$ is continuous also in this topology, i.e., in the space $(T_\alpha H^1(S^1,M),\tau_{L^2})$, notice that it vanishes in the subspace $T_\alpha S_\gamma$, which is closed in $(T_\alpha H^1(S^1,M),\tau_{L^2})$, from (ii) in Claim~\ref{cl:of}. Moreover, $T_\alpha S_\gamma$ is finite--codimensional in $T_\alpha H^1(S^1,M)$. In fact, from (ii) of Claim~\ref{cl:of}, $A_\alpha\cap T_\alpha H^1(S^1,M)\subset T_\alpha S_\gamma$, and hence the linear inclusion $$\frac{T_\alpha H^1(S^1,M)}{T_\alpha S_\gamma}\longhookrightarrow\frac{\sect^{L^2}(\alpha^*TM)}{A_\alpha}\cong\mathcal D_\alpha$$ is injective. Thus, since $\mathcal D_\alpha$ is finite--dimensional, also $T_\alpha H^1(S^1,M)/T_\alpha S_\gamma$ is finite--dimensional.

Therefore we may apply Lemma~\ref{le:grandetauskao}, which gives that $$\dd E_g(\alpha):(T_\alpha H^1(S^1,M),\tau_{L^2})\la\R$$ is continuous. From this continuity and density of $T_\alpha H^1(S^1,M)$ in the space $\sect^{L^2}(\alpha^*TM)$, it follows that the functional $\dd E_g(\alpha)$ admits a continuous extension to $\sect^{L^2}(\alpha^*TM)$. Since this is a Hilbert space, from the Riesz Representation Theorem~\ref{thm:riesz}, there exists $u\in\sect^{L^2}(\alpha^*TM)$ such that
\begin{equation}\label{eq:egrepresentado}
\dd E_g(\alpha)v = \int_{S^1} g(u,v)\;\dd z, \quad v\in T_\alpha H^1(S^1,M).
\end{equation}
Let $\{e_i(t)\}_{i=1}^m$, $t\in [0,2\pi]$ be a $g$--parallel orthonormal frame of $\alpha^*TM$, i.e., a $g$--orthonormal frame\footnote{Recall Definition~\ref{def:gorthframe}.} formed by vectors $e_i\in\sect^{H^1}(\alpha^*TM)$ along $\alpha$ that are $g$--parallel. For some integrals in the sequel, it will be handy to consider the parameter $t\in [0,2\pi]$ of the frame varying in this interval rather than using directly $z\in S^1$. Thus, we now implicitly assume a composition with the parameterization $$[0,2\pi]\ni t\longmapsto e^{it}\in S^1$$ of the circle using the interval $[0,2\pi]$. Decompose $u$ and $v$ with respect to this frame,
\begin{eqnarray*}
u=\sum_{i=1}^m u_i(t)e_i(t) && t\in [0,2\pi] \\
v=\sum_{i=1}^m v_i(t)e_i(t) && t\in [0,2\pi],
\end{eqnarray*}
and define for each $1\leq i\leq m$, $$U_i(t)=u_i(0)+\int_0^t u_i(s)\;\dd s.$$ Notice that since $u_i\in L^2([0,2\pi],\R)$, it follows that $U_i\in H^1([0,2\pi],\R)$, for all $1\leq i\leq m$. Since \eqref{eq:egrepresentado} holds for all $v\in T_\alpha H^1(S^1,M)$, in particular it holds supposing $v_i\in C^\infty_c(\,]0,2\pi[,\R)$, for all $1\leq i\leq m$. Since the above frame is orthonormal, consider $\delta_i=g(e_i,e_i)=\pm1$. We may then compute
\begin{eqnarray*}
\dd E_g(\alpha)v &=& \int_{S^1} g(u,v)\;\dd z\\
&=& \int_0^{2\pi} \sum_{i=1}^m \delta_iu_iv_i\;\dd t\\
&=& \int_0^{2\pi} \sum_{i=1}^m \delta_iU'_iv_i\;\dd t\\
&=& -\int_0^{2\pi} \sum_{i=1}^m \delta_iU_iv'_i\;\dd t\\
&=& -\int_0^{2\pi} \langle \widetilde U,\widetilde{v}'\rangle\;\dd t,
\end{eqnarray*}
where $\widetilde U=(\delta_iU_i)_{i=1}^m$, and $\widetilde{v}'=(v'_i)_{i=1}^m$. Notice that since the chosen frame is $g$--parallel, $$\D^g v(t)=\sum_{i=1}^m v'_i(t)e_i(t)$$ is represented in this frame by $\widetilde{v}'$, which is the ordinary derivative of $\widetilde v=(v_i)_{i=1}^m$. Let $w=(g(\dot\alpha,e_i))_{i=1}^m$. Then, from \eqref{eq:dfdgammaa},
\begin{eqnarray*}
\dd E_g(\alpha)v&=&\int_{S^1} g(\dot\alpha,\D^g v)\;\dd z\\
&=& \int_0^{2\pi} \sum_{i=1}^m v'_ig(\dot\alpha,e_i)\;\dd t.\\
&=& \int_0^{2\pi} \langle w,\widetilde{v}'\rangle\;\dd t.
\end{eqnarray*}
Thus, from the above two computations, $$\int_0^{2\pi} \langle\widetilde U+w,\widetilde{v}'\rangle\;\dd t=0$$ for all $\widetilde{v}\in C^\infty_c(\,]0,2\pi[,\R^m)$. From Lemma~\ref{le:distder0}, it follows that $\widetilde U+w$ is constant almost everywhere. This means that there exists $a=(a_i)_{i=1}^m\in\R^m$, with $\delta_iU_i(t)+g(\dot\alpha(t),e_i(t))=a_i$, for almost all $t\in [0,2\pi]$, for all $1\leq i\leq m$. Thus, $$\delta_ig(\dot\alpha(t),e_i(t))=\delta_ia_i-U_i(t)$$ for almost all $t\in [0,2\pi]$. Since $\dot\alpha\in\sect^{L^2}(\alpha^*TM)$ decomposes in the same frame as 
\begin{equation}\label{eq:alphadotrep}
\dot\alpha(t)=\sum_{i=1}^m \delta_ig(\dot\alpha(t),e_i(t))e_i(t),
\end{equation}
it follows that $\dot\alpha$ concides almost everywhere with a section os Sobolev class $H^1$. From Corollary~\ref{cor:milagre}, this implies that $\alpha$ is of class $C^1$, and $\dot\alpha$ is always equal to this Sobolev $H^1$ section, given by the right--hand side of \eqref{eq:alphadotrep}. Hence, $\alpha\in H^2(S^1,M)$.

Using that $\alpha\in S_\gamma\cap H^2(S^1,M)$, from Claim~\ref{cl:bah} there is a decomposition $$T_\alpha H^1(S^1,M)=\mathcal D_\alpha\oplus T_\alpha S_\gamma.$$ Let $v\in\mathcal D_\alpha$. Then $v=\lambda\dot\alpha$, and hence
\begin{eqnarray*}
\dd E_g(\alpha)v &=& \int_{S^1} g(\dot\alpha,\D^g v)\;\dd z\\
&=& \int_{S^1} \lambda g(\dot\alpha,\D^g \dot\alpha)\;\dd z\\
&=& \tfrac{\lambda}{2}\int_{S^1} \left.\frac{\dd}{\dd s}g(\dot\alpha(s),\dot\alpha(s))\right|_{s=z} \;\dd z\\
&=& 0,
\end{eqnarray*}
since $S^1$ has empty boundary. Therefore, $\dd E_g(\alpha)$ vanishes identically on $\mathcal D_\alpha$. By hypothesis, it also vanishes identically on $T_\alpha S_\gamma$. Thus, $\dd E_g(\alpha)=0$, i.e., $\alpha\in H^2(S^1,M)$ is a critical point of $E_g$, concluding the proof of Claim~\ref{cl:vale2}.

\begin{claim}\label{cl:vale3}
If $\alpha\in S_\gamma$ is a critical point of $E_g$, then $T_\alpha H^1(S^1,M)=\mathcal D_\alpha\oplus T_\alpha S_\gamma$.
\end{claim}

From Proposition~\ref{prop:critgenenfunc}, if $\alpha\in S_\gamma$ is a critical point of $E_g$, then it is a periodic $g$--geodesic. In particular, from Corollary~\ref{cor:geodck}, $\alpha\in H^2(S^1,M)$. Thus, from Claim~\ref{cl:bah} there exists the required decomposition $T_\alpha H^1(S^1,M)=\mathcal D_\alpha\oplus T_\alpha S_\gamma.$

\begin{claim}
There exists an open subset $\mathfrak U$ of $H^1(S^1,M)$ and a sequence $\{\gamma_n\}_{n\in\N}$ in $\mathfrak U$ such that $(\mathfrak U,\{S_{\gamma_n}\}_{n\in\N})$ is a generalized slice for the action of $S^1$ on $H^1(S^1,M)$ with respect to $E_g$.
\end{claim}

For each $\gamma\in H^2(S^1,M)$, consider the submanifold $S_\gamma$ of $H^1(S^1,M)$ given by Claim~\ref{cl:of}. From Remark~\ref{re:weakreg}, the map $$\rho_\gamma:S^1\ni e^{i\theta}\longmapsto\rho(e^{i\theta},\gamma)=\gamma(e^{i\theta}\,\cdot)\in H^1(S^1,M)$$ is of class $C^1$. In addition, Claim~\ref{cl:bah} gives a decomposition $T_\gamma H^1(S^1,M)=\im\dd\rho_\gamma(1)\oplus T_\gamma S_\gamma$. Thus, we may apply Proposition~\ref{prop:tauskehocara}, which implies that the subset $\rho(S^1\times S_\gamma)=\bigcup_{\alpha\in S_\gamma} \rho_\alpha(S^1)$ is a neighborhood of $\gamma\in H^1(S^1,M)$. 

Thus, for each $\gamma\in H^2(S^1,M)$, let $U_\gamma\subset\rho(S^1\times S_\gamma)$ be an open subset of $H^1(S^1,M)$ containing $\gamma$, and define $$\mathfrak U=\bigcup_{\gamma\in H^2(S^1,M)} U_\gamma.$$ This is clearly an open subset of $H^1(S^1,M)$ that contains\footnote{Since critical points of $E_g$ are periodic $g$--geodesics, from Corollary~\ref{cor:geodck}, they are automatically of class $C^2$, in particular of Sobolev class $H^2$.} all critical points of $E_g$, and since it is second--countable, by the Lindel\"of property, the open cover $\{U_\gamma\}_{\gamma\in H^2(S^1,M)}$ of $\mathfrak U$ admits a countable subcover $\{U_{\gamma_n}\}_{n\in\N}$. Consider the family $\{S_{\gamma_n}\}_{n\in\N}$ of submanifolds associated to the sequence $\{\gamma_n\}_{n\in\N}$ in $\mathfrak U$.

Since $\mathfrak U\subset\bigcup_{\gamma\in H^2(S^1,M)}\rho(S^1\times S_\gamma)$, condition (i) of Definition~\ref{def:genslice} is trivially satisfied. Moreover, conditions (ii) and (iii) of this definition are direct consequences of Claims~\ref{cl:vale2} and~\ref{cl:vale3}, respectively. Therefore, $(\mathfrak U,\{S_{\gamma_n}\}_{n\in\N})$ is a generalized slice for the action of $S^1$ on $H^1(S^1,M)$ with respect to $E_g$, concluding the proof.
\end{proof}


\part{Genericity of nondegenerate geodesics}

\chapter{Abstract genericity criteria}
\label{chap4}

In this chapter we discuss the concept of {\em genericity}, and give a few abstract criteria under which a certain property is {\em generic}. In general terms, a property is generic if it holds for {\em typical examples}, in other words, if {\em almost all} objects satisfy it. This notion can be formally defined in terms of {\em measures} or {\em topologies}, depending on the type of ambient space considered. Since our default ambient spaces are infinite--dimensional manifolds, where there is no clear concept of Lebesgue measure, the most natural definition of genericity is in terms of its topology, as follows.

\begin{definition}\label{def:generic}
A subset $\mathcal G$ of a metric space $E$ is said to be \emph{generic}\index{Generic property} in $E$ if it contains a $G_\delta$ of dense subsets, that is, a countable intersection of open dense subsets of $E$. Elements of $\mathcal G$ are said to be {\em generic elements of $E$}, and if these satisfy a certain property, then such property is also said to be {\em generic} in $E$.\index{Generic set}
\end{definition}

\begin{remark}\label{re:stupidremark}
As an obvious consequence of this definition, if a certain subset $S$ of a metric space $E$ contains a generic subset of $E$, then $S$ is also generic in $E$. This fact implies, for instance, that any union of generic subsets is also generic.
\end{remark}

\begin{remark}
Apart from giving intuition about the nature of mathematical objects, genericity plays an important role in the reliability mathematical models. Namely, due to inner inaccuracies in observation, the only relevant physical properties of a model are those generic in the adequate topology. This stability guarantees that such inaccuracies are physically neglectable. In particular, this is of great relevance in astrophysics concerning measurements of effects modeled by general relativity, see Hawking \cite{haw2}.
\end{remark}

Although this definition is given for subsets of metric spaces, it is clearly valid for more general topological spaces. Once more, all of our applications deal with metrizable ambients and hence we shall restrict to the case of (complete) metric spaces.

\begin{example}\label{ex:gensubsets1}
Some generic properties are quite intuitive, as the following examples indicate. A generic polynomial of degree $n$ with real coefficients has $n$ distinct complex roots. Generically, a plane in $\R^3$ intersects the three coordinate axes in three distinct points. A generic pair of lines in $\R^3$ are skew, i.e., non parallel and disjoint. 

Nevertheless, genericity can also be tricky at times. For instance, a generic function $f\in C^0([a,b],\R)$ is differentiable at no point of $[a,b]$, nor is it monotone on any subinterval of $[a,b]$. A proof can be found in Pugh \cite{pugh}, using the Lebesgue Monotone Differentiation Theorem. Further interesting generic properties of continuous real functions are given in \cite{boas,bruckner1,bruckner2,rooij}. For more basic examples of generic properties see Example~\ref{ex:gensubsets2} and Remarks~\ref{rm:isotrivial} and~\ref{re:genericvsdense}.
\end{example}

Let us mention a few elementary properties of generic sets, which are in great part direct consequences of basic properties of $G_\delta$ sets.\index{$G_\delta$ set} Recall that a subset of a metric space is a $G_\delta$ if it is given as the countable intersection of open subsets. Countable intersections and finite unions of $G_\delta$'s are still a $G_\delta$, and open or closed subsets are $G_\delta$'s. The complementary of a $G_\delta$ set is called an $F_\sigma$ set, which is hence given as countable union of closed subsets.\index{$F_\sigma$ set}

\begin{lemma}\label{le:genopenintersect}
Let $E$ be a metric space, $\mathcal G$ a generic subset and $A$ an open subset. Then $\mathcal G\cap A$ is generic in $A$.
\end{lemma}

\begin{proof}
If $\mathcal G$ is generic, it contains a $G_\delta$, $D=\bigcap_{n\in\N} D_n$, where $D_n$ are open dense subsets of $E$. The fact that $\mathcal G\cap A$ is generic in $A$ is evident considering $\bigcap_{n\in\N} (D_n\cap A)$, which is a dense $G_\delta$ of $A$ contained in $\mathcal G\cap A$.
\end{proof}

\begin{lemma}\label{le:genunion}
Let $E$ be a metric space and $A$ and $B$ open subsets of $E$. If $\mathcal G$ is generic in both $A$ and $B$, then it is also generic in $A\cup B$.
\end{lemma}

\begin{proof}
Since $G$ is generic in $A$ and $B$, there exist dense $G_\delta$'s $$D^A=\bigcap_{n\in\N} D^A_n\;\;\mbox{ and }\;\;D^B=\bigcap_{n\in\N} D^B_n$$ of $A$ and $B$ respectively, contained in $\mathcal G$. Since $D^A$ is dense in $A$ and $D^B$ is dense in $B$, it follows that $D^A\cup D^B$ is dense in $A\cup B$, for the closure of a union is the union of closures. Moreover, $$D^A\cup D^B=\left(\bigcap_{n\in\N}D^A_n\right)\cup\left(\bigcap_{n\in\N}D^B_n\right)=\bigcap_{n\in\N}(D^A_n\cup D^B_n),$$ and since $D^A_n$'s and $D^B_n$'s are open in $A$ and $B$ respectively, which are in turn open subsets of $E$, it follows that $(D^A_n\cup D^B_n)$ is an open subset of $A\cup B$. Hence $D^A\cup D^B$ is a (dense) $G_\delta$ of $A\cup B$ contained in $\mathcal G$, and therefore $\mathcal G$ is generic in $A\cup B$.
\end{proof}

\begin{lemma}\label{le:gencountintersect}
Let $E$ be a metric space and $\{\mathcal G_n\}_{n\in\N}$ a countable family of generic subsets. Then the intersection $\bigcap_{n\in\N}\mathcal G_n$ is also generic in $E$.
\end{lemma}

\begin{proof}
From genericity of each $\mathcal G_n$ in $E$, there exist $D_m^n$ open dense subsets of $E$, with $D^n=\bigcap_{m\in\N} D_m^n\subset\mathcal G_n$. Let $$D=\bigcap_{n\in\N}D^n=\bigcap_{n,m\in\N} D_m^n.$$ Then $D\subset\bigcap_{n\in\N}\mathcal G_n$ is the countable intersection of open dense subsets in $E$, hence also $\bigcap_{n\in\N}\mathcal G_n$ is generic in $E$.
\end{proof}

\begin{example}\label{ex:gensubsets2}
An open dense subset is clearly generic, since it contains a countable (finite) intersection of open dense subsets. A counter example for the converse is for instance the set of irrational numbers $\R\setminus\Q$, that is generic in $\R$ however not open. For more details, see Remark~\ref{re:genericvsdense}.

An important example of generic subset that is open and dense is the following. Let $V$ be a finite--dimensional vector space and consider $\GL(V)$ the set of all automorphisms of $V$. Then, it can be written as $$\GL(V)={\det}^{-1}\big(\R\setminus\{0\}\big),$$ where $\det:\Lin(V)\to\R$ is the determinant function. Since $\det$ is continuous, $\GL(V)$ is open. Moreover, it is dense in $\Lin(V)$ by standard arguments.\footnote{Suppose $A\in\Lin(V)$ is non invertible and let $\varepsilon>0$. Consider the polynomial function $p(t)=\det((1-t)A+t\id)$, $t\in\R$. Since $p$ is a polynomial, it has finitely many (hence isolated) zeros. Thus, since $p(0)=0$, there exists $\delta>0$ such that there are no other zeros of $p$ in $(-\delta,\delta)$ other than $0$. Therefore, if $t\in (-\delta,\delta)\setminus\{0\}$ then $(1-t)A+t\id\in\GL(V)$. Choose $0<t_0<\min\{\delta,\frac{\varepsilon}{\|A-\id\|}\}$. Then $(1-t_0)A+t_0\id\in\GL(V)$ and its distance from $A$ is less then $\varepsilon$, proving that $\GL(V)$ is dense in $\Lin(V)$.} Therefore, $\GL(V)$ is generic in $\Lin(V)$. Since operators of $\GL(V)$ are invertible operators, we may conclude that {\em an operator $T\in\Lin(V)$ is generically invertible}, or is {\em generically an isomorphism}.
\end{example}

Notice that a non invertible operator, or a singular matrix, is a quite symmetric object, when compared to invertible operators, that are {\em generic}. This is a sort of general rule, in the sense that highly symmetric objects are almost never {\em generic}. All genericity results in Chapters~\ref{chap5} and~\ref{chap6} are in this direction, asserting for instance that on a given manifold, two points\footnote{This is the content of Theorem~\ref{thm:biljavapic}, of \cite{biljavapic}. Nevertheless, we prove (see Theorem~\ref{thm:bigone}) that it is possible to extend this concept to much more general {\em endpoints conditions}, such as two submanifolds or any {\em admissible GEC}, see Definition~\ref{def:admgbc}.} are not conjugate in a generic semi--Riemannian metric.

\begin{remark}\label{rm:isotrivial}
This general idea regarding symmetries also holds in a more precise sense, considering isometries of a manifold. Namely, it is possible to prove that a generic Riemannian metric $g$ on a smooth finite--dimensional manifold $M$ has trivial isometry group $\Iso(M,g)=\{\id\}$, see Definition~\ref{def:isometries}. This result was proved by Ebin \cite{ebin}, however there are simpler proofs of this fact.
\end{remark}

We now recall the celebrated Baire Theorem, whose proof can be found for instance in Manetti \cite{manetti}. Among other consequences, it is the key fact used in the proof of basic functional analysis results such as the Open Mapping Theorem (or Banach--Schauder Theorem) and the Closed Graph Theorem.

\begin{bairethm}\label{thm:baire}\index{Theorem!Baire}
Let $\{D_n\}_{n\in\N}$ be a countable family of open dense subsets of a complete metric space. Then the intersection $\bigcap_{n\in\N} D_n$ is dense.
\end{bairethm}

\begin{remark}\label{re:genericvsdense}
By the Baire Theorem~\ref{thm:baire}, a generic subset (of a complete metric space) is automatically dense. This implies that arbitrarily small perturbations turn any element generic.

Notice however that genericity is a {\em much stronger} condition than being dense. For instance, the intersection of two generic subsets is generic (see Lemma~\ref{le:gencountintersect}), while the intersection of two dense subsets might be even empty. Namely, consider the set $\Q$ of rational numbers. Both $\Q$ and its complementary $\R\setminus\Q$ are dense, but their intersection is empty. Moreover, since $\Q$ is countable, it may be regarded as the countable union of its points, which are closed subsets of empty interior in $\R$. Therefore $\Q$ is an $F_\sigma$ with empty interior, and hence its complementary $\R\setminus\Q$ is a dense $G_\delta$, in particular, generic. Notice that $\R\setminus\Q$ is not open, however generic; and $\Q$ is not generic, however dense.
\end{remark}

\section{Sard--Smale Theorem and Transversality Theorem}
\label{sec:sardsmaleetc}

All genericity results presented in this text use directly or indirectly the Sard--Smale Theorem~\ref{thm:sardsmale}. This theorem is an infinite--dimensional extension of the celebrated Sard Theorem~\ref{thm:sard}, originally proved in 1942 by Sard \cite{sard} in the finite--dimensional context. It asserts that, under suitable regularity conditions, the set of critical values of a map between finite--dimensional manifolds has Lebesgue measure zero. As remarked above, there is no clear extension of the notion of null Lebesgue measure for subsets of Banach manifolds. Thus, its infinite--dimensional version states that the set of critical values of a sufficiently regular map\footnote{The regularity hypotheses from the finite--dimensional version are maintained and further {\em Fredholmness} assumptions are necessary.} has {\em generic}\footnote{Some textbooks choose to call such a subset {\em residual}. Since this term suggests {\em of small size}, we prefer not to use this terminology, and rather state that its complement is generic.} complement, in the sense of Definition~\ref{def:generic}. This important result was proved by Smale in the sixties in \cite{smale}, and is clearly of great value when dealing with generic properties. One of the main reasons for this is that if values are generically regular, from Proposition~\ref{prop:regularvalue} preimages (of values) are generically submanifolds.

In this section, we state the Sard Theorem~\ref{thm:sard} and use it to prove the Sard--Smale Theorem~\ref{thm:sardsmale}, exploring also further aspects related to the genericity of transversality using these results. A complete proof of the Sard Theorem~\ref{thm:sard}, originally given by Sard \cite{sard} in 1942, can be found in most differential topology textbooks, such as \cite{hirsch,milnortop}.

\begin{sardthm}\label{thm:sard}\index{Theorem!Sard}
Let $f:\R^n\to\R^m$ be a $C^k$ map, such that $k>\max\{n-m,0\}$. Then the set of critical values of $f$ has Lebesgue measure zero in $\R^m$.
\end{sardthm}

The adequate infinite--dimensional context to generalize this result is considering a $C^k$ nonlinear Fredholm map $f:X\to Y$ between Banach manifolds $X$ and $Y$, see Definitions~\ref{def:chartatlasmnfld} and~\ref{def:nonlinfred}. For this section, consider $f:X\to Y$ such a map. At a further point, we will also need to assume separability of the Banach manifolds $X$ and $Y$ and a regularity condition on $f$, namely $k>\max\{\ind(f),0\}$.

The following result asserts that after a suitable change of coordinates, $f$ differs from the identity by a nonlinear map between finite--dimensional manifolds.

\begin{lemma}\label{le:frednonlinstruc}
Let $f:X\to Y$ be a $C^k$ nonlinear Fredholm map. Then for any $x_0\in X$ there exists a Banach space $B$, finite--dimensional spaces $K$ and $C$, a $C^k$ map $g:B\oplus K\to C$ and local charts\footnote{Recall Definition~\ref{def:chartatlasmnfld}.} $\varphi:U\to B\oplus K$ around $x_0\in X$ and $\psi:V\to B\oplus C$ around $f(x_0)\in Y$, such that
\begin{equation}\label{eq:frednonlinstruc}
\psi\circ f\circ\varphi^{-1}(b,k)=(b,g(b,k)), \quad b\in B,k\in K,
\end{equation}
\begin{equation*}
\xymatrix@+15pt{
X\ar@(u,u)[rrr]^f &\ar@{_{(}->}[l] U\ar[d]_{\varphi}\ar[r]^{f|_U} & V\ar[d]^{\psi} \ar@{^{(}->}[r] & Y \\
& B\oplus K\ar[r]_{(\id,g)} & B\oplus C &
}
\end{equation*}
\end{lemma}

\begin{proof}
Consider arbitrary local charts $(U,\varphi)$ around $x_0\in X$ and $(V,\psi)$ around $f(x_0)\in Y$ taking values on Banach spaces $B_1$ and $B_2$ respectively. Denote by $\widetilde f:B_1\to B_2$ the representation of $f$ in these local charts. Since $f$ is a nonlinear Fredholm map, $B=\im\dd\widetilde f(\varphi(x_0))$ is finite--codimensional, hence complemented, from Lemma~\ref{le:finitecomplemented}. Consider $C$ a (finite--dimensional) complement to $B$, so that $B_2=B\oplus C$. Notice that, from Lemma~\ref{le:comp1}, $C$ is topologically isomorphic to $\coker\dd\widetilde f(\varphi(x_0))$.

Let $\widetilde f=\big(\widetilde f_1,\widetilde f_2\big)$. Then, by the above construction, $$\widetilde f_1:B_1\la B$$ is a submersion at $\varphi(x_0)\in B_1$. Notice that $\ker\dd\widetilde f_1(\varphi(x_0))=\ker\dd\widetilde f(\varphi(x_0))$, since $\dd\widetilde f_2(\varphi(x_0))=0$. Define $K=\ker\dd\widetilde f(\varphi(x_0))$, and notice it is finite--dimensional, because $f$ is Fredholm.

We now use the local form of submersions (see Remark~\ref{re:localforms}). There exists a $C^k$ diffeomorphism $\mathfrak d$ between open subsets of $B_1$ and $B\oplus K$ (the local chart in the domain) and a $C^k$ nonlinear map $g:B\oplus K\to C$ (the representation of $\tilde f$ in using this chart and the identity on the counter domain), such that the following diagram is commutative
\begin{equation}
\xymatrix@+15pt{
B_1\ar[r]^(.4){\widetilde f}\ar[d]_{\mathfrak d} & B\oplus C\\
B\oplus K\ar@(ru,d)[ru]_{(\id,g)} &
}
\end{equation}
where by $(\id,g)$ we denote the map $$B\oplus K\ni (b,k)\longmapsto (b,g(b,k))\in B\oplus C.$$

Replacing $\varphi$ with $\mathfrak d\circ\varphi$ and shrinking $U$ if necessary, so that $\varphi(U)$ is in the domain of $\mathfrak d$, we obtain the desired local charts such that \eqref{eq:frednonlinstruc} holds, concluding the proof.
\end{proof}

\begin{remark}\label{re:kcnames}
The subspaces $K$ and $C$ in Lemma~\ref{le:frednonlinstruc} are clearly named this way in reference to $\ker\dd f(x_0)$ and $\coker\dd f(x_0)$, since they are clearly topologically isomorphic. Recall that since $f$ is Fredholm, these are finite--dimensional spaces. Furthermore, notice that
\begin{equation}\label{eq:indfkc}
\ind(f)=\dim K-\dim C.
\end{equation}
\end{remark}

\begin{remark}
As a consequence of this Lemma~\ref{le:frednonlinstruc}, it follows that the preimage $f^{-1}(\{y\})$ of a point by a $C^k$ nonlinear Fredholm map is locally homeomorphic to the preimage $g^{-1}(\psi(y))$ of a point by a $C^k$ map between finite--dimensional manifolds. This fact paves the way to several topics in {\em deformation theory} for complex manifolds, see Kuranishi \cite{kuran}.
\end{remark}

\begin{lemma}\label{le:critvalemptyint}
Let $f:X\to Y$ be a $C^k$ nonlinear Fredholm map, with $k>\max\{\ind(f),0\}$. Then the set of critical values of $f$ has empty interior.
\end{lemma}

\begin{proof}
From Definition~\ref{def:critical}, the set of critical values of $f$ is the image $f(\crit(f))$ of the critical set of $f$. It clearly suffices to prove that the intersection of each open subset of $Y$ with $f(\crit(f))$ has empty interior. Thus, since the question is local, from Lemma~\ref{le:frednonlinstruc}, we may use local charts and reduce the problem to the case where $f$ is of the form
\begin{eqnarray}\label{eq:fbk}
f: B\oplus K &\la &B\oplus C \nonumber\\
(b,k)&\longmapsto &(b,g(b,k)),
\end{eqnarray}
where $B$ is a Banach space, $K$ and $C$ are finite--dimensional spaces and $g:B\oplus K\to C$ is a $C^k$ map.

The proof will be by contradiction, using the Sard Theorem~\ref{thm:sard}. Suppose that $V\subset B\oplus C$ is a nonempty open subset of critical values of $f$. Notice that a point $(b,k)\in B\oplus K$ is a critical point of $f$ if and only if $k$ is a critical point of $$g_b:K\ni k\longmapsto g(b,k)\in C.$$ In fact, from \eqref{eq:fbk}, it follows that
\begin{equation}
\dd f(b,k)=\left[\begin{array}{c c} \id & 0 \\ & \\ \dfrac{\partial g}{\partial b}(b,k) &\dfrac{\partial g}{\partial k}(b,k) \end{array}\right]
\end{equation}
and hence $\dd f(b,k)$ is not surjective if and only if $\dd g_b(k)=\frac{\partial g}{\partial k}(b,k)$ is not surjective, i.e., $k$ is a critical point of $g_b$.

For each $b\in B$, define $C_b=\{c\in C:(b,c)\in V\}$, which is clearly an open subset of $C$. Since we are assuming that $V$ is a nonempty open subset of critical values of $f$, for some $b\in B$, $C_b$ is nonempty. For any $c\in C_b$, the value $(b,c)\in V$ is a critical value of $f$. Hence there exists $k_c\in K$ such that $f(b,k_c)=(b,c)$, with $(b,k_c)\in\crit(f)$. Thus, from the above discussion, such a $k_c$ is a critical point of $g_b$, and hence $c$ is a critical value of $g_b$. Since $c\in C_b$ was arbitrarily chosen, this implies that $C_b$ is an open subset of critical values of $g_b:K\to C$.

Finally, since $$k>\ind(f)\stackrel{\eqref{eq:indfkc}}{=}\dim K-\dim C,$$ the Sard Theorem~\ref{thm:sard} applies to $g_b$, implying that the set of its critical values has Lebesgue measure zero in $C$. This contradicts the existence of the nonempty open subset $C_b$ of critical values, concluding the proof.
\end{proof}

\begin{lemma}\label{le:locallyclosed}
Consider a map $f:B\oplus K\to B\oplus C$ of the form \eqref{eq:fbk} and let $U\subset B\oplus K$ be an open subset. Let $U_1\subset B$ and $U_2\subset K$ be open subsets with $\overline{U_2}$ compact\footnote{Recall that $K$ is finite--dimensional, see Remark~\ref{re:kcnames}.} and $\overline{U_1}\times\overline{U_2}\subset U$. Then the restriction $$f|_{\overline{U_1}\times\overline{U_2}}:\overline{U_1}\times\overline{U_2}\to B\oplus C$$ is a closed map, i.e., maps closed subsets of the domain to closed subsets of the counter domain.
\end{lemma}

\begin{proof}
Let $F\subset \overline{U_1}\times\overline{U_2}$ be a closed subset and $\{f(b_n,k_n)\}_{n\in\N}$ a convergent sequence in $B\oplus C$, where $\{(b_n,k_n)\}_{n\in\N}$ is a sequence on $F$. Since $f$ is of the form \eqref{eq:fbk}, convergence of $$f(b_n,k_n)=(b_n,g(b_n,k_n))$$ implies the convergence of $\{b_n\}_{n\in\N}$ to a limit $b_\infty\in B$. Since $\overline{U_2}$ is compact, up to passing to a subsequence, we may assume that $\{k_n\}_{n\in\N}$ converges to a limit $k_\infty\in K$. Hence, from continuity of $f$, we have that $\{f(b_n,k_n)\}_{n\in\N}$ converges to $f(b_\infty,c_\infty)\in f(F)$. Therefore, $f(F)$ is closed in $B\oplus C$, concluding the proof.
\end{proof}

\begin{lemma}\label{le:critvalfsigma}
Consider a map $f:B\oplus K\to B\oplus C$ of the form \eqref{eq:fbk}, and suppose $B$ is {\em separable}. If $U$ is an open subset of $X$, then $f(U\cap\crit(f))$ is an $F_\sigma$, i.e., a countable union of closed subsets.
\end{lemma}

\begin{proof}
Since $B$ is supposed to be separable, the sum $B\oplus K$ clearly satisfies the {\em Lindel\"of property}.\footnote{Since Banach spaces are first--countable, separability is equivalent to second--countability, and also to satisfying the so--called Lindel\"of property. Recall that the Lindel\"of property for a topological space $X$ asserts that every open cover of $X$ admits a countable subcover. The finite--dimensional space $K$ automatically satisfies this property, hence if $B$ is separable, we may assume the Lindel\"of property to hold for $B\oplus K$.} Thus, we may consider a countable open cover $\bigcup_{n\in\N} (U_1^n\times U_2^n)$ of $U$, where $U_1^n\subset C$ and $U_2^n\subset K$ satisfy the conditions of Lemma~\ref{le:locallyclosed}, i.e., $\overline{U_1^n}\times\overline{U_2^n}\subset U$ and $\overline{U_2^n}\subset K$ is compact.

From Lemma~\ref{le:critclosed}, the set $\crit(f)$ is closed in $B\oplus K$. Thus, the intersection $(\overline{U_1^n}\times\overline{U_2^n})\cap\crit(f)$ is closed in $\overline{U_1^n}\times\overline{U_2^n}$ for every $n\in\N$. Therefore, from Lemma~\ref{le:locallyclosed},
\begin{equation*}
f(U\cap\crit(f))=\bigcup_{n\in\N} f\big((\overline{U_1^n}\times\overline{U_2^n})\cap\crit(f)\big)
\end{equation*}
is a countable union of {\em closed} subsets of $B\oplus C$, i.e., an $F_\sigma$.
\end{proof}

We are now ready to prove the Sard--Smale Theorem~\ref{thm:sardsmale}, simply applying the previous lemmas.

\begin{sardsthm}\label{thm:sardsmale}\index{Theorem!Sard--Smale}
Let $f:X\to Y$ be a $C^k$ nonlinear Fredholm map between separable Banach manifolds, with $k>\max\{\ind(f),0\}$. Then the set of regular values of $f$ is generic.
\end{sardsthm}

\begin{proof}
Since $Y$ is separable (hence second--countable) and union of generic subsets is generic (see Remark~\ref{re:stupidremark}), it suffices\footnote{For a slightly more precise justification of this fact, see Remark~\ref{re:gotu}.} to prove that every $y\in Y$ admits an open neighborhood $V$, such that the regular values of $f$ in $V$ form a generic subset of $V$. If $V$ does not intersect $\im f$, there is nothing to do, since all values in $V$ are trivially regular (recall Definition~\ref{def:regular}). Suppose there exists $x_0\in X$, with $f(x_0)\in V$. We have to prove that the regular values of $f$ in $V$ contain an open $G_\delta$ of $V$, which is equivalent to proving that its complementary $f(U\cap\crit(f))$ is contained in an $F_\sigma$ with empty interior.

Shrinking $U$ and $V$ if necessary, we may assume that these are domains of charts $\varphi$ and $\psi$ of $X$ and $Y$ respectively, as in Lemma~\ref{le:frednonlinstruc}. In other words, we may assume that $f$ is locally represented by \eqref{eq:fbk}, i.e.,
\begin{eqnarray*}
f: B\oplus K &\la &B\oplus C\\
(b,k)&\longmapsto &(b,g(b,k)),
\end{eqnarray*}
where $B$ is a Banach space, $K$ and $C$ are finite--dimensional spaces such that \eqref{eq:indfkc} holds and $g:B\oplus K\to C$ is a $C^k$ map. From Lemma~\ref{le:critvalfsigma}, $f(U\cap\crit(f))$ is a countable union of closed subsets of $V$. Since $k>\max\{\ind(f),0\}$, we may also apply\footnote{Notice that in this step we use the finite--dimensional Sard Theorem~\ref{thm:sard} on $g:K\to C$, to guarantee that $\crit(f)$ has empty interior. Its hypotheses are satisfied since $\ind(f)=\dim K-\dim C<k$.} Lemma~\ref{le:critvalemptyint}, that implies that each of these closed subsets has empty interior. Therefore, $f(U\cap\crit(f))$ is an $F_\sigma$ with empty interior, concluding the proof.
\end{proof}

\begin{remark}\label{re:gotu}
Notice that we proved above that $f(U\cap\crit(f))$ is a countable union of closed subsets with empty interior in $V$. The meticulous reader may argue that this does not automatically imply that $f(\crit(f))$ is a countable union of closed subsets with empty interior {\em in $Y$}.

Notice however that each subset with empty interior in $V$ also has empty interior in $Y$. Using again its second--countability, we may cover $Y$ with a countable number of open subsets. Considering the union of the closure of these open subsets intersected with the closed subsets of $V$ with empty interior we have an $F_\sigma$ {\em in $Y$} with empty interior.
\end{remark}

We end this section with a brief discussion on the consequences of the Sard Theorem~\ref{thm:sard} and the Sard--Smale Theorem~\ref{thm:sardsmale} related to genericity of transversality, see Definition~\ref{def:transversality}. More precisely, we now prove the so--called Transversality Theorem, that asserts that a sufficiently regular map is generically transverse to a fixed submanifold of the counter domain. 

\begin{transvthm}\label{thm:transversality}\index{Theorem!Transversality}
Let $X$, $Y$ and $Z$ be separable Banach manifolds, $W$ a submanifold of $Z$ and $f:X\times Y\to Z$ a $C^k$ map. For every $(x,y)\in f^{-1}(W)$, denote $q:T_{f(x,y)}Z\to T_{f(x,y)}Z/T_{f(x,y)}W$ the quotient map and suppose that $q\circ\frac{\partial f}{\partial y}:T_yY\to T_{f(x,y)}Z/T_{f(x,y)}W$ is a Fredholm operator such that $k>\max\left\{\ind \left(q\circ\frac{\partial f}{\partial y}(x,y)\right),0\right\}$. In addition, suppose that for every $(x,y)\in f^{-1}(W)$ the operator
\begin{equation}\label{eq:Ltransvthm}
T_xX\oplus T_yY\xrightarrow{\;\;\dd f(x,y)\;\;} T_{f(x,y)}Z\xrightarrow{\;\; q\;\;}\frac{T_{f(x,y)}Z}{T_{f(x,y)}W}
\end{equation}
is surjective. Denoting $f_x:Y\to Z$ the map $f_x(y)=f(x,y)$, the following is a generic subset of $X$,
\begin{equation}
\mathcal G=\{x\in X: f_x \mbox{ is transverse to }W\}.
\end{equation}
\end{transvthm}

\begin{proof}
For all $(x,y)\in f^{-1}(W)$, consider the Banach spaces $V_1=T_xX$, $V_2=T_yY$, $H=T_{f(x,y)}Z/T_{f(x,y)}W$ and the surjective operator $L=q\circ\dd f(x,y)$, given by \eqref{eq:Ltransvthm}. Since $L|_{V_2}=q\circ\frac{\partial f}{\partial y}(x,y)$ is Fredholm, from Lemma~\ref{le:lemmatransvthm}, the subspace $\ker L=\dd f(x,y)^{-1}[T_{f(x,y)}W]$ is complemented. Together with surjectivity of $L$, this implies that $f$ is transverse to $W$. From Proposition~\ref{prop:transvsubmnfld}, $f^{-1}(W)$ is a $C^k$ submanifold of $X\times Y$ and $T_{(x,y)}f^{-1}(W)=\dd f(x,y)^{-1}[T_{f(x,y)}W]=\ker L$.

Let $\Pi:X\times Y\to X$ be the projection onto the first coordinate. On the one hand, $(x,y)\in f^{-1}(W)$ is a regular point of $\Pi|_{f^{-1}(W)}:f^{-1}(W)\to X$ if
\begin{equation*}
\dd\Pi(x,y)|_{\ker L}:\underbrace{T_{(x,y)}f^{-1}(W)}_{\ker L}\longhookrightarrow \underbrace{T_xX\oplus T_yY}_{V_1\oplus V_2}\xrightarrow{\;\;\dd\Pi(x,y) \;\;}\underbrace{T_xX}_{V_1}
\end{equation*}
is surjective and has complemented kernel.\footnote{Notice that the derivative $\dd\Pi(x,y)$ coincides with the (linear) projection $T_xX\oplus T_yY\to T_yY$ onto the first variable.} On the other hand, $f_x$ is transverse to $W$ at $y\in Y$ if
\begin{equation*}
L|_{V_2}:\underbrace{T_yY}_{V_2}\longhookrightarrow \underbrace{T_xX\oplus T_yY}_{V_1\oplus V_2}\xrightarrow{\;\;L\;\;}\underbrace{T_{f(x,y)}Z/T_{f(x,y)}W}_{H}
\end{equation*}
is surjective and has complemented kernel. From Lemma \ref{le:lemmatransvthm}, both kernels are isomorphic to $\ker L\cap(\{0\}\oplus V_2)$, and since we are assuming that $L|_{V_2}=q\circ\frac{\partial f}{\partial y}(x,y)$ is Fredholm, these kernels are also finite--dimensional, hence trivially complemented. Also from Lemma \ref{le:lemmatransvthm}, $\dd\Pi(x,y)|_{\ker L}$ is surjective if and only if $L|_{V_2}$ is surjective. Thus, $(x,y)\in f^{-1}(W)$ is a regular point of $\Pi|_{f^{-1}(W)}:f^{-1}(W)\to X$ if and only if $f_x$ is transverse to $W$ at $y\in Y$. Therefore, $x$ is a regular value of $\Pi|_{f^{-1}(W)}$ if and only if $f_x$ is transverse to $W$.

Once more from Lemma \ref{le:lemmatransvthm}, since $L|_{V_2}=q\circ\frac{\partial f}{\partial y}(x,y)$ is a Fredholm operator, also $\dd\Pi(x,y)|_{T_{(x,y)}f^{-1}(W)}$ is a Fredholm operator, and $$\ind\left(q\circ\frac{\partial f}{\partial y}(x,y)\right)=\ind\left(\dd\Pi(x,y)|_{T_{(x,y)}f^{-1}(W)}\right).$$ Thus, $\Pi|_{f^{-1}(W)}$ is a $C^k$ nonlinear Fredholm map, with $$k>\max\left\{\ind\left(\dd\Pi(x,y)|_{T_{(x,y)}f^{-1}(W)}\right),0\right\},$$ and hence the Sard--Smale Theorem \ref{thm:sardsmale} gives genericity of the set of regular values of $\Pi_{f^{-1}(W)}$. Since $x$ is a regular value of $\Pi|_{f^{-1}(W)}$ if and only if $f_x$ is transverse to $W$, it follows that $\mathcal G$ is generic, concluding the proof.
\end{proof}

\begin{remark}
The above Transversality Theorem~\ref{thm:transversality} clearly holds for finite--dimensional manifolds, in which case several hypotheses are automatically verified. For instance, trivially the operator $q\circ\frac{\partial f}{\partial y}(x,y)$ is Fredholm, and all considered subspaces are complemented. Moreover, surjectivity of \eqref{eq:Ltransvthm} is equivalent to transversality of $f:X\times Y\to Z$ to $W$ at $(x,y)$.
\end{remark}

\begin{remark}\label{re:transvgensubmnfld}
The Transversality Theorem~\ref{thm:transversality} implies that two submanifolds are generically transverse, see Remark~\ref{re:abouttransv}. Let us give a more precise statement of this fact for finite--dimensional manifolds. Let $W$ be a submanifold of $Z$ and consider $f_x:Y\to Z$ a family of immersions of a manifold $Y$ into $Z$ parameterized by $x\in X$. Then, provided that $f:X\times Y\to Z$ is transverse to $W$ and sufficiently regular, the submanifolds $f_x(Y)$ and $W$ are transverse in $Z$ for a generic set of parameters $x\in X$.
\end{remark}

\section{Genericity criteria with transversality}
\label{sec:abstractgenericity}


Assume $Y$ is a Hilbert manifold and $f_x:Y\rightarrow\R$ is a family of functionals parameterized in an open subset of a Banach manifold $X$. In this section, we are interested in establishing abstract criteria for $f_x$ to be generically Morse, see Definition~\ref{def:degnondegstdeg}. More precisely, we want to prove that the following subset of parameters is generic in $X$, $$\mathcal G=\{x\in X:f_x:Y\to\R\mbox{ is Morse}\}.$$ For this, we adopt a standard transversality approach, inspired by Proposition~\ref{prop:morseifftransv}. This allows to consider $\mathcal G$ as the set of $x\in X$ such that $\frac{\partial f}{\partial y}(x,\cdot\,):Y\to TY^*$ is transverse to the null section $\mathbf 0_{TY^*}$, and use an appropriate version of the Transversality Theorem~\ref{thm:transversality} to prove its genericity. Indeed, under suitable transversality hypotheses, the preimage of the null section by $\frac{\partial f}{\partial y}$, is an embedded submanifold of $X\times Y$, the projection $\Pi:X\times Y\rightarrow X$ is a nonlinear Fredholm map of index zero and its critical values are precisely the set of parameters $x$ such that $f_x$ has some degenerate critical point in $Y$. Therefore, the problem of genericity of strongly nondegenerate critical points is reduced to a matter of regular values of a nonlinear Fredholm map. Then, genericity of $\mathcal G$ follows as a simple consequence of the Sard--Smale Theorem~\ref{thm:sardsmale}.

This approach follows the lines of standard transversality arguments present in \cite{AbbMaj2,abrahamrobbin,chillingsworth}. These were extended to the Banach and Hilbert manifolds setting in a more recent paper by White \cite{white}, that introduced the elegant method described above relating degeneracy for $f_x$ with criticality for $\Pi$ restricted to an adequate domain. We will now use these ideas to give a detailed proof of a first genericity criterion, with a formulation closely adapted from Biliotti, Javaloyes and Piccione \cite{biljavapic}.

\begin{abgen}\label{crit:abstractgenericity}\index{Abstract Genericity Criterion}
Consider $X$ a separable Banach manifold, $Y$ a separable Hilbert manifold and $\mathcal{U}\subset X\times Y$ an open subset. Let $f:\mathcal{U}\rightarrow\R$ be a $C^k$ functional and assume that for every $(x_0,y_0)\in\mathcal{U}$ such that $\frac{\partial f}{\partial y}(x_0,y_0)=0$, the following conditions hold:
\begin{itemize}
\item[(i)] the Hessian $$\frac{\partial^2 f}{\partial y^2}(x_0,y_0):T_{y_0}Y\la T_{y_0}Y^*\cong T_{y_0}Y$$ is a (self--adjoint) Fredholm operator;
\item[(ii)] for all $w\in\ker\left[\frac{\partial^2 f}{\partial y^2}(x_0,y_0)\right]\setminus\{0\}$, there exists $v\in T_{x_0}X$ such that $$\frac{\partial^2 f}{\partial x \partial y}(x_0,y_0)(v,w)\neq 0.$$
\end{itemize}
For each $x\in X$, let $\mathcal{U}_x=\{y\in Y:(x,y)\in\mathcal{U}\}$ and $f_x(y)=f(x,y)$ for all $y\in\mathcal U_x$. Then the following is a generic subset of $X$,
\begin{equation}\label{eq:gabstgen}
\mathcal G=\{x\in X: f_x:\mathcal{U}_x\to\R\mbox{ is Morse}\}.
\end{equation}
\end{abgen}

Before giving the proof, we make a couple of remarks on how the second partial derivatives mentioned in conditions (i) and (ii) can be defined without the use of a connection, under the hypotheses of the criterion.

\begin{remark}\label{re:remarkpd1}
Given $y\in Y$ with $(x,y)\in\mathcal U$ for some $x\in X$, since $x\mapsto \frac{\partial f}{\partial y}(x,y)$ takes values on the fixed Hilbert space $T_{y}Y^*$, the mixed derivative in condition (ii) is well defined without the use of a connection on $TY^*$.

More precisely, for such $y\in Y$, we have the map
\begin{eqnarray}\label{eq:partialy}
\frac{\partial f}{\partial y}(\, \cdot \,,y):\Pi(\mathcal U) &\la & T_{y}Y^* \\
x &\longmapsto & \tfrac{\partial f}{\partial y}(x,y) \nonumber
\end{eqnarray}
that can be once more differentiated, obtaining the mixed derivative
\begin{equation*}
\frac{\partial^2 f}{\partial x \partial y}(x,y):T_xX\la T_yY^*.
\end{equation*}
\end{remark}

\begin{remark}\label{re:remarkpd2}
If $\frac{\partial f}{\partial y}(x_0,y_0)=0$, the second partial derivative $\frac{\partial^2 f}{\partial y^2}(x_0,y_0)$ can be defined as the Hessian of $y\mapsto f(x_0,y)$ at the critical point $y_0$, also without depending on the choice of a connection, see Definition~\ref{def:hess}.

For the proof of the criterion, it suffices to adopt the following equivalent definition, given in Remark~\ref{re:equivalenthess}. The bilinear symmetric map
\begin{equation*}
\frac{\partial^2 f}{\partial y^2}(x,y):T_yY\times T_yY\la\R
\end{equation*}
is defined at the diagonal (at a pair $(v,v)\in T_yY\times T_yY$) using an auxiliary curve $\alpha:(-\varepsilon,\varepsilon)\to Y$, with $\alpha(0)=y$ and $\dot\alpha(0)=v$, by $$\frac{\partial^2 f}{\partial y^2}(x,y)(v,v)=(f\circ\alpha)''(0),$$ and extended to $T_yY\times T_yY$ by polarization.\footnote{More generally, if $f:X\to\R$ is a $C^2$ function defined on a Banach manifold, then $d^2f(x_0)(v,v)$ can be defined analogously with a curve $\alpha:(-\varepsilon,\varepsilon)\to X$. Inductively, if $f$ is $C^k$ and $d^jf(x_0)$ vanishes for $1\leq j\leq k-1$, $d^kf(x_0)$ can be defined in a similar way. Such definitions with auxiliary curves can be equivalently given in terms of a connection, and it can be proved that there is no dependence on the choice of this connection.} Hence, if $\frac{\partial f}{\partial y}(x,y)=0$, we have the second partial derivative
\begin{equation*}
\frac{\partial^2 f}{\partial y^2}(x,y):T_yY\la T_yY^*.
\end{equation*}
\end{remark}

\begin{proof}
We now proceed to the proof of the Abstract Genericity Criterion~\ref{crit:abstractgenericity}, through four claims.

First, given any $(x_0,y_0)\in\mathcal U$, denote by $0_{y_0}\in T_{y_0}Y^*$ the zero. If $\frac{\partial f}{\partial y}(x_0,y_0)=0_{y_0}$, then $T_{(y_0,0_{y_0})}TY^*$ canonically decomposes in the direct sum a horizontal and a vertical part, respectively tangent to the null section and to the fibers of $TY^*$. This is the automatic infinite--dimensional extension of \eqref{eq:tangentnullsec}, see Remark~\ref{re:nullsection}. More precisely, the tangent space to the null section of $TY^*$ at $(y_0,0_{y_0})$ is canonically identified as
\begin{equation}\label{eq:ident1}
T_{(y_0,0_y)}\mathbf 0_{TY^*}\cong T_{y_0}Y,
\end{equation}
and hence
\begin{equation}\label{eq:ident2}
T_{(y_0,0_{y_0})}TY^*\cong T_{y_0}Y\oplus T_{y_0}Y^*.
\end{equation}

Second, notice that \eqref{eq:partialy} induces a \emph{global partial derivative}
\begin{eqnarray*}
\frac{\partial f}{\partial y}:\mathcal U &\la & TY^* \\
(x,y) &\longmapsto & \left(y,\tfrac{\partial f}{\partial y}(x,y)\right).
\end{eqnarray*}
and if $(x_0,y_0)\in\mathcal U$ is such that $\frac{\partial f}{\partial y}(x_0,y_0)=0_{y_0}$, the above map can be differentiated again, obtaining
\begin{eqnarray}\label{eq:monster}
\mathrm d\left(\frac{\partial f}{\partial y}\right)(x_0,y_0):T_{x_0}X\oplus T_{y_0}Y &\la & T_{\left(y_0,0_{y_0}\right)}TY^*\cong T_{y_0}Y\oplus T_{y_0}Y^* \nonumber\\
(v,w) &\longmapsto &\left(w,\tfrac{\partial^2 f}{\partial x\partial y}(x_0,y_0)v+\tfrac{\partial^2 f}{\partial y^2}(x_0,y_0)w\right)
\end{eqnarray}
where \eqref{eq:ident2} is used and the second partial derivatives are in the sense of Remarks~\ref{re:remarkpd1} and~\ref{re:remarkpd2}. Observe that if $\frac{\partial f}{\partial y}(x_0,y_0)\neq 0_{y_0}$, it would be necessary to have a connection on $TY^*$, otherwise the vertical component of \eqref{eq:monster} would not be well defined, see Remarks~\ref{re:ttx} and~\ref{re:hessnoncrit}.

\begin{claim}\label{cl:claim1}
The map $\frac{\partial f}{\partial y}:\mathcal U\to TY^*$ is transverse to the null section $\mathbf 0_{TY^*}$ of $TY^*$ if and only if (ii) holds, which is also equivalent to
\begin{equation}\label{eq:transversal}
\ker\left(\frac{\partial^2 f}{\partial y^2}(x_0,y_0)\right)\bigcap\im\left(\frac{\partial^2 f}{\partial x\partial y}(x_0,y_0)\right)^\perp=\{0\}.
\end{equation}
\end{claim}

Denote by $p_{y_0}:T_{(y_0,0_{y_0})}TY^*\to T_{y_0}Y^*$ the projection correspondent to the decomposition \eqref{eq:ident2}. Then $$p_{y_0}\circ\mathrm d\left(\frac{\partial f}{\partial y}\right)(x_0,y_0):T_{x_0}X\oplus T_{y_0}Y\la T_{y_0}Y^*$$ is given by the direct sum of the bounded linear maps
\begin{equation}\label{eq:L1L2}
\begin{aligned}
L_1=\dfrac{\partial^2 f}{\partial x\partial y}(x_0,y_0):T_{x_0}X &\la T_{y_0}Y^*\cong T_{y_0}Y \\
L_2=\dfrac{\partial^2 f}{\partial y^2}(x_0,y_0):T_{y_0}Y &\la T_{y_0}Y^*\cong T_{y_0}Y.
\end{aligned}
\end{equation}

From Definition~\ref{def:transversality}, transversality of $\frac{\partial f}{\partial y}:\mathcal U\to TY^*$ to the null section $\mathbf 0_{TY^*}$ means that for all $(x_0,y_0)\in\mathcal U$ such that $\frac{\partial f}{\partial y}(x_0,y_0)=0_{y_0}$, $$\left[\mathrm d\left(\frac{\partial f}{\partial y}\right)(x_0,y_0)\right]^{-1}T_{(y_0,0_{y_0})} \mathbf 0_{TY^*}$$ is complemented in $T_{x_0}X\oplus T_{y_0}Y$ and $$\im\left[ \mathrm d\left(\frac{\partial f}{\partial y}\right)(x_0,y_0)\right]+T_{(y_0,0_{y_0})} \mathbf 0_{TY^*}=T_{(y_0,0_{y_0})}TY^*.$$ From \eqref{eq:L1L2}, using identifications \eqref{eq:ident1} and \eqref{eq:ident2}, this is equivalent to 
\begin{itemize}
\item[--] $\ker (L_1\oplus L_2)$ is complemented in $T_{x_0}X\oplus T_{y_0}Y$;
\item[--] $L_1\oplus L_2$ is surjective.
\end{itemize}

The first condition holds as a consequence of (i), since $\frac{\partial^2 f}{\partial y^2}(x_0,y_0)$ is Fredholm and hence has finite--dimensional kernel and finite--codimensional image. From Lemma~\ref{le:finitecomplemented}, both subspaces are complemented. Thus, applying Proposition~\ref{prop:sumcomplement} it follows that $\ker (L_1\oplus L_2)$ is complemented in $T_{x_0}X\oplus T_{y_0}Y$.

Therefore, transversality of $\frac{\partial f}{\partial y}:\mathcal U\to TY^*$ to $\mathbf 0_{TY^*}$ is now equivalent to surjectivity of $L_1\oplus L_2$. From self--adjointness and Fredholmness of $L_2$, Lemma~\ref{le:abs1} applies. This gives that $L_1\oplus L_2$ is surjective if and only if (ii) holds, which in turn is also obviously equivalent to \eqref{eq:transversal},\footnote{Since (ii) and \eqref{eq:transversal} are both equivalent to transversality of $\frac{\partial f}{\partial y}$ to the null section of $TY^*$, they will be henceforth labeled as \emph{transversality conditions} in this context.} concluding the proof of Claim~\ref{cl:claim1}.

\begin{claim}\label{cl:claim2}
The subset
\begin{equation*}
\mathfrak{M}=\left\{(x,y)\in\mathcal U:\frac{\partial f}{\partial y}(x,y)=0 \right\}
\end{equation*}
is an embedded $C^{k-1}$ submanifold of $X\times Y$, and at each $(x_0,y_0)\in\mathfrak M$, its tangent space is
\begin{multline}\label{eq:tangentM}
T_{(x_0,y_0)}\mathfrak{M}=\Big\{(v,w)\in T_{x_0}X\oplus T_{y_0}Y: \\ \left[\tfrac{\partial^2 f}{\partial x\partial y}(x_0,y_0)\oplus\tfrac{\partial^2 f}{\partial y^2}(x_0,y_0)\right](v,w)=0\Big\}.
\end{multline}
\end{claim}

From Claim~\ref{cl:claim1}, (ii) implies that the map $\frac{\partial f}{\partial y}:\mathcal U\to TY^*$ is transverse to the null section $\mathbf 0_{TY^*}$ of $TY^*$. Thus, the above claim is an immediate consequence of Proposition~\ref{prop:transvsubmnfld}. Furthermore, notice that in the notation \eqref{eq:L1L2}, the subspace \eqref{eq:tangentM} of $T_{x_0}X\oplus T_{y_0}Y$ is exactly the (complemented) space $\ker(L_1\oplus L_2)$.

\begin{claim}\label{cl:claim3}
Let $\Pi:X\times Y\to X$ be the projection onto the first variable. Then the restriction $\Pi\vert_{\mathfrak M}$ is a nonlinear $C^{k-1}$ Fredholm map of index zero, and $(x_0,y_0)\in\mathfrak M$ is a regular point of $\Pi\vert_{\mathfrak M}$ if and only if $y_0$ is a {\em strongly nondegenerate} critical point of the functional $$f_{x_0}:\mathcal U_{x_0}\owns y\longmapsto f(x_0,y)\in\R,$$ where $\mathcal{U}_{x_0}=\{y\in Y:(x_0,y)\in\mathcal{U}\}$.
\end{claim}

Fix $(x_0,y_0)\in\mathfrak M$. Then the derivative of $\Pi|_\mathfrak M$ at this point is
\begin{eqnarray*}
\mathrm d\Pi(x_0,y_0)\big|_{T_{(x_0,y_0)}\mathfrak M}:T_{(x_0,y_0)}\mathfrak M &\la &T_{x_0}X \\
(v,w) &\longmapsto & v
\end{eqnarray*}
and hence has kernel $T_{(x_0,y_0)}\mathfrak M\cap (\{0\}\oplus T_{y_0}Y)$, which, from \eqref{eq:tangentM}, is clearly identified as
\begin{equation}\label{eq:kerL2}
\ker \mathrm d\Pi(x_0,y_0)\big|_{T_{(x_0,y_0)}\mathfrak M}\cong \ker\left(\frac{\partial^2 f}{\partial y^2}(x_0,y_0)\right),
\end{equation}
and since $\frac{\partial^2 f}{\partial y^2}(x_0,y_0)$ is Fredholm, the right--hand side (hence both sides) are finite dimensional. In addition, if $v$ is in the image $\mathrm d\Pi(x_0,y_0)\left(T_{(x_0,y_0)}\mathfrak M\right)$, then from \eqref{eq:tangentM} there exists $w\in T_{y_0}Y$ such that $$\left[\frac{\partial^2 f}{\partial x\partial y}(x_0,y_0)\oplus\frac{\partial^2 f}{\partial y^2}(x_0,y_0)\right](v,w)=0,$$ hence
\begin{equation}\label{eq:projectionimage}
\mathrm d\Pi(x_0,y_0)\left(T_{(x_0,y_0)}\mathfrak M\right)=\left[\frac{\partial^2 f}{\partial x\partial y}(x_0,y_0)\right]^{-1}\im\left(\frac{\partial^2 f}{\partial y^2}(x_0,y_0)\right).
\end{equation}

Since $\frac{\partial^2 f}{\partial y^2}(x_0,y_0)$ is Fredholm, its image has finite codimension in $T_{y_0}Y$. Applying Lemma~\ref{le:abs2}, it follows that also $\mathrm d\Pi(x_0,y_0)\left(T_{(x_0,y_0)}\mathfrak M\right)$ has finite codimension in $T_{x_0}X$, hence $\mathrm d\Pi(x_0,y_0)\big|_{T_{(x_0,y_0)}\mathfrak M}$ is a Fredholm operator. More precisely,
\begin{multline*}
\codim_{T_{x_0}X}\left[\im \mathrm d\Pi(x_0,y_0)\big|_{T_{(x_0,y_0)}\mathfrak M}\right]=\codim_{T_{y_0}Y}\im\left(\frac{\partial^2 f}{\partial y^2}(x_0,y_0)\right) \\
-\codim_{T_{y_0}Y}\left[\im\left(\tfrac{\partial^2 f}{\partial x\partial y}(x_0,y_0)\right)+\im\left(\tfrac{\partial^2 f}{\partial y^2}(x_0,y_0)\right)\right],
\end{multline*}
and, since by \eqref{eq:transversal} the operator $\left[\tfrac{\partial^2 f}{\partial x\partial y}(x_0,y_0)\oplus\tfrac{\partial^2 f}{\partial y^2}(x_0,y_0)\right]$ is surjective, it follows that the last term in the right--hand side of the above expression is null. Therefore, using that $\frac{\partial^2 f}{\partial y^2}(x_0,y_0)$ is self--adjoint, it follows that
\begin{eqnarray*}
\dim\ker\left(\frac{\partial^2 f}{\partial y^2}(x_0,y_0)\right) &\stackrel{\eqref{eq:ateminhamaesabe}}{=}& \dim\im\left(\frac{\partial^2 f}{\partial y^2}(x_0,y_0)\right)^\perp \\
&=& \codim\im\left(\frac{\partial^2 f}{\partial y^2}(x_0,y_0)\right) \\
&=& \codim_{T_{x_0}X}\left[\im\dd\Pi(x_0,y_0)\big|_{T_{(x_0,y_0)}\mathfrak M}\right],
\end{eqnarray*}
thus $\ind\left(\mathrm d\Pi(x_0,y_0)\big|_{T_{(x_0,y_0)}\mathfrak M}\right)=0$. This proves the claim that $\Pi|_{\mathfrak M}$ is a nonlinear $C^{k-1}$ Fredholm map of index zero.

From \eqref{eq:projectionimage}, it is clear that $(x_0,y_0)\in\mathfrak M$ is a regular point of $\Pi|_{\mathfrak M}$, i.e., $\dd\Pi(x_0,y_0)\big|_{T_{(x_0,y_0)}\mathfrak M}$ is surjective, if and only if\footnote{Notice that its kernel $\ker\dd\Pi(x_0,y_0)\big|_{T_{(x_0,y_0)}\mathfrak M}$ given by \eqref{eq:kerL2} is finite--dimensional, hence complemented as a consequence of Lemma~\ref{le:finitecomplemented}.}
\begin{equation*}
\im\left(\frac{\partial^2 f}{\partial x\partial y}(x_0,y_0)\right)\subset\im\left(\frac{\partial^2 f}{\partial y^2}(x_0,y_0)\right).
\end{equation*}
Taking orthogonal complements and using once more that $\frac{\partial^2 f}{\partial y^2}(x_0,y_0)$ is self--adjoint, this is equivalent to
\begin{equation*}
\im\left(\frac{\partial^2 f}{\partial x\partial y}(x_0,y_0)\right)^\perp\supset\im\left(\frac{\partial^2 f}{\partial y^2}(x_0,y_0)\right)^\perp\stackrel{\eqref{eq:ateminhamaesabe}}{=}\ker\left(\frac{\partial^2 f}{\partial y^2}(x_0,y_0)\right).
\end{equation*}
Since, from \eqref{eq:transversal}, $$\ker\left(\frac{\partial^2 f}{\partial y^2}(x_0,y_0)\right)\bigcap\im\left(\frac{\partial^2 f}{\partial x\partial y}(x_0,y_0)\right)^\perp=\{0\},$$ $(x_0,y_0)$ is a regular point of $\Pi|_\mathfrak M$ if and only if $\ker\left(\frac{\partial^2 f}{\partial y^2}(x_0,y_0)\right)$ is trivial. Using \eqref{eq:kerL2}, this is in turn equivalent to $y_0$ being a nondegenerate critical point of $f_x$. Finally, from Lemma~\ref{le:selfadjointzero}, the self--adjoint Fredholm operator $\frac{\partial^2 f}{\partial y^2}(x_0,y_0)$ must have index zero. From Lemma~\ref{le:fred0} it follows that this operator is injective if and only if it is surjective. Thus, $(x_0,y_0)$ is a regular point of $\Pi|_\mathfrak M$ if and only if $y_0$ is a strongly nondegenerate critical point of $f_x$, concluding the proof of Claim~\ref{cl:claim3}.

\begin{claim}\label{cl:claim4}
The subset $\mathcal G$, given by \eqref{eq:gabstgen}, is generic in $X$.
\end{claim}

From Claim~\ref{cl:claim3}, the set of $\mathcal G$ of parameters $x\in X$ such that the functional $$f_x:\mathcal{U}_x\owns y\longmapsto f(x,y)\in\R$$ is Morse coincides with the set of regular values of $\Pi|_\mathfrak M$. Since this is a nonlinear $C^{k-1}$ Fredholm map of index zero between separable Banach manifolds, the Sard--Smale Theorem~\ref{thm:sardsmale} applies, and it follows that this is a generic subset of $X$, completing the proof of this last claim.
\end{proof}

\begin{remark}
The proof of the Abstract Genericity Criterion~\ref{crit:abstractgenericity} may be severely simplified by the use of Proposition~\ref{prop:morseifftransv} and the Transversality Theorem~\ref{thm:transversality}. Using Claim~\ref{cl:claim1}, all hypotheses of the Transversality Theorem~\ref{thm:transversality} are verified for the map $\frac{\partial f}{\partial y}:\mathcal U\to TY^*$, with respect to the submanifold $\mathbf 0_{TY^*}$. Thus, generically on $x$, $\frac{\partial f}{\partial y}(x,\cdot\,)$ is transverse to $\mathbf 0_{TY^*}$. Finally, from Proposition~\ref{prop:morseifftransv}, this happens if and only if $x\in\mathcal G$, proving that $\mathcal G$ is generic in $X$.

The actual proof given above is rather longer, however more detailed on how criticality for $f_x$ is equivalent to degeneracy for $\Pi\vert_{\mathfrak M}$, and how the Sard--Smale Theorem~\ref{thm:sardsmale} is used to prove genericity of $\mathcal G$. However, the methods involved are clearly the same.
\end{remark}
 

We end this section with a second abstract criterion, with slightly weaker hypotheses but with the same setting as the Abstract Genericity Criterion~\ref{crit:abstractgenericity}. Namely, we consider the same family of parameterized variation problems $f:\mathcal U\subset X\times Y\to\R$ however we allow condition (ii) to be verified only in a distinguished subset $\mathfrak C$ of $$\mathfrak{M}=\left\{(x,y)\in\mathcal U:\frac{\partial f}{\partial y}(x,y)=0\right\}.$$ The conclusion will then be that the parameters $x\in X$ for which these distinguished critical points of $f_x$ are strongly nondegenerate is generic in $X$.

\begin{abgen}\label{crit:weakabstractgenericity}\index{Abstract Genericity Criterion}
Consider $X$ be a separable Banach manifold, $Y$ a separable Hilbert manifold and $\mathcal{U}\subset X\times Y$ an open subset. Let $f:\mathcal{U}\rightarrow\R$ be a $C^k$ functional, $\mathfrak{M}=\left\{(x,y)\in\mathcal U:\frac{\partial f}{\partial y}(x,y)=0\right\}$, and $\mathfrak C$ a subset of distinguished pairs $(x,y)\in\mathfrak M$. Suppose the following conditions hold:
\begin{itemize}
\item[(i)] for every $(x_0,y_0)\in\mathfrak M$, the Hessian $$\frac{\partial^2 f}{\partial y^2}(x_0,y_0):T_{y_0}Y\la T_{y_0}Y^*\cong T_{y_0}Y$$ is a (self--adjoint) Fredholm operator;
\item[(ii)] for every $(x_0,y_0)\in\mathfrak C$, for all $w\in\ker\left[\frac{\partial^2 f}{\partial y^2}(x_0,y_0)\right]\setminus\{0\}$, there exists $v\in T_{x_0}X$ such that $$\frac{\partial^2 f}{\partial x \partial y}(x_0,y_0)(v,w)\neq 0.$$
\end{itemize}
For each $x\in X$, let $\mathcal{U}_x=\{y\in Y:(x,y)\in\mathcal{U}\}$, $\mathfrak C_x=\{y\in Y:(x,y)\in\mathfrak C\}$ and $f_x(y)=f(x,y)$ for all $y\in\mathcal U_x$. Then the following is a generic subset of $X$,
\begin{equation*}
\mathcal G_{\mathfrak C}=\{x\in X: \mbox{ all } y\in\mathfrak C_x\mbox{ are strongly nondegenerate for }f_x:\mathcal{U}_x\to\R\}.
\end{equation*}
\end{abgen}

\begin{proof}
This second abstract criterion is in fact a simple consequence of the Abstract Genericity Criterion~\ref{crit:abstractgenericity}. From Lemma~\ref{le:transvopen}, applied to $\frac{\partial f}{\partial y}:\mathcal U\to TY^*$, the following is an open subset of $\mathfrak M$,
\begin{equation*}
\mathfrak A=\left\{(x,y)\in\mathfrak M:\frac{\partial f}{\partial y}:\mathcal U\to TY^* \mbox{ is transverse to }\mathbf 0_{TY^*} \mbox{ at } (x,y)\right\}.
\end{equation*}
Thus, there exists $\mathcal V$ an open subset in $\mathcal U$, such that $\mathcal V\cap\mathfrak M=\mathfrak A$.

Under condition (i), Claim~\ref{cl:claim1} gives that $(x,y)\in\mathfrak A$ if and only if for all $w\in\ker\left[\frac{\partial^2 f}{\partial y^2}(x,y)\right]\setminus\{0\}$, there exists $v\in T_{x}X$ such that $\frac{\partial^2 f}{\partial x \partial y}(x,y)(v,w)\neq 0.$ Consider the restriction $f\vert_{\mathcal V}:\mathcal V\subset\mathcal U\to\R$. Then, conditions (i) and (ii) of the Abstract Genericity Criterion~\ref{crit:abstractgenericity} apply to $f\vert_{\mathcal V}$, and hence the following is generic in $X$,
\begin{equation*}
\mathcal G_{\mathcal V}=\{x\in X: f_x\vert_{\mathcal V_x}:\mathcal V_x\to\R\mbox{ is Morse}\},
\end{equation*}
where $\mathcal V_x=\{y\in Y:(x,y)\in\mathcal V\}$ and $f_x\vert_{\mathcal V_x}$ is the restriction of $f_x:\mathcal U_x\to\R$ to this open subset. Once more, from Claim~\ref{cl:claim1}, $\mathfrak C_x$ is contained in $\mathcal V_x$ for every $x\in X$. Therefore, $\mathcal G_{\mathcal V}$ is contained in $\mathcal G_{\mathfrak C}$ and hence, from Remark~\ref{re:stupidremark}, $\mathcal G_{\mathfrak C}$ is generic in $X$.
\end{proof}

\section{Equivariant genericity criteria}
\label{sec:equivariantgenericity}

In this section, we consider the same family of functionals $f:\mathcal U\subset X\times Y\to\R$ parametrized in a Banach manifold $X$, however with the additional hypothesis that there is a (non necessarily differentiable) action of a finite--dimensional Lie group $G$ on $Y$ and $f(x,\cdot\,)$ and $\mathcal U$ are $G$--invariant. We will assume also that this action is by diffeomorphisms and that it admits a generalized slice with respect to $f_x$ for all $x\in\Pi(\mathcal U)$, see Definition~\ref{def:genslice}. In this context, we will prove equivariant genericity criteria that give abstract sufficient conditions on $f$ to guarantee the genericity of the set of parameters $x\in X$ for which $f_x$ is $G$--Morse, see Lemma~\ref{le:invkernel} and Definition~\ref{def:gmorse}.

\begin{eqgen}\label{crit:equivariantgenericity}\index{Equivariant Genericity Criterion}\index{Abstract Genericity Criterion}
Consider $X$ a separable Banach manifold, $Y$ a separable Hilbert manifold, $\mathcal{U}\subset X\times Y$ an open subset and $G$ a finite--dimensional Lie group with an action $\mu:G\times Y\to Y$ by diffeomorphisms. Suppose that $\mathcal U$ is $G$--invariant and that $f$ is $G$--invariant in the second variable. Suppose also the existence of submanifolds $Y_2\subset Y_1\subset Y$ such that all critical points of $f_x=f(x,\cdot\,)$ are contained in $Y_2$ for all $x\in\Pi(U)$, and for all $y\in Y_2$, the subspace $\mathcal D_y$ of $T_yY_1$ is well--defined.\footnote{See Remark~\ref{re:weakreg} and \eqref{eq:dgammaspangamma}.} Assume the existence of a generalized slice $(\mathfrak U,\{S_n\}_{n\in\N})$ for the action of $G$ on $Y$ with respect to $f_x$, for all $x\in\Pi(U)$ and that for every $(x_0,y_0)\in\mathcal{U}$ such that $\frac{\partial f}{\partial y}(x_0,y_0)=0$, the following conditions hold:
\begin{itemize}
\item[(eq-i)] the Hessian $$\frac{\partial^2 f}{\partial y^2}(x_0,y_0):T_{y_0}Y\la T_{y_0}Y^*\cong T_{y_0}Y$$ is a (self--adjoint) Fredholm operator;
\item[(eq-ii)] for all $w\in\ker\left[\frac{\partial^2 f}{\partial y^2}(x_0,y_0)\right]\setminus\mathcal D_{y_0}$, there exists $v\in T_{x_0}X$ such that $$\frac{\partial^2f}{\partial x\partial y}(x_0,y_0)(v,w)\ne0.$$
\end{itemize}
For each $x\in X$, let $\mathcal{U}_x=\{y\in Y:(x,y)\in\mathcal{U}\}$ and $f_x(y)=f(x,y)$ for all $y\in\mathcal U_x$. Then the following is a generic subset of $X$,
\begin{equation}\label{eq:geqgen}
\mathcal G=\{x\in X: f_x:\mathcal{U}_x\to\R\mbox{ is } G\mbox{--Morse}\}.
\end{equation}
\end{eqgen}

\begin{proof}
The idea of the proof is to apply the Abstract Genericity Criterion~\ref{crit:abstractgenericity} to the restrictions of $f$ to the elements $S_n$ of the generalized slice. Let $\mathcal U_n$ be the open subsets of $X\times S_n$ defined by $$\mathcal U_n=\{(x,y)\in X\times S_n:\{x\}\times G(y)\cap\mathcal U\neq\emptyset\},$$ and consider $f_n:\mathcal U_n\to\R$ the restrictions $f|_{\mathcal U_n}$.

Let us verify that each $f_n:\mathcal U_n\to\R$ satisfies the hypotheses of the Abstract Genericity Criterion~\ref{crit:abstractgenericity}. Given $(x_0,y_0)\in\mathcal U_n$ such that $\frac{\partial f}{\partial y}(x_0,y_0)=0$, the decomposition\footnote{Recall that this decomposition exists from property (iii) of the generalized slice, see Definition~\ref{def:genslice}.}
\begin{equation}\label{eq:decompty0}
T_{y_0}Y=T_{y_0}S_n\oplus\mathcal D_{y_0}
\end{equation}
induces a decomposition of operators defined in $T_{y_0}Y$. Thus, \eqref{eq:decompty0} induces a decomposition of $\frac{\partial f}{\partial y}(x_0,y_0):T_{y_0}Y\to\R$ as the direct sum of $\frac{\partial f_n}{\partial y}(x_0,y_0):T_{y_0}S_n\to\R$ and the null functional\footnote{Recall that, since $f$ is $G$--invariant in the second variable, as observed in Lemma~\ref{le:invkernel} the subspace $\mathcal D_{y_0}$ is contained in the kernel of $\frac{\partial f}{\partial y}(x_0,y_0)$, and hence also in the kernel of $\frac{\partial^2f}{\partial x\partial y}(x_0,y_0)$.} of $\mathcal D_{y_0}$. Hence, since $\frac{\partial f}{\partial y}(x_0,y_0)=0$, it follows that also $\frac{\partial f_n}{\partial y}(x_0,y_0)=0$.

From (eq-i), the Hessian $\frac{\partial^2f}{\partial y^2}(x_0,y_0)$ is a Fredholm operator. Using \eqref{eq:decompty0}, it decomposes as the sum of
\begin{equation}\label{eq:hessfn}
\frac{\partial^2 f_n}{\partial y^2}(x_0,y_0):T_{y_0}S_n\la T_{y_0}S_n^*\cong T_{y_0}S_n
\end{equation}
and the null operator of $\mathcal D_{y_0}$. Thus, \eqref{eq:hessfn} is given by the restriction of a Fredholm operator to a finite codimensional space, which is hence Fredholm. Therefore condition (i) of the Abstract Genericity Criterion~\ref{crit:abstractgenericity} holds.

As for condition (ii), from the above decomposition of $\frac{\partial^2f}{\partial y^2}(x_0,y_0)$, if $w\in\ker\left[\frac{\partial^2 f_n}{\partial y^2}(x_0,y_0)\right]\setminus\{0\}$, then\footnote{Here we use property (ii) of the generalized slice, see Definition~\ref{def:genslice}.} $w\in\ker\left[\frac{\partial^2 f}{\partial y^2}(x_0,y_0)\right]\setminus\mathcal D_{y_0}$.  Thus, from (eq-ii), there exists $v\in T_{x_0}X$ such that $\frac{\partial^2 f}{\partial x\partial y}(x_0,y_0)(v,w)\neq 0$. Using \eqref{eq:decompty0}, also $$\frac{\partial^2 f}{\partial x\partial y}(x_0,y_0):T_{x_0}X\times T_{y_0}Y\la\R$$ decomposes as the sum of $\frac{\partial^2 f_n}{\partial x\partial y}(x_0,y_0)$ an the null functional of $T_{x_0}X\oplus\mathcal D_{y_0}$. Hence, $\frac{\partial^2f_n}{\partial x\partial y}(x_0,y_0)(v,w)\neq0$, proving that condition (ii) also holds.

Therefore we may apply the Abstract Genericity Criterion~\ref{crit:abstractgenericity} to each $f_n:\mathcal U_n\to\R$, obtaining genericity of $$\mathcal G_n=\big\{x\in X:f_n(x,\cdot\,)\mbox{ is Morse}\big\}$$ in $X$, for all $n\in\N$. Proposition~\ref{prop:morsegmorse} gives that\footnote{Here we use property (i) of the generalized slice, see Definition~\ref{def:genslice}.} $f$ is $G$--Morse if and only if $f_n$ is Morse for all $n\in\N$, hence $\mathcal G=\bigcap_{n\in\N}\mathcal G_n$. Since this is the intersection of a countable family of generic subsets of $X$, Lemma~\ref{le:gencountintersect} gives that $\mathcal G$ is generic in $X$, concluding the proof.
\end{proof}

Analogously to the Abstract Genericity Criterion~\ref{crit:weakabstractgenericity}, we now give a second equivariant criterion, with slightly weaker hypotheses but with the same setting as the Equivariant Genericity Criterion~\ref{crit:equivariantgenericity}. Namely, we consider the same family of parameterized $G$--invariant variation problems $f:\mathcal U\subset X\times Y\to\R$ however we allow condition (eq-ii) to be verified only in a distinguished subset $\mathfrak C$ of critical points of $f_x$. The conclusion will then be that the parameters $x\in X$ for which these distinguished critical points of $f_x$ are $G$--nondegenerate is generic in $X$.

\begin{eqgen}\label{crit:weakequivariantgenericity}\index{Equivariant Genericity Criterion}\index{Abstract Genericity Criterion}
Consider $X$ a separable Banach manifold, $Y$ a separable Hilbert manifold, $\mathcal{U}\subset X\times Y$ an open subset and $G$ a finite--dimensional Lie group with an action $\mu:G\times Y\to Y$ by diffeomorphisms. Suppose that $\mathcal U$ is $G$--invariant and that $f$ is $G$--invariant in the second variable. Suppose also the existence of submanifolds $Y_2\subset Y_1\subset Y$ such that all critical points of $f_x=f(x,\cdot\,)$ are contained in $Y_2$ for all $x\in\Pi(U)$, and for all $y\in Y_2$, the subspace $\mathcal D_y$ of $T_yY_1$ is well--defined.\footnote{See Remark~\ref{re:weakreg} and \eqref{eq:dgammaspangamma}.} Assume the existence of a generalized slice $(\mathfrak U,\{S_n\}_{n\in\N})$ for the action of $G$ on $Y$ with respect to $f_x$, for all $x\in\Pi(U)$. Consider $\mathfrak{M}=\left\{(x,y)\in\mathcal U:\frac{\partial f}{\partial y}(x,y)=0\right\}$, and $\mathfrak C$ a subset of distinguished pairs $(x,y)\in\mathfrak M$. Suppose the following conditions hold:
\begin{itemize}
\item[(eq-i)] for every $(x_0,y_0)\in\mathfrak M$, the Hessian $$\frac{\partial^2 f}{\partial y^2}(x_0,y_0):T_{y_0}Y\la T_{y_0}Y^*\cong T_{y_0}Y$$ is a (self--adjoint) Fredholm operator;
\item[(eq-ii)] for every $(x_0,y_0)\in\mathfrak C$, for all $w\in\ker\left[\frac{\partial^2 f}{\partial y^2}(x_0,y_0)\right]\setminus\mathcal D_{y_0}$, there exists $v\in T_{x_0}X$ such that $$\frac{\partial^2 f}{\partial x \partial y}(x_0,y_0)(v,w)\neq 0.$$
\end{itemize}
For each $x\in X$, let $\mathcal{U}_x=\{y\in Y:(x,y)\in\mathcal{U}\}$, $\mathfrak C_x=\{y\in Y:(x,y)\in\mathfrak C\}$ and $f_x(y)=f(x,y)$ for all $y\in\mathcal U_x$. Then the following is a generic subset of $X$,
\begin{equation*}
\mathcal G_{\mathfrak C}=\{x\in X: \mbox{ all } y\in\mathfrak C_x\mbox{ are }G\mbox{--nondegenerate for } f_x:\mathcal{U}_x\to\R\}.
\end{equation*}
\end{eqgen}

\begin{proof}
Once more, let $$\mathcal U_n=\{(x,y)\in X\times S_n:\{x\}\times G(y)\cap\mathcal U\neq\emptyset\},$$ and consider $\widetilde f_n:\mathcal U_n\to\R$ the restrictions $f|_{\mathcal U_n}$. From Lemma~\ref{le:transvopen}, applied to each $\frac{\partial\widetilde f_n}{\partial y}:\mathcal U_n\to TY^*$, the following is an open subset of $\mathfrak M$,
\begin{equation*}
\mathfrak A_n=\left\{(x,y)\in\mathcal U_n\cap\mathfrak M:\frac{\partial\widetilde f_n}{\partial y}:\mathcal U_n\to TY^* \mbox{ is transverse to }\mathbf 0_{TY^*} \mbox{ at } (x,y)\right\}.
\end{equation*}
Thus, there exists $\mathcal V_n$ an open subset in $\mathcal U_n$, such that $\mathcal V_n\cap\mathfrak M=\mathfrak A_n$. Consider $f_n:\mathcal V_n\to\R$ the restrictions $\widetilde f_n|_{\mathcal V_n}$.

Let us verify that each $f_n:\mathcal V_n\to\R$ satisfies the hypotheses of the Abstract Genericity Criterion~\ref{crit:abstractgenericity}. Given $(x_0,y_0)\in\mathcal V_n$ such that $\frac{\partial f_n}{\partial y}(x_0,y_0)=0$ 
the same decomposition \eqref{eq:decompty0},
\begin{equation*}
T_{y_0}Y=T_{y_0}S_n\oplus\mathcal D_{y_0}
\end{equation*}
induces a decomposition of operators defined in $T_{y_0}Y$. Thus, there is a decomposition of $\frac{\partial f}{\partial y}(x_0,y_0):T_{y_0}Y\to\R$ as the direct sum of $\frac{\partial f_n}{\partial y}(x_0,y_0):T_{y_0}S_n\to\R$ and the null functional\footnote{Recall that, since $f$ is $G$--invariant in the second variable, as observed in Lemma~\ref{le:invkernel} the subspace $\mathcal D_{y_0}$ is contained in the kernel of $\frac{\partial f}{\partial y}(x_0,y_0)$, and hence also in the kernel of $\frac{\partial^2f}{\partial x\partial y}(x_0,y_0)$.} of $\mathcal D_{y_0}$. Hence, if $\frac{\partial f}{\partial y}(x_0,y_0)=0$, it follows that also $\frac{\partial f_n}{\partial y}(x_0,y_0)=0$.

From (eq-i), the Hessian $\frac{\partial^2f}{\partial y^2}(x_0,y_0)$ is a Fredholm operator. Using \eqref{eq:decompty0}, it decomposes as \eqref{eq:hessfn},
\begin{equation*}
\frac{\partial^2 f_n}{\partial y^2}(x_0,y_0):T_{y_0}S_n\la T_{y_0}S_n^*\cong T_{y_0}S_n
\end{equation*}
and the null operator of $\mathcal D_{y_0}$. Thus, $\frac{\partial^2 f_n}{\partial y^2}(x_0,y_0)$ is given by the restriction of a Fredholm operator to a finite codimensional space, which is hence Fredholm. Therefore condition (i) of the Abstract Genericity Criterion~\ref{crit:abstractgenericity} holds.

As for condition (ii), from the above decomposition of $\frac{\partial^2f}{\partial y^2}(x_0,y_0)$, if $w\in\ker\left[\frac{\partial^2 f_n}{\partial y^2}(x_0,y_0)\right]\setminus\{0\}$, then\footnote{Here we use property (ii) of the generalized slice, see Definition~\ref{def:genslice}.} $w\in\ker\left[\frac{\partial^2 f}{\partial y^2}(x_0,y_0)\right]\setminus\mathcal D_{y_0}$. Since $(x_0,y_0)\in\mathfrak A_n$ it follows that there exists $v\in T_{x_0}X$ such that $\frac{\partial^2 f_n}{\partial x\partial y}(x_0,y_0)(v,w)\neq 0$. Using again \eqref{eq:decompty0}, also $$\frac{\partial^2 f}{\partial x\partial y}(x_0,y_0):T_{x_0}X\times T_{y_0}Y\la\R$$ decomposes as the sum of $\frac{\partial^2 f_n}{\partial x\partial y}(x_0,y_0)$ an the null functional of $T_{x_0}X\oplus\mathcal D_{y_0}$. Hence, $\frac{\partial^2f_n}{\partial x\partial y}(x_0,y_0)(v,w)\neq0$, proving that condition (ii) also holds.

Therefore we may apply the Abstract Genericity Criterion~\ref{crit:abstractgenericity} to each $f_n:\mathcal V_n\to\R$, obtaining genericity of $$\mathcal G_n=\big\{x\in X:f_n(x,\cdot\,)\mbox{ is Morse}\big\}$$ in $X$, for all $n\in\N$. Consider the intersection $\mathcal G=\bigcap_{n\in\N}\mathcal G_n$. Since this is the intersection of a countable family of generic subsets of $X$, Lemma~\ref{le:gencountintersect} gives that $\mathcal G$ is generic in $X$. Proposition~\ref{prop:morsegmorse} gives that\footnote{Here we use property (i) of the generalized slice, see Definition~\ref{def:genslice}.} $f$ is $G$--Morse if and only if $f_n$ is Morse for all $n\in\N$, hence $\mathcal G$ is the set of parameters $x\in X$ such that $f_x$ is $G$--Morse. The set $\mathcal G_{\mathfrak C}$ clearly contains $\mathcal G$ and hence, from Remark~\ref{re:stupidremark}, $\mathcal G_{\mathfrak C}$ is generic in $X$.
\end{proof}

\chapter{Periodic geodesics and the Bumpy Metric Theorem}
\label{chap5}

In this chapter we prove the first genericity result of the text, a non compact semi--Riemannian version of the Bumpy Metric Theorem. The celebrated Bumpy Metric Theorem is one of the central results in the theory of generic properties of geodesic flows, and several applications and generalizations are present in the literature. 

Although the Bumpy Metric Theorem~\ref{thm:bumpy} is a result of independent interest, it is also employed in the proofs of other results in the next chapter, as Theorem~\ref{thm:bigone}, that gives new generic properties regarding nondegeneracy of semi--Riemannian geodesics with general endpoints conditions. In fact, the main reason for the importance of the classic Bumpy Metric Theorem is that it is keystone for several other genericity results of geodesic flows, as for instance the ones established in \cite{ballmann,contpater,klitak}.

Let us begin by recalling the definition of {\em bumpy} metric, and providing a variational characterization of this property.

\begin{definition}\label{def:bumpy}
A semi--Riemannian metric $g\in\met_\nu^k(M)$ is {\em bumpy}\index{Bumpy metric}\index{Metric!bumpy} if for every periodic $g$--geodesic $\gamma:S^1\to M$, the only nontrivial periodic Jacobi fields along $\gamma$ are constant multiples of $\dot\gamma$.
\end{definition}

Clearly, there is a variational characterization of this fact, using the $G$--invariant notion of degeneracy introduced in Chapter~\ref{chap35}, see Definition~\ref{def:gmorse}. In this context, the action involved is the reparameterization action\footnote{See Example~\ref{ex:s1h1}.} of $S^1$ on $H^1(S^1,M)$, which leaves invariant\footnote{See Example~\ref{ex:mus1h1inv}.} the (second variable of the) generalized energy functional for periodic curves \eqref{eq:efunctper},
\begin{equation*}
E:\mathcal A_{g_\mathrm A,\nu}\times H^1(S^1,M)\ni (g,\gamma)\longmapsto\tfrac12\int_{S^1} g(\dot\gamma,\dot\gamma)\;\dd z\in\R,
\end{equation*}
where $\mathcal A_{g_\mathrm A,\nu}$ is given by \eqref{eq:agnu}. Recall this is an open subset of an affine separable Banach space formed by semi--Riemannian metrics, that depends on the choice of an auxiliary semi--Riemannian metric $g_\mathrm A$ of index $\nu$, see Proposition~\ref{prop:affineworks}. 

Analogously to what was discussed in Section~\ref{sec:genenfunc}, more precisely in Proposition~\ref{prop:fck}, this is a $C^k$ functional. In addition, the first variable should be thought of as a parameter, and we will frequently denote $E_g(\gamma)=E(g,\gamma)$. If $\frac{\partial E}{\partial\gamma}(g_0,\gamma_0)=0$, then $\gamma_0$ is a periodic $g_0$--geodesic, see Proposition~\ref{prop:critgenenfunc}. This setup allows to give the following obvious characterization of bumpy metrics in $\mathcal A_{g_\mathrm A,\nu}$.

\begin{lemma}\label{le:bumpychar}
A metric $g\in\A_{g_\mathrm A,\nu}$ is bumpy if and only if $$E_g:H^1(S^1,M)\la\R$$ is a $S^1$--Morse functional.
\end{lemma}

The Bumpy Metric Theorem states that bumpy metrics on $M$ form a \emph{generic} subset of $\A_{g_\mathrm A,\nu}$. In view of Lemma~\ref{le:bumpychar}, this is equivalent to the functional $E_g$ being $S^1$--Morse for a generic parameter $g\in\A_{g_\mathrm A,\nu}$. Before getting to some technical lemmas necessary for the proof, let us give a brief overview of the history of this theorem.

The Riemannian version of the Bumpy Metric Theorem is attributed to Abraham \cite{abraham} in 1970, who was the first to formulate its statement and to use the term {\em bumpy}. In fact, the result was announced in 1968 by Abraham, at a conference on global analysis at Berkeley. It seems to be among the first of a numerous sequence of theorems on the {\em generic} behavior of dynamical systems, particularly geodesic flows. The main motivation of Abraham to introduce the concept of bumpy metrics is due to the classic conjecture that every compact Riemannian manifold oughts to admit infinitely many geometrically distinct periodic geodesics. It is claimed in \cite{abraham} that for metrics with non discrete isometry groups, this result is {\em obvious}, and hence the interest in the {\em generic} case of minimal symmetry. In addition, it is conjectured that every bumpy metric on a compact manifold admits infinitely many distinct periodic geodesics. This conjecture on bumpy metrics was proved to hold by Rademacher \cite{rademacher} in 1989, hence the classic conjecture {\em generically} holds.

Strangely enough, the first complete proof of the Bumpy Metric Theorem is due to Anosov \cite{anosov} in 1982, more than ten years after Abraham's paper and originally published in Russian. Anosov \cite{anosov} gives concrete examples of why some attempted proofs of the Bumpy Metric Theorem by Klingenberg \cite{klitak} in 1972 are incorrect. In the mean time, several authors began to use the Bumpy Metric Theorem to establish other genericity results for geodesic flows and more general dynamical systems, among which Klingenberg \cite{kli,klitak} himself. As pointed out by Anosov, the main problem with Klingenberg's proof in \cite{klitak} is that it employs a perturbation argument along a specific degenerate periodic geodesic, making it nondegenerate. However, it relinquishes the effect that this metric change might have, possibly causing other periodic geodesics to degenerate. Among more recent significative extensions of the Bumpy Metric Theorem, we highlight the results of Gon\c{c}alves Miranda \cite{GonMir} on genericity of periodic trajectories in the context of magnetic flows on a surface, which allows to establish an extension of the Kupka--Smale Theorem.

Despite giving a quite cumbersome proof with fairly involved technical arguments, Anosov \cite{anosov} claims that there might have been other correct proofs before that date, however so cumbersome that were not published. Anosov's proof and subsequent applications employ arguments dating from the time of Poincar\'e, Birkhof and Toponogov. However, in the eighties, there already were more modern tools to approach this type of problem, for instance the ones developed in the so--called Ljusternik--Schnirelmann theory. In this sense, Anosov \cite{anosov} claims having worked on a proof using {\em tubular neighborhoods}, inspired by Peixoto \cite{peixoto}, but that revealed being much more cumbersome that the classic approach adopted.

In the last years, the totally unexplored semi--Riemannian version of several important generic properties of geodesic flows came to the attention of Biliotti, Javaloyes and Piccione \cite{biljavapic}. The motivation to explore generic properties of this type of dynamical systems come mainly from Lorentzian geometry and its implications on general relativity, but also from general theory of semi--Riemannian manifolds and Morse theory, due to recent results of Abbondandolo and Majer \cite{AbbMej2,AbbMaj,AbbMaj2}.

This inaugurating article of Biliotti, Javaloyes and Piccione \cite{biljavapic} in 2009 paved the way to several further investigations of genericity of nondegeneracy in semi--Riemannian geodesic flows. We stress that, unlike the Riemannian case, for non necessarily positive--definite metrics there is a significative qualitative change of the structure of the geodesic flow when passing from negative to positive values of energy. At that point, the most natural candidate to be extended to this semi--Riemannian realm was the Bumpy Metric Theorem, since it remains a central result in the area and proved to have numerous applications. Nevertheless, the transversality techniques used in \cite{biljavapic} alone proved not enough to give a complete answer to this problem, due to possible presence of {\em strongly degenerate}\footnote{See Section~\ref{sec:strongdeg}.} periodic geodesics. This issue was solved in the recent preprint \cite{biljavapic2} of the same authors in 2010, where a complete proof of the semi--Riemannian Bumpy Metric Theorem is given, using a combination of the original approach of Anosov \cite{anosov} and transversality techniques of \cite{biljavapic}.

Among several usages of this semi--Riemannian Bumpy Metric Theorem of \cite{biljavapic2}, we mention the genericity results of nondegeneracy of semi--Riemannian geodesics under general endpoints conditions, in the recent paper of Bettiol and Giamb\`o \cite{metmna}. To carry out the main applications, a subtle refinement of the semi--Riemannian Bumpy Metric Theorem of \cite{biljavapic2} is needed. Namely, a {\em non compact} reformulation is required. Up to date, all versions of the Bumpy Metric Theorem in \cite{abraham,anosov,biljavapic2,klitak} were stated for {\em compact} manifolds. Our proof of the Bumpy Metric Theorem~\ref{thm:bumpy} for non necessarily compact semi--Riemannian manifolds details the tools needed for this refinement, which turns out to follow almost immediately from the preceding versions of the Bumpy Metric Theorem. We also stress the importance of removing any compactness assumptions in the semi--Riemannian version of this result, in face of topological obstructions to the existence of metrics of given index in {\em compact} manifolds, as studied in Section~\ref{sec:topobst}. For instance, from Propositions~\ref{prop:noncompactlorentz} and~\ref{prop:existlorentzian}, every non compact manifold admits a Lorentzian metric, while for compact manifolds this only holds under the additional hypothesis that the Euler class vanishes.

A natural challenge to extend any such genericity results to the non compact case is that there is no canonical separable Banach space structure on the space of semi--Riemannian metrics on a non compact manifold. Thus, we use the tools developed in Chapter~\ref{chap3} regarding this structure on {\em smaller} subsets of metrics of the form \eqref{eq:agnu}, namely metrics that are asymptotically equal to some fixed auxiliary metric $g_\mathrm A$ of the same index $\nu$, see Proposition~\ref{prop:affineworks}. Genericity results are then formulated relatively to such open subsets, that have all necessary structures to carry out the analysis in the sequel.

Let us give an idea of the possible approaches to the bumpy problem. There are essentially two ways of characterizing nondegeneracy of periodic geodesics, corresponding to the \emph{dynamical} and the \emph{variational} approaches. On the one hand, the dynamical approach consists in studying periodic geodesics as fixed points for the {\em Poincar\'e map}, or {\em first recurrence map}. On the other hand, the variational approach consists in studying geodesics as critical points of the energy functional defined in the free loop space.

Denote by $T^1M$ the unit tangent bundle of $M$ relatively to an auxiliary Riemannian metric $g_\mathrm R$, see Definition~\ref{def:unittangentbundle}. For each $v\in T^1M$, let $\gamma_v:\left[0,+\infty\right[\to M$ be the unique $g$--geodesic with $\dot\gamma(0)=v$. From the dynamical viewpoint, nondegeneracy of a periodic geodesic $\gamma_0:\left[0,+\infty\right[\to M$ of period $\omega$ means that the map $$\R_+\times T^1M\ni(t,v)\longmapsto \big(v,\dot\gamma_v(t)\big)\in T^1M\times T^1M$$ is transverse to the diagonal of $T^1M\times T^1M$ at $(\omega,v_0)$, where $v_0=\dot{\gamma_0}(0)$. The proof of the Riemannian Bumpy Metric Theorem by Anosov \cite{anosov} uses this approach, and it employs the transversality theorem.

The dynamical approach does not work well when considering semi--Riemannian metrics, starting from the observation that even the notion of unit tangent bundle itself is not very meaningful in semi--Riemannian geometry. Distinguishing \emph{causal} notions of unit tangent bundles, i.e., timelike, lightlike and spacelike, is also not very meaningful when dealing with families of metrics.

The proof of the semi--Riemannian Bumpy Metric Theorem by Biliotti, Javaloyes and Piccione \cite{biljavapic2} uses a variational approach. From this viewpoint, nondegeneracy for a periodic geodesic $\gamma$ means that $\gamma$ is a nondegenerate critical point of the energy functional \eqref{eq:efunctper}, in the invariant sense of Definition~\ref{def:gmorse}. More precisely, nondegeneracy means that the kernel of its index form \eqref{eq:indexform} is one--dimensional, consisting only of constant multiples of the tangent field $\dot\gamma$, i.e., the distribution $\mathcal D_\gamma$ of Definition~\ref{def:dy}. In order to deal with the invariance of the energy functional under the action of $S^1$, see Example~\ref{ex:mus1h1inv}, the Equivariant Genericity Criterion~\ref{crit:equivariantgenericity} is used together with the construction of a generalized slice for the action of $S^1$ on $H^1(S^1,M)$ given by Proposition~\ref{prop:existslice}. Recall that this only holds in the weak regularity context described in Remark~\ref{re:weakreg}.

Using such techniques, in Section~\ref{sec:weakbumpy} we prove a Weak Bumpy Metric Theorem~\ref{thm:weakbumpy}, that guarantees nondegeneracy only of prime geodesics, see Definition~\ref{def:prime}. Genericity of nondegeneracy of iterates does not follow from this equivariant variational setup due to a subtle technical problem, that will be discussed in Section~\ref{sec:strongdeg}. In order to deal with iterates, we follow the ingenious idea of Anosov \cite{anosov}, also used by Biliotti, Javaloyes and Piccione \cite{biljavapic2} with suitable modifications that make it work in the non compact semi--Riemannian case. For this, in Section~\ref{sec:iterates} we introduce families of metrics $\M_K(a,b)$ parameterized by two positive real numbers $a,b\in\R$ that correspond to the \emph{period} and to the \emph{minimal period} of periodic geodesics, and a compact subset $K$ of $M$. For the semi--Riemannian extension, the notion of period (which is meaningless in the case of lightlike geodesics\footnote{Recall Definition~\ref{def:causalchar}.}) is replaced by notions of \emph{energy} relatively to an auxiliary Riemannian metric $g_\mathrm R$. In the final Section~\ref{sec:bumpy}, we prove that the set of metrics that are bumpy for geodesics in $K$ corresponds to the countable intersection $\bigcap_{n\ge1}\M_K(n,n)$, and a proof of its genericity is obtained by showing that each $\M_K(a,b)$ is open and dense in the set of metrics. Finally, to conclude the Bumpy Metric Theorem~\ref{thm:bumpy}, we use a simple exhaustion by compacts argument.

\section{Weak Bumpy Metric Theorem}
\label{sec:weakbumpy}

As mentioned above, prime and iterate geodesics will be treated separately. In this section, we prove a weak version of the Bumpy Metric Theorem~\ref{thm:bumpy}, regarding prime geodesics. This result is a subtle generalization of \cite[Proposition 3.4]{biljavapic2}, in that it does not require compactness of the manifold. Finally, it will be later used to establish genericity of all periodic geodesics, in Section~\ref{sec:bumpy}.

\begin{weakbumpythm}\label{thm:weakbumpy}\index{Theorem!Weak Bumpy Metric}\index{Bumpy Metric Theorem!Weak}
Let $M$ be a smooth $m$--dimensional manifold and fix $\mathds{E}$ a separable $C^k$ Whitney type Banach space of sections of $TM^*\vee TM^*$ that tend to zero at infinity, with $k\geq 3$. Fix $\nu\in\{0,\dots,m\}$ an index and let $g_\mathrm A\in\met_\nu^k(M)$ be such that $$\sup_{x\in M}\|g_\mathrm A(x)^{-1}\|_\mathrm R<+\infty.$$ Then the following is a generic subset of $\A_{g_\mathrm A,\nu}$\footnote{Recall Proposition~\ref{prop:affineworks}.}
$$\mathcal{G}_*(M)=\left\{g\in\A_{g_\mathrm A,\nu}:\text{ all {\em prime }} g\text{--geodesics are } S^1\mbox{--nondegenerate}\right\}.$$
\end{weakbumpythm}

\begin{proof}
The proof is in great part adapted from the proof of \cite[Proposition 3.4]{biljavapic2}. More precisely, we will apply the Equivariant Genericity Criterion~\ref{crit:weakequivariantgenericity} to the generalized energy functional for periodic curves \eqref{eq:efunctper}, $$E:\A_{g_\mathrm A,\nu}\times H^1(S^1,M)\ni (g,\gamma)\longmapsto E_g(\gamma)=\tfrac12\int_{S^1} g(\dot\gamma,\dot\gamma)\;\dd z\in\R,$$ which is invariant under the reparameterization action of $S^1$ on $H^1(S^1,M)$, see Example~\ref{ex:mus1h1inv}. Let $\mathcal U=\A_{g_\mathrm A,\nu}\times H^1(S^1,M)$. Recall that this is a $C^k$ functional and if a pair $(g_0,\gamma_0)\in\mathcal U$ satisfies $\frac{\partial E}{\partial\gamma}(g_0,\gamma_0)=0$, then $\gamma_0$ is a periodic $g_0$--geodesic, see Proposition~\ref{prop:critgenenfunc}. Consider also the set of distinguished critical points to be the set of prime geodesics, $$\mathfrak C=\left\{(g_0,\gamma_0)\in\mathcal U:\gamma_0\in H_*^1(S^1,M)\mbox{ and }\frac{\partial E}{\partial\gamma}(g_0,\gamma_0)=0\right\}.$$ From Lemma~\ref{le:mus1h1diffeos}, Propositions~\ref{prop:affineworks} and~\ref{prop:existslice} and Corollary~\ref{cor:h1s1}, the above context satisfies the hypotheses of the Equivariant Genericity Criterion~\ref{crit:weakequivariantgenericity}.

Let us verify that conditions (eq-i) and (eq-ii) hold. Condition (eq-i) follows from Proposition~\ref{prop:fredholmness} setting $\p=\Delta$. To verify condition (eq-ii), we use a local perturbation argument. This condition asserts that given $(g_0,\gamma_0)\in\mathfrak C$, for all $J\in\ker\left[\frac{\partial^2 E}{\partial\gamma^2}(g_0,\gamma_0)\right]\setminus\operatorname{span}\,\dot{\gamma_0}$, there must exist $h\in T_{g_0}\A_{g_\mathrm A,\nu}$ such that the mixed derivative \eqref{eq:mixedderivative}, given by $$\frac{\partial^2 E}{\partial g\partial\gamma}(g_0,\gamma_0)(h,J)=\int_{S^1} h(\dot{\gamma_0},\D J)+\tfrac{1}{2}\nabla h(J,\dot{\gamma_0},\dot{\gamma_0})\;\dd z$$ does not vanish. Notice that since $\gamma_0$ is prime, it has only a finite number of self intersections, see Proposition~\ref{prop:selfintersections}. From Lemma~\ref{le:parallelfinite}, since $J$ is not a multiple of $\dot{\gamma_0}$, the set of $z\in S^1$ such that $J$ is parallel to $\dot{\gamma_0}$ is finite. Thus, there exists an open nonempty connected subset $I\subset S^1$ such that
\begin{itemize}
\item[$I$--1:] $\gamma_0(I)\cap\gamma_0(S^1\setminus I)=\emptyset$;
\item[$I$--2:] $J$ is not parallel to $\dot{\gamma_0}$ at any time in $I$.
\end{itemize}

In order to construct the required $h\in\mathds{E}$ such that $\frac{\partial^2 E}{\partial g\partial\gamma}(g_0,\gamma_0)(h,J)\ne 0$, we apply Lemma~\ref{le:extension} to the vector bundle $TM^*\vee TM^*$. Let $U\subset M$ be any open subset containing $\gamma_0(I)$ such that
\begin{itemize}
\item[$U$:] $\gamma_0(t)\in U$ if and only if $t\in I$.
\end{itemize}
For instance, $U$ can be taken as the complement of $\gamma_0(S^1\setminus I)$. Let $H\in\sect^k(\gamma_0^*(TM^*\vee TM^*))$ be the identically null section and choose any $K\in\sect^{k}(\gamma_0^*(TM^*\vee TM^*))$ that satisfies $$K(\dot{\gamma_0},\dot{\gamma_0})\geq 0\;\;\mbox{ and }\;\;\int_{S^1} K(z)(\dot{\gamma_0}(z),\dot{\gamma_0}(z))\;\dd z>0,$$ for instance, $K(z)=g_\mathrm R(\gamma_0(z))$. Reducing the size of $I$ if necessary, we may assume that the result of Lemma~\ref{le:extension} holds. This gives a globally defined section $h\in\sect^k(TM^*\vee TM^*)$ with compact support contained in $U$ such that
\begin{equation*}
h(\gamma_0(z))=0\;\;\mbox{ and }\;\;\nabla_{J(z)} h=K(z), \quad\mbox{ for all } z\in I.
\end{equation*}
Clearly $h\in\mathds E$, since all $C^k$ sections of $E$ with compact support are in $\mathds E$. Finally, from the above construction,
\begin{eqnarray*}
\frac{\partial^2 E}{\partial g\partial\gamma}(g_0,\gamma_0)(h,J) &=& \int_{S^1} h(\dot{\gamma_0},\D J)+\tfrac{1}{2}\nabla h(J,\dot{\gamma_0},\dot{\gamma_0})\;\dd z\\
&=&\tfrac{1}{2}\int_{S^1} K(z)(\dot{\gamma_0}(z),\dot{\gamma_0}(z))\;\dd z\\
&>&0.
\end{eqnarray*}
Therefore, condition (eq-ii) holds.

The Equivariant Genericity Criterion~\ref{crit:weakequivariantgenericity} then gives genericity of the set $\mathcal G_*(M)$ of $g\in \A_{g_\mathrm A,\nu}$ such that all prime $g$--geodesics are $S^1$--nondegenerate, concluding the proof.
\end{proof}

\section{Strongly degenerate geodesics}
\label{sec:strongdeg}

In this section, we briefly explain the reason why the above local perturbation argument employed to verify condition (ii) of the Equivariant Genericity Criterion~\ref{crit:weakequivariantgenericity} fails in the case of iterate geodesics. Namely, this is due to possible existence of a particularly degenerate class of periodic geodesics, called {\em strongly degenerate} geodesics. Such geodesics will also play a special role in Chapter~\ref{chap6} when dealing with GECs that admit periodic geodesics.

\begin{definition}\label{def:defstdeg}
Let $\gamma:[0,1]\rightarrow M$ be a $g$--geodesic. Then $\gamma$ is said to be \emph{strongly degenerate}\index{Geodesic!strongly degenerate}\index{Strongly degenerate geodesic} if there exists an integer $k\ge2$ such that:
\begin{itemize}
\item[(a)] $\gamma\left(t+\tfrac{i}{k}\right)=\gamma(t)$, for all $i\in\{0,\dots,k-1\}$ and $t\in \left[0,\tfrac{1}{k}\right[$;
\item[(b)] $\gamma$ admits a Jacobi field $J\ne0$, such that $\sum\limits_{i=0}^{k-1} J\left(t+\tfrac{i}{k}\right)=0$, for all $t\in \left[0,\tfrac{1}{k}\right[$.
\end{itemize}
\end{definition}

\begin{figure}[htf]
\vspace{-0.5cm}
\begin{center}
\includegraphics[scale=0.5]{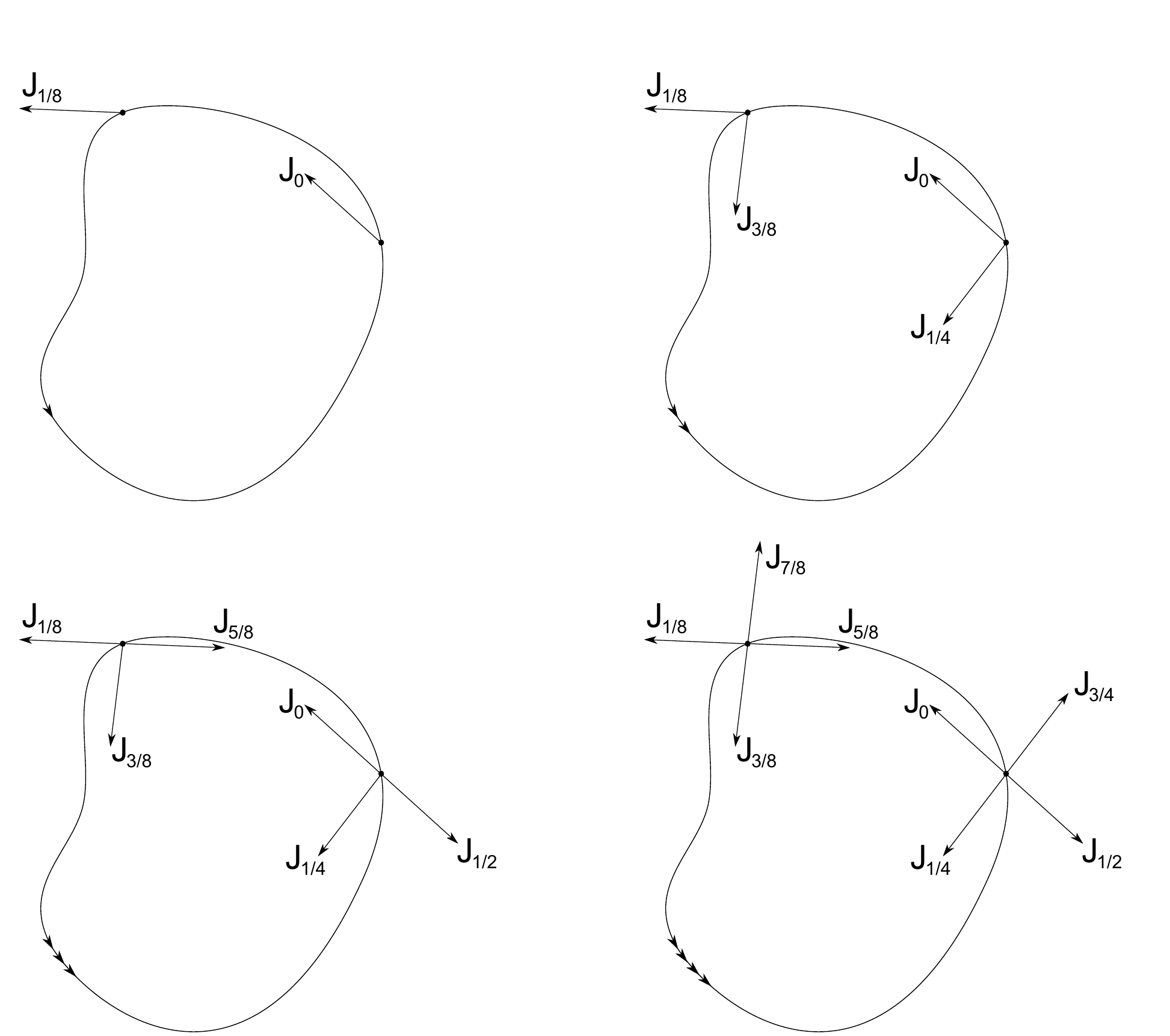}
\begin{pgfpicture}
\end{pgfpicture}
\end{center}
\caption{Example of a strongly degenerate geodesic with $k=4$, and a Jacobi field $J$ satisfying (b), illustrated at $t=0$, $t=\tfrac18$ and iterates.}\label{fig:stdeg}
\end{figure}

Observe that if $\gamma$ is strongly degenerate, then it is automatically a periodic geodesic\footnote{Recall that from item (a), $\gamma$ is a {\em geodesic loop}. In order to verify it is indeed a {\em periodic geodesic}, one has to check that $\dot\gamma$ is periodic, with the same period. This is easily done deriving the condition (a) at appropriate values of $t$.} with period $\omega=\tfrac{1}{k}$, for some $k\ge2$. More precisely, it is an iterate geodesic, see Definition~\ref{def:prime}. This allows to consider $\gamma$ defined in the domain $$S^1\cong\frac{[0,1]}{\big\{0,\tfrac1k,\ldots,\tfrac{k-1}{k},1\big\}},$$ nevertheless when dealing with strongly degenerate geodesics, we will prefer to adopt $[0,1]$ as its domain. Such convention will prove handy to deal with strongly degenerate $(g,\p)$--geodesics and at same time with non periodic $(g,\p)$--geodesics in Chapter~\ref{chap6}.

\begin{remark}\label{re:whenpertfails}
The transversality condition (ii) of the abstract genericity criteria of Chapter~\ref{chap4} trivially fails in the presence of a strongly degenerate geodesic. Indeed, this is the {\em only case} in which a local perturbation argument similar to the one used in the proof of the Weak Bumpy Metric Theorem~\ref{thm:weakbumpy} does not apply. Namely, if $\gamma_0$ is a strongly degenerate $g_0$--geodesic, then it admits a nontrivial Jacobi field $J$ which satisfies (b) of Definition~\ref{def:defstdeg}. For this $J$, the right--hand side of \eqref{eq:mixedderivative} is identically null {\em for any section} $h$ of $TM^*\vee TM^*$. Therefore, \eqref{eq:mixedderneq0} trivially fails.

This leads to the development of alternative methods to deal with the strongly degenerate geodesics, in order to prove the complete Bumpy Metric Theorem~\ref{thm:bumpy}. In Chapter~\ref{chap6}, it will be seen that such methods can be avoided in the case of $(g,\p)$--geodesics with a subtle trick employing the Bumpy Metric Theorem~\ref{thm:bumpy}. Moreover, it will also be proved in Theorem~\ref{thm:transvholds} that the local perturbation argument above can be adapted to non periodic geodesics, provided they are not parts of a strongly degenerate geodesic. This theorem will explore, in its full generality, the range of the local perturbation argument introduced in the above proof of the Weak Bumpy Metric Theorem~\ref{thm:weakbumpy}.
\end{remark}

We conclude this section with two final results on strongly degenerate geodesics.

\begin{proposition}\label{prop:stronglydegenerate}
Suppose $\gamma:[0,1]\to M$ is a strongly degenerate $g$--geodesic. Then $\gamma$ is an $S^1$--degenerate critical point of \eqref{eq:efunctper}, i.e., $\gamma$ admits a nontrivial periodic Jacobi field $J$ that is not a constant multiple of $\dot\gamma$.
\end{proposition}

\begin{proof}
Take $J$ a Jacobi field as in (b). Then $J$ is not everywhere parallel to $\dot\gamma$, otherwise it would follow that $\dot\gamma=0$. Comparing condition (b) at $t=0$ and $t=\tfrac{1}{k}$, one obtains that
\begin{equation}\label{eq:sdeg1}
J(0)-J(1)=\sum_{i=0}^{k-1} J\left(\tfrac{i}{k}\right)-\sum_{i=0}^{k-1} J\left(\tfrac{1}{k}+\tfrac{i}{k}\right)=0.
\end{equation}
Moreover, from (b), $V(t)=\sum_{i=0}^{k-1} J\left(t+\tfrac{i}{k}\right)$ is the identically null vector field. Thus $\D V(t)=0$ for all $t\in [0,1]$. As a result, analogously to \eqref{eq:sdeg1},
\begin{equation}\label{eq:sdeg2}
\D J(0)-\D J(1)=\D V(0)-\D V(\tfrac{1}{k})=0.
\end{equation}
From \eqref{eq:sdeg1} and \eqref{eq:sdeg2}, it follows that $J$ is a periodic Jacobi field along $\gamma$ with respect to $g$. Thus, the same $J$ that degenerates $\gamma$ as a $g$--geodesic is also periodic and is not a constant multiple of $\dot\gamma$, hence in $\mathcal D_\gamma$. This concludes the proof, since there are no closed complements of $\mathcal D_\gamma$ where the restriction of \eqref{eq:indexform} gives an isomorphism, for its kernel intersects every such complement non trivially in $J$, see Definition~\ref{def:gmorse}.
\end{proof}

\begin{proposition}\label{prop:tangentisstdeg}
Let $\gamma:[0,1]\rightarrow M$ be a periodic $g$--geodesic with period $\tfrac{1}{k}$, for some $k\ge2$. Suppose $\gamma$ admits a nontrivial Jacobi field $J$ such that there exists $\lambda:[0,1]\to\R$ with
\begin{equation}\label{eq:sommaJ}
\sum\limits_{i=0}^{k} J\left(t+\tfrac{i}{k}\right)=\lambda(t)\dot{\gamma}(t), \quad t\in \left[0,\tfrac{1}{k}\right[.
\end{equation}
Then $\gamma$ is strongly degenerate.
\end{proposition}

\begin{proof}
By adding a suitable multiple of $\dot\gamma$ to $J$, that depends on $\lambda$ and $k$, one easily obtains a Jacobi field along $\gamma$ that satisfies condition (b) of Definition~\ref{def:defstdeg}. Therefore, in this case $\gamma$ is strongly degenerate.
\end{proof}

\section{Iterate geodesics}
\label{sec:iterates}

In this section, we study the problem of $S^1$--nondegeneracy for iterates, following closely the approach of Biliotti, Javaloyes and Piccione \cite{biljavapic2}, inspired by Anosov \cite{anosov}. The methods developed in the sequel take into account possible presence of strongly degenerate geodesics, giving the necessary tools to complete the proof of the semi--Riemannian Bumpy Metric Theorem in its full generality.

Recall that a periodic curve is called an {\em iterate} if it has nontrivial isotropy with respect to the reparameterization action \eqref{eq:mus1h1}, see Figure~\ref{fig:iterates}. From Lemma~\ref{le:cyclicgroup}, this isotropy group is a finite cyclic group, hence isomorphic to $\Z_n$. Let us denote $\gamma^{(n)}\in H^1(S^1,M)$ such an iterate curve, with $\# S^1_{\gamma^{(n)}}=n$. This means that $\gamma^{(n)}$ is given as $n$--fold iteration of a {\em prime} curve, i.e., there exists $\gamma\in H^1_*(S^1,M)$ such that $$\gamma^{(n)}(z)=\gamma(z^n), \quad z\in S^1.$$ Notice that, in this case, there is a clear relation between the energies of the iterate $\gamma^{(n)}$ and its prime generator $\gamma$. Notice first that
\begin{equation}\label{eq:gammangamma}
\begin{aligned}
\dot{\left(\gamma^{(n)}\right)}(z) &= \dd\gamma^{(n)}(z)iz\\
&= \dd\gamma(z^n)nz^{n-1}iz\\
&= n\dd\gamma(z^n)iz^n\\
&= n\dot\gamma(z^n),
\end{aligned}
\end{equation}
hence,
\begin{equation}\label{eq:ene}
\begin{aligned}
E_\mathrm R(\gamma^{(n)}) &= \tfrac12\int_{S^1} g_\mathrm R(n\dot\gamma(z^n),n\dot\gamma(z^n))\;\dd z\\
&= \tfrac{n^2}{2}\int_{S^1} g_\mathrm R(\dot\gamma(z^n),\dot\gamma(z^n))\;\dd z\\
&= \tfrac{n^2}{2}\int_{S^1} g_\mathrm R(\dot\gamma(w),\dot\gamma(w))\;\dd w\\
&= n^2 E_\mathrm R(\gamma),
\end{aligned}
\end{equation}
where the third equality holds by a simple change of variables\footnote{Notice that if $f:S^1\to\R$ is continuous, then curiously for any $n\in\N$, $$\int_{S^1}f(w)\;\dd w=\int_{S^1}f(z^n)\;\dd z.$$} $w=z^n$.

\begin{remark}
Obviously, from \eqref{eq:gammangamma} it is also possible to infer that $L_\mathrm R(\gamma^{(n)})=nL_\mathrm R(\gamma)$, but we shall deal with $E_\mathrm R$ rather than $L_\mathrm R$.
\end{remark}

Inspired by the relation \eqref{eq:ene}, we define two different ways of measuring the energy of an iterate with respect to the fixed auxiliary Riemannian metric $g_\mathrm R$, as follows.

\begin{definition}
Let $\gamma\in H^1(S^1,M)$. Define the {\em total energy}\index{Energy!total} of $\gamma$ to be
\begin{equation}
\ea(\gamma)=E_\mathrm R(\gamma)=\tfrac12\int_{S^1} g_\mathrm R(\dot\gamma,\dot\gamma)\;\dd z,
\end{equation}
and the {\em minimal energy}\index{Energy!minimal} of $\gamma$ to be
\begin{equation}
\eo(\gamma)=\frac{\ea(\gamma)}{(\# S^1_\gamma)^2}=\frac{1}{2(\# S^1_\gamma)^2}\int_{S^1} g_\mathrm R(\dot\gamma,\dot\gamma)\;\dd z,
\end{equation}
where $\# S^1_\gamma$ is the cardinality of the isotropy group of $\gamma$, see Lemma~\ref{le:cyclicgroup}. In this sense, \eqref{eq:ene} yields that the minimal energy $\eo(\gamma^{(n)})$ of an $n$--fold iterate is the total energy $\ea(\gamma)$ of the (unique) prime generator curve $\gamma\in H^1_*(S^1,M)$.

In other words, $\ea(\gamma)$ gives the $g_\mathrm R$--energy of {\em all} turns $\gamma$ makes while $\eo(\gamma)$ gives the $g_\mathrm R$--energy of only {\em one} turn of $\gamma$.
\end{definition}

\begin{remark}
Clearly, if $\gamma\in H^1_*(S^1,M)$ is prime, then $\ea(\gamma)=\eo(\gamma)$.
\end{remark}

\begin{remark}
According to the subscript $_\mathrm R$ notation being used throughout the text to highlight dependence\footnote{See, for instance, Remark~\ref{re:dependgr}.} on the choice of the auxiliary Riemannian metric $g_\mathrm R$, the quantities $\ea(\gamma)$ and $\eo(\gamma)$ should also carry a $_\mathrm R$, since they obviously depend on this choice. Nevertheless, for the sake of simplifying notation, we will deliberately omit this subindex and assume a given choice of auxiliary Riemannian metric $g_\mathrm R$.
\end{remark}

\begin{remark}
Analogously to Remark~\ref{re:dependgr}, we stress that although it is not desirable to have dependence on the choice of an auxiliary Riemannian metric for defining $\ea(\gamma)$ and $\eo(\gamma)$, this is the best possible setting for the desired applications. Namely, we must have $$0\leq\eo(\gamma)\leq\ea(\gamma),$$ for all $\gamma\in H^1(S^1,M)$, which would not hold replacing $g_\mathrm R$ with a semi--Riemannian metric, for instance.
\end{remark}

Henceforth, fix $\mathds{E}$ a separable $C^k$ Whitney type Banach space of sections of $TM^*\vee TM^*$ that tend to zero at infinity, with $k\geq 2$. Fix also $\nu\in\{0,\dots,m\}$ an index and let $g_\mathrm A\in\met_\nu^k(M)$ be such that $$\sup_{x\in M}\|g_\mathrm A(x)^{-1}\|_\mathrm R<+\infty.$$ Under these choices, $\A_{g_\mathrm A,\nu}$ is an open subset of an affine separable Banach space,\footnote{Recall Proposition~\ref{prop:affineworks}.} hence is a separable Banach manifold.

We now introduce families of metrics parameterized by two positive real numbers $a,b\in\R$ and a compact subset $K$ of $M$. These families are the natural extension of the families considered by Anosov \cite{anosov} and Biliotti, Javaloyes and Piccione \cite{biljavapic2} in their proof of the Bumpy Metric Theorem.

\begin{definition}\label{def:mkab}
Given the above choices of $\nu$ and $g_\mathrm A$, for each compact subset $K$ of $M$ and positive real numbers $0<a\leq b<+\infty$, define
\begin{equation*}
\M_K(a,b)=\left\{g\in\A_{g_\mathrm A,\nu}:\begin{array}{c}\mbox{ every periodic }g\mbox{--geodesic }\gamma\mbox{ with image in } K, \\ \eo(\gamma)\leq a\mbox{ and }\ea(\gamma)\leq b\mbox{ is }S^1\mbox{--nondegenerate}\end{array}\right\}
\end{equation*}
\end{definition}

\begin{lemma}\label{le:mabmab}
Given a compact $K\subset M$, if $a_1\leq a_2$ and $b_1\leq b_2$, then $$\M_K(a_2,b_2)\subset\M_K(a_1,b_1).$$
\end{lemma}

\begin{proof}
Immediate from Definition~\ref{def:mkab}.
\end{proof}

The following results on absence of short periodic geodesics and accumulation of degenerate periodic geodesics will be later used for our proof of the Bumpy Metric Theorem~\ref{thm:bumpy}.

\begin{lemma}\label{le:noshortgeods}
Given $g_0\in\A_{g_\mathrm A,\nu}$ and $K\subset M$ compact, there exists $r>0$ and an open neighborhood $\mathcal V$ of $g_0$ in $\A_{g_\mathrm A,\nu}$ such that for every $g\in\mathcal V$ no non constant periodic $g$--geodesics with image contained in $K$ have image also contained in a ball of $g_\mathrm R$--radius less than or equal to $r$. In particular, there exists $\widehat a>0$ such that for all $g\in\mathcal V$ and all prime $g$--geodesics $\gamma$ with image in $K$, $\ea(\gamma)\geq \widehat a$.
\end{lemma}

\begin{proof} 
Given any $p\in K$, there exists an open neighborhood $U_p$ of $p$ in $M$ and an open neighborhood $\mathcal V^p$ of $g$ in $\A_{g_\mathrm A,\nu}$ such that, for all $g\in\mathcal V^p$, the open subset $U_p$ is contained in a $g$--convex neighborhood of $p$, see Definition~\ref{def:normalradius}. By compactness of $K$, it can covered by a finite union $\{U_{p_i}\}_{i=1}^n$ of such open subsets. Let $r$ be the Lebesgue number of this open cover relatively to the metric induced by $g_{\mathrm R}$. It follows that every ball of $g_{\mathrm R}$--radius less than or equal to $r$ is contained in some $U_{p_i}$, and thus it cannot contain any non constant periodic $g$--geodesic for any $g\in\mathcal V=\bigcap_{i=1}^n\mathcal V^{p_i}$. This concludes the proof.
\end{proof}

\begin{lemma}\label{le:lemaconver}
Let $K\subset M$ be compact and $\{g_n\}_{n\in\N}$ be a sequence in $\A_{g_\mathrm A,\nu}$ converging to $g_\infty\in\A_{g_\mathrm A,\nu}$. Let $\{\gamma_n\}_{n\in\N}$ be curves with image contained in $K$ such that for every $n\in\N$, $\gamma_n$ is a degenerate $g_n$--geodesic and there exists $b>0$ such that $\ea(\gamma_n)\leq b$. Then there exists a subsequence of $\{\gamma_n\}_{n\in\N}$ that converges to a non constant degenerate geodesic $\gamma_\infty$ of $g_\infty$, also contained in $K$.
\end{lemma}

\begin{proof}
Since $\ea(\gamma_n)\le b$, there exists $t_n\in[0,1]$ such that $$g_\mathrm R\big({\dot\gamma_n}(t_n),{\dot\gamma_n}(t_n)\big)\le b^2, \quad n\in\N.$$ Up to passing to a subsequence, assume that $\{t_n\}_{n\in\N}$ converges to $t_\infty\in[0,1]$ and $\{\dot{\gamma_n}(t_n)\}_{n\in\N}$ converges to $v\in T_{p_\infty}M$ as $n$ tends to $\infty$, with $p_\infty=\lim_{n\to+\infty}\gamma_n(t_\infty)$.

Let $\gamma_\infty$ be the solution of $\D^{g_\infty}\dot\gamma=0$ with initial conditions $\gamma_\infty(t_\infty)=p_\infty$ and $\dot{\gamma_\infty}(t_\infty)=v$. From continuous dependence of solutions of ODEs on initial conditions, it is easy to see that $\gamma_\infty$ is the $C^k$--limit of the sequence $\{\gamma_n\}_{n\in\N}$. In addition, $\gamma_\infty$ is clearly a periodic $g_\infty$--geodesic with $\ea(\gamma_\infty)\le b$ contained in $K$. It is also non constant, since if it were constant, there would be nontrivial periodic geodesics relatively to metrics arbitrarily near $g_\infty$ whose images lie in $K$ and in balls of $g_{\mathrm R}$--radius arbitrarily small, which contradicts Lemma~\ref{le:noshortgeods}.

Finally, $\gamma_\infty$ is a {\em degenerate} $g_\infty$--geodesic. Let $J_n$ be a periodic Jacobi field along $\gamma_n$ which is not a multiple of the tangent field $\dot{\gamma_n}$. By adding to $J_n$ a suitable multiple of $\dot{\gamma_n}$, one can assume that $J_n(0)$ is $g_{\mathrm R}$--orthogonal to $\dot{\gamma_n}(0)$. In addition, using an adequate normalization, it is also possible to assume that $\max\big\{\Vert J_n(0)\Vert_\mathrm R,\Vert \D^{g_n}J_n(0)\Vert_\mathrm R\big\}=1$. Again, up to subsequences, the initial conditions converge $$\lim\limits_{n\to+\infty}J_n(0)=v\in T_{\gamma_\infty(0)}M\;\quad\;\lim\limits_{n\to+\infty}\D^{g_n}J_n(0)=w\in T_{\gamma_\infty(0)}M.$$ By continuity, $v$ is $g_{\mathrm R}$--orthogonal to $\dot{\gamma_\infty}(0)$ and
\begin{equation}\label{eq:maxvw}
\max\big\{\Vert v\Vert_{\mathrm R},\Vert w\Vert_{\mathrm R}\big\}=1.
\end{equation}
The solution $J_\infty$ of the $g_\infty$--Jacobi equation along $\gamma_\infty$ with the above limit initial conditions is the $C^k$--limit of the Jacobi fields $J_n$, and thus periodic. In addition, it is not a multiple of the tangent field $\dot{\gamma_\infty}$. Indeed, if $J_\infty$ were a multiple of $\dot{\gamma_\infty}$, since $v$ is $g_\mathrm R$--orthogonal to $\dot{\gamma_\infty}(0)$, it would be $v=0$ and $w=0$, which contradicts \eqref{eq:maxvw}. Hence $\gamma_\infty$ is degenerate, which concludes the proof.
\end{proof}

\begin{corollary}\label{cor:mkabopen}
Let $K\subset M$ be compact. Then for all $0<a\le b$, the subset $\M_K(a,b)$ is open in $\A_{g_\mathrm A,\nu}$, see Definition~\ref{def:mkab}.
\end{corollary}

\begin{proof}
Let us prove that the complementary $\A_{g_\mathrm A,\nu}\setminus\M_K(a,b)$ is closed. Assume that $\{g_n\}_{n\in\N}$ is a sequence in $\A_{g_\mathrm A,\nu}\setminus\M_K(a,b)$ that converges to $g_\infty\in\A_{g_\mathrm A,\nu}$. Then every $g_n$ has a non constant degenerate periodic geodesic $\gamma_n$ with image in $K$, $\ea(\gamma_n)\le b$ and $\eo(\gamma_n)\le a$. From Lemma~\ref{le:lemaconver}, there exists a subsequence of $\{\gamma_n\}_{n\in\N}$ that converges to a non constant $g_\infty$--degenerate periodic geodesic $\gamma_\infty$ with image in $K$. From Lemma~\ref{le:noshortgeods}, there exists $\widehat a>0$ such that, for $n$ sufficiently large, $$\eo(\gamma_n)=\frac{\ea(\gamma_n)}{(\# S^1_{\gamma_n})^2}\ge \widehat a,$$ i.e., the total energy of a nontrivial prime geodesic relatively to a metric near $g_\infty$ is greater or equal to $\widehat a$. 

Thus, $\# S^1_{\gamma_n}$ is bounded, hence up to passing to a subsequence, we may assume $N=\# S^1_{\gamma_n}$ is constant for all $n\in\N$. In addition, considering the limit when $n$ tends to $\infty$ in $$\gamma_n\left(t+\frac1N\right)=\gamma_n(t),$$ it follows from pointwise convergence that $\# S^1_{\gamma_n}\ge N$. Therefore, $\eo(\gamma_\infty)\le a$ and $g_\infty\in\A_{g_\mathrm A,\nu}\setminus\M_K(a,b)$, which is hence closed, concluding the proof.
\end{proof}

\section{Bumpy Metric Theorem}
\label{sec:bumpy}

In this section, we give a complete proof of an extension of the semi--Riemannian Bumpy Metric Theorem of Biliotti, Javaloyes and Piccione \cite[Theorem 3.14]{biljavapic2} to the non compact case, as discussed in the beginning of this chapter. More precisely, we establish genericity in the $C^k$--topology of semi--Riemannian metrics of given index over a non necessarily compact manifold that have no $S^1$--degenerate periodic geodesics. The proof is an adaptation of results in \cite{anosov,biljavapic2} combined with an exhaustion argument.

In addition, we stress that in the non compact case, there is no canonical separable Banach space structure on the space of semi--Riemannian metrics. Thus, we use the tools developed in Chapter~\ref{chap3} regarding this structure on subsets of metrics $\A_{g_\mathrm A,\nu}$. Recall that in the last section an index $\nu\in\{0,\dots,m\}$ was fixed and an auxiliary semi--Riemannian metric $g_\mathrm A\in\met_\nu^k(M)$ was chosen, satisfying $$\sup_{x\in M}\|g_\mathrm A(x)^{-1}\|_\mathrm R<+\infty.$$ Recall also that from Proposition~\ref{prop:affineworks}, this set $\A_{g_\mathrm A,\nu}$ is an open subset of an affine separable Banach space. We will establish genericity of bumpy metrics in this open subset.

\begin{proposition}\label{prop:a1}
Let $K\subset M$ be compact, $g_0\in\A_{g_\mathrm A,\nu}$ and $\gamma_0\in H^1(S^1,M)$ be a nondegenerate periodic $g_0$--geodesic with image contained in $K$. Then, there exists a neighborhood $\mathcal U_0$ of $g_0$ in $\A_{g_\mathrm A,\nu}$ and a $C^k$ map $$\gamma:\mathcal U_0\la H^1(S^1,M)$$ such that $\gamma(g)$ is a $g$--geodesic for all $g\in\mathcal U_0$. Moreover, for $g$ in $\mathcal U_0$, $\gamma(g)$ is the unique periodic $g$--geodesic near $\gamma_0$, and it is nondegenerate.
\end{proposition}

\begin{proof}
From Proposition~\ref{prop:existslice}, there exists a generalized slice $(\mathfrak U,\{S_n\}_{n\in\N})$ for the action of $S^1$ on $H^1(S^1,M)$. This means that there exists $n\in\N$ such that $S^1(\gamma_0)\cap S_{n}\neq\emptyset$, see Definition~\ref{def:genslice}. Moreover, every metric $g\in\A_{g_\mathrm A,\nu}$ admits a (nondegenerate) periodic geodesic near $\gamma_0$ if and only if the $g$--energy functional $E_g\vert_{S_{n}}:S_{n}\to\R$ has a (nondegenerate) critical point in $S_{n}$. Consider the restriction $E:\A_{g_\mathrm A,\nu}\times S_{n}\to\R$ of the generalized energy functional \eqref{eq:efunct}, and its partial derivative $$\frac{\partial E}{\partial\gamma}:\A_{g_\mathrm A,\nu}\times S_{n}\la TS_{n}^*.$$ Since $\gamma_0$ is a nondegenerate $g$--geodesic, $\frac{\partial E}{\partial\gamma}(g_0,\gamma_0)\in\mathbf 0_{TS_{n}^*}$ and $\frac{\partial E}{\partial\gamma}$ is transverse to $\mathbf 0_{TS_{n}^*}$ at $(g_0,\gamma_0)$. From Proposition~\ref{prop:transvsubmnfld}, the inverse image $$\left(\frac{\partial E}{\partial\gamma}\right)^{-1}\left(\mathbf 0_{TS_{n}^*}\right)$$ is a $C^k$ embedded submanifold of $\A_{g_\mathrm A,\nu}\times S_{n}$. From the Implicit Function Theorem, there exists an open neighborhood $\mathcal U_0$ of $g_0$, such that this submanifold is the graph of a $C^k$ map
\begin{eqnarray*}
\gamma_0:\mathcal U_0 &\la& H^1(S^1,M)\\
g&\longmapsto&\gamma_0(g).
\end{eqnarray*}
By continuity, for $g$ near $g_0$, the periodic geodesic $\gamma(g)$ is nondegenerate. Moreover, if $g\in\mathcal U_0$, then $\gamma_0(g)$ is the unique periodic $g$--geodesic near $\gamma_0$ and it is nondegenerate, concluding the proof.
\end{proof}

\begin{remark}
A more elegant proof of Proposition~\ref{prop:a1} above is possible using an {\em equivariant} version of the Implicit Function Theorem, in preparation by Bettiol, Piccione and Siciliano.
\end{remark}

\begin{proposition}\label{prop:lightpert}
Let $\gamma_0\in H^1(S^1,M)$ be a nondegenerate lightlike\footnote{See Definition~\ref{def:causalchar}.} $g_0$--geodesic. Then, arbitrarily near $g_0$ in $\A_{g_\mathrm A,\nu}$ there exist metrics $\widetilde{g}$ having spacelike or timelike periodic nondegenerate geodesics near $\gamma_0$. Such metrics $\widetilde{g}$ can also be chosen in such way that $g_0-\widetilde{g}$ vanishes outside an arbitrarily prescribed open subset $U$ of $M$ containing the image of $\gamma_0$.
\end{proposition}

\begin{proof}
Consider the open neighborhood $\mathcal U_0$ and the $C^k$ map $\gamma:\mathcal U_0\to H^1(S^1,M)$ given by Proposition~\ref{prop:a1}. Then the above claim is equivalent to the following function changing sign in arbitrary neighborhoods of $g_0$, $$\phi:\mathcal U_0\ni g\longmapsto E(g,\gamma(g))\in\R,$$ where $E$ is the generalized energy functional \eqref{eq:efunct}. Notice that $\phi(g_0)=0$, so hence if $\phi$ does not change sign in some neighborhood of $g_0$, then $g_0$ would be a local extremum of $\phi$. In this case, it would be $\dd\phi(g_0)h=0$ for all $h\in\mathds E$, and 
\begin{eqnarray*}
\dd\phi(g_0)h &=& \frac{\partial f}{\partial g}(g_0,\gamma_0)h+\frac{\partial f}{\partial\gamma}(g_0,\gamma_0)\circ\dd\gamma(g_0)h\\
&=& \frac{\partial f}{\partial g}(g_0,\gamma_0)h\\
&=& \tfrac12\int_{S^1}h(\dot{\gamma_0},\dot{\gamma_0})\;\dd z.
\end{eqnarray*}
Nevertheless, the integral on the right hand side in the above equality cannot vanish for all $h\in\mathds E$. For instance, if $h$ is everywhere positive definite, i.e., a Riemannian metric tensor on $M$, then such quantity is strictly positive. Thus, arbitrary neighborhoods of $g_0$ contain metrics with timelike and metrics with spacelike periodic geodesics near $\gamma_0$. Once more, nondegeneracy follows from continuity.

In addition, assume that $\widetilde{g}$ is such a metric. By continuity, the difference $h=g_0-\widetilde{g}$ may be assumed sufficiently small so that for all $t\in[0,1]$, the sum $g_0+th$ is nondegenerate on $M$. If $U$ is any open subset of $M$ containing the image of $\gamma_0$, let $b:M\to[0,1]$ be a smooth function that is identically equal to $1$ near the image of $\gamma_0$ and vanishes outside $U$. Then $\widetilde{g}=g_0+bh$ coincides with $g_0$ outside $U$ and satisfies the required properties.
\end{proof}

\begin{corollary}\label{cor:klingworks}
Let $\gamma_0$ be an arbitrary prime $g_0$--geodesic with image in $K$ and $U$ an open subset of $M$ contained in $K$, that contains the image of $\gamma_0$. Then, arbitrarily near $g_0$ in $\A_{g_\mathrm A,\nu}$ there exists $g$ satisfying
\begin{itemize}
\item[(i)] the difference $g-g_0$ has support in $U$;
\item[(ii)] the unique prime $g$--geodesic $\gamma$ near $\gamma_0$, given by\footnote{See Proposition~\ref{prop:a1}.} $\gamma(g)$, is nondegenerate and its two--fold covering $\gamma^{(2)}$ is also nondegenerate.
\end{itemize}
\end{corollary}

\begin{proof}
In the Riemannian case, this result\footnote{In fact, a more general result on the linearized Poincar\'e map of $\gamma_0$.} was proved by Klingenberg \cite[Proposition 3.3.7]{kli}, also present in a previous article by Klingenberg and Takens \cite{klitak}. This result ensures that in the perturbed metric $g$, the curve $\gamma_0$ remains a geodesic, i.e., $\gamma(g)=\gamma_0$. The proof employs only symplectic arguments, not using the positive--definite character of the metric. Thus, it carries over to the semi--Riemannian context, except for \emph{one point}. Namely, in the use of {\em Fermi coordinates} along $\gamma_0$, it is used that the tangent spaces $T_{\gamma_0(t)}M$ along $\gamma_0$ are spanned by the tangent vector $\dot\gamma_0(t)$ and its orthogonal space $\dot{\gamma_0}(t)^\perp$. In the general semi--Riemannian case, this fails to be true exactly when $\gamma_0$ is lightlike.

Nevertheless, under these circumstances, Proposition~\ref{prop:lightpert} applies. More precisely, it guarantees that it is possible to first perturb the metric $g_0$ to a new metric $\widetilde{g}$ arbitrarily near $g_0$, that coincides with $g_0$ outside $U$, and such that $\gamma=\gamma(\widetilde{g})$ has image contained in $U$ and it is not lightlike. Then, Klingenberg's perturbation argument can be applied to $\widetilde{g}$ along $\gamma$, with support in $U$, yielding a new metric having $\gamma$ and its two--fold covering $\gamma^{(2)}$ as nondegenerate geodesics.
\end{proof}

We now prove the key fact used to establish genericity of nondegeneracy for iterate geodesics. It is in great part an adaptation of \cite[Lemma 3.11]{biljavapic2}.

\begin{proposition}\label{prop:delcazzo}
Let $K\subset M$ be compact. Then for all $a>0$, $\M_K(a,2a)$ is dense in $\M_K(a,a)$.
\end{proposition}

\begin{proof}
Let $g_0\in\M_K(a,a)$ and $\mathcal U$ be an arbitrary open neighborhood of $g_0$ in $\M_K(a,a)$. There exists only a {\em finite number} of geometrically distinct\footnote{Recall that from Example~\ref{ex:periodicgeod}, two periodic geodesics $\gamma_1,\gamma_2:S^1\to M$ are \emph{geometrically distinct} if $\gamma_1(S^1)$ and $\gamma_2(S^1)$ are distinct. In particular, geometrically distinct geodesics belong to different orbits of the action of $S^1$ on $H^1(S^1,M)$.} prime $g_0$--geodesics $\{\gamma_j\}_{j=1}^r$ of total energy less than or equal to $a$, and they are all nondegenerate by assumption. Namely, if there were infinitely many, since their image is in the compact subset $K\subset M$, they would accumulate to a necessarily degenerate prime $g_0$--geodesic of energy less than or equal to $a$, contradicting Lemma~\ref{le:lemaconver}.

For each $1\leq j\leq r$, Proposition~\ref{prop:a1} implies existence of open neighborhoods $\mathcal U_j$ of $g_0$ and $C^k$ maps
\begin{equation}\label{eq:gammmajimplicit}
\gamma_j:\mathcal U_j\la H^1(S^1,M),
\end{equation}
such that $\gamma_j(g)$ is the unique periodic $g$--geodesic near $\gamma_j$, and it is nondegenerate. Let $\mathcal V=\bigcap_{j=1}^r\mathcal U_j$. This is an open neighborhood of $g_0$ where the maps \eqref{eq:gammmajimplicit} are well--defined for all $1\leq j\leq r$, and satisfy
\begin{itemize}
\item[(a)] $\gamma_j(g)$ is a prime nondegenerate $g$--geodesic for all $g\in\mathcal V$;
\item[(b)] given $g\in\mathcal V$, if $\gamma$ is a prime $g$--geodesic near one of the $\gamma_j$'s, then $\gamma$ coincides with $\gamma_j(g)$.
\end{itemize}

\begin{claim}\label{cl:onlyggj}
Given $g$ sufficiently near $g_0$, then every periodic $g$--geodesics $\alpha$ with $\ea(\alpha)\le a$ coincides with one of the $\gamma_j(g)$'s.
\end{claim}

In fact, assume that this were not the case. Then there would exist a convergent sequence $\{g_n\}_{n\in\N}$ to $g_0$ and a sequence $\{\alpha_n\}_{n\in\N}$ of periodic $g_n$--geodesics with $\ea(\alpha_n)\le a$ and such that $\alpha_n$ does not coincide with any of the $\{\gamma_j(g_n)\}_{j=1}^r$. From (b), $\alpha_n$ must then stay away from some open subset of $H^1(S^1,M)$ containing the $\gamma_j$'s. Arguing as in the proof of Corollary~\ref{cor:mkabopen}, one would then obtain a $C^2$--limit $\alpha_\infty$ of (a suitable subsequence of) $\alpha_n$, which is a $g_0$--geodesic with $\ea(\alpha_\infty)\le a$, and that does not coincide with any of the $\gamma_j$'s. Since this is impossible, Claim~\ref{cl:onlyggj} is proved.

Finally, there exists $g\in\mathcal V$ such that all the $\gamma_j(g)$ are nondegenerate, as well as their two--fold iterates $\gamma_j(g)^{(2)}$. This follows from Corollary~\ref{cor:klingworks}. More precisely, Corollary~\ref{cor:klingworks} has to be used repeatedly for each $\gamma_j(g)$, $1\leq j\leq r$ and the perturbation at the $(j+1)^{\mbox{\tiny st}}$ step has to be sufficiently small so that $\gamma_1(g),\ldots,\gamma_j(g)$ remain nondegenerate together with their two--fold coverings. Moreover, as observed above, the perturbation $g$ of $g_0$ can be chosen in such a way that $g$ has no periodic geodesic of minimal energy less than or equal to $a$ that does not coincide
with any of the $\gamma_j(g_0)$'s. Then, it follows that $g\in\M_K(a,2a)$, for all its periodic geodesics of minimal energy less than or equal to $a$ and their two--fold coverings are nondegenerate. Finally, $$g\in\mathcal V\cap\M_K(a,2a)\subset\mathcal U\cap\M_K(a,2a),$$ which proves that $\mathcal M(a,2a)$ is dense in $\mathcal M(a,a)$.
\end{proof}

We are now ready to prove our generalized version of the Bumpy Metric Theorem.

\begin{bumpythm}\label{thm:bumpy}\index{Theorem!Bumpy Metric}\index{Bumpy Metric Theorem}
Let $M$ be a smooth $m$--dimensional manifold and fix $\mathds{E}$ a separable $C^k$ Whitney type Banach space of sections of $TM^*\vee TM^*$ that tend to zero at infinity, with $k\geq 3$. Fix $\nu\in\{0,\dots,m\}$ an index and let $g_\mathrm A\in\met_\nu^k(M)$ be such that $$\sup_{x\in M}\|g_\mathrm A(x)^{-1}\|_\mathrm R<+\infty.$$ Then the following set is generic in $\A_{g_\mathrm A,\nu}$\footnote{Recall Proposition~\ref{prop:affineworks}.}
$$\mathcal{G}_\Delta(M)=\left\{g\in\A_{g_\mathrm A,\nu}: g\text{ is bumpy}\right\}.$$
\end{bumpythm}

\begin{proof}
The proof is in great part adapted from the proofs of \cite[Theorem 3.14]{biljavapic2} and \cite[Theorem 1]{anosov}, however it is extended here to the non compact case, as discussed in the beginning of the chapter. We will first prove four claims about the sets $\M_K(a,b)$, see Definition~\ref{def:mkab}. These results, together with Proposition~\ref{prop:delcazzo}, allow to use an inductive argument similarly to the one employed by Anosov \cite{anosov}, concluding that each $\M_K(n,n)$ is dense in $\A_{g_\mathrm A,\nu}$. From Corollary~\ref{cor:mkabopen}, these are also open subsets, hence their intersection is generic. Finally, an exhaustion argument is used to finish the proof, removing dependence on the compact $K\subset M$.

First, notice that $\mathcal{G}_\Delta(M)$ can be regarded as $$\mathcal{G}_\Delta(M)=\left\{g\in\A_{g_\mathrm A,\nu}:\mbox{ all periodic }g\text{--geodesics are } S^1\mbox{--nondegenerate}\right\},$$ and consider $K$ any compact subset of $M$.

\begin{claim}\label{cl:b1}
For all $a>0$, $\mathcal G_*(M)\cap\M_K(a,2a)\subset\M_K(2a,2a)$.
\end{claim}

Choose $g\in\mathcal G_*(M)\cap\M_K(a,2a)$ and let $\gamma$ be a periodic $g$--geodesic with $\ea(\gamma)\le 2a$. If $\gamma$ is prime, then it is nondegenerate, for $g\in\mathcal G_*(M)$. If $\# S^1_{\gamma}\ge2$, then $\eo(\gamma)\le a$, and thus $\gamma$ is nondegenerate, for $g\in\M_K(a,2a)$.

\begin{claim}\label{cl:b2}
For all $a>0$, $\M_K\left(\tfrac32a,\tfrac32a\right)\cap\M_K(a,2a)$ is dense in $\M_K(a,2a)$.
\end{claim}

If we prove that $\mathcal G_*(M)\cap\M_K(a,2a)$ is contained in $\M_K\left(\tfrac32a,\tfrac32a\right)\cap\M_K(a,2a)$, then Claim~\ref{cl:b2} follows automatically from the Weak Bumpy Theorem~\ref{thm:weakbumpy}, since it implies that $\mathcal G_*(M)\cap\M_K(a,2a)$ is dense in $\M_K(a,2a)$, see Lemma~\ref{le:genopenintersect} and Corollary~\ref{cor:mkabopen}. In fact, choose $g\in\mathcal G_*(M)\cap\M_K(a,2a)$ and let $\gamma$ be a periodic $g$--geodesic such that $\ea(\gamma)\le\frac32a$. If $\gamma$ is prime, then it is nondegenerate because $g\in\mathcal G_*(M)$. If $\# S^1_{\gamma}\ge2$, then $\eo(\gamma)\le\frac34 a<a$, and thus $\gamma$ is nondegenerate, because $g\in\M_K(a,2a)$.

From Claim~\ref{cl:b2} and Proposition~\ref{prop:delcazzo}, it follows that for all $a>0$,
\begin{equation}\label{eq:3232aa}
\M_K\left(\tfrac32a,\tfrac32a\right)\mbox{ is dense in }\M_K(a,a).
\end{equation}

\begin{claim}\label{cl:b4}
For all $b>a$, $\M_K(b,b)$ is dense in $\M_K(a,a)$.
\end{claim}

An immediate induction argument using \eqref{eq:3232aa} shows that for all $n\in\N$, $\M_K\left((\frac32)^na,(\frac32)^na\right)$ is dense in $\M_K(a,a)$. Choosing $n$ such that $(\frac32)^n>b$, from Lemma~\ref{le:mabmab}, $$\M_K\left((\tfrac32)^na,(\tfrac32)^na\right)\subset\M_K(b,b)\subset\M_K(a,a),$$ and therefore $\M(b,b)$ is dense in $\M_K(a,a)$, proving Claim~\ref{cl:b4}.

Notice that for $b\le a$, $\M_K(b,b)$ contains $\M_K(a,a)$. Hence, for all $a$ and $b$, $\M_K(a,a)\cap\M_K(b,b)$ is dense in $\M_K(a,a)$.

\begin{claim}\label{cl:b5}
For all $a>0$, $\M_K(a,a)$ is dense in $\A_{g_\mathrm A,\nu}$.
\end{claim}

Fix $a>0$ and $g$ in $\A_{g_\mathrm A,\nu}$. From Lemma~\ref{le:noshortgeods}, there exists $\widehat a>0$ such that all periodic $g$--geodesics have total energy greater than or equal to $\widehat a$. Thus, $g\in\M_K\left(\tfrac12{\widehat a},\tfrac12{\widehat a}\right)$. Given any neighborhood $\mathcal U$ of $g$ in $\A_{g_\mathrm A,\nu}$, since $\M_K\left(\tfrac12{\widehat a},\tfrac12{\widehat a}\right)$ is open in $\A_{g_\mathrm A,\nu}$, by Claim~\ref{cl:b4}, $\mathcal U\cap\M_K\left(\tfrac12{\widehat a},\tfrac12{\widehat a}\right)\cap\M_K(a,a)$ is nonempty. Thus, $\M_K(a,a)$ is dense in $\A_{g_\mathrm A,\nu}$, proving Claim~\ref{cl:b5}.

Let $$\mathcal G_\Delta^K(M)=\left\{g\in\A_{g_\mathrm A,\nu}:\begin{array}{c}\text{ all periodic } g\text{--geodesics with image in } K \\ \mbox{ are }S^1\mbox{--nondegenerate}\end{array}\right\}.$$ It is clear that $$\mathcal G_\Delta^K(M)=\bigcap_{n\in\N}\M_K(n,n),$$ and this is a countable intersection of open and dense subsets of $\A_{g_\mathrm A,\nu}$, see Corollary~\ref{cor:mkabopen} and Claim~\ref{cl:b5}. Thus, for any compact $K\subset M$, the subset $\mathcal G_\Delta^K(M)$ is generic in $\A_{g_\mathrm A,\nu}$.

Consider an exhaustion of $M$ by compact subsets, i.e., a sequence of compact subsets $\{K_n\}_{n\in\N}$ of $M$, with $K_n$ contained in the interior of $K_{n+1}$ for all $n\in\N$ and $M=\bigcup_{n\in\N} K_n$. It is also clear that $$\mathcal G_\Delta(M)=\bigcap_{n\in\N}\mathcal G_\Delta^{K_n}(M),$$ and this is a countable intersection of generic subsets of $\A_{g_\mathrm A,\nu}$. Therefore, from Lemma~\ref{le:gencountintersect}, the subset of bumpy metrics $\mathcal G_\Delta(M)$ is generic in $\A_{g_\mathrm A,\nu}$, concluding the proof.
\end{proof}

%
%

\section{\texorpdfstring{Bumpy Metric Theorem in the $C^\infty$--topology}{Bumpy Metric Theorem in the smooth topology}}
\label{sec:smooth1}

In this last section, we extend the Bumpy Metric Theorem~\ref{thm:bumpy} to the $C^\infty$--topology, following closely the approach of Biliotti, Javaloyes and Piccione \cite[Appendix B]{biljavapic2}. Recall that the statement of the Bumpy Metric Theorem~\ref{thm:bumpy} guarantees genericity of bumpy metrics in open subsets of the form $\A_{g_\mathrm A,\nu}$, described in Proposition~\ref{prop:affineworks}. These are formed by metrics in $\met_\nu^k(M)$, hence of class $C^k$. Replacing $\met_\nu^k(M)$ with $\met_\nu^\infty(M)$ restrains the use of most techniques developed in Chapters~\ref{chap3} and~\ref{chap4}, that apply only to Banach spaces. As observed in Remark~\ref{re:sectinftyfrechet}, requiring smoothness of tensors gives rise to a Fr\'echet structure in the space of sections, see Definition~\ref{def:frechet}.

None of the genericity criteria established in Chapter~\ref{chap4} applies in this context, since the Sard--Smale Theorem~\ref{thm:sardsmale}, which is keystone in the proof of all these criteria, does not have a sufficiently strong version for Fr\'echet spaces. Nevertheless, as pointed out by Biliotti, Javaloyes and Piccione \cite{biljavapic}, there is a standard intersection argument due to Taubes, discussed by Floer, Hofer and Salamon \cite{fhs}, that allows to overcome these technical difficulties. This is a sort of general algorithm to extend genericity results in the $C^k$--topology to the $C^\infty$--topology, used in several papers such as \cite{metmna,biljavapic,biljavapic2,gj}. Namely, it uses the genericity of a certain property in the $C^k$--topology, for $k\geq k_0$, to infer genericity of the same property in the $C^\infty$--topology.

Let us first make a few remarks on the considered topologies. Fix an index $\nu\in\{0,\dots,m\}$ and a smooth auxiliary metric $g_\mathrm A\in\met_\nu^\infty(M)$, such that $\sup_{x\in M}\|g_\mathrm A(x)^{-1}\|_\mathrm R<+\infty.$ For each $k\geq 3$, let $\mathds{E}^k$ be any separable $C^k$ Whitney type Banach space of sections of $TM^*\vee TM^*$ that tend to zero at infinity, for instance $\mathds E^k=\sect_0^k(TM^*\vee TM^*)$. Then, from Proposition~\ref{prop:affineworks},
\begin{equation*}
\A_{g_\mathrm A,\nu}^k=(g_\mathrm A+\mathds E^k)\cap\met_\nu^k(M)
\end{equation*}
is an open subset of the affine Banach space $g_\mathrm A+\mathds E^k$. The countable intersection
\begin{equation}\label{eq:ainfty}
\A_{g_\mathrm A,\nu}^\infty=\bigcap_{k\geq 3}\A_{g_\mathrm A,\nu}^k
\end{equation}
hence admits a family of inclusions $i_k:\A_{g_\mathrm A,\nu}^\infty\hookrightarrow\A_{g_\mathrm A,\nu}^k$, for all $k\geq3$. The topology considered in this intersection $\A_{g_\mathrm A,\nu}^\infty$ is the one induced by the whole family of inclusions $\{i_k\}_{k\geq3}$, i.e., the smallest topology that makes all of these inclusions continuous. Equivalently, this topology on $\A_{g_\mathrm A,\nu}^\infty$ is such that a subset $U\subset\A_{g_\mathrm A,\nu}^\infty$ is open if and only if there exist $k_0\geq3$ and an open subset $U^{k_0}$ of $\A_{g_\mathrm A,\nu}^{k_0}$ such that $U=U^{k_0}\cap\A_{g_\mathrm A,\nu}^\infty$.

This is the most natural topology to be considered in this intersection, and coincides with the so--called $C^\infty$ topology\index{Topology!$C^\infty$} of $$\A_{g_\mathrm A,\nu}^\infty=(g_\mathrm A+\mathds E^{\infty})\cap\met_\nu^\infty(M)$$ considered as an affine Fr\'echet space, where $\mathds E^\infty=\bigcap_{k\geq3}\mathds E^k$, or simply $\sect_0^\infty(TM^*\vee TM^*)$ in case $\mathds E^k$ was chosen as $\sect^k_0(TM^*\vee TM^*)$. The Fr\'echet space $\mathds E^\infty$ is here considered with the topology induced by the countable family of semi--norms given by the Banach norms $\|\cdot\|_{\mathds E^k}$ of each separable $C^k$ Whitney type Banach space of sections of $TM^*\vee TM^*$ that tend to zero at infinity, see Definition~\ref{def:ckwhitbanachspace} and Lemma~\ref{le:tvshell}.

\begin{remark}\label{re:finerthanany}
Notice that the topology on $\A_{g_\mathrm A,\nu}^\infty$ is hence finer than any topology $\tau_k$ induced by a single inclusion map $i_k:\A_{g_\mathrm A,\nu}^\infty\hookrightarrow\A_{g_\mathrm A,\nu}^k$, i.e., a $\tau_k$--open subset is always open in the considered topology (and the converse does not necessarily hold). In fact, the considered topology on $\A_{g_\mathrm A,\nu}^\infty$ is given by $\bigcup_{k\geq3}\tau_k$.
\end{remark}

Our version of the Bumpy Metric Theorem~\ref{thm:bumpy} in this $C^\infty$--topology will give genericity of smooth bumpy metrics in $\A_{g_\mathrm A,\nu}^\infty$, for a given choice of an index $\nu$, a smooth auxiliary metric $g_\mathrm A$ and a family $\{\mathds E^k\}_{k\geq3}$ where each $\mathds E^k$ is a separable $C^k$ Whitney type Banach spaces of sections of $TM^*\vee TM^*$ that tend to zero at infinity. More precisely, define
\begin{equation*}
\mathcal G_\Delta^\infty(M)=\{g\in\A_{g_\mathrm A,\nu}^\infty : g\mbox{ is bumpy}\},
\end{equation*}
and notice that $\mathcal G_\Delta^\infty(M)=\bigcap_{k\geq3}\mathcal G^k_\Delta(M)$, where
\begin{equation*}
\mathcal G^k_\Delta(M)=\left\{g\in\A^k_{g_\mathrm A,\nu}: g \mbox{ is bumpy}\right\}.
\end{equation*}
Recall that for each $k\geq 3$, the Bumpy Metric Theorem~\ref{thm:bumpy} gives genericity of $\mathcal G^k_\Delta(M)$ in $\A^k_{g_\mathrm A,\nu}$. We will also need the following elementary lemma.

\begin{lemma}\label{le:dotausk}
Let $X$ be a metric space, $D$ a dense subset of $X$ and $U$ an open and dense subset of $X$. Then $U\cap D$ is dense in $D$. 
\end{lemma}

\begin{proof}
Let $V$ be any nonempty open subset of $X$. Then since $D$ is dense in $X$, $V\cap D\neq\emptyset$. Since $U$ is open, also $U\cap V$ is open and nonempty, hence $U\cap V\cap D\neq\emptyset$. Thus, the nonempty subset $V\cap D$ open on $D$ intersects $U\cap D$. Since every open subset of $D$ is of this form, it follows that $U\cap D$ is dense in $D$.
\end{proof}

We are now ready to state and prove the $C^\infty$ version of the Bumpy Metric Theorem~\ref{thm:bumpy}, using the above results.

\begin{cibumpythm}\index{Theorem!$C^\infty$ Bumpy Metric}\index{Bumpy Metric Theorem!$C^\infty$}
Consider choices as above for $\nu$, $g_\mathrm A$ and $\{\mathds E^k\}_{k\geq 3}$, and the $C^\infty$--topology induced in the intersection $\A_{g_\mathrm A,\nu}^\infty$. Then $\mathcal G_\Delta^\infty(M)$ is generic in $\A_{g_\mathrm A,\nu}^\infty$.
\end{cibumpythm}

\begin{proof}
As above mentioned, this proof is in great part adapted from \cite{biljavapic2}. Since we are not assuming compactness of the base manifold $M$, it will be necessary to use an exhaustion argument. Thus, consider $\{K_j\}_{j\in\N}$ an exhaustion of $M$ by compacts, i.e., a sequence of compact subsets $\{K_j\}_{j\in\N}$ of $M$, with $K_j$ contained in the interior of $K_{j+1}$ for all $j\in\N$ and $M=\bigcup_{j\in\N} K_j$. Define for each $j\in\N$ and $n\in\N$ the following subsets of $\A_{g_\mathrm A,\nu}^\infty$
\begin{equation}\label{eq:inftynj}
\A^{\infty,n,j}_{g_\mathrm A,\nu}=\left\{g\in\A^\infty_{g_\mathrm A,\nu}: \begin{array}{c} \mbox{all periodic } g\mbox{--geodesics }\gamma\mbox{ with } \gamma(S^1)\subset K_j \\ \mbox{ and } \ea(\gamma)<n \mbox{ are nondegenerate}\end{array}\right\}.
\end{equation}
Notice that $\mathcal G_\Delta^\infty(M)=\bigcap_{n,j\in\N}\A^{\infty,n,j}_{g_\mathrm A,\nu}$, hence it suffices to prove that for each $n\in\N$ and $j\in\N$, the subset $\A^{\infty,n,j}_{g_\mathrm A,\nu}$ is open and dense in $\A^{\infty}_{g_\mathrm A,\nu}$. It then follows that $\mathcal G_\Delta^\infty(M)$ contains a countable intersection of open dense subsets, and is hence generic.\footnote{Recall Definition~\ref{def:generic}.}

\begin{claim}\label{cl:njopen}
For each $n\in\N$ and $j\in\N$ the subset $\A^{\infty,n,j}_{g_\mathrm A,\nu}$ is open in $\A^{\infty}_{g_\mathrm A,\nu}$.
\end{claim}

Notice that\footnote{See \eqref{eq:inftynj} and \eqref{eq:ainfty}.} for each $k\geq3$, 
\begin{equation}\label{eq:seila}
\A^{\infty,n,j}_{g_\mathrm A,\nu}=\A^{\infty}_{g_\mathrm A,\nu}\cap\M^k_{K_j}(n,n),
\end{equation}
where $\M_{K_j}^k(n,n)$ is given in Definition~\ref{def:mkab}, and the index $k$ stress the choice of regularity $C^k$ in that definition. More precisely, $$\M^k_{K_j}(n,n)=\left\{g\in\A^k_{g_\mathrm A,\nu}: \begin{array}{c} \mbox{all periodic } g\mbox{--geodesics }\gamma\mbox{ with } \gamma(S^1)\subset K_j \\ \mbox{ and } \ea(\gamma)<n \mbox{ are nondegenerate}\end{array}\right\}.$$ Fix $k\geq 3$. Then Corollary~\ref{cor:mkabopen} gives that $\M^k_{K_j}(n,n)$ is open in $\A^k_{g_\mathrm A,\nu}$ for all $n,j\in\N$. From Remark~\ref{re:finerthanany}, this implies that $\A^{\infty,n,j}_{g_\mathrm A,\nu}$ is also open in $\A^{\infty}_{g_\mathrm A,\nu}$, since\footnote{Notice that the intersection $\A^{\infty,n,j}_{g_\mathrm A,\nu}=\A^{\infty}_{g_\mathrm A,\nu}\cap\M^k_{K_j}(n,n)$ is open in $\A^{\infty}_{g_\mathrm A,\nu}$ with the topology induced by the inclusion $i_k:\A^{\infty}_{g_\mathrm A,\nu}\hookrightarrow\A^k_{g_\mathrm A,\nu}$. The $C^\infty$--topology on $\A^{\infty}_{g_\mathrm A,\nu}$ is finer than any of these topologies, for it is induced by the entire family $\{i_k\}_{k\geq3}$, hence $\A^{\infty,n,j}_{g_\mathrm A,\nu}$ is open in $\A^{\infty}_{g_\mathrm A,\nu}$.} it is $\tau_k$--open, concluding the proof of Claim~\ref{cl:njopen}.

\begin{claim}\label{cl:njdense}
For each $n\in\N$ and $j\in\N$ the subset $\A^{\infty,n,j}_{g_\mathrm A,\nu}$ is dense in $\A^{\infty}_{g_\mathrm A,\nu}$.
\end{claim}

Fix $n,j\in\N$, $k\geq3$ and consider once more the intersection \eqref{eq:seila}. Corollary~\ref{cor:mkabopen} gives that $\M^k_{K_j}(n,n)$ is open in $\A^k_{g_\mathrm A,\nu}$, and Claim~\ref{cl:b5} in the proof of the Bumpy Metric Theorem~\ref{thm:bumpy} gives that $\M^k_{K_j}(n,n)$ is dense in $\A^k_{g_\mathrm A,\nu}$. In addition, from the Stone--Weierstrass Theorem~\ref{thm:stw2}, it is easy to conclude that $\A^\infty_{g_\mathrm A,\nu}$ is also dense in $\A^k_{g_\mathrm A,\nu}$. Therefore, setting $X=\A^k_{g_\mathrm A,\nu}$, $U=\M^k_{K_j}(n,n)$ and $D=\A^\infty_{g_\mathrm A,\nu}$ in Lemma~\ref{le:dotausk}, it follows that $\A^{\infty,n,j}_{g_\mathrm A,\nu}=\A^{\infty}_{g_\mathrm A,\nu}\cap\M^k_{K_j}(n,n)$ is dense in $\A^{\infty}_{g_\mathrm A,\nu}$, concluding the proof.
\end{proof}

\chapter[Nondegeneracy under GEC]{Nondegeneracy under GEC}
\label{chap6}

In the last chapter, we proved a semi--Riemannian version of the Bumpy Metric Theorem, which is a central result in the theory of generic properties of geodesic flows. It is therefore natural to ask whether other classic generic properties of Riemannian geodesic flows extend to the semi--Riemannian context in a similar fashion. In fact, this problem was proposed by Biliotti, Javaloyes and Piccione \cite{biljavapic}, together with a genericity result on nondegeneracy of semi--Riemannian geodesics joining to {\em distinct} points. More recently, this result was extended by Bettiol and Giamb\`o \cite{metmna} to the context of {\em general endpoints conditions}, or {\em GEC}s, see Definition~\ref{def:gec}. Such extension is the main issue of this chapter.

In general terms, instead of considering semi--Riemannian geodesics joining two points, the results of \cite{metmna} allow to consider arbitrary endpoints conditions for geodesics, expressed in terms of a submanifold $\p$ of the product $M\times M$. In Section~\ref{sec:gecs}, the geometry of both fixed endpoints and general endpoints conditions were studied, see Lemma~\ref{le:opqsmnfld} and Proposition~\ref{prop:opmsubmfld}. In addition, the associated geodesic variational problems were explored in details in Sections~\ref{sec:genenfunc},~\ref{sec:genindexform} and~\ref{sec:pjacobi}. In this chapter, we aim to first give sufficient conditions on GECs for the genericity of nondegeneracy statement to hold in the $C^k$--topology. In Section~\ref{sec:admissgecs} we define {\em admissibility} for GECs, and finally in Section~\ref{sec:gnggec} we give a detailed proof of the main result in \cite{metmna}, Theorem~\ref{thm:bigone}.

Before getting to details, let us give further motivations for the study of generic properties of geodesics in semi--Riemannian manifolds. Other than Lorentzian geometry and its implications in general relativity, an important motivation comes from Morse theory. Indeed, a crucial assumption for developing a Morse theory for geodesics between fixed points is that the two arbitrarily fixed distinct points must be non conjugate. Recent works by Abbondandolo and Majer \cite{AbbMej2,AbbMaj,AbbMaj2} connect Morse relations for critical points of the semi--Riemannian energy functional to the homology of a doubly infinite chain complex, the Morse--Witten complex, constructed out of the critical points of a strongly indefinite Morse functional, using the dynamics of the gradient flow. The Morse relations for critical points are obtained computing the homology of this complex, which in the standard Morse theory is isomorphic to the singular homology of the base manifolds. Abbondandolo and Majer \cite{AbbMej2} also managed to prove stability of this homology with respect to small perturbations of the metric structure. Thus, it is important to ask whether it is possible to perturb a metric in such a way that the non conjugacy property between two points is preserved. A positive answer\footnote{As a matter of fact, there is a much simpler proof of this perturbation property. In Remark~\ref{re:sacana} we give a general idea of this proof, that only guarantees the existence of a such perturbation that destroys conjugacy between two fixed points. However, the result of Biliotti, Javaloyes and Piccione \cite{biljavapic}, stated in Theorem~\ref{thm:biljavapic}, is much stronger. It asserts {\em genericity} of the set of metrics for which two points are not conjugate, and not only its {\em density}, see Remark~\ref{re:genericvsdense}.} to this question is given by Biliotti, Javaloyes and Piccione \cite{biljavapic}, and the results of Bettiol and Giamb\`o \cite{metmna} assert that this property remains valid when considering, more generally, the non focality property between a point and a submanifold, see Definition~\ref{def:focalpoint}. In fact, more general {\em generic non focality} assertions may be inferred from \cite{metmna}, see Corollaries~\ref{cor:nonconjugacy},~\ref{cor:nonfocality} and~\ref{cor:nonfocality2}.

Let us begin by recalling the main result of Biliotti, Javaloyes and Piccione \cite[Proposition 4.3]{biljavapic}, which can be adapted to our context in the following way.

\begin{theorem}\label{thm:biljavapic}
Let $M$ be a smooth manifold and $\mathds E$ be a separable $C^k$ Whitney type Banach space of sections of $TM^*\vee TM^*$ that tend to zero at infinity, with $k\geq 3$. Fix $\nu\in\{0,\dots,m\}$ an index and let $g_\mathrm A\in\met_\nu^k(M)$ be such that $$\sup_{x\in M}\|g_\mathrm A(x)^{-1}\|_\mathrm R<+\infty.$$ Given any pair of {\em distinct} points $p,q\in M$, the set of semi--Riemannian metrics $g\in\mathcal A_{g_\mathrm A,\nu}$ such that all $g$--geodesics joining $p$ and $q$ are nondegenerate is generic in $\mathcal A_{g_\mathrm A,\nu}$.\footnote{Recall Proposition~\ref{prop:nondegenerateopen} and \eqref{eq:anup}.}
\end{theorem}

\begin{remark}
Although essentially proved in \cite[Proposition 4.3]{biljavapic}, Theorem~\ref{thm:biljavapic} above is stated in a slightly different way of the paper's original result, that allows to give a complete proof of the statement. This is basically due to the fact that the {\em separability}\footnote{Separability is a {\em necessary} condition to use the Sard--Smale Theorem~\ref{thm:sardsmale}, which is keystone in the proof of the Abstract Genericity Criterion~\ref{crit:abstractgenericity}.} problem of the space of metrics is ignored in \cite{biljavapic}, and this leads to some necessary adaptations. In fact, this is the subject dealt with in the end of Section~\ref{sec:banachspacetensors}, which we now briefly recall.

The ideal candidate to $\mathds E$ described in \cite[Example 1]{biljavapic}, that corresponds to $\sect_b^k(TM^*\vee TM^*)$ is non separable as pointed out in Remark~\ref{re:separability}. The easiest solution to this problem is replacing $\mathds E$ with a separable subspace of $\sect_b^k(TM^*\vee TM^*)$, for instance $\sect_0^k(TM^*\vee TM^*)$. However, this would cause the intersection $\sect_0^k(TM^*\vee TM^*)\cap\met_\nu^k(M)$ to have empty interior, as pointed out in Remark~\ref{re:emptyinterior}. We then replace this $\mathds E$ with an affine translation $g_\mathrm A+\mathds E$, where $g_\mathrm A$ satisfies a suitable uniform nondegeneracy property, see \eqref{eq:boundedfromzero}. Such affine space has the topology induced by the translation of $g_\mathrm A$, hence is clearly separable. This is the adequate\footnote{More than mathematically more suitable for our purposes, this setting of {\em asymptotically equal to $g_\mathrm A$} metrics is a relevant generalization of the so--called asymptotically flat space--times. The relativistic meaning and physical relevance of such conditions on the metrics is discussed in Remark~\ref{re:physics}.} setting for semi--Riemannian metrics on non compact manifolds for the type of genericity argument that follows. Our generalization of Theorem~\ref{thm:biljavapic} given by Theorem~\ref{thm:bigone} is stated an proved in such context.
\end{remark}

\begin{remark}\label{re:sacana}
If we were only interested in proving {\em density}\footnote{Recall that density is a much weaker property then genericity, see Remark~\ref{re:genericvsdense}.} of the set of semi--Riemannian metrics $g\in\mathcal A_{g_\mathrm A,\nu}$ such that all $g$--geodesics joining $p$ and $q$ are nondegenerate, there would be a quite simpler approach to the problem.

In fact, given $g\in\A_{g_\mathrm A,\nu}$, suppose $p$ and $q$ are $g$--conjugate. From Proposition~\ref{prop:conjugatecritexp}, a point $q'$ is $g$--conjugate to $p$ if and only if it is a critical value of the $g$--exponential map $\exp_p:T_pM\to M$. Applying the Sard Theorem~\ref{thm:sard}, it follows that the set of points $q'\in M$ such that $p$ is not $g$--conjugate to $q'$ is generic in $M$, in particular dense in $M$, see Remark~\ref{re:genericvsdense}. Choose such a $q'$ near $q$, that is not $g$--conjugate to $p$, as illustrated below.

\begin{figure}[htf]
\begin{center}
\vspace{-0.75cm}
\includegraphics[scale=1]{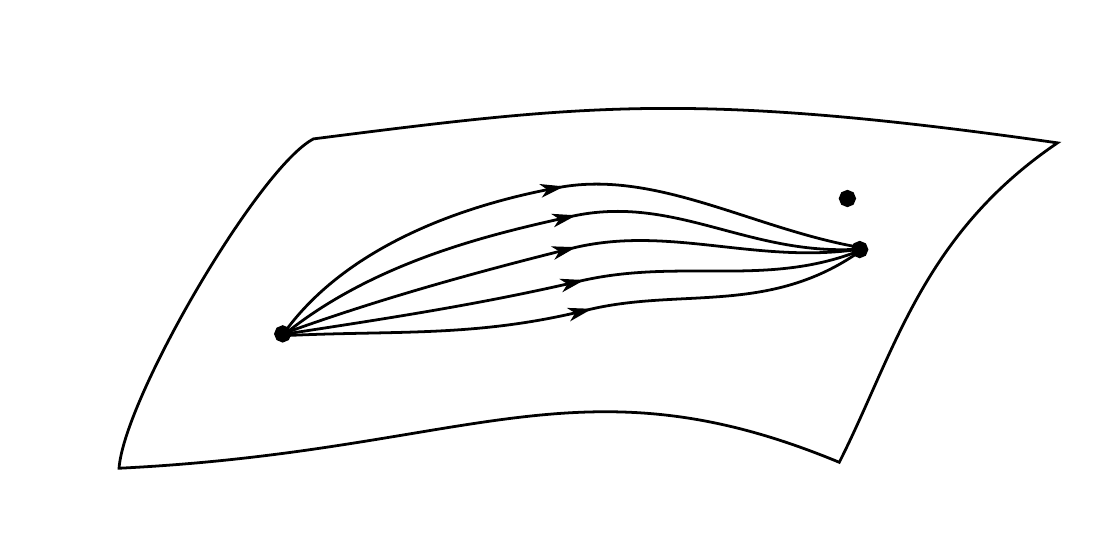}
\begin{pgfpicture}
\pgfputat{\pgfxy(-2.3,3)}{\pgfbox[center,center]{$q$}}
\pgfputat{\pgfxy(-2.3,3.8)}{\pgfbox[center,center]{$q'$}}
\pgfputat{\pgfxy(-9,2.1)}{\pgfbox[center,center]{$p$}}
\end{pgfpicture}
\vspace{-0.75cm}
\end{center}
\end{figure}

It obviously suffices to suppose that $q$ and $q'$ are in the same connected component of $p\in M$. Since we may regard the action of the diffeomorphism group $\diff(M)$ on this connected component of $M$ as transitive,\footnote{Recall Definition~\ref{def:actionobjects}. The infinite--dimensional group $G=\diff(M)$ is not a Lie group, however several important techniques may be used. In fact, to prove that its action on a connected $M$ is transitive, it suffices to prove that each orbit $G(x)$ is open. This implies that $M\setminus G(x)=\bigcup_{y\notin G(x)}G(y)$ is also open, since it is the union of the other orbits, and hence $G(x)$ is closed. Being open and closed, since $M$ is connected, $G(x)=M$ and hence the action is transitive. Notice that if $M$ is not connected, the transitivity holds for points in the same connected component (as required in Remark \ref{re:sacana}).

In order to verify that $G(x)$ is open, given $x'$ near $x$ it is possible to consider a local chart around $x$ and obtain a diffeomorphism of the domain of this chart that maps $x$ to $x'$ and coincides with the identity near the boundary of the chart. This is done with the images of $x$ and $x'$ in Euclidean space and then conjugated with the chart. Setting this diffeomorphism equal to the identity of $M$ outside the chart, we have a global diffeomorphism of $M$ that maps $x$ to $x'$, and hence $x$ admits an open neighborhood contained in $G(x)$, proving it is an open subset.} there exists a diffeomorphism $f:M\to M$ such that $f(q)=q'$. Then, the points $p$ and $q$ are not conjugate in the pull--back metric $f^*g$, by construction. In addition, continuity arguments prove that if $q'$ is sufficiently near $q$, then $f$ is sufficiently near the identity so that $f^*g$ is sufficiently near $g$. This implies that arbitrarily small perturbations of $g$ destroy the conjugacy property of $p$ and $q$, proving density of the desired subset. Notice however that this does not imply Theorem~\ref{thm:biljavapic}, since genericity is a property stronger than density.
\end{remark}

\section{Admissibility of GECs}
\label{sec:admissgecs}

In order to state our generalization of Theorem~\ref{thm:biljavapic} in the context of GECs, it is necessary to analyze in more details some nondegeneracy properties of submanifolds $\p$ of $M\times M$. We begin with a technical remark on the openness of this nondegeneracy property of semi--Riemannian submanifolds. This will be later used to endow $\p$ with a product metric.

From the reasons presented in Remark~\ref{re:whyg}, the natural choice is to consider the metric \eqref{eq:overlineg}, i.e., the restriction of the ambient space metric $g\oplus (-g)$ to $\p$,
\begin{equation*}
\overline{g}=g\oplus(-g)\in\met_m^k(M\times M).
\end{equation*}

\begin{remark}
Notice that the index of $\overline{g}$ is always equal to $m=\dim M$, with no dependence of $\nu$.
\end{remark}

\begin{proposition}\label{prop:nondegenerateopen}
Let $\mathds{E}$ be a $C^k$ Whitney type Banach space of sections of $E=TM^*\vee TM^*$ that tend to zero at infinity, $\nu\in\{0,\dots,m\}$ a fixed index and $g_\mathrm A\in\met_\nu^k(M)$ a metric satisfying \eqref{eq:boundedfromzero}. Consider the nonempty open subset $\A_{g_\mathrm A,\nu}=(g_\mathrm A+\mathds E)\cap\met_\nu^k(M)$ of metrics studied in Proposition~\ref{prop:affineworks}. Let $\p\subset M\times M$ be a \emph{compact} submanifold. Then the following subset is open in $\A_{g_\mathrm A,\nu}$
\begin{equation}\label{eq:anup}
\A_{g_\mathrm A,\nu,\p}=\big\{g\in\A_{g_\mathrm A,\nu}: \overline{g}\in\met_m^k(M\times M,\p)\big\},
\end{equation}
recall \eqref{eq:nondegmet} in Definition~\ref{def:nondegmet}.
\end{proposition}

\begin{proof}
Suppose $\A_{g_\mathrm A,\nu,\p}\ne\A_{g_\mathrm A,\nu}$, otherwise the statement is trivially verified. For each $g\in\A_{g_\mathrm A,\nu}$, consider the product metric $\overline{g}$. Let $\{g_n\}_{n\in\N}$ be a convergent sequence in $\A_{g_\mathrm A,\nu}\setminus\A_{g_\mathrm A,\nu,\p}$ and $\{\overline{g_n}\}_{n\in\N}$ the correspondent sequence in $\met_n^k(M\times M)\setminus \met_n^k(M\times M,\p)$, with $\lim_{n\to+\infty} \overline{g_n}=\overline{g_\infty}$.  From identifications \eqref{ident:bilin}, consider the symmetric tensor $i^*\overline{g_n}$ at each $p$ as a linear operator $$(i^*\overline{g_n})_p:T_p\p\la T_p\p^*\cong T_p\p,$$ denoted with the same symbol. Since for all $n$, $i^*\overline{g_n}$ is a degenerate symmetric bilinear tensor on $\p$, there exists $p_n\in\p$ and $V_n\subset T_{p_n}\p$, with $\dim V_n\geq 1$, such that $V_n\subset\ker (i^*\overline{g_n})_{p_n}$, see Definition~\ref{def:nondegenerate}. Choosing $r$ to be the minimum of $\dim V_n$, without loss of generality it is possible to assume that for all $n$, $\dim V_n=r\geq 1$.

Thus $\{V_n\}_{n\in\N}$ is a sequence in the $r$--Grassmannian bundle\footnote{See Example~\ref{ex:grassmann}.} $\gr_r(\p)$, which is compact, since $\p$ is compact. Up to subsequences, there exists $V_\infty\in\gr_r(\p)$ limit of the sequence $\{V_n\}_{n\in\N}$. By continuity of this convergence, there exists a limit point $p_\infty\in \p$, and $V_\infty\subset\ker (i^*\overline{g_\infty})_{p_\infty}$. Therefore, as $\dim V_\infty=r\geq 1$, the limit metric tensor $\overline{g_\infty}$ is also in $\met_n^k(M\times M)\setminus\met_n^k(M\times M,\p)$, hence $g_\infty\in\A_{g_\mathrm A,\nu}\setminus\A_{g_\mathrm A,\nu,\p}$.
\end{proof}

\begin{remark}
Notice that the above proof does not use the particular fact that tensors of $\mathds E$ tend to zero at infinity, or that $g_\mathrm A$ satisfies \eqref{eq:boundedfromzero}. These hypotheses are only made in order to provide the same context of that in the definition of $\A_{g_\mathrm A,\nu}$ in Proposition~\ref{prop:affineworks}. Indeed, this openness of nondegeneracy is a much more general result, that will however be applied to the context above.
\end{remark}

Recall that from Definition~\ref{def:gpgeod}, a curve $\gamma\in\op(M)$ is a $(g,\p)$--geodesic if $\frac{\partial E}{\partial \gamma}(g,\gamma)=0$, i.e., if $\gamma$ is a critical point of $E_g:\op(M)\to\R$. From Proposition~\ref{prop:critgenenfunc}, this is also equivalent to $\gamma$ being a $g$--geodesic that satisfies $$(\dot\gamma(0),\dot\gamma(1))\in T_{(\gamma(0),\gamma(1))}\p^\perp,$$ where $^\perp$ denotes orthogonality relatively to $\overline g$. For our main result on generic nondegeneracy of $(g,\p)$--geodesics, it is necessary to have a lower bound on the Riemannian length of such geodesics, analogously to Lemma~\ref{le:noshortgeods}. To this aim we introduce the following.

\begin{definition}\label{def:admgbc}
A GEC $\p$ will be said to be \emph{$(g_\mathrm A,\nu)$--admissible}\index{GEC!admissible} if
\begin{itemize}
\item[(i)] $\p$ is compact;
\item[(ii)] $\A_{g_\mathrm A,\nu,\p}$ given by \eqref{eq:anup} is \emph{nonempty};
\item[(iii)] for every $g_0\in\A_{g_\mathrm A,\nu,\p}$, there exists an open neighborhood $\mathcal V$ of $g_0$ in $\A_{g_\mathrm A,\nu,\p}$ and $a>0$, such that for all $g\in\mathcal V$ and all $(g,\p)$--geodesics $\gamma$, $L_\mathrm R(\gamma)\geq a$.
\end{itemize}
In case the pair $(g_\mathrm A,\nu)$ is evident from the context, we will simply say that $\p$ is {\em admissible}.
\end{definition}

It is easy to see that this definition does not depend on the choice of the auxiliary Riemannian metric $g_\mathrm R$.

\begin{remark}
Regarding emptiness of $\A_{g_\mathrm A,\nu,\p}$, recall that from Remark~\ref{re:metmpmightbeempty} the set $\met_\nu^k(M\times M,\p)$ may be empty depending on the topology of $\p$. Topological obstructions to the existence of metrics of given index were studied in Section~\ref{sec:topobst}, particularly the case of Lorentzian metrics and metrics on spheres, see Propositions~\ref{prop:noncompactlorentz},~\ref{prop:existlorentzian} and Theorems~\ref{thm:spheres1} and~\ref{thm:spheres2} respectively.

Thus, depending on the topology of $\p$, it may not admit any metrics of the form $g\oplus (-g)$, and in this case $\A_{g_\mathrm A,\nu,\p}=\emptyset$. For instance, if $M$ is three--dimensional and $\p$ is homeomorphic to the sphere $S^4$, then $\met_{3}^k(M\times M,\p)=\emptyset$. This follows easily from the following facts. On the one hand, the restriction to $\p$ of any metric tensor on $M\times M$ having index equal to $3$ cannot be positive or negative definite. On the other hand, $\p$ does not admit any metric tensor of index $1$ or $2$, since $\p$ does not admit distributions\footnote{Recall that from Proposition~\ref{prop:metricdistribution}, existence of a semi--Riemannian metric of index $\nu$ is equivalent to the existence of a distribution of rank $\nu$ with the same regularity.} of rank $1$ or $2$.

Much more general examples of homotopy types for $\p$ that for certain dimensions of $M$ imply emptiness of $\met_\nu^k(M\times M,\p)$ may be obtained from Theorem~\ref{thm:spheres2}.
\end{remark}

\begin{remark}
Let us briefly justify the hypotheses (i), (ii) and (iii) for $(g_\mathrm A,\nu)$--admissibility of a GEC $\p$. The subset $\A_{g_\mathrm A,\nu,\p}$ is genuinely the natural set of metrics to be considered in this context, and for this reason, admissibility of a GEC $\p$ is defined in such way that $\A_{g_\mathrm A,\nu,\p}\ne\emptyset$. More precisely, genericity of metrics without degenerate $(g,\p)$--geodesics will be established in this open subset, which is thus required to be nonempty, so that our statement is nontrivial. For a detailed study of why this is an appropriate domain to use our techniques, see Section~\ref{sec:banachspacetensors}.

It is also crucial to consider only nondegenerate metrics on $\p$ because the submanifold geometry of $\p$ determines the behavior of variational fields correspondent to curves with these conditions, see \eqref{eq:indexform}. More precisely, Lemma~\ref{le:jacobinotparallel} would not hold in case $\p$ was degenerate, see Remark~\ref{re:whyg}. Compactness of $\p$ is also a fundamental assumption, not only because it is used to prove Proposition~\ref{prop:nondegenerateopen} above, but also because we shall use boundedness of $\p$ to get the desired conditions on limits of curves satisfying such GEC. In this sense, we also need to replace the result of Lemma~\ref{le:noshortgeods} that gives a lower bound on the Riemannian length of geodesics in the sense of (iii), which is now required as a hypothesis on $\p$ for its admissibility.
\end{remark}

\begin{remark}
Notice that even for indexes $\nu=0$ and $\nu=m$, in which we are essentially dealing with {\em Riemannian manifolds}, the sets $\A_{g_\mathrm A,\nu}$ and $\A_{g_\mathrm A,\nu,\p}$ may not coincide. The key fact is that the metric $\overline g=g\oplus (-g)$ is {\em always semi--Riemannian}, of index $m$. For instance, if $\p$ is tangent to the diagonal $\Delta$ of $M\times M$ at any point, then $\A_{g_\mathrm A,\nu,\p}$ is trivially empty, since any metric in $\A_{g_\mathrm A,\nu}$ degenerates at this point, see Remark~\ref{re:diagonaldoesntwork}.

The only situation in which $\A_{g_\mathrm A,\nu,\p}=\A_{g_\mathrm A,\nu}$ automatically is when $\p$ is a point. This corresponds to fixed endpoints conditions $\{p\}\times\{q\}$, even if $p=q$. In this case, $\p$ trivially satisfies conditions (i) and (ii) of Definition~\ref{def:admgbc}.
\end{remark}

Some classes of GECs introduced in Example~\ref{ex:gecs} are clearly admissible. Let us now comment a few examples.

\begin{example}\label{ex:triviallyadmissible}
First, if $\p$ is compact, has no $\nu$--topological obstructions and satisfies $\p\cap\Delta=\emptyset$, then it is admissible. In this case, to verify condition (iii) of Definition~\ref{def:admgbc} it is enough to set $$a=\min_{(p,q)\in\p} d_\mathrm R(p,q),$$ where $d_\mathrm R$ denotes the $g_\mathrm R$--distance in $M$, see Definition~\ref{def:distance}. It is not difficult to see that there are no restrictions on the auxiliary metric $g_\mathrm A$ for admissibility in this case, provided that $\p$ has no topological obstructions to the existence of such metrics. For instance, this is the case of a fixed endpoints condition $\p=\{p\}\times\{q\}$, with $p\neq q$.
\end{example}

\begin{example}\label{ex:Pq}
Another class of admissible GECs is given by $\p=P\times \{q\}$, where $P\subset M$ is a compact submanifold and $q\in M$, as described in Example~\ref{ex:gecs}. We are clearly supposing that there are no topological obstructions on $P$ for a given choice of index $\nu$.

There are two possible situations, depending on the relative position of $P$ and $q$. Namely, if $q\notin P$, then $\p\cap\Delta=\emptyset$, hence it is also in the previous class of Example~\ref{ex:triviallyadmissible}. However, if $q\in P$, the proof of Lemma~\ref{le:noshortgeods} can be used to verify that $\p$ is admissible. In fact, although stated only for periodic geodesics, its proof is automatically valid considering non constant geodesic loops instead of periodic geodesics, hence gives the required condition on $\p$. Note that the same holds for the transpose $\p^t=\{q\}\times P$, see Remark~\ref{re:transpose}.
\end{example}

\begin{remark}\label{re:diagonaldoesntwork}
The diagonal case $\p=\Delta$ was already mentioned in Examples~\ref{ex:gecs} and~\ref{ex:gecs2}. Clearly, if $M$ is non compact, $\p$ is not admissible. Moreover, $\A_{g_\mathrm A,\nu,\p}$ is trivially empty for every $\nu$, since the tangent space to $\Delta$ at $(x,x)$ is the diagonal of $T_xM\oplus T_xM$, and any metric of the form $\overline g$ vanishes identically in such subspace. More generally, any $\p$ somewhere tangent to $\Delta$ is trivially degenerate for a metric of the form $\overline g$. Thus, these are not admissible GECs.

In this sense, we cannot expect to use GECs to generalize the Bumpy Metric Theorem~\ref{thm:bumpy}. Instead, we {\em use} the Bumpy Metric Theorem~\ref{thm:bumpy} to establish genericity of metrics without degenerate geodesics under GECs that may intersect $\Delta$ transversally, see Proposition~\ref{prop:admissibility}. Regarding the case $\p=\Delta$, the main generic property studied in this chapter, Theorem~\ref{thm:bigone}, coincides for such $\p$ with the statement of the Bumpy Metric Theorem~\ref{thm:bumpy}, interpreting degeneracy in the adequate sense, see Definition~\ref{def:gmorse} and Example~\ref{ex:gecs2}.
\end{remark}

We shall now establish the admissibility of a larger class of GECs that intersect $\Delta$, using a transversality approach, see Definition~\ref{def:transversality} and Remark~\ref{re:abouttransv}. To this aim, recall the estimate of the decrease of the difference between the normalized tangent field to a geodesic at its endpoints, in terms of its length, given by Lemma~\ref{le:short}.

\begin{proposition}\label{prop:admissibility}
If a GEC $\p$ intersects $\Delta$ transversally\footnote{That is, $T_{(x,x)}\p +\Delta=T_xM\oplus T_xM$, for all $x\in\p\cap\Delta$ (recall Definition~\ref{def:transversality}.}, then $\p$ satisfies (iii) of Definition~\ref{def:admgbc}. In particular, if in addition $\p$ is compact and has no topological obstructions, i.e., satisfies (i) and (ii), then it is admissible.
\end{proposition}

\begin{proof}
We proceed by contradiction. Since the weak Whitney $C^1$--topology is first countable, assuming $\p$ is does not satisfy (iii) implies that there exists a sequence $\{g_n\}_{n\in\N}$ in $\A_{g_\mathrm A,\nu,\p}$ converging to some $g_\infty\in\A_{g_\mathrm A,\nu,\p}$ in the weak Whitney $C^1$--topology and a sequence $\{\gamma_n\}_{n\in\N}$ in $\op(M)$ of non constant $(g_n,\p)$--geodesics such that $\lim_{n\to+\infty} L_\mathrm R(\gamma_n)=0$. Since $\p$ is compact, up to taking subsequences, we may assume that there exists $x\in M$ such that $(x,x)\in\p$ and both $\lim_{n\to+\infty} \gamma_n(0)=x$, $\lim_{n\to+\infty} \gamma_n(1)=x$.

By taking a local chart of $M$ around $x$, we can assume that we are in open subset $U\subset \R^m$. Let $K\subset U$ be any compact neighborhood of $x$, so that there exists $n_0$ such that for $n\geq n_0$, $\gamma_n([0,1])\subset K$. Since $L_\mathrm R(\gamma_n)$ tends to zero, then also the Euclidean length of $\gamma_n$ tends to zero. From Lemma~\ref{le:short}, it follows that, $$\lim_{n\to+\infty} \left(\frac{\dot\gamma_n(0)}{\|\dot\gamma_n(0)\|}-\frac{\dot\gamma_n(1)}{\|\dot\gamma_n(1)\|}\right)=0,$$ and up to taking subsequences, we can assume that both $\frac{\dot\gamma_n(0)}{\|\dot\gamma_n(0)\|}$ and $\frac{\dot\gamma_n(1)}{\|\dot\gamma_n(1)\|}$ converge to unitary vectors. However, from the above limit, both tend to the \emph{same unitary vector $v\in\R^m$}.

We claim that $(v,v)\in T_{(x,x)}\p^\perp$, where $^\perp$ denotes orthogonality with respect to $\overline{g_\infty}$, and that this concludes the proof. Indeed, suppose the claim to be true. Then $$(v,v)\in T_{(x,x)}\p^\perp\cap\Delta=(T_{(x,x)}\p+\Delta^\perp)^\perp.$$ It is easy to see that $\Delta^\perp=\Delta$; and since we assumed $T_{(x,x)}\p+\Delta=T_xM\oplus T_xM$, its orthogonal complement with respect to $\overline{g_\infty}$ is trivial. Hence $v=0$, which gives the desired contradiction.

It remains to prove the above claim that $(v,v)\in T_{(x,x)}\p^\perp$. Consider $\mathcal O$ the open neighborhood\footnote{Using the identification above given by a local chart $U$ of $M$ around $x$, since the restriction map $\A_{g_\mathrm A,\nu,\p}\owns h\mapsto h|_U\in\met_\nu^k(U)$ is continuous in the considered topologies, the open neighborhood of $g_\infty$ in $\A_{g_\mathrm A,\nu,\p}$ can be taken as the preimage of $\mathcal O$ by this restriction map.} of $g_\infty\in\met_\nu^k(U)$ in the weak Whitney $C^1$--topology given by Lemma~\ref{le:short} with the choices above. Then, for all $g\in\mathcal O$ it is possible to give the following estimate for any $g$--geodesic $\gamma$ with image lying in $K$,
{\allowdisplaybreaks
\begin{eqnarray*}
\left|\frac{\dd}{\dd t} \log \|\dot\gamma(t)\|\right| &=&\frac{\left|\langle\dot\gamma,\ddot\gamma\rangle\right|}{\|\dot\gamma\|^2}\\
&\stackrel{\eqref{eq:geodequation}}{=}&\frac{\left|\langle\dot\gamma,\chr^{g}(\gamma)(\dot\gamma,\dot\gamma)\rangle\right|}{\|\dot\gamma\|^2}\\
&\leq& \frac{\|\chr^g(\gamma)\|\|\dot\gamma\|^3}{\|\dot\gamma\|^2}\\
&\leq& \kappa\|\dot\gamma\|,
\end{eqnarray*}}where $\kappa=\max_{x\in K} \|\Gamma^{g_\infty}(x)\| +1$ is again the same as in Lemma~\ref{le:short}. Hence, integrating the above inequality in $[0,1]$, it follows that
{\allowdisplaybreaks
\begin{eqnarray*}
\left|\log\frac{\|\dot\gamma(1)\|}{\|\dot\gamma(0)\|}\right| &=& \left|\int_0^1 \frac{\dd}{\dd t}\log\|\dot\gamma\|\;\dd t\right|\\
&\leq&\int_0^1 \left|\frac{\dd}{\dd t}\log\|\dot\gamma\|\right|\;\dd t\\
&\leq&\kappa\int_0^1\|\dot\gamma\|\;\dd t.
\end{eqnarray*}}Applying this estimate to the $(g_n,\p)$--geodesics $\gamma_n$, since its Euclidean length tend to zero, one concludes that $$\lim_{n\to+\infty} \frac{\|\dot{\gamma_n}(1)\|}{\|\dot{\gamma_n}(0)\|}=1.$$ Moreover, for each $n$, $$\left(\frac{\dot{\gamma_n}(0)}{\|\dot{\gamma_n}(1)\|},\frac{\dot{\gamma_n}(1)}{\|\dot{\gamma_n}(1)\|}\right)\in T_{(\gamma_n(0),\gamma_n(1))}\p^{\perp_n},$$ where $^{\perp_n}$ denotes orthogonality with respect to $\overline{g_n}$, and $\lim_{n\to+\infty} \frac{\dot{\gamma_n}(1)}{\|\dot{\gamma_n}(1)\|}=v$.

From the limits $$\lim_{n\to+\infty}\frac{\dot{\gamma_n}(0)}{\|\dot{\gamma_n}(0)\|}=v  \;\; \mbox{ and } \;\; \lim_{n\to+\infty}\frac{\|\dot{\gamma_n}(1)\|}{\|\dot{\gamma_n}(0)\|}=1,$$ it follows that also $\lim_{n\to+\infty}\frac{\dot{\gamma_n}(0)}{\|\dot{\gamma_n}(1)\|}=v$. Since $\p$ is compact, this proves the claim that $(v,v)\in T_{(x,x)}\p^\perp$, concluding the proof.
\end{proof}

\begin{remark}
Since admissibility of $\p$ can be characterized by its transversality to $\Delta$, it follows from the Transversality Theorem~\ref{thm:transversality} that GECs are {\em generically} admissible, see Remark~\ref{re:transvgensubmnfld}.
\end{remark}

To end this section, we analyze admissibility of the GECs given in Example~\ref{ex:gecs}.

\begin{example}\label{ex:adgecs}
A fixed endpoints condition $\p=\{p\}\times\{q\}$ is always admissible. Namely, $\p$ is compact and since the tangent space to $\p$ is trivial, it follows that $\A_{g_\mathrm A,\nu,\p}=\A_{g_\mathrm A,\nu}$ for any $\nu$ and $g_\mathrm A\in\met_\nu^k(M)$. Hence the nondegeneracy condition is empty, and (i) and (ii) trivially hold. In addition, regarding condition (iii), it falls in the class of endpoints conditions of the form $\p=P\times\{q\}$, where $P\subset M$ is a compact submanifold, explored in Example~\ref{ex:Pq}. Thus $\p=\{p\}\times\{q\}$ is admissible, even if\footnote{For technical reasons, it necessary to assume $p\neq q$ in the proof of Theorem~\ref{thm:biljavapic}, to guarantee non existence of strongly degenerate geodesics.} $p=q$. Notice that it is not necessary to use Proposition~\ref{prop:admissibility} to infer this conclusion.

Replacing one of the points with a compact submanifold, we fall in the previous case $\p=P\times\{q\}$ discussed in Example~\ref{ex:Pq}. Replacing both points with compact submanifolds gives $\p=P\times Q$, as in Example~\ref{ex:gecs}. Provided that these submanifolds have no $\nu$--topological obstructions, they satisfy (i) and (ii). As for (iii), if $P\cap Q=\emptyset$, it trivially holds as discussed above. However, if $P\cap Q\neq\emptyset$, it is easy to see that $\p$ is transverse to $\Delta$ if and only if $P$ and $Q$ are transverse submanifolds of $M$. In this case, from Proposition~\ref{prop:admissibility}, (iii) holds. In particular, from the Transversality Theorem~\ref{thm:transversality}, two generic compact submanifolds $P$ and $Q$ without $\nu$--topological obstructions give rise to an admissible GEC, since generic submanifolds are transversal, see Remark~\ref{re:transvgensubmnfld}.

Finally, as already mentioned in Remark~\ref{re:diagonaldoesntwork}, the diagonal GEC $\p=\Delta$ is not admissible.
\end{example}

\section{Generic properties of geodesics under GEC}
\label{sec:gnggec}

In this section, we give a detailed proof of the main result of Bettiol and Giamb\`o \cite{metmna}, on genericity of nondegeneracy of $(g,\p)$--geodesics, see Theorem~\ref{thm:bigone}. More precisely, given choices of an index $\nu$, an auxiliary metric $g_\mathrm A\in\met_\nu^k(M)$ satisfying \eqref{eq:boundedfromzero} and an admissible GEC $\p$, we establish genericity of metrics $g\in\A_{g_\mathrm A,\nu,\p}$, see \eqref{eq:anup}, such that the $g$--energy functional \eqref{eq:efunct} is Morse. As explained in the beginning of this chapter, such result extends Theorem~\ref{thm:biljavapic}, which corresponds to the case $\p=\{p\}\times\{p\}$, to the GEC context.

Apart from direct applications to obtain genericity of non focality properties among others, Theorem~\ref{thm:bigone} gives an affirmative answer to some questions of Biliotti, Javaloyes and Piccione \cite{biljavapic} concerning more general settings for their result. For instance, it is conjectured in \cite{biljavapic} that the same genericity result holds for geodesics joining $p$ and $q$ even if $p=q$. Although this result trivially follows from the Bumpy Metric Theorem~\ref{thm:bumpy}, it also follows from Theorem~\ref{thm:bigone} setting $\p=\{p\}\times\{p\}$. More generally, Theorem~\ref{thm:bigone} gives a much wider context in which a property similar to Theorem~\ref{thm:biljavapic} holds. In addition, as mentioned above in Remark~\ref{re:diagonaldoesntwork}, one cannot expect to use GECs to prove the Bumpy Metric Theorem~\ref{thm:bumpy}, for $\p=\Delta$ is not an admissible GEC. In fact, the Bumpy Metric Theorem~\ref{thm:bumpy} will be used in the proof of Theorem~\ref{thm:bigone} in the case $\p\cap\Delta\neq\emptyset$.

Before stating and proving such generic property, we need a generalization of the local perturbation argument employed in the proof of the Weak Bumpy Metric Theorem~\ref{thm:weakbumpy} to verify the transversality condition (ii) of the genericity criterion. Namely, as stated in Remark~\ref{re:whenpertfails}, this perturbation argument only fails in the presence of strongly degenerate geodesics, see Definition~\ref{def:defstdeg}. The following result proves this statement, using minor adaptations to fit the context of GECs. For instance, depending on the geometry of $\p$, there might be $(g,\p)$--geodesics that have infinitely many self intersections, and in this case it is necessary to appeal to a parity trick. Nevertheless, the essential ideas for the local perturbation are the same as in the Weak Bumpy Metric Theorem~\ref{thm:weakbumpy}.

\begin{theorem}\label{thm:transvholds}
Fix $(g_0,\gamma_0)\in\mathcal U$ such that $\frac{\partial E}{\partial\gamma}(g_0,\gamma_0)=0$, see Proposition~\ref{prop:critgenenfunc}. Suppose that the $g_0$--geodesic $\gamma_0$ is {\em not strongly degenerate}\footnote{Recall Definition~\ref{def:defstdeg}.}. Then for every $\p$--Jacobi field $J\in\ker\left[\frac{\partial^2 E}{\partial\gamma^2}(g_0,\gamma_0)\right]\setminus\{0\}$ along $\gamma_0$, there exists $h\in\mathds E$ such that
\begin{equation}\label{eq:mixedderneq0}
\frac{\partial^2 E}{\partial g\partial\gamma}(g_0,\gamma_0)(h,J) \stackrel{\eqref{eq:mixedderivative}}{=}\int_0^1 h(\dot{\gamma_0},\D J)+\tfrac{1}{2}\nabla h(J,\dot{\gamma_0},\dot{\gamma_0})\;\dd t\neq 0.
\end{equation}
\end{theorem}

\begin{proof}
This proof is in great part adapted from \cite[Proposition 4.3]{biljavapic}. Let $J\in\ker\left[\frac{\partial^2 E}{\partial\gamma^2}(g_0,\gamma_0)\right]\setminus\{0\}$ be a nontrivial $\p$--Jacobi field along $\gamma_0$. The main idea is to use a {\em local perturbation argument} that will employ $J$ to construct\footnote{Using the extension Lemma~\ref{le:extension}.} a section $h$ with the required regularity, having compact support contained in a neighborhood of a segment of $\gamma_0$ where $J$ is not parallel to $\dot{\gamma_0}$.

We will split the proof of the existence of such $h$ such that \eqref{eq:mixedderneq0} holds in three claims, to deal with the possibly infinite number of self intersections of $\gamma_0$. The geodesic $\gamma_0$ has either infinite or finite self intersections. From Proposition~\ref{prop:selfintersections}, these possibilities correspond respectively to $\gamma_0$ being a portion of a periodic geodesic with period $\omega<1$ or not. The first possibility will be subdivided again in two cases, namely corresponding to when $\gamma_0$ has endpoints that coincide or not. Notice that the second possibility, with a finite number of self intersections, covers {\em prime geodesics}, i.e. periodic geodesics that are not $k$--fold iteration of other $g_0$--geodesics. Notice that under the hypotheses on $\gamma_0$, these cases cover all possibilities, since $\gamma_0$ must fall in one of the following cases above discussed:
\begin{itemize}
\item[-- ] $\gamma_0$ is {\em not} a portion of a periodic geodesic of period $\omega<1$;
\item[-- ] $\gamma_0$ is a portion of a periodic geodesic of period $\omega<1$, however with distinct endpoints;
\item[-- ] $\gamma_0$ is an iterate geodesic.
\end{itemize}
We will respectively cover each of these possibilities in the following three claims.

\begin{claim}\label{cl:finiteinter}
The theorem holds if $\gamma_0$ is {\em not} a portion of a periodic geodesic with period $\omega<1$.\footnote{Notice that $\gamma_0$ may be a prime geodesic with these hypothesis.}
\end{claim}

From Proposition~\ref{prop:selfintersections}, in this case $\gamma_0$ has only a finite number of self intersections. Thus, there exists a nonempty open interval $I\subset [0,1]$ such that 
\begin{itemize}
\item[$I$--1:] $\gamma_0(I)\cap\gamma_0([0,1]\setminus I)=\emptyset$;
\item[$I$--2:] $J$ is not parallel to $\dot{\gamma_0}$ at any time in $I$.
\end{itemize}
Indeed such an interval exists, since the first condition is feasible due to the finiteness of self intersections of $\gamma_0$ and the second is also admissible as a consequence of Corollary~\ref{cor:pjacobiparallelfinite}.

In order to construct the required $h\in\mathds{E}$ such that $\frac{\partial^2 E}{\partial g\partial\gamma}(g_0,\gamma_0)(h,J)\ne 0$, we apply Lemma~\ref{le:extension} to the vector bundle $TM^*\vee TM^*$. Let $U\subset M$ be any open subset containing $\gamma_0(I)$ such that
\begin{itemize}
\item[$U$:] $\gamma_0(t)\in U$ if and only if $t\in I$.
\end{itemize}
For instance, $U$ can be taken as the complement of $\gamma_0([0,1]\setminus I)$. Let $H\in\sect^k(\gamma_0^*E)$ be the identically null section and $K\in\sect^k(\gamma_0^*E)$ any symmetric bilinear form continuous on $t$, that satisfies $$K(\dot{\gamma_0},\dot{\gamma_0})\geq 0\;\;\mbox{ and }\;\;\int_I K(t)(\dot{\gamma_0}(t),\dot{\gamma_0}(t))\;\dd t>0,$$
for instance, $K(t)=g_\mathrm R(\gamma_0(t))$. Reducing the size of $I$ if necessary, we may assume that the result of Lemma~\ref{le:extension} holds.

\begin{figure}[htf]
\begin{center}
\includegraphics[scale=0.5]{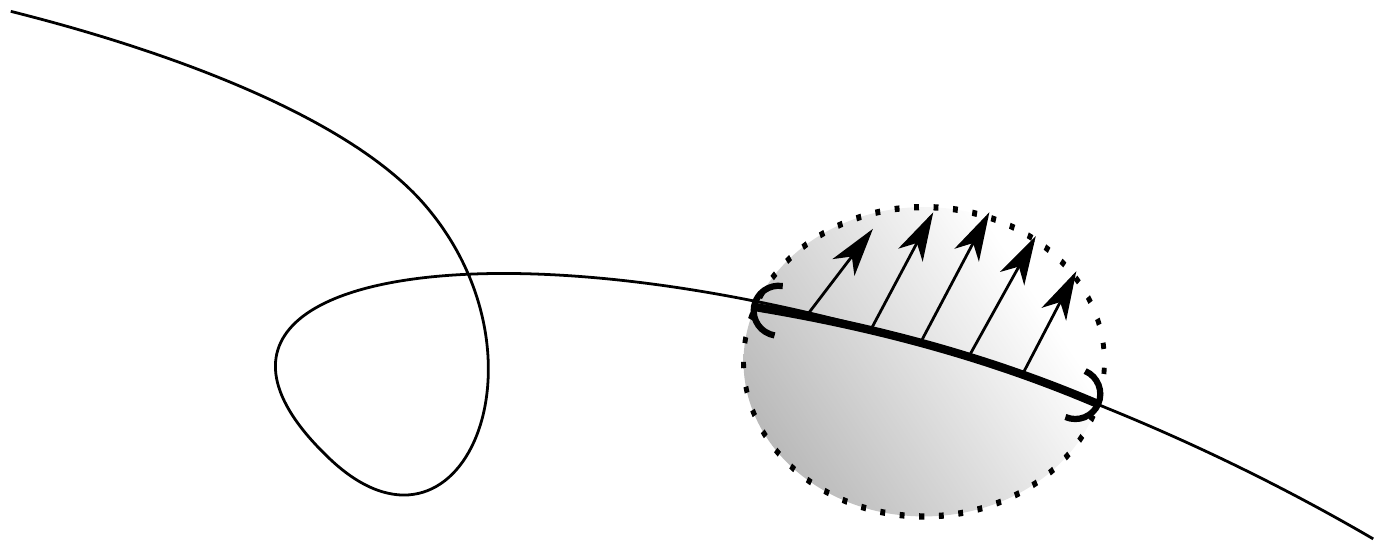}
\begin{pgfpicture}
\pgfputat{\pgfxy(-6,2.8)}{\pgfbox[center,center]{$\gamma_0$}}
\pgfputat{\pgfxy(-3.3,0)}{\pgfbox[center,center]{$U$}}
\pgfputat{\pgfxy(-2,2)}{\pgfbox[center,center]{$K$}}
\pgfputat{\pgfxy(-1,0.9)}{\pgfbox[center,center]{$\gamma_0(I)$}}
\end{pgfpicture}
\end{center}
\caption{Local perturbation along $\gamma_0$ to obtain the section $h\in\sect^k(TM^*\vee TM^*)$ that ensures \eqref{eq:mixedderneq0}.}\label{fig:localperturb}
\end{figure}

This gives a globally defined section $h\in\sect^k(TM^*\vee TM^*)$ with compact support contained in $U$ such that
\begin{equation}\label{eq:hhk}
h(\gamma_0(t))=0\;\;\mbox{ and }\;\;\nabla_{J(t)} h=K(t), \quad\mbox{ for all } t\in I.
\end{equation}
Clearly $h\in\mathds E$, since all $C^k$ sections of $E$ with compact support are in $\mathds E$. Finally, from the above construction,
\begin{eqnarray*}
\frac{\partial^2 E}{\partial g\partial\gamma}(g_0,\gamma_0)(h,J) &=& \int_0^1 h(\dot{\gamma_0},\D J)+\tfrac{1}{2}\nabla h(J,\dot{\gamma_0},\dot{\gamma_0})\;\dd t\\
&=&\tfrac{1}{2}\int_I K(t)(\dot{\gamma_0}(t),\dot{\gamma_0}(t))\;\dd t\\
&>&0.
\end{eqnarray*}

\begin{claim}\label{cl:infiniteinter1}
The theorem holds if $\gamma_0$ is a portion of a periodic geodesic with period $\omega<1$, however with distinct endpoints.
\end{claim}

To prove this second claim we adapt the local perturbation argument above using a {\em parity trick} to find an open interval $I\subset [0,1]$ with the required properties. Let $p=\gamma_0(0)$ and $q=\gamma_0(1)$ and define
$$t_*=\min\,\{t>0:\gamma_0(t)=q\}\;\; \mbox{ and }\;\; k_*=\max\,\{k\in\N:k\omega<1\}.$$
Then, it follows that
$$k_*\ge 1,\quad 0<t_*<\omega,\quad k_*\omega+t_*=1.$$
Notice that $p\ne q$ is a necessary hypothesis here. Indeed, $p=q$ would imply $t_*=\omega$, and the parity trick below fails in this case.

Consider the geodesics $\gamma_1=\gamma_0\vert_{[0,t_*]}$ and $\gamma_2=\gamma_0\vert_{[t_*,T]}$, where $\gamma_2$ has the opposite orientation of $\gamma_0$. Both $\gamma_1$ and $\gamma_2$ have finitely many self intersections and join $p$ and $q$. Thus Claim~\ref{cl:finiteinter} applies to both and hence there exist nonempty open intervals $I_1=\,]a_1,b_1[\,\subset[0,t_*]$ and $I_2=\,]a_2,b_2[\,\subset[t_*,\omega]$ such that
\begin{itemize}
\item[$I_1$--1:] $\gamma_0(I_1)\cap\gamma_0\big(([0,t_*]\setminus I_1)\cup [t_*,\omega]\big)=\emptyset$;
\item[$I_2$--1:] $\gamma_0(I_2)\cap\gamma_0\big(([t_*,\omega]\setminus I_2)\cup [0,t_*]\big)=\emptyset$.
\end{itemize}
Analogously to the proof of Claim~\ref{cl:finiteinter}, there exist open subsets $U_1,U_2\subset M$, with $\gamma(I_j)\subset U_j$, $j=1,2$, satisfying
\begin{itemize}
\item[$U_1$:] $\gamma_0(t)\in U_1$ for some $t\in I_1$ if and only if $t-i\omega\in I_1$ for some $i\in\{0,\ldots,k_*\}$;
\item[$U_2$:] $\gamma_0(t)\in U_2$ for some $t\in I_2$ if and only if $t-i\omega\in I_2$ for some $i\in\{0,\ldots,k_*-1\}$.
\end{itemize}
Once more, these may be taken as $U_j=\gamma_0([0,1]\setminus I_j)$.

For $j=1,2$, consider the $\p$--Jacobi fields $W^j$ along $\gamma_j$ defined by
\begin{equation}\label{eq:defW1W2}
\begin{aligned}
W^1(t)&=\displaystyle\sum_{i=0}^{k_*}J(t+i\omega),\quad t\in I_1\\
W^2(t)&=\displaystyle\sum_{i=0}^{k_*-1} J(t+i\omega),\quad t\in I_2
\end{aligned}
\end{equation}
It is impossible that both $W^1$ and $W^2$ are everywhere parallel to $\dot{\gamma_0}$ at $I_1$ and $I_2$ respectively, for otherwise from \eqref{eq:defW1W2} one would easily conclude that $J$ is everywhere parallel to $\dot{\gamma_0}$, contradicting Lemma~\ref{le:parallelfinite} (and Corollary~\ref{cor:pjacobiparallelfinite}). Thus, we may assume that, for instance $W^1$, is not everywhere parallel to $\dot{\gamma_0}$ on $I_1$. This means that there are only points where $W^1(t)$ is parallel to $\dot{\gamma_0}(t)$. Reducing the size of $I_1$ if necessary, we can assume that $W^1(t)$ is never a multiple of $\dot{\gamma_0}(t)$ on $I_1$.

At this point it is possible to repeat exactly the same construction from Claim~\ref{cl:finiteinter} replacing the Jacobi field $J$ with $W^1$. From Lemma~\ref{le:extension}, there exists $h\in\sect^k(TM^*\vee TM^*)$ with compact support contained in $U_1$ with prescribed values $H$ and covariant derivative $K$ in the direction $W^1$ along $\gamma_0\vert_{I_1}$, analogously to \eqref{eq:hhk}. Choosing $H$ and $K$ as in the proof of Claim~\ref{cl:finiteinter}, it follows that
\begin{eqnarray}\label{eq:mixedderneq0holds}
\frac{\partial^2 E}{\partial g\partial\gamma}(g_0,\gamma_0)(h,J) &=& \int_0^1 h(\dot{\gamma_0},\D J)+\tfrac12\nabla h(J,\dot{\gamma_0},\dot{\gamma_0})\;\dd t \nonumber \\
&=& \tfrac12\sum_{i=0}^{k_*}\int_{a_1+i\omega}^{b_1+i\omega}\nabla h(J,\dot{\gamma_0},\dot{\gamma_0})\;\dd t\nonumber \\
&=& \tfrac12\int_{a_1}^{b_1}\nabla h(W^1,\dot{\gamma_0},\dot{\gamma_0})\;\dd t \\
&=& \tfrac12\int_I K(t)(\dot{\gamma_0}(t),\dot{\gamma_0}(t))\;\dd t \nonumber \\
&>&0.\nonumber 
\end{eqnarray}

\begin{claim}\label{cl:infiniteinter2}
The theorem holds if $\gamma_0$ is an iterate geodesic.
\end{claim}

The local perturbation argument used in Claims~\ref{cl:finiteinter} and~\ref{cl:infiniteinter2} can also be adapted to this last case where $\gamma_0$ is an iterate geodesic, provided it is {\em not strongly degenerate}. This is a simple scholium from Claim~\ref{cl:infiniteinter1}. Under these hypotheses, $\gamma_0$ is a periodic geodesic with period $\omega=\tfrac{1}{k}$, for some $k\ge 2$. Notice that for $k=1$, $\gamma_0$ is a prime geodesic and this case was already covered by Claim~\ref{cl:finiteinter}.

Analogously to \eqref{eq:defW1W2}, define the $\p$--Jacobi field $$W(t)=\sum\limits_{i=0}^{k-1} J\left(t+\tfrac{i}{k}\right),\quad t\in [0,1].$$ We claim that a sufficient condition to apply the local perturbation argument is that
\begin{equation}\label{eq:whenworks}
W(t_0)\ne 0\;\;\mbox{ for some }\;\; t_0\in [0,1].
\end{equation}
Before verifying that indeed this is a sufficient condition, notice that since $\gamma_0$ is {\em not} strongly degenerate, \eqref{eq:whenworks} clearly holds, see Definition~\ref{def:defstdeg}.

Finally, let us prove that \eqref{eq:whenworks} allows to apply the local perturbation argument as above. By continuity, from \eqref{eq:whenworks}, there exists a nonempty open interval $I\subset [0,1]$ around such $t_0$ where $W$ does not vanish, with the same properties of the intervals $I$ considered above. Namely,
\begin{itemize}
\item[$I$--1:] $\gamma_0\big([0,\tfrac1k]\setminus I\big)\cap\gamma_0(I)=\emptyset$;
\item[$I$--2:] $W$ is not parallel to $\dot{\gamma_0}$ at any $t\in I$.
\end{itemize}
Once more, the second condition is feasible as a consequence of Corollary~\ref{cor:pjacobiparallelfinite}. It is also easy to obtain an open neighborhood $U$ of $\gamma_0(I)$ such that
\begin{itemize}
\item[$U$:] $\gamma_0(t)\in U$ for some $t\in I$ if and only if $t-\tfrac{i}{k}\in I$ for some $i\in\{0,\ldots,k\}$.
\end{itemize}
Again, take for instance $U=M\setminus\gamma_0([0,1]\setminus I)$. This gives a situation totally analogous to the one illustrated in Figure~\ref{fig:localperturb}.

Reducing the size of $I$ if necessary, we may assume that the result of Lemma~\ref{le:extension} holds. Once more, this gives a globally defined section $h\in\sect^k(TM^*\vee TM^*)$ with compact support contained in $U$ and prescribed values $H$ and covariant derivative $K$ in the direction $W$ along $\gamma_0\vert_{I}$. Once more, $h\in\mathds E$, since it has compact support. Prescribing appropriate values again for $H$ and $K$, exactly as in the end of the proof of Claim~\ref{cl:infiniteinter1}, a computation similar to \eqref{eq:mixedderneq0holds} proves that \eqref{eq:mixedderneq0} holds for this $h$.

This concludes the proof, since all the three possibilities for $\gamma_0$ described above have been covered.
\end{proof}

\begin{remark}
Theorem~\ref{thm:transvholds} guarantees that transversality condition (ii) of the Abstract Genericity Criterion~\ref{crit:abstractgenericity} holds for the geodesic setup unless $\gamma_0$ is a strongly degenerate geodesic.\footnote{Recall Definition~\ref{def:defstdeg}.} More precisely, from Remark~\ref{re:whenpertfails}, this is the only case in which the local perturbation argument above used fails. Recall that if $\gamma_0$ is a strongly degenerate $(g_0,\p)$--geodesic, then it admits a nontrivial Jacobi field $J$ which satisfies (b) of Definition~\ref{def:defstdeg}. For this $J$, the right--hand side of \eqref{eq:mixedderivative} is identically null for any section $h$ of $TM^*\vee TM^*$, hence \eqref{eq:mixedderneq0} trivially fails.
\end{remark}

We are now ready to prove our main genericity result, on nondegeneracy of geodesics under GECs. The proof will be done in two steps. First, we consider the case $\p\cap\Delta=\emptyset$ and apply the Abstract Genericity Criterion~\ref{crit:abstractgenericity} using a local perturbation argument, proved in Theorem~\ref{thm:transvholds}. Secondly, we treat the special case $\p\cap\Delta\neq\emptyset$ using its admissibility, since the abstract criterion fails due to the possible presence of strongly degenerate geodesics.

We stress that this case is \emph{not} an immediate consequence of the first case $\p\cap\Delta=\emptyset$ and the Bumpy Metric Theorem~\ref{thm:bumpy}. Indeed, if $\gamma\in\op(M)$ is a periodic $(g,\p)$--geodesic, the notions of degeneracy as a $(g,\p)$--geodesic and as a periodic geodesic \emph{do not} coincide, see Remark~\ref{re:degeneracynotions}. For this, we use a more elaborate argument, which employs both the Bumpy Metric Theorem~\ref{thm:bumpy} and the Abstract Genericity Criterion~\ref{crit:abstractgenericity} in a different way.

\begin{theorem}\label{thm:bigone}
Let $M$ be a smooth $m$--dimensional manifold and fix $\mathds{E}$ a separable $C^k$ Whitney type Banach space of sections of $E=TM^*\vee TM^*$ that tend to zero at infinity, with $k\geq 3$. Fix $\nu\in\{0,\dots,m\}$ an index and let $g_\mathrm A\in\met_\nu^k(M)$ be such that $$\sup_{x\in M}\|g_\mathrm A(x)^{-1}\|_\mathrm R<+\infty.$$ Consider $\p$ an $(g_\mathrm A,\nu)$--admissible GEC. Then the following is a generic subset of $\A_{g_\mathrm A,\nu,\p}$\footnote{Recall Proposition~\ref{prop:nondegenerateopen} and \eqref{eq:anup}.}
$$\mathcal{G}_\p(M)=\left\{g\in\A_{g_\mathrm A,\nu,\p}: \begin{array}{c} \text{ all } (g,\p)\text{--geodesics }\gamma\in\op(M) \\ \text{ are nondegenerate }\end{array}\right\}.$$
\end{theorem}

\begin{proof}
We shall prove the genericity of $\mathcal{G}_\p(M)$ in $\A_{g_\mathrm A,\nu,\p}$ through a sequence of four claims. The first claim establishes the genericity of $\mathcal G_\p(M)$ if $\p\cap\Delta=\emptyset$, using the Abstract Genericity Criterion~\ref{crit:abstractgenericity}. The second claim deals with the case $\p\cap\Delta\ne\emptyset$, setting the context to prove genericity of $\mathcal G_\p(M)$ for such GECs using the Abstract Genericity Criterion~\ref{crit:abstractgenericity} and the Bumpy Metric Theorem~\ref{thm:bumpy} in a more technical argument. Finally, the last two claims guarantee that the second claim holds.

\begin{claim}\label{cl:bigclaim1}
$\mathcal G_\p(M)$ is generic in $\A_{g_\mathrm A,\nu,\p}$ if $\p\cap\Delta=\emptyset$.
\end{claim}

To prove genericity of $\mathcal G_\p(M)$ in this case, we apply the Abstract Genericity Criterion~\ref{crit:abstractgenericity} to the generalized energy functional \eqref{eq:efunct}, $$E:\mathcal U\ni (g,\gamma)\longmapsto E_g(\gamma)=\tfrac{1}{2}\int_0^1 g(\dot\gamma,\dot\gamma)\;\dd t\in\R,$$ where $\mathcal U=\A_{g_\mathrm A,\nu,\p}\times\op(M)$. Recall that this criterion states that under two conditions (i) and (ii) on the points $(g_0,\gamma_0)\in\mathcal U$ such that $\frac{\partial E}{\partial\gamma}(g_0,\gamma_0)=0$, the set of parameters $g\in\A_{g_\mathrm A,\nu,\p}$ such that $E_g$ is a Morse function is generic in $\A_{g_\mathrm A,\nu,\p}$. From Proposition~\ref{prop:critgenenfunc}, $E_g$ is Morse if and only if all $(g,\p)$--geodesics are nondegenerate. Thus, it suffices to verify these two conditions to obtain the desired genericity of $\mathcal G_\p(M)$ for $\p\cap\Delta=\emptyset$.

Condition (i) of the Abstract Genericity Criterion~\ref{crit:abstractgenericity} is an immediate consequence of Proposition~\ref{prop:fredholmness}. Namely, this proposition asserts that given $(g_0,\gamma_0)\in\mathcal U$ such that\footnote{See Proposition~\ref{prop:critgenenfunc} for a characterization of this fact.} $\frac{\partial E}{\partial\gamma}(g_0,\gamma_0)=0$, the index form $$\frac{\partial^2 E}{\partial\gamma^2}(g_0,\gamma_0):T_{\gamma_0}\op(M)\la T_{\gamma_0}\op(M)^*\cong T_{\gamma_0}\op(M)$$ given by \eqref{eq:indexform} is represented by a self--adjoint Fredholm operator. This is exactly the content of condition (i).

As for condition (ii) of the Abstract Genericity Criterion~\ref{crit:abstractgenericity}, it is an immediate consequence of Theorem~\ref{thm:transvholds}. Namely, condition (ii) asserts that given $(g_0,\gamma_0)\in\mathcal U$ such that $\frac{\partial E}{\partial\gamma}(g_0,\gamma_0)=0$, for all $J\in\ker\left[\frac{\partial^2 E}{\partial\gamma^2}(g_0,\gamma_0)\right]\setminus\{0\}$ there must exist $h\in T_{g_0}\A_{g_\mathrm A,\nu,\p}$ such that the mixed derivative \eqref{eq:mixedderivative}, $$\frac{\partial^2 E}{\partial g\partial\gamma}(g_0,\gamma_0)(h,J)=\int_0^1 h(\dot{\gamma_0},\D J)+\tfrac{1}{2}\nabla h(J,\dot{\gamma_0},\dot{\gamma_0})\;\dd t$$ does not vanish. Notice that since we are assuming $\p\cap\Delta=\emptyset$, the $g_0$--geodesic $\gamma_0$ has distinct endpoints. Hence, from Claims~\ref{cl:finiteinter} and~\ref{cl:infiniteinter1} in the proof of Theorem~\ref{thm:transvholds}, condition (ii) is verified in this case.

Therefore, if $\p\cap\Delta=\emptyset$, the Abstract Genericity Criterion~\ref{crit:abstractgenericity} implies that $\mathcal G_\p(M)$ is generic in $\A_{g_\mathrm A,\nu,\p}$. This concludes the proof of Claim~\ref{cl:bigclaim1}.

\begin{claim}\label{cl:bigclaim2}
$\mathcal G_\p(M)$ is generic in $\A_{g_\mathrm A,\nu,\p}$ if $\p\cap\Delta\ne\emptyset$.
\end{claim}

Define for each $n\in\N$,
\begin{equation}\label{eq:Ralpha}
\mathcal{R}_n=\left\{g\in\mathcal{A}_{g_\mathrm A,\nu,\p}:\begin{array}{c}\mbox{ all } (g,\p)\mbox{--geodesics }\gamma\mbox{ with } \\ L_\mathrm R (\gamma)\leq n\mbox{ are nondegenerate} \end{array}\right\}.
\end{equation}
Since $\mathcal{G}_\p(M)=\bigcap_{n\in\N}\mathcal{R}_n$, from Lemma~\ref{le:gencountintersect} it suffices to prove that each $\mathcal R_n$ is open and dense in $\A_{g_\mathrm A,\nu,\p}$. We now prove separately that each $\mathcal R_n$ is open, using the Arzel\`a--Ascoli Theorem; and dense, using the Abstract Genericity Criterion~\ref{crit:abstractgenericity} together with the Bumpy Metric Theorem~\ref{thm:bumpy}.

\begin{claim}\label{cl:bigclaim3}
$\mathcal R_n$ is open in $\mathcal A_{g_\mathrm A,\nu,\p}$ for every $n\in\N$.
\end{claim}

Let $\{g_i\}_{i\in\N}$ be a convergent sequence in $\mathcal A_{g_\mathrm A,\nu,\p}\setminus\mathcal R_n$, with $\lim_{i\to+\infty} g_i=g_\infty$. From the definition of $\mathcal R_n$, for each $i\in\N$ there exists a degenerate $(g_i,\p)$--geodesic $\gamma_i$ with $L_\mathrm R(\gamma_i)\leq n$. Since $\p$ is compact and $L_\mathrm R(\gamma_i)\leq n$, by the Arzel\`a--Ascoli Theorem, up to subsequences, there exists a convergent sequence $\{t_i\}_{i\in\N}$ in $[0,1]$ with $\lim_{i\to+\infty} t_i=t_\infty$ such that $\left\|\dot{\gamma_i}(t_i)\right\|_\mathrm R\leq n$ for all $i\in\N$, and $\dot{\gamma_i}(t_i)$ converges to $v\in T_{p_\infty}M$, with $p_\infty=\lim_{i\to+\infty}\gamma_i(t_\infty)$. From continuous dependence of ODE's solutions on initial conditions, it is easy to see that the solution $\gamma_\infty$ of $\D^{g_\infty}\dot\gamma=0$ with initial conditions $\gamma(t_\infty)=p_\infty$ and $\dot\gamma(t_\infty)=v$ is the $C^2$--limit of the sequence of geodesics $\gamma_i$. Therefore $\gamma_\infty$ is a $(g_\infty,\p)$--geodesic, and obviously $L_\mathrm R(\gamma_\infty)\leq n$.

Moreover, $\gamma_\infty$ is \emph{non constant}. This follows from the fact that $\p$ is admissible.\footnote{Indeed, condition (iii) of Definition~\ref{def:admgbc} of admissibility is used only in this part of the proof.} Hence there exists $a>0$ such that $L_\mathrm R(\gamma_i)\geq a$ for large $i$, since $g_i$ will be in any open neighborhoods of $g_\infty$ in $\mathcal A_{g_\mathrm A,\nu,\p}$.

In order to prove that such $\gamma_\infty$ is a \emph{degenerate} $(g_\infty,\p)$--geodesic, for each $i$ let $J_i$ be a nontrivial $\p$--Jacobi field along $\gamma_i$. Then $J_i$ is the solution of a second order ODE whose initial conditions converge to initial conditions of the $\p$--Jacobi fields equation along the $g_\infty$--geodesic $\gamma_\infty$. More precisely, for each $i$, $J_i$ is a nontrivial $\p$--Jacobi field, that in particular satisfies the $g_i$--Jacobi equation \eqref{eq:jacobi}, $$\D^{g_i}J_i=R^{g_i}(\dot{\gamma_i},J_i)\dot{\gamma_i}.$$ By adding a suitable multiple of $\dot{\gamma_i}$, we may assume that $J_i(0)$ is $g_\mathrm R$--orthogonal to $\dot{\gamma_i}(0)$. In addition, using an adequate normalization it is also possible to assume that $\max\{\|J_i(0)\|_{\mathrm R},\|\D^{g_i} J_i(0)\|_{\mathrm R}\}=1.$ Again, up to subsequences, the initial conditions converge, $$\lim_{i\to+\infty} J_i(0)=v\in T_{\gamma_\infty(0)}M, \;\quad\; \lim_{i\to+\infty} \D^{g_i} J_i(0)=w\in T_{\gamma_\infty(0)}M.$$ By continuity, $v$ is $g_\mathrm R$--orthogonal to $\dot{\gamma_\infty}(0)$, and \begin{equation}\label{eq:maxinfty} \max\{\|v\|_{\mathrm R},\|w\|_{\mathrm R}\}=1. \end{equation} The solution of the $g_\infty$--Jacobi equation along $\gamma_\infty$ with the above limit initial conditions is a $\p$--Jacobi field $J_\infty$ that is also the $C^2$--limit of the $\p$--Jacobi fields $J_i$. Finally, it is not a multiple of the tangent field $\dot{\gamma_\infty}$. Indeed, if $J_\infty$ were a multiple of $\dot{\gamma_\infty}$, since $v$ is $g_\mathrm R$--orthogonal to $\dot{\gamma_\infty}(0)$, it would be $v=0$ and $w=0$, which contradicts \eqref{eq:maxinfty}. Hence $g_\infty\in\mathcal A_{g_\mathrm A,\nu,\p}\setminus\mathcal R_n$, which proves that $\mathcal R_n$ is an open subset.

\begin{claim}\label{cl:bigclaim4}
$\mathcal R_n$ is dense in $\mathcal A_{g_\mathrm A,\nu,\p}$ for every $n\in\N$.
\end{claim}

For each $n\in\N$, define the following subsets of $\mathcal{A}_{g_\mathrm A,\nu,\p}$, $$\mathcal{B}_n=\left\{g\in\mathcal{A}_{g_\mathrm A,\nu,\p}:\begin{array}{c}\mbox{ all periodic } g\mbox{--geodesics }\gamma\mbox{ with }  \\ L_\mathrm R (\gamma)\leq n\mbox{ are nondegenerate} \end{array}\right\},$$ \smallskip $$ \mathcal{D}_n=\left\{g\in\mathcal{A}_{g_\mathrm A,\nu,\p}: \begin{array}{c} \mbox{ all } g\mbox{--geodesics }\gamma\mbox{ with }L_\mathrm R (\gamma)<n \mbox{ that are} \\ \mbox{periodic or }(g,\p)\mbox{--geodesics are nondegenerate} \end{array} \right\}.$$

It is easy to see that for each $n$, $\mathcal{D}_{n+1}\subset\mathcal{R}_n$. From the Bumpy Metric Theorem~\ref{thm:bumpy}, each $\mathcal B_n$ is open and dense in $\mathcal A_{g_\mathrm A,\nu,\p}$. Hence to prove that $\mathcal R_n$ is dense in $\mathcal A_{g_\mathrm A,\nu,\p}$, it suffices to prove that $\mathcal D_n$ is dense in $\mathcal B_n$. To this aim, for each $n$ we use the Abstract Genericity Criterion~\ref{crit:abstractgenericity} again. The setting is the same geodesic setup used in Claim~\ref{cl:bigclaim1}, with the only difference being the domain of the generalized energy functional \eqref{eq:efunct}, which we now take as the open subset $$\mathcal U_n=\mathcal B_n\times\{\gamma\in\op(M):L_\mathrm R(\gamma)<n\}.$$ This means that we are dealing only with bumpy metrics, i.e., without degenerate periodic geodesics.

Let us prove that conditions (i) and (ii) of the Abstract Genericity Criterion~\ref{crit:abstractgenericity} are verified also in this context, concluding the proof. Since it is local, condition (i) follows again from Proposition~\ref{prop:fredholmness}. Theorem~\ref{thm:transvholds} implies that the transversality condition (ii) would only fail in the presence of strongly degenerate geodesics. Nevertheless, there cannot be critical points of the form $(g_0,\gamma_0)$, where $\gamma_0$ is a strongly degenerate $(g_0,\p)$--geodesic. This follows from Proposition~\ref{prop:stronglydegenerate}, since $\gamma_0$ would also be a {\em degenerate periodic geodesic}, contradicting $g_0\in\mathcal B_n$. Thus, condition (ii) is verified and the Abstract Genericity Criterion~\ref{crit:abstractgenericity} applies also in this setting. Therefore $\mathcal R_n$ is generic, in particular dense, in $\mathcal B_n$ hence also in $\A_{g_\mathrm A,\nu,\p}$.

Therefore, each $\mathcal R_n$ is open and dense in $\A_{g_\mathrm A,\nu,\p}$, hence generic. From Lemma~\ref{le:gencountintersect}, the countable intersection $\mathcal{G}_\p(M)=\bigcap_{n\in\N} \mathcal{R}_n$ is also generic in $\mathcal A_{g_\mathrm A,\nu,\p}$. This concludes the proof of Claim~\ref{cl:bigclaim2}, which combined with Claim~\ref{cl:bigclaim1}, implies that $\mathcal G_\p(M)$ is generic in $\mathcal A_{g_\mathrm A,\nu,\p}$, concluding the proof.
\end{proof}

We end this section with a few examples of applications of Theorem~\ref{thm:bigone} in the case of the admissible GECs given in Example~\ref{ex:adgecs}, regarding conjugacy and focality properties.

\begin{corollary}\label{cor:nonconjugacy}
Let $p,q\in M$. For a generic metric $g\in\A_{g_\mathrm A,\nu}$, the points $p$ and $q$ are not $g$--conjugate.
\end{corollary}

\begin{proof}
Set $\p=\{p\}\times\{q\}$. As explained in Example~\ref{ex:adgecs}, this is always an admissible GEC, even if $p=q$. Thus, Theorem~\ref{thm:bigone} applies and gives genericity of the set $\mathcal G_\p(M)$ of metrics $g$ in $\A_{g_\mathrm A,\nu,\p}=\A_{g_\mathrm A,\nu}$ such that all $g$--geodesics joining $p$ and $q$ are nondegenerate. From Example~\ref{ex:jac}, this set coincides with the set of metrics $g\in\A_{g_\mathrm A,\nu}$ such that $p$ and $q$ are not $g$--conjugate, concluding the proof.
\end{proof}

\begin{remark}
The above corollary may also be seen as a corollary of Theorem~\ref{thm:biljavapic}, by Biliotti, Javaloyes and Piccione \cite{biljavapic}. Nevertheless, the case $p=q$ of geodesic loops was left open in \cite{biljavapic}, and is proved to hold with the above GECs approach.
\end{remark}

\begin{corollary}\label{cor:nonfocality}
Let $P$ be a submanifold of $M$ and $q\in M$, such that $\p=P\times\{q\}$ is admissible.\footnote{Admissibility of this class of GECs is discussed in Example~\ref{ex:adgecs}.} For a generic metric $g\in\A_{g_\mathrm A,\nu,\p}$, $q$ is not $g$--focal to $P$.
\end{corollary}

\begin{proof}
Applying Theorem~\ref{thm:bigone} to $\p=P\times\{q\}$ we obtain genericity of the set $\mathcal G_\p(M)$ of metrics $g$ in $\A_{g_\mathrm A,\nu,\p}$ such that all $(g,\p)$--geodesics are nondegenerate. From Examples~\ref{ex:gecs}, \ref{ex:gecs2} and \ref{ex:jac}, these are $g$--geodesics $\gamma:[0,1]\to M$ that are $g$--orthogonal to $P$ at $\gamma(0)$ and do not admit any $g$--Jacobi field $J$ satisfying $J(0)\in T_{\gamma(0)}P$, $J(1)=0$ and $$\D^g J(0)+\s^P_{\gamma(0)}(J(0))\in T_{\gamma(0)}P^\perp.$$ From Definition~\ref{def:focalpoint}, the generic set $\mathcal G_\p(M)$ coincides with the set of metrics $g\in\A_{g_\mathrm A,\nu,\p}$ such that $q$ is not $g$--focal to $P$, concluding the proof.
\end{proof}

\begin{corollary}\label{cor:nonfocality2}
Let $P$ and $Q$ be submanifolds of $M$, such that $\p=P\times Q$ is admissible.\footnote{Admissibility of this class of GECs is also discussed in Example~\ref{ex:adgecs}.} For a generic metric $g\in\A_{g_\mathrm A,\nu,\p}$, $P$ and $Q$ are not $g$--focal.
\end{corollary}

\begin{proof}
Applying Theorem~\ref{thm:bigone} to $\p=P\times Q$ we obtain genericity of the set $\mathcal G_\p(M)$ of metrics $g$ in $\A_{g_\mathrm A,\nu,\p}$ such that all $(g,\p)$--geodesics are nondegenerate. From Examples~\ref{ex:gecs}, \ref{ex:gecs2} and \ref{ex:jac}, these are $g$--geodesics $\gamma:[0,1]\to M$ that are $g$--orthogonal to $P$ at $\gamma(0)$ and $Q$ at $\gamma(1)$ and do not admit any $g$--Jacobi field $J$ satisfying $J(0)\in T_{\gamma(0)}P$, $J(1)\in T_{\gamma(1)}Q$ and
\begin{equation*}
\begin{aligned}
\D^g J(0)+\s_{\dot\gamma(0)}^P(J(0)) &\in T_{\gamma(0)}P^\perp \\
\D^g J(1)+\s_{\dot\gamma(1)}^Q(J(1)) &\in T_{\gamma(1)}Q^\perp,
\end{aligned}
\end{equation*}
where $^\perp$ is orthogonality with respect to the metrics on $P$ and $Q$ induced by $g$. From Definition~\ref{def:focalsubmanifolds}, the generic set $\mathcal G_\p(M)$ coincides with the set of metrics $g\in\A_{g_\mathrm A,\nu,\p}$ such that $P$ and $Q$ are not $g$--focal, concluding the proof.
\end{proof}

\section{\texorpdfstring{Genericity in the $C^\infty$--topology}{Genericity in the smooth topology}}
\label{sec:smooth2}

Analogously to Section~\ref{sec:smooth1}, in this section we extend the generic property described in Theorem~\ref{thm:bigone} from the case of the $C^k$--topology to the $C^\infty$--topology. Consider once more an index $\nu\in\{0,\dots,m\}$, a smooth auxiliary metric $g_\mathrm A\in\met_\nu^\infty(M)$, such that $\sup_{x\in M}\|g_\mathrm A(x)^{-1}\|_\mathrm R<+\infty$, and a $(g_\mathrm A,\nu)$--admissible GEC $\p$. For each $k\geq 3$, let $\mathds{E}^k$ be any separable $C^k$ Whitney type Banach space of sections of $TM^*\vee TM^*$ that tend to zero at infinity, for instance $\mathds E^k=\sect_0^k(TM^*\vee TM^*)$. Then, from Proposition~\ref{prop:affineworks},
\begin{equation*}
\A_{g_\mathrm A,\nu}^k=(g_\mathrm A+\mathds E^k)\cap\met_\nu^k(M)
\end{equation*}
is an open subset of the affine Banach space $g_\mathrm A+\mathds E^k$. In Section~\ref{sec:smooth1}, we endowed the countable intersection
\begin{equation*}
\A_{g_\mathrm A,\nu}^\infty=\bigcap_{k\geq 3}\A_{g_\mathrm A,\nu}^k
\end{equation*}
with the so--called $C^\infty$--topology, which is the smallest topology that makes all inclusions $i_k:\A_{g_\mathrm A,\nu}^\infty\hookrightarrow\A_{g_\mathrm A,\nu}^k$ continuous. Moreover, from Proposition~\ref{prop:nondegenerateopen}, the subset $\A_{g_\mathrm A,\nu,\p}^k$ of $\A_{g_\mathrm A,\nu}^k$, formed by\footnote{Recall \eqref{eq:anup}.} metrics $g$ such that $\overline g$ is nondegenerate on $\p$, is open, for each $k\geq3$. Similarly, consider
\begin{equation*}
\A_{g_\mathrm A,\nu,\p}^\infty=\bigcap_{k\geq3}\A_{g_\mathrm A,\nu,\p}^k,
\end{equation*}
which is open in $\A_{g_\mathrm A,\nu}^\infty$. In fact, for any $k\geq3$, the subset $\A_{g_\mathrm A,\nu,\p}^k$ is open in $\A_{g_\mathrm A,\nu}^k$. In addition, $$\A_{g_\mathrm A,\nu,\p}^\infty=\A_{g_\mathrm A,\nu,\p}^k\cap\A_{g_\mathrm A,\nu}^\infty,$$ and hence $\A_{g_\mathrm A,\nu,\p}^\infty$ is $\tau_k$--open and therefore open in $\A_{g_\mathrm A,\nu}^\infty$, see Remark~\ref{re:finerthanany}.

Our version of the generic property stated in Theorem~\ref{thm:bigone} in the $C^\infty$--topology will give genericity of smooth metrics without $(g,\p)$--degenerate geodesics in $\A_{g_\mathrm A,\nu,\p}^\infty$, for a given choice of an index $\nu$, a smooth auxiliary metric $g_\mathrm A$, a family $\{\mathds E^k\}_{k\geq3}$ where each $\mathds E^k$ is a separable $C^k$ Whitney type Banach spaces of sections of $TM^*\vee TM^*$ that tend to zero at infinity, and a $(g_\mathrm A,\nu)$--admissible GEC $\p$. More precisely, define
\begin{equation*}
\mathcal G_\p^\infty(M)=\left\{g\in\A_{g_\mathrm A,\nu,\p}^\infty: \begin{array}{c} \text{ all } (g,\p)\text{--geodesics }\gamma\in\op(M) \\ \text{ are nondegenerate }\end{array}\right\}.
\end{equation*}
and notice that $\mathcal G_\p^\infty(M)=\bigcap_{k\geq3}\mathcal G^k_\p(M)$, where
\begin{equation*}
\mathcal G^k_\p(M)=\left\{g\in\A^k_{g_\mathrm A,\nu,\p}: \begin{array}{c} \text{ all } (g,\p)\text{--geodesics }\gamma\in\op(M) \\ \text{ are nondegenerate }\end{array}\right\}.
\end{equation*}
Recall that for each $k\geq 3$, Theorem~\ref{thm:bigone} gives genericity of $\mathcal G^k_\p(M)$ in $\A^k_{g_\mathrm A,\nu,\p}$.

\begin{theorem}\label{thm:bigonecinfty}
Consider choices of $\nu$, $g_\mathrm A$, $\p$ and $\{\mathds E^k\}_{k\geq 3}$ as described above, and the $C^\infty$--topology induced in the intersection $\A_{g_\mathrm A,\nu,\p}^\infty$. Then the subset $\mathcal G^\infty_\p(M)$ is generic in $\mathcal A^\infty_{\nu,\p}$.
\end{theorem}

\begin{proof}
This proof is in great part adapted from Bettiol and Giamb\`o \cite[Proposition 5.12]{metmna}, and employs the same techniques used for instance in \cite{biljavapic,biljavapic2,fhs,gj}, described in Section~\ref{sec:smooth1}. For each $n\in\N$ let
\begin{equation}\label{eq:inftyrn}
\mathcal{R}_n^\infty=\left\{g\in\mathcal{A}_{g_\mathrm A,\nu,\p}^\infty:\begin{array}{c}\mbox{ all } (g,\p)\mbox{--geodesics }\gamma\mbox{ with } \\ L_\mathrm R (\gamma)\leq n\mbox{ are nondegenerate} \end{array}\right\}.
\end{equation}
Notice that $\mathcal G_\p^\infty(M)=\bigcap_{n\in\N}\mathcal R_n^\infty$, hence it suffices to prove that for each $n\in\N$, the subset $\mathcal R_n^\infty$ is open and dense in $\A^{\infty}_{g_\mathrm A,\nu,\p}$. It then follows that $\mathcal G_\p^\infty(M)$ contains a countable intersection of open dense subsets, and is hence generic.\footnote{Recall Definition~\ref{def:generic}.}

\begin{claim}\label{cl:ropen}
For each $n\in\N$ the subset $\mathcal R_n^\infty$ is open in $\A^{\infty}_{g_\mathrm A,\nu,\p}$.
\end{claim}

Notice that\footnote{See \eqref{eq:inftyrn} and \eqref{eq:ainfty}.} for each $k\geq3$, 
\begin{equation}\label{eq:seila2}
\mathcal R_n^\infty=\A_{g_\mathrm A,\nu,\p}^\infty\cap\mathcal R_n^k
\end{equation}
where $\mathcal{R}_n^k$ is given by \eqref{eq:Ralpha}, and the index $k$ stress the choice of regularity $C^k$ in that definition. More precisely, 
\begin{equation*}
\mathcal{R}_n^k=\left\{g\in\mathcal{A}_{g_\mathrm A,\nu,\p}^k:\begin{array}{c}\mbox{ all } (g,\p)\mbox{--geodesics }\gamma\mbox{ with } \\ L_\mathrm R (\gamma)\leq n\mbox{ are nondegenerate} \end{array}\right\}.
\end{equation*}
Observe that $\mathcal R_n^\infty=\bigcap_{k\geq3}\mathcal R_n^k$. 

Fix $k\geq 3$. Then Claim~\ref{cl:bigclaim3} gives that $\mathcal R_n^k$ is open in $\A^k_{g_\mathrm A,\nu,\p}$ for all $n\in\N$. From Remark~\ref{re:finerthanany}, this implies that $\mathcal R_n^\infty$ is also open in $\A^{\infty}_{g_\mathrm A,\nu,\p}$, since\footnote{Notice that the intersection $\mathcal R_n^\infty=\A^{\infty}_{g_\mathrm A,\nu,\p}\cap\mathcal R_n^k$ is open in $\A^{\infty}_{g_\mathrm A,\nu,\p}$ with the topology induced by the inclusion $i_k:\A^{\infty}_{g_\mathrm A,\nu,\p}\hookrightarrow\A^k_{g_\mathrm A,\nu,\p}$. The $C^\infty$--topology on $\A^{\infty}_{g_\mathrm A,\nu,\p}$ is finer than any of these topologies, for it is induced by the entire family $\{i_k\}_{k\geq3}$, hence $\mathcal R_n^\infty$ is open in $\A^{\infty}_{g_\mathrm A,\nu,\p}$.} it is $\tau_k$--open, concluding the proof of Claim~\ref{cl:ropen}.

\begin{claim}\label{cl:rdense}
For each $n\in\N$ the subset $\mathcal R_n^\infty$ is dense in $\A^{\infty}_{g_\mathrm A,\nu}$.
\end{claim}

Fix $n\in\N$, $k\geq3$ and consider once more the intersection \eqref{eq:seila2}. Claim~\ref{cl:bigclaim3} in the proof of Theorem~\ref{thm:bigone} gives that $\mathcal R_n^k$ is open in $\A^k_{g_\mathrm A,\nu,\p}$, and Claim~\ref{cl:bigclaim4} gives that $\mathcal R_n^k$ is dense in $\A^k_{g_\mathrm A,\nu,\p}$. In addition, from the Stone--Weierstrass Theorem~\ref{thm:stw2}, it is easy to conclude that $\A^\infty_{g_\mathrm A,\nu,\p}$ is also dense in $\A^k_{g_\mathrm A,\nu,\p}$. Therefore, setting $X=\A^k_{g_\mathrm A,\nu,\p}$, $U=\mathcal R_n^k$ and $D=\A^\infty_{g_\mathrm A,\nu,\p}$ in Lemma~\ref{le:dotausk}, it follows that $\mathcal R_n^\infty=\A^{\infty}_{g_\mathrm A,\nu,\p}\cap\mathcal R_n^k$ is dense in $\A^{\infty}_{g_\mathrm A,\nu,\p}$, concluding the proof.
\end{proof}

\begin{remark}
Analogously to Theorem~\ref{thm:bigonecinfty}, the $C^\infty$ version of the statements of Corollaries~\ref{cor:nonconjugacy},~\ref{cor:nonfocality} and~\ref{cor:nonfocality2} are automatically valid.
\end{remark}

\chapter{Final remarks and considerations}
\label{chap7}

In this short chapter, we make some final remarks on the topics discussed in the previous chapters, briefly mentioning some interesting details.

With regard to the main result in Chapter~\ref{chap5}, the Bumpy Metric Theorem~\ref{thm:bumpy}, there are several important consequences of this result well--studied in the literature. After the first article of Abraham \cite{abraham}, where the compact Riemannian version of this theorem was first announced, many other authors developed further genericity results for periodic geodesics and, more generally, periodic orbits of certain well--behaved Hamiltonian flows.

A well--known result of this type is the genericity statement of Klingenberg and Takens \cite{klitak}, on the $k$--jet of the {\em Poincar\'e map}, or {\em first recurrence map}, of any periodic geodesic. Roughly, it asserts that given $Q$ an open dense and invariant subset of the space of $k$--jets, for a generic metric in the $C^{k+1}$--topology the Poincar\'e map of every periodic geodesic belongs to $Q$. Biliotti, Javaloyes and Piccione \cite{biljavapic2} managed to use further perturbation properties on lightlike geodesics and establish a (compact) semi--Riemannian version of the Klingenberg--Takens generic property, \cite[Corollary 4.2]{biljavapic2}. At this stage, given the techniques used to prove the non compact version of the semi--Riemannian Bumpy Metric Theorem, Theorem~\ref{thm:bumpy}, it is reasonable to expect that the same perturbation arguments of Biliotti, Javaloyes and Piccione \cite{biljavapic2} may be used to obtain the non compact semi--Riemannian version of this generic property.

In addition, Contreras--Barandiar\'an and Paternain \cite{contpater} successfully used the classic Bumpy Metric Theorem to prove that generic Riemannian geodesic flows have positive {\em topological entropy}. The topological entropy $h_{\mathrm{top}}(g)$ of a metric $g$ is a dynamical invariant that roughly measures orbit structure complexity of a flow.\footnote{The following interesting characterization of this invariant was given by Ma\~n\'e \cite{mane}. Denoting $a_n(p,q)$ the number of geodesic segments joining $p$ and $q$ with length less then $n$, $$h_{\mathrm{top}}(g)=\lim_{n\to+\infty}\tfrac1n\log\int_{M\times M}a_n(p,q)\;\dd p\,\dd q.$$} In particular, positiveness of $h_{\mathrm{top}}(g)$ implies that the average number of geodesic segments joining $p$ and $q$ grows exponentially with length.

Counter--examples by Meyer and Palmore \cite{mp} point out that abstract Hamiltonian systems cannot be considered for generalizations of the Bumpy Metric Theorem to a more comprehensive class of dynamical flows. Basically, the dynamics of solutions differ in distinct energy levels, and hence the nondegeneracy property fails to be generic. In the particular case of geodesic flows, energy levels are well--organized, since adequate reparameterizations of periodic geodesics give other periodic geodesics with any prescribed energy, see \eqref{eq:ene}. The counter--examples in this paper also prove false a conjecture of Abraham and Marsden \cite{abmakekol} on genericity of nondegeneracy for periodic orbits. It exhibits a cylinder of periodic orbits in which the conjectured generic property is violated at isolated values of the energy, in a generic way.

Nevertheless, in some particular cases of Hamiltonian flows it is possible to infer a {\em bumpy--type} result. For instance, Gon\c{c}alves Miranda \cite{GonMir} recently proved genericity of nondegenerate periodic trajectories for the magnetic flow on surfaces. This result is also used to establish an extension of the Kupka--Smale Theorem, in this context. Thanks to these and many other implications of the Bumpy Metric Theorem, it became a central result in the theory of generic properties of flows. In this sense, it is also reasonable to expect that several of these subsequent results may be generalized for instance to the case of semi--Riemannian metrics, using the Bumpy Metric Theorem~\ref{thm:bumpy}.

Regarding the genericity results of nondegeneracy of semi--Riemannian geodesics under GECs in Chapter~\ref{chap6}, other than the given geometric applications,\footnote{For instance, Corollaries~\ref{cor:nonconjugacy},~\ref{cor:nonfocality},~\ref{cor:nonfocality2}.} they have relevant implications in general relativity. For instance, the problem of light conjugacy between an event and an observer may be regarded as a focality issue between a point and a submanifold in a space--time. Giamb\`o, Giannoni and Piccione \cite{ggp} studied genericity of nondegeneracy for lightlike geodesics in stationary space--times conjugating an event and an observer. Namely, after reducing the original problem to a Finsler geodesic problem via a second order Fermat principle for light rays, transversality techniques similar to those studied in Chapter~\ref{chap4} are used to establish the desired genericity.

In addition, it is possible to give more precise information on how to perturb a metric $g\in\A_{g_\mathrm A,\nu,p}\setminus\mathcal G_\p(M)$ to an arbitrarily close generic metric $g'\in\mathcal G_\p(M)$ without degenerate $(g,\p)$--geodesics.\footnote{See the context and notation of Theorem~\ref{thm:bigone}.} Namely, applying the same ideas used by Biliotti, Javaloyes and Piccione \cite[Section 4.3]{biljavapic} it is possible to prove that $g'$ may be taken conformal to $g$, i.e., there exists a $C^k$ positive function $f:M\to\R_+$ such that $g'=fg\in\mathcal G_\p(M)$, with $f$ arbitrarily close to the constant function equal to $1$. This gives a qualitative refinement of Theorem~\ref{thm:bigone}, and the proof is very similar to the proof of \cite[Proposition 4.4]{biljavapic}. More precisely, it employs the same arguments of Theorem~\ref{thm:transvholds} on $C^k$ sections of the trivial bundle $M\times\R$, which are $C^k$ functions on $M$. The absence of strongly degenerate geodesics then allows to conclude the existence of a such positive function $f$ that gives the desired generic metric conformal to $g$.

Finally, it is also possible to explore consequences of Theorem~\ref{thm:bigone} for GECs in more specific ambients, such as {\em orthogonally split} metrics, {\em globally hyperbolic space--times} and {\em stationary space--times}. This is also possible in a similar fashion to \cite[Sections 4.4-4.6]{biljavapic}, where results are explored in the case of fixed endpoints conditions $\p=\{p\}\times\{q\}$, with $p\neq q$.

\backmatter
\renewcommand{\chaptermark}[1]{\markboth{#1}{}}
\renewcommand{\sectionmark}[1]{\markright{#1}}


\clearpage
\phantomsection
\bibliographystyle{amsplain}


\clearpage
\phantomsection
\printindex

\end{document}